\documentclass{amsart}
\usepackage[applemac]{inputenc}
\usepackage[T1]{fontenc}
\usepackage{hyperref}
\usepackage{amssymb}
\usepackage[matrix,arrow]{xy}
\usepackage{mathrsfs}
\newtheorem{thm}[subsubsection]{Theorem}
\newtheorem{defi}[subsubsection]{Definition}
\newtheorem{prop}[subsubsection]{Proposition}

\newtheorem{lem}[subsubsection]{Lemma}
\newtheorem{coro}[subsubsection]{Corollary}
\newtheorem{conj}[subsubsection]{Conjecture}
\theoremstyle{remark}

\newtheorem{rem}[subsubsection]{Remark}
\newtheorem{ex}[subsubsection]{Example}
\newtheorem{Principle}[subsubsection]{Principle}

\newcommand{\colim}{\mathrm{colim}}

\newcommand{\C}{\mathbb{C}}
\newcommand{\R}{\mathbb{R}}
\newcommand{\qq}{\mathbb{Q}}

\newcommand{\ZZ}{\mathbb{Z}}
\newcommand{\Z}{\mathbb{Z}}
\newcommand{\N}{\mathbb{N}}

\newcommand{\Spec}{\mathrm{Spec}}
\newcommand{\Spf}{\mathrm{Spf}}
\newcommand{\Spa}{\mathrm{Spa}}

\newcommand{\HH}{\mathrm{H}}

\newcommand{\ocal}{\mathcal{O}}
\newcommand{\oscr}{\mathscr{O}}

\def\WM{{}^MW}

\title{Higher Coleman Theory}
\author[G. Boxer]{George Boxer}  \email{george.boxer@universite-paris-saclay.fr} \address{Institut de math\'ematique d'Orsay,
Universit\'e Paris-Saclay,
B\^atiment 307, rue Michel Magat,
F-91405 Orsay Cedex,
France}
\author[V. Pilloni]{Vincent Pilloni}\email{vincent.pilloni@universite-paris-saclay.fr}\address{Institut de math\'ematique d'Orsay,
Universit\'e Paris-Saclay,
B\^atiment 307, rue Michel Magat,
F-91405 Orsay Cedex,
France}
\begin{document}

\maketitle

\begin{abstract}We develop local cohomology techniques to study the finite slope part of the coherent cohomology of Shimura varieties. The local cohomology groups we consider are a generalization of overconvergent modular forms, and they are defined by using a stratification on the Shimura variety obtained from the Bruhat stratification on a flag variety via the Hodge-Tate period map.  We construct a spectral sequence from the local cohomologies to the classical cohomology and use it to obtain classicality and vanishing results. We also develop a theory of $p$-adic families and construct eigenvarieties.  As an application, we prove some new properties of Galois representations arising from certain non-regular algebraic cuspidal automorphic representations.
\end{abstract}

\tableofcontents

\section{Introduction}

In the 90's, Coleman proved his classicality theorem for overconvergent $p$-adic modular forms \cite{MR1369416} and developed the theory of $p$-adic families of finite slope overconvergent modular forms \cite{1997InMat.127..417C}.  Subsequently, Coleman and Mazur constructed the eigencurve \cite{MR1696469}.  Since then, these works have been generalized to a wide class of Shimura varieties.

These Coleman theories concern classical and overconvergent modular forms, or zeroth cohomology of automorphic vector bundles on Shimura varieties.  Recently we have come to expect \cite{pilloniHidacomplexes}, \cite{BCGP}, \cite{loeffler2019higher}, \cite{Boxer-Pilloni} that there should be some analogous ``higher Coleman theories'' concerning the higher coherent cohomology of Shimura varieties.  The goal of this paper is to develop such theories for a rather general class of Shimura varieties.

\subsection{Setup}
Let $(G,X)$ be a Shimura datum of abelian type, so by definition $G/\qq$ is a reductive group and $X$ is a complex analytic space with an action of $G(\mathbb{R})$ satisfying a list of axioms (\cite{MR546620}). There are two opposite parabolic subgroups of $G$ attached to $(G,X)$, called $P_\mu$ and $P_\mu^{std}$, with common Levi $M_\mu$. The space $X$ embeds $G(\mathbb{R})$-equivariantly as an open subspace of  $FL_{G,\mu}^{std}(\C) = G/P_{\mu}^{std}(\C)$. This is the Borel embedding. 

For any neat compact open subgroup $K \subseteq G(\mathbb{A}_f)$, we let $S_K(\mathbb{C}) = G(\qq) \backslash X \times G(\mathbb{A}_f)/K$ be the corresponding Shimura variety over $\C$. This is a finite disjoint union of arithmetic quotients of $X$. 

Any representation of  $P_\mu^{std}$ defines a $G$-equivariant vector bundle over $FL_{G,\mu}^{std}$. Let $Z_s \subseteq G$ be the maximal torus of the center of $G$ which splits over $\mathbb{R}$ but contains no $\mathbb{Q}$-split subtorus.  We let $M_\mu^c = M_\mu/Z_s$. By pull back to $X$ and descent to $S_K(\mathbb{C})$, we obtain a functor from the category of representations of $M^c_\mu$ to the category of vector bundles on $S_K(\mathbb{C})$, called (totally decomposed) automorphic vector bundles. 
We make a choice of Borel subgroup $B$ contained in $P_\mu$, let $T$ be a maximal torus contained in this Borel, and let $T^c = T/Z_s$. We label irreducible representations of $M^c_\mu$ by their highest weight in $X^\star(T^c)^{M_\mu,+}$. For any weight $\kappa \in X^\star(T^c)^{M_\mu,+}$ we let $\mathcal{V}_{\kappa}$ be the corresponding vector bundle over $S_K(\C)$. 

The Shimura variety $S_K(\C)$ has a structure of an algebraic variety $S_K$ defined over a number field $E$, called the reflex field. For a combinatorial choice $\Sigma$ of cone decomposition, there are projective compactifications $S_{K,\Sigma}^{tor}$ whose (reduced) boundary $D_{K,\Sigma} =  S_{K,\Sigma}^{tor} \setminus S_K$ is a Cartier divisor. The vector bundles $\mathcal{V}_\kappa$ admit models over $S_K$ and canonical extensions $\mathcal{V}_{\kappa,\Sigma}$ to $S_{K, \Sigma}^{tor}$ (\cite{MR1044823}, \cite{harris-ann-arb}).  We also have $\mathcal{V}_{\kappa,\Sigma}(-D_{K,\Sigma})$, the so called sub-canonical extension (its sections are holomorphic cusp forms).

This paper is devoted to the study of the coherent cohomology of weight $\kappa$: $\HH^i(S^{tor}_{K,\Sigma}, \mathcal{V}_{\kappa,\Sigma})$, $\HH^i(S^{tor}_{K,\Sigma}, \mathcal{V}_{\kappa,\Sigma}(-D_{K,\Sigma}))$ as well as the interior cohomology:
$$ \overline{\HH}^i(S^{tor}_{K,\Sigma}, \mathcal{V}_{\kappa,\Sigma}) = \mathrm{Im} ( {\HH}^i(S^{tor}_{K,\Sigma}, \mathcal{V}_{\kappa,\Sigma}(-D_{K,\Sigma})) \rightarrow {\HH}^i(S^{tor}_{K,\Sigma}, \mathcal{V}_{\kappa,\Sigma})).$$
They are independent of the cone decomposition $\Sigma$ and we denote them simply by $\HH^i(K,\kappa)$, $\HH^i(K,\kappa,cusp)$, and $\overline{\HH}^i(K,\kappa)$.

We now fix a rational prime $p$ and pass to the $p$-adic theory.  We assume throughout this paper that $G_{\qq_p}$ is quasi-split (this assumption is necessary to have finite slope families over all of weight space).  We fix a finite extension $F$ of $\qq_p$ over which $G_{\qq_p}$ splits, and we assume that we have chosen $T\subseteq B\subseteq P_\mu$ so that $T$ and $B$ are defined over $\qq_p$ and $P_\mu$ is defined over $F$.  We now pass to analytic geometry and let $\mathcal{S}_{K,\Sigma}^{tor}/\Spa(F,\ocal_F)$ be the adic Shimura variety.  We assume for the rest of the introduction that $K=K^pK_p$ where $K_p\subset G(\qq_p)$ is a certain congruence subgroup (denoted $K_{p,1,0}$ in \ref{section-compact-open-subgroups}) possessing an Iwahori factorization
$$K_p=(\overline{U}(\qq_p)\cap K_p)\cdot T(\ZZ_p)\cdot (U(\qq_p)\cap K_p)$$
where $U$ and $\overline{U}$ denote the unipotent radicals of $B$ and the opposite Borel.  When $G_{\qq_p}$ is unramified, $K_p$ will be an Iwahori subgroup.

Our study of the $p$-adic properties of coherent cohomology is based on the Hodge-Tate period map introduced by Scholze \cite{scholze-torsion} and refined in \cite{MR3702677} and \cite{diao2019logarithmic}.  More specifically we will use the existence of a diagram
\begin{eqnarray*}
\xymatrix{\mathcal{S}_{K^p}\ar[r]^{\pi_{HT}}\ar[d]&\mathcal{FL}_{G,\mu}\ar[d]\\
\mathcal{S}_{K^pK_p}\ar[r]^{\pi_{HT,K_p}}&\mathcal{FL}_{G,\mu}/K_p}
\end{eqnarray*}
where $\mathcal{FL}_{G,\mu}=P_\mu\backslash G$ is the adic space associated to the Hodge-Tate flag space $FL_{G,\mu}=P_\mu\backslash G$.  Here we only view $\mathcal{FL}_{G,\mu}/K_p$ as a topological space, and the ``truncated'' Hodge-Tate period map $\pi_{HT,K_p}$ as a continuous map.  The map $\pi_{HT}$ is equivariant for the right actions of the group $G(\qq_p)$, while the map $\pi_{HT,K_p}$ is equivariant for the action of the Hecke algebra $\mathcal{H}_p=F[K_p\backslash G(\qq_p)/K_p]$ by correspondences.  It also extends to a map on the toroidal compactification $\pi_{HT,K_p}^{tor}:\mathcal{S}_{K^pK_p,\Sigma}^{tor}\to\mathcal{FL}_{G,\mu}/K_p$.  The basic idea in this paper is to use $\pi_{HT,K_p}^{tor}$ to define certain support conditions for coherent cohomology on the Shimura variety, and to study the dynamical properties of certain Hecke operators at $p$.

We introduce some further notation.  For a choice of $+$ or $-$ which we denote from now on by $\pm$, we consider the monoids
$$T^{\pm}=\{t\in T(\qq_p)\mid v(\alpha(t))\geq0,\forall \alpha\in\Phi^{\pm}\}$$
where $\Phi^+$ denotes the positive roots of $G$ with respect to $T\subseteq B$, and also the associated commutative subalgebras of the Iwahori Hecke algebra
$$\mathcal{H}_p^{\pm}=\mathrm{span}\{[K_ptK_p]\mid t\in T^\pm\} \subseteq\mathcal{H}_p.$$
These should be thought of as the algebra of ``$U_p$-type'' Hecke operators.  The reason we consider both the $+$ and $-$ algebras is in order to study duality (we note that $[K_ptK_p]$ and $[K_pt^{-1}K_p]$ are transposes as correspondences), but the reader might assume that all $\pm$'s are just $+$ on first reading.

We let $W$ be the (absolute) Weyl group of $G$ and $W_M\subseteq W$ the Weyl group of $M_\mu$.  We write $\ell:W\to\ZZ$ for the length function and we write $w_0\in W$ and $w_{0,M}\in W_M$ for the longest elements.  In this paper a crucial role will be played by the set $\WM\subseteq W$ of minimal length coset representatives for $W_M\backslash W$.  It has a unique longest element $w_0^M=w_{0,M}w_0$ of length $d=\dim S_K=\dim FL_{G,\mu}$.  We define $\ell_{\pm}:\WM\to[0,d]$ by $\ell_+=\ell$ and $\ell_-(w)=d-\ell(w)$.

The set $\WM$ will index the ``higher Coleman theories'' constructed in this paper.  We recall how it arises naturally in the study of coherent cohomology, both from the archimedean and $p$-adic points of view.  

The coherent cohomologies $\HH^i(K,\kappa)$ for $\kappa\in X^\star(T^c)^{M_\mu,+}$ can be computed in terms of automorphic forms on $G$ (\cite{harris-ann-arb}, \cite{su2018coherent}).  Cuspidal automorphic representations $\pi=\pi_\infty\otimes\pi_f$ contribute to coherent cohomology according to the archimedean component $\pi_\infty$.  The (essentially) tempered $\pi_\infty$ which contribute to coherent cohomology have been completely classified: according to work of Blasius-Harris-Ramakrishnan, Mirkovich, Schmid, and Williams, they are exactly the so called non-degenerate limits of discrete series (see \cite{harris-ann-arb}, theorems 3.4 and 3.5).  According to Harish-Chandra, for each element of $\WM$ there is a corresponding family of (limits of) discrete series representations, parameterized by the infinitesimal character (more accurately, this would be true if $G(\mathbb{R})$ were semisimple and simply connected).

We recall how the limits of discrete series contribute to coherent cohomology.  Beginning with a (dominant) ``C-algebraic'' infinitesimal character $\nu+\rho\in X^\star(T^c)_{\mathbb{R}}^+$ with $\nu\in X^\star(T^c)$, then for a choice of $w\in\WM$, if $\kappa_w=-w_{0,M}w(\nu+\rho)-\rho$ is $M_\mu$-dominant, then there is a (limit of) discrete series with infinitesimal character $\nu+\rho$ contributing to coherent cohomology in weight $\kappa_w$ and degree $\ell(w)$.  We recall some features of this description:
\begin{itemize}
\item The case of $w=1$ corresponds to the family of holomorphic (limits of) discrete series.  They contribute to degree zero coherent cohomology.
\item When $\nu+\rho$ is regular, $\kappa_w$ is automatically $M_\mu$-dominant, and the corresponding representations of $G(\mathbb{R})$ are discrete series.  Moreover a $\kappa\in X^\star(T^c)^{M_\mu,+}$ with $\kappa+\rho$ regular will be of the form $\kappa_w$ for a unique $\nu$ and $w\in\WM$, and so cusp forms which are tempered at $\infty$ contribute to cohomology in weight $\kappa$ in only a single degree.
\item When $\nu+\rho$ is not regular, there will be some subset of $\WM$ (possibly empty) for which $\kappa_w$ is $M_\mu$-dominant (we emphasize in particular that this set may be nonempty but not contain $1$, and this is a major motivation for studying higher coherent cohomology of Shimura varieties).  In this case the corresponding representations of $G(\mathbb{R})$ are limits of discrete series.  Beginning with a $\kappa\in X^\star(T)^+$ with $\kappa+\rho$ irregular, we will have $\kappa=\kappa_w$ for multiple $w\in\WM$ (with a range of lengths), and so the cohomology in weight $\kappa$ will have contributions from cuspidal automorphic representations tempered at $\infty$ in a range of degrees.
\end{itemize}
We note that in this paper, our study of coherent cohomology is purely geometric, and makes no use of the theory of automorphic forms or the results recalled above (except in the application to local-global compatibility for irregular automorphic representations in section \ref{subsec-localglobal}).  Nonetheless, it is useful to keep in mind, for interpreting the results that follow.

Now we pass to $p$-adic geometry, beginning with the Hodge-Tate flag space $FL_{G,\mu}=P_\mu\backslash G$. The set   $\WM$ indexes the Schubert stratification of $FL_{G,\mu}$:
$$FL_{G,\mu}=P_\mu\backslash G=\coprod_{w\in WM}P_\mu\backslash P_\mu w B$$
Passing to a more dynamical point of view, the set of fixed points for the action of $T$ on $FL_{G,\mu}$ is exactly
$$\mathrm{Fix}_T(FL_{G,\mu})=\{P_\mu w\mid w\in\WM\}$$
as indeed there is one $T$ fixed point in each Schubert cell.  A closely related fact is that the set of points of $\mathrm{FL}_{G,\mu}(\C_p)/K_p$ fixed by all of the ``$U_p$-type'' Hecke correspondences $[K_ptK_p]$ for $t\in T^\pm$ is
$$\{P_\mu w K_p\mid w\in \WM\}.$$
(Here we say that a point $xK_p\in \mathrm{FL}_{G,\mu}(\C_p)/K_p$ is fixed by $[K_ptK_p]$ means that $x\in xK_ptK_p$.)

According to Scholze-Weinstein \cite{MR3272049}, $FL_{G,\mu}(\C_p)/K_p$ should be thought of as the set of $p$-divisible groups ``with $G$-structure'' over $\ocal_{C_p}$, equipped with a $K_p$ level structure (on the generic fiber).  In particular, the points $P_\mu wK_p$ for $w\in\WM$ correspond to certain very special such $p$-divisible groups, which are fixed points for the action of the ``$U_p$-type'' Hecke correspondences.  We can now pass to the Shimura variety via the Hodge-Tate period map.  According to Caraiani-Scholze \cite{MR3702677}, the fibers of $\pi_{HT}$ are perfectoid Igusa varieties.  In this paper we will consider the ``Igusa varieties'' $(\pi_{HT,K_p}^{tor})^{-1}(P_\mu wK_p)\subset \mathcal{S}_{K,\Sigma}^{tor}$ at finite level.  We view them as some kind of fixed subspaces for the ``$U_p$-type'' Hecke correspondences.  We emphasize that in contrast to the situation at infinite level, they are not Zariski closed, and we do not try to give them any geometric structure (in fact, we will only consider neighborhoods of them, constructed using the Hodge-Tate period map).

We give two illustrative examples:
\begin{ex}If $(G,X)=(\mathrm{GSp}_{2g},\mathcal{H}_g^{\pm})$ is the Siegel Shimura datum and $K_p$ is Iwahori, then the corresponding Shimura varieties are moduli spaces of polarized abelian varieties of dimension $g$.  The polarized $p$-divisible groups with Iwahori level structures $(\mathcal{G},\lambda,\{H_i\}_{0\leq i\leq 2g})$ corresponding to the elements of $\WM$ are all ordinary (i.e. $\mathcal{G}=\mu_{p^\infty}^g\times (\qq_p/\ZZ_p)^g$) and the element $w\in\WM$ measures the relative position of the canonical subgroup $H_{can}=\mu_p^g$ with the symplectic flag $$0=H_0\subset H_1\subset\cdots\subset H_{2g}=\mathcal{G}[p]$$ giving the Iwahori level structure (alternatively, the $2^g$ elements of $\WM$ correspond to the $2^g$ possibilities that $H_i/H_{i-1}$ for $i=1,\ldots,g$ is either multiplicative or \'etale).  Passing to the Shimura variety, the ``Igusa varieties'' $\pi_{HT,K_p}^{-1}(P_\mu wK_p)$ are simply the (closures of the) corresponding components of the (quasi-compact, open) ordinary locus.
\end{ex}
\begin{ex}If $(G,X)=(\mathrm{Res}_{F/\qq}\mathrm{GL}_2,(\mathcal{H}_1^{\pm})^{[F:\qq]})$ is a Hilbert Shimura datum for $F$ a totally real field in which $p$ remains prime (for simplicity), then the corresponding Shimura varieties are (coarse) moduli spaces of abelian varieties of dimension $[F:\qq]$ with an action of $\ocal_F$.  In this case, $\WM=W=\{1,w\}^{\mathrm{Hom}(F,\C_p)}$.  For $I\subseteq\mathrm{Hom}(F,\C_p)$ we write $w_I\in W$ for the element which is $1$ in the factors corresponding to the elements of $I$.  Then the $p$-divisible group with $\ocal_F$-action and Iwahori level structure corresponding to $w_I$ is $(\mathrm{LT}_I\times\mathrm{LT}_{I^c},\mathrm{LT}_{I^c}[p])$ where $I^c=\mathrm{Hom}(F,\C_p)\setminus I$, and $\mathrm{LT}_I/\ocal_{\C_p}$ is the unique $p$-divisible group with $\ocal_F$-action of height $[F:\qq]$ and such that $\mathrm{Lie}_{LT_I}=\oplus_{\tau\in I}\C_p(\tau)$ as $F\otimes\C_p$ modules, where $\C_p(\tau)$ denotes the $F\otimes\C_p$ module on which $F$ acts via $\tau$.  The isogeny corresponding to quotienting by $\mathrm{LT}_{I}[p]$ witnesses these $p$-divisible groups as fixed points for the ``$U_p$-type'' Hecke correspondences.  In particular, in contrast to the previous example, only $1=w_{\mathrm{Hom}(F,\C_p)}$ and $w_0=w_\emptyset$ correspond to ordinary $p$-divisible groups.
\end{ex}

\subsection{Overconvergent cohomologies}
The first aim of this paper is the construction, for all $w\in\WM$, weights $\kappa\in X^{\star}(T^c)^{M_\mu,+}$, and choices $\pm$ of $+$ or $-$, of some finite slope, overconvergent cohomologies $\HH_w^i(K^p,\kappa)^{\pm,fs}$ and $\HH_w^i(K^p,\kappa,cusp)^{\pm,fs}$.  These spaces carry actions of the algebras $\mathcal{H}_p^\pm$ for which the ``$U_p$-type'' operators $[K_ptK_p]$ for $t\in T^\pm$ act invertibly (hence ``finite slope'') as well as actions of the prime to $p$ Hecke operators.

We explain the rough idea of the construction: we start with the complex $\mathrm{R}\Gamma_{\mathcal{U}\cap\mathcal{Z}}(\mathcal{U},\mathcal{V}_{\kappa,\Sigma})$ of cohomology with support, where $\mathcal{U}\cap\mathcal{Z}\subset\mathcal{S}_{K,\Sigma}^{tor}$ is a suitably chosen, locally closed neighborhood of the ``Igusa variety'' $\pi_{HT,K_p}^{-1}(w\cdot K_p)$.  The support conditions are chosen in order to have an action of the Hecke operators $[K_ptK_p]$ for $t\in T^{\pm}$, and so that furthermore sufficiently regular Hecke operators act compactly.  Then we use the spectral theory of compact operators to take the finite slope part, i.e. pass to the part where the Hecke operators $[K_ptK_p]$ for $t\in T^{\pm}$ act invertibly.  The space $\HH^0_1(K^p,\kappa)^{+,fs}$ is nothing but the usual space of finite slope overconvergent modular forms of weight $\kappa$.

The following theorem summarizes the basic results about the finite slope overconvergent cohomologies.

\begin{thm}\label{thm-overconvergent-introduction}
\begin{enumerate}
\item {\bf(Spectral sequence to classical cohomology)} There is a $\mathcal{H}_p^{\pm}$-equivariant spectral sequence
$$\mathbf{E}_1^{p,q}=\bigoplus_{w\in\WM,\ell_{\pm}(w)=p}\HH_w^{p+q}(K^p,\kappa)^{\pm,fs}\Rightarrow \HH^{p+q}(K,\kappa)^{\pm,fs}$$
and similarly for cuspidal cohomology.
\item {\bf(Analytic continuation)} The finite slope, overconvergent cohomologies $\mathrm{R}\Gamma_w(K^p,\kappa)^{\pm,fs}$ and $\mathrm{R}\Gamma_w(K^p,\kappa,cusp)^{\pm,fs}$, together with their Hecke action, can be computed on arbitrarily small neighborhoods of the ``Igusa variety'' $\pi_{HT,K_p}^{-1}(w\cdot K_p)\subseteq\mathcal{S}_{K,\Sigma}^{tor}$.  (See sections \ref{subsection-analytic-continuation} and \ref{section-overconvergent-cohomologies} for a precise statement.)
\item {\bf(Duality)} There is a perfect pairing
$$\langle,\rangle:\HH^i_w(K^p,\kappa,cusp)^{\pm,fs}\times \HH^{d-i}_w(K^p,\kappa^\vee)^{\mp,fs}\to F$$
such that for all $t\in T^\pm$, the Hecke operators $[K_ptK_p]$ on the left and $[K_pt^{-1}K_p]$ on the right are adjoint.  Moreover, these pairings are compatible, via the spectral sequence of part (1), with the classical Serre duality pairings
$$\HH^i(K,\kappa,cusp)\times \HH^{d-i}(K,\kappa^\vee)\to F.$$
(Here $\kappa^\vee:=-2\rho_{nc}-w_{0,M}\kappa$ where $2\rho_{nc}\in X^\star(T)$ is the sum of the positive roots of $G$ which are not roots of $M_\mu$.  The canonical bundle of $S^{tor}_{K,\Sigma}$ is $\mathcal{V}_{-2\rho_{nc},\Sigma}(-D_{K,\Sigma})$, and the dual of $\mathcal{V}_{\kappa,\Sigma}$ is $\mathcal{V}_{-w_{0,M}\kappa,\Sigma}$.)
\item {\bf(Vanishing)} We have 
\begin{eqnarray*}
\HH^i_w(K^p,\kappa)^{\pm,fs}&=&0\quad\mathrm{for}\quad i<\ell_{\pm}(w),\\
\HH^i_w(K^p,\kappa,cusp)^{\pm,fs}&=&0\quad\mathrm{for}\quad i>\ell_{\pm}(w),
\end{eqnarray*}
and so in particular the interior cohomology $\overline{\HH}_w^i(K^p,\kappa)^{\pm,fs}$ vanishes except when $i=\ell_{\pm}(w)$.
\end{enumerate}
\end{thm}

The spectral sequence is simply the finite slope part of the spectral sequence of a filtration on $S_{K,\Sigma}^{tor}$ defined using $\pi_{HT}$ and the Bruhat stratification of the flag variety.  We note that if $K_p$ is Iwahori, then the ``$U_p$-type'' operators $[K_ptK_p]$ for $t\in T^{\pm}$ are already invertible in the Iwahori Hecke algebra $\mathcal{H}_p$, and hence the classical cohomology at Iwahori level is already ``finite slope'', i.e. $\HH^i(K,\kappa)^{\pm,fs}=\HH^i(K,\kappa)$.  The ``analytic continuation'' result follows from a dynamical study of the Hecke operators an their interactions with the support conditions.  The vanishing theorem is ultimately deduced from the affineness of the Hodge-Tate period map.

If the Shimura variety is compact, then the vanishing theorem implies that the overconvergent cohomologies $\mathrm{H}^\star_w(K^p,\kappa)^{\pm,fs}$ are concentrated in the single degree $\ell_{\pm}(w)$, and so the spectral sequence takes the particularly simple form of a single complex
$$ \HH^0_{Id/w_0^M}( K^p, \kappa)^{\pm,fs} \rightarrow \cdots \rightarrow \bigoplus_{w \in \WM,\ell_{\pm}(w)=i} \HH^{i}_{w}(K^p, \kappa)^{\pm,fs}\to \cdots \rightarrow \HH^{d}_{w_0^M/Id} ( K^p, \kappa)^{\pm,fs}$$
whose cohomology is the classical cohomology $\HH^i(K,\kappa)$.  We call this the Cousin complex of the Shimura variety, in analogy with the work of Kempf \cite{MR509802} on the cohomology of the flag variety.

\subsection{Slope bounds and classicality}

The key input for proving the classicality theorem is a lower bound for the slopes (i.e. $p$-adic valuations of eigenvalues) of the Hecke operators acting on finite slope, overconvergent cohomology.  More precisely we prove the following (see theorems \ref{thm-slopes} and \ref{thm-strongslopes-interior}).

\begin{thm}[Slope bounds]
For any $t\in T^{\pm}$ and any eigenvalue $\lambda$ of $[K_ptK_p]$ acting on $\overline{\HH}_w^i(K^p,\kappa)^{\pm,fs}$ we have
\begin{eqnarray*}
&v(\lambda)\geq v((w^{-1}w_{0,M}(\kappa+\rho)+\rho)(t))& \textrm{in the $+$ case}\\
&v(\lambda)\geq v((w^{-1}(\kappa+\rho)-\rho)(t))& \textrm{in the $-$ case}
\end{eqnarray*}
\end{thm}

Here $\rho\in X^\star(T)_\qq$ is half the sum of the positive roots of $G$ and, given an algebraic character $\mu\in X^\star(T)$, we may view it as a character $\mu:T(F)\to F^\times$, so that it makes sense to evaluate it on $t\in T(\qq_p)$ and take the valuation.  We remark that an eigenclass is ``ordinary'' if the inequalities of the theorem are in fact equalities.

Actually we expect (conjecture \ref{conj-strongslopes}) that the same bound holds for the full overconvergent cohomologies $\HH^i_w(K^p,\kappa)^{\pm,fs}$ and $\HH^i_w(K^p,\kappa,cusp)^{\pm,fs}$, but we were only able to prove a slightly weaker bound (theorem \ref{thm-slopes}).  The improved bound on interior cohomology above is proved by a bootstrapping trick using $p$-adic interpolation.

From the slope bound we can immediately deduce a classicality theorem in regular weight.  The idea is that if we fix $\kappa$, the lower bounds vary with $w\in\WM$, and so there will be a certain range of ``small slopes'' which can only occur for a single $w$.  Then by passing to the small slope part of the spectral sequence we will obtain an isomorphism between the small slope parts of the $w$ overconvergent and classical cohomologies.

\begin{thm}[Classicality theorem]\label{thm-classicity-intro} Let $\kappa \in X^\star(T^c)^{M,+}$ be a weight such that $\kappa + \rho$ is $G$-regular.  Let $w_+$ (resp. $w_-$) be the unique $w\in\WM$ such that $-w^{-1}w_{0,M}(\kappa+\rho)$ (resp. $w^{-1}(\kappa+\rho)$) is $G$-dominant.

Then we have  a quasi-isomorphism:
$$\mathrm{R}\Gamma_{w_\pm}(K^p, \kappa)^{\pm,sss^M(\kappa)} =  \mathrm{R}\Gamma(K^p, \kappa)^{\pm,sss^M(\kappa)}$$
and similarly for cuspidal cohomology.  Here the superscript is the ``strongly small slope part'', i.e. the part where the eigenvalues of the Hecke operators $[K_ptK_p]$ for $t\in T^{\pm}$ satisfy certain upper bounds (see section \ref{section-slope-conditions} for the precise definition).

If the Shimura variety is compact, then ``strongly small slope'' can be replaced by the weaker ``small slope'' condition $\pm,ss^M(\kappa)$ (again see section \ref{section-slope-conditions}).  Even if the Shimura variety is not compact it is still true that a small slope overconvergent cusp form (i.e. an element of $H^0_1(K^p,\kappa,cusp)^{+,ss^M(\kappa)})$ is classical.
\end{thm}

\begin{rem} Many cases of this theorem for the degree $0$  cohomology of  PEL Shimura varieties  were already proven, see for example  \cite{MR1369416}, \cite{MR2219265}, \cite{MR2783930}, \cite{MR3488741}, \cite{MR3581176}, \cite{MR3742470}.
\end{rem}
\begin{rem} Even in cases of degree 0 cohomology where classicality theorems were already known, this theorem often improves on the slope bound.  For example in the Siegel case, with the usual labelling of weights $\kappa = (k_1, \cdots, k_g)$ with $k_1 \geq \cdots \geq k_g$, and for the usual ``minuscule'' Hecke operator $U_p$, the theorem states that any finite slope overconvergent cusp form of weight $\kappa$ and $U_p$-slope $< k_g-g$ is classical. The bound  previously proven in  \cite{MR2783930} and \cite{MR3488741} was $<k_g- \frac{g(g+1)}{2}$.  Holomorphic discrete series contribute to the $\HH^0$ in weights with $k_g>g$, while holomorphic limits of discrete series contribute when $k_g=g$.  Thus we cannot expect any ``numerical'' classicality results when $k_g=g$, while the theorem is non-vacuous for all $k_g>g$.
\end{rem}

\begin{rem}We conjecture that the theorem should hold with ``small slope'' instead of ``strongly small slope'' for all Shimura varieties.  This would follow from the conjectural improved slope bound \ref{conj-strongslopes}.  Moreover we believe that the ``small slope'' condition is optimal.  We remark that an ``ordinary'' cohomology class is small slope in all weights, but it may only be strongly small slope under an additional regularity condition.
\end{rem}

\begin{rem}
We cannot expect to have any ``numerical'' classicality criteria when $\kappa+\rho$ is not $G$-regular.  In this case, there are multiple $w\in\WM$ for which $-w^{-1}w_{0,M}(\kappa+\rho)$ (or $w^{-1}(\kappa+\rho)$) are $G$-dominant, and the lower bound on the slopes are the same for these $w$.  Thus we can impose some small slope condition to kill the contributions of the remaining $w\in\WM$ to the spectral sequence, but there will still be further cancellation among those that remain.
\end{rem}

\begin{ex}
We explicate all the results so far in the case of the modular curve, where everything was already done in \cite{Boxer-Pilloni}.  Then $\WM=W=\{1,w\}$ has only two elements, and the spectral sequence reduces to a four term exact sequence
$$0\rightarrow \HH^0(K,k)\rightarrow \HH^0_1(K^p,k)^{+,fs}\rightarrow \HH^1_w(K^p,k)^{+,fs}\rightarrow \HH^1(K,k)\to 0$$
relating the the classical coherent cohomology of the modular line bundle $\omega^k$, $k\in\ZZ$, to the two overconvergent cohomologies, and moreover $H^0_1(K,k)^{+,fs}$ is nothing but the usual space of finite slope overconvergent modular forms of weight $k$.

Then the slope bound says that the ``unnormalized $U_p$ operator'' $U_p^{naive}$ has slopes $\geq1$ on $\HH^0_1(K,k)^{+,fs}$, and $\geq k$ on $\HH^1_w(K,k)^{+,fs}$.  Motivated by this we renormalize the $U_p$ operator by setting $U_p=p^{-\min(1,k)}U_p^{naive}$, and then from the four term exact sequence above and the slope bounds, we immediately deduce Coleman's classicality theorem
$$\HH^0_1(K^p,k)^{<k-1}=\HH^0(K,k)^{<k-1}$$
in weights $k\geq 2$, as well as its ``$\HH^1$-analog''
$$\HH^1_w(K^p,k)^{<1-k}=\HH^1(K,k)^{<1-k}$$
in weights $k\leq 0$.  (Here the superscripts $<\star$ mean we take the slope $<\star$ part for the normalized $U_p$ operator.)
\end{ex}

Combining the slope bounds with the vanishing theorem for overconvergent cohomology, we obtain vanishing theorems for the small slope part of classical cohomology.  For any $\kappa \in X_\star(T^c)^{M_\mu,+}$, there is a range $[\ell_{\min}(\kappa), \ell_{\max}(\kappa)] \subseteq [0, d]$ where cuspidal automorphic representations which are (limits of) discrete series at $\infty$ contribute.  We recall a combinatorial description of this range: let $C(\kappa)^+ = \{ w \in W, -w^{-1}w_{0,M} (\kappa + \rho) \in X^\star(T)_{\qq}^+\}$ and then let $\ell_{\min}(\kappa) = \min_{w \in C(\kappa)^+} \ell(w)$, $\ell_{\max}(\kappa) = \max_{w \in C(\kappa)^+} \ell(w)$.  We note that when $\kappa+\rho$ is regular, then $C(\kappa)^+=\{w_+\}$ where $w_+$ is as in theorem \ref{thm-classicity-intro} above, and so $\ell_{\min}(\kappa)=\ell_{\max}(\kappa)=\ell(w_+)$ and hence the range is reduced to a single degree.

Then we have the following vanishing theorem for the small slope part of classical cohomology (theorems \ref{thm-12-main} and \ref{thm-strong-vanishing-ss}).

\begin{thm}[{Vanishing for classical cohomology}]\label{thm-1-intro}  For any $\kappa \in X_\star(T^c)^{M_\mu,+}$,
\begin{enumerate}
\item  $\overline{\HH}^i(K, \kappa)^{ss^M(\kappa)}$ is concentrated in the range  $[\ell_{\min}(\kappa), \ell_{\max}(\kappa)]$,
\item ${\HH}^i(K, \kappa, cusp)^{sss^M(\kappa)}$ is concentrated in the range  $[0, \ell_{\max}(\kappa)]$,
\item ${\HH}^i(K, \kappa)^{sss^M(\kappa)}$ is concentrated in the range  $[\ell_{\min}(\kappa),d]$.
\end{enumerate}
\end{thm}

Again we conjecture that strongly small slope can be replaced by small slope in the second and third points, and this would follow from conjecture \ref{conj-strongslopes}.  We also prove a similar result at deeper levels at $p$, and give a more representation theoretic statement (see definition \ref{defi-small-slope-rep}).

\begin{rem} 
This result should be compared to vanishing theorems in \cite{MR3571345} or \cite{blasius-harris-ramak} where there is no small slope condition but instead a regularity condition on the weight $\kappa$.
\end{rem}

Using Faltings's dual BGG spectral sequence, we can also deduce vanishing theorems for de Rham and hence Betti cohomology.  Let $\nu \in X^\star(T^c)^+$ and let $W_\nu$ be the corresponding irreducible representation of $G$ and $W_\nu^\vee$ its contragredient.   We can attach to it a local system $\mathcal{W}^\vee_\nu$ on $S_{K}(\C)$. We have the Betti cohomology groups $\HH^\star(S_K(\C), \mathcal{W}_\nu^\vee)$, $\HH^\star_c(S_K(\C), \mathcal{W}_\nu^\vee)$ and the interior cohomology $\overline{\HH}^\star(S_K(\C), \mathcal{W}_\nu^\vee) = \mathrm{Im}( \HH^\star_c(S_K(\C), \mathcal{W}^\vee_\nu) \rightarrow \HH^\star(S_K(\C), \mathcal{W}^\vee_\nu))$.  We have the following (Theorems \ref{thm-12-main} and \ref{thm-strong-vanishing-ss}): 

\begin{thm}[Vanishing of Betti cohomology]\label{thm-2-intro} For any $\nu \in X_\star(T^c)^{+}$,
\begin{enumerate}
\item $\overline{\HH}^i(K, \mathcal{W}_\nu^\vee)^{ss_b(\nu)}$ is concentrated in the middle degree $d$, 
\item ${\HH}_c^i(K, \mathcal{W}_\nu^\vee)^{sss_b(\nu)}$ is concentrated in the range  $[0, d]$,
\item ${\HH}^i(K, \mathcal{W}_\nu^\vee)^{sss_b(\nu)}$ is concentrated in the range  $[d,2d]$.
\end{enumerate}
\end{thm} 

\begin{rem} The conditions $ss_b(\nu)$ and $sss_b(\nu)$ are the union of all conditions $ss^M(\kappa)$ and $sss^M(\kappa)$ respectively where $\kappa \in X_\star(T^c)^{M,+}$ runs through the set $\{-ww_0(\nu+\rho)-\rho, w \in \WM\}$.
\end{rem}

In \cite{MR3702677} and \cite{caraiani2019generic}, Caraiani and Scholze proved a similar concentration result for the Betti cohomology of  unitary Shimura varieties under a genericity condition for the action of the spherical Hecke algebra at a prime number $\ell$.  Their result is much more powerful because it also applies to the cohomology with  coefficients in a $p$-torsion local system for a prime $p\neq \ell$.  In this section we have seen three kinds of conditions that can be used to kill cohomology outside of the expected range of degrees: small slope at $p$, genericity at some prime $\ell$, and tempered at $\infty$.  We give the simplest example of a cohomology class outside of the expected degree, and explain how all of these conditions fail.

\begin{ex} Consider a compact Shimura curve $S_K$ associated to a quaternion algebra over $\qq$ split at $\infty$ and $p$.  Consider the constant function $1 \in \HH^0(S_K, \oscr_{S_K})$, which comes from the trivial automorphic representation. The degree $0$ cohomology is the ``wrong'' degree in this weight, in the sense that the interesting cohomology of $\oscr_{S_K} = \mathcal{V}_0$ sits in degree $1$.  We observe that the cohomology class $1$ is:
\begin{enumerate}
\item Non tempered at $\infty$, since $\mathrm{Vol}( \mathrm{PGL}_2(\R)) = \infty$.
\item Not small slope at $p$,   because the $U_p$ eigenvalue of the trivial representation is $p$, and $v(p) =1$. The small slope condition in weight $0$ is having $U_p$-eigenvalue of slope $< 1$. 
\item Not generic at any prime $\ell$, because the semi-simple conjugacy class attached to the trivial representation at $\ell$ via the local Langlands correspondence is $\mathrm{diag}( \ell^{\frac{1}{2}}, \ell^{-\frac{1}{2}})$. 
\end{enumerate}
\end{ex}

\subsection{Locally analytic cohomology, interpolation, and the eigenvariety} Our next results concern the $p$-adic interpolation of the overconvergent cohomologies over weight space, and the construction of the eigenvariety.  In the classical work \cite{MR1696469} (revisited in \cite{MR3265287} and \cite{MR3097946}), the eigencurve was constructed by interpolating the Hecke action on the spaces of overconvergent modular forms $\HH^0_{1}(K^p, k)^{+,fs}$ in the weight $k$. This approach of interpolating overconvergent modular forms  was  generalized to Siegel varieties in \cite{MR3275848} (see also \cite{MR3609197}, \cite{MR3563721}, \cite{MR3975499}, \cite{brasca2020padic},  \cite{birkbeck2020overconvergent} for further generalizations). 
In another direction, $p$-adic interpolation of the Betti cohomologies $\HH^i(S_K(\C), \mathcal{W}_\nu^\vee)$ was considered (and in this setting, one  considers general arithmetic quotients of locally symmetric space, not only Shimura varieties). See for example  \cite{ashstevensslopes}, \cite{MR2846490}, \cite{MR3692014}, \cite{MR2207783}. We also mention the case of algebraic modular forms where the group $G(\mathbb{R})$ is compact modulo center, which has been studied intensively, and for which there is no distinction between Betti and coherent cohomology (see for example \cite{MR2392353}, \cite{MR2075765}, \cite{MR2769113}). 

In this paper, we construct eigenvarieties by interpolating the local cohomologies $\HH^i_w(K^p,\kappa)^{\pm,fs}$.  For the case of overconvergent modular forms $\HH^0_{1}(K^p,\kappa)^{+,fs}$, we recover certain of the constructions recalled above. In this case, the improvement is that we are not assuming that the group $G_{\qq_p}$ is unramified or that there is a nonempty ordinary locus.

Let $\Lambda=\ZZ_p[[T^c(\ZZ_p)]]$ be the Iwasawa algebra and let $\mathcal{W} = \Spa(\Lambda,\Lambda) \times_{\Spa(\ZZ_p, \ZZ_p)} \Spa(\qq_p, \ZZ_p)$ be the weight space of continuous characters of $T^c(\ZZ_p)$.  In order to construct families, we must first replace the finite slope, overconvergent cohomologies $\HH^i_w(K^p,\kappa)^{\pm,fs}$, which are defined for algebraic weights $\kappa\in X^\star(T^c)^{+,M_\mu}$, with some finite slope, locally analytic overconvergent cohomologies $\HH^i_{w,an}(K^p,\nu)^{\pm,fs}$, which are defined for any $p$-adic weight $\nu:T^c(\ZZ_p)\rightarrow\mathbb{C}_p^\times$.

To define these cohomologies, we use the same support conditions, but we replace the automorphic vector bundles $\mathcal{V}_\kappa$ with certain Banach sheaves $\mathcal{V}_\nu^{an}$ and $\mathcal{D}_\nu^{an}$.  These Banach sheaves are defined on neighborhoods of $(\pi_{HT,K_p}^{tor})^{-1}(w\cdot K_p)$, using a reduction of the $M_\mu^c$-torsor which is used to define the sheaves $\mathcal{V}_\kappa$.  This reduction is defined using the Hodge-Tate period map (see section \ref{section-reduction}).  Locally, the sheaves $\mathcal{V}_\nu^{an}$ (resp. $\mathcal{D}_\nu^{an}$) are modeled on a locally analytic induction (resp. locally analytic distributions) of a certain profinite subgroup of $M_\mu(F)$ (actually with a fixed ``radius of analyticity'' which we suppress from the notation, as it will not matter when we pass to the finite slope part of cohomology).  We define $\HH^i_{w,an}(K^p,\nu)^{+,fs}$ using $\mathcal{V}_\nu^{an}$ and $\HH^i_{w,an}(K^p,\nu)^{-,fs}$ using $\mathcal{D}_\nu^{an}$.  We must consider both of the sheaves $\mathcal{V}_\nu^{an}$ and $\mathcal{D}_\nu^{an}$ in order to have a duality theory, and we note that here we break the symmetry between the $+$ and $-$ theories.

The finite slope, locally analytic overconvergent cohomologies $\HH^i_{w,an}(K^p,\nu)^{\pm,fs}$ carry a Hecke action of the monoid $T^\pm$ (see section \ref{subsubsec-hecke-action} for the definition), where one should think of the action of $t\in T^{\pm}$ as something like the ``$U_p$-type'' Iwahori Hecke operator $[K_ptK_p]$, except that one does not actually get an action of $\mathcal{H}_p^{\pm}$, as the ``diamond operators'' $t\in T(\ZZ_p)$ do not act trivially: in fact they act by $\nu(t)^{\pm}$.  The finite slope, locally analytic overconvergent cohomologies enjoy similar properties to parts (2),(3),(4) of theorem \ref{thm-overconvergent-introduction} (see section \ref{subsection-analytic-duality} for duality and theorems \ref{thm-one-half-conj} and \ref{theo-vanishing-big-sheaves} for vanishing).

We now explain the relation between the two cohomologies.  Fix $w\in\WM$ and let $\kappa\in X^\star(T^c)^{M_\mu,+}$ be an algebraic weight.  We may associate a $p$-adic weight $\nu:T(\ZZ_p)\to F^\times$, as follows: we first consider $\nu_{alg}=-w^{-1}w_{0,M}(\kappa+\rho)-\rho\in X^\star(T^c)$, take the associated character $T(F)\to F^\times$, and define $\nu$ to be its restriction to $T(\ZZ_p)$.  We emphasize that when $G$ does not split over $\qq_p$, neither $\kappa$ nor $w$ need be defined over $\qq_p$, and the (absolute) Weyl group $W$ need not act on $T^c(\ZZ_p)$ nor on $\mathcal{W}$.

From the construction we have maps of sheaves in a neighborhood of $\pi_{HT,K_p}^{-1}(w\cdot K_p)$
$$\mathcal{V}_{\kappa}\hookrightarrow\mathcal{V}_\nu^{an}\quad\mathrm{and}\quad\mathcal{D}_\nu^{an}\twoheadrightarrow\mathcal{V}_{\kappa^\vee}$$
which are locally modeled on the inclusion of an algebraic induction into a locally analytic induction and its dual.   Passing to cohomology with supports, we obtain maps
$$\HH^i_w(K^p,\kappa)^{+,fs}\rightarrow\HH^i_{w,an}(K^p,\nu)^{+,fs}(-\nu_{alg})$$
$$\HH^i_{w,an}(K^p,\nu)^{-,fs}(\nu_{alg})\to \HH^i_w(K^p,\kappa^\vee)^{-,fs}$$
where we have twisted the Hecke action of $T^\pm$ on the analytic cohomology in order to make these maps Hecke equivariant (the twist of $(-\nu_{alg})$ means that we multiply the action of $t\in T^+$ by $\nu_{alg}(t)^{-1}$.  Note that with this twist, $T(\ZZ_p)$ acts trivially.)

Then we can state our second classicality theorem (see corollary \ref{coro-class-overconvergent-analytic}).

\begin{thm}[Classicality at the level of the sheaf]
The two maps above become isomorphisms after passing to the $+,sss_{M,w}(\kappa)$ part in the $+$ case and the $-,sss_{M,w}(\kappa^\vee)$ part in the $-$ case.  If the Shimura variety is compact, the same statement is true with weaker conditions $+,ss_{M,w}(\kappa)$ and $-,ss_{M,w}(\kappa^\vee)$ (again see section \ref{section-slope-conditions} for the precise definitions of these small slope conditions).
\end{thm}

\begin{rem}
We try to clarify the switch from $\kappa$ to $\nu$, and why it is really $\nu$ and not $\kappa$ that is the natural weight parameter on the eigenvariety.  The situation will become much clearer if we start by fixing a $G$-dominant $\nu_{alg}\in X^\star(T^c)^+$.  Then associated to each $w\in\WM$ we can consider $\kappa_w=-w_{0,M}w(\nu_{alg}+\rho)-\rho$.  Combining the two classicality theorems we have an isomorphism
$$\HH^i_{w,an}(K^p,\nu)^{+,sss_w(\nu_{alg})}=\HH^i(K,\kappa_w)^{+,sss_w(\nu_{alg})}$$
where the small slope condition $+,sss_w(\nu_{alg})$ is simply the combination of $+,sss^M(\kappa_w)$ and $+,sss_{M,w}(\kappa_w)$ (if the Shimura variety is compact, we can use the condition $ss(\nu_{alg})$, which is independent of $w$ and is the usual small slope condition in the algebraic modular form setting).

Now a stable $L$-packet where the archimedean components are discrete series with infinitesimal character $\nu_{alg}+\rho$ should give the same contribution to the cohomologies $\HH^{\ell(w)}(K,\kappa_w)$ for $w\in\WM$.  If the contributions to the $\HH^{\ell(w)}(K,\kappa_w)^{+,sss_w(\nu_{alg})}$ are nonempty, then we will see eigenclasses with the same system of Hecke eigenvalues in all of the $\HH^{\ell(w)}_{w,an}(K^p,\nu)^{+,fs}$.  These eigenclasses should all correspond to the same point of the eigenvariety, and thus we should use $\nu$, which is nothing but a $\rho$ shift of the dominant representative of the infinitesimal character, rather than the $\kappa_w$, which vary with $w\in\WM$.
\end{rem}

We finally turn to results about the eigenvariety $\pi:\mathcal{E}\rightarrow\mathcal{W}$.  The eigenvariety comes with  a map $T(\qq_p) \rightarrow \oscr_{\mathcal{E}}$ (compatible with the map $T(\ZZ_p) \rightarrow \ZZ_p[[T^c(\ZZ_p)]] \stackrel{\pi^\star} \rightarrow \oscr_{\mathcal{E}}$) and a map $\mathcal{H}^S \rightarrow \oscr_{\mathcal{E}}$ where $\mathcal{H}^S$ is the spherical Hecke algebra away from  the finite set of primes $S$ which contains $p$ and all primes $\ell$ where $K_\ell$ is not hyperspecial.  Points of the eigenvariety are pairs $( \lambda_p, \lambda^S)$ where $\lambda_p$ is a character of $T(\qq_p)$,  and $\lambda^S$ is a system of eigenvalues for the prime to $S$ spherical Hecke algebra.  The projection to $\mathcal{W}$ is given by  $\lambda_p \mapsto \nu = \lambda_p \vert_{T^c(\ZZ_p)}$.  When $\nu$ is associated to an algebraic weight $\nu_{alg}\in X^\star(T)$, we let $\lambda_p^{sm} = \lambda_p \nu_{alg}^{-1}$ be the smooth part of $\lambda_p$ (it is trivial on $T(\ZZ_p)$).  We have the following theorem on the existence and properties of the eigenvariety (Theorems \ref{thm-eigenvariety} and \ref{thm-eigenvariety2}): 
\begin{thm}\label{thm-eigenvariety-intro} 
 
 The eigenvariety $\pi:\mathcal{E} \rightarrow \mathcal{W}$ is locally quasi-finite and partially proper. It carries coherent sheaves $$\HH^k_{w,an}(K^p,\nu^{un})_{\mathcal{Z}}^{\pm,fs},\qquad \HH^k_{w,an}(K^p,\nu^{un}, cusp)_{\mathcal{Z}}^{\pm,fs}$$
 for all $w\in\WM$ and $k\in\ZZ$
 and they satisfy the following properties:
 \begin{enumerate} 
\item {\bf (Any classical, finite slope eigenclass gives a point of the eigenvariety)} For any $\kappa\in X^\star(T^c)^{M,+}$ and any system of Hecke eigenvalues $(\lambda^{sm}_p,\lambda^S)$ occurring in $\HH^i(K,\kappa)^{+,fs}$ there is a $w=w_Mw^M\in W$ with $w_M\in W_M$, $w^M\in\WM$, so that if $\nu:T^c(\ZZ_p)\to F^\times$ is associated with $\nu_{alg}=-w^{-1}w_{0,M}(\kappa_{alg}+\rho)-\rho$, then $(\lambda^{sm}_p \nu_{alg},\lambda^S)$ is a point of the eigenvariety $\mathcal{E}$ which lies in the support of $\HH^k_{w^M,an}(K^p,\nu^{un})_{\mathcal{Z}}^{+,fs}$ for some $k$.  The analogous result hold for the cuspidal and $-$ cohomologies.

\item {\bf (Small slope points of the eigenvariety in regular, locally algebraic weights are classical)} Conversely if $\nu_{alg}\in X^\star(T^c)^+$ is a $G$-dominant weight and $(\lambda_p,\lambda^S)$ is a point of $\mathcal{E}$ in weight $\nu$ and in the support of $\HH^k_{w,an}(K^p,\nu^{un})_{\mathcal{Z}}^{+,fs} $ for some $w\in\WM$, and if $\lambda_p^{sm}$ satisfies $+,sss_w(\nu)$ then $(\lambda^{sm}_p,\lambda^S)$ occurs in $\HH^i(K,\kappa)^{+,fs}$ for $\kappa=-w_{0,M}w(\nu_{alg}+\rho)-\rho$ and some $i$.  The analogous result hold for the cuspidal and $-$ cohomologies.

\item {\bf (Serre duality interpolates over the eigenvariety)} We have pairings: 
 $$\HH^k_{w,an}(K^p,\nu^{un})_{\mathcal{Z}}^{\pm,fs} \otimes \HH^{d-k}_{w,an}(K^p,\nu^{un}, cusp)_{\mathcal{Z}}^{\mp,fs} \rightarrow  \pi^{-1} \oscr_{\mathcal{W}}.$$
and these pairings are compatible with Serre duality under the classicality theorem. 

\item{\bf (Interior cohomology classes deform over the weight space)} Any interior cohomology eigenclass $c \in \overline{\HH}^\star(K,  \kappa)^{\pm,fs}$ defines a point on  a component of the eigenvariety of dimension equal to the dimension of the weight space. 
\end{enumerate} 
\end{thm} 

\begin{rem} It is plausible that the eigenvariety $\mathcal{E}$ coincides with an eigenvariety constructed via Betti cohomology interpolation, as in \cite{MR3692014}. One can also believe that there is a $p$-adic Eichler-Shimura theory relating both constructions. See \cite{MR3315057}, as well as some forthcoming work of Juan Esteban Rodriguez. 
\end{rem}

\begin{rem} We can define certain $\pi^{-1} \oscr_{\mathcal{W}}$-torsion free sheaves $\overline{\HH}^{\ell_{\pm}(w)}_{w,an}(K^p, \nu^{un})_{\mathcal{Z}}^{\pm,fs}$ over the eigenvariety. This sheaves interpolate the various modules of the interior Cousin complex (which can be used to compute the interior cohomology). For all $w \in \WM$,  we let $\mathcal{E}_w^!$ be the support of $\overline{\HH}^{\ell_{\pm}(w)}_{w,an}(K^p, \nu^{un})_{\mathcal{Z}}^{\pm,fs}$, which is a union of  irreducible components of the eigenvariety of dimension equal to $\dim~\mathcal{W}$. Any interior cohomology class lifts to a point on  $\cup_{w \in \WM} \mathcal{E}^!_w$. 
\end{rem}

 \begin{rem} For $\mathrm{GSp}_4/L$ with $L$ a totally real field, variants or special cases of this theory are considered in \cite{pilloniHidacomplexes}, \cite{BCGP}, \cite{loeffler2019higher}.
\end{rem}

\begin{rem} The point $(4)$ of the theorem is an advantage of the method we use to construct the eigenvarieties. Such a result was only available in a limited number of cases (e.g. Shimura sets and automorphic forms contributing to  cuspidal coherent $\HH^0$).
\end{rem}

Finally, using point $(4)$ we can give a new construction of Galois representations associated to certain automorphic representations realizing in the coherent cohomology of Shimura varieties, but not in the Betti cohomology. This construction is via $p$-adic interpolation, and yields new results on local-global compatibility at $p$.  In \cite{F-Pilloni}, section 9, we defined a certain class of cuspidal automorphic representations for the group $\mathrm{GL}_n/L$ where $L$ is a totally real or CM number field.  These are called weakly regular, odd, algebraic, essentially (conjugate) self dual, cuspidal automorphic representations. Informally, the weak regularity condition means that we allow the weights of the infinitesimal character  to repeat at most twice. 

  \begin{thm} Let $\pi$ be a  weakly regular, odd, algebraic, essentially (conjugate) self dual, cuspidal automorphic representation of $\mathrm{GL}_n/L$ with $\pi^c = \pi^\vee \otimes \chi$ and infinitesimal character $\lambda = ( ( \lambda_{1, \tau}, \ldots, \lambda_{n, \tau})_{\tau \in \mathrm{Hom}(L,\C)})$ with $\lambda_{1, \tau} \geq \cdots \geq \lambda_{n, \tau}$. 
  Then for each isomorphism $\iota:\overline{\qq}_p\simeq \C$ there is a continuous
  Galois representation $\rho_{\pi, \iota}: G_L \rightarrow \mathrm{GL}_n(\overline{\qq}_p)$
  such that:
  \begin{enumerate}
  \item $\rho_{\pi, \iota}^c \simeq \rho^\vee \otimes \epsilon_p^{1-n} \otimes \chi_{\iota}$ where $\chi_{\iota}$ is the $p$-adic realization of $\chi$.
 \item $\rho_{\pi, \iota}$ is unramified at all finite places $v \nmid p$ for which $\pi_v$ is unramified and one has
$$\iota WD (\rho_{\pi, \iota} \vert_{G_{L_v}})^{F-ss} \simeq \mathrm{rec} (\pi_v \otimes \vert \det \vert_v^{\frac{1-n}{2}}).$$
\item $\rho_{\pi, \iota}$ has generalized $\iota^{-1}\circ\tau$-Hodge--Tate weights $(-\lambda_{n,\tau} + \frac{n-1}{2}, \cdots,
 -\lambda_{1,\tau} + \frac{n-1}{2})$.
 \item Let $v \mid p$ be a place of $L$ and assume that $\pi_v$ is a regular principal series. Then $\rho_{\pi,\iota}\vert_{G_{L_v}}$ is potentially crystalline and 
 $$\iota WD (\rho_{\pi, \iota} \vert_{G_{L_v}})^{F-ss} \simeq \mathrm{rec} (\pi_v \otimes \vert \det \vert_v^{\frac{1-n}{2}}).$$
\end{enumerate}
 \end{thm}
 
Here the hypothesis that $\pi_v$ is a regular principal series in part (4) is a ``$p$-distinguished'' condition, which in particular implies that $\mathrm{rec} (\pi_v \otimes \vert \det \vert_v^{\frac{1-n}{2}})$ is a sum of $n$ distinct characters, so that we can prove that $\rho_{\pi,\iota}|G_{L_v}$ is potentially crystalline by finding its periods one at a time.
 
When $\pi$ is regular rather than just weakly regular and odd, then a stronger form of the theorem holds according to results of Bellaiche, Caraiani, Chenevier, Clozel, Harris, Kottwitz, Labesse, Shin, Taylor, etc. (see \cite{MR3272052}, \cite{BLGGT}) including purity and local-global compatibility at all places.  The above theorem is deduced from these results by $p$-adic interpolation.   Moreover even in the irregular case, the Galois representations in the theorem have already been constructed by different methods in \cite{MR3512528} (and in special cases in \cite{2015arXiv150705922B}, \cite{MR3989256}).  The novelty in the above theorem is the results towards local global compatibility at $p$ in points (3) and (4).

The theorem is proved by first descending $\pi$ to a cuspidal automorphic representation on a quasi-split unitary group over $L$, which then contributes to the interior coherent cohomology of the corresponding Shimura variety.  Then we apply point (4) of theorem \ref{thm-eigenvariety-intro} to get a point on the eigenvariety which is a limit of regular weight classical points, and conclude using results of Kisin \cite{MR1992017} on the interpolation of crystalline periods in analytic families, as in the work of Jorza and Mok \cite{MR3039824}, \cite{MR3200667}.

\begin{rem} As explained, our proof of this theorem  makes use of the results of \cite{MR3338302} to descend certain automorphic representations  to quasi-split unitary groups, and is therefore conditional on the results announced in  \cite{MR3135650}.
\end{rem}

\subsection{A toy model: the cohomology of flag varieties and the Cousin complex of Kempf}

There is a strong analogy between the coherent cohomology of Shimura varieties, and the much simpler coherent cohomology of flag varieties.  In \cite{MR509802}, Kempf introduced a local cohomology method for computing the coherent cohomology of flag varieties, which may be viewed as a ``toy model'' for the methods of this paper.  We briefly review this computation and some of the analogies (see section \ref{subsec-kempf} for more details).

For this section of the introduction, which is independent of the rest of the paper, we let $G$ be a split reductive group over a field $F$ and let $T\subset B\subset G$ be a Borel and maximal torus.  Let $FL = B\backslash G$ be the flag variety for $G$ and write $\pi : G \rightarrow FL$ for the projection. Let $d$ be the dimension of $FL$. Let $W$ be the Weyl group of $G$ with length function $\ell$, and let $w_0\in W$ be the longest element.  For any weight $\kappa \in X^\star(T)$, we associate a $G$-equivariant line bundle $\mathcal{L}_\kappa$ over $FL$, whose sections on an open $U \hookrightarrow FL$ are
$$\mathcal{L}_\kappa(U)=\{f : \pi^{-1}(U) \rightarrow \mathbb{A}^1\mid f(bg)=(w_0\kappa)(b)f(g),\quad\forall b\in B,g\in \pi^{-1}(U)\}.$$
The right action of $G$ on $\mathrm{FL}$ induces a left action on the cohomology groups $\HH^i(FL, \mathcal{L}_\kappa)$.  If $\kappa$ is dominant, then $\HH^0(FL, \mathcal{L}_\kappa)$ contains a highest weight representation of weight $\kappa$.

The classical Borel-Weil-Bott theorem describes the cohomology of the sheaves $\mathcal{L}_\kappa$ over $FL$ when the characteristic of $F$ is $0$: 
\begin{thm}[\cite{MR2015057}, 5.5, corollary]\label{thmBWB} Assume that $\mathrm{char}(F) =0$. Let $\kappa \in X^\star(T)$ then:
\begin{enumerate}
\item If there exists no $w \in W$ such that $w(\kappa+\rho)-\rho$ is  dominant (i.e. $\kappa+\rho$ is irregular) then $\HH^i(FL, \mathcal{L}_\kappa) =0$ for all $i$.
\item If there exists $w \in W$ such that $w(\kappa+\rho)-\rho$ is dominant, then there is a unique such $w$, and $\HH^i(FL, \mathcal{L}_\kappa)=0$ if $\ell(w) \neq i$, while  $\HH^{\ell(w)}(FL, \mathcal{L}_\kappa)$ is the highest weight $w(\kappa+\rho)-\rho$ representation.
\end{enumerate}
\end{thm}

Following section 12 of \cite{MR509802}, one can study the cohomology of the sheaves $\mathcal{L}_\kappa$ over $FL$ with the help of the Bruhat stratification $FL = \cup_{w \in W} B\backslash Bw B$.  Given $w\in W$, let $C_w=B\backslash BwB$, and let $U_w\subseteq \mathrm{FL}$ be an open so that $C_w\subseteq U_w$ is closed (for example one can take $U_w=B\backslash Bw_0Bw_0^{-1}w$).  Then one can introduce the local cohomology groups:
$$\HH_w^i(\kappa):=\HH^i_{C_w}(U_w,\mathcal{L}_\kappa).$$
Unlike the cohomology of the full flag variety, these local cohomologies do not carry an action of $G$.  However they do still carry an action of $B$, as well as of the Lie algebra $\mathfrak{g}$ of $G$.

Then we have the following results:
\begin{itemize}
\item {\bf(Vanishing)} The local cohomologies $H^i_w(\kappa)$ vanish unless $i=d-\ell(w)$.
\item {\bf(Cousin complex)} For each weight $\kappa\in X^\star(T)$ there is a Grothendieck-Cousin complex
$$\HH^0_{w_0}(\kappa)\to\bigoplus_{w\in W,\ell(w)=d-1} \HH^1_{w}(\kappa)\to\cdots\to\bigoplus_{w\in W,\ell(w)=d-i}\HH^i_w(\kappa)\to\cdots\to \HH^d_1(\kappa)$$
computing the cohomology $\HH^\star(FL,\mathcal{L}_\kappa)$.  This complex arises as the $E_1$ page of the spectral sequence of the following filtration of $FL$ associated to the Bruhat stratification:
$$Z^i=\bigcup_{w\in W,\ell(w)\leq d-i} C_w.$$

We also remark that this complex is closely related to the BGG resolution.  For example, if the characteristic of $F$ is zero $\kappa$ is dominant, then this complex is exactly the dual of the BGG resolution of the irreducible representation of highest weight $-w_0\kappa$.
\item {\bf(Weight bounds)}
The analog of the Hecke action of $\mathcal{H}_p^{\pm}$ on the finite slope, overconvergent cohomologies is the action of the torus $T$ on the local cohomologies $\HH^i_w(\kappa)$ (there is no analog of taking the ``finite slope part'', as we already have the action of a group).  Then the analog of the slope bound is the following ``weight bound'': any weight $\lambda\in X^\star(T)$ occurring in $H_w^{d-\ell(w)}$ satisfies $$\lambda\leq w^{-1}w_0(\kappa+\rho)-\rho$$ where we write $\lambda\leq\lambda'$ if $\lambda'-\lambda=\sum_{\alpha\in\Delta}c_\alpha\alpha$ with $c_\alpha\geq 0$ where $\Delta$ denotes the simple roots of $G$ associated to $T\subseteq B$.  Actually this weight bound is an immediate consequence of a much better result, which we lack an analog of in the Shimura variety setting: the character of $T$ acting on $\HH^i_w(\kappa)$ equals that of the Verma module of highest weight $w^{-1}w_0(\kappa+\rho)-\rho$.
\item {\bf(``Classicality'' and vanishing for the ``big weight'' part)} In analogy with passing to the small slope part in the Shimura variety setting, we can try to impose a ``big weight'' condition in order to kill most (ideally all but one) of the terms in the Cousin complex.

We make some combinatorial definitions: let $C(\kappa)=\{w\in W\mid w^{-1}w_0(\kappa+\rho)\in X^\star(T)_{\mathbb{R}}^+\}$ be the subset of $W$ for which the weight bound is as big as possible.  We say that a weight $\lambda\in X^\star(T)^+$ is a big weight (depending on $\kappa$) if $\lambda\not\leq w^{-1}w_0(\kappa+\rho)-\rho$.  The point of this definition is exactly that by the weight bound, $\HH_w^{d-\ell(w)}(\kappa)$ will not contain any big weights unless $w\in C(\kappa)$.  Then from this and the Cousin complex we have the following immediate consequences, which can be viewed as a weak form of the Borel-Weil-Bott theorem (as well as toy models for theorems \ref{thm-classicity-intro}, \ref{thm-1-intro}):
\begin{itemize}
\item $\HH^i(FL,\mathcal{L}_\kappa)^{bw}=0$ unless $i\in [\min_{w\in C(\kappa)} d-\ell(w),\max_{w\in C(\kappa)}d-\ell(w)]$.
\item If $\kappa+\rho$ is regular, so that $C(\kappa)=\{w\}$ consists of a single element, then we have the following ``classicality theorem'':
$$\HH^i_w(\kappa)^{bw}=\HH^i(FL,\mathcal{L}_\kappa)^{bw}$$
and both sides vanish when $i\not=d-\ell(w)$.
\end{itemize}
Although these results are not as strong as the Borel-Weil-Bott theorem, we emphasize that we have not assumed that $\mathrm{char}(F)=0$.  Moreover, in characteristic zero it is not too hard to deduce the full Borel-Weil-Bott theorem using some basic properties of the BGG category $\mathcal{O}$ (see section \ref{subsec-kempf}).

\item {\bf(Interpolation)} We assume $F$ has characteristic 0.  Although the line bundles $\mathcal{L}_\kappa$ and their cohomologies only exist for integral weights $\kappa\in X^\star(T)$, the local cohomologies $\HH^i_w(X,\kappa)$, together with their $\mathfrak{g}$ action can be defined for all $\kappa\in X^\star(T)^+\otimes F=\mathfrak{t}^\vee$, where $\mathfrak{t}$ is the Lie algebra of $T$.  The families of $\mathfrak{g}$-modules obtained in this way are called twisted Verma modules (see \cite{MR1985191}).
\end{itemize}

\bigskip

\noindent
\textbf{Acknowledgements}:  We would like to thank D. Loeffler and S. Zerbes for asking about the possibility to generalize the results of \cite{pilloniHidacomplexes} to more general situations and for many discussions about this paper. A first generalization was found in our joint work with C. Skinner \cite{loeffler2019higher}. We further thank D. Loeffler for his careful reading of a first draft of the manuscript, and his many useful comments which in particular lead us to prove theorem \ref{thm-perfect-pairing} and theorem \ref{theorem-coho-van2}  which where stated as conjectures in the draft.   It is also a pleasure to thank F. Calegari and T. Gee,  some of the ideas were  discovered while working on  \cite{BCGP}.  We also thank F. Andreatta, A. Iovita, J. Rodrigues and J.E. Rodriguez for many interesting exchanges in Lyon.  We thank  A. Caraiani  and the London number theory group for their interest in our work (and their patience), as well as the Montr\'eal number theory group for giving us the opportunity to give lectures on this topic. 
We finally acknowledge the financial support of the ERC-2018-COG-818856-HiCoShiVa.

\section{Cohomological preliminaries}

\subsection{Cohomology with support in a closed subspace} We recall the notion and the basic properties of the cohomology of an abelian sheaf on a topological space,  with support in a closed subspace. A reference for this material is \cite{MR2171939}, chapter I. 
Let $X$ be a topological space. We let $Ab_X$ be the category of abelian sheaves over $X$.  If $X$ is a point, $Ab_X$ is simply $Ab$ the category of abelian groups. We let $\mathcal{D}(Ab_X)$ be the derived category of $Ab_X$.
Let $i : Z \hookrightarrow X$ be a closed subspace.
 For an object $\mathscr{F}$ of $Ab_X$, we let $\Gamma_Z(X, \mathscr{F})$ be  the subgroup of $\HH^0(X, \mathscr{F})$ of sections whose support is included in $Z$.  We let $\mathrm{R}\Gamma_Z(X, -) : \mathcal{D}(Ab_X) \rightarrow \mathcal{D}(Ab)$ be the right derived functor of $\Gamma_Z(X,-)$ (see  \cite{stacks-project}, Tag 079V, in particular for the unbounded version). 
Let $U= X \setminus Z$ and let $\mathscr{F}$ be an object of $\mathcal{D}(Ab_X)$. We have an exact triangle in $\mathcal{D}(Ab)$ (\cite{MR2171939}, I, corollaire 2.9):

$$ \mathrm{R}\Gamma_Z(X, \mathscr{F}) \rightarrow \mathrm{R}\Gamma(X, \mathscr{F}) \rightarrow \mathrm{R}\Gamma(U, \mathscr{F}) \stackrel{+1}\rightarrow $$

 We have the classical pushforward functor $i_\star : Ab_{Z} \rightarrow Ab_{X}$ and it  admits a right adjoint $i^! : Ab_{X} \rightarrow Ab_{Z}$ which can be described as follows: Let $W \subseteq Z$ be an open subset. Let $W' \subseteq X$ be an open subset of $X$ such that $W = W' \cap Z$. For any object $\mathscr{F}$ of $Ab_X$,  we have  $i^! \mathscr{F}(W) = \Gamma_{W} ( W', \mathscr{F}\vert_{W'})$.  It follows that  $\Gamma_Z(X, \mathscr{F}) = \HH^0(X, i_\star i^! \mathscr{F})$. The functor $i^!$ has a right derived functor   $\mathrm{R}i^! : \mathcal{D}(Ab_{X}) \rightarrow \mathcal{D}(Ab_{Z})$. Moreover   $\mathrm{R}\Gamma_Z(X, \mathscr{F}) = \mathrm{R}\Gamma(X, i_\star \mathrm{R} i^! \mathscr{F})$.

Some  properties of the cohomology with support are: 
\begin{enumerate}
\item (corestriction)[\cite{MR2171939}, I, Proposition 1.8] If $Z \subseteq Z'\subseteq X$ are closed subsets, there is a corestriction map $\mathrm{R}\Gamma_Z(X, \mathscr{F}) \rightarrow \mathrm{R}\Gamma_{Z'}(X, \mathscr{F})$. 

\item (pull-back) If we have a continuous map $f:X\to X'$, closed subsets $Z\subset X$, $Z'\subset X'$ satisfying $f^{-1}(Z')\subseteq Z$, and a sheaf $\mathscr{F}$ on $X'$, then there is a map $\mathrm{R}\Gamma_{Z'}(X', \mathscr{F}) \rightarrow  \mathrm{R}\Gamma_{Z}(X, f^{-1} \mathscr{F})$.

\item (Change of ambient space)[\cite{MR2171939}, I, Proposition 2.2]  If we have $Z \subset U \subset X$ for some open $U$ of $X$, then the pull back map   $\mathrm{R}\Gamma_Z(X, \mathscr{F}) \rightarrow \mathrm{R}\Gamma_{Z}(U, \mathscr{F}) $ is a quasi-isomorphism. 
\end{enumerate}

The above properties imply easily the following lemma: 

\begin{lem}\label{lem-direct-sum-complex-support} Let $Z_1, Z_2 \subseteq X$ be two disjoint closed subsets. Then the natural map given by pushforward: 
$$ \mathrm{R}\Gamma_{Z_1} (X, \mathscr{F}) \oplus \mathrm{R}\Gamma_{Z_2} (X, \mathscr{F})   \rightarrow \mathrm{R}\Gamma_{Z_1 \cup Z_2} (X, \mathscr{F}) $$ is a quasi-isomorphism.
\end{lem}
\begin{proof} For $i \in \{1,2\}$, we let $i_i : Z_i \hookrightarrow X$. We finally let $i : Z_1 \cup Z_2 \hookrightarrow X$.  The lemma follows from the claim that for any $\mathscr{F}\in Ob(Ab_X)$, the natural map $ (i_1)_\star i_1^! \mathscr{F} \oplus (i_2)_\star i_2^! \mathscr{F} \rightarrow i_\star i^! \mathscr{F} $ is an isomorphism of sheaves. This is a local computation. Since $X = Z_1^c \cup Z_2^c$ and the claim holds true over any open subset of $Z_1^c$ or $Z_2^c$, we are done. 
\end{proof}

\medskip

We now discuss the construction of the trace map in the context of schemes or adic spaces  and finite locally free morphisms (\cite{MR1734903}, Sect. 1.4.4).

\begin{lem}\label{lem-trace-1}  Consider a commutative  diagram  of  topological spaces:

\begin{eqnarray*}
\xymatrix{ Z \ar[r] \ar[d]& X \ar[d]^f \\
Z' \ar[r] & X'}
\end{eqnarray*}
with $X$ and $X'$  schemes (resp. adic spaces), $f$ a finite locally free morphism of schemes (resp. adic spaces), $Z' \rightarrow X'$ and $Z \rightarrow X$  closed subspaces. 
Let  $\mathscr{F}$ be a sheaf of $\oscr_{X'}$-modules. Then there is a map $\mathrm{R}\Gamma_{Z}(X, f^\star \mathscr{F}) \rightarrow  \mathrm{R}\Gamma_{Z'}(X', \mathscr{F})$.

\end{lem}
\begin{proof} We first recall that the category of sheaves of $\oscr_T$-modules on a ringed space $(T, \oscr_T)$ has enough injectives (\cite{stacks-project}, Tag 01DH). It follows that it is enough to construct a functorial map $\Gamma_{Z}(X, f^\star \mathscr{F}) \rightarrow  \Gamma_{Z'}(X', \mathscr{F}) $  for  sheaves  $\mathscr{F}$ of $\oscr_X$-modules. 
If we let $Z'' = f^{-1} (Z')$, then we have a map $\Gamma_{Z}(X, f^\star \mathscr{F}) \rightarrow  \Gamma_{f^{-1}(Z')}(X, f^\star \mathscr{F})$. Therefore, it suffices to consider the case where $Z = f^{-1}(Z')$.  We have a trace map $\mathrm{Tr} : f_\star \oscr_{X} \rightarrow \oscr_{X'}$. Moreover, the natural morphism  $ f_\star \oscr_{X} \otimes_{\oscr_{X'}} \mathscr{F} \rightarrow f_\star f^\star \mathscr{F} $ is an isomorphism. 
We therefore have a trace map $\mathrm{Tr} : f_\star f^\star \mathscr{F} \rightarrow \mathscr{F}$.  Let us complete the above diagram into:
\begin{eqnarray*}
\xymatrix{ Z \ar[r] \ar[d]& X \ar[d]^f & U \ar[d]^g \ar[l]\\
Z' \ar[r] & X' & U' \ar[l]_{j'}}
\end{eqnarray*}

where $U' = X' \setminus Z'$ and $U = X \setminus Z$.  We have a commutative diagram:

\begin{eqnarray*}
\xymatrix{ f_\star f^\star \mathscr{F} \ar[r] \ar[d]^{\mathrm{Tr}} & j'_\star g_\star g^\star (j')^\star \mathscr{F} \ar[d] \\
\mathscr{F} \ar[r] & j'_\star  (j')^\star \mathscr{F} }
\end{eqnarray*}

We deduce that the trace map $\Gamma(X, f^\star \mathscr{F}) \rightarrow \Gamma(X', \mathscr{F})$  induces a trace map  $\Gamma_Z(X, f^\star \mathscr{F}) \rightarrow \Gamma_{Z'}(X', \mathscr{F})$. 
\end{proof}

\subsection{Cup products} Let $(X, \oscr_X)$ be a ringed space. Let $K = \Gamma(X, \oscr_X)$. 

\begin{prop}\label{prop-construction-cuprod} Let $Z_1, Z_2  \subseteq X$ be two closed subsets. Let $\mathscr{F}$ and $\mathscr{G}$ be two flat sheaves of $\oscr_X$-modules. 
There is a natural map, functorial in $\mathscr{F}$ and $\mathscr{G}$, as well as the support $Z_1$ and $Z_2$:
$$ \mathrm{R}\Gamma_{Z_1} ( X, \mathscr{F}) \otimes_K^L \mathrm{R}\Gamma_{Z_2} ( X, \mathscr{G}) \rightarrow \mathrm{R}\Gamma_{Z_1 \cap Z_2}( X, \mathscr{F}\otimes^L_{\oscr_X} \mathscr{G}).$$ 
\end{prop}
\begin{proof} Let $Z_3 = Z_1 \cap Z_2$. For $1 \leq j \leq 3$,  we let $i_j : Z_i \hookrightarrow X$ be the inclusion. For  $\mathscr{F}$ and $\mathscr{G}$ sheaves of $\oscr_X$-modules we have a map $(i_{1})_\star i_1^! \mathscr{F} \otimes_{\oscr_{X}} (i_{2})_\star i_2^! \mathscr{G}  \rightarrow (i_{3})_\star i_3^! (\mathscr{F} \otimes \mathscr{G})$.
 We claim that there is a map in the derived category: 
 $$ (i_{1})_\star  \mathrm{R}i_1^! \mathscr{F} \otimes^L_{\oscr_{X}} (i_{2})_\star \mathrm{R}i_2^! \mathscr{G}  \rightarrow (i_{3})_\star \mathrm{R}i_3^! (\mathscr{F} \otimes \mathscr{G}).$$
 Indeed, taking the Godement resolution (\cite{stacks-project}, Tag 0FKR)  gives quasi-isomorphisms $\mathscr{F} \rightarrow \mathscr{F}^\bullet$ and $\mathscr{G} \rightarrow \mathscr{G}^\bullet$ where $\mathscr{F}^\bullet$ and $\mathscr{G}^\bullet$ are bounded below complexes of  flasque sheaves of $\oscr_{X}$-modules and for all $x \in X$, the maps $\mathscr{F}_x \rightarrow \mathscr{F}_x^\bullet$ and $\mathscr{G}_x \rightarrow \mathscr{G}_x^\bullet$ are homotopy equivalences in the category of $\oscr_{X,x}$-modules. In particular, $ \mathscr{F}^\bullet$ and $\mathcal{G}^\bullet$ are $K$-flat (by \cite{stacks-project}, Lemma 06YB and the property that $\mathscr{F}_x$ and $\mathscr{G}_x$ are flat $\oscr_{X}$-modules). Since $\mathscr{F}^\bullet$ and $\mathscr{G}^\bullet$ are $K$-flat, the map $\mathscr{F} \otimes \mathscr{G} \rightarrow \mathrm{Tot} \big( \mathscr{F}^\bullet \otimes \mathscr{G}^\bullet\big) $ is a quasi-isomorphism.   We see that $(i_{1})_\star  \mathrm{R}i_1^! \mathscr{F}$ is computed by $(i_{1})_\star i_1^! \mathscr{F}^\bullet$ and $(i_{2})_\star  \mathrm{R}i_2^! \mathscr{G}$ is computed by $(i_{2})_\star i_2^! \mathscr{G}^\bullet$. Taking $K$-flat resolutions $A^\bullet \rightarrow (i_{1})_\star i_1^! \mathscr{F}^\bullet $ and $B^\bullet \rightarrow (i_{2})_\star i_2^! \mathscr{G}^\bullet$, we see that $\mathrm{Tot}( A^\bullet \otimes_{\oscr_X} B^\bullet)$ computes $$ (i_{1})_\star  \mathrm{R}i_1^! \mathscr{F} \otimes^L_{\oscr_{X}} (i_{2})_\star \mathrm{R}i_2^! \mathscr{G}  $$ and there is a composite map  $$\mathrm{Tot}( A^\bullet \otimes_{\oscr_X} B^\bullet) \rightarrow \mathrm{Tot} \big((i_{1})_\star i_1^! \mathscr{F}^\bullet \otimes_{\oscr_X} (i_{2})_\star i_2^! \mathscr{G}^\bullet) $$ $$ \rightarrow (i_3)_\star i_3^! (\mathrm{Tot} \big( \mathscr{F}^\bullet \otimes \mathscr{G}^\bullet\big) ).$$
  Taking a $K$-injective  resolution $\mathrm{Tot}( \mathscr{F}^\bullet \otimes \mathscr{G}^\bullet) \rightarrow C^\bullet$ we finally find that $(i_3)_\star i_3^!(C^\bullet)$ computes $(i_{3})_\star \mathrm{R}i_3^! (\mathscr{F} \otimes \mathscr{G})$ and we have a morphism  $$ (i_3)_\star i_3^! (\mathrm{Tot} \big( \mathscr{F}^\bullet \otimes \mathscr{G}^\bullet\big) ) \rightarrow (i_3)_\star i_3^!(C^\bullet).$$
 There is also a usual cup-product map by \cite{stacks-project}, Tag 0FPJ: 
  $$ \mathrm{R}\Gamma(X, (i_{1})_\star  \mathrm{R}i_1^! \mathscr{F}) \otimes^L_K \mathrm{R}\Gamma(X, (i_{2})_\star  \mathrm{R}i_2^! \mathscr{G}) \rightarrow 
 \mathrm{R}\Gamma(X,  (i_{1})_\star  \mathrm{R}i_1^! \mathscr{F} \otimes^L_{\oscr_{X}} (i_{2})_\star \mathrm{R}i_2^! \mathscr{G}).$$
Combining the two maps gives the map of the proposition. 
\end{proof}

\subsection{The spectral sequence of a filtered topological space}\label{section-spectral-sequence}

 Let $X$ be a topological space, $\mathscr{F}$ a sheaf of abelian groups,  and let $W \subseteq Z$ be two closed subspaces of $X$. We can define $\mathrm{R}\Gamma_{Z/W}(X, \mathscr{F}) = \mathrm{R}\Gamma_{Z\setminus W} (X \setminus W, \mathscr{F})$. If $Z \subseteq Z'$ and $W \subseteq W'$ we have a map  $\mathrm{R}\Gamma_{Z/W}(X, \mathscr{F}) \rightarrow  \mathrm{R}\Gamma_{Z'/W'}(X, \mathscr{F})$. 

If we have $Z_3 \subseteq Z_2 \subseteq Z_1$, then there is an exact triangle (\cite{MR509802}, lemma 7.6): 
$\mathrm{R}\Gamma_{Z_2/Z_3}(X, \mathscr{F}) \rightarrow \mathrm{R}\Gamma_{Z_1/Z_3}(X, \mathscr{F}) \rightarrow \mathrm{R}\Gamma_{Z_1/Z_2}(X, \mathscr{F}) \stackrel{+1}\rightarrow  $

Assume that there is a filtration by closed subsets $X =Z_0 \supseteq Z_1 \supseteq \cdots Z_r = \emptyset$. 
Then there is a spectral sequence of filtered topological space (\cite{Hartshorne}, p. 227):
$$ E_1^{pq} = \HH^{p+q}_{Z_p/Z_{p+1}}(X, \mathscr{F}) \Rightarrow \HH^{p+q}(X, \mathscr{F})$$
which we can visualize as follows: 
\begin{tiny}
\begin{eqnarray*}
\xymatrix{ \vdots & \vdots & \vdots \\
\HH^1_{Z_0/Z_1}(X, \mathscr{F}) & \HH^2_{Z_1/Z_2}(X, \mathscr{F}) & \cdots \\
\HH^0_{Z_0/Z_1}(X, \mathscr{F}) & \HH^1_{Z_1/Z_2}(X, \mathscr{F}) & \cdots \\
& \HH^0_{Z_1/Z_2}(X, \mathscr{F}) & \cdots }
\end{eqnarray*}
\end{tiny}

The differential $d_1^{p,q} : \HH^{p+q}_{Z_p/Z_{p+1}}(X, \mathscr{F}) \rightarrow \HH^{p+q+1}_{Z_{p+1}/Z_{p+2}}(X, \mathscr{F})$ is the boundary map in the long exact sequence associated with the triangle: 
$$\mathrm{R}\Gamma_{Z_{p+1}/Z_{p+2}}(X, \mathcal{F}) \rightarrow \mathrm{R}\Gamma_{Z_p/Z_{p+2}}(X, \mathcal{F}) \rightarrow \mathrm{R}\Gamma_{Z_p/Z_{p+1}}(X, \mathcal{F}) \stackrel{+1}\rightarrow. $$

\subsection{The category of projective Banach modules} In this work  we will consider cohomologies  that will be naturally represented by complexes of Banach modules (or projective, inductive limits of such complexes). We therefore recall the basics of this theory.  Our discussion follows \cite{MR2846490}, section 2. Let $(A,A^+)$ be a complete Tate algebra over a non-archimedean field $(F, \ocal_F)$. We let $\varpi \in \ocal_F$ be a pseudo-uniformizer. 
\subsubsection{Modules} We let $\mathbf{Mod}(A)$ be the abelian category of $A$-modules, $\mathcal{C}(A)$ be the category of complexes of $A$-modules, $\mathcal{K}(A)$ its homotopy category and $\mathcal{D}(A)$ the derived category. We let $\mathcal{K}^{perf}(A)$ be the homotopy category of the category of bounded complexes of finite projective $A$-modules. These are called perfect complex. The category $\mathcal{K}^{perf}(A)$ is a full subcategory of the category $\mathcal{D}(A)$. 

\subsubsection{Banach modules} A Banach $A$-module $M$ is a topological $A$-module  whose topology can be described as follows: Let $A_0$ be an open and bounded subring of $A$. Then $M$ contains an open and bounded sub $A_0$-module $M_0$ which is $\varpi$-adically complete and separated (and $M = M_0[\frac{1}{\varpi}]$).  We let $\mathbf{Ban}(A)$ be the category whose objects are  Banach $A$-modules, and whose morphisms are continuous $A$-linear morphisms.  This is a quasi-abelian category in the sense of \cite{MR1779315}, def. 1.1.3. This means that $\mathbf{Ban}(A)$ is an additive category with kernels and cokernels,  strict epimorphisms are stable under pullback and strict monomorphisms are stable under pushout. Recall that  strict morphisms $f$ are those for which the natural morphism $\mathrm{coim}(f) \rightarrow \mathrm{im}(f)$ is an isomorphism in $\mathbf{Ban}(A)$.  
\subsubsection{The derived category}
We let $\mathcal{C}(\mathbf{Ban}(A))$ be the category of complexes of Banach modules, and $\mathcal{K}(\mathbf{Ban}(A))$ be its homotopy category. A complex $M^\bullet$ in $\mathcal{C}(\mathbf{Ban}(A))$ is strictly exact if each differential $d_i : M_i \rightarrow M_{i+1}$ is strict and the morphism $\mathrm{Im}(d_i) \rightarrow \mathrm{ker}(d_{i+1})$ is an isomorphism.  A morphism $f : M^\bullet \rightarrow N^\bullet$ in $\mathcal{C}(\mathbf{Ban}(A))$ is a strict quasi-isomorphism if its cone is strictly exact.  It is a general property of quasi-abelian categories that if  a morphism  is strictly exact, so is any other morphism in the same homotopy class. One can consider the derived category $\mathcal{D}(\mathbf{Ban}(A))$ obtained by inverting in $\mathcal{K}(\mathbf{Ban}(A))$ the strict quasi-isomorphisms  (\cite{MR1779315}, def. 1.2.16, \cite{MR2846490}, sect. 2.1.3).    In the case of the category $\mathbf{Ban}(A)$ we have the following lemma which helps recognize  strict quasi-isomorphisms:
\begin{lem} Let $M^\bullet$ be a complex in $\mathcal{C}(\mathbf{Ban}(A))$. Then $M^\bullet$ is exact as a complex of $A$-modules  if and only if it is strictly exact.  A morphism $f : M^\bullet \rightarrow N^\bullet$ in $\mathcal{C}(\mathbf{Ban}(A))$ is a quasi-isomorphism of complexes in $\mathbf{Mod}(A)$ if and only if it is a strict quasi-isomorphism.
\end{lem}
\begin{proof} It follows from the open mapping theorem that any epimorphism in $\mathbf{Ban}(A)$ is open and therefore strict. The first claim follows. The second claim is deduced by considering the cone of $f$.
\end{proof}
Because of this lemma, strict quasi-isomorphisms will simply be called quasi-isomorphisms. 

\subsubsection{Projectives} Let $I$ be a set. Denote by $A(I)$ the submodule of $A^I$ of sequences of  elements of $A$ indexed by $I$ converging to $0$ according to the filter in $I$ of the complement of the finite subsets of $I$.  This module can also be described as follows. Let $A_0$ be an open and bounded subring of $A$. Let $A_0(I)$ be the $\varpi$-adic completion of the free $A_0$-module with basis $I$. Then $A(I) = A_0(I)[\frac{1}{\varpi}]$. We see that $A(I)$ is a Banach $A$-module. A Banach $A$-module $M$ is orthonormalizable if there exists a set $I$ and an isomorphism $M \simeq A(I)$. 
 A Banach $A$-module  is called projective if it is a direct factor of an orthonormalizable Banach  $A$-module.  We let $\mathcal{C}^{proj}(\mathbf{Ban}(A))$ be the category of bounded complexes of projective Banach $A$-modules and  $\mathcal{K}^{proj}(\mathbf{Ban}(A))$ be the homotopy category.  There is a natural functor $\mathcal{K}^{proj}(\mathbf{Ban}(A)) \rightarrow \mathcal{D}(\mathbf{Ban}(A))$ and it is fully faithful (\cite{MR1779315}, prop. 1.3.22, \cite{MR2846490}, lem. 2.1.8).   In particular we deduce the following lemma which will be used in this work:
 \begin{lem}\label{lem-invert-quasi-isom} Let $f : M^\bullet \rightarrow N^\bullet$ be a   quasi-isomorphism in $\mathcal{C}^{proj}(\mathbf{Ban}(A))$. Then $f$ admits and inverse $g : N^\bullet \rightarrow M^\bullet$ up to homotopy. 
 \end{lem}

 Recall that finite type $A$-modules are canonically   Banach $A$-modules. Therefore we can view $\mathcal{K}^{perf}(A)$  as a full subcategory  of $\mathcal{K}^{proj}(\mathbf{Ban}(A))$.

\subsubsection{Limits of Banach spaces}
We let $\mathrm{Pro}_{\mathbb{N}}(\mathbf{Ban}(A))$ be the category of countable projective systems of $\mathbf{Ban}(A)$.  Its objects are contravariant functors $\mathbb{N} \rightarrow \mathbf{Ban}(A)$.  An object of this category will be denoted  $``\lim_{i \in \mathbb{N}}"K_i$ with $K_i \in Ob(\mathbf{Ban}(A))$.
We let $\mathrm{Ind}_{\mathbb{N}}(\mathbf{Ban}(A))$ be the category of countable inductive systems of $\mathbf{Ban}(A)$. Its objects are covariant functors $\mathbb{N} \rightarrow \mathbf{Ban}(A)$.  An object of this category will be denoted  $``\colim_{i \in \mathbb{N}}"K_i$ with $K_i \in Ob(\mathbf{Ban}(A))$.
We have  a limit functor $\lim : \mathrm{Pro}_{\mathbb{N}}(\mathbf{Ban}(A)) \rightarrow  \mathbf{Mod}(A)$, and a colimit functor $\colim : \mathrm{Ind}_{\mathbb{N}}(\mathbf{Ban}(A)) \rightarrow  \mathbf{Mod}(A)$.  Of course, these functors completely forget the topology. We could consider instead limit and colimit functors in the category of locally convex $A$-modules instead, but this turns out to be unnecessary to us. 
We extend $\lim$ and $\colim$ to functors between the category of complexes:
$\lim : \mathcal{C}(\mathrm{Pro}_{\mathbb{N}}(\mathbf{Ban}(A))) \rightarrow  \mathcal{C}(A)$ and $\colim :  \mathcal{C}(\mathrm{Ind}_{\mathbb{N}}(\mathbf{Ban}(A))) \rightarrow  \mathcal{C}(A)$.

We let $\mathrm{Pro}_{\mathbb{N}} (\mathcal{K}^{proj}(\mathbf{Ban}(A)))$ be the category  whose objects are  projective systems of complexes  $\{K_i \in Ob(\mathcal{K}^{proj}(\mathbf{Ban}(A)))\}_{i \in \N}$ and the $K_i$'s have  non-zero cohomology in a uniformly bounded range of degrees. We denote an object of this category by $``\lim_{i \in \mathbb{N}}"K_i$  (we are not evaluating the limit).
 Let $``\lim_{i \in \mathbb{N}}"K_i \in Ob(\mathrm{Pro}_{\mathbb{N}} (\mathcal{K}^{proj}(\mathbf{Ban}(A))))$. Then there exists an object (unique up to a non unique isomorphism) $\lim_{i \in \mathbb{N}}K_i \in \mathcal{D}(A)$ by \cite{stacks-project}, Tag 08TB.  And similarly, for any morphism $``\lim_{i \in \mathbb{N}}"K_i  \rightarrow ``\lim_{i \in \mathbb{N}}"K'_i \in \mathrm{Pro}_{\mathbb{N}} (\mathcal{K}^{proj}(\mathbf{Ban}(A)))$, there is  a morphism $\lim K_i \rightarrow \lim K'_i$. 
 If  $``\lim_{i \in \mathbb{N}}"K_i \in Ob(\mathrm{Pro}_{\mathbb{N}} (\mathcal{K}^{proj}(\mathbf{Ban}(A))))$ can be represented by 
an object in $\mathcal{C}(\mathrm{Pro}_{\mathbb{N}}(\mathbf{Ban}(A)))$, then $\lim K_i$ is obtained by taking the degreewise projective limit.

We also let $\mathrm{Ind}_{\mathbb{N}} (\mathcal{K}^{proj}(\mathbf{Ban}(A)))$ be the category  whose objects are  inductive systems of complexes  $\{K_i \in Ob(\mathcal{K}^{proj}(A))\}_{i \in \N}$ and the $K_i$'s have  non-zero cohomology in a uniformly bounded range of degrees. We denote an object of this category by $``\colim_{i \in \mathbb{N}}"K_i$ (we are not evaluating the limit either).

Let $``\colim_{i \in \mathbb{N}}"K_i \in Ob(\mathrm{Ind}_{\mathbb{N}} (\mathcal{K}^{proj}(\mathbf{Ban}(A))))$. Then there exists an object (unique up to a non unique isomorphism) $\colim_{i \in \mathbb{N}}K_i \in \mathcal{D}(A)$ by \cite{stacks-project}, Tag 0A5K.  And similarly, for any morphism $``\colim_{i \in \mathbb{N}}"K_i  \rightarrow ``\colim_{i \in \mathbb{N}}"K'_i \in \mathrm{Ind}_{\mathbb{N}} (\mathcal{K}^{proj}(\mathbf{Ban}(A)))$,  there is  a morphism $\colim K_i \rightarrow \colim K'_i$. 
If  $``\colim_{i \in \mathbb{N}}"K_i \in Ob(\mathrm{Ind}_{\mathbb{N}} (\mathcal{K}^{proj}(\mathbf{Ban}(A))))$ can be represented by 
an object in $\mathcal{C}(\mathrm{Ind}_{\mathbb{N}}(\mathbf{Ban}(A)))$, then $\colim K_i$ is obtained by taking the degreewise inductive limit. 

\subsubsection{Compact operators} Recall that a  morphism $T : M \rightarrow N$  in $\mathbf{Ban}(A)$ is called compact if it is a limit of finite rank operators (for the supremum norm of operators, or equivalently the strong topology on $\mathrm{Hom}_{\mathbf{Ban}(A)}(M,N)$).  
A morphism $T :  M^\bullet \rightarrow N^\bullet$  in $\mathcal{C}(\mathbf{Ban}(A))$ is called compact if it is compact in each degree. 

\begin{defi} A morphism $T :  M^\bullet \rightarrow N^\bullet$  in $\mathcal{K}^{proj}(\mathbf{Ban}(A))$ is compact if it has a compact representative in $\mathcal{C}^{proj}(\mathbf{Ban}(A))$. 
\end{defi}

We need to extend these definitions  to the case of objects in $\mathrm{Pro}_{\mathbb{N}}(\mathcal{K}^{proj}(\mathbf{Ban}(A)))$ or $\mathrm{Ind}_{\mathbb{N}}(\mathcal{K}^{proj}(\mathbf{Ban}(A)))$. Let $T : ``\lim_i"M_i^\bullet \rightarrow  ``\lim_i" N_i^\bullet$ be a morphism in  $\mathrm{Pro}_{\mathbb{N}}(\mathcal{K}^{proj}(\mathbf{Ban}(A)))$. 
We say that $T$ is compact if  there exists a compact operator  $T' :  M^\bullet \rightarrow  N^\bullet \in \mathcal{K}^{proj}(\mathbf{Ban}(A))$, and a commutative diagram: 
\begin{eqnarray*}
\xymatrix{ M^\bullet \ar[r]^{T'} & N^\bullet \ar[d] \\
``\lim_i"M_i^\bullet \ar[u]\ar[r]^{T} & ``\lim_i" N_i^\bullet}
\end{eqnarray*}

\begin{lem}\label{lem-potent-compact-factor} Let $``\lim_i"M_i^\bullet \in  Ob(\mathrm{Pro}_{\mathbb{N}}(\mathcal{K}^{proj}(\mathbf{Ban}(A))))$ and let $T$ be a  compact endomorphism of $``\lim_i"M_i^\bullet$.  Then $T$ induces canonically a compact endomorphism $T_i$ of $M^\bullet_i$ for $i$ large enough and there are factorization diagrams: 
\begin{eqnarray*}
\xymatrix{ M^\bullet_{i+1} \ar[d]\ar[r]^{T_{i+1}} & M_{i+1}^\bullet \ar[d] \\
M_i^\bullet \ar[ru]\ar[r]^{T_i} & M_i^\bullet}
\end{eqnarray*}
\end{lem}
\begin{proof} By definition, the map $``\lim_i"M_i^\bullet \rightarrow M^\bullet$ factors into $``\lim_i"M_i^\bullet \rightarrow M_i ^\bullet\rightarrow M^\bullet$ for some $i$ large enough. 
The map $N^\bullet \rightarrow ``\lim_i"M_i^\bullet$ is given by a collection of compatible maps $N^\bullet \rightarrow M_i^\bullet$.  The lemma follows.
\end{proof}

Let $T : ``\colim_i"M_i^\bullet \rightarrow  ``\colim_i" N_i^\bullet$ be a morphism in  $\mathrm{Ind}_{\mathbb{N}}(\mathcal{K}^{proj}(\mathbf{Ban}(A)))$. 
We say that $T$ is compact if  there exists a compact operator  $T' :  M^\bullet \rightarrow  N^\bullet \in \mathcal{K}^{proj}(\mathbf{Ban}(A))$, and a commutative diagram: 
\begin{eqnarray*}
\xymatrix{ M^\bullet \ar[r]^{T'} & N^\bullet \ar[d] \\
``\colim_i"M_i^\bullet \ar[u]\ar[r]^{T} & ``\colim_i" N_i^\bullet}
\end{eqnarray*}

\begin{lem}\label{lem-potent-compact-factor2} Let $``\colim_i"M_i^\bullet \in  Ob(\mathrm{Ind}_{\mathbb{N}}(\mathcal{K}^{proj}(\mathbf{Ban}(A))))$ and let $T$ be a  compact endomorphism of $``\colim_i"M_i^\bullet$.  Then $T$ induces canonically a compact endomorphism $T_i$ of $M^\bullet_i$ for $i$ large enough and there are factorization diagrams: 
\begin{eqnarray*}
\xymatrix{ M^\bullet_{i+1} \ar[rd]\ar[r]^{T_{i+1}} & M_{i+1}^\bullet  \\
M_i^\bullet \ar[u] \ar[r]^{T_i} & M_i^\bullet \ar[u]}
\end{eqnarray*}
\end{lem}
\begin{proof} By definition, the map $N^\bullet \rightarrow ``\colim_i"M_i^\bullet$ factors into $N^\bullet \rightarrow M_i^\bullet \rightarrow  ``\colim_i"M_i^\bullet $ for some $i$ large enough. 
The map $ ``\colim_i"M_i^\bullet \rightarrow M^\bullet$ is given by a collection of compatible maps $ M_i^\bullet \rightarrow M^\bullet$.  The lemma follows.
\end{proof}

\begin{defi}\label{defi-notop-compact} Let  $T : N^\bullet \rightarrow M^\bullet$ be a map in $\mathcal{D}(A)$.  We say that $T$ is compact if it can be represented by a compact map  in $\mathrm{Pro}_{\mathbb{N}}(\mathcal{K}^{proj}(\mathbf{Ban}(A)))$ or  $\mathrm{Ind}_{\mathbb{N}}(\mathcal{K}^{proj}(\mathbf{Ban}(A)))$. 
\end{defi}

\begin{defi} Let $M^\bullet \in  Ob(\mathcal{D}(A))$ and let $T \in \mathrm{End}_{\mathcal{D}(A)}(M^\bullet)$. We say that  $T$ is potent compact if  for some $n \geq 0$,  ${T}^n$ is compact.  
\end{defi}

\subsection{The cohomology of Banach sheaves} 
In this section we explain how we can obtain complexes of Banach modules from the cohomology of rigid analytic varieties. We use the theory of adic spaces  described in  \cite{MR1734903} and \cite{MR1306024} for example. 

\subsubsection{Sheaves of Banach modules over adic spaces} In this section we recall some material from \cite{MR3275848}, appendix A. Let $F$ be a non archimedean field with ring of integers $\ocal_F$. Let $\varpi \in F$ be a topologically nilpotent unit. 
 
 We  recall that there is a good theory of coherent sheaves on finite type adic spaces  over $\Spa (F, \ocal_F)$: If $\mathcal{X} =\Spa (A,A^+)$ is affinoid and $\mathscr{F}$ is a coherent sheaf on $\mathcal{X}$, then $\HH^i(\mathcal{X},\oscr_{\mathcal{X}}) =0$ for $i \neq 0$, $M = \HH^0(\mathcal{X}, \oscr_{\mathcal{X}})$ is an $A$-module of finite type and the canonical map ${M} \otimes_{A} \oscr_{\mathcal{X}} \rightarrow \mathscr{F}$ is an isomorphism (\cite{MR1306024}, thm. 2.5). Moreover,  $M$ is canonically a Banach $A$-module (\cite{MR1306024}, lem. 2.4). 
 It follows that a coherent sheaf $\mathscr{F}$ over   a     finite type  adic space $\mathcal{X}$  is a sheaf of topological $\oscr_{\mathcal{X}}$-modules. In this paper we will have to manipulate topological sheaves which are not coherent.

\begin{defi}\label{defi-banach-sheaf} Let $\mathcal{X}$ be a     finite type  adic space  over $\Spa (F, \ocal_F)$. A sheaf $\mathscr{F}$ of topological $\oscr_{\mathcal{X}}$-modules is called a Banach sheaf if:
 \begin{enumerate}
 \item For any quasi-compact open $\mathcal{U} \hookrightarrow \mathcal{X}$, $\mathscr{F}(\mathcal{U})$ is a Banach $\oscr_{\mathcal{X}}(\mathcal{U})$-module,
 \item There is an affinoid covering $\mathcal{X} = \cup_i \mathcal{U}_i$, such that  for any affinoid $\mathcal{V} \subset \mathcal{U}_i$, the continuous restriction map $\mathscr{F}(\mathcal{U}_i) \rightarrow \mathscr{F}(\mathcal{V})$ induces a topological isomorphism: 
 $\oscr_{\mathcal{X}}(\mathcal{V}) \hat{\otimes}_{\oscr_{\mathcal{X}}(\mathcal{U}_i)} \mathscr{F}(\mathcal{U}_i) \rightarrow \mathscr{F}(\mathcal{V})$,
 \end{enumerate}
 
 A Banach sheaf $\mathscr{F}$ is called locally projective if there is a covering as in $(2)$ such that $ \mathscr{F}(\mathcal{U}_i)$ is a projective Banach $\oscr_{\mathcal{X}}(\mathcal{U}_i)$-module. 
 \end{defi}

 Any coherent sheaf on $\mathcal{X}$ is therefore a Banach sheaf and a coherent sheaf is a locally projective Banach sheaf if and only if it is locally free.
 Banach sheaves over $\mathcal{X}$ form a full subcategory  of the category of topological $\oscr_{\mathcal{X}}$-modules. We introduce compact morphisms in this context. 
 \begin{defi}\label{defi-compact-map}  Let $\mathcal{X}$ be an adic space of finite type over $\Spa(F, \ocal_F)$.  Let $\mathscr{F}$ and $\mathscr{G}$ be two locally projective Banach sheaves. Let $\phi : \mathscr{F} \rightarrow \mathscr{G}$ be a continuous morphism of $\oscr_{\mathcal{X}}$-modules. We say that the map $\phi$ is compact if there is a covering $\mathcal{X} = \cup_i \mathcal{U}_i$ satisfying the point $(2)$ of definition \ref{defi-banach-sheaf} for both $\mathscr{F}$ and $\mathscr{G}$, such that the map $\phi : \mathscr{F}(\mathcal{U}_i) \rightarrow \mathscr{G}(\mathcal{U}_i)$ is a compact map of $\oscr_{\mathcal{X}}(\mathcal{U}_i)$-modules.
\end{defi}

\begin{rem} Note that if $\mathscr{G}$ is coherent, any morphism  to $\mathscr{G}$ is compact. \end{rem}
 
 \subsubsection{Cohomological properties of Banach sheaves}\label{section-Cohomological properties of Banach sheaves}
  We warn the reader that Banach sheaves which are not coherent sheaves are pathological in general. In particular it is not true that on an affinoid $\mathcal{X}$, a Banach sheaf has trivial higher cohomology groups, nor is it true that a Banach sheaf is the sheaf associated with its global sections.  Here is nevertheless a simple situation where this holds. Let $\mathcal{X} = \Spa(A,A^+)$ be an affinoid and let $M$ be a projective Banach $A$-module. Let $\mathscr{F} = M \hat{\otimes}_A \oscr_{\mathcal{X}}$ be the pre-sheaf whose value on an affinoid  open $\mathcal{U} = \Spa(B,B^+)$ of $\mathcal{X}$ is $M \hat{\otimes}_A B$.  
 
 \begin{lem} The pre-sheaf $\mathscr{F}$ is a sheaf and $\HH^i(\mathcal{X}, \mathscr{F})=0$ for all $i >0$.
 \end{lem}
 \begin{proof} We reduce to the case that $M$ is orthonormalizable, and everything follows from the known properties of $\oscr_{\mathcal{X}}$.
 \end{proof}
 
   We now introduce a certain class of Banach sheaves that have better cohomological properties, following \cite{MR3275848}, appendix A. These are Banach sheaves admitting  formal models which can be controlled in a certain sense. We thus begin by discussing formal Banach sheaves over formal schemes.

\begin{defi} Let $\mathfrak{X} \rightarrow \Spf~\ocal_F$ be a finite type formal scheme over $\Spf(\ocal_F)$. A sheaf $\mathfrak{F}$ of $\oscr_{\mathfrak{X}}$-modules is called a formal Banach sheaf if $\mathfrak{F}$ is flat as an $\ocal_F$-module, $\mathscr{F}_n := \mathfrak{F}/\varpi^n$ is a quasi-coherent sheaf, and   $\mathfrak{F} = \lim_n \mathscr{F}_n$. 
\end{defi}
 
 A formal Banach sheaf is called \emph{flat} if $\mathscr{F}_n$ is a flat $\oscr_{\mathfrak{X}}/\varpi^n$-module for all $n$. It is called \emph{locally projective} if $\mathscr{F}_n$ is a locally projective $\oscr_{\mathfrak{X}}/\varpi^n$-module for all $n$. A formal Banach sheaf is called \emph{small} if there exists a coherent sheaf $\mathscr{G}$ over $\mathfrak{X}$ with the property that $\mathscr{F}_1$ is the inductive limit of coherent  sub sheaves $\mathscr{F}_1 = \mathrm{colim}_{j \in \mathbb{Z}_{\geq 0}} \mathscr{F}_{1,j}$ and  $\mathscr{F}_{1,j}/\mathscr{F}_{1,j-1}$   is a direct summand of $\mathscr{G}$ for all $j \geq 0$. 
 
 The relevance of the smallness assumption is given by the following theorem:
 
 \begin{thm}\label{thm-small-vanish} Let $\mathfrak{X} \rightarrow \Spf~\ocal_F$ be a finite type formal scheme  and let $\mathfrak{F}$ be a small formal Banach sheaf. Assume that $\mathfrak{X}$ has an ample invertible sheaf, and that the generic fiber $\mathcal{X}$ of $\mathfrak{X}$ is affinoid. Then $\HH^i(\mathfrak{X}, \mathfrak{F}) \otimes_{\ocal_F} F = 0 $ for all $i >0$. 
 \end{thm}
 
 \begin{proof} This is   \cite{MR3275848}, theorem A.1.2.2. In the reference, the formal scheme $\mathfrak{X}$ is assumed to be normal and quasi-projective, but the only property needed in the proof is the existence of an ample sheaf on $\mathfrak{X}$.
 \end{proof}
 
Let $\mathfrak{X} \rightarrow \Spf~\ocal_F$ be a finite type formal scheme,  and let $\mathcal{X} \rightarrow \Spa (F, \ocal_F)$ be the generic fiber  of $\mathfrak{X}$. Thus $(\mathcal{X}, \oscr_{\mathcal{X}}^+) = \lim_{\mathfrak{X}'} (\mathfrak{X}',  \oscr_{\mathfrak{X}'})$ where the limit runs over all admissible blow-ups of $\mathfrak{X}$.  Let $\mathfrak{F}$ be a Banach sheaf over $\mathfrak{X}$. For any admissible blow-up $ f : \mathfrak{X}' \rightarrow \mathfrak{X}$, we let $\mathfrak{F}_{\mathfrak{X}'} =  \lim_{n}  f^\star \mathscr{F}_n$. We let $\mathscr{F} = \lim_{\mathfrak{X}'} \mathfrak{F}_{\mathfrak{X}'} [1/\varpi]$. This is a sheaf over $\mathcal{X}$ that we call the generic fiber of $\mathfrak{F}$.
 
 \begin{thm}\label{thm-formal-B-sheaves} Let $\mathfrak{X} \rightarrow \Spf~\ocal_F$ be a finite type formal scheme, with generic fiber $\mathcal{X}$.  We have the following properties:
 \begin{enumerate} 
 \item There is a  ``generic fiber'' functor   going from the category of flat formal Banach sheaves over $\mathfrak{X}$ to the category of Banach sheaves over $\mathcal{X}$, described by the procedure $\mathfrak{F} \mapsto \mathscr{F}$.
 \item  If $\mathcal{U} \hookrightarrow \mathcal{X}$ is a quasi-compact open subset and $\mathfrak{U}' \hookrightarrow \mathfrak{X}'$ is a formal model for the map $\mathcal{U} \hookrightarrow \mathcal{X}$, $\mathscr{F}(\mathcal{U}) = \mathfrak{F}_{\mathfrak{X}'}(\mathfrak{U}')[1/\varpi]$. 
 \item The property $(2)$ of definition \ref{defi-banach-sheaf} holds over the generic fiber of any affine covering of $\mathfrak{X}$. 
 \item The generic fiber functor sends locally projective formal Banach sheaves to locally projective Banach sheaves. 
 \item Let $\mathscr{F}$ be a Banach sheaf arising from a flat small formal Banach sheaf.  Then for any affinoid  open subset $\mathcal{U} \hookrightarrow \mathcal{X}$ we have 
 $\HH^i(\mathcal{U}, \mathscr{F}) = 0$ for all $i >0$.  
 \end{enumerate}
 \end{thm}
 \begin{proof} The first three points are \cite{MR3275848}, proposition A.2.2.3.  For the fourth point, let $\mathfrak{U}$ be an open affine of $\mathfrak{X}$. Let $M_n = \HH^0(\mathfrak{U}, \mathscr{F}_n)$, $A_n = \HH^0(\mathfrak{U}, \oscr_{\mathfrak{X}}/\varpi^n)$, $M = \lim_n M_n$, and $A = \lim_n A_n$. We claim that $M$ is a direct factor of the completion of a free $A$-module.  Let us pick a surjection $A_1^I \rightarrow M_1$ and for any $n$, we can lift it successively to surjections $A_n^I \rightarrow M_n$. We need to prove that we can find a compatible system of sections $s_n : M_n \rightarrow A_n^I$. It suffices to show that the map $\mathrm{Hom}_A( M_n , A_n^{I}) \rightarrow \mathrm{Hom}_A( M_{n-1}, A_{n-1}^{I})$ is surjective. This follows from the short exact sequence  $0 \rightarrow \mathrm{Hom}_{A_n} (M_n, A^I_{1})  \rightarrow \mathrm{Hom}_{A_n} (M_n, A^I_{n})  \rightarrow \mathrm{Hom}_{A_n} (M_n, A^I_{n-1}) \rightarrow 0$.
 
 For the last point, let $\mathfrak{X}$ be a formal model of $\mathcal{X}$ and $\mathfrak{F}$ be a small flat formal Banach sheaf over $\mathfrak{X}$.  Let $\mathfrak{U}$ be an affine formal model of $\mathcal{U}$.    Let $\cup_i \mathcal{U}_i = \mathcal{U}$ be a finite affinoid cover of $\mathcal{U}$.   Let $\mathfrak{U}' $ be an admissible blow-up of $\mathfrak{U}$ with the property that $\cup_i \mathcal{U}_i = \mathcal{U}$ is the generic fiber of a covering of $\mathfrak{U}'$ and there is a map $\mathfrak{U}' \rightarrow \mathfrak{X}$ inducing the map $\mathcal{U} \rightarrow \mathcal{X}$. Note that $\mathfrak{U}'$ has an ample invertible sheaf, since it is a blow-up of an affine formal scheme.  We can apply  theorem \ref{thm-small-vanish} to $\mathfrak{F}_{\mathfrak{U}'}$, the pull-back to $\mathfrak{U}'$ of $\mathfrak{F}$, which is still small by flatness. This shows that the \v{C}ech cohomology of $\mathcal{U}$ with respect to the  covering $\cup_i \mathcal{U}_i$ vanishes. Since this holds for any finite cover, and $\mathcal{U}$ is quasi-compact, we deduce that $\HH^i(\mathcal{U}, \mathscr{F}) = 0$. 
  \end{proof}

\begin{defi} Let $\mathcal{X}$ be a     finite type  adic space  over $\Spa (F, \ocal_F)$. A Banach sheaf $\mathscr{F}$ is called a small, locally projective Banach sheaf if  it arises as the generic fiber of a small, locally projective formal Banach sheaf.
\end{defi}

\begin{rem} A locally free coherent sheaf over $\mathcal{X}$ is a small, locally projective Banach sheaf by the flattening techniques of \cite{MR308104}.
\end{rem}

\begin{rem} We don't know if, for $\mathcal{X} = \Spa(A,A^+)$ affinoid and $\mathscr{F}$ a small, locally projective Banach sheaf on $\mathcal{X}$, it is true that $\mathscr{F}(\mathcal{X})$ is a projective Banach $A$-module and the map $\mathscr{F}(\mathcal{X}) \hat{\otimes}_A \oscr_{\mathcal{X}} \rightarrow \mathscr{F}$ is an isomorphism. \end{rem} 

\subsubsection{Acyclicity of quasi-Stein spaces} In our arguments, it will often be useful to consider not only affinoid covers of adic spaces, but also some quasi-Stein covers. 

\begin{defi}[\cite{MR210949}, def. 2.3]  We say that an adic space $\mathcal{X} \rightarrow \Spa (F, \ocal_F)$ is quasi-Stein if $\mathcal{X} = \cup_{i \in \ZZ_{\geq 0}} \mathcal{X}_i$  is a countable increasing union  of finite type affinoid adic spaces $\mathcal{X}_i \rightarrow \Spa (F, \ocal_F)$ and $\oscr_{\mathcal{X}_{i+1}} \rightarrow \oscr_{\mathcal{X}_i}$ has dense image.  
\end{defi}

\begin{ex} Here are some examples of quasi-Stein adic spaces: 
\begin{itemize}
\item An affinoid space, like the unit ball $\mathbb{B}(0, 1)$. 
\item A Stein space like the open unit ball:  ${\mathbb{B}}^o(0,1) =  \cup_{n } \mathbb{B}(0, \vert p^{\frac{1}{n}} \vert)$.
\item A ``mixed'' situation like $ \mathbb{B}^o(0,1) \times_{\Spa (F, \ocal_F)} \mathbb{B}(0,1)$. 
\end{itemize}
\end{ex}

We also recall the following classical acyclicity result: 
\begin{thm}[\cite{MR210949}, Satz 2.4]\label{thmKhiel}  Let $\mathcal{X}$ be a quasi-Stein adic space and let $\mathscr{F}$ be a coherent sheaf over $\mathcal{X}$. Then $\mathrm{H}^i(\mathcal{X}, \mathscr{F}) = 0$ for all $i >0$. 
\end{thm}

We also have:

\begin{prop}  Let $\mathcal{X} = \Spa (A, A^+)$ be an affinoid finite type adic space, let $M$ be a projective Banach $A$-module, and let $\mathscr{F} = M \hat{\otimes}_A \oscr_{\mathcal{X}}$. Let $\mathcal{U} \hookrightarrow \mathcal{X}$  be a quasi-Stein  open subset. Then $\mathrm{H}^i(\mathcal{U}, \mathscr{F}) = 0$ for all $i >0$. 
\end{prop}
\begin{proof} We reduce to the case where $M$ is orthonormalizable, and therefore to the case of $\oscr_{\mathcal{X}}$ where we can apply theorem \ref{thmKhiel}.
 \end{proof}

\subsubsection{Cohomology complexes}
We now illustrate how one can obtain  complexes of Banach modules.  We denote by $\mathcal{D}(F)$ the derived category of the category of $F$-vector spaces. We also let $\mathbf{Ban}(F)$ be the category of Banach $F$-vector spaces. 
For an adic space $\mathcal{X}$ over $\Spa(F, \ocal_F)$, and a sheaf $\mathscr{F}$ of $\oscr_{\mathcal{X}}$-module, the cohomology groups  $\mathrm{R}\Gamma(\mathcal{X}, \mathscr{F})$ are objects of the category $\mathcal{D}(F)$. Nevertheless, they often  carry more structure and can be represented by objects of the category $\mathcal{C}(\mathbf{Ban}(F))$ of  complexes of Banach modules.   We formalize this in   this section.
\begin{rem} Our approach is rather elementary and we content ourselves with statements saying that a given cohomology can be represented by certain complexes of  topological $F$-vector spaces.  This is enough for our purposes. We simply point out in this remark a more conceptual approach yielding stronger results. Let $\mathscr{F}$ be a coherent sheaf over an adic space $\mathcal{X}$ of finite type over $\Spa (F, \oscr_F)$. Since $\mathscr{F}$ carries a canonical topology, we can view $\mathscr{F}$ as sheaf of locally convex $F$-vector spaces over $\mathcal{X}$.
Let $\mathbf{Loc}_F$ be the category of locally convex $F$-vector spaces. This is a quasi-abelian category with enough injectives (\cite{MR1749013}). One can therefore derive the functor $\mathrm{H}^0(\mathcal{X}, \mathscr{F})$ in this category  and  view $\mathrm{R}\Gamma(\mathcal{X}, \mathscr{F})$ as an object of $\mathcal{D}(\mathbf{Loc}_F)$.
\end{rem}

\begin{lem}\label{lem-represent-Banach}
Let $\mathcal{X}$ be a  separated  finite type  adic space  over $\Spa (F, \ocal_F)$. Let $\mathscr{F}$ be a locally projective Banach sheaf over $\mathcal{X}$. Let $\mathcal{U} \subseteq \mathcal{X}$ be a quasi-compact open subset. Let $\mathcal{Z} \subseteq \mathcal{X}$ be a closed subset, with quasi-compact complement. Then one can represent    $\mathrm{R}\Gamma_{\mathcal{Z} \cap \mathcal{U}} ( \mathcal{U}, \mathscr{F})$ by  an object of   $\mathcal{K}^{proj}(\mathbf{Ban}(F))$. 
\end{lem}
\begin{proof}  We have an exact triangle $\mathrm{R}\Gamma_{\mathcal{Z} \cap \mathcal{U}}(\mathcal{U}, \mathscr{F}) \rightarrow \mathrm{R}\Gamma(\mathcal{U}, \mathscr{F}) \rightarrow \mathrm{R}\Gamma(\mathcal{U} \setminus (\mathcal{Z} \cap \mathcal{U}), \mathscr{F}) \stackrel{+1}\rightarrow $ and therefore, we are reduced to prove the claim for $\mathrm{R}\Gamma(\mathcal{U}, \mathscr{F})$ and $\mathrm{R}\Gamma(\mathcal{U} \setminus (\mathcal{Z} \cap \mathcal{U}), \mathscr{F})$. Finally, it suffices to prove that for a quasi-compact open $\mathcal{U} \subseteq \mathcal{X}$, $\mathrm{R}\Gamma(\mathcal{U}, \mathscr{F}) \in Ob(\mathcal{K}^{proj}(\mathbf{Ban}(F)))$. We can compute the cohomology by considering an  $\mathscr{F}$-acyclic affinoid covering of $\mathcal{U}$ for $\mathscr{F}$, and the associated \v{C}ech complex by \cite{MR1306024}, thm. 2.5. Then each of the terms of the \v{C}ech complex carries a canonical structure of Banach $F$-algebra.
\end{proof}

\bigskip

\begin{lem}\label{lem-complex-frechet}
Let $\mathcal{X}$ be a  separated finite type  adic space  over $\Spa (F, \ocal_F)$. Let $\mathscr{F}$ be a locally projective Banach sheaf over $\mathcal{X}$. Let $\mathcal{U} \subseteq \mathcal{X}$ be an open subset which is a finite union of quasi-Stein spaces. Let $\mathcal{Z} \subseteq \mathcal{X}$ be a closed subset, whose complement is a finite union of quasi-Stein spaces. Then one can represent   $\mathrm{R}\Gamma_{\mathcal{Z} \cap \mathcal{U}} ( \mathcal{U}, \mathscr{F})$  by an object of    $\mathrm{Pro}_{\mathbb{N}}(\mathcal{K}^{proj}(\mathbf{Ban}(F)))$. 
\end{lem}
\begin{proof} As in the proof of  lemma \ref{lem-represent-Banach}, we are reduced to showing that  for an  open $\mathcal{U} \subseteq \mathcal{X}$ which is a finite union of quasi-Stein spaces, $\mathrm{R}\Gamma(\mathcal{U}, \mathscr{F}) \in Ob(\mathcal{C}(\mathrm{Pro}_{\mathbb{N}}(\mathbf{Ban}(F))))$. We let $\mathcal{U} = \cup_k \mathcal{V}_k$ be a finite covering of $\mathcal{U}$ by $\mathscr{F}$-acyclic  quasi-Stein spaces $\mathcal{V}_k$. We let $\mathcal{V}_k= \cup_{i \geq 0} \mathcal{V}_{k,i}$ where $\mathcal{V}_{k,i}$ is affinoid. We let $\mathcal{U}_i = \cup_{k} \mathcal{V}_{k,i}$. We let $\mathrm{R}\Gamma(\mathcal{U}, \mathscr{F}) = ``\lim"\mathrm{R}\Gamma(\mathcal{U}_i, \mathscr{F})$.
\end{proof}

\subsubsection{Compact morphisms in the cohomology of adic spaces}

We  give some examples of compact morphisms arising from maps between the cohomology of adic spaces. First, let us fix a standard notation.   Let  $T$ be a topological space and let  $S$ be a subset of $T$. Then we denote by  $\overline{S}$  the closure of  $S$ in $T$ and  by  $\overset{\circ}{S}$ the interior of $S$ in $T$.

\begin{lem} Let $\mathcal{X} \rightarrow \Spa(F,\ocal_F)$ be a proper adic space and let $\mathscr{F}$ be a locally free sheaf of finite rank over $\mathcal{X}$. Let $\mathcal{U}' \subseteq \mathcal{U} \subseteq \mathcal{X}$ be quasi-compact open subsets. Assume that $\overline{\mathcal{U}'} \subseteq \mathcal{U}$.  Then the map $$ \mathrm{R}\Gamma (\mathcal{U}, \mathscr{F}) \rightarrow \mathrm{R}\Gamma (\mathcal{U}', \mathscr{F})$$ is compact.
\end{lem}

\begin{proof} We claim that there exists a formal model $\mathfrak{X} \rightarrow \Spf~\ocal_F$ of $\mathcal{X}$, and two opens $\mathfrak{U}$ and $\mathfrak{U}'$ of $\mathfrak{X}$ with generic fibers $\mathcal{U}$ and $\mathcal{U}'$, such that $\overline{\mathfrak{U}'} \subseteq \mathfrak{U}$.  Remark that $\overline{\mathfrak{U}'} \rightarrow \Spf~\ocal_F$ is proper because $\mathcal{X}$ was assumed to be proper (\cite{MR1032938}, Th. 3.1). We deduce that $\mathcal{U}'$ is relatively compact in $\mathcal{U}$   (\cite{MR1032938}, lem. 2.5). 
We prove the claim. Recall that the ringed space $(\mathcal{X}, \oscr_{\mathcal{X}}^+)$ is the inverse limit of the ringed spaces $(\mathfrak{X}, \oscr_{\mathfrak{X}})$ where $\mathfrak{X}$ runs over all the formal models of $\mathcal{X}$.   For a cofinal subset of $\mathfrak{X}$,  we have opens $\mathfrak{U}_{\mathfrak{X}}$ and $\mathfrak{U}'_{\mathfrak{X}}$ of $\mathfrak{X}$ with generic fiber $\mathcal{U}$ and $\mathcal{U}'$. We let $\overline{\mathcal{U}'}_{\mathfrak{X}}$ be the generic fiber of $\overline{\mathfrak{U}'}_{\mathfrak{X}}$. Then $\overline{\mathcal{U}'} = \cap_{\mathfrak{X}} \overline{\mathcal{U}'}_{\mathfrak{X}}$. The topological space $\mathcal{X}$ equipped with the constructible topology is compact (in fact this is a profinite set). We have that $(\mathcal{U})^c = \cup_{\mathfrak{X}} (\overline{\mathcal{U}'}_{\mathfrak{X}})^c \cap (\mathcal{U})^c$. Since $(\mathcal{U})^c$ is compact, we deduce that there is a model $\mathfrak{X}$ such that $ \overline{\mathcal{U}'}_{\mathfrak{X}} \subseteq \mathcal{U}$, and therefore $\overline{\mathfrak{U}'}_{\mathfrak{X}} \subseteq \mathfrak{U}_{\mathfrak{X}}$. This finishes the proof of the claim.

Let $\mathcal{U}' = \cup_{i \in I} \mathcal{U}'_i$ be a finite  affinoid cover of $\mathcal{U}'$. By \cite{MR1032938}, thm. 5.1, for each $i$, there exists an affinoid  $\mathcal{U}'_i \subseteq \mathcal{U}_i \subseteq \mathcal{U}$ such that $\mathcal{U}'_i$ is relatively compact in $\mathcal{U}_i$ (equivalently $\overline{\mathcal{U}'_i} \subseteq \mathcal{U}_i$).  Let $\mathcal{U}'' = \cup_{i \in I} \mathcal{U}_i$. We claim that the map $$ \mathrm{R}\Gamma (\mathcal{U}'', \mathscr{F}) \rightarrow \mathrm{R}\Gamma (\mathcal{U}', \mathscr{F})$$ is compact. This will prove that the map of the lemma is compact because it factors over a compact map.   These last  cohomologies  can be represented by  the \v{C}ech complex with respect to $\cup_i \mathcal{U}_i$ and $\cup_i \mathcal{U}_i'$. We  now prove  that the maps $\mathscr{F}(\cap_{J \subseteq I} \mathcal{U}_i) \rightarrow \mathscr{F}(\cap_{J \subseteq I}\mathcal{U}'_i)$ are compact. We reproduce the argument of   \cite{MR2154369}, prop. 2.4.1. We have just seen that we can find a formal model $\mathfrak{V}' \hookrightarrow \mathfrak{V} \hookrightarrow \mathfrak{X}$ (with maps being Zariski open immersions) for the maps $\cap_{J \subseteq I}\mathcal{U}'_i \rightarrow \cap_{J \subseteq I} \mathcal{U}_i \rightarrow \mathcal{X}$ with the property that the map $\mathfrak{V}' \hookrightarrow \mathfrak{V}$ factors into  $\mathfrak{V}' \hookrightarrow \overline{\mathfrak{V}'}  \hookrightarrow  \mathfrak{V}$. We also let $\mathfrak{F}$ be a  flat coherent sheaf over $\mathfrak{X}$ with generic fiber $\mathscr{F}$. Let $\mathfrak{F}_n = \mathfrak{F}/\varpi^n $. We have that $\mathscr{F}(\cap_{J \subseteq I} \mathcal{U}_i) = \mathfrak{F}(\mathfrak{V}) [\frac{1}{\varpi}]$, and the image of $\mathfrak{F}(\mathfrak{V})$ in $ \mathscr{F}(\cap_{J \subseteq I} \mathcal{U}_i)$ is an open and bounded submodule. This also applies to the $'$ situation. We claim that the image of $\mathfrak{F}(\mathfrak{V})$ in $\mathfrak{F}(\mathfrak{V}')/\varpi^n$ is a finite $\ocal_F$-module. This will prove that the map $\mathscr{F}(\cap_{J \subseteq I} \mathcal{U}_i) \rightarrow \mathscr{F}(\cap_{J \subseteq I}\mathcal{U}'_i)$ is compact. 
The image of $\mathfrak{F}(\mathfrak{V})$ in $\mathfrak{F}(\mathfrak{V}')/\varpi^n$ is contained (inside $\mathfrak{F}_n(\mathfrak{V}')$) in the image of $\mathfrak{F}_n(\mathfrak{V})$. But the map $\mathfrak{F}_n(\mathfrak{V}) \rightarrow \mathfrak{F}_n(\mathfrak{V}')$ factors over $\mathfrak{F}_n(\mathfrak{V}) \rightarrow \mathfrak{F}_n(\overline{\mathfrak{V}'}) \rightarrow \mathfrak{F}_n(\mathfrak{V}')$.  The $\ocal_F$-module  $\mathfrak{F}_n(\overline{\mathfrak{V}'})$  is finite since $\overline{\mathfrak{V}'}$ is proper. 

\end{proof}

\begin{lem}\label{lem-criterium-compacity} Let $\mathcal{X} \rightarrow \Spa(F,\ocal_F)$ be a proper adic space and let $\mathscr{F}$ be a locally free sheaf of finite rank over $\mathcal{X}$. Let $\mathcal{U}' \subseteq \mathcal{U}$  be two quasi-compact open, and $\mathcal{Z} \subseteq \mathcal{Z}'$ be two  closed subspaces, with quasi-compact complements.  Assume that $\overline{\mathcal{U}}' \subseteq \mathcal{U}$ and $\mathcal{Z} \subseteq \overset{\circ}{\mathcal{Z}'}$. Then the map 
$$ \mathrm{R}\Gamma_{\mathcal{Z} \cap \mathcal{U} } (\mathcal{U}, \mathscr{F}) \rightarrow \mathrm{R}\Gamma_{\mathcal{Z}' \cap \mathcal{U} '} (\mathcal{U}', \mathscr{F})$$ is compact.
\end{lem}
\begin{proof}  It suffices to see that $\mathrm{R}\Gamma (\mathcal{U}, \mathscr{F}) \rightarrow \mathrm{R}\Gamma (\mathcal{U}', \mathscr{F})$ and $\mathrm{R}\Gamma (\mathcal{U}\setminus (\mathcal{Z}\cap \mathcal{U}), \mathscr{F}) \rightarrow \mathrm{R}\Gamma (\mathcal{U}' \setminus (\mathcal{Z}'\cap \mathcal{U}'), \mathscr{F})$ are compact. This follows from the fact that $\overline{\mathcal{U}}' \subseteq \mathcal{U}$ and $\overline{\mathcal{U}' \setminus (\mathcal{Z}'\cap {\mathcal{U}'})} \subseteq \overline{\mathcal{U}}' \setminus (\overset{\circ}{\mathcal{Z}'}\cap \overline{\mathcal{U}}')\subseteq  \mathcal{U} \setminus  (\mathcal{Z} \cap \mathcal{U}) $ and the previous lemma. 
\end{proof}

\bigskip

We now give a stronger form of the lemma.

\begin{lem}\label{lem-criterium-compacity2}  Assume that $\mathcal{X}$ is proper. Let $\mathcal{U}' \subseteq \mathcal{U}$  be two  open subsets which are finite unions of quasi-Stein spaces, and $\mathcal{Z} \subseteq \mathcal{Z}'$ be two  closed subspaces, whose complements are finite unions of quasi-Stein spaces.  Assume that there exists a quasi-compact open $\mathcal{U}''$ such that ${\mathcal{U}}' \cap \mathcal{Z}' \subseteq \mathcal{U}''$ with $\overline{\mathcal{U}''} \subseteq \mathcal{U}$, as well as two closed subset $\mathcal{Z}'' \subseteq \mathcal{Z}'''$ with quasi-compact complement,  such that  $\mathcal{Z}\cap \mathcal{U} \subseteq \mathcal{Z}^{''}$, $\mathcal{Z}^{''} \subseteq  \overset{\circ}{\mathcal{Z}'''}$ and $\mathcal{Z}^{'''} \subseteq {\mathcal{Z}'}$. Then the map 
$$ \mathrm{R}\Gamma_{\mathcal{Z} \cap \mathcal{U} } (\mathcal{U}, \mathscr{F}) \rightarrow \mathrm{R}\Gamma_{\mathcal{Z}' \cap\mathcal{ U} '} (\mathcal{U}', \mathscr{F})$$ is compact. 
\end{lem}
\begin{proof}  We can write $\mathcal{U} = \cup_n \mathcal{U}_n$ with $\mathcal{U}_n$ quasi-compact. Since $\overline{\mathcal{U}}'' \subseteq \mathcal{U}$ and $\mathcal{X}$ equipped with the constructible topology is compact, we deduce that $\overline{\mathcal{U}}'' \subseteq \mathcal{U}_n$ for $n$ large enough.   The map  $\mathrm{R}\Gamma_{\mathcal{Z}'' \cap \mathcal{U}_n } (\mathcal{U}_n, \mathscr{F}) \rightarrow  \mathrm{R}\Gamma_{\mathcal{Z}''' \cap  \mathcal{U}'' }( \mathcal{U}'', \mathscr{F})$ is compact by lemma \ref{lem-criterium-compacity}.
There is a restriction-corestriction map  $\mathrm{R}\Gamma_{\mathcal{Z} \cap \mathcal{U} } (\mathcal{U}, \mathscr{F}) \rightarrow \mathrm{R}\Gamma_{\mathcal{Z}'' \cap \mathcal{U}_n } (\mathcal{U}_n, \mathscr{F})$.

We also have  a restriction-corestriction map $$ \mathrm{R}\Gamma_{\mathcal{Z}''' \cap \mathcal{U}'' } (\mathcal{U}'', \mathscr{F}) \rightarrow \mathrm{R}\Gamma_{\mathcal{Z}' \cap \mathcal{U}' } (\mathcal{U}' \cap \mathcal{U}'', \mathscr{F}).$$
On the other hand, $\mathrm{R}\Gamma_{\mathcal{Z}' \cap  \mathcal{U}'}( \mathcal{U}', \mathscr{F}) = \mathrm{R}\Gamma_{\mathcal{Z}' \cap  \mathcal{U}' }( \mathcal{U}' \cap \mathcal{U}'', \mathscr{F})$. 
All together, we deduce that  the map 
$$ \mathrm{R}\Gamma_{\mathcal{Z} \cap\mathcal{ U} } (\mathcal{U}, \mathscr{F}) \rightarrow \mathrm{R}\Gamma_{\mathcal{Z}' \cap \mathcal{U} '} (\mathcal{U}', \mathscr{F})$$
factors through the compact map $\mathrm{R}\Gamma_{\mathcal{Z}'' \cap \mathcal{U}_n } (\mathcal{U}_n, \mathscr{F}) \rightarrow  \mathrm{R}\Gamma_{\mathcal{Z}''' \cap  \mathcal{U}'' }( \mathcal{U}'', \mathscr{F})$ and is compact. 
\end{proof}

We finally conclude this section with a last lemma where we deal with Banach sheaves which are not necessarily coherent.

\begin{lem}\label{last-lemma-compact}  Let $\mathcal{X}$, $\mathcal{U}$, $\mathcal{U}'$, $\mathcal{Z}$, $\mathcal{Z}'$ be as in lemma \ref{lem-criterium-compacity2}. Let $\mathscr{F}$ and $\mathscr{G}$ be two locally projective Banach sheaves and let $\phi : \mathscr{F} \rightarrow \mathscr{G}$ be a compact morphism.  The morphism $$ \mathrm{R}\Gamma_{\mathcal{Z} \cap \mathcal{U} } (\mathcal{U}, \mathscr{F}) \rightarrow \mathrm{R}\Gamma_{\mathcal{Z}' \cap \mathcal{U} '} (\mathcal{U}', \mathscr{G})$$ is compact. 
\end{lem}

\begin{proof} Easy and left to the reader.
\end{proof}

\subsection{Integral structures on Banach sheaves} We now consider integral structures on Banach sheaves, but this time more in the spirit of analytic geometry. 
 We let $\mathcal{X}$ be a  separated  adic space locally of finite type over $\Spa(F, \ocal_F)$. We let $\mathscr{F}$ be a locally projective Banach sheaf  over $\mathcal{X}$. We can view $\mathscr{F}$ as a sheaf on the \'etale site of $\mathcal{X}$ by \cite{MR1603849}.
 
 \begin{defi}\label{defi-integral-structure-sheaf} An integral structure on $\mathscr{F}$ is   a sheaf $\mathscr{F}^+$ of $\oscr_{\mathcal{X}}^+$-modules  on the \'etale site of $X$, such that:
 \begin{enumerate} 
 \item $\mathscr{F}^+ \hookrightarrow \mathscr{F}$ and $\mathscr{F}^+ \otimes_{\ocal_F} F = \mathscr{F}$, 
 \item There is  an \'etale cover  $\coprod U_i \rightarrow \mathcal{X}$ by affinoid spaces such that $\mathscr{F}^+(U_i)$ is the completion of a  free $\oscr_{\mathcal{X}}^+(U_i)$-module, and the canonical map 
  $\mathscr{F}^+(U_i) \hat{\otimes}_{\oscr_{\mathcal{X}}^+(U_i)} \oscr_{U_i}^+ \rightarrow    \mathscr{F}^+\vert_{U_i}$ is an isomorphism. 
  \end{enumerate} 
  \end{defi}
  
  \begin{rem} A stronger property would be to ask that the \'etale cover $\coprod U_i \rightarrow \mathcal{X}$ is in fact an analytic cover. In our applications, this stronger property will not be satisfied. Indeed, we will produce sheaves arising from torsors under various groups, and these torsors are usually only trivial locally for the \'etale topology. However, we will not consider the \'etale cohomology $\HH^i_{et}(\mathcal{X},\mathscr{F}^+)$, but only the analytic cohomology $\HH^i_{an}(\mathcal{X},\mathscr{F}^+)$. 
  \end{rem} 
  
  \begin{lem}\label{lem-bounded-torsion-integral} In the situation of definition \ref{defi-integral-structure-sheaf}, assume moreover that $\mathcal{X}$ is reduced. Let  $U \rightarrow \mathcal{X}$ be an \'etale map. 
Then $\mathscr{F}^+(U)$ is an open and bounded submodule of $\mathscr{F}(U)$. 
  \end{lem}
  \begin{proof} Let $I$ be a finite set and let $\coprod_{i \in I} U_i \rightarrow U$ be an \'etale cover such each $U_i$ is affinoid and  $\mathscr{F}^+\vert_{U_i}$ is the completion of a  free sheaf of $\oscr_{U_i}^+$-modules. By the sheaf property we have an exact sequence:
  $$ 0 \rightarrow  \mathscr{F}(U) \rightarrow \prod_{i} \mathscr{F}(U_i) \rightarrow \prod_{i,j} \mathscr{F}(U_i \times_{U} U_j)$$
and $\mathscr{F}(U)$  is a closed Banach subspace of $ \prod_{i} \mathscr{F}(U_i)$. Since $\prod_{i} \mathscr{F}^+(U_i)$ is open and bounded in $\prod_{i} \mathscr{F}(U_i)$ (using that $\oscr^+_{\mathcal{X}}(U_i)$ is bounded by the reducedness hypothesis),  $\mathscr{F}^+(U) = \prod_{i} \mathscr{F}^+(U_i) \cap \mathscr{F}(U)$ is open and bounded in $\mathscr{F}^+(U)$.
  \end{proof}

  We now will elaborate on a result of Bartenwerfer which we first recall. 
  
  \begin{thm}[\cite{MR488562}]\label{thm-ref-Barten} Let $\mathcal{X}$   be an affinoid  smooth  adic space  over $\Spa(F, \ocal_F)$. There exists $N \in \ZZ_{\geq 0}$ such that $\HH^i_{an}(\mathcal{X}, \oscr_{\mathcal{X}}^+)$ is annihilated by $p^N$  for all $i >0$. 
  \end{thm}
  \begin{proof} Bartenwerfer's result is stated for \v{C}ech cohomology. By \cite{pilloniHidacomplexes} prop. 3.1.1, this implies the claim for cohomology. 
  \end{proof}
  
  \begin{lem}\label{lem-Cechversusanalytic} Let $\mathcal{X}$   be an affinoid  smooth  adic space  over $\Spa(F, \ocal_F)$. Let $\mathscr{F}$ be a locally projective Banach sheaf which is assumed to be associated to its global sections (i.e. satisfies  point $(2)$ in definition  \ref{defi-banach-sheaf}).  Let $\mathscr{F}^+$ be an integral structure on $\mathscr{F}$.   There exists $N\in \ZZ_{\geq 0}$ such that  for all $i >0$ the  cohomology groups $\HH^i_{an}(\mathcal{X}, \mathscr{F}^+)$ are annihilated by $p^N$. 
  \end{lem}
  \begin{proof}  Let $\mathcal{X} = \Spa(A,A^+)$.  By assumption, $\mathscr{F} = M \hat{\otimes}_{A} \oscr_\mathcal{\mathcal{X}}$ is associated to a projective Banach $A$-module $M$. Let $I$ be a set such that $A(I) = M \oplus N$.  Let $M^+ = A^+(I) \cap M$ and $N^+= A^+(I) \cap N$. The injective map $M^+ \oplus N^+ \rightarrow A^+(I)$ has cokernel of bounded torsion.  Moreover, if we let $M_+$ be the image of $A^+(I)$ in $M$ under the projection orthogonal to $N$, then $M^+ \hookrightarrow M_+$ has cokernel of bounded torsion. We deduce that there exists an integer $N$ such that the multiplication $p^N : M^+ \rightarrow M^+$ factors through:  
  $ M^+ \rightarrow A^+(I) \rightarrow M^+$ (where the first map is the inclusion, the second map is the orthogonal projection with respect to $N$ composed with multiplication by $p^N$ and followed by the inclusion $p^N M_+ \subseteq M^+$).  
  It follows from theorem \ref{thm-ref-Barten}  that $\HH^i (\mathcal{X}, \oscr_{\mathcal{X}}^+ \hat{\otimes}_{A^+} A^+(I))$ is of bounded torsion for all $i>0$. We let $\mathscr{M}^+$ be the subsheaf of $\oscr_{\mathcal{X}}^+ \hat{\otimes}_{A^+} A^+(I) \cap \mathscr{F}$ equal to the image of the sheaf associated to the presheaf  $\oscr_{\mathcal{X}}^+ \hat{\otimes}_{A^+} M^+$ (in other words, it is the subsheaf of $\mathscr{F}$ of sections which can locally be written as tensors   in $\oscr_{\mathcal{X}}^+ \hat{\otimes}_{A^+} M^+$).
  
We see that  multiplication by $p^N : \mathscr{M}^+ \rightarrow \mathscr{M}^+$ factors through   $ \mathscr{M}^+ \rightarrow \oscr_{\mathcal{X}}^+ \hat{\otimes}_{A^+}A^+(I) \rightarrow \mathscr{M}^+$. We deduce that 
  $\HH^i (\mathcal{X}, \mathscr{M}^+)$ is of bounded torsion for all $i>0$.
  After rescaling, we may assume that $M^+ \subseteq \mathscr{F}^+(\mathcal{X})$, with cokernel of bounded torsion by lemma \ref{lem-bounded-torsion-integral}. We therefore get a morphism: 
  $\mathscr{M}^+  \rightarrow \mathscr{F}^+$. We claim that this morphism has cokernel of bounded torsion. 
Let $\cup_{i \in I} U_i \rightarrow \mathcal{X}$ be an affinoid  \'etale covering with the property that $\mathscr{F}^+\vert_{U_i}$ is associated with its global sections. We may assume that the set $I$ is finite. It suffices to show that $\oscr_{\mathcal{X}}^+(U_i) \hat{\otimes}_{A^+} M^+  \rightarrow \mathscr{F}^+(U_i)$ has cokernel of bounded torsion. This follows since the image of $\oscr_{\mathcal{X}}^+(U_i) \hat{\otimes}_{A^+} M^+$ is open. 
  We finally deduce that $\HH^i_{an}(\mathcal{X}, \mathscr{F}^+)$ is of bounded torsion for all $i >0$. 
 \end{proof}
  
\subsection{Duality for analytic adic spaces}
 In this section we fix $\mathcal{X}$ a proper smooth adic space over $\Spa(F, \ocal_F)$ of pure dimension $d$. Let $\Omega^{d}_{\mathcal{X}/F}$ be the canonical sheaf, equal to $\Lambda^d \Omega^1_{\mathcal{X}/F}$.  Let $\iota  : \mathcal{Z} \hookrightarrow \mathcal{X}$ be a closed subset, equal to the closure of a quasi-compact open subset $\mathcal{U}$ of $\mathcal{X}$.
 One can consider the cohomology group $\HH^d_{\mathcal{Z}}(\mathcal{X}, \Omega^{d}_{\mathcal{X}/F})$.

 \begin{thm}[\cite{MR1739729}, \cite{MR1464367}]\label{thm-duality-GK} \begin{enumerate}
 \item There is a trace map $\mathrm{tr}_{\mathcal{Z}} : \HH^d_{\mathcal{Z}}(\mathcal{X}, \Omega^{d}_{\mathcal{X}/F}) \rightarrow F$. 
  \item If $\mathcal{Z} \subseteq \mathcal{Z}'$, there is a factorization $$\mathrm{tr}_{\mathcal{Z}} : \HH^d_{\mathcal{Z}}(\mathcal{X}, \Omega^{d}_{\mathcal{X}/K})  \rightarrow \HH^d_{\mathcal{Z}'}(\mathcal{X}, \Omega^{d}_{\mathcal{X}/F}) \stackrel{\mathrm{tr}_{\mathcal{Z}'}}\rightarrow F.$$
  \item For any coherent sheaf $\mathscr{F}$ defined on a neighborhood of $\mathcal{Z}$ the map $\mathrm{tr}_{\mathcal{Z}}$ induces  a pairing:
  $$\mathrm{Ext}^i_{\iota^{-1} \oscr_{\mathcal{X}}} (\iota^{-1} \mathscr{F}, \iota^{-1} \Omega^d_{\mathcal{X}/F}) \times \HH^{d-i}_{\mathcal{Z}}(\mathcal{X}, \mathscr{F}) \rightarrow F$$
 \item When $\mathcal{Z}=\mathcal{X}$ the cohomology groups are finite dimensional $F$-vector spaces,  and the pairing is perfect.
 \item When $\mathcal{U}$ is affinoid and $\mathscr{F}$ is a locally free sheaf, $\mathrm{Ext}^0_{\iota^{-1} \oscr_{\mathcal{X}}}(\iota^{-1} \mathscr{F}, \iota^{-1} \Omega^d_{\mathcal{X}/F})$  is  a compact  inductive limit of $F$-Banach spaces, $\HH^d_{\mathcal{Z}}(\mathcal{X}, \mathscr{F})$  is a compact projective limit of $F$-Banach  spaces, and the pairing is a topological duality between  locally convex topological $F$-vector spaces spaces identifying each space with the strong dual of the other.  Moreover, $\mathrm{Ext}^i_{\iota^{-1} \oscr_{\mathcal{X}}} (\iota^{-1} \mathscr{F}, \iota^{-1} \Omega^d_{\mathcal{X}/F})$ and $\HH^{d-i}_{\mathcal{Z}}(\mathcal{X}, \mathscr{F})$ vanish for $i \neq 0$.  \end{enumerate}
  \end{thm}
  
  \begin{rem}\label{rem-dagger} We give some translations to the  language of \cite{MR1739729}. Let $\mathcal{U}^\dag$ be the dagger space attached to $\mathcal{U} \hookrightarrow \mathcal{X}$. Let $\mathscr{F}$ be a coherent sheaf defined on a neighborhood of $\mathcal{Z}$ in $\mathcal{X}$. Then $\HH^i(\mathcal{U}^\dag, \mathscr{F}) = \HH^i(\mathcal{Z}, \iota^{-1} \mathscr{F})$ and  
  $\HH^i_c(\mathcal{U}^\dag,  \mathscr{F}) = \HH^i_{\mathcal{Z}}(\mathcal{X}, \mathscr{F})$.
 \end{rem}
\section{Flag variety}\label{section-Flagvar}

\subsection{Bruhat decomposition} Let $F$ be a non archimedean local field of mixed characteristic with residue field $k$ of characteristic $p$ and discrete valuation $v$
normalized so that $v(p)=1$.
 Let $G \rightarrow \ocal_F$ be a split reductive group with maximal torus $T$ contained in a Borel $B$ with unipotent radical $U$. Let $B \subseteq P$ be a parabolic with Levi $M \subseteq P$ containing $T$.
  \begin{rem} In our application to Shimura varieties, the unipotent radical of $P$ will be abelian. Nevertheless, this assumption is not relevant for the moment.  \end{rem}

  We let $\Phi$ be the set of roots, $\Delta$ and $\Phi^+$ the subsets of simple and positive roots corresponding to our choice of $B$, and $\Phi^- = - \Phi^+$. We let $\Phi^+_M$ be the subset of positive roots which lie in the Lie algebra of $M$ and $\Phi^{+,M}  = \Phi^+ \setminus \Phi^+_M$. We let $\Phi^-_M = - \Phi^+_M$ and $\Phi^{-,M} = - \Phi^{+,M}$.

 We let $W$ be the Weyl group of $G$. We denote by $\ell : W \rightarrow \ZZ_{\geq 0}$ the length function.  We denote by $w_0$ the longest element of $W$.  For each $w \in W$, we choose a representative in $N(T)$  that we still denote $w$. The group $W$ acts on the left on the cocharacter group $X_{\star}(T)$ and on the character group $X^\star(T)$ on the left as well via the formula $w\kappa(t) = \kappa(w^{-1} t w)$ for $\kappa \in X^\star(T)$ and $w \in W$. Let $W_M$ be the Weyl group of $M$. The quotient $W_M \backslash W$ has a set of coset representatives of minimal length (the Kostant representatives) called $\WM$.  This is the subset of $W$ of elements  $w$ that satisfy $\Phi_M^+ \subseteq w \Phi^+$.

  Let $FL = P\backslash G \rightarrow \Spec~\ocal_F$ be the flag variety associated with $P$.  We let $d=\#\Phi^{+,M}$ be the (relative) dimension of $FL$.  The group $G$ acts on the right on $FL$. 
 
For any $w \in\WM$, we let  $C_w = P \backslash P w B$ be the Bruhat cell corresponding to $w$. We have the decomposition into $B$-orbits $FL = \coprod_{w \in\WM} C_w$.  We can also consider the opposite Bruhat cell: $C^w = P \backslash P w \overline{B}$ for the opposite Borel $\overline{B}$. 

We let $X_w $ be the Schubert variety equal to the Zariski closure of $C_w$ in $FL$. We also let $X^w$ be the opposite Schubert variety, equal to the Zariski closure of $C^w$ in $FL$. 
There is a partial order $\leq $ on $\WM$   for which $X_w = \cup_{w' \leq w} C_{w'}$ and $X^w = \cup_{w' \geq w} C^{w'}$. For the length function $\ell : \WM \rightarrow [0, \mathrm{dim} FL]$, we have  $\ell(w) = \dim C_w$ (here dimensions are relative dimensions over $\Spec~\ocal_F$). 

We also  define $Y_w = \cup_{w' \geq w} C_{w}$ for all $w\in\WM$. This is an open subscheme of $FL$ containing $C_{w}$. 

\begin{lem}\label{lem-oppositeSchubert} We have an inclusion $X^w \hookrightarrow Y_w$.
\end{lem}

\begin{proof} By \cite{MR2017071}, I, lemma 1, we know that $X^w \cap X_v \neq \emptyset \Leftrightarrow v \geq w$. We deduce that $X^w \cap C_v \neq \emptyset  \Rightarrow  v \geq w $, so that $X^w \subseteq \cup_{v \geq w} C_v$.  
\end{proof}

\bigskip

The following lemma gives a description of the Bruhat cells. For all $\alpha \in \Phi$, we let $U_{\alpha}$ be the one parameter subgroup corresponding to $\alpha$. 

\begin{lem}\label{lem-cell-product} The product map (for any ordering of the factors)
\begin{eqnarray*}
\prod_{\alpha \in  (w^{-1} \Phi^{-,M}) \cap \Phi^+} U_\alpha  & \rightarrow & C_w \\
(x_\alpha) &\mapsto &w \prod_{\alpha} x_\alpha
\end{eqnarray*}
is an isomorphism of schemes.
\end{lem}

\begin{proof} We need to prove that the map $\prod_{\alpha \in  (w^{-1} \Phi^{-,M}) \cap \Phi^+} U_\alpha  \rightarrow (B\cap w^{-1} P w) \backslash B$ is an isomorphism. This follows easily from the following facts:
\begin{itemize}
\item $B = T \times \prod_{\alpha \in \Phi^+} U_\alpha$ (in any order),
\item  $\Phi^+ = (w^{-1} \Phi^{-,M}) \cap \Phi^+ \coprod (w^{-1} \Phi^+) \cap \Phi^+ \coprod (w^{-1} \Phi^-_M) \cap \Phi^+$,
\item $B\cap w^{-1} P w = T \times \prod_{\alpha \in  (w^{-1} \Phi^+) \cap \Phi^+ \coprod (w^{-1} \Phi^-_M) \cap \Phi^+} U_\alpha$.
\end{itemize}
\end{proof}

We also introduce the translated open cell $U_w:=C_{w_0}\cdot w_0^{-1}w$, so that by the lemma we have an isomorphism (for any ordering of the roots):  
\begin{eqnarray*}
\prod_{\alpha \in  w^{-1} \Phi^{-,M}} U_\alpha  & \rightarrow & U_w \\
(x_\alpha) &\mapsto &w \prod_{\alpha} x_\alpha.
\end{eqnarray*}
We see that $U_w$ contains the Bruhat cell $C_w$ as a closed subset.

\subsection{Interlude: the cohomology of the flag variety}\label{subsec-kempf}

In this section, which is independent of the rest of the paper, we discuss the coherent cohomology of the flag variety following \cite{MR509802}, section 12. We assume here that $P=B$ is a Borel subgroup. We consider the stratification 
$ Z_0 = FL \supseteq Z_1 \supseteq \cdots \supseteq Z_d \supseteq Z_{d+1} = \emptyset$
where $d = \mathrm{dim} FL$ and $Z_i = \bigcup_{ w \in W, \ell(w) = d-i} X_w$. 
Let $\kappa \in X^\star(T)$. We associate to $\kappa$ a $G$-equivariant line bundle $\mathcal{L}_\kappa$ on $FL$ as follows: if $\pi : G \rightarrow FL$ is the projection map, then for any open $U \hookrightarrow FL$, $$\mathcal{L}_\kappa(U) = \{ f : \pi^{-1}(U) \rightarrow \mathbb{A}^1\mid f(bu) = (w_0\kappa)(b) f(u),~\forall (b,u) \in B\times \pi^{-1}(U) \}.$$ 

We consider the spectral sequence associated to the filtration (section \ref{section-spectral-sequence}) $$E_1^{p,q} = \HH^{p+q}_{Z_p/Z_{p+1}} (FL, \mathcal{L}_\kappa) \Rightarrow \HH^{p+q}(FL, \mathcal{L}_\kappa).$$
The abutment of this spectral sequence carries an action of $G$.  The terms on the left do not carry an action of $G$, but then do carry an action of $B$, as well as the Lie algebra $\mathfrak{g}$ (see \cite{MR509802}, lemma 12.8), and the spectral sequence is equivariant for these actions.

In order to study this spectral sequence, we need the following basic result: 
\begin{lem}\label{lem-computation-coho-affine-space} Let $m \geq n$. Let $\mathbb{A}^m = \Spec~\ocal_F[x_1, \cdots, x_m]$ and $\mathbb{A}^n = \Spec~\ocal_F[x_1, \cdots, x_n]$. Let $\mathbb{A}^n \hookrightarrow \mathbb{A}^m$ be the closed immersion given by $x_{n+1} = \cdots = x_m=0$. We have $\HH^i_{\mathbb{A}^n}(\mathbb{A}^m) = 0$ if $i\neq m-n$, and $$\HH^{m-n}_{\mathbb{A}^n}(\mathbb{A}^m)  = \bigoplus_{k_1, \cdots, k_n \geq 0, \\~ k_{n+1}, \cdots, k_m < 0} \ocal_F \cdot \prod_{i=1}^m x_i^{k_i}$$
\end{lem}
\begin{proof} One can use a Koszul complex. See \cite{MR2171939}, expos\'e II, proposition 5.
\end{proof} 

Let $\rho =  \frac{1}{2} \sum_{\alpha \in \Phi^+} \alpha$. We define the dotted action of $W$ on $X^\star(T)$ by $w\cdot \kappa = w (\kappa + \rho)-\rho$. We also denote $\HH^\star_w(FL, \mathcal{L}_\kappa) = \HH^\star_{C_w}(U_w, \mathcal{L}_\kappa)$. 

\begin{lem}\label{lem-BGG-character} We have a decomposition $$\HH^{p+q}_{Z_p/Z_{p+1}} (FL, \mathcal{L}_\kappa) = \bigoplus_{w\in W, \ell(w)=d-p} \HH^{p+q}_w( FL, \mathcal{L}_\kappa)$$ and these groups vanish when $q\not=0$.  Moreover, $T$ acts on $\HH^{d-\ell(w)}_w( FL, \mathcal{L}_\kappa)$ with character:
$$ \frac{ [(w^{-1}w_0)\cdot \kappa]} {\prod_{\alpha \in \Phi^-} ([1]-[\alpha])}  $$
\end{lem}

Here the character is viewed as a function $X^\star(T)\to\ZZ$ and for $\lambda\in X^\star(T)$, we write $[\lambda]$ for the function which is 1 on $\lambda$ and 0 otherwise.  The multiplication is convolution.  We note that this says that the character of $\HH^{d-\ell(w)}_w(FL,\mathcal{L}_\kappa)$ is the same as that of the Verma module with highest weight $(w^{-1}w_0)\cdot\kappa$.

\begin{proof} Since $Z_p \setminus Z_{p+1} = \coprod_{w \in W, \ell(w) = d-p} C_w$ and $U_w$ is a neighborhood of $C_w$ in $FL$, we deduce from lemma \ref{lem-direct-sum-complex-support} that 
$$\HH^{p+q}_{Z_p/Z_{p+1}} (FL, \mathcal{L}_\kappa) = \oplus_{w \in W, \ell(w) =d-p} \HH^{p+q}_{C_w}(U_w, \mathcal{L}_\kappa).$$
By lemma \ref{lem-cell-product}, after choosing an ordering of the roots, we have $U_w\simeq \mathbb{A}^d$ with coordinates $x_\alpha$ for $\alpha\in w^{-1}\Phi^-$, and $C_w\simeq\mathbb{A}^{\ell(w)}$ with coordinates $x_\alpha$ for $\alpha\in w^{-1}\Phi^-\cap\Phi^+$.  Moreover the inclusion $C_w\subseteq U_w$ is given by the vanishing of $x_\alpha$ for $\alpha\in w^{-1}\Phi^-\cap \Phi^-$.  The torus $T$ acts on $U_w$ and $C_w$, and the action on the coordinates is given by $t\cdot x_\alpha=(-\alpha)(t)x_\alpha$.

Over $U_w$, the line bundle $\mathcal{L}_\kappa$ has a non-vanishing section $s$ given by the function $bw_0uw_0^{-1}w\mapsto (w_0\kappa)(b)$ on $\pi^{-1}(U_w)=Bw_0 Uw_0^{-1}w$.  We compute that the action of $T$ on $s$ is given by $t\cdot s=(w^{-1}w_0\kappa)(t)s$.

We now deduce from lemma \ref{lem-computation-coho-affine-space} that $\HH^{p+q}_{C_w}(U_w, \mathcal{L}_\kappa)$ is concentrated in degree $p$, and the cohomology is isomorphic to the free $\ocal_F$-module: $$\bigoplus_{k_\alpha \geq 0 ~\forall\alpha \in  w^{-1}\Phi^- \cap  \Phi^+ ,~ \\ k_{\alpha} < 0 ~\forall\alpha \in  w^{-1}\Phi^-\cap  \Phi^-} \ocal_F \cdot\prod x_\alpha^{k_\alpha}s$$
from which we can read off the  character of the $T$-action:
\begin{eqnarray*}
\mathrm{ch}\big(\HH^{p}_{C_w}(U_w, \mathcal{L}_\kappa)\big )&=&\sum_{ k_\alpha \geq 0 ~\forall\alpha \in  w^{-1}\Phi^+ \cap  \Phi^-, k_{\alpha} > 0 ~\forall\alpha \in  w^{-1}\Phi^- \cap  \Phi^-}[w^{-1}w_0\kappa+\sum k_\alpha \alpha]\\
&=&[w^{-1}w_0\kappa+\sum_{\alpha\in w^{-1}\Phi^-\cap\Phi^-}\alpha]\frac{1}{\prod_{\alpha\in\Phi^-}[1]-[\alpha]}\\
\end{eqnarray*}
Finally we note $$(w^{-1}w_0)\cdot\kappa=w^{-1}w_0(\kappa+\rho)-\rho=w^{-1}w_0\kappa-w^{-1}\rho-\rho=w^{-1}w_0\kappa+\sum_{\alpha\in w^{-1}\Phi^-\cap\Phi^-}\alpha.$$

\end{proof}

 It follows from the lemma that the complex $E_1^{\bullet, 0}$ from the spectral sequence (the Grothendieck-Cousin complex):
 $$  \mathcal{C}ous(\kappa) :  0 \rightarrow  \HH^0_{Z_0/Z_1}(FL, \mathcal{L}_\kappa) \rightarrow \cdots  \rightarrow \HH^d_{Z_d/Z_{d+1}}(FL, \mathcal{L}_\kappa) \rightarrow 0$$ computes $\mathrm{R}\Gamma(FL, \mathcal{L}_\kappa)$. Each group decomposes   $$\HH^p_{Z_p/Z_{p+1}}(FL, \mathcal{L}_\kappa) =  \bigoplus_{w \in W, \ell(w) = d-p} \HH^p_{w} (FL, \mathcal{L}_\kappa)$$ and we have established a precise formula for the weights of $T$ on each module.

There is a partial order on $X^\star(T)$ where $\lambda\geq 0$ if and only if $\lambda$ is a sum of positive roots.  Lemma \ref{lem-BGG-character} tells us that the weights occurring in $\HH^{d-\ell(w)}_w(FL,\mathcal{L}_\kappa)$ are exactly those which are $\leq(w^{-1}w_0)\cdot\kappa$.  In particular certain ``big weights'' will occur in as few of the terms of the Cousin complex as possible.  We begin with the following lemma:
\begin{lem}\label{lem-bw-definition}
Let $\nu\in X^\star(T)$ be such that $\nu+\rho$ is dominant.  Then the following conditions on a weight $\lambda\in X^\star(T)$ are equivalent:
\begin{enumerate}
\item $\lambda\not\leq w\cdot \nu$ for all $w\in W$ with $w\cdot\nu\not=\nu$.
\item $\lambda\not\leq s_\alpha\cdot \nu$ for all $\alpha\in\Delta$ with $s_\alpha\cdot\nu\not=\nu$.
\end{enumerate}
Moreover if we additionally assume that $\lambda\leq\nu$ then we have the further equivalent condition:
\begin{enumerate}
\setcounter{enumi}{2}
\item $\lambda=\nu-\sum_{\alpha\in\Delta}n_\alpha\alpha$ with $n_\alpha<\langle \alpha^\vee,\nu\rangle+1$ for all $\alpha\in\Delta$ with $\langle\alpha^\vee,\nu\rangle+1>0$.
\end{enumerate}
\end{lem}
\begin{proof}
Clearly the first condition implies the second.  For the converse, writing $w$ as a reduced product of simple reflections, there must be at least one factor $s_\alpha$ with $s_\alpha\cdot\nu\not=\nu$.  Then $w\geq s_\alpha$, and so $w\cdot\nu\leq s_\alpha\cdot\nu$ (see lemma \ref{lem-bruhat-inequality}) and hence $\lambda\not\leq s_\alpha\cdot\nu$ implies $\lambda\not\leq w\cdot\nu$.

The equivalence of the second and third points follows from the formula $s_\alpha\cdot\nu=\nu-(\langle\alpha^\vee,\nu\rangle+1)\alpha$.
\end{proof}
We say that a weight $\lambda$ satisfying the conditions of the proposition has big weight (with respect to $\nu$) and for a $T$-module $M$ which is a direct sum of its weight spaces we denote by $M^{bw(\nu)}$ the direct sum of its weight spaces corresponding to big weights.  We note that if $\nu+\rho$ is regular the last condition may be expressed as $\lambda>\nu-\sum_{\alpha\in\Delta}\langle(\alpha^\vee,\nu\rangle+1)\alpha$.

The set $X^\star(T)_\qq^+-\rho$ is a fundamental domain for the dotted action of $W$ on $X^\star(T)_\qq$, and so there is a unique $\nu\in X^\star(T)_\qq^+-\rho$ with $\nu\in W\cdot\kappa$.  We let $C(\kappa) = \{ w \in W\mid (w^{-1}w_0)\cdot\kappa =\nu\}$.  This set is a right torsor under $W_\nu=\{w\in W\mid w\cdot\nu=\nu\}$, and hence it is always nonempty and consists of a single element if and only if $\kappa+\rho$ is regular. For example, if $\kappa$ is dominant, $C(\kappa) = \{w_0\}$.

\begin{prop}\label{prop-miniBWB} The cohomology complex $\mathrm{R}\Gamma( FL, \mathcal{L}_\kappa)^{bw(\nu)}$ is a perfect complex of $\ocal_F$-modules of amplitude  $[  \mathrm{min}_{w \in C(\kappa)} (d-\ell(w)), \mathrm{max}_{w \in C(\kappa)} (d-\ell(w))]$. 
\end{prop}
\begin{proof} 
We will show that for $w\in W\setminus C(\kappa)$, the big weight part of the term in the Grothendieck-Cousin complex corresponding to $w$ vanishes.  Indeed, by lemma \ref{lem-BGG-character}, all the weights of $\HH^{d-\ell(w)}_w(FL,\mathcal{L}_\kappa)$ are $\leq(w^{-1}w_0)\cdot\kappa$, but $(w^{-1}w_0)\cdot\kappa\not=w'\cdot\nu\not=\nu$ and so $\HH^{d-\ell(w)}_w(FL,\mathcal{L}_\kappa)^{bw(\nu)}=0$.
\end{proof}

\begin{rem}
In particular when $\kappa+\rho$ is regular so that $C(\kappa)=\{w\}$ consists of a single element, we have that $\HH^\star(FL,\mathcal{L}_\kappa)^{bw(\nu)}=\HH_w^\star(FL,\mathcal{L}_\kappa)^{bw(\nu)}$.  We view this statement as a sort of analog of Coleman's classicality theorem, where the algebraic local cohomology groups $\HH_w^\star$ play the role of the overconvergent cohomology groups introduced in this paper, and the big weight condition plays the role of the small slope condition.
\end{rem}

We emphasize that the vanishing of Proposition \ref{prop-miniBWB} is characteristic independent.  Of course in characteristic zero, the classical Borel-Weil-Bott theorem gives a precise description of $\HH^\star(FL,\mathcal{L}_\kappa)$.  For the sake of completeness, we explain how the Borel-Weil-Bott theorem may be deduced from the computation of $\HH^\star(FL,\mathcal{L}_\kappa)$ via the Cousin complex and basic properties of the BGG category $\mathcal{O}$.
\begin{thm}[\cite{MR2015057}, 5.5, corollary]\label{thmBWB-proof} Let $\kappa \in X^\star(T)$.
\begin{enumerate}
\item If there exists no $w \in W$ such that $w \cdot \kappa $ is  dominant then $\HH^i(FL, \mathcal{L}_\kappa) \otimes F =0$ for all $i $.
\item If there exists $w \in W$ such that $w\cdot \kappa$ is dominant, then  there is a unique such $w$, and $\HH^i(FL, \mathcal{L}_\kappa) \otimes F=0$ if $\ell(w) \neq i$, while  $\HH^{\ell(w)}(FL, \mathcal{L}_\kappa) \otimes F$ is a highest weight $w\cdot\kappa$ representation.
\end{enumerate}
\end{thm}
\begin{proof} It suffices to prove the theorem after tensoring  with $\bar{F}$, an algebraic closure of $F$. For simplicity, we slightly change our notation and we assume that $F$ is algebraically closed for this proof.  There is a famous sub-category of the category of $\mathcal{U}(\mathfrak{g})$-modules, called the category $\ocal$. See \cite{MR2428237}, chapter 1 for the  definition and properties of the category $\ocal$. We recall a number of basic results concerning the category $\mathcal{O}$ that will be used in the argument.   The category $\ocal$ is abelian,  Artinian and Noetherian. The simple objects are indexed by weights $\lambda \in X^\star(T) \otimes F$ and denoted by $L_\lambda$.   If the module $L_\lambda$ is finite dimensional then $\lambda \in X^\star(T)^+$ is dominant, and  moreover  $L_\lambda$ arises from the highest weight $\lambda$ representation of $G$.  For all $\lambda \in X^\star(T) \otimes F$ we also denote by $M_\lambda = \mathcal{U}(\mathfrak{g}) \otimes_{\mathcal{U}(\mathfrak{b})} F(\lambda)$ the Verma module of weight $\lambda$. The simple module $L_\lambda$ is the unique simple quotient of $M_\lambda$. The Grothendieck group of $\ocal$, denoted by $K(\ocal)$, is the free module on the $[L_\lambda]$. We denote by $M \mapsto [M]$ the semi-simplification map from $\ocal$ to $K(\ocal)$.  In $K(\ocal)$ we have $[M_\lambda] = \oplus_{w\cdot \lambda \leq \lambda} a(w\cdot \lambda, \lambda) [L_{w\cdot \lambda}]$ with $a(\lambda, \lambda) =1$.  Since any element $M \in \ocal$ has diagonalizable $\mathfrak{t}$-action, we can associate to $M$ its formal character $\mathrm{ch} M$ which is an element of the group $\mathcal{X}$ of functions $X^\star(T) \otimes F \rightarrow \ZZ$.   The character is additive on short exact sequence and we get a map $K(\mathcal{O}) \rightarrow \mathcal{X}$, $[M] \mapsto \mathrm{ch}[M]$. Moreover, this last map is a group injection. We denote by $\mathcal{X}_{\ocal}$ its image.  Finally, any $\mathcal{U}(\mathfrak{g})$-module with diagonalizable $\mathfrak{t}$-action, and whose formal character  belongs to $\mathcal{X}_\ocal$ is an object of $\ocal$.

 It follows from lemma \ref{lem-BGG-character} that $\mathrm{ch}(\HH^p_{Z_p/Z_{p+1}}(FL, \mathcal{L}_\kappa) \otimes_{\ocal_F } F) = \oplus_{w, \ell(w) = p} \mathrm{ch}(M_{w \cdot \kappa})$.  The Grothendieck-Cousin complex $$0 \rightarrow  \HH^0_{Z_0/Z_1}(FL, \mathcal{L}_\kappa) \otimes F \rightarrow \cdots  \rightarrow \HH^d_{Z_d/Z_{d+1}}(FL, \mathcal{L}_\kappa) \otimes F \rightarrow 0$$  carries an action $\mathcal{U}(\mathfrak{g})$ by \cite{MR509802}, lemma 12.8, and  is therefore  a complex in the category $\ocal$.  
At this stage, we see that if none of the elements in the set $\{ w\cdot \kappa\}$ is dominant, then none of the  $\HH^p_{Z_p/Z_{p+1}}(FL, \mathcal{L}_\kappa) \otimes_{\ocal_F } F$ contains a finite dimensional subquotient. Otherwise there is a unique $w$ such that $w\cdot \kappa$ is dominant.  We see that for $p=\ell(w)$, $[\HH^p_{Z_p/Z_{p+1}}(FL, \mathcal{L}_\kappa) \otimes_{\ocal_F } F]$ has a unique finite dimensional constituent (with multiplicity one) equal to $[L_{w\cdot \kappa}]$. On the other hand $[\HH^i_{Z_i/Z_{i+1}}(FL, \mathcal{L}_\kappa) \otimes_{\ocal_F } F]$ has no finite dimensional constituent for $i \neq \ell(w)$. 
 The cohomology groups $\HH^i(FL, \mathcal{L}_\kappa) \otimes_{\ocal_F } F$ are finite dimensional because $FL$ is proper. If none of the elements in the set $\{ w\cdot \kappa\}$ is dominant, the cohomology is therefore trivial.  If there is a unique $w$ such that $w\cdot \kappa$ is dominant, we see that $\HH^i(FL, \mathcal{L}_\kappa) \otimes_{\ocal_F } F =0$ if $i \neq \ell(w)$, and $\HH^{\ell(w)}(FL, \mathcal{L}_\kappa) \otimes_{\ocal_F } F = L_{w\cdot \kappa}$. 
\end{proof}

\subsection{Analytic geometry}\label{section-analytic-geometry}
If $S \rightarrow \Spec~\ocal_F$ is a finite type morphism of schemes, we let $\mathcal{S} = S  \times_{\Spec~\ocal_F} \Spa (F, \ocal_F)$ be the associated analytic adic space and ${S}_k = S \times_{\Spec~\ocal_F} \Spec~k
$ be the special fiber. One can also consider $\mathcal{S}^{an}$, the analytification of the scheme $S \times_{\Spec~\ocal_F} \Spec~F$, and there is a map $\mathcal{S} \rightarrow \mathcal{S}^{an}$ which is an isomorphism when $S$ is proper (see \cite{MR1306024}, section 4). There is a  continuous specialization map  $\mathrm{sp}_{\mathcal{S}} : \mathcal{S} \rightarrow {S}_k$ and the preimage of a subset $U \subset {S}_k$ is denoted by $\mathrm{sp}_{\mathcal{S}}^{-1}(U)$.   If $U$ is a locally closed subset of $S_k$, we let $]U[_{\mathcal{S}}$ be the interior of $\mathrm{sp}^{-1}(U)$. This is an adic space, called the tube of $U$ (see \cite{Berthelot}). The difference between $\mathrm{sp}_{\mathcal{S}}^{-1}(U)$ and $]U[_{\mathcal{S}}$ consists only of certain higher rank points. The tube $]U[_{\mathcal{S}}$ is the adic space associated to a  ``classical'' rigid space, while $\mathrm{sp}_{\mathcal{S}}^{-1}(U)$ is not in general. 

\subsubsection{The Iwahori decomposition}  Let $\mathcal{G}$ be the quasi-compact adic space associated to $G$ and let $\mathrm{Iw}  = ] {B}_k[_{\mathcal{G}}$ be the Iwahori subgroup of $\mathcal{G}$.

For any root $\alpha \in \Phi$, we have an algebraic root space $U_\alpha \rightarrow \Spec~\ocal_F$. We let $\mathcal{U}_\alpha$ be the corresponding quasi-compact  adic space (isomorphic to a unit ball) and we let $\mathcal{U}_\alpha^o = ] \{1\}[_{\mathcal{U}_\alpha}$ be the tube of  the identity element (isomorphic to an increasing union of balls of radii $r<1$).   We also let $\mathcal{U}_\alpha^{an}$ be the analytification of $U_\alpha$ (isomorphic to the affine line). 

 The following result gives a strong form of the Iwahori decomposition.

\begin{prop}\label{prop-superIwahoridecomposition} Let $\alpha_1, \cdots, \alpha_n$ be an enumeration of the roots in $\Phi$. 
The product map
$$\mathcal{T} \times \prod \mathcal{U}_{\alpha_i}^{\star_i} \rightarrow \mathrm{Iw}$$
is an isomorphism of analytic adic spaces,
where $\mathcal{U}_{\alpha_i}^{\star_i} = \mathcal{U}_{\alpha_i}$ if $\alpha_i \in \Phi^+$ and   $\mathcal{U}_{\alpha_i}^{\star_i} = \mathcal{U}_{\alpha_i}^o$ if $\alpha_i \in \Phi^-$. 
\end{prop}

\begin{rem} The existence of a product decomposition $\mathcal{T}(K) \times \prod \mathcal{U}_{\alpha_i}^{\star_i}(K) \rightarrow \mathrm{Iw}(K)$  for $K$ a discretely valued field is a consequence of Bruhat-Tits theory \cite{MR546588}, sect. 3.1.1.
\end{rem}
\begin{proof} We let $\mathfrak{Iw}$ be the formal group scheme obtained by completing the group $G$ along the closed subscheme $B_{k}$. For each root $\alpha$, we let $\mathfrak{U}_\alpha = \Spf(\ocal_F \langle T \rangle)$ be the formal completion of $U_\alpha$ along its special fiber, and we let $\mathfrak{U}^o_\alpha = \Spf(\ocal_F[[T]])$ be the formal completion of $\mathfrak{U}_{\alpha}$ at the identity. We also let $\mathfrak{T}$ be the formal completion of $T$ along its special fiber. 

We consider the map of formal schemes: $\mathfrak{T} \times \prod \mathfrak{U}_{\alpha_i}^{\star_i} \rightarrow \mathfrak{Iw}$ which yields the map of the proposition after passing to the generic fiber.  We will show that this map of formal schemes is an isomorphism.  Passing to underlying reduced schemes, we obtain the isomorphism $T_k \times \prod_{ \alpha \in \Phi^+} U_{\alpha,k} \rightarrow B_k$.  Hence it suffices to show that it is formally \'etale.

We claim that this map is formally \'etale. Since it is an isomorphism on the associated reduced schemes because $T_k \times \prod_{ \alpha \in \Phi^+} U_{\alpha,k} \rightarrow B_k$ is an isomorphism, we deduce that  the map is  an isomorphism. The associated map on the generic fiber (which is the map of the lemma) is an isomorphism. 
We are left to prove that the map is formally \'etale. Let $k'$ be a finite field extension of $k$. Let $(t, u_{\alpha_i} ) \in \mathfrak{T} \times \prod \mathfrak{U}_{\alpha_i}^{\star_i} (k')$. Note that  $u_{\alpha_i} = 1$ if $\alpha_i \in \Phi^-$. The map on Zariski tangent spaces is given by: 
$$ (t(1+ \epsilon T) , u_{\alpha_i}(1 + \epsilon U_{\alpha_i})_{1 \leq i \leq  n}) \mapsto t(1+\epsilon T) \prod_i u_{\alpha_i}(1 + \epsilon U_{\alpha_i})$$
 where $(T, (U_{\alpha_i})_{1 \leq i \leq  n}) \in \mathrm{Lie} (T_k) \otimes_k k' \oplus_{i} \mathrm{Lie} (U_{\alpha_i}) \otimes_k k'$, and there is an equality: 
$$ t(1+\epsilon T) \prod_i u_{\alpha_i}(1 + \epsilon U_{\alpha_i}) = t\prod_i u_{\alpha_i} (1 + \epsilon \mathrm{Ad}((\prod_i u_{\alpha_i})^{-1}) ( T)  + \sum_{i= 2}^{n+1}\mathrm{Ad} ( (\prod_{k=i}^n u_{\alpha_k})^{-1}) U_{\alpha_{i-1}}).$$ 

Therefore, if we identify the tangent space of $\mathfrak{T} \times \prod \mathfrak{U}_{\alpha_i}^{\star_i} $ at $(t, u_{\alpha_i} )$  and the tangent space  of $\mathfrak{Iw}$ at $t \prod  u_{\alpha_i}$ with $$\mathrm{Lie} (G_k) \otimes k' = \mathrm{Lie} (T_k) \otimes_k k' \oplus_{i} \mathrm{Lie} (U_{\alpha_i}) \otimes_k k',$$ the map on tangent spaces is given by the endomorphism: 
$$ (T, (U_{\alpha_i})_{1 \leq i \leq  n}) \mapsto \mathrm{Ad}((\prod_i u_{\alpha_k})^{-1}) ( T) +  \sum_{i= 2}^{n+1}\mathrm{Ad} ( (\prod_{k=i}^n u_{\alpha_k})^{-1})(U_{\alpha_{i-1}}).$$
Observe that  $(\prod_{k=i}^n u_{\alpha_k})^{-1} \in U(k')$ for all $1 \leq i \leq n$. 
By Chevalley's commutativity relations (\cite{MR3616493}, chapter 6), we know that for any $u \in U(k')$, any $\alpha \in \Phi \cup \{0\}$ and any $v \in \mathrm{Lie}( G_k)_{\alpha} \otimes k'$, $Ad(u).v = v + w$ where $w \in \bigoplus_{\alpha' > \alpha} \mathrm{Lie} (G_k)_{\alpha'}\otimes k'$. A simple inductive argument proves that the map on tangent spaces is an isomorphism.
\end{proof}

\subsubsection{Tubes of Bruhat cells}

For any $w\in\WM$, we now consider the tube of the Bruhat cell  $]C_{w,k}[_{\mathcal{FL}} = \mathcal{P} \backslash \mathcal{P}w \mathrm{Iw}$, as well as the tubes $]X_{w,k}[_{\mathcal{FL}}$ and $]Y_{w,k}[_{\mathcal{FL}}$. It follows from lemma \ref{lem-oppositeSchubert} that $\mathcal{X}^w \hookrightarrow ]Y_{w,k}[_{\mathcal{FL}}$. 

We can now give a very precise description of the tubes of Bruhat cells.

 \begin{coro}\label{coro-description-bruhat-cell} For any $w \in \WM$,  and for any ordering of the roots in $\Phi$, we have an isomorphism of analytic spaces: 
 \begin{eqnarray*}
 \prod_{\alpha \in  (w^{-1} \Phi^{-,M}) \cap \Phi^+} \mathcal{U}_{\alpha} \times \prod_{ \alpha \in  (w^{-1} \Phi^{-,M}) \cap \Phi^-} \mathcal{U}_{\alpha}^o &\rightarrow&   ]C_{w,k}[_{\mathcal{FL}}  \\
 (u_{\alpha})_{\alpha \in  w^{-1}\Phi^{-,M}} & \mapsto & w\prod_\alpha u_{\alpha}
 \end{eqnarray*}
 where the product $\prod_\alpha u_{\alpha}$ is taken according to our fixed ordering. 
\end{coro} 
\begin{proof} This follows easily from proposition \ref{prop-superIwahoridecomposition}.
\end{proof}
 
 \subsubsection{Analytic subspaces of $ \mathcal{U}_w^{an}$}
We now consider $\mathcal{U}_w^{an}$, the analytification of $U_w$.  This open subset of $\mathcal{FL}$ contains $]C_{w,k}[_{\mathcal{FL}}$,  and we have an isomorphism  of analytic spaces (for any ordering of the roots): 
 \begin{eqnarray*}
 \prod_{\alpha \in  (w^{-1} \Phi^{-,M}) } \mathcal{U}^{an}_{\alpha} &\rightarrow&  \mathcal{U}_w^{an}  \\
 (u_{\alpha})_{\alpha \in  w^{-1}\Phi^{-,M}} & \mapsto & w\prod_\alpha u_{\alpha}
 \end{eqnarray*}
 
 We now introduce certain quasi-Stein subspaces of $\mathcal{U}_w^{an}$ as well as some partial closures of them. In order to do this we introduce a little bit of language. 
 
 If $T$ is a topological space and $S \subseteq T$ is a subset, we let as usual $\overline{S}$ be the closure of $S$ in $T$.  We will use repeatedly the property that in a spectral space, the closure of a pro-constructible set is the set of all its specializations (see \cite{scholze2017etale}, lemma 2.4).   
 
 Let $\mathcal{X}$ be an adic space. Using the Yoneda lemma,  we identify $\mathcal{X}$ with a functor on complete Huber pairs $(A,A^+)$.  Let   $T$ be a subset of $\mathcal{X}$. In general $T$ is not an adic space but using the Yoneda point of view,   we can attach to $T$ a subfunctor of $\mathcal{X}$:  we let $T((A,A^+))$ be the subset of  $\mathcal{X}((A,A^+))$ of morphisms $\Spa (A,A^+) \rightarrow \mathcal{X}$ whose image factors through $T$. We observe that this functor is completely determined by its value on fields  $(K,K^+)$.   In general, we say that a subfunctor $F$ of $\mathcal{X}$ is a subset of $\mathcal{X}$ if it is the functor attached to a subset of $\mathcal{X}$. 
 
 \begin{ex} Let $\mathbb{B}(0,1)$ and $\mathbb{B}^o(0,1)$ be respectively the quasi-compact unit ball and the open unit ball of radius $1$  over $\Spa(F,\ocal_F)$. These are open subsets of the affine line $(\mathbb{A}^1)^{an}$.  For any $(F,\ocal_F)$-complete Huber pair $(A,A^+)$, we have $(\mathbb{A}^1)^{an}((A,A^+)) = A $, $\mathbb{B}(0,1)((A,A^+))= A^+$, $\mathbb{B}^o(0,1)((A,A^+)) = \{ a \in A^+,~\exists n \in \ZZ_{\geq 1}, a^n \in pA^+\}$. The closed subset  $\overline{\mathbb{B}(0,1)}$ is the union of $\mathbb{B}(0,1)$ together with a rank two point whose maximal generalization is the Gauss point and which points towards infinity. This is still an adic space equal to $\Spa ( F\langle X \rangle, \ocal_F + \mathfrak{m}_{\ocal_F} X + \mathfrak{m}_{\ocal_F}X^2 + \cdots )$.  We have $\overline{\mathbb{B}(0,1)}((A,A^+)) = \{ a \in A, \forall n \in \ZZ_{\geq 1}, pa^n \in A^+\}$. The closed subset  $\overline{\mathbb{B}^o(0,1)}$ is the union of $\mathbb{B}^o(0,1)$ together with a rank two point whose maximal generalization is the Gauss point and which points towards the center of the disc. This is not adic space.  We have $\overline{\mathbb{B}^o(0,1)}((K,K^+)) = K^{++} = \{ a \in K, \vert a \vert < 1\}$. 
 \end{ex}

   We identify each $\mathcal{U}_\alpha$ with the  unit ball of center $0$ and coordinate $u_\alpha$ with its additive group law (the coordinate $u_\alpha$ is well defined up to multiplication by a unit). For all  $m \in \qq \cup \{-\infty\} $ and all $\alpha \in \Phi$, we let $\mathcal{U}_{\alpha, m} = \{ \vert. \vert  \in \mathcal{U}^{an}_{\alpha}, ~\vert u_\alpha \vert \leq \vert p^m \vert \}$ and $\mathcal{U}^o_{\alpha, m} = \cup_{m' >m} \mathcal{U}_{\alpha,m'}$.  For all $m,n \in \qq \cup \{-\infty\}$, we let $]C_{w,k}[_{m,n}$ be the image of $$ w\prod_{\alpha \in  (w^{-1} \Phi^{-,M}) \cap \Phi^-} \mathcal{U}_{\alpha,m}^o \times \prod_{\alpha \in  (w^{-1} \Phi^{-,M}) \cap \Phi^+} \mathcal{U}_{\alpha,n} \rightarrow \mathcal{U}_w^{an},$$
 For all $m,n \in \qq$, we let $]C_{w,k}[_{\overline{m},n}$ be the image of $$ w\prod_{\alpha \in  (w^{-1} \Phi^{-,M}) \cap \Phi^-} \overline{\mathcal{U}_{\alpha,m}^o} \times \prod_{\alpha \in  (w^{-1} \Phi^{-,M}) \cap \Phi^+} \mathcal{U}_{\alpha,n} \rightarrow \mathcal{U}_w^{an},$$
 we let $]C_{w,k}[_{m,\overline{n}}$ be the image of $$ w\prod_{\alpha \in  (w^{-1} \Phi^{-,M}) \cap \Phi^-} \mathcal{U}_{\alpha,m}^o \times \prod_{\alpha \in  (w^{-1} \Phi^{-,M}) \cap \Phi^+} \overline{\mathcal{U}_{\alpha,n}} \rightarrow \mathcal{U}_w^{an},$$
 and  we let $]C_{w,k}[_{\overline{m},\overline{n}}$ be the image of $$ w\prod_{\alpha \in  (w^{-1} \Phi^{-,M}) \cap \Phi^-}\overline{ \mathcal{U}^o_{\alpha,m}} \times \prod_{\alpha \in  (w^{-1} \Phi^{-,M}) \cap \Phi^+} \overline{\mathcal{U}_{\alpha,n}} \rightarrow \mathcal{U}_w^{an}.$$

Clearly, $]C_{w,k}[ = ]C_{w,k}[_{0,0}$, $]C_{w,k}[_{m,n} \subseteq ]C_{w,k}[$ if  $m, n \geq 0$, and moreover then we have $]C_{w,k}[_{m,n} = ]C_{w,k}[_{m,0} \cap ]C_{w,k}[_{0,n}$, $]C_{w,k}[_{\overline{m},n} = ]C_{w,k}[_{\overline{m},0} \cap ]C_{w,k}[_{0,n}$   and $]C_{w,k}[_{m,\overline{n}} = ]C_{w,k}[_{m,0} \cap ]C_{w,k}[_{0,\overline{n}}$.

\begin{rem} In the formulas defining the sets $]C_{w,k}[_{m,n}$ and their partial closures,  the image is taken in the sheaf theoretic sense. If the spaces at hand are adic spaces, this is also the images of the corresponding morphisms of adic spaces. 
\end{rem} 

\begin{rem} In the above formulas, we make the product for any ordering of the roots in $(w^{-1} \Phi^{-,M}) \cap \Phi^+$ or  $(w^{-1} \Phi^{-,M}) \cap \Phi^-$. See \cite{MR3616493}, lemma 17 for a justification that the order doesn't matter. We also point out that in our applications to Shimura varieties the unipotent radical of $P$ is abelian so that the root groups $U_\alpha$ for $\alpha \in w^{-1} \Phi^{-,M}$ commute with each other.
\end{rem}

\subsubsection{Orbits of Cells}\label{section-orbitsofcells}
 We can give a more group theoretic description of certain of  the above sets.  We introduce some subgroups of $\mathcal{G}$.
 
 For $m\in\qq_{\geq 0}$ we let $\mathcal{G}_{B,m}$ (resp. $\mathcal{G}_{U,m}$, $\mathcal{G}_{\overline{B},m}$, $\mathcal{G}_{\overline{U},m}$) for the affinoid subgroup of $\mathcal{G}$ of elements reducing to $B$ (resp. $U$, $\overline{B}$, $\overline{U}$) mod $p^m$.  We also define $\mathcal{G}_{B,m}^\circ=\cup_{m'>m}\mathcal{G}_{B,m'}$, $\mathcal{G}_{U,m}^\circ=\cup_{m'>m}\mathcal{G}_{U,m'}$.
 
 Then for all $m,n\in\qq_{\geq 0}$ we let $\mathcal{G}_{m,n}=\mathcal{G}_{B,m}^\circ\cap\mathcal{G}_{\overline{B},n}$.  In particular we note that $\mathcal{G}_{0,0}=\mathrm{Iw}$.  We also let $\mathcal{G}^1_{m,n}=\mathcal{G}_{U,m}^\circ\cap\mathcal{G}_{\overline{U},n}$.
 
 We would also like to introduce some ``partial closures'' of these groups. In general these won't be adic spaces, but simply certain subsets of  $\mathcal{G}^{an}$.  For all $m,n\in\qq_{\geq 0}$ we let $\mathcal{G}_{\overline{m},n}=\overline{\mathcal{G}_{B,m}^\circ}\cap\mathcal{G}_{\overline{B},n}$ and $\mathcal{G}^1_{\overline{m},n}=\overline{\mathcal{G}_{U,m}^\circ}\cap\mathcal{G}_{\overline{U},n}$, where the closures are taken inside $\mathcal{G}$.  We note that $\mathcal{G}_{\overline{0},0}=\mathrm{sp}^{-1}(B_k)$, the closure of $\mathrm{Iw}$ in $\mathcal{G}$.  For all $m, n \in \qq_{\geq 0}$, we let $\mathcal{G}_{m, \overline{n}}$ and $\mathcal{G}^1_{m, \overline{n}}$ be set of all specializations of $\mathcal{G}_{m,n}$ and $\mathcal{G}^1_{m,n}$ respectively in $\mathcal{G}^{an}$. For all $m,n\in\qq_{\geq0}$ with $n>0$ we also observe that  $\mathcal{G}_{m,\overline{n}}=\mathcal{G}_{B,m}^\circ\cap\overline{\mathcal{G}_{\overline{B},n}}$ and $\mathcal{G}^1_{m,\overline{n}}=\mathcal{G}_{U,m}^\circ\cap\overline{\mathcal{G}_{\overline{U},n}}$ where  the closures are taken inside $\mathcal{G}$. 
 
 \begin{lem} The subsets $\mathcal{G}_{m,n}$, $\mathcal{G}_{\overline{m}, n}$ and $\mathcal{G}_{m, \overline{n}}$, $\mathcal{G}^1_{m,n}$, $\mathcal{G}^1_{\overline{m}, n}$ and $\mathcal{G}^1_{m, \overline{n}}$ of $\mathcal{G}^{an}$ are subgroups (in the sense that their functor of points is a subgroup of $\mathcal{G}^{an}$).
 \end{lem}
 \begin{proof} It is clear that  $\mathcal{G}_{B,m}$, $\mathcal{G}_{U,m}$, $\mathcal{G}_{\overline{B},n}$, $\mathcal{G}_{\overline{U},n}$ are groups. Therefore, the $\mathcal{G}_{m,n}$ are groups.  We deduce that the $\mathcal{G}_{m, \overline{n}}$ are groups. Indeed, $\mathcal{G}_{m, \overline{n}}((K,K^+)) = \mathcal{G}_{m, {n}}(K, K^0)$ for any field (where $K^0$ is the ring of power bounded elements in $K$). We also check that  $\mathcal{G}_{\overline{m}, 0}$ are groups. Indeed, $\mathcal{G}_{\overline{m}, 0}((K,K^+)) = \{ g \in G(K^+), g \mod p^m K^{++} \in B(K^+/p^mK^{++})\}$. We finally deduce that $\mathcal{G}_{\overline{m}, n}$ is a subgroup.
 \end{proof}
 
Note that the groups $\mathcal{G}_{m,n}$, $\mathcal{G}_{\overline{m},n}$, and $\mathcal{G}_{m,\overline{n}}$ all have the same rank 1 points.  Moreover they all have Iwahori decompositions:
\begin{prop}\label{prop-superIwahoridecomposition2} Let $w \in W$.  Then for $m,n\in\qq_{\geq 0}$ and for $\mathcal{G}'$ one of $\mathcal{G}^\star_{m,n}$, $\mathcal{G}^\star_{\overline{m},n}$, $\mathcal{G}^\star_{m,\overline{n}}$ with $\star\in\{\emptyset,1\}$ the product map 
$$\mathcal{T}\cap\mathcal{G}' \times \prod_{ \alpha_i \in w^{-1} \Phi^+} \mathcal{U}_{\alpha_i}^{\star_i}  \prod_{ \alpha_i \in w^{-1} \Phi^-} \mathcal{U}_{\alpha_i}^{\star_i} \rightarrow \mathcal{G}'$$
is an isomorphism, where if $\alpha_i\in \Phi^+$ then $\mathcal{U}_{\alpha_i}^{\star_i}$ is $\mathcal{U}_{\alpha_i,n}$ (resp. $\overline{\mathcal{U}_{\alpha_i,n}}$) if $\mathcal{G}'=\mathcal{G}_{m,n}^\star,\mathcal{G}_{\overline{m},n}^\star$ (resp. $\mathcal{G}'=\mathcal{G}_{m,\overline{n}}^\star$) while if $\alpha_i\in\Phi^-$ then $\mathcal{U}_{\alpha_i}^{\star_i}$ is $\mathcal{U}^o_{\alpha_i,m}$ (resp. $\overline{\mathcal{U}^o_{\alpha_i,m}}$) if $\mathcal{G}'=\mathcal{G}_{m,n}^\star,\mathcal{G}_{m,\overline{n}}^\star$. (resp. $\mathcal{G}'=\mathcal{G}_{\overline{m},n}^\star$.)
\end{prop}
\begin{proof} We first give the argument for $\mathcal{G}_{\overline{0}, 0}$. Since $\mathcal{G}_{\overline{0}, 0} = \mathrm{sp}^{-1}(B_k)$, and $B_k \hookrightarrow w^{-1} \overline{B_k} B_k w$, it follows that  $\mathcal{G}_{\overline{0},0} \hookrightarrow \mathcal{T} \times \prod_{\alpha \in w^{-1} \Phi^-} \mathcal{U}_\alpha \prod_{\alpha \in w^{-1} \Phi^+} \mathcal{U}_\alpha = \mathrm{sp}^{-1}(w^{-1} \overline{B_k} B_k w)$. We now let $g \in \mathcal{G}_{\overline{0},0}(K,K^+)$ for a field $K$ and valuation ring $K^+ \subseteq K$.  Let $K^{++}$ be the maximal ideal of $K^+$.  By definition,  $g \in G(K^+)$ and its image $\overline{g}$ in $G(K^+/K^{++})$ lies in $B$. Let $g = t \prod_{\alpha \in w^{-1} \Phi^+} u_\alpha \prod_{\alpha \in w^{-1}} u_\alpha$. We have $\overline{g} =\overline{ t} \prod_{\alpha \in w^{-1} \Phi^+}\overline{ u_\alpha} \prod_{\alpha \in w^{-1}} \overline{u_\alpha}$ and we find that $\overline{u_\alpha} = 1$ if $\alpha \in \Phi^-$ by unicity of the decomposition. The case of $\mathcal{G}_{0,0}$ has already been treated in proposition \ref{prop-superIwahoridecomposition}. The case of $\mathcal{G}_{0, \overline{0}}$ follows. The general case is left to the reader. 
\end{proof}
\begin{rem} For the group $\mathcal{G}_{0,0} = \mathrm{Iw}$, the decomposition holds for any ordering of the roots $\alpha \in \Phi$ by proposition \ref{prop-superIwahoridecomposition}. We don't know if this property holds for $\mathcal{G}_{\overline{0},0}$ for example. 
\end{rem}

\begin{coro} We have that for all $m,n\in\qq_{\geq 0}$, $]C_{w,k}[_{m,n} = \mathcal{P} \backslash \mathcal{P}w \mathcal{G}_{m,n}=\mathcal{P} \backslash \mathcal{P}w \mathcal{G}_{m,n}^1$, $]C_{w,k}[_{\overline{m},n} = \mathcal{P} \backslash \mathcal{P}w \mathcal{G}_{\overline{m},n}=\mathcal{P} \backslash \mathcal{P}w \mathcal{G}_{\overline{m},n}^1$ and , $]C_{w,k}[_{m,\overline{n}} = \mathcal{P} \backslash \mathcal{P}w \mathcal{G}_{m,\overline{n}}=\mathcal{P} \backslash \mathcal{P}w \mathcal{G}_{m,\overline{n}}^1$.
\end{coro}
\begin{proof} We only prove the identity  $]C_{w,k}[_{m,n} = \mathcal{P} \backslash \mathcal{P}w \mathcal{G}_{m,n}$, the other identities are proved similarly. 
By proposition \ref{prop-superIwahoridecomposition2}, we have that 
\begin{eqnarray*}
\mathcal{P} \backslash \mathcal{P}w \mathcal{G}_{m,n} & =&  w . w^{-1} \mathcal{P} w \cap \mathcal{G}_{m,n} \backslash \mathcal{G}_{m,n} \\
&=& w.  \prod_{\alpha_i \in w^{-1}\Phi^{-}_M} \mathcal{U}^{\star_i}_{\alpha_i} \prod_{\alpha_i \in w^{-1}\Phi^+}   \mathcal{U}_{\alpha_i}^{\star_i}  \backslash  \prod_{ \alpha_i \in w^{-1} \Phi^+} \mathcal{U}_{\alpha_i}^{\star_i}  \prod_{ \alpha_i \in w^{-1} \Phi^-} \mathcal{U}_{\alpha_i}^{\star_i}  \\
&=& w \prod_{ \alpha \in w^{-1} \Phi^{-,M}}  \mathcal{U}_{\alpha_i}^{\star_i}
\end{eqnarray*}
where if $\alpha_i\in \Phi^+$ then $\mathcal{U}_{\alpha_i}^{\star_i}$ is $\mathcal{U}_{\alpha_i,n}$ and if $\alpha_i\in\Phi^-$ then $\mathcal{U}_{\alpha_i}^{\star_i}$ is $\mathcal{U}^o_{\alpha_i,m}$.
\end{proof}

\begin{lem}\label{lem-normalsubgroups} For all $m\in\ZZ_{\geq0}$ and $k\in\ZZ_{\geq 0}$, the groups  $\mathcal{G}^1_{m+k,k}$, $\mathcal{G}^1_{\overline{m+k},k}$ and $\mathcal{G}^1_{m+k, \overline{k}}$   are normalized by $\mathcal{G}_{m,0}$.
\end{lem}

\begin{proof} The case $k=0$ is trivial. We assume that $k \geq 1$. We can find a closed embedding $G \hookrightarrow \mathrm{GL}_r$ and a Borel $B_{\mathrm{GL_r}}$ of $\mathrm{GL}_r$ with the property that $B = G \cap B_{GL_n}$. Indeed, first consider a faithful representation of $G$ into $\mathrm{GL}_r$, then consider the action of $B$ on the flag variety of Borels of $\mathrm{GL}_r$.  This action of a solvable group on a proper scheme must have a fixed point $B_{GL_r}$, and then $B\subseteq G\cap B_{GL_r}$. 
But then we must have $B=G\cap B_{GL_r}$ as the later is a solvable subgroup of $G$ containing $B$. Therefore, the problem is reduced to the case of the group $\mathrm{GL}_r$. 
We now consider certain sub-algebras of the  algebra of $r \times r$ matrices  $\mathcal{M}^{an}_r$. For all $s \in \qq$, we let $\mathbb{B}(0,s)$ the (quasi-compact) ball of center $0$ and radius $s$, and $\mathbb{B}^{o}(0,s) = \cup_{s' < s} \mathbb{B}(0,s')$ be the ``open'' ball of radius $s$ (which is a Stein space). For any $m \geq 0$, we let: 
$$\mathrm{Lie}_{m,0}(S, S^+) = $$ $$ \{ (a_{i,j}) \in \mathcal{M}^{an}_r(S,S^+),~a_{i,j} \in \mathbb{B}^{o}(0,\vert p^m\vert )(S,S^+)~\textrm{if}~i\geq j,  a_{i,j} \in {\mathbb{B}(0,1)}(S,S^+)~\textrm{if}~i\leq j\},$$ $$\mathrm{Lie}_{\overline{m},0}(S, S^+) = $$ $$ \{ (a_{i,j}) \in \mathcal{M}^{an}_r(S,S^+), a_{i,j} \in \overline{\mathbb{B}^{o}(0,\vert p^m\vert)}(S,S^+)~\textrm{if}~i \geq j,   a_{i,j} \in {\mathbb{B}(0,1)}(S,S^+)~\textrm{if}~i\leq j\},$$
$$ \mathrm{Lie}_{{m},\overline{0}}(S, S^+) = $$ $$ \{ (a_{i,j}) \in  \mathcal{M}^{an}_r(S,S^+),~a_{i,j} \in \mathbb{B}^{o}(0,\vert p^m\vert)(S,S^+)~\textrm{if}~i \geq j\,  a_{i,j} \in \overline{\mathbb{B}(0,1)}(S,S^+)~\textrm{if}~i\leq j\}.$$
We claim that these algebras are stable under the adjoint action of $\mathcal{G}_{m,0}$. Since $\mathrm{Lie}_{m,0}$ is the Lie algebra of $\mathcal{G}^1_{m,0}$, and $\mathcal{G}_{m,0}$ normalizes $\mathcal{G}^1_{m,0}$, this case follows easily. The other cases are elementary to check by hand. 
A typical element of $\mathcal{G}^1_{m+k,k}$ (resp. $\mathcal{G}^1_{\overline{m+k},k}$, resp. $\mathcal{G}^1_{{m+k},\overline{k}}$) writes $(1+p^k g )$ where $g \in \mathrm{Lie}_{m,0}$ (resp. $\mathrm{Lie}_{\overline{m},0}$, resp. $\mathrm{Lie}_{m,\overline{0}}$).  For $h \in \mathcal{G}_{m,0}$, we have $h(1+p^kg)h^{-1} = 1 + p^k \mathrm{Ad}(h).g$. 
\end{proof}

\begin{lem}\label{lem-nice-disjoint-union}  Let $w \in \WM$.  Let $m \geq 0$, and let $K_{p} \subset \mathcal{G}_{m,0}$ be a profinite subgroup. For all $k \geq 0$,   the  sets  $$\textrm{$]C_{w,k}[_{m+k,k}K_{p}$,  $]C_{w,k}[_{\overline{m+k},k}K_{p} $ and $]C_{w,k}[_{{m+k}, \overline{k}}K_{p} $}$$ are a finite disjoint union of translates of the form  $$\textrm{$]C_{w,k}[_{m+k,k}h$, $]C_{w,k}[_{\overline{m+k},k}h $, and respectively $]C_{w,k}[_{{m+k}, \overline{k}}h $ for $h \in  K_p$. }$$ 
\end{lem}

 \begin{proof} We prove the first statement, the others are identical.  For $h\in K_p$ we have $]C_{w,k}[_{m+k,k}h=\mathcal{P}\backslash\mathcal{P}w\mathcal{G}_{m+k,k}^1h=\mathcal{P}\backslash\mathcal{P}wh\mathcal{G}_{m+k,k}^1$ by Lemma \ref{lem-normalsubgroups}.
This proves that if  for $h,h'\in K_p$, $]C_{w,k}[_{m+k,k}h  \cap  ]C_{w,k}[_{m+k,k} h' \neq \emptyset$, then $]C_{w,k}[_{m+k,k}h  =  ]C_{w,k}[_{m+k,k} h'$.  Therefore, $]C_{w,k}[_{m+k,k} K_p$ is a disjoint union of translates. This disjoint union is finite because $\mathcal{G}_{m+k,k} \cap K_{p}$ is of finite index in $K_{p}$.  
 \end{proof}
 
 \bigskip

 We now consider certain  intersections. 
 
 \begin{lem}\label{lem-nice-intersections} Let $w \in \WM$.  Let $m, n\in \ZZ_{\geq 0}$ and let $K_p$ be a subgroup of $\mathcal{G}_{m,0}$.  Then we have:
 \begin{eqnarray*}
 ]C_{w,k}[_{m,0} K_{p} \cap ]C_{w,k}[_{0,n} K_{p} &= &]C_{w,k}[_{m,n} K_p,\\
 ]C_{w,k}[_{m, \overline{0}} K_{p} \cap ]C_{w,k}[_{0,\overline{n}} K_{p} &=& ]C_{w,k}[_{m,\overline{n}} K_p,\\  
 ]C_{w,k}[_{\overline{m},0} K_{p} \cap ]C_{w,k}[_{\overline{0},n} K_{p} & = & ]C_{w,k}[_{\overline{m},n} K_p.
 \end{eqnarray*}
 \end{lem}
 
\begin{proof} We only check the first statement as the others are identical.  As $K_p\subset\mathcal{G}_{m,0}$, $]C_{w,k}[_{m,0} K_{p} = ]C_{w,k}[_{m,0}$. Let $x \in  ]C_{w,k}[_{m,0} K_{p} \cap ]C_{w,k}[_{0,n} K_{p}$. Then, there exists $h \in K_p$, $xh \in ]C_{w,k}[_{m,0} \cap ]C_{w,k}[_{0,n} = ]C_{w,k}[_{m,n}$. The reverse inclusion is trivial.
\end{proof}

 \subsubsection{Intersections  and unions of  cells} We prove a few more results concerning the intersections and unions of various subsets of $\mathcal{FL}$ we have introduced so far. 
  \begin{lem}\label{lem-writing-intersections} For any integer $r$, we have:
 \begin{eqnarray*}
 \mathcal{FL} & =& ] \cup_{w, \ell(w)  \geq r} C_{w,k} [ \coprod \overline{] \cup_{w, \ell(w) \leq r-1} C_{w,k}[} \\
  &=& \overline{] \cup_{w, \ell(w)  \geq r} C_{w,k} [ }\coprod {] \cup_{w, \ell(w) \leq r-1} C_{w,k}[}
  \end{eqnarray*}
  \end{lem}
  
\begin{proof} The first equality is obtained by applying $\mathrm{sp}^{-1}$ to the decomposition of the special fiber:
$$FL_k = \cup_{w, \ell(w)  \geq r} C_{w,k}  \coprod  \cup_{w, \ell(w) \leq r-1} C_{w,k}$$
For the second equality, first observe that both $ \overline{] \cup_{w, \ell(w)  \geq r} C_{w,k} [ }$ and ${] \cup_{w, \ell(w) \leq r-1} C_{w,k}[}$ are stable under specialization and generalization. 
Since the disjoint union contains all rank one points of $\mathcal{FL}$, it contains all of them.  Since the intersection contains no rank one points, it is empty.
\end{proof} 

\begin{lem} \label{topological-lemma-easy}
Let $w, w' \in \WM$. Assume that $\ell(w) \leq \ell(w')$ and $w \neq w'$.
Then:
\begin{enumerate}
\item $\overline{]X_{w,k}[} \cap ]Y_{w',k}[ = \emptyset$,
\item $]X_{w,k}[\cap\overline{]Y_{w',k}[}=\emptyset$.
\end{enumerate}
\end{lem}
\begin{proof}  We prove the first point. Note that $ \overline{]X_{w,k}[} = \mathrm{sp}^{-1}(X_{w,k})$ and $]Y_{w',k}[ = \mathrm{sp}^{-1}(Y_{w',k})$. Thus the first point follows from the fact that $X_{w,k} \cap Y_{w',k} = \emptyset$. We prove the second point. Since $]Y_{w',k}[$ is quasi-compact open,  $\overline{]Y_{w',k}[}$ is the set of all its specializations. Therefore, both $\overline{]Y_{w',k}[}$ and $]X_{w,k}[$ are stable under generalization. But they have no common rank one point, again as $X_{w,k} \cap Y_{w',k} = \emptyset$.
\end{proof}

\begin{lem}\label{lem-topological-description}
\begin{enumerate}
\item For all $m, n \in \qq$,  we have $\overline{]C_{w,k}[_{m,n}} = ]C_{w,k}[_{\overline{m}, \overline{n}} \subseteq \mathcal{U}_w^{an}$,
\item We have $]C_{w,k}[_{0, - \infty} \subseteq ]X_{w,k}[.$
\item We have $]C_{w,k}[_{-\infty, 0} \subseteq ]Y_{w,k}[.$
\item We have $\overline{]Y_{w,k}[} \cap ] X_{w, k}[  = ]C_{w,k}[_{0, \overline{0}}\subseteq \overline{]C_{w,k}[}$. 
\item We have ${]Y_{w,k}[} \cap \overline{ ] X_{w, k}[}  = ]C_{w,k}[_{\overline{0}, {0}} \subseteq \overline{]C_{w,k}[}$. 
\end{enumerate}
\end{lem}
\begin{proof} We begin with the first point.  Let $Z =FL \setminus U_w$, a closed subscheme. We let $\mathcal{Z}$ be the associated analytic space which is proper. 
We have $\mathcal{FL} = \mathcal{U}_w^{an} \coprod \mathcal{Z}$.  Since  $\mathcal{Z}$ is proper, we deduce that $\mathcal{U}_{w}^{an}$ is stable under specializations.  Let us consider $]C_{w,k}[_{m,n}$. We claim that $\overline{]C_{w,k}[_{m,n}} \subseteq \mathcal{U}_{w}^{an}$. To prove this, we can first find a quasi-compact open $\mathcal{U}$ such that $]C_{w,k}[_{m,n}\subseteq \mathcal{U} \subseteq \mathcal{U}_w^{an}$ and we see that $\overline{\mathcal{U}}$ is the set of all its specializations and is therefore included in $\mathcal{U}_w^{an}$. Now it is clear that $\overline{]C_{w,k}[_{m,n}} = ]C_{w,k}[_{\overline{m}, \overline{n}}$. 

 We now prove the second point. We observe that $w \prod_{\alpha \in w^{-1}\Phi^{-,M} \cap \Phi^+} \mathcal{U}_\alpha^{an} = \mathcal{C}_w^{an} \subseteq \mathcal{X}_w \subseteq ]X_{w,k}[$ and since $]X_{w,k}[$ is invariant under multiplication by the Iwahori subgroup,  $$w \prod_{\alpha \in  (w^{-1} \Phi^{-,M}) \cap \Phi^+} {\mathcal{U}^{an}_{\alpha}} \times \prod_{ \alpha \in  (w^{-1} \Phi^{-,M}) \cap \Phi^-} {\mathcal{U}_{\alpha}^o } \subseteq ]X_{w,k}[.$$  
We check the third point. We observe that  $w \prod_{\alpha \in w^{-1}\Phi^{-,M} \cap \Phi^-} \mathcal{U}_\alpha^{an} = \mathcal{C}^{w,an} \subseteq \mathcal{X}^w \subseteq ]Y_{w,k}[$ and again the conclusion follows because $]Y_{w,k}[$ is invariant under multiplication by the Iwahori subgroup. 

We prove the fourth point. Since $]Y_{w,k}[$ is quasi-compact, it is constructible, and therefore $\overline{]Y_{w,k}[}$ is the set of all specializations of   $]Y_{w,k}[$ in $\mathcal{FL}$. Let $x \in \overline{]Y_{w,k}[} \cap ] X_{w, k}[$ and let $y$ be the maximal generalization of $x$ in $\mathcal{FL}$. Then $y \in ]Y_{w,k}[ \cap ]X_{w,k}[ = ] C_{w,k}[$.   The subset of $\overline{]C_{w,k}[}$ consisting of points whose maximal generalization is in $]C_{w,k}[$ is exactly $w \prod_{\alpha \in  (w^{-1} \Phi^{-,M}) \cap \Phi^+} \overline{\mathcal{U}_{\alpha}} \times \prod_{ \alpha \in  (w^{-1} \Phi^{-,M}) \cap \Phi^-} {\mathcal{U}_{\alpha}^o }$. This proves that $\overline{]Y_{w,k}[} \cap ] X_{w, k}[$ is included in that set.  The converse inclusion follows easily from the second and third points. 
  We prove the last point. Since $\overline{ ] X_{w, k}[} = \mathrm{sp}^{-1}(X_{w,k})$ and $]Y_{w,k}[ = \mathrm{sp}^{-1}(Y_{w,k})$, we deduce that ${]Y_{w,k}[} \cap \overline{ ] X_{w, k}[}  = \mathrm{sp}^{-1}(C_{w,k}) \subseteq \overline{] C_{w, k}[}$ has exactly the announced description.  
\end{proof}

\begin{lem}\label{lem-decomposition-cells} \begin{enumerate}
\item We have $]X_{w,k}[ = ]C_{w,k}[_{0, \overline{0}}  \coprod \cup_{w' < w} ]X_{w',k}[$.
\item  We have $]Y_{w,k}[ = ]C_{w,k}[_{\overline{0}, {0}}  \coprod \cup_{w' > w} ]Y_{w',k}[$.
\item We have $\overline{]X_{w,k}[}= ]C_{w,k}[_{\overline{0}, {0}}  \coprod \cup_{w' < w}\overline{ ]X_{w',k}[}$.
\item  We have $\overline{]Y_{w,k}[}= ]C_{w,k}[_{{0}, \overline{0}}  \coprod \cup_{w' > w} \overline{]Y_{w',k}[}$.
\end{enumerate}
\end{lem}

\begin{proof} We prove the first point and the direct inclusion.  We first observe that since $X_{w',k}$ is closed, $]X_{w',k}[$ is a finite union of Stein spaces. Therefore $X_{w',k}$ is stable under specialization. Let $x \in  ]X_{w,k}[$. If its maximal generalization is in $]X_{w',k}[$ for some $w' < w$, then actually $x \in ]X_{w',k}[$. Otherwise, $x \in \overline{]C_{w,k}[}\cap ]X_{w,k}[ = ]C_{w,k}[_{0, \overline{0}}$ by lemma \ref{lem-topological-description}. The converse  inclusion follows from the same lemma. The union is disjoint by lemma \ref{topological-lemma-easy}. 
We prove the second point. We note that $]Y_{w,k}[ = \mathrm{sp}^{-1}( Y_{w,k})$ and therefore, $$]Y_{w,k}[ =  \mathrm{sp}^{-1}(C_{w,k}) \cap ]Y_{w,k}[ \bigcup\cup_{w' > w} ]Y_{w',k}[.$$ By lemma  \ref{lem-topological-description} $ \mathrm{sp}^{-1}(C_{w,k}) \cap ]Y_{w,k}[ \subset ]C_{w,k}[_{\overline{0}, 0}$. The reverse inclusion follows also from the lemma.  The union is disjoint by lemma \ref{topological-lemma-easy}.
We prove the third point. We have $X_{w,k} = C_{w,k} \coprod  \cup_{w'<w}X_{w',k}$. Taking $\mathrm{sp}^{-1}$ gives the identity. We prove the last point. All three spaces $\overline{]Y_{w,k}[}$, $]C_{w,k}[_{{0}, \overline{0}} = \overline{]Y_{w,k}[} \cap \overline{]X_{w,k}[}$ and $\cup_{w' > w} \overline{]Y_{w',k}[}$ are stable under specialization and generalization. It suffices therefore to prove the identity on rank one points, and it is clear. 
\end{proof}

\subsection{Dynamics}\label{section-dynamics}

Let $v : F \rightarrow \R \cup \{ + \infty\}$ be the $p$-adic valuation normalized by $v(p) =1$.  We consider certain sub-semigroups of $T(F)$.  We let $T^+(F) = \{ t \in T(F), v( \alpha(t)) \geq 0, ~\forall \alpha \in \Phi^+\}$, $T^{++}(F) =  \{ t \in T(F), v( \alpha(t)) > 0, ~\forall \alpha \in \Phi^+\}$, $T^-(F) = \{ t \in T(F), v( \alpha(t)) \leq 0, ~\forall \alpha \in \Phi^+\}$, $T^{--}(F) = \{ t \in T(F), v( \alpha(t)) < 0, ~\forall \alpha \in \Phi^+\}$.

\begin{lem}\label{lemma-dynamical1} \begin{enumerate}
\item If $t \in T^+(F)$, $]X_{w,k}[ . t \subseteq ]X_{w,k}[$ for all $w \in ~\WM$. 
\item If $t \in T^{-}(F)$, $]Y_{w,k}[ . t \subseteq ]Y_{w,k}[$ for all $w \in ~\WM$.
\end{enumerate}
\end{lem}
\begin{proof} By \cite{MR1207303}, corollary 4.2, we reduce to check the inclusions on rank $1$ points. 
 For the first point, it suffices therefore to prove that $]C_{w,k}[.t \subseteq ]X_{w,k}[$. Let $w b u \in ]C_{w,k}[$ with $b \in \mathcal{B}$ and $u \in \prod_{\alpha \in \Phi^-} \mathcal{U}_\alpha^o$. We find that $w b u t = w t^{-1} b t t^{-1} u t$. Now, $\mathrm{Ad}(t^{-1}) u \in \prod_{\alpha \in \Phi^-} \mathcal{U}_\alpha^o$, while $\mathrm{Ad}(t^{-1})b \in \mathcal{B}^{an}$. In particular  $w \mathrm{Ad}(t^{-1}) b \in \mathcal{X}_{w} \subseteq  ]X_{w,k}[_{\mathcal{FL}}$.   Therefore $w \mathrm{Ad}(t^{-1}) b \mathrm{Ad}(t^{-1}) u \in ]X_{w,k}[$. 
For the second point,  it suffices to prove that $]C_{w,k}[.t \subseteq ]Y_{w,k}[$. Let $w  u b \in ]C_{w,k}[$ with $b \in \mathcal{B}$ and $u \in \overline{\mathcal{U}}^o$. We find that $w  u b t = w t^{-1} u t t^{-1} b t$. Now, $\mathrm{Ad}(t^{-1}) b \in {\mathcal{B}}$, while $\mathrm{Ad}(t^{-1}) u  \in \overline{\mathcal{B}}^{an}$. In particular, $w \mathrm{Ad}(t^{-1}) u \in \mathcal{X}^{w} \subseteq  ]Y_{w,k}[_{\mathcal{FL}}$.   Therefore $w \mathrm{Ad}(t^{-1}) u \mathrm{Ad}(t^{-1}) b \in ]Y_{w,k}[$.
\end{proof}

\bigskip

For $t \in T^{++}(F)$,   we define $\min(t) = \min_{\alpha \in \Phi^+} v(\alpha(t))$  and $\max(t) = \max_{\alpha \in \Phi^+} v(\alpha(t))$. For $t \in T^{--}(F)$, we let  $\min(t) = \min_{\alpha \in \Phi^-} v(\alpha(t))$  and $\max(t) = \max_{\alpha \in \Phi^-} v(\alpha(t))$. We note that $\min(t) > 0$ and that $\min(t) = \min(t^{-1})$ and $\max(t) = \max(t^{-1})$ for $t \in T^{++}(F)$. 

\begin{lem}\label{lemma-strongcontraction}   \begin{enumerate}
\item Let $t \in T^{++}(F)$.   For all $w \in \WM$, $m,n \in \qq$,  we have 
\begin{eqnarray*}
]C_{w,k}[_{m+\max(t),n-\min (t)} &\subseteq& ]C_{w,k}[_{m,n} . t \subseteq ]C_{w,k}[_{m+\min(t),n-\max(t)}\\
]C_{w,k}[_{\overline{m+\max(t)},n-\min (t)} & \subseteq & ]C_{w,k}[_{\overline{m},n} . t \subseteq ]C_{w,k}[_{\overline{m+\min(t)},n-\max(t)},\\
]C_{w,k}[_{m+\max(t),\overline{n-\min (t)}} & \subseteq & ]C_{w,k}[_{m,\overline{n}} . t \subseteq ]C_{w,k}[_{m+\min(t),\overline{n-\max(t)}}.
\end{eqnarray*}
\item Let  $t \in T^{--}(F)$.  For all $w \in \WM$ and $m,n \in \qq$, we have
\begin{eqnarray*}
]C_{w,k}[_{m-\min(t),n+\max(t)} &\subseteq& ]C_{w,k}[_{m,n} . t \subseteq ]C_{w,k}[_{m-\max(t),n+\min(t)}\\
]C_{w,k}[_{\overline{m-\min(t)},n+\max (t)} & \subseteq & ]C_{w,k}[_{\overline{m},n} . t \subseteq ]C_{w,k}[_{\overline{m-\max(t)},n+\min(t)},\\
]C_{w,k}[_{m-\min(t),\overline{n+\max (t)}} & \subseteq & ]C_{w,k}[_{m,\overline{n}} . t \subseteq ]C_{w,k}[_{m-\max(t),\overline{n+\min(t)}}.
\end{eqnarray*}
\end{enumerate}
\end{lem}
\begin{proof} 
Easy and left to the reader. 
\end{proof}

\subsection{Dynamics of correspondences}
\subsubsection{Certain compact open subgroups}\label{section-compact-open-subgroups} We now assume that the group $G_F = G \times_{\Spec~\ocal_F} \Spec~F$ is defined and quasi-split over $\qq_p$. Therefore,  we have a reductive group $G_{\qq_p}$ with Borel $B_{\qq_p}$ and moreover the group $G_{\qq_p}$ splits over the extension $F$ of $\qq_p$. We have a reductive model $G$ of $G_{F}$ over $\Spec~\ocal_F$.  The Borel $B_{\qq_p}$ base changes to a Borel $B_F$ of $G_F$ which extends to a Borel $B$ of $G$. Similarly, we have a maximal torus $T_{\qq_p} \subseteq B_{\qq_p} $ of $G_{\qq_p}$, and its base change $T_F$ extends to a maximal (split) torus of $G$.  We will often drop the subscripts $\qq_p$ or $F$ when the context is clear. 
\begin{rem} Starting from the next section, we will consider a Shimura datum $(G,X)$, where $G$ is a reductive group over $\qq$. The group $G_{\qq_p}$ that we consider here will be the base change to $\qq_p$ of the group $G$ which is part of the Shimura datum. We apologize for this slightly inconsistent notation.
\end{rem} 
 
 We let $T(\ZZ_p)$ be the maximal compact subgroup of $T(\qq_p)$. Note that $T(\ZZ_p) = T(\ocal_F) \cap T(\qq_p)$. 
 We have an exact sequence $0 \rightarrow \ocal_F^\times \rightarrow F^\times \stackrel{v}\rightarrow \qq$ and tensoring with $X_\star(T)$ and taking Galois invariants, 
 we obtain the sequence: $0 \rightarrow T(\ZZ_p) \rightarrow T(\qq_p) \stackrel{v}\rightarrow X_\star(T^d) \otimes \qq$ where $T^d$ stands for the maximal split torus inside $T$ and $X_\star(T^d)$ for its cocharacter group. The image of $T(\qq_p)$ in $X_\star(T^d) \otimes \qq$ is easily seen to be a $\ZZ$-lattice. 
 
 We let $T^+ = T(\qq_p) \cap T^+(F)$, $T^{++} = T(\qq_p) \cap T^{++}(F)$, $T^-=T(\qq_p)\cap T^-(F)$, and $T^{--}=T(\qq_p)\cap T^{--}(F)$. These are semigroups in $T(\qq_p)$ and one proves easily that they generate   $T(\qq_p)$ (since there are regular elements in the maximal split torus $T^d$). 
 
 We will now consider certain compact open subgroups of $G(\ocal_F)$.   For all $m\in\ZZ_{\geq 0}$, we let $\tilde{K}_{p,m}\subseteq G(\ocal_F)$ be the preimage of $B(\ocal_F/p^m)$ under the map $G(\ocal_F)\to G(\ocal_F/p^m)$. We observe that  $\tilde{K}_{p,0}=G(\ocal_F)$ and $\tilde{K}_{p,1}$ is the Iwahori subgroup of $G(\ocal_F)$ with respect to the Borel $B(\ocal_F)$.  For $b\in\ZZ_{\geq 0}$ we let $\tilde{K}_{p,b,b} \subseteq G(\ocal_F)$ be the preimage of $U(\ocal_F/p^b)$ under the map $G(\ocal_F)\to G(\ocal_F/p^b)$.  Finally for $m\geq b\geq 0$ we let $\tilde{K}_{p,m,b}=\tilde{K}_{p,m}\cap \tilde{K}_{p,b,b}$.  In other words $\tilde{K}_{p,m,b}$ is the subgroup of $G(\ocal_F)$ of elements whose reduction mod $p^m$ lies in $B$, and whose reduction mod $p^b$ lies in $U$.  We note that we have $\tilde{K}_{p,m,b}\subset\mathcal{G}_{m-1,0}$.

For $m\geq b\geq 0$ and $m>0$, the groups $\tilde{K}_{p,m,b}$ have an Iwahori decomposition, in the sense that the product map
$$\overline{\tilde{U}}_m\times \tilde{T}_b\times {U}(\ocal_F)\to K_{p,m,b}$$
is a bijection, where $\tilde{T}_b=\ker(T(\ocal_F)\to T(\ocal_F/p^b))$ and $\overline{\tilde{U}}_m=\ker(\overline{U}(\ocal_F)\to\overline{U}(\ocal_F/p^{m}))$.

We now let $K_{p,m,b} = G(\qq_p) \cap \tilde{K}_{p,m,b}$. This is a compact open subgroup of $G(\qq_p)$. 

For $m\geq b\geq 0$ and $m>0$, the groups $K_{p,m,b}$ have an Iwahori decomposition, in the sense that the product map
$$\overline{U}_m\times T_b\times U(\ZZ_p)\to K_{p,m,b}$$
is a bijection, where $T_b= \tilde{T}_b \cap T(\qq_p)$ and $\overline{U}_m= \overline{\tilde{U}}_m \cap\overline{U}(\qq_p)$, $U(\ZZ_p) = U(\ocal_F) \cap U(\qq_p)$. 

\begin{rem}
We note that the groups $K_{p,m,b}$ depend (even up to conjugation in general), on the choice of the reductive model $G$ over $\ocal_F$.
\end{rem}

\begin{rem}\label{rem-Weil-restriction}If the group $G_{\qq_p}$ is unramified, then we can choose a reductive model $G_{\ZZ_p}$ of $G_{\qq_p}$ with Borel and maximal torus $B_{\ZZ_p}$ and $T_{\ZZ_p}$.  In this case, if we take $G_{\ocal_F}=G_{\ZZ_p}\times_{\ZZ_p}\Spec\ocal_F$, then we have that $K_{p,m}$ is the preimage of $B(\ZZ_p/p^m)$ under $G(\ZZ_p)\to G(\ZZ_p/p^m)$, and $K_{p,b,b}$ is the preimage of $U(\ZZ_p/p^b)$ under $G(\ZZ_p)\to G(\ZZ_p/p^b)$.  In particular, $K_{p,1,0}$ is an Iwahori subgroup, and $K_{p,1,1}$ is a pro-$p$ Iwahori subgroup.

In the ramified case, there is still a notion of an Iwahori subgroup of $G(\qq_p)$ defined using Bruhat-Tits theory.  However, our subgroup $K_{p,1,0}$ will usually not be an Iwahori subgroup.  In the general case, we found it convenient to work with the groups $K_{p,m,b}$ constructed as above, and it makes no difference for the purposes of this paper whether $K_{p,1,0}$ is an Iwahori.
\end{rem}
 
\subsubsection{Change of group}\label{subsecchangegroup}
In this paragraph we make a short digression that will be useful when we deal with abelian type Shimura varieties. Assume that we have an epimorphism of reductive groups $G_{\qq_p} \rightarrow G'_{\qq_p}$ with central kernel. This implies that  $G$ and $G'$ have the same adjoint group $G^{ad}$. We assume that these group split over $F$ and we fix models over $\ocal_F$, that we denote $G$ and $G'$ together with a map $G \rightarrow G'$ over $\Spec~\ocal_F$, extending the map $G_F \rightarrow G'_F$. We also assume that $G_{\qq_p}$ and $G'_{\qq_p}$ are quasi-split and pick Borels $B$ and $B'$ defined over $\qq_p$ such that $B \rightarrow B'$.  We can therefore define compact open subgroups $K_{p,m,b} \subseteq G(\qq_p)$ and $K'_{p,m,b} \subseteq G'(\qq_p)$ as in section \ref{section-compact-open-subgroups}.  

\begin{lem}\label{lem-compact-abelian} The groups $K_{p,m,b}$ and $K'_{p,m,b}$ have the same image in $G^{ad}(\qq_p)$.
\end{lem}
\begin{proof} This follows from the Iwahori decomposition of these groups. Since $G_{\qq_p} \rightarrow G'_{\qq_p}$ is an epimorphism with central kernel, the induced map on root groups are isomorphisms. 
\end{proof}

\subsubsection{Hecke algebras}\label{sect-Hecke-algebras} 

We let $\mathcal{H}_{p,m,b}$ be the Hecke algebras $\ZZ[ K_{p,m,b} \backslash G(\qq_p) / K_{p,m,b}]$.  We denote by  $\mathcal{H}_{p,m,b}^+$ the sub-algebra generated by the double cosets $[K_{p,m,b} t K_{p,m,b}]$ with $t \in T^+$ and by $\mathcal{H}_{p,m,b}^{++}$ the ideal generated by $[K_{p,m,b} t K_{p,m,b}]$ with $t \in T^{++}$. We define similarly $\mathcal{H}_{p,m,b}^-$ and $\mathcal{H}_{p,m,b}^{--}$. 

\begin{lem}\label{lemma-casselman} For all $m\geq b\geq 0$ with $m>0$, the map $t \mapsto  [K_{p,m,b} t K_{p,m,b}]$ induces isomorphisms
$\ZZ[T^+/T_b]  \rightarrow \mathcal{H}_{p,m,b}^+$ and $\ZZ[T^-/T_b]  \rightarrow \mathcal{H}_{p,m,b}^-$.
\end{lem}
\begin{proof} This is \cite{CasselmanNotes}, lem. 4.1.5.  (Alternatively it can be deduced from lemma \ref{lem-iwahori-comp} below.)
\end{proof}

\subsubsection{Action of correspondences on the flag variety}

 Let $K_{p} \subseteq G_0(\qq_p)$ be a compact open subgroup. 
 
 We consider the quotient space $\mathcal{FL}/K_{p}$. This space carries an action by correspondences  of double cosets $K_{p} gK_{p}$ for $g \in G_0(\qq_p)$. 
Namely, given $x K_{p} \in \mathcal{FL}/K_{p}$, we let $xK_{p} . K_{p} gK_{p} = \{ x\tilde{u}g K_p,~u \in K_p/ (gK_pg^{-1} \cap K_p),~\tilde{u} \in K_p~\textrm{lifts}~u\}$. 

We now consider a compact open  $K_p=K_{p,m',b}$ for $m'\geq b\geq 0$ and $m'>0$.

\begin{lem}\label{lem-dynamic-corres} Let $w \in \WM$. \begin{enumerate}
\item Let $t \in T^{+}$. The sequence $\{]X_{w,k}[ . (K_p t K_p)^m \}_{m\geq 0}$ is nested.
\item Let $t \in T^{-}$. The sequence $\{]Y_{w,k}[ . (K_p t K_p)^n\}_{n\geq 0}$ is nested. 
\item Let $t \in T^{++}$ and $m \geq 0$. We have   $$ ]C_{w,k}[_{ m \max(t),\overline{- m \min(t)}}. K_p \subseteq $$ $$]X_{w,k}[ . (K_p t K_p)^m \subseteq ]C_{w,k}[_{ m \min(t),\overline{- m \max(t)}}.K_p  \bigcup    \cup_{w' < w} ]X_{w',k}[. $$ 
\item  Let $t \in T^{--}$ and $n \geq 0$. We have  $$ ]C_{w,k}[_{\overline{- n\min(t)},  n\max(t)}.K_p \subseteq $$ $$ ]Y_{w,k}[. (K_p t K_p)^n \subseteq ]C_{w,k}[_{\overline{- n\max(t)},  n\min(t)}.K_p \bigcup \cup_{w'>w} ]Y_{w',k}[.$$

\item  Let $t \in T^{++}$. For all $m, n \in \ZZ_{\geq 0}$, 
$$ ]X_{w,k}[ . (K_p t K_p)^m  \cap \overline{]Y_{w,k}[} . (K_p t^{-1} K_p)^n \subseteq ]C_{w,k}[_{m \min(t), \overline{0}} K_p \cap ]C_{w,k}[_{0, \overline{n \min(t)}} K_p.$$

\item  Let $t \in T^{--}$.  For all $m, n \in \ZZ_{\geq 0}$, 
$$ \overline{]X_{w,k}[ }. (K_p t^{-1} K_p)^m  \cap {]Y_{w,k}[} . (K_p t K_p)^n \subseteq ]C_{w,k}[_{\overline{m\min(t)}, {0}} K_p \cap ]C_{w,k}[_{\overline{0},{n}\min(t)} K_p.$$

\item Let $t \in T^{--}$. For all $n \geq d$, $\overline{]Y_{w,k}[} . (K_p t K_p)^n \subseteq ]Y_{w,k}[$.

\end{enumerate}
\end{lem}

\begin{proof}  We observe that $]X_{w,k}[ . (K_p t K_p)^m = ]X_{w,k}[ t^m K_p$ and $]Y_{w,k}[ . (K_p t K_p)^m = ]Y_{w,k}[ t^m K_p$, using the very definition of the action, and also noting that $(K_p t K_p)^m = K_p t^m K_p$ by lemma \ref{lemma-casselman}. Therefore, the first and second points follow from lemma \ref{lemma-dynamical1}.   

We note that $]X_{w,k}[ =  ]C_{w,k}[_{0, \overline{0}} \bigcup \cup_{w'<w} ]X_{w',k}[$ and $]Y_{w,k}[ = ]C_{w,k}[_{\overline{0}, 0} \bigcup \cup_{w'>w} ]Y_{w',k}[$ by lemma \ref{lem-decomposition-cells}. 
The third and fourth point follow from this, the first two points, and lemma \ref{lemma-strongcontraction}.

We know that $]X_{w,k}[ \cap \overline{]Y_{w,k}[} = ]C_{w,k}[_{0, \overline{0}}$ by lemma \ref{lem-topological-description}. Moreover, $]C_{w,k}[_{0, \overline{0}} \cap ]X_{w',k}[ = \emptyset$ if $w' \neq w$ and $]C_{w,k}[_{0, \overline{0}} \cap ]Y_{w',k}[ =  \emptyset$ if $w' \neq w$ (because $]C_{w,k}[_{0, \overline{0}} $, $ ]X_{w',k}[ $ and $]Y_{w,k}[$ are all stable under generalization and they have no common rank one point). Therefore, 

$$ ]X_{w,k}[ . (K_p t K_p)^m  \cap \overline{]Y_{w,k}[} . (K_p t^{-1} K_p)^n \subseteq $$ $$ ]C_{w,k}[_{0, \overline{0}} \cap ]C_{w,k}[_{ m\min(t),\overline{-m\max(t)} } K_p \cap  ]C_{w,k}[_{\overline{- n\max(t^{-1})},\overline{n\min(t^{-1})}}.K_p 
\subseteq $$ $$ ]C_{w,k}[_{m \min(t), \overline{0}} K_p \cap ]C_{w,k}[_{0, \overline{n\min(t)}} K_p.$$

The proof of the sixth  point is similar.  We pass to the last point. By taking closures, we find that $\overline{]Y_{w,k}[}. (K_p t K_p) \subseteq  ]C_{w,k}[_{\overline{-\max(t)},  \overline{\min(t)}}.K_p \bigcup \cup_{w'>w}\overline{ ]Y_{w',k}[}$.  Since  $]C_{w,k}[_{\overline{-\max(t)},  \overline{\min(t)}}.K_p \subseteq ]Y_{w,k}[$, we deduce that for any $x \in \overline{]Y_{w,k}[}$, and any $y \in x. K_ptK_p$, then either $y \in ]Y_{w,k}[$ or $y \in \overline{]Y_{w',k}[}$ for some $w' > w$. We conclude by induction on $d-\ell(w)$.

\end{proof}

\section{Shimura varieties}

Let $(G,X)$ be a Shimura datum.  Thus $X$ is a $G(\R)$-conjugacy class of homomorphism $h : \mathrm{Res}_{\C/\R} \mathbb{G}_m \rightarrow G_{\R}$ satisfying a list of familiar axioms (\cite{MR546620}, section 2.1):

\begin{enumerate}
\item For all $h \in X$, the Hodge structure on $\mathfrak{g}_{\R}$ has weight $(1,-1)$, $(0,0)$ and $(-1,1)$. 
\item The involution $\mathrm{Ad}(h(i))$ is a Cartan involution on $G^{ad}_{\R}$.
\item The group $G^{ad}$ has no simple $\qq$-factor whose real points are compact.
\end{enumerate}

\begin{rem} The condition $(3)$ is mostly  irrelevant for this paper.  This axiom is used in order to apply the strong approximation theorem, which is important for the description of connected components of Shimura varieties, and therefore the theory of canonical models over the reflex field.  This is important for certain rationality questions, but will not play much of a role in this paper.  Groups $G$ whose real points are compact modulo the center give Shimura sets. Automorphic forms on these groups are called algebraic automorphic forms, and all our results are  well known   for Shimura sets:  all the vanishing theorems are trivial, whereas the theory of overconvergent modular forms has been extensively developed in this setting (\cite{MR2392353}, \cite{MR2075765}, \cite{MR2769113}). 
We  keep this third axiom for convenience, as it is used in many references, but believe all our results  hold without this assumption, with minor modifications of the proofs.
\end{rem}

  Via base change to $\C$, we have $(\mathrm{Res}_{\C/\R} \mathbb{G}_m )_{\C} =  \mathbb{G}_m \times \mathbb{G}_m$ (given by $z \mapsto (z, \bar{z})$) and via projection to the first factor, $X$ induces a conjugacy class of cocharacters $\mathbb{G}_m \rightarrow G_{\C}$.  We fix one such cocharacter $\mu$.  Associated to $\mu$ we have two opposite parabolic subgroups $P^{std}_\mu = \{ g \in G_\C, \lim_{t \rightarrow \infty} \mathrm{Ad}(\mu(t)) g ~\textrm{exists}\}$ and $P_{\mu} = \{ g \in G_\C, \lim_{t \rightarrow 0} \mathrm{Ad}(\mu(t)) g ~\textrm{exists}\}$.  We also let $M_\mu$ be the Levi quotient of $P_\mu$ and $P_\mu^{std}$. We let $FL^{std}_{G,\mu} = G_\C/P^{std}_\mu$ and $FL_{G,\mu} = P_{\mu} \backslash G_\C$ be the Flag varieties.
Let $E$ be the reflex field, which is the field of definition of the conjugacy class of $\mu$.  The two flag varieties are defined over $\Spec~E$. 

 Let $K \subset G(\mathbb{A}_f)$ be a  neat compact open subgroup, and let $S_K\to\Spec~E$ be the canonical model of the corresponding Shimura variety \cite{MR1044823}.  We have $S_K(\C) = G(\qq) \backslash X \times G(\mathbb{A}_f)/K$. In the rest of this paper, all the compact open subgroups $K \subset G(\mathbb{A}_f)$ are assumed to be neat, so  we do not always repeat this assumption.

The most fundamental Shimura data are  the Siegel data  $(\mathrm{GSp}_{2g}, \mathcal{H}_g)$ for all $g \in \ZZ_{\geq 1}$, where $\mathcal{H}_g$ is the Siegel (upper and lower) half space of matrices $M \in \mathrm{M}_{g}(\C)$ such that $~^tM = M$ and $\mathrm{Im}(M)$ is definite. The corresponding Shimura varieties parametrize abelian varieties of dimension $g$, with a polarization and level structure prescribed by $K$. 
A Shimura datum $(G,X)$ is of Hodge type if it admits an embedding in a Siegel datum. All PEL Shimura data are of Hodge type.  A Shimura datum $(G,X)$ is of abelian type if there exists a datum $(G_1,X_1)$ of Hodge type and a central isogeny $G_1^{der} \rightarrow G^{der}$ which induces an isomorphism of connected Shimura data $(G^{ad}, X^{\mathrm{+}}) = (G_1^{ad}, X_1^{+})$  where $X^{+}$ is  a connected component of $X$ (and similarly $X_1^+$ is a connected component of  $X_1$). 
\begin{ex} Let $L$ be a totally real number field. The datum $(\mathrm{Res}_{L/\qq} \mathrm{GSp}_{2g}, \mathcal{H}_g^{[L:\qq]})$ is of abelian type. We call it  a symplectic datum. In the case $g=1$, the corresponding Shimura varieties are the Hilbert modular varieties.
\end{ex}

We let $S^\star_K \rightarrow \Spec~E$ be the minimal compactification of $S_K$. Depending on the auxiliary choice  of a projective cone decomposition $\Sigma$, we let $S^{tor}_{K, \Sigma} \rightarrow \Spec~E$ be the toroidal compactification of $S_K$ corresponding to $\Sigma$. The (reduced) boundary $D_{K,\Sigma} = S^{tor}_{K, \Sigma} \setminus S_{K}$ is a Cartier divisor. The cone decompositions $\Sigma$ are partially ordered by inclusion, and any two cone decompositions admit a  common refinement.  The cone decompositions $\Sigma$ which are such that  $S^{tor}_{K, \Sigma}\to\Spec~E$ is smooth and projective are cofinal among all cone decompositions, and we usually choose them this way.  We refer to \cite{PINK}  for the construction of these compactifications.

\subsection{Automorphic vector bundles and their cohomology}
\subsubsection{Automorphic vector bundles}\label{section-automorphicvectorbundles}
We choose a field extension $F$ of $E$ which splits $G$ and we work over $F$ in this section.  In most of this paper, $F$ will be a finite extension of $\qq_p$, but this is not necessary for the moment.   We can choose  $\mu$ so that it is defined over $F$, and choose a maximal split torus $T$ with $\mu(\mathbb{G}_m)\subseteq T\subseteq M_\mu$.

We let $Z_s(G)$ be the largest subtorus of the center $Z(G)$ which is $\mathbb{R}$-split but contains no $\qq$-split subtorus.  We let $G^c = G/Z_s(G)$, and define $M_\mu^c$, $T^c$, $P_\mu^c$, and $P_\mu^{c,std}$ similarly.

\begin{rem} In the Hodge type case, $Z_s(G) = \{1\}$. For the symplectic datum  $(\mathrm{Res}_{L/\qq} \mathrm{GSp}_{2g}, \mathcal{H}_g^{[L:\qq]})$, $Z_s(G)$ is the kernel of the norm $\mathrm{Res}_{L/\qq} \mathbb{G}_m \rightarrow \mathbb{G}_m$. 
\end{rem}

 Let  $\mathrm{Rep}(M^c_\mu)$ be the category of finite dimensional algebraic representations of the reductive group $M^c_\mu$ on $F$-vector spaces.   

By \cite{MR1044823}, thm. 5.1 and \cite{MR997249}, thm. 4.2 we have a right $M^c_\mu$-torsor $M_{dR}$ over $S^{tor}_{K, \Sigma}$, and it corresponds to  a functor:  
\begin{eqnarray*}
\mathrm{Rep}(M^c_\mu) &\rightarrow& VB(S_{K,\Sigma}^{tor}) \\
V & \mapsto & \mathcal{V}_{K,\Sigma}
\end{eqnarray*}
where $VB(S_{K,\Sigma}^{tor})$ is the category of locally free sheaves of finite rank over $S_{K,\Sigma}^{tor}$.  This functor is compatible in a natural way with change of level $K$ and of cone decompositions $\Sigma$.  The locally free sheaves in the essential image of this functor are called (totally decomposed) automorphic vector bundles.
 They carry an equivariant action of $G(\mathbb{A}_f)$.
\begin{rem} We recall the description of $M_{dR} \times_{S^{tor}_{K,\Sigma}} S_K(\C)$ as a complex analytic space. First, we have the Borel embedding  $\beta : X \hookrightarrow FL_{G,\mu}^{std}(\C) =G^c(\C)/P_\mu^{c,std}(\C) $ which sends $h \in X$ to the parabolic stabilizing the Hodge filtration. The Borel embedding is equivariant for the left action of $G(\R)$.  We have  the canonical map $G^c(\C) \rightarrow FL_{G,\mu}^{std}(\C)$. This map is a  right $P_\mu^{c,std}(\C)$-torsor and is $G(\C)$-equivariant for the left action.  Similarly, the canonical map $G^c(\C)/U_{P_\mu^{c,std}}(\C) \rightarrow   FL_{G,\mu}^{std}(\C)$ is a right $M_\mu^c(\C)$-torsor (where $U_{P_\mu^{c,std}}$ is the unipotent radical of $P_\mu^{c,std}$). 
Then we have:   $$M_{dR} \times_{{S^{tor}_{K,\Sigma}}}  S_K(\C) = G(\qq) \backslash \beta^{-1} (G^c(\C)/U_{P_\mu^{c,std}}(\C)) \times {G}(\mathbb{A}_f) / K.$$
\end{rem}

\begin{rem} We give a description of $M_{dR}$ for the Siegel Shimura datum $(\mathrm{GSp}_{2g}, \mathcal{H}_g)$. Let us first  work over $S_K$.  Let $\omega_A$ be the conormal sheaf of the universal abelian scheme $A \rightarrow {S}_K$ along its unit section and let $\mathrm{Lie}(A)$ be its dual. The torsor ${M}_{dR}$ parametrizes  trivializations   $ \psi_1 \oplus \psi_2 :  \oscr_{S_K}^g \oplus \oscr_{S_K}^g \rightarrow \mathrm{Lie}(A) \oplus \omega_{A^t}$, such that under the isomorphism $\mathrm{Lie}(A)^\vee = \omega_{A^t}$ given by the polarization, we have $\psi_1 = c (\psi_2^{-1})^t$ for a unit $c \in  \oscr_{S_K}^\times$. 
The extension of $M_{dR}$ over $S_{K,\Sigma}^{tor}$ has a similar description. Indeed, $A$ can be extended to a semi-abelian scheme $A_\Sigma$ over ${S}_{K,\Sigma}^{tor}$, and we can let ${M}_{dR}$ be the torsor of trivializations  of $\mathrm{Lie}(A_\Sigma) \oplus \omega_{A_\Sigma^t} $, compatible, up to a unit, with the polarization.
\end{rem}

We now make the construction of  automorphic vector bundles explicit, and label them using weights.  We begin by making a choice of positive roots of $T$: we first choose a set of compact positive roots $\Phi^+_c$ which lie in $\mathfrak{m}_{\mu}$ the Lie algebra of $M_\mu$. We then choose the non-compact positive roots $\Phi^+_{nc}$ to lie in  $\mathfrak{g}/\mathfrak{p}_{\mu}^{std}$ where $\mathfrak{p}_{\mu}^{std}$ is the Lie algebra of $P_\mu^{std}$. We let $\Phi^+ = \Phi^+_c  \coprod \Phi^+_{nc}$. 
\begin{rem} This choice implies  that the Borel corresponding to $\Phi^+$ is included in $P_\mu$, or equivalently that the cocharacter $\mu$ of the Shimura datum is dominant. In  section \ref{section-Flagvar} we fixed a parabolic $P$ of $G$ containing a Borel $B$. In the applications to Shimura varieties, $P$ will be $P_\mu$. Therefore our convention is also compatible with the choice made in section \ref{section-Flagvar}.
\end{rem}

We let $X^\star(T)^{M_\mu,+}$ be the cone of characters of $T$ which are dominant for $\Phi^+_c$. We label irreducible representations of $M_{\mu}$ by their highest weight  $\kappa \in X^\star(T)^{M_\mu,+}$. An explicit construction of the highest weight $\kappa$ representation $V_\kappa$ is as follows. Let $w_{0,M}$ be the longest element of the Weyl group of $M_\mu$. For any $\kappa \in X^\star(T)^{M_\mu,+}$ we consider the space $V_\kappa$ of functions $f : M_\mu \rightarrow \mathbb{A}^1$ such that $f(mb) = (w_{0,M} \kappa) (b^{-1}) f(m)$ for all $m \in M_\mu$ and $b \in B \cap M_\mu$.  The action of $M_\mu$ on itself via left translation induces a left action on $V_\kappa$, i.e. $(m'\cdot f)(m)=f({m'}^{-1}m)$. 
The irreducible representations of $M^c_{\mu}$ are the irreducible representations of $M_\mu$ labelled by  dominant characters $\kappa$  of $T^c$.  We let $X^\star(T^c)^{M_\mu,+}$ be the cone of these characters. 

We denote by $\mathcal{V}_{\kappa, K, \Sigma}$ the locally free sheaf associated to  the irreducible representation of highest weight $\kappa$ of $M^c_\mu$.   Concretely, we consider the right torsor $g : M_{dR} \rightarrow {S}^{tor}_{K, \Sigma}$ and we let $\mathcal{V}_{\kappa, K, \Sigma}$ be the subsheaf of $g_\star \oscr_{M_{dR}}$ of sections $f(m) $ such that $f(m b) = (w_{0,M} \kappa) (b^{-1})f(m)$ for all $b \in B \cap M_\mu$. We will often abbreviate $\mathcal{V}_{\kappa, K, \Sigma}$ to $\mathcal{V}_{\kappa}$. 

We also introduce the cuspidal subsheaf $\mathcal{V}_{\kappa, K, \Sigma}(-D_{K,\Sigma})$ of sections vanishing on the boundary divisor $D_{K,\Sigma} \hookrightarrow S_{K,\Sigma}^{tor}$. Again, we often  abbreviate this sheaf to $\mathcal{V}_{\kappa}(-D)$. 

\subsubsection{The cohomology of automorphic vector bundles}

We let $\pi_{K, \Sigma} : S_{K,\Sigma}^{tor} \rightarrow S_{K}^\star$ be the projection from the toroidal  to the minimal compactification. 

\begin{thm}\label{vanishingtominimal}  We have $\mathrm{R}^i(\pi_{K,\Sigma})_\star \mathcal{V}_{\kappa, K, \Sigma}(-nD_{K,\Sigma}) =0$ for all $i >0$ and all $n \geq 1$.
\end{thm}
\begin{proof} In the PEL case (and for $n=1$), this is \cite{Lan2016}, thm. 8.6. We give an argument which follows closely \cite{MR3275848} and which is also similar to \emph{loc. cit}. Let $x \in {S}^{\star}_{K}$.  We write $\widehat{{S}^{tor}_{K, \Sigma}}^x $ for the formal completion of $S^{tor}_{K,\Sigma}$ along $\pi_{K,\Sigma}^{-1}(x)$. By the theorem on formal functions (\cite{stacks-project}, Tag 02O7) it suffices to prove that $\HH^i(\widehat{{S}^{tor}_{K, \Sigma}}^x,  \mathcal{V}_{\kappa, K, \Sigma}(-nD_{K,\Sigma})) = 0$ for all $i > 0$ and $n \geq 1$. 

 We may now use the description of $\widehat{{S}^{tor}_{K, \Sigma}}^x$ in terms of the local charts, following the notations of \cite{MR3948111}. The argument is mostly about torus embedding, so we do not need to explain in detail the structure of the toroidal compactification, but just recall what is strictly necessary.   Suppose that $x$ belongs to a boundary component indexed by a cusp label representative $\Phi$. 
There is a tower of spaces:
$$ \mathbf{S}_{K_\Phi}(Q_\Phi, D_\Phi) \stackrel{g}\rightarrow \mathbf{S}_{K_\Phi}( \overline{Q}_\Phi, \overline{D}_\Phi) \rightarrow \mathbf{S}_{K_\Phi}(G_{\Phi,h}, D_{\Phi,h})$$
where  $ \mathbf{S}_{K_\Phi}(Q_\Phi, D_\Phi) \rightarrow \mathbf{S}_{K_\Phi}( \overline{Q}_\Phi, \overline{D}_\Phi)$ is a torsor under a torus  $\mathbf{E}_K(\Phi)$. There is a locally free sheaf $\mathcal{V}_{\kappa,K}$ over $\mathbf{S}_{K_\Phi}( \overline{Q}_\Phi, \overline{D}_\Phi)$. 
There is a twisted torus embedding  $\mathbf{S}_{K_\Phi}(Q_\Phi, D_\Phi) \rightarrow \mathbf{S}_{K_\Phi}(Q_\Phi, D_\Phi,\Sigma(\Phi))$ which depends on the choice of $\Sigma$. There is an arithmetic group $\Delta_{K}(\Phi)$ acting on $X^\star(\mathbf{E}_K(\Phi))$ and on $\mathbf{S}_{K_\Phi}(Q_\Phi, D_\Phi) \hookrightarrow \mathbf{S}_{K_\Phi}(Q_\Phi, D_\Phi,\Sigma(\Phi))$. Let $\overline{g} : \mathbf{S}_{K_\Phi}(Q_\Phi, D_\Phi, \Sigma(\Phi)) \rightarrow \mathbf{S}_{K_\Phi}( \overline{Q}_\Phi, \overline{D}_\Phi)$.  The arithmetic group also acts on  $\overline{g}^\star \mathcal{V}_{\kappa,K}$. 
We have a $\Delta_K(\Phi)$-invariant closed   subscheme $\mathbf{Z}_{K_\Phi}(Q_\Phi, D_\Phi,\Sigma(\Phi)) \hookrightarrow \mathbf{S}_{K_\Phi}(Q_\Phi, D_\Phi, \Sigma(\Phi))$.
There is a finite morphism  $\mathbf{S}_{K_\Phi}(G_{\Phi,h}, D_{\Phi,h}) \rightarrow  {S}^{\star}_{K}$  whose image contains $x$. 
There is a series of morphisms: 
$$ \mathbf{Z}_{K_\Phi}(Q_\Phi, D_\Phi,\Sigma(\Phi)) \rightarrow \mathbf{S}_{K_\Phi}( \overline{Q}_\Phi, \overline{D}_\Phi) \rightarrow \mathbf{S}_{K_\Phi}(G_{\Phi,h}, D_{\Phi,h}).$$
We let $\mathbf{Z}_{K_\Phi}(Q_\Phi, D_\Phi,\Sigma(\Phi))_x$ be the closed subspace equal to the inverse image of $x$. The main result on the description of toroidal compactifications states that: 

$$\widehat{{S}^{tor}_{K, \Sigma}}^x \simeq \Delta_K(\Phi) \backslash \big( \widehat{ \mathbf{S}_{K_\Phi}(Q_\Phi, D_\Phi,\Sigma(\Phi))}^{\mathbf{Z}_{K_\Phi}(Q_\Phi, D_\Phi,\Sigma(\Phi))_x}\big)$$
Moreover, the sheaf $\overline{g}^\star \mathcal{V}_{\kappa,K}$ on $\mathbf{S}_{K_\Phi}(Q_\Phi, D_\Phi,\Sigma(\Phi))$ descends to $$\Delta_K(\Phi) \backslash \big( \widehat{ \mathbf{S}_{K_\Phi}(Q_\Phi, D_\Phi,\Sigma(\Phi))}^{\mathbf{Z}_{K_\Phi}(Q_\Phi, D_\Phi,\Sigma(\Phi))_x}\big)$$ and identifies with $\mathcal{V}_{\kappa, K, \Sigma}\vert_{\widehat{{S}^{tor}_{K, \Sigma}}^x}$.
We let $D$ be the boundary divisor ${S}^{tor}_{K,\Sigma} \setminus {S}_{K}$. Under the isomorphism above, it corresponds to the divisor $D_0 =  \mathbf{S}_{K_\Phi}(Q_\Phi, D_\Phi,\Sigma(\Phi)) \setminus \mathbf{S}_{K_\Phi}(Q_\Phi, D_\Phi)$. 

There is a divisor $D'$ on $ {S}^{tor}_{K,\Sigma}$ which  has exactly the same support as $D$ and such that $\oscr_{{S}^{tor}_{K,\Sigma}}(-D')$ is ample relatively to the minimal compactification. It corresponds to a divisor $D'_0$ on  $\mathbf{S}_{K_\Phi}(Q_\Phi, D_\Phi,\Sigma(\Phi))$.

We will prove the following statement: for any $C \in \ZZ_{\geq 0}$, there exists $s \geq C$ and a finite morphism
$$ \psi : \widehat{{S}^{tor}_{K, \Sigma}}^x  \rightarrow \widehat{{S}^{tor}_{K, \Sigma}}^x $$ such that $\psi^\star \mathcal{V}_{\kappa,K,\Sigma} = \mathcal{V}_{\kappa,K,\Sigma}$ and $\oscr(-nD) \rightarrow \psi_\star \oscr(-sD')$ is a split injection.
This implies the theorem, as we deduce that for all $i \geq 0$, $\HH^i(\widehat{{S}^{tor}_{K, \Sigma}}^x, \mathcal{V}_{\kappa,K,\Sigma} (-nD))$ is a direct factor of $\HH^i(\widehat{{S}^{tor}_{K, \Sigma}}^x, \mathcal{V}_{\kappa,K,\Sigma} (-sD'))$. Taking $C$ large enough, this last group vanishes for $i>0$. 

We now prove the claim about the existence of $\psi$. Our proof follows \cite{MR3275848}, p. 679.  We will construct everything on $\mathbf{S}_{K_\Phi}(Q_\Phi, D_\Phi,\Sigma(\Phi))$.
For any integer $\ell$, the multiplication by $\ell$-map on the torus $\mathbf{E}_K(\Phi)$ induces a finite morphism $\psi_\ell  : \mathbf{S}_{K_\Phi}(Q_\Phi, D_\Phi,\Sigma(\Phi)) \rightarrow \mathbf{S}_{K_\Phi}(Q_\Phi, D_\Phi,\Sigma(\Phi))$ which is $\Delta_K(\Phi)$ equivariant and for which $\psi_\ell^\star \overline{g}^\star \mathcal{V}_{\kappa,K} = \overline{g}^\star \mathcal{V}_{\kappa,K}$. The morphisms $\psi_\ell$ induces a finite morphism  $ \widehat{{S}^{tor}_{K, \Sigma}}^x \rightarrow  \widehat{{S}^{tor}_{K, \Sigma}}^x$. 
We have $D'_0 = \sum_{\sigma \in \Sigma(\Phi)(1)}  a_\rho D_\rho$ where $\Sigma(\Phi)(1)$ is the set of one dimensional faces in $\Sigma(\Phi)$ and $D_0 = \sum_{\sigma \in \Sigma(\Phi)(1)}  D_\rho$. Since $D'_0$ is $\Delta_K(\Phi)$-equivariant and $\Sigma(\Phi)(1)/\Delta_K(\Phi)$ is finite, we deduce that  for any $s \geq 1$, there exists $\ell$ such that $0  < s a_\rho \leq \ell$. It follows that the round-down of the $\qq$-divisor  $- \ell^{-1} D'_0$ is $-D_0$. 
We deduce from \cite{MR2810322}, lem. 9.3.4 that $ \oscr(-D_0) \hookrightarrow (\psi_{\ell})_\star  \oscr(-sD'_0) $ is a split injection. Pulling back this morphism by $\psi_n$, we deduce that $ \oscr(-nD_0) \hookrightarrow (\psi_{\ell})_\star  \oscr(-snD'_0) $ is a split injection.
\end{proof} 

We let $\pi_{K,K',\Sigma,\Sigma'} : S_{K,\Sigma}^{tor} \rightarrow S_{K',\Sigma'}^{tor}$ be the map associated to a change of level (for $K\subseteq K'$) and cone decomposition.
\begin{thm}\label{thm-van-between-toroidal}   
We have $\mathrm{R}^i(\pi_{K,K',\Sigma,\Sigma'})_\star\mathcal{V}_{\kappa,K,\Sigma}=0$ and $\mathrm{R}^i(\pi_{K,K',\Sigma,\Sigma'})_\star\mathcal{V}_{\kappa,K,\Sigma}(-D_{K,\Sigma})=0$ for all $i>0$.  Moreover if $K\subseteq K'$ is normal then we have
$$((\pi_{K,K',\Sigma, \Sigma'})_\star \mathcal{V}_{\kappa, K, \Sigma})^{K'/K} =  \mathcal{V}_{\kappa, K', \Sigma'}$$  and $$((\pi_{K,K',\Sigma, \Sigma'})_\star \mathcal{V}_{\kappa, K, \Sigma}(-D_{K,\Sigma}))^{K'/K} =  \mathcal{V}_{\kappa, K', \Sigma'}(-D_{K', \Sigma'}).$$ 
\end{thm}
\begin{proof} This is easily extracted from \cite{MR1064864}, section 2 (look in particular at proposition 2.4 and proposition 2.6 and their proofs). In the PEL case, this is explicitly  \cite{Lan2016}, prop. 7.5.
\end{proof}

In particular from the case that $K=K'$ we deduce that the cohomologies $\mathrm{R}\Gamma(S_{K,\Sigma}^{tor}, \mathcal{V}_{\kappa, K, \Sigma})$ and  $\mathrm{R}\Gamma(S_{K,\Sigma}^{tor}, \mathcal{V}_{\kappa, K, \Sigma}(-D_{K,\Sigma}))$ do not depend on $\Sigma$. 

We also recall that the maps $\pi_{K,K',\Sigma,\Sigma'}$ have fundamental classes in the sense of \cite{F-Pilloni} section 2.3.  Namely we have maps $\oscr_{S_{K,\Sigma}^{tor}}\to \pi_{K,K',\Sigma,\Sigma'}^!\oscr_{S_{K',\Sigma'}^{tor}}$ and $\oscr_{S_{K,\Sigma}^{tor}}(-D_{K,\Sigma})\to \pi_{K,K',\Sigma,\Sigma}^!\oscr_{S_{K',\Sigma'}^{tor}}(-D_{K',\Sigma'})$ or equivalently, by adjunction, trace maps $\mathrm{R}(\pi_{K,K',\Sigma,\Sigma'})_\star\oscr_{S_{K,\Sigma}^{tor}}=(\pi_{K,K',\Sigma,\Sigma'})_\star\oscr_{S_{K,\Sigma}^{tor}}\to\oscr_{S_{K',\Sigma'}^{tor}}$ and $\mathrm{R}(\pi_{K,K',\Sigma,\Sigma'})_\star(\oscr_{S_{K,\Sigma}^{tor}}(-D_{K,\Sigma}))=(\pi_{K,K',\Sigma,\Sigma'})_\star(\oscr_{S_{K,\Sigma}^{tor}}(-D_{K,\Sigma}))\to\oscr_{S_{K',\Sigma'}^{tor}}(-D_{K',\Sigma'})$, both of which extend the trace map for the finite \'{e}tale morphism $\pi_{K,K'}:S_K\to S_{K'}$.

\subsection{Action of the Hecke algebra}

\subsubsection{Action of Hecke correspondences on cohomology}\label{section-hecke-construction}
Let $K_1,K_2\subseteq G(\mathbb{A}_f)$ be open compact subgroups, and let $g\in G(\mathbb{A}_f)$.  To this data we associate a Hecke correspondence
\begin{eqnarray*}
\xymatrix{ & S_{K_1 \cap g K_2 g^{-1}} \ar[rd]^{p_1} \ar[ld]_{p_2} & \\
S_{K_2} & & S_{K_1}}
\end{eqnarray*}
where $p_1$ is the forgetful map corresponding to the inclusion $K_1\cap gK_2g^{-1}\subseteq K_1$, and $p_2$ is the composition of the action map $[g]:S_{K_1\cap gK_2g^{-1}}\to S_{g^{-1}K_1g\cap K_2}$ and the forgetful map corresponding to the inclusion $g^{-1}K_1g\cap K_2\subseteq K_2$.

For any weight $\kappa\in X^\star(T^c)^{M_\mu,+}$ we have a cohomological correspondence $(p_1)_\star p_2^\star\mathcal{V}_{\kappa,K_2}\to\mathcal{V}_{\kappa,K_1}$.  This is obtained by combining the trace map $\mathrm{tr}_{p_1}:(p_1)_\star\ocal_{S_{K_1\cap gK_2g^{-1}}}\to\ocal_{S_{K_1}}$ for the finite \'etale map $p_1$ with the isomorphism $p_2^\star\mathcal{V}_{\kappa,K_2}\simeq p_1^\star\mathcal{V}_{\kappa,K_1}$, which is itself composed of the isomorphism $p_1^\star\mathcal{V}_{\kappa,K_1}\simeq\mathcal{V}_{\kappa,K_1\cap gK_2g^{-1}}$ and the similar isomorphism for the forgetful map $S_{g^{-1}K_1g\cap K_2}\to S_{K_2}$, as well as the action map $[g]^\star\mathcal{V}_{\kappa,g^{-1}K_1g\cap K_2}\simeq\mathcal{V}_{\kappa,K_1\cap gK_2g^{-1}}$.

One readily verifies that the cohomological correspondence $((S_{K_1\cap gK_2g^{-1}},p_1,p_2),(p_1)_\star p_2^\star\mathcal{V}_{\kappa,K_2}\to\mathcal{V}_{\kappa,K_1})$ only depends on the double coset $K_1gK_2$, up to canonical isomorphism.

Now for suitable choices of cone decomposition we have a diagram
\begin{eqnarray*}
\xymatrix{ & S^{tor}_{K_1 \cap g K_2 g^{-1},\Sigma''} \ar[rd]^{p_1} \ar[ld]_{p_2} & \\
S^{tor}_{K_2,\Sigma'} & & S^{tor}_{K_1,\Sigma}}
\end{eqnarray*}
and we claim that our cohomological correspondence on the interior extends to a cohomological correspondence $\mathrm{R}(p_1)_\star p_2^\star\mathcal{V}_{\kappa,K_2,\Sigma'}=(p_1)_\star p_2^\star\mathcal{V}_{\kappa,K_2,\Sigma'}\to\mathcal{V}_{\kappa,K_1,\Sigma''}$ as well as a cuspidal version  $\mathrm{R}(p_1)_\star p_2^\star\mathcal{V}_{\kappa,K_2,\Sigma'}(-D_{K_2,\Sigma'})=(p_1)_\star p_2^\star\mathcal{V}_{\kappa,K_2,\Sigma'}(-D_{K_2,\Sigma'})\to\mathcal{V}_{\kappa,K_1,\Sigma''}(-D_{K_1,\Sigma})$.  We have already discussed the extensions of the trace map in the previous section.  As for the action map, it also induces an isomorphism of canonical extensions $p_2^\star\mathcal{V}_{\kappa,K_2,\Sigma}\to p_1^\star\mathcal{V}_{\kappa,K_1,\Sigma}$ as well as a morphism of subsheaves $p_2^\star(\mathcal{V}_{\kappa,K_2,\Sigma}(-D_{K_2,\Sigma'}))\to (p_1^\star\mathcal{V}_{\kappa,K_1,\Sigma})(-D_{K_1\cap gK_2g^{-1},\Sigma''})$.

Finally these cohomological correspondences induces maps on cohomology in the usual way, namely we denote by $[K_1gK_2]$ the composition
$$\mathrm{R}\Gamma(S_{K_2,\Sigma'}^{tor},\mathcal{V}_{\kappa,K_2,\Sigma'})\to\mathrm{R}\Gamma(S_{K_1\cap gK_2g^{-1},\Sigma''}^{tor},p_2^*\mathcal{V}_{\kappa,K_2,\Sigma'})=$$
$$\mathrm{R} \Gamma(S_{K_1,\Sigma}^{tor},\mathrm{R}(p_1)_*p_2^*\mathcal{V}_{\kappa,K_2,\Sigma'})\to \Gamma(S_{K_1,\Sigma}^{tor},\mathcal{V}_{\kappa,K_1,\Sigma})$$
where the first map is $p_2^*$ and the last map uses the cohomological correspondence.  We have a similar definition for cuspidal cohomology.

\subsubsection{Composition of Hecke correspondences}\label{section-compos-Hecke-op}
Let us briefly explain the point of this section.  We have defined an action of individual Hecke operators on coherent cohomology, and we would like to show that this actually gives an action of the Hecke algebra.  The standard proof (see \cite[Prop. 2.6]{MR1064864}) is to pass to a limit over all $K$ to obtain a representation of $G(\mathbb{A}_f)$, and then use the purely group theoretic fact that Hecke algebras act on its invariants.  However in this paper we will also study coherent cohomology with support conditions which are not $G(\qq_p)$ invariant, and so we no longer expect an action of the full Hecke algebra, but only of certain Hecke operators which preserve the support conditions in a suitable sense.  We would still like to know that these Hecke operators compose according to relations in the abstract Hecke algebra.

For this reason we develop a different approach to composition of Hecke operators, by directly studying the geometric composition of Hecke correspondences.  This material is presumably well known but we lack a reference (see \cite[VII \S3]{MR1083353} for a closely related discussion.)

We recall the formalism of double coset multiplication.  Fix a Haar measure on $G(\mathbb{A}_f)$ and let $C_c^\infty(G(\mathbb{A}_f),\mathbb{R})$ be the Hecke algebra with its convolution product.  For $K_1,K_2\subset G(\mathbb{A}_f)$ we consider the free group $\ZZ[K_1\backslash G(\mathbb{A}_f)/K_2]$ and we write the basis elements as $[K_1gK_2]$.  We have an embedding $i_{K_1,K_2}:\ZZ[K_1\backslash G(\mathbb{A}_f)/K_2]\to C_c^\infty(G(\mathbb{A}_f),\mathbb{R})$ which sends $[K_1gK_2]$ to $\frac{1}{\sqrt{\mathrm{vol}(K_1)\mathrm{vol}(K_2)}}\mathbf{1}_{K_1gK_2}$, where $\mathbf{1}_{K_1gK_2}$ denotes the characteristic function of the double coset $K_1gK_2$.

For $K_1,K_2,K_3\subset G(\mathbb{A}_f)$ open compact, there is a product map
$$\ZZ[K_1\backslash G(\mathbb{A}_f)/K_2]\times\ZZ[K_1\backslash G(\mathbb{A}_f)/K_2]\to\ZZ[K_1\backslash G(\mathbb{A}_f)/K_3]$$
which can be defined by
$$i_{K_1,K_3}([K_1gK_2][K_2hK_3])=i_{K_1,K_2}([K_1gK_2])\star i_{K_2,K_3}([K_2hK_3])$$
To see that the right hand side is in the image of $i_{K_1,K_3}$ note that
$$(\mathbf{1}_{K_1gK_2}\star\mathbf{1}_{K_2hK_3})(x)=\mathrm{vol}(K_1gK_2\cap xK_3h^{-1}K_2)$$
is an integer multiple of $\mathrm{vol}(K_2)$.  We also note that this definition is independent of the choice of Haar measure.

It follows from the definition that double coset multiplication is associative and satisfies 
$$([K_1gK_2][K_2hK_3])^t=[K_2hK_3]^t[K_1gK_2]^t$$
where the transpose map $(-)^t:\ZZ[K_1\backslash G(\mathbb{A}_f)/K_2]\to\ZZ[K_2\backslash G(\mathbb{A}_f)/K_1]$ is defined by $[K_1gK_2]^t=[K_2g^{-1}K_1]$.  This corresponds to transposes of Hecke correspondences.

We now give a formula for double coset multiplication which is closely related to geometric composition of Hecke correspondences.
\begin{prop}\label{prop-conv-doublecoset}
Let $k_1,\ldots,k_n\in K_2$ be a set of representatives for the double cosets $(K_2\cap g^{-1}K_1g)\backslash K_2/(K_2\cap hK_3h^{-1})$.  Then we have
$$[K_1gK_2][K_2hK_3]=\sum_{i=1}^nc(k_i)[K_1gk_ihK_3]$$
where $c(k_i)=[g^{-1}K_1g\cap(k_ih)K_3(k_ih)^{-1}:g^{-1}K_1g\cap K_2\cap(k_ih)K_3(k_ih)^{-1}]$
\end{prop}

\begin{proof}
We first note that for $x,y\in G(\mathbb{A}_f)$ we have $\mathbf{1}_{K_1x}\star\mathbf{1}_{yK_3}=\mathrm{vol}(x^{-1}K_1x\cap yK_3y^{-1})\mathbf{1}_{K_1xyK_3}$.
Decomposing $K_1gK_2$ into right $K_1$ cosets and $K_2hK_3$ into left $K_3$ cosets we have
$$\mathbf{1}_{K_1gK_2}\star\mathbf{1}_{K_2hK_3}=\sum_{\substack{x\in K_1\backslash K_1gK_2\\y\in K_2hK_3/K_3}}\mathrm{vol}(x^{-1}K_1x\cap yK_3y^{-1})\mathbf{1}_{K_1xyK_3}.$$
Now note that the terms of this sum are constant in $K_2$ orbits, for the action $k\cdot (K_1x,yK_3)=(K_1xk^{-1},kyK_3)$.  We have a bijection
$$(K_2\cap g^{-1}K_1g)\backslash K_2/(K_2\cap hK_3h^{-1})\to K_2\backslash((K_1\backslash K_1gK_2)\times (K_2hK_3/K_3))$$
sending the double coset of $k$ to the orbit of $(K_1g,khK_3)$.  Moreover the stabilizer of this orbit is $K_2\cap g^{-1}K_1g\cap(kh)K_3(kh)^{-1}$.  It follows that
$$\mathbf{1}_{K_1gK_2}\star\mathbf{1}_{K_2hK_3}=\sum_{i=1}^n\frac{\mathrm{vol}(K_2)\mathrm{vol}(g^{-1}K_1g\cap (k_ih)K_3(k_ih)^{-1})}{\text{vol}(K_2\cap g^{-1}K_1g\cap (k_ih)K_3(k_ih)^{-1})}\mathbf{1}_{K_1gk_ihK_3}.$$
which translated back into double coset multiplication gives the proposition.
\end{proof}

Suppose we have $K_1,K_2,K_3\subset G(\mathbb{A}_f)$ open compact subgroups and $g,h\in G(\mathbb{A}_f)$.  We have a diagram of correspondences
\begin{eqnarray*}
\xymatrix{ & & C  \ar[rd]^{r_1} \ar[ld]_{r_2} & &\\ & S_{K_2\cap hK_3h^{-1}} \ar[rd]^{q_1} \ar[ld]_{q_2} &  & S_{K_1 \cap g K_2 g^{-1}} \ar[rd]^{p_1} \ar[ld]_{p_2} & \\S_{K_3} & &
S_{K_2} & & S_{K_1}}
\end{eqnarray*}
where the middle diamond is cartesian.  We denote $s_1=p_1r_1$, $s_2=q_2r_2$.

It is not true in general that the correspondence $(C,s_1,s_2)$ is isomorphic to a disjoint union of Hecke correspondences.  However this is almost the case in a way we now explain.  

We first construct another correspondence $(C',s_1',s_2')$ between $S_{K_1}$ and $S_{K_3}$.  Let $k_1,\ldots,k_n$ be as in proposition \ref{prop-conv-doublecoset}.   Now let $C'=\coprod_i S_{K_1\cap gK_2g^{-1}\cap (gk_ih)K_3(gk_ih)^{-1}}$, let $s_1':C'\to S_{K_1}$ be the forgetful map on each component, and let $s_2':C'\to S_{K_3}$ be the action map $[gk_ih]$ followed by the forgetful map.  Then we note that while the components of $C'$ are not necessarily Hecke correspondences themselves, we nonetheless have commutative diagrams
\begin{eqnarray*}
\xymatrix{&S_{K_1\cap gK_2g^{-1}\cap (gk_ih)K_3(gk_ih)^{-1}} \ar[d]\ar[ddl]\ar[ddr]&\\
&S_{K_1\cap(gk_ih)K_3(gk_ih)^{-1}} \ar[dl]\ar[dr]\\
S_{K_3}& &S_{K_1}}
\end{eqnarray*}
where the vertical arrow and rightward arrows are forgetful maps and the leftward arrows are the action maps for $[gk_ih]$ followed by forgetful maps.  In particular the bottom correspondence is exactly the Hecke correspondence corresponding to the double coset $K_1gk_ihK_3$.

\begin{prop}\label{prop-geometric-comp}
There is an isomorphism of correspondences $(C,s_1,s_2)\simeq (C',s_1',s_2')$.
\end{prop}

We will use the following group theoretic lemma.
\begin{lem}\label{lem-double-orbit}
Let $K$ be a group and $H_1,H_2\subseteq K$ subgroups.  Let $X$ be a right $K$-torsor.  Then there is a bijection
\begin{eqnarray*}
\coprod_{H_1gH_2\in H_1\backslash K/H_2}X/gH_1g^{-1}\cap H_2&\to& X/H_1\times X/H_2\\
x(gH_1g^{-1}\cap H_2)&\mapsto&(xgH_1,xH_2)
\end{eqnarray*}
\end{lem}

\begin{proof}[Proof of proposition \ref{prop-geometric-comp}]
We recall that a point of the Shimura variety $S_K$ over a connected locally Noetherian base is some data which is independent of $K$ (an isogeny class of abelian varieties with Hodge tensors) to which one associates a right $G(\mathbb{A}_f)$-torsor, and a $K$ level structure is just a $K$ orbit in this $G(\mathbb{A}_f)$-torsor.  Moreover for any $K'\subset K$ the forgetful map $S_{K'}\to S_K$ at the level of points sends the $K'$-orbit to the $K$-orbit generated by it, and the action map $[g]:S_{K}\to S_{g^{-1}Kg}$ is given by right multiplication by $g$.

With this observation the proof of the proposition proceeds formally.  We can consider a slightly expanded diagram.
\begin{eqnarray*}
\xymatrix{ & & C''  \ar[rd] \ar[ld] & C\ar[l]_{[g]}\ar[dr]&\\
 & S_{K_2\cap hK_3h^{-1}} \ar[rd] \ar[ld] &  & S_{g^{-1}K_1g \cap K_2 } \ar[ld] & S_{K_1 \cap g K_2 g^{-1}} \ar[d]\ar[l]_{[g]}\\
S_{K_3} & & S_{K_2} & & S_{K_1}}
\end{eqnarray*}
Here all three quadrilaterals are Cartesian.  We can compute $C''$ using lemma \ref{lem-double-orbit}.  Indeed giving a point of $C''$ lying over a point of $S_{K_2}$ is the same as giving a pair of an $H_1=K_2\cap hK_3h^{-1}$ and $H_2=g^{-1}K_1g\cap K_2$ orbit inside a $K=K_2$-torsor $X$.  We deduce from lemma \ref{lem-double-orbit} an isomorphism
$$C''\simeq\coprod_{i=1}^n S_{g^{-1}K_1g\cap K_2\cap (k_ih)K_3(k_ih)^{-1}}$$
where on each factor, the map to $S_{g^{-1}K_1g\cap K_2}$ is the forgetful map, while the map to $S_{K_2\cap hK_3h^{-1}}$ is the action map $[k_i$] followed by the forgetful map.  Then the isomorphism $C\simeq C'$ is obtained as the composition
$$C\overset{[g]}{\to} C''\to\coprod_{i=1}^n S_{g^{-1}K_1g\cap K_2\cap (k_ih)K_3(k_ih)^{-1}}\overset{[g]^{-1}}{\to}C'.$$
Moreover from the descriptions of the maps out of $C''$ above one immediately deduces the compatibility between $s_1,s_1'$ and $s_2,s_2'$.
\end{proof}

The following proposition is readily deduced from \cite[Prop. 2.6]{MR1064864}, but we give an alternative proof which will also work for cohomology with support.
\begin{prop}\label{prop-comp-classical-coh}
We have $[K_1gK_2]\circ [K_2hK_3]=[K_1gK_2][K_2hK_3]$ as maps $$\mathrm{R}\Gamma(S_{K_3,\Sigma''}^{tor},\mathcal{V}_{\kappa,K_3,\Sigma''})\to\mathrm{R}\Gamma(S_{K_1,\Sigma}^{tor},\mathcal{V}_{\kappa,K_1,\Sigma}).$$
\end{prop}

\begin{proof}
We can choose toroidal compactifications of all the Shimura varieties occurring in this section: $S_{K_i}$, $i=1,2,3$, $S_{K_1\cap gK_2g^{-1}}$,$S_{K_2\cap hK_3h^{-1}}$, $S_{K_1\cap (gk_ih)K_3(gk_ih)^{-1}}$, $S_{K_1\cap gK_2g^{-1}\cap (gk_ih)K_3(gk_ih)^{-1}}$ so that all the diagrams appearing above extend to the compactifications (we will not be particularly consistent or careful with our labelling of the various cone decompositions in this argument, and they will be denoted by $\Sigma_{\ast}$ where $\ast$ is an index). In particular we obtain a compactification $C^{tor}$ of the correspondence $C$ which is isomorphic to a disjoint union of toroidal compactifications of Shimura varieties.  Moreover using this description, we can construct cohomological correspondences over $C^{tor}$ as in section \ref{section-hecke-construction} using the action maps for $gk_ih$.   We first claim that this cohomological correspondence acts on cohomology by the linear combination of Hecke operators on the right hand side of the formula in proposition \ref{prop-conv-doublecoset}.

Consider a commutative diagram \begin{eqnarray*}
\xymatrix{&S^{tor}_{K_1\cap gK_2g^{-1}\cap (gk_ih)K_3(gk_ih)^{-1}, \Sigma_1} \ar[d]^c\ar[ddl]_{q'}\ar[ddr]^{p'}&\\
&S^{tor}_{K_1\cap(gk_ih)K_3(gk_ih)^{-1}, \Sigma_2} \ar[dl]_{q}\ar[dr]^{p}\\
S^{tor}_{K_3, \Sigma_3}& &S^{tor}_{K_1, \Sigma_4}}
\end{eqnarray*}

We claim that the following diagram commutes (where the vertical map is induced by the adjunction $\mathrm{Id} \Rightarrow c_\star c^\star$, the other maps are given by the cohomological correspondence, and $c(k_i)$ is the generic degree of the generically finite flat map $c$):   

\begin{eqnarray*}
\xymatrix{ \mathrm{R}p'_\star (q')^\star \mathcal{V}_{\kappa, K_3, \Sigma_3} \ar[r]^{tr_{p'}} & \mathcal{V}_{\kappa, K_1, \Sigma_1} \\
 \mathrm{R}p_\star (q)^\star \mathcal{V}_{\kappa, K_3, \Sigma_3} \ar[ur]^{ c(k_i) tr_p} \ar[u]& }
\end{eqnarray*}

These are maps of locally free sheaves by theorem \ref{thm-van-between-toroidal}, and the commutativity of the diagram follows from its commutativity away from the boundary which  is clear.

Hence to prove the proposition, we need to see that the action of the cohomological correspondence $C^{tor}$ is equal to $[K_1gK_2]\circ[K_2hK_3]$. 

Consider the diamond from the proof of proposition \ref{prop-geometric-comp}
\begin{eqnarray*}
\xymatrix{&{C''}^{tor}\ar[dr]^{q'}\ar[dl]_{p'}&\\ S_{K_2\cap hK_3h^{-1},\Sigma_a}^{tor}\ar[dr]^q& &S_{g^{-1}K_1g\cap K_2,\Sigma_b}^{tor}\ar[dl]_p\\& S_{K_2,\Sigma'}^{tor}&}
\end{eqnarray*}
which is cartesian away from the boundary.  We would like to know that
\begin{tiny} \begin{eqnarray*}
\xymatrix{&\mathrm{R}\Gamma({C''}^{tor},(q')^\star\mathcal{V}_{\kappa,g^{-1}K_1g\cap K_2,\Sigma_b})\ar[dr]^{\mathrm{tr}_{q'}}&\\
\mathrm{R}\Gamma(S_{K_2\cap hK_3h^{-1},\Sigma_a}^{tor},\mathcal{V}_{\kappa,K_2\cap hK_3h^{-1},\Sigma_a})\ar[dr]^{\mathrm{tr}_q}\ar[ur]^{(p')^\star}&&\mathrm{R}\Gamma(S_{g^{-1}K_1g\cap K_2,\Sigma_b}^{tor},\mathcal{V}_{\kappa,g^{-1}K_1g\cap K_2,\Sigma_b})\\
&\mathrm{R}\Gamma(S_{K_2,\Sigma'}^{tor},\mathcal{V}_{\kappa,K_2,\Sigma'})\ar[ur]^{p^\star}&}
\end{eqnarray*}
\end{tiny}
commutes, and similarly for cuspidal cohomology.  Note that here what we write as $(p')^\star$ is really the composition
$$\mathrm{R}\Gamma(S_{K_2\cap hK_3h^{-1},\Sigma_a}^{tor},\mathcal{V}_{\kappa,K_2\cap hK_3h^{-1},\Sigma_a})\to\mathrm{R}\Gamma({C''}^{tor},(p')^\star\mathcal{V}_{\kappa,K_2\cap hK_3h^{-1},\Sigma_a})\to$$
$$\mathrm{R}\Gamma({C''}^{tor},(q')^\star\mathcal{V}_{\kappa,g^{-1}K_1g\cap K_2,\Sigma_b})$$
where the first map is literally $(p')^\star$ and the second map is described on each component of ${C''}^{tor}$ as the action map $[k_i]$ (see the description of $p'$,$q'$ given in the proof of Proposition \ref{prop-geometric-comp}.)

We claim  that the following diagram commutes: 
\begin{eqnarray*}
\xymatrix{ \mathrm{R}(q')_\star (p')^\star q^\star \mathcal{V}_{\kappa, K_2\cap hK_3h^{-1},\Sigma_a} = \mathrm{R}(q')_\star (q')^\star p^\star \mathcal{V}_{\kappa, K_2\cap hK_3h^{-1},\Sigma_a}  \ar[r] & p^\star \mathcal{V}_{\kappa, g^{-1}K_1g\cap K_2,\Sigma_b} \\
 p^\star \mathrm{R} q_\star q^\star \mathcal{V}_{\kappa K_2\cap hK_3h^{-1},\Sigma_a} \ar[ur] \ar[u]& }
\end{eqnarray*}

This is a diagram  of locally free sheaves by theorem \ref{thm-van-between-toroidal},  and therefore the commutativity  can be checked away from the boundary, where this is clear.  
\end{proof}

\begin{rem}
We remark that for $K'\subseteq K\subset G(\mathbb{A}_f)$ open compact subgroups, the Hecke operator $[K'1K]$ and $[K1K']$ are the pullback and trace for the forgetful map $S_{K',\Sigma'}\to S_{K,\Sigma}$.  Moreover if $g\in G(\mathbb{A}_f)$ then $[Kg(g^{-1}Kg)]$ is simply the action map $[g]^\star$.  Finally for $K_1,K_2\subset G(\mathbb{A}_f)$ open compact subgroups and $g\in G(\mathbb{A}_f)$, we have a factorization $$[K_1gK_2]=[K_11(K_1\cap gK_2g^{-1})][(K_1\cap gK_2g^{-1})g(g^{-1}K_1g\cap K_2)][(g^{-1}K_1g\cap K_2)1K_2]$$
of $[K_1gK_2]$ into a pullback map, an action map, and a trace, which is essentially the definition of $[K_1gK_2]$.
\end{rem}

\subsubsection{Serre duality}
The dualizing sheaf of $S_{K,\Sigma}$ is   $\mathcal{V}_{-2\rho_{nc}, K, \Sigma}(-D_{K,\Sigma})$ where $\rho_{nc}$ is half the sum of the non-compact positive roots. Indeed, recall  our choice of non-compact positive roots being in $\mathfrak{g}/\mathfrak{p}_\mu^{std}$ and that $\mathfrak{g}/\mathfrak{p}_\mu^{std}$ is the tangent space at the identity of $FL_{G,\mu}^{std}$ and then apply \cite{MR1064864}, proposition 2.2.6. 

The Serre dual of the automorphic sheaf $\mathcal{V}_{\kappa, K, \Sigma}$ is therefore $\mathcal{V}_{-2\rho_{nc}- w_{0,M}\kappa, K, \Sigma}(-D_{K,\Sigma})$ where $w_{0,M}$ is the longest element of the Weyl group of $M_\mu$.  Serre duality is:

\begin{prop}\label{prop-classical-duality}
There is a Serre duality isomorphism
$$D_{F}(\mathrm{R}\Gamma(S_{K,\Sigma}^{tor}, \mathcal{V}_{\kappa, K, \Sigma}))[-d]\simeq\mathrm{R}\Gamma(S_{K,\Sigma}^{tor}, \mathcal{V}_{-2\rho_{nc}- w_{0,M_\mu}\kappa, K, \Sigma}(-D_{K,\Sigma}))$$
where $D_{F}(-)=\mathrm{RHom}_{F}(-,F)$ is the dualizing functor for $F$-vector spaces and $d$ is the dimension of $S_{K}$.  Moreover  this isomorphism, is compatible with the Hecke action in the sense the action of $[KgK]$ on the left matches the action of $[KgK]^t=[Kg^{-1}K]$ on the right.
\end{prop} 
\begin{proof}  For the existence of the duality pairing we refer to \cite{Hartshorne}. We simply prove the formula for the adjoint.  We denote by $$D_{K'}(-) = \underline{\mathrm{RHom}}(-,\mathcal{V}_{-2\rho_{nc}, K', \Sigma'''}(-D_{K',\Sigma'''}))$$ the dualizing functor on $S_{K',\Sigma'''}^{tor}$ for any compact open $K'$ (and cone decomposition $\Sigma'''$). Let $f = [K g K]$ be a characteristic function to which we associate a Hecke correspondence $ S^{tor}_{K, \Sigma} \stackrel{p_1}\leftarrow  S^{tor}_{K \cap g K g^{-1}, \Sigma''} \stackrel{p_2}\rightarrow S^{tor}_{K, \Sigma'}$. 
The action of $f$ on the cohomology arises from a cohomological correspondence (see section \ref{section-hecke-construction}): 
$$ f : p_2^\star \mathcal{V}_{\kappa, K, \Sigma} \stackrel{f_\kappa}\rightarrow p_1^\star\mathcal{V}_{\kappa, K, \Sigma} \stackrel{Id \otimes \mathrm{tr}_{p_1}}\rightarrow p_1^! \mathcal{V}_{\kappa, K, \Sigma}.$$
We find  (since duality switches $^\star$ and $^!$) that $$D_{K \cap g Kg^{-1} }(f) : p_1^\star D_K(\mathcal{V}_{\kappa, K, \Sigma}) \stackrel{Id \otimes \mathrm{tr}_{p_1}}\rightarrow p_1^! D_K(\mathcal{V}_{\kappa, K, \Sigma}) \rightarrow p_2^! D_K(\mathcal{V}_{\kappa, K, \Sigma}).$$
Remark that 
\begin{eqnarray*}
p_1^! D_K(\mathcal{V}_{\kappa, K, \Sigma}) &=& p_1^\star \mathcal{V}_{\kappa, K, \Sigma}^\vee \otimes p_1^! \mathcal{V}_{-2\rho_{nc}, K, \Sigma}(-D_{K,\Sigma}) \\
 p_2^! D_K(\mathcal{V}_{\kappa, K, \Sigma'}) &=& p_2^\star \mathcal{V}_{\kappa, K, \Sigma'}^\vee \otimes p_2^! \mathcal{V}_{-2\rho_{nc}, K, \Sigma}(-D_{K,\Sigma'})
 \end{eqnarray*}
 
 We have a canonical isomorphism $$\mathrm{Id} : p_1^! \mathcal{V}_{-2\rho_{nc}, \Sigma}(-D_{K,\Sigma})  = p_2^! \mathcal{V}_{-2\rho_{nc}, K, \Sigma'}(-D_{K,\Sigma'}),$$
 as both sheaves identify with the canonical sheaf of $ S^{tor}_{K \cap g K g^{-1}, \Sigma''}$. The map $p_1^! D_K(\mathcal{V}_{\kappa, K, \Sigma}) \rightarrow p_2^! D_K(\mathcal{V}_{\kappa, K, \Sigma})$ therefore writes  $ f_{\kappa}^\vee  \otimes \mathrm{Id}$. Let $f^t = [K g^{-1} K]$.  We observe that $f_\kappa^\vee = (f^t)_{-w_{0,M} \kappa}$ by definition (this identity boils down to $A^t = ((A^{-1})^{-1})^t$ for a matrix in $\mathrm{GL}_n$).
 We now claim that  we have a commutative diagram:
 
 \begin{eqnarray*}
 \xymatrix{  p_1^\star D_K(\mathcal{V}_{\kappa, K, \Sigma}) \ar[r]^{\mathrm{Id} \otimes \mathrm{tr}_{p_1}}& p_1^! D_K(\mathcal{V}_{\kappa, K, \Sigma})\ar[rr]^{ f_\kappa^{\vee} \otimes \mathrm{Id}} && p_2^! D_K(\mathcal{V}_{\kappa, K, \Sigma}) \\
   p_1^\star D_K(\mathcal{V}_{\kappa, K, \Sigma}) \ar[r] \ar[u]^{\mathrm{Id}} & p_1^\star D_K(\mathcal{V}_{\kappa, K, \Sigma}) \ar[u]^{\mathrm{Id} \otimes \mathrm{tr}_{p_1}} \ar[rr]^{\small{f^t_{-w_{0,M} \kappa - 2\rho_{nc}}}} & &p_2^\star D_K(\mathcal{V}_{\kappa, K, \Sigma}) \ar[u]^{\mathrm{Id} \otimes \mathrm{tr}_{p_2}}}
 \end{eqnarray*}
 This implies that $D_{K \cap g Kg^{-1} }(f) = f^t$. 
 The commutativity of the diagram now boils down to the commutativity of:
 \begin{eqnarray*}
 \xymatrix{  p_1^! \mathcal{V}_{-2\rho_{nc}, K, \Sigma}(-D_{K,\Sigma})  \ar[rr]^{\mathrm{Id}} && p_2^! \mathcal{V}_{-2\rho_{nc}, K, \Sigma}(-D_{K,\Sigma}) \\
 p_1^\star \mathcal{V}_{-2\rho_{nc}, K, \Sigma}(-D_{K,\Sigma}) \ar[rr]^{ f^t_{-2\rho_{nc}}} \ar[u]^{\mathrm{tr}_{p_1}} & &p_2^\star \mathcal{V}_{-2\rho_{nc}, K, \Sigma}(-D_{K,\Sigma})  \ar[u]^{\mathrm{tr}_{p_2}}}
 \end{eqnarray*}
 It is sufficient to prove the commutativity outside of the boundary. 
For any level $K$, the identification of the canonical sheaf of $S_{K}$ with    $\mathcal{V}_{-2\rho_{nc}, K}$  is functorial in the tower of Shimura varieties. 
Therefore, for the action map $[g] : S_{K \cap g K g^{-1}} \rightarrow S_{K \cap g^{-1} K g}$, we find that the map $[g]^\star  \mathcal{V}_{-2\rho_{nc}, K\cap g^-1 K g} \rightarrow  \mathcal{V}_{-2\rho_{nc}, K\cap g K g^{-1}}$ is the canonical isomorphism between canonical sheaves. We finally deduce that the map $(f^t)_{-2\rho_{nc}} : p_1^\star  \mathcal{V}_{-2\rho_{nc}, K} \rightarrow  p_2^\star \mathcal{V}_{-2\rho_{nc}, K}$ decomposes as: $$p_1^\star  \mathcal{V}_{-2\rho_{nc}, K} \rightarrow \mathcal{V}_{-2\rho_{nc}, K\cap g K g^{-1}} \rightarrow g^\star \mathcal{V}_{-2\rho_{nc}, K\cap g^{-1} K g} \rightarrow  p_2^\star  \mathcal{V}_{-2\rho_{nc}, K}$$ where the first map is induced by $\mathrm{tr}_{p_1}$, the second map is the canonical isomorphism, and the last map is induced by $g^\star( \mathrm{tr}_{p'_2})^{-1}$ where $p'_2 : S_{ K\cap g^{-1} K g} \rightarrow S_{K}$. This is telling us that the diagram commutes.
 \end{proof}

\begin{rem} One proves more generally that the adjoint of an Hecke operator $[K_1gK_2]$ for two (not necessarily equal) compact open  subgroup $K_1$ and $K_2$ is $[K_1gK_2]^t=[K_2g^{-1}K_1]$. Details are left to the reader.
\end{rem}

\subsubsection{The finite slope part of classical cohomology}\label{sect-finite-slope-classical-coho} We now assume that $G_{\qq_p}$ is quasi-split.  We assume that $K=K^p\times K_p$ where $K_p=K_{p,m,b}$ for $m\geq b\geq0$, $m>0$ is one of the subgroups with an Iwahori decomposition introduced in section \ref{section-compact-open-subgroups}.  We recall that for a choice of $+$ or $-$ we have commutative sub-algebras $\mathcal{H}_{p,m,b}^\pm$ of $\mathbb{Z}[K_{p,m,b} \backslash G(\qq_p) /K_{p,m,b}]$. The subalgebra $\mathcal{H}_{p,m,b}^\pm$ is generated by the double cosets $[K_{p,m,b} t K_{p,m,b}]$ with $t \in T^\pm$.  We have isomorphisms $\mathcal{H}_{p,m,b}^\pm = \ZZ[T^\pm/T_b]$. We also have the ideals $\mathcal{H}_{p,m,b}^{\pm\pm}$ generated by the double cosets $[K_{p,m,b} t K_{p,m,b}]$ with $t \in T^{\pm\pm}$.  The anti-involution of $\ZZ[K_{p,m,b} \backslash G(\qq_p) /K_{p,m,b}]$ defined by inversion exchanges  $\mathcal{H}_{p,m,b}^+$ with $\mathcal{H}_{p,m,b}^-$ and $\mathcal{H}_{p,m,b}^{++}$ with $\mathcal{H}_{p,m,b}^{--}$.

We let $$\mathrm{R}\Gamma(S_{K^pK_{p,m,b},\Sigma}^{tor}, \mathcal{V}_{\kappa, K^pK_{p,m,b}, \Sigma})^{\pm,fs} =$$ $$ \mathrm{R}\Gamma(S_{K^pK_{p,m,b},\Sigma}^{tor}, \mathcal{V}_{\kappa, K^pK_{p,m,b}, \Sigma}) \otimes^L_{\qq[T^\pm/T_b]}\qq[T(\qq_p)/T_b]$$ be the finite slope direct factor of $\mathrm{R}\Gamma(S_{K^pK_{p,m,b},\Sigma}^{tor}, \mathcal{V}_{\kappa, K^pK_{p,m,b}, \Sigma})$ for the operators in $[K_{p,m,b}tK_{p,m,b}]$ for $t\in T^\pm$.  We have a similar definition for cuspidal cohomology.  We note that the monoids $T^\pm$ have the property that for any $t\in T^{\pm\pm}$ and $s\in T^\pm$ there is an $s'\in T^\pm$ and an $n>0$ such that $ss'=t^n$.  It follows that the finite slope part can also be described as the finite slope part fo the single operator $[K_{p,m,b}tK_{p,m,b}]$ for any $t\in T^{\pm\pm}$.

We note that the Serre duality pairing of Proposition \ref{prop-classical-duality} restricts to a duality $$D_{F}(\mathrm{R}\Gamma(S_{K^pK_{p,m,b},\Sigma}^{tor}, \mathcal{V}_{\kappa, K^pK_{p,m,b}, \Sigma})^{\pm,fs})[-d]\simeq\mathrm{R}\Gamma(S_{K^pK_{p,m,b},\Sigma}^{tor}, \mathcal{V}_{-2\rho_{nc}- w_{0,M}\kappa, K^pK_{p,m,b}, \Sigma}(-D)^{\mp,fs})$$ on finite slope parts.

\begin{rem} Assume that the group $G_{\qq_p}$ is unramified. Then, as explained in remark \ref{rem-Weil-restriction} we can arrange so that the compact $K_{p,1,0}$ is an Iwahori, and the compact $K_{p,1,1}$ is a pro-$p$ Iwahori. Let $K_p = K_{p,1,1}$ or $K_{p,1,0}$.  All the Hecke operators $[K_p t K_p]$  for $t \in T^+$ or $t \in T^-$ are already invertible in $\qq[ K_p \backslash G(\qq_p) /K_p ]$ (see \cite[Cor. 1]{MR2122539})
 and so $$\mathrm{R}\Gamma(S_{K^pK_p,\Sigma}^{tor}, \mathcal{V}_{\kappa, K^pK_p, \Sigma})^{\pm, fs} = \mathrm{R}\Gamma(S_{K^pK_p,\Sigma}^{tor}, \mathcal{V}_{\kappa, K^pK_p, \Sigma})$$ and similarly for cuspidal cohomology.
\end{rem}
 
Next we recall how the finite slope part behaves under certain changes of level.  We first recall some classical relations in these Hecke algebras.

\begin{lem}\label{lem-iwahori-comp}
Let $K_1,K_2,K_3\subseteq G(\qq_p)$ be open compact subgroups with Iwahori decompositions $K_i=K_i^-K_i^0K_i^+$.  Let $t_1,t_2\in T(\qq_p)$.  Suppose that $t_1^{-1}K_1^-t_1\cap t_2K_3^-t_2^{-1}\subseteq K_2^-\subseteq t_1^{-1}K_1^-t_1$, $t_1^{-1}K_1^+t_1\cap t_2K_3^+t_2^{-1}K_2^+\subseteq t_2K_3^+t_3$ and $K_1^0\cap K_3^0\subseteq K_2^0\subseteq K_1^0K_3^0$.  Then $[K_1t_1K_2][K_2t_2K_3]=[K_1t_1t_2K_3]$
\end{lem}
\begin{proof}
This is an immediate consequence of proposition \ref{prop-conv-doublecoset} upon noting that our hypotheses imply $K_2=(K_2\cap t_1^{-1}K_1t_1)(K_2\cap t_2K_3t_2^{-1})$ and $t_1^{-1}K_1t_1\cap t_2K_3t_2^{-1}\subseteq K_2$.
\end{proof}

\begin{lem}\label{lemma-relations-Hecke-algebra}
Let $m'\geq b'\geq 0$ and $m\geq b\geq 0$ satisfy $m'\geq m>0$ and $b'\geq b$.
\begin{enumerate}
\item For all $t\in T^+$ we have $$[K_{p,m',b'}tK_{p,m',b'}][K_{p,m',b'}1K_{p,m,b}]=[K_{p,m',b'}1K_{p,m,b}][K_{p,m,b}tK_{p,m,b}].$$
\item For all $t\in T^-$ we have $$[K_{p,m,b}tK_{p,m,b}][K_{p,m,b}1K_{p,m',b'}]=[K_{p,m,b}1K_{p,m',b'}][K_{p,m',b'}tK_{p,m',b'}].$$
\item For all $t\in T^{++}$ with $\min(t)\geq1$ (see \ref{section-dynamics}) we have factorizations: 
$$[K_{p,m,b}tK_{p,m,b}]=[K_{p,m,b}tK_{p,m+1,b}][K_{p,m+1,b}1K_{p,m,b}]$$ and $$[K_{p,m+1,b}tK_{p,m+1,b}]=[K_{p,m+1,b}1K_{p,m,b}][K_{p,m,b}tK_{p,m+1,b}].$$
\item For all $t\in T^{--}$ with $\min(t)\geq 1$ we have factorizations: $$[K_{p,m,b}tK_{p,m,b}]=[K_{p,m,b}1K_{p,m+1,b}][K_{p,m+1,b}tK_{p,m,b}]$$ and $$[K_{p,m+1,b}tK_{p,m+1,b}]=[K_{p,m+1,b}tK_{p,m,b}][K_{p,m,b}1K_{p,m+1,b}].$$
\end{enumerate}
\end{lem}

\begin{proof}
We note that the second and fourth points are just the transposes of the first and third.  The first and third points are immediate consequences of lemma \ref{lem-iwahori-comp}.
\end{proof}

Here are the consequences on cohomology:
\begin{coro}\label{cor-classical-coh-diagrams}
Let $m'\geq b'\geq 0$ and $m\geq b\geq 0$ satisfy $m'\geq m>0$ and $b'\geq b$.
\begin{enumerate}
\item For all $t \in T^{+}$, the following diagram commutes:
\begin{eqnarray*}
\xymatrix{ \mathrm{R}\Gamma(S_{K^p K_{p,m',b'},\Sigma}^{tor}, \mathcal{V}_{\kappa, K^p K_{p,m',b'}, \Sigma}) \ar[rr]^{[K_{p,m',b'} t K_{p,m',b'}] } & &\mathrm{R}\Gamma(S_{K^p K_{p,m',b'},\Sigma}^{tor}, \mathcal{V}_{\kappa, K^p K_{p,m',b'}, \Sigma}) \\
\mathrm{R}\Gamma(S_{K^p K_{p,m,b},\Sigma}^{tor}, \mathcal{V}_{\kappa, K^p K_{p,m,b}, \Sigma}) \ar[rr]^{[K_{p,m,b} t K_{p,m,b}] } \ar[u] && \mathrm{R}\Gamma(S_{K^p K_{p,m,b},\Sigma}^{tor}, \mathcal{V}_{\kappa, K^p K_{p,m,b}, \Sigma})  \ar[u]}
\end{eqnarray*}
\item For all $m' \geq m$ and $t \in T^{-}$, the following diagram commutes:
\begin{eqnarray*}
\xymatrix{ \mathrm{R}\Gamma(S_{K^p K_{p,m',b'},\Sigma}^{tor}, \mathcal{V}_{\kappa, K^p K_{p,m',b'}, \Sigma}) \ar[rr]^{[K_{p,m',b'} t K_{p,m',b'}] }\ar[d]^{tr} & &\mathrm{R}\Gamma(S_{K^p K_{p,m',b'},\Sigma}^{tor}, \mathcal{V}_{\kappa, K^p K_{p,m',b'}, \Sigma})\ar[d]^{tr} \\
\mathrm{R}\Gamma(S_{K^p K_{p,m,b},\Sigma}^{tor}, \mathcal{V}_{\kappa, K^p K_{p,m,b}, \Sigma}) \ar[rr]^{[K_{p,m,b} t K_{p,m,b}] } && \mathrm{R}\Gamma(S_{K^p K_{p,m,b},\Sigma}^{tor}, \mathcal{V}_{\kappa, K^p K_{p,m,b}, \Sigma}) }
\end{eqnarray*}
\item For all $m$ and $t \in T^{++}$ with $\min(t)\geq 1$, there is a factorization:
\begin{eqnarray*}
\xymatrix{ \mathrm{R}\Gamma(S_{K^p K_{p,m+1,b},\Sigma}^{tor}, \mathcal{V}_{\kappa, K^p K_{p,m+1,b}, \Sigma}) \ar[rrd]\ar[rr]^{[K_{p,m+1,b} t K_{p,m+1,b}] } & &\mathrm{R}\Gamma(S_{K^p K_{p,m+1,b},\Sigma}^{tor}, \mathcal{V}_{\kappa, K^p K_{p,m+1,b}, \Sigma}) \\
\mathrm{R}\Gamma(S_{K^p K_{p,m,b},\Sigma}^{tor}, \mathcal{V}_{\kappa, K^p K_{p,m,b}, \Sigma}) \ar[rr]^{[K_{p,m,b} t K_{p,m,b]} } \ar[u] && \mathrm{R}\Gamma(S_{K^p K_{p,m,b},\Sigma}^{tor}, \mathcal{V}_{\kappa, K^p K_{p,m,b}, \Sigma})  \ar[u]}
\end{eqnarray*}
\item For all $m$ and $t \in T^{--}$ with $\min(t)\geq 1$, there is a factorization:
\begin{eqnarray*}
\xymatrix{ \mathrm{R}\Gamma(S_{K^p K_{p,m+1,b},\Sigma}^{tor}, \mathcal{V}_{\kappa, K^p K_{p,m+1,b}, \Sigma}) \ar[rr]^{[K_{p,m+1,b} t K_{p,m+1,b}] }\ar[d]^{tr} & &\mathrm{R}\Gamma(S_{K^p K_{p,m+1,b},\Sigma}^{tor}, \mathcal{V}_{\kappa, K^p K_{p,m+1,b}, \Sigma}) \ar[d]^{tr} \\
\mathrm{R}\Gamma(S_{K^p K_{p,m,b},\Sigma}^{tor}, \mathcal{V}_{\kappa, K^p K_{p,m,b}, \Sigma}) \ar[rr]^{[K_{p,m,b} t K_{p,m,b}] }\ar[rru]  && \mathrm{R}\Gamma(S_{K^p K_{p,m,b},\Sigma}^{tor}, \mathcal{V}_{\kappa, K^p K_{p,m,b}, \Sigma}) }
\end{eqnarray*}

\item We have the same results for cuspidal cohomology. 

\end{enumerate}
\end{coro}

We deduce the following classical corollary 

\begin{coro} 
\begin{enumerate}
\item For all $m' \geq m\geq b$ with $m>0$, the pullback map 
$$ \mathrm{R}\Gamma(S_{K^pK_{p,m,b},\Sigma}^{tor}, \mathcal{V}_{\kappa, K^pK_{p,m,b}, \Sigma})^{+, fs}  \rightarrow  \mathrm{R}\Gamma(S_{K^pK_{p,m',b},\Sigma}^{tor}, \mathcal{V}_{\kappa, K^pK_{p,m',b}, \Sigma})^{+, fs}$$ 
and the trace map
$$ \mathrm{R}\Gamma(S_{K^pK_{p,m',b},\Sigma}^{tor}, \mathcal{V}_{\kappa, K^pK_{p,m',b}, \Sigma})^{-, fs}  \rightarrow  \mathrm{R}\Gamma(S_{K^pK_{p,m,b},\Sigma}^{tor}, \mathcal{V}_{\kappa, K^pK_{p,m,b}, \Sigma})^{-, fs}$$ are quasi-isomorphisms, compatible with the action of $\qq[T(\qq_p)/T_b]$, and the same statements are true for cuspidal cohomology.  Moreover these isomorphisms are compatible with Serre duality.
\item For all $m\geq b'\geq b$ with $m>0$, the pullback map
$$ \mathrm{R}\Gamma(S_{K^pK_{p,m,b},\Sigma}^{tor}, \mathcal{V}_{\kappa, K^pK_{p,m,b}, \Sigma})^{+, fs}  \rightarrow  (\mathrm{R}\Gamma(S_{K^pK_{p,m,b'},\Sigma}^{tor}, \mathcal{V}_{\kappa, K^pK_{p,m,b'}, \Sigma})^{+, fs})^{T_b/T_{b'}}$$
and the trace map
$$ (\mathrm{R}\Gamma(S_{K^pK_{p,m,b'},\Sigma}^{tor}, \mathcal{V}_{\kappa, K^pK_{p,m,b'}, \Sigma})^{-, fs})^{T_b/T_{b'}}  \rightarrow  \mathrm{R}\Gamma(S_{K^pK_{p,m,b},\Sigma}^{tor}, \mathcal{V}_{\kappa, K^pK_{p,m,b}, \Sigma})^{-, fs}$$ are quasi-isomorphisms, compatible with the action of $\qq[T(\qq_p)/T_b]$, and the same statements are true for cuspidal cohomology.  Moreover these isomorphisms are compatible with Serre duality.
\end{enumerate}
\end{coro}
\begin{proof}
Immediate from corollary \ref{cor-classical-coh-diagrams} and theorem \ref{thm-van-between-toroidal}.
\end{proof}

Now let $\chi:T(\ZZ_p)\to \overline{F}^\times$ be a finite order character.  For all $m\geq b\geq\mathrm{cond}(\chi)$ with $m>0$, the spaces $\mathrm{R}\Gamma(S_{K^pK_{p,m,b},\Sigma}^{tor}, \mathcal{V}_{\kappa, K^pK_{p,m,b}, \Sigma})^{+, fs}[\chi]$ are canonically isomorphic.  We denote this space by $\mathrm{R}\Gamma(K^p,\kappa,\chi)^{+,fs}$.  We define in the same way $\mathrm{R}\Gamma(K^p,\kappa,\chi)^{-,fs}$, and the cuspidal versions $\mathrm{R}\Gamma(K^p,\kappa,\chi,cusp)^{+,fs}$ and $\mathrm{R}\Gamma(K^p,\kappa,\chi,cusp)^{-,fs}$

These are  the classical finite slope cohomologies, for the tame level $K^p$, $M_\mu$-dominant  algebraic weight $\kappa$, and nebentypus $\chi$.  They satisfy a Serre duality: $$D_{F}(\mathrm{R}\Gamma(K^p,\kappa,\chi)^{\pm,fs})[-d]\simeq \mathrm{R}\Gamma(K^p,-2\rho_{nc}- w_{0,M_\mu}\kappa,\chi^{-1},cusp)^{\mp,fs}.$$

\begin{rem}\label{rem-Nebentypus in the abelian case} We explain some  constraints on the possible nebentypus $\chi$ for which we  have non trivial cohomology in the case that the subgroup $Z_s(G)$ of the center $Z(G)$ of $G$ is non-trivial (a case that may only occur for  non Hodge type Shimura datum). For simplicity we put $Z=Z(G)$ and $Z_s = Z_s(G)$. 
Recall that $S_K(\C) = G(\qq) \backslash X \times G(\mathbb{A}_f)/ K$. Passing to the limit over $K$, we find that $S(\C) = \lim_K S_K(\C) =  G(\qq) \backslash X \times G(\mathbb{A}_f)/ \overline{Z(\qq)}$ where $\overline{Z(\qq)}$ is the closure of $Z(\qq)$ in $G(\mathbb{A}_f)$ (see \cite{MR546620}, sect. 2.19). 
Let $K' \subseteq K$ be two compact open subgroups, with $K'$ normal in $K$. The map $S_{K'}(\C)  \rightarrow S_{K}(\C)$ is a covering with Galois group $K/ (K'. \overline{Z(\qq)} \cap K)$.
If $K$ is neat, $\overline{Z(\qq)} \cap K$ is a subgroup of $\overline{Z_s(\qq)}$. In the case that $Z_s= \{1\}$, we deduce that 
$S_{K'}(\C)  \rightarrow S_{K}(\C)$ is  a Galois covering with group $K/K'$. 
In general, it follows  that the action of   $T(\ZZ_p)$ on $\mathrm{R}\Gamma(S_{K^pK_{p,m,b},\Sigma}^{tor}, \mathcal{V}_{\kappa, K^pK_{p,m,b}, \Sigma})^{+, fs}$ factors through an action of $T(\ZZ_p)/  Z'(K^p)$, where  $Z'(K^p) = \mathrm{Im} ((\overline{Z_s(\qq)} \cap K^pK_{p,m,0}) \rightarrow T(\ZZ_p))$.  
We deduce  that when $Z_s \neq 1$, if $\chi$ does not factor through $T(\ZZ_p)/  Z'(K^p)$, then  $\mathrm{R}\Gamma(K^p,\kappa,\chi)^{\pm,fs} =0$ for all $\kappa$. To avoid this situation, we will often impose that $\chi$ factors through a character of $T^c(\ZZ_p)$.
\end{rem}

\subsection{Jacquet Modules}\label{subsection-jacquet}  In this section we translate the finite slope condition into more representation theoretic terms.  We keep assuming that $G_{\qq_p}$ is quasi-split with Borel $B$. We let $U$ be the unipotent radical of $B$. 
Let $\pi$ be a smooth admissible representation of $G(\qq_p)$ with coefficient in a field of characteristic $0$. We let $\pi(U) \subseteq \pi$ be the submodule generated by the elements $n\cdot v-v$ for $n \in U(\qq_p)$ and $v \in \pi$.  We let $\pi_U = \pi / \pi(U)$ be the Jacquet module of $\pi$ (with respect to $U$).  This is a smooth admissible representation of $T(\qq_p)$ by \cite{CasselmanNotes}, thm. 3.3.1. Moreover, the functor $\pi \mapsto \pi_U$ is an exact functor by \cite{CasselmanNotes}, prop. 3.3.2.  We can define similarly the Jacquet module $\pi_{\overline{U}}$ with respect to $\overline{U}$. Note that conjugation by the longest element $w_0$ of the Weyl group realizes an isomorphism from $\pi_U$ to $\pi_{\overline{U}}$. 
Let $\psi : T(\qq_p) \rightarrow \C^\times$ be a continuous character. We let $\iota_B^G (\psi) = \{ f : G(\qq_p) \rightarrow \C,~\textrm{smooth},~f( b g) = \psi(b)f(g) \}$, equipped with the left action induced by right translation of $G(\qq_p)$ on itself.  We define similarly $\iota_{\overline{B}}^G (\psi)$. 

The adjunction formula of  \cite{CasselmanNotes}, thm. 3.2.4 states that $$\mathrm{Hom}_{G(\qq_p)} ( \pi, \iota_B^G(\psi)) = \mathrm{Hom}_{T(\qq_p)}( \pi_U, \psi)$$ and $$\mathrm{Hom}_{G(\qq_p)} ( \pi, \iota_{\overline{B}}^G(\psi)) = \mathrm{Hom}_{T(\qq_p)}( \pi_{\overline{U}}, \psi).$$

Let $K = K_{p,m,b}$. The algebra $\mathcal{H}_{p,m,b}^{\pm} = \ZZ[ T^{\pm}/T_b]$ (by lemma \ref{lemma-casselman}) acts  on $\pi^{K_{p,m,b}}$, and we can define $\pi^{K_{p,m,b}, \pm,fs} \subseteq \pi^{K_{p,m,b}}$ as the sub-vector space where the operators $[K_{p,m,b} t K_{p,m,b}]$ for $t \in T^\pm$ act invertibly. 

\begin{prop}\label{prop-casell-JM} The natural map $\pi^{K_{p,m,b}, +,fs} \rightarrow  \pi_U^{T_b}$  is an isomorphism which is $T^+/T_b$  equivariant, and  the natural map $\pi^{K_{p,m,b}, -,fs} \rightarrow  \pi_{\overline{U}}^{T_b}$ is an isomorphism which is $T^{-}/T_b$ equivariant. 
\end{prop}
\begin{proof} See \cite{CasselmanNotes}, Lemma 4.1.1 and proposition 4.1.4.
\end{proof}
\begin{prop} For a smooth irreducible  representation $\pi$ of $G(\qq_p)$, the following   properties are equivalent:
\begin{enumerate}
\item  There exists $m\geq  b \geq 0$  such that $\pi^{K_{p,m,b}, +,fs} \neq 0$,  
\item There exists $m\geq  b \geq 0$  such that $\pi^{K_{p,m,b}, -,fs} \neq 0$,
\item There exists a character $\psi$ of $T(\qq_p)$ such that  $\pi \hookrightarrow \iota_B^G \psi$,
\item  There exists a character $\psi'$ of $T(\qq_p)$ such that $\pi \hookrightarrow \iota_{\overline{B}}^G \psi'$.
\end{enumerate}
\end{prop}

\begin{proof} The points $(1)$ and $(2)$ are equivalent because the non-vanishing of $\pi^{K_{p,m,b}, \pm,fs} $ is equivalent to the non-vanishing of      $\pi_U^{T_b} $ and $\pi_{\overline{U}}^{T_b}$ respectively by proposition \ref{prop-casell-JM}. But conjugation by $w_0$ realizes an isomorphism between  these spaces. Similarly $(3)$ is equivalent to $(4)$. If we assume $(3)$, the adjunction formula shows that $\pi_U \neq 0$, hence there exists $b$ such that $\pi_U^{T_b} \neq 0$ and $\pi^{K_{p,m,b}, +,fs} \neq 0$ for any $m \geq b$ by proposition \ref{prop-casell-JM}. Conversely, if $\pi^{K_{p,m,b}, +,fs} \neq 0$, then $\pi_U^{T_b} \neq 0$ and by adjunction, there is a non zero map: $\pi \rightarrow \iota_B^G \psi$ for a character $\psi$. Since $\pi$ is irreducible, this map is injective. 
\end{proof}

If one of the equivalent properties of the proposition is satisfied, we say that an irreducible smooth representation $\pi$ is a \emph{finite slope} representation. 

Let $\pi$ be an admissible representation of $G(\qq_p)$. By adjunction, we have a morphism $\pi \rightarrow \iota_B^G \pi_U$ and we let $\pi^{fs}$ be the image of this morphism.  We call $\pi^{fs}$ the finite slope part of $\pi$. 

\begin{prop}The following properties are satisfied:

\begin{enumerate} 
\item The $G(\qq_p)$-representations $\pi^{fs}$ is a direct summand of $\pi$.
\item Any irreducible factor of $\pi^{fs}$ is   a finite slope representation.
\item Any irreducible factor of $\pi$ which is a finite slope representation lies in $\pi^{fs}$. 
\item $\pi^{fs}$ is the  sub-representation  of $\pi$ generated by the $\pi^{K_{p,m,b}, \pm,fs}$ for all $m \geq b \geq 0$.
\end{enumerate}
\end{prop}
\begin{proof} By the  Bernstein decomposition of the category of smooth  representations (\cite{renard},  VI.7.2)  we find that  $\pi = \pi' \oplus \pi''$  where $\pi'$ satisfies the properties $(1)$, $(2)$ and $(3)$ of the representation $\pi^{fs}$ of the proposition.
We have that $\pi''_U = 0$ because the Jacquet functor is exact  and  $\pi''$ has no finite slope irreducible sub-quotient. We deduce that $\pi_U = \pi'_U$. Moreover the morphism $\pi \rightarrow  \iota_B^G \pi_U$ factors into $\pi' \rightarrow \iota_B^G \pi_U$ and it follows again from the exactness of the Jacquet functor  that the map $\pi'  \rightarrow \iota_B^G \pi_U$ is injective. Therefore $\pi' = \pi^{fs}$. Let $\pi'''$ be the sub-representation of $\pi$ generated by $\pi^{K_{p,m,b}, \pm,fs}$ for all $m \geq b \geq 0$. We see that $\pi''' \subseteq \pi^{fs}$. But it follows from proposition \ref{prop-casell-JM} that $\pi'''_U = \pi^{fs}_U$. Therefore $\pi^{fs}/\pi'''$ has trivial Jacquet module, hence contains no finite slope sub-quotient and has to be trivial. 
\end{proof}

Let us denote by $$\HH^i(K^p, \kappa) = \colim_{K_p} \HH^i(S_{K^pK_p,\Sigma}^{tor}, \mathcal{V}_{\kappa, K^pK_p, \Sigma}),$$ $$\HH^i(K^p, \kappa,cusp) = \colim_{K_p} \HH^i(S_{K^pK_p,\Sigma}^{tor}, \mathcal{V}_{\kappa, K^pK_p, \Sigma}(-D_{K^pK_p, \Sigma}))$$ and $$\overline{\HH}^i(K^p, \kappa) = \mathrm{Im}(\HH^i(K^p, \kappa,cusp) \rightarrow \HH^i(K^p, \kappa)).$$

These are smooth admissible $G(\qq_p)$-representations. We can consider their finite slope parts $\HH^i(K^p, \kappa)^{fs}$, $\HH^i(K^p, \kappa, cusp)^{fs}$ and  $\overline{\HH}^i(K^p, \kappa)^{fs}$. These are direct summands of $\HH^i(K^p, \kappa)$, $\HH^i(K^p, \kappa, cusp)$ and  $\overline{\HH}^i(K^p, \kappa)$ respectively, and are   generated respectively as $G(\qq_p)$-representations  by the vector spaces $\HH^i(K^p, \kappa, \chi)^{\pm,fs}$, $\HH^i(K^p, \kappa, \chi, cusp)^{\pm,fs}$ and  $\overline{\HH}^i(K^p, \kappa, \chi)^{\pm,fs} = \mathrm{Im}(\HH^i(K^p, \kappa, \chi, cusp)^{\pm,fs} \rightarrow \HH^i(K^p, \kappa, \chi)^{\pm,fs})$ for all characters $\chi : T(\ZZ_p) \rightarrow \overline{F}^\times$.

\subsection{The Hodge-Tate period morphism} In this section we recall a number of results concerning the Hodge-Tate period morphism and infinite level Shimura varieties.  We now assume (unless explicitly mentioned) that $F$ is a finite extension of $\qq_p$ such that we have an embedding $E \hookrightarrow F$ and such that $G$ splits over $F$.  In this paper, the rationality questions with respect to $E$ are not very important. We will frequently allow ourselves to enlarge $F$ if necessary. Let $\mathcal{S}^{an}_K = (S_K \times \Spec~F)^{an}$, $\mathcal{S}^\star_K = (S^\star_K \times \Spec~F)^{{an}}$, $\mathcal{S}^{tor}_{K,\Sigma} = (S^{tor}_{K,\Sigma} \times \Spec~F)^{{an}}$, $ \mathcal{FL}_{G,\mu} = (FL_{G,\mu} \times \Spec~F)^{{an}}$ (see section \ref{section-analytic-geometry} for the meaning of the superscript $an$).   The first of these spaces is not quasi-compact if the Shimura variety is not proper, the  other three spaces are quasi-compact. 
We will also consider the groups $\mathcal{G}^{an} = (G \times \Spec~\qq_p)^{an}$, $\mathcal{P}_\mu^{an} = (P_\mu \times \Spec~F)^{an}$, $\mathcal{M}_\mu^{an}  =  (M_\mu \times \Spec~F)^{an}$.  

\subsubsection{Inverse limit of adic spaces} 
We start by a definition following  \cite{MR3272049}, sect. 2.4. Our definition is slightly more restrictive but fits in our setting. 

\begin{defi}\label{defi-tower} Let $\{ \mathcal{X}_i\}_{i \in I}$ be a  cofiltered inverse system of locally of finite type adic spaces over $\Spa(F, \ocal_F)$, with finite transition maps. Let $\mathcal{X}$ be a perfectoid space  with compatible maps $\mathcal{X} \rightarrow \mathcal{X}_i$. 

We say that $\mathcal{X} \sim \lim_{i\in I} \mathcal{X}_i$ if:
\begin{enumerate}
\item The maps $\mathcal{X} \rightarrow \mathcal{X}_i$ induces an homeomorphism of topological spaces $\vert \mathcal{X} \vert = \lim_i \vert \mathcal{X}_i\vert$.
\item There is a covering of $\mathcal{X}$ by open affinoids $U = \Spa(A, A^+)$ such  that  $U$ is the preimage of an affinoid $U_i = \Spa(A_i, A_i^+) \subseteq \mathcal{X}_i$   for a cofinal subset of $I$ and the map $\colim_i A_i \rightarrow A$ has dense image. 
\end{enumerate}
\end{defi}

If $\mathcal{X}  \sim \lim_{i\in I} \mathcal{X}_i$, then the diamond $ \lim_i \mathcal{X}_i^{\Diamond}$ is representable by the perfectoid space $\mathcal{X}$ by \cite{MR3272049}, prop. 2.4.5.  In particular, $\mathcal{X}$ is unique up to a unique isomorphism. 

In the notation of point $(2)$, we see that $A^0$ is a ring of definition of $A$ (because $A$ is uniform) and $A^0$ is the completion of $\colim_i A^0_i$ with respect to the $p$-adic topology.

\begin{defi}\label{defi-of-good} Let $\{ \mathcal{X}_i\}_{i \in I}$ be a  cofiltered inverse system of locally of finite type adic spaces over $\Spa(F, \ocal_F)$ with finite transition maps. Let $\mathcal{X}$ be a perfectoid space and assume that $\mathcal{X} \sim \lim_i \mathcal{X}_i$. 
We say that an open affinoid subset $U \hookrightarrow \mathcal{X}$ is good if  it satisfies the second property of definition \ref{defi-tower}. 
We say that an open affinoid $U_i  \hookrightarrow \mathcal{X}_i$ is pregood  if the open subset   $\mathcal{X} \times_{\mathcal{X}_i} \mathcal{U}_i$ of $\mathcal{X}$ is good.
\end{defi}

\begin{rem} We remark that a rational subset of a good open affinoid is also good by \cite{MR3204346}, proposition 2.22.  
\end{rem} 
\begin{rem} It is conjectured in   \cite{MR3204346}, conjecture 2.24 and proposition 2.26 that any open affinoid in $\mathcal{X}_i$ is pregood.
\end{rem}

We end this paragraph with  two  useful lemmas. 

\begin{lem}\label{lem-connected-perfectoid} Let $\{ \mathcal{X}_i\}_{i \in I}$ be a  cofiltered inverse system of locally of finite type adic spaces over $\Spa(F, \ocal_F)$ with finite transition maps.  For each $i$ let $\Pi_i$ be the set of connected components of $\mathcal{X}_i$ which we assume to be finite. Let $\Pi = \lim_i \Pi_i$. For any $e \in \Pi$ we get a cofiltered inverse system $\{ \mathcal{X}_{i,e} \}$.  If  there is a perfectoid space $\mathcal{X}$ such that $\mathcal{X} \sim \lim_i \mathcal{X}_i$ then for  all $e \in \Pi$, there is a perfectoid space $\mathcal{X}_e \sim \lim_i \mathcal{X}_{i,e}$.

\end{lem}
\begin{proof} We reduce to the affine case. Let $\mathcal{X}_i = \Spa(A_i,A_i^+)$. For each $e \in \Pi_i$, we have $\mathcal{X}_{i,e} = \Spa(A_{i,e},A_{i,e}^+)$ and $A_{i} = \prod_{e \in \Pi_i} A_{i,e}$. We may assume that   all rings $A_i$ are reduced (by taking the reduction). In particular $A_i^0$ is open and bounded.  This does not affect the $\sim$-limit because perfectoid spaces are reduced.   We may also assume that all maps $A_i \rightarrow A_j$ are injective for $i, j \in I$ and $j \mapsto i$ (replacing $A_i$ by its image in $A_j$).

We assume that there is a perfectoid space $\mathcal{X} = \Spa(A,A^+) \sim \lim_i \mathcal{X}_i$.  Then $A^0$  is the $p$-adic completion of $\colim_i A_i^0$ and $A^0$ is a perfectoid $\ocal_F$-algebra: there is $\varpi \in A^0$ with $\varpi^p \mid p$ and the Frobenius morphism  $\phi : A^0/\varpi^p \rightarrow A^0/\varpi^p$ is surjective.  Moreover, $A = A^0[1/p]$ and $A^+$ is the closure of $\colim A_i^+$ in $A$.  By approximation, we may assume that $\varpi \in A^0_{i}$ and by projection we get an element $\varpi \in A_{i,e}$. 

We need to see that the map  $ \phi : \colim_{i} A_{i,e}^0/\varpi^p \rightarrow \colim_{i}  A_{i,e}^0/\varpi^p$ is surjective. Let $x_{i,e} \in A_{i,e}/\varpi^p$. Since $A^0$ is perfectoid, we see that there exists $j \mapsto i$ and $y_{j,e} \in A_{j,e}/\varpi^p$ such that $y_{j,e}^p = x_{i,e}$.
\end{proof}

\begin{lem}\label{lem-group-perfectoid} Let $\{ \mathcal{X}_i\}_{i \in I}$ be a  cofiltered inverse system of locally of finite type separated adic spaces over a perfectoid field $\Spa(F, \ocal_F)$ with finite transition maps. Let $G$ be a finite group acting on the inverse system via $\Spa(F,\ocal_F)$-morphisms. Let $\mathcal{X}$ be a perfectoid space such that $\mathcal{X} \sim \lim_i \mathcal{X}_i$. Assume that   for some index $i$, we have a $G$-invariant covering of $\mathcal{X}_i$ by pregood affinoids.  Then the categorical quotient $\mathcal{Y}_i = \mathcal{X}_i/G$ is representable by an adic space for a cofinal subset of $I$, the categorial quotient $\mathcal{Y} = \mathcal{X}/G$ is representable by a perfectoid space, and $\mathcal{Y} \sim \lim_i \mathcal{Y}_i$.
\end{lem}

\begin{proof} The quotient $\mathcal{X}_i/G$ exist by  \cite{DHansen}, thm. 1.3. We may reduce to the affine case with $\mathcal{X} = \Spa(A,A^+)$ and $\mathcal{X}_i = \Spa(A_i, A_i^+)$.  By \cite{DHansen}, thm 1.4, $A^G$ is perfectoid. It is clear that $\colim A_i^{G}$ is dense in $A^G$ since we have a projector $A \rightarrow A^G$, $a \mapsto \frac{1}{\vert G \vert} \sum_{g \in G} g.a$. 
\end{proof}

\subsubsection{Siegel Shimura varieties}\label{section-torsors-Siegel}
We assume in this paragraph that $(G,X)$ is the Siegel Shimura datum $(\mathrm{GSp}_{2g}, \mathcal{H}_g)$. Let $K = K^p K_p \subseteq G(\mathbb{A}_f)$ be a compact open subgroup. The reflex field is $\qq$ and the Shimura variety $S_K$ is a moduli space of  abelian varieties $A$, with a level structure and polarization (prescribed by $K$). 

The Shimura variety ${S}_K$ carries a right pro-\'etale  $G(\qq_p)$-torsor. Namely, equip $\qq_p^{2g}$ with the standard symplectic form and consider   the torsor of isomorphisms $\qq_p^{2g} \rightarrow \HH_{1}(A, \qq_p)$, respecting the symplectic forms up to a similitude factor, where $A$ is the universal abelian scheme (defined up to isogeny) and $\HH_{1}(A, \qq_p) = V_p(A)$ is the rational Tate module of $A$. After choosing a geometric point $\overline{x} \rightarrow S_K$,  this torsor corresponds to a representation of the algebraic fundamental group $\pi_1(S_K, \overline{x}) \rightarrow G(\qq_p)$. The image of this morphism lies in the compact open subgroup $K_p \subset G(\qq_p)$ and the corresponding $K_p$-torsor is realized geometrically by the tower of Shimura varieties $\lim_{K'_p \subseteq K_p} S_{K^pK_p'}$.   
By pullback to the adic space $\mathcal{S}^{an}_K$, we get a pro-\'etale $G(\qq_p)$-torsor $\mathcal{G}^{an}_{pet,p}$.  If $K_p \subseteq G(\ZZ_p)$, this torsor has a $G(\ZZ_p)$-reduction of group structure that we denote by $\mathcal{G}_{pet,p}$.

The (relative) Hodge-Tate filtration  is the exact sequence of pro-\'etale sheaves over $\mathcal{S}^{an}_K$: 
$$ 0 \rightarrow \mathrm{Lie}(A) \otimes_{\oscr_{\mathcal{S}^{an}_K}} \hat{\oscr}_{\mathcal{S}^{an}_K} \rightarrow  \HH_1(A, \qq_p) \otimes_{\qq_p} \hat{\oscr}_{\mathcal{S}_K} \rightarrow \omega_{A^t}  \otimes_{\oscr_{\mathcal{S}^{an}_K}} \hat{\oscr}_{\mathcal{S}^{an}_K}  \rightarrow 0$$

 The groups $\mathcal{P}^{an}_{\mu}$, $\mathcal{G}^{an}$ and $\mathcal{M}^{an}_{\mu}$  are naturally sheaves over the \'etale site of $\mathcal{S}^{an}_K$. We extend them to  sheaves on the pro-\'etale site of $\mathcal{S}^{an}_K$ as follows: if $H$ is any of these groups, and  if $U \rightarrow \mathcal{S}^{an}_K $ is an object of the pro-\'etale site, we let $$H(U) = H(\hat{\oscr}_{\mathcal{S}^{an}_K}(U), \hat{\oscr}^+_{\mathcal{S}^{an}_K}(U)).$$ 
\begin{rem} There is also an extension to the pro-\'etale site of  $\mathcal{P}^{an}_{\mu}$, $\mathcal{G}^{an}$ and $\mathcal{M}^{an}_{\mu}$   where one takes sections with values in the uncompleted structure sheaf, but we will not need to consider this extension.
\end{rem}

The Hodge-Tate filtration gives a $\mathcal{P}^{an}_{\mu}$-reduction of structure group $\mathcal{P}^{an}_{HT}$ of the  $\mathcal{G}^{an}$-torsor $\mathcal{G}^{an}_{pet, p} \times^{\mathcal{G}^{an}(\qq_p)} \mathcal{G}^{an}$. Namely, we consider trivializations of $ \HH_1(A, \qq_p) \otimes_{\qq_p} \hat{\oscr}_{\mathcal{S}_K}$ which respect the filtration. This is a right $\mathcal{P}^{an}_{\mu}$-torsor.

We can consider the pushout $\mathcal{P}^{an}_{HT} \times^{\mathcal{P}^{an}_{\mu}} \mathcal{M}^{an}_\mu := \mathcal{M}^{an}_{HT}$. This pro-\'etale torsor actually identifies canonically with (the pull back to the pro-\'etale site of) the analytic  torsor  $\mathcal{M}^{an}_{dR}$, which is the analytification of $M_{dR}$ (see section \ref{section-automorphicvectorbundles}).

The  pro-\'etale $K_p$-torsor $\lim_{K'_p\subseteq K_p} S_{K^pK'_p}$ over $S_{K}$ extends to a pro-Kummer \'etale $K_p$-torsor $\lim_{K'_p\subseteq K_p} S^{tor}_{K^pK'_p,\Sigma}$ over $S^{tor}_{K,\Sigma}$. We can pull it back to the analytic space $\mathcal{S}_{K,\Sigma}^{tor}$ (see \cite{diao2019logarithmic} for the definition of the pro-Kummer \'etale site). By pushout along $K_p \rightarrow G(\qq_p)$, we get a pro-Kummer \'etale  $G(\qq_p)$-torsor over $\mathcal{S}^{tor}_{K,\Sigma}$ extending $\mathcal{G}^{an}_{pet,p}$, which we also denote by $\mathcal{G}^{an}_{pet,p}$. If $K_p \subseteq G(\ZZ_p)$ we also have the pro-Kummer \'etale $G(\ZZ_p)$-torsor $\mathcal{G}_{pet,p}$ over $\mathcal{S}^{tor}_{K,\Sigma}$.

Let $A_\Sigma$ be the semi-abelian scheme over $\mathcal{S}^{tor}_{K,\Sigma}$. The Hodge-Tate exact sequence extends to a sequence over the pro-Kummer \'etale site  $$ 0 \rightarrow \mathrm{Lie}(A_\Sigma) \otimes_{\oscr_{\mathcal{S}^{tor}_{K,\Sigma}}} \hat{\oscr}_{\mathcal{S}^{tor}_{K,\Sigma}} \rightarrow  \HH_1(A_\Sigma, \qq_p) \otimes_{\qq_p} \hat{\oscr}_{\mathcal{S}^{tor}_{K,\Sigma}} \rightarrow \omega_{A_\Sigma^t}  \otimes_{\oscr_{\mathcal{S}^{tor}_{K,\Sigma}}} \hat{\oscr}_{\mathcal{S}^{an}_K}  \rightarrow 0$$

 Therefore, the torsors $\mathcal{P}^{an}_{HT}$ and $\mathcal{M}^{an}_{HT}$ extend over $\mathcal{S}^{tor}_{K,\Sigma}$, and as usual, we continue to use the same notation for the extensions.  
Moreover, again  by construction, the torsors $\mathcal{M}^{an}_{HT}$ and $\mathcal{M}^{an}_{dR}$ are canonically identified. 

\subsubsection{Perfectoid Siegel Shimura varieties}\label{section-perfectoid-siegel}
By \cite{scholze-torsion}, thm. III.3.17 there is a  perfectoid space $\mathcal{S}^\star_{{K^p}}  \sim \lim_{K_p} \mathcal{S}^\star_{{K^pK_p}}$.

By \cite{scholze-torsion}, we have a $G(\qq_p)$-equivariant map  $\pi_{HT} : \mathcal{S}^\star_{{K^p}} \rightarrow \mathcal{FL}_{G,\mu}$.   Moreover,  there exists an affinoid covering $\mathcal{FL}_{G,\mu} = \cup_i V_i$ such that for each $i$, $\pi_{HT}^{-1}(V_i)$ is a good affinoid perfectoid open subset of $\mathcal{S}^\star_{{K^p}}$ (see definition \ref{defi-of-good}).

The construction of the map $\pi_{HT} : \mathcal{S}^\star_{{K^p}} \rightarrow \mathcal{FL}_{G,\mu}$ is delicate at the boundary, but over the complement of the boundary $
\mathcal{S}^{an}_{{K^p}} \sim \lim_{K_p} \mathcal{S}^{an}_{K^pK_p}$ it has a simple description  which is given below.

 By \cite{MR3512528}, thm. 0.4,  for any cone decomposition $\Sigma$, there is also  a  perfectoid space $\mathcal{S}^{tor}_{{K^p},\Sigma}  \sim \lim_{K_p} \mathcal{S}^{tor}_{{K^pK_p},\Sigma}$ (it is important that the cone decomposition does not depend on $K_p$ for the limit to be perfectoid) and  we have a map $\mathcal{S}^{tor}_{{K^p},\Sigma}  \rightarrow \mathcal{S}^\star_{{K^p}} $ of perfectoid spaces induced by the maps at finite level $K_p$.

The torsor $\mathcal{G}^{an}_{pet, p}$ becomes trivial over $\mathcal{S}^{tor}_{{K^p},\Sigma}$ and we therefore get a Hodge-Tate period map $\pi^{tor}_{HT} : \mathcal{S}^{tor}_{{K^p},\Sigma} \rightarrow \mathcal{FL}_{G,\mu}$.  Let us explain very concretely how this map is defined. Let $\Spa(R,R^+)$ be a perfectoid affinoid open subset of $\mathcal{S}^{tor}_{{K^p},\Sigma}$. We can evaluate the sequence $$ 0 \rightarrow \mathrm{Lie}(A_\Sigma) \otimes_{\oscr_{\mathcal{S}^{tor}_{K,\Sigma}}} \hat{\oscr}_{\mathcal{S}^{tor}_{K,\Sigma}} \rightarrow  \HH_1(A_\Sigma, \qq_p) \otimes_{\qq_p} \hat{\oscr}_{\mathcal{S}^{tor}_{K,\Sigma}} \rightarrow \omega_{A_\Sigma^t}  \otimes_{\oscr_{\mathcal{S}^{tor}_{K,\Sigma}}} \hat{\oscr}_{\mathcal{S}^{an}_K}  \rightarrow 0$$
on $(R,R^+)$ (viewed as an object of the pro-Kummer-\'etale site of $\mathcal{S}^{tor}_{K,\Sigma}$) and use the trivialization $ \qq_p^{2g} \simeq \HH_1(A_\Sigma, \qq_p)$ to get an exact sequence:
$$0 \rightarrow   \mathrm{Lie}(A_\Sigma)\otimes R \rightarrow R^{2g} \rightarrow  \omega_{A_\Sigma^t} \otimes R \rightarrow 0$$
After localizing, we may even assume that $ \mathrm{Lie}(A_\Sigma)\otimes R$ and $\omega_{A_\Sigma^t} \otimes R$ are free $R$-modules. 
Now let $0 \rightarrow R^g \rightarrow R^{2g} \rightarrow R^g \rightarrow 0$ be the (polarized) chain with automorphism group  $P_\mu(R)$. 
We have that $\mathcal{P}_{HT}^{an}(R,R^+) = $ $$\mathrm{Isom}_{\mathrm{symp}} ( 0 \rightarrow R^g \rightarrow R^{2g} \rightarrow R^g \rightarrow 0, 0 \rightarrow   \mathrm{Lie}(A_\Sigma)\otimes R \rightarrow R^{2g} \rightarrow  \omega_{A_\Sigma^t} \otimes R \rightarrow 0) $$
and  $\mathcal{P}_{HT}^{an}(R,R^+) \subseteq  \mathrm{Isom}_{\mathrm{symp}}(R^{2g}) = \mathrm{GSp}_{2g}(R).$ 
This is a right $P_\mu(R)$-torsor and there is an element $x \in G(R)$ such that $\mathcal{P}_{HT}^{an}(R,R^+) = x P_\mu(R)$. The automorphism group of $0 \rightarrow   \mathrm{Lie}(A_\Sigma)\otimes R \rightarrow R^{2g} \rightarrow  \omega_{A_\Sigma^t} \otimes R \rightarrow 0$ is $x P_\mu(R) x^{-1}$.  Finally  we let $\pi_{HT}^{tor} (\Spa(R,R^+)) = x^{-1} \in FL_{G,\mu}(R)$. 
 
 \begin{rem} We are forced to use $x^{-1}$ above because $ FL_{G,\mu} = P_\mu \backslash G$. Note that taking the quotient by the left action of $P_\mu$ is natural because right translation on $G$ defines a right $G$-action of $FL_{G,\mu}$ and  the map $\pi_{HT}  : \mathcal{S}^\star_{{K^p}} \rightarrow \mathcal{FL}_{G,\mu}$ is equivariant for the right $G(\qq_p)$-action.  We chose to  define $\mathcal{P}_{HT}^{an}$ as a right $\mathcal{P}_\mu^{an}$-torsor, because we want to identify the torsors  $\mathcal{M}_{HT}^{an}$  and $\mathcal{M}_{dR}^{an}$. But  in the classical theory, $M_{dR}$ is a right torsor.  It means that in our convention, the torsor $\mathcal{P}_{HT}^{an}$ is pulled back via $\pi_{HT}^{tor}$ from the torsor $\mathcal{G}^{an} \rightarrow \mathcal{FL}_{G,\mu}$, $x \mapsto x^{-1}$. 
 \end{rem}

The maps $\pi^{tor}_{HT} $ and $\pi_{HT}$ coincide by construction on the open subset  $\mathcal{S}^{an}_{{K^p}}$. We deduce that we have a commutative diagram:
\begin{eqnarray*}
\xymatrix{ \mathcal{S}^{tor}_{{K^p},\Sigma} \ar[rd]^{\pi_{HT}^{tor}} \ar[d] & \\
\mathcal{S}^{\star}_{{K^p}} \ar[r]^{\pi_{HT}}& \mathcal{FL}_{G,\mu}}
\end{eqnarray*}

The key properties of this diagram that we will  use are:
\begin{itemize}
 \item The pull back of the torsor $\mathcal{G}^{an}/U_{\mathcal{P}_\mu^{an}} \rightarrow  \mathcal{FL}_{G,\mu}$ via $\pi_{HT}^{tor}$ is $\mathcal{M}_{HT}^{an}$ and this is canonically identified with the pull back via $\mathcal{S}^{tor}_{{K^p},\Sigma} \rightarrow \mathcal{S}^{tor}_{{K^pK_p},\Sigma}$ of $\mathcal{M}_{dR}^{an}$. 
\item The map $\pi_{HT}$ is affine.
\end{itemize}

\subsubsection{Formal models of perfectoid Siegel Shimura varieties}\label{sect-formal-perf-Siegel} We need to consider formal models of the perfectoid Siegel Shimura varieties in order to be able  to use the vanishing result below (theorem \ref{thm-formal-vanishingtominimal}). We first recall a number of statements from \cite{MR3512528}.
Let $K = K^p K_p$.   For $K_p = \mathrm{GSp}_{2g}(\ZZ_p)$, we have natural models over $\Spec~\ZZ_p$ for $S_{K^pK_p}$, $S^\star_{K^pK_p}$, and $S^{tor}_{K^pK_p,\Sigma}$ that we denote $\mathbf{S}_{K^pK_p}$, $\mathbf{S}^\star_{K^pK_p}$, and $\mathbf{S}^{tor}_{K^pK_p,\Sigma}$. We can also consider the corresponding $p$-adic formal schemes $\mathfrak{S}_{K^pK_p}$, $\mathfrak{S}^\star_{K^pK_p}$, and $\mathfrak{S}^{tor}_{K^pK_p,\Sigma}$.  We denote by $A$ the semi-abelian scheme over $\mathfrak{S}^{tor}_{K^pK_p,\Sigma}$  (it is defined up to prime-to-$p$ isogeny by our choice of level structure).  By the construction of the minimal compactification, the line bundle $\det\omega_A$ on $\mathfrak{S}_{K^pK_p}$ extends to a line bundle that we also denote by $\det\omega_A$ on $\mathfrak{S}^\star_{K^pK_p}$, and whose pullback to $\mathfrak{S}^{tor}_{K^pK_p,\Sigma}$ agrees with the extension defined by the semi-abelian scheme $A$.

Now let $K_p \subseteq \mathrm{GSp}_{2g}(\ZZ_p)$ be an open subgroup. We can define $\mathfrak{S}_{K^pK_p}$, $\mathfrak{S}^\star_{K^pK_p}$, and $\mathfrak{S}^{tor}_{K^pK_p,\Sigma}$ as the normalizations of $\mathfrak{S}_{K^p\mathrm{GSp}_{2g}(\ZZ_p)}$, $\mathfrak{S}^\star_{K^p\mathrm{GSp}_{2g}(\ZZ_p)}$, and $\mathfrak{S}^{tor}_{K^p\mathrm{GSp}_{2g}(\ZZ_p),\Sigma}$ in $\mathcal{S}_{K^pK_p}$, $\mathcal{S}^\star_{K^pK_p}$, and $\mathcal{S}^{tor}_{K^pK_p,\Sigma}$ respectively. 

Let us denote by $K_{p,n} = \{ M \in  \mathrm{GSp}_{2g}(\ZZ_p),~M = 1 ~\mod p^n\}$. This is the principal level $p^n$ subgroup. Consider the module $\mathbb{Z}^{2g}$
with canonical basis $e_1, \cdots,e_{2g}$, equipped with the standard symplectic form $\langle, \rangle$ given by $\langle e_{i}, e_{2g-i+1} \rangle =1$ for $1\leq i\leq g$, and  $\langle e_{i}, e_{j} \rangle =0$ if $i+j \neq 2g+1$. Over $\mathcal{S}_{K^pK_{p,n}}$ we have a symplectic isomorphism $(\ZZ/p^n\ZZ)^{2g} \rightarrow A[p^n]$ (up to a similitude factor), and it extends to a morphism of group schemes  over  $\mathfrak{S}_{K^pK_{p,n}}$, $(\ZZ/p^n\ZZ)^{2g} \rightarrow A[p^n]$. 
We have the Hodge-Tate morphism $\mathrm{HT}_n : A[p^n] \rightarrow \omega_{A^t}/p^n$. Using the prime-to-$p$ polarization, we can identify $\omega_{A^t}$ and $\omega_{A}$. We therefore have sections $\mathrm{HT}_n(e_i) \in \HH^0( \mathfrak{S}_{K^pK_{p,n}}, \omega_{A}/p^n)$. We also have a map $\Lambda^g  \mathrm{HT}_n : \Lambda^g(\ZZ/p^n\ZZ)^{2g} \rightarrow \det \omega_A/p^n$. Let  $r=$${g}\choose{2g}$. Let  $f_1, \cdots, f_r$ be a basis of $\Lambda^g \ZZ^{2g}$ obtained by taking exterior products of $e_1, \cdots, e_{2g}$. We get sections $\Lambda^g \mathrm{HT}_n(f_i) \in  \HH^0( \mathfrak{S}_{K^pK_{p,n}}, \det \omega_{A}/p^n)$.

By \cite{MR3512528},  proposition 1.7 and corollary 1.7, the sections $\mathrm{HT}_n(e_i)$ extend to sections $\mathrm{HT}_n(e_i) \in \HH^0( \mathfrak{S}^{tor}_{K^pK_{p,n},\Sigma}, \omega_{A}/p^n)$ and the sections $\Lambda^g \mathrm{HT}_n(f_i)$ extend to sections $\Lambda^g \mathrm{HT}_n(f_i) \in  \HH^0( \mathfrak{S}^{\star}_{K^pK_{p,n}}, \det \omega_{A}/p^n)$. 

It follows that over $\mathfrak{S}^{\star}_{K^pK_{p,n}}$ we have a morphism $\Lambda^g  \mathrm{HT}_n : \Lambda^g (\ZZ/p^n\ZZ)^{2g} \rightarrow \det \omega_A/p^n$. These morphisms satisfy the natural compatibilities as $n$ varies. The cokernel of this morphism is killed by $p^{\frac{g}{p-1}}$ (resp. $4^{g}$ if $p=2$) by \cite{MR2673421}, theorem 7.
We let $\det \omega_A^{mod,n}$ be the subsheaf of $\det \omega_A$ which is the inverse image of $$ \mathrm{Im}(\Lambda^g  \mathrm{HT}_n  (  \Lambda^g (\ZZ/p^n\ZZ)^{2g}) \otimes \oscr_{\mathfrak{S}^{\star}_{K_{p,n}K^p}} \rightarrow \det \omega_A/p^n)$$ in $\det \omega_{A}$.   We have $p^{\frac{g}{p-1}}  \det \omega_A \subseteq \det \omega_A^{mod,n} \subseteq \det \omega_A$
(resp. $4^{{g}}  \det \omega_A \subseteq \det \omega_A^{mod,n} \subseteq \det \omega_A$ if $p=2$) .

 Let $f : \mathfrak{S}^\star_{K^pK_{p,n+1}} \rightarrow \mathfrak{S}^\star_{K^pK_{p,n}}$ be the projection. We clearly have a map $\det \omega_A^{mod,n+1} \rightarrow \mathrm{Im} (f^\star \det \omega_A^{mod,{n}} \rightarrow \det \omega_A)$. This map is an isomorphism    if $n \geq \frac{g}{p-1}$ (resp. $n \geq 2g$ if $p=2$). We now assume that $n$ is larger than $n_0 = \frac{g}{p-1}$ (resp. $n_0=2g$ if $p=2$).   We simply denote $\det \omega_A^{mod,n}$ by $\det \omega_A^{mod}$, this is a  subsheaf of $\det \omega_A$.  The sheaf $\det \omega_A^{mod}$ is not locally free. We can perform a blow up to make it locally free. Let us start with a definition: 

\begin{defi} Let $\mathfrak{X}$ be a $p$-adic formal scheme, locally of finite type  over $\Spf~\ZZ_p$,  and let $\mathcal{I}$ be a coherent sheaf of ideals such that $p \in \sqrt{\mathcal{I}}$. We let $\mathrm{BL}_{\mathcal{I}}(\mathfrak{X})$ be the $p$-adic formal scheme obtained by taking the admissible blow-up of $\mathfrak{X}$ at $\mathcal{I}$.  We let $\mathrm{NBL}_{\mathcal{I}}(\mathfrak{X})$ be the normalization of $\mathrm{BL}_{\mathcal{I}}(\mathfrak{X})$, this is the normalized blow-up. 
\end{defi}

\begin{rem} The normalization of a formal scheme is well defined in our context by \cite{MR1697371}, coro. 1.2.3.
\end{rem}

Let $\mathcal{I}_n$ be the sheaf of ideals of $\oscr_{\mathfrak{S}^{\star}_{K^pK_{p,n}}}$ given by $\mathcal{I}_n= \{ a \in \oscr_{\mathfrak{S}^{\star}_{K^pK_{p,n}}}, a \det \omega_A \subseteq \det \omega_{A}^{mod} \}$.  
We let ${ \mathfrak{S}}^{\star,mod}_{K^pK_{p,n}} = \mathrm{NBL}_{\mathcal{I}_n}( \mathfrak{S}^{\star}_{K^pK_{p,n}})$.
Over ${ \mathfrak{S}}^{\star,mod}_{K^pK_{p,n}}$, the sheaf $\det \omega_{A}^{mod}$ is locally free and  moreover the map  $$\Lambda^g \mathrm{HT}_n : \Lambda^g (\ZZ/p^n\ZZ)^{2g} \rightarrow \det \omega_{A}/p^n$$ induces a surjective map   $$\Lambda^g \mathrm{HT}'_n : \Lambda^g (\ZZ/p^n\ZZ)^{2g} \rightarrow \det \omega^{mod}_{A}/p^{n-\frac{g}{p-1}}$$ (resp. $ \Lambda^g \mathrm{HT}'_n : \Lambda^g (\ZZ/p^n\ZZ)^{2g} \rightarrow \det \omega^{mod}_{A}/p^{n-2{g}}$ if $p=2$). 

\begin{lem} Let $n \geq n_0$. We have a commutative diagram 
\begin{eqnarray*}
\xymatrix{ \mathfrak{S}^{\star,mod}_{K^pK_{p,n+1}} \ar[r] \ar[d]^{h} & \mathfrak{S}^{\star}_{K^pK_{p,n+1}} \ar[d]^{f} \\
\mathfrak{S}^{\star,mod}_{K^pK_{p,n}} \ar[r] & \mathfrak{S}^{\star}_{K^pK_{p,n}}}
\end{eqnarray*}
where the vertical maps are finite.
Moreover, $h^\star \det \omega_{A}^{mod} = \det \omega_A^{mod}$ and $h^\star \Lambda^g \mathrm{HT}'_n(f_i) = \Lambda^g \mathrm{HT}'_{n+1}(f_i)  \in \HH^0( \mathfrak{S}^{\star,mod}_{K^pK_{p,n+1}}, \det \omega_{A}^{mod}/p^{n-\frac{g}{p-1}})$ (resp. $\in \HH^0( \mathfrak{S}^{\star,mod}_{K^pK_{p,n+1}}, \det \omega_{A}^{mod}/p^{n-2g})$ if $p=2$). 
\end{lem} 

\begin{proof}   We first observe that $\mathrm{Im}( f^\star \mathcal{I}_n \rightarrow \oscr_{\mathfrak{S}^{\star}_{K^pK_{p,n+1}}}) = \mathcal{I}_{n+1}$. 
It follows that we have maps $\mathrm{BL}_{\mathcal{I}_n}{\mathfrak{S}}^{\star}_{K^pK_{p,n}}  \times_{ \mathfrak{S}^{\star,mod}_{K^pK_{p,n}}}  \mathfrak{S}^{\star}_{K^pK_{p,n+1}} \rightarrow \mathrm{BL}_{\mathcal{I}_{n+1}}{\mathfrak{S}}^{\star}_{K^pK_{p,n+1}} \rightarrow  \mathrm{BL}_{\mathcal{I}_{n}}{\mathfrak{S}}^{\star}_{K^pK_{p,n}}$. 
The map $\mathrm{BL}_{\mathcal{I}_{n+1}}{\mathfrak{S}}^{\star}_{K^pK_{p,n+1}} \rightarrow  \mathrm{BL}_{\mathcal{I}_{n}}{\mathfrak{S}}^{\star}_{K^pK_{p,n}}$ is therefore finite, and so is the map 
$h : {\mathfrak{S}}^{\star,mod}_{K^pK_{p,n+1}} \rightarrow  {\mathfrak{S}}^{\star,mod}_{K^pK_{p,n}}$.
We have a surjective map  $ h^\star \det \omega_{A}^{mod}  \rightarrow \det \omega_{A}^{mod}$ of invertible sheaves. This map is therefore an isomorphism. The last compatibility follows from the property that $f^\star \Lambda^g \mathrm{HT}_n(f_i) = \Lambda^g \mathrm{HT}_{n+1}(f_i)  \in \HH^0( \mathfrak{S}^{\star}_{K^pK_{p,n+1}}, \det \omega_{A}/p^{n})$. 
\end{proof}

We can similarly define $\mathfrak{S}^{\star, mod}_{K^pK_p}$ for any open $K_p \subseteq K_{p,n_0}$. The ideal $\mathcal{I}_{n_0}$ pulls back to an ideal $\mathcal{I}_{K_p}$ of $\mathfrak{S}^{\star}_{K^pK_p}$ and we let $\mathfrak{S}^{\star,mod}_{K^pK_p} = \mathrm{NBL}_{\mathcal{I}_{K_p}}(\mathfrak{S}^{\star}_{K^pK_p})$. If $K_p = K_{p,n}$ for $n\geq n_0$, we recover the previous definition.  

We finally let ${ \mathfrak{S}}^{\star,mod}_{K^p} = \lim_n { \mathfrak{S}}^{\star,mod}_{K^pK_{p,n}}$ where the inverse limit is taken in the category of $p$-adic formal schemes.  This inverse limit exists because the transition morphisms are affine. In the limit we have a  map $\Lambda^g \mathrm{HT} : \Lambda^g \ZZ_p^{2g}  \rightarrow \det \omega_A^{mod}$ whose linearization is surjective. It follows that  we have a morphism $\pi_{HT} : { \mathfrak{S}}^{\star,mod}_{K^p} \rightarrow \mathbb{P}^{r-1}$. Let $X_1, \cdots, X_r$ be the homogeneous coordinates on $\mathbb{P}^{r-1}$. For all $1 \leq i \leq r$, let $\mathfrak{U}_i$ be the formal open subscheme defined by the condition $X_i \neq 0$.  Let $\mathcal{U}_i$ be its generic fiber. 

\begin{prop}\label{prop-bretagne} \begin{enumerate}
\item The formal scheme $ { \mathfrak{S}}^{\star,mod}_{K^p}$ is integral perfectoid and its generic fiber is the perfectoid space $\mathcal{S}^\star_{K^p}$. 
 \item The Hodge-Tate map factors through a map $\pi_{HT} : \mathfrak{S}^{\star,mod}_{K^p} \rightarrow \mathfrak{FL}_{G, \mu}$ and it induces the Hodge-Tate period map 
 $\pi_{HT} : \mathcal{S}^{\star}_{K^p} \rightarrow \mathcal{FL}_{G, \mu}$ of section \ref{section-perfectoid-siegel} on the generic fiber. 
\item For all $1 \leq i \leq r$, we have that $\pi_{HT}^{-1}(\mathcal{U}_i)$ is a good open affinoid subset of $\mathcal{S}^\star_{K^p}$.  
\end{enumerate}
\end{prop}
\begin{proof} See \cite{MR3512528}, thm 1.18.  This relies  on the main theorems of \cite{scholze-torsion}.
\end{proof}

\begin{rem} For $n$ large enough,  the open subsets $\pi_{HT}^{-1}(\mathcal{U}_i)$ come from open subsets $\pi_{HT}^{-1}(\mathcal{U}_i)_{K_{p,n}} \hookrightarrow \mathcal{S}_{K^pK_{p,n}}^{\star}$. One can define (\cite{scholze-torsion}, p. 72) a formal model $\mathfrak{S}_{K^pK_{p,n}}^{\star-HT}$ by gluing the  formal schemes $\Spf~\HH^0(\pi_{HT}^{-1}(\mathcal{U}_i)_{K_{p,n}}, \oscr_{\mathcal{S}_{K^pK_{p,n}}}^+)$.  We will not use this formal model.
\end{rem}

We can perform  similar  constructions with the toroidal compactification at level $K_p \subseteq K_{p,n_0}$. Namely, the ideal $\mathcal{I}_{K_p}$ pulls back to an ideal $\mathcal{J}_{K_p}$ of $\mathfrak{S}^{tor}_{K^pK_p,\Sigma}$. 
We denote by $\mathfrak{S}^{tor,mod}_{K^pK_p,\Sigma} = \mathrm{NBL}_{\mathcal{J}_{K_p}}({\mathfrak{S}}^{tor}_{K^pK_p,\Sigma})$. We have  natural morphisms $\mathfrak{S}^{tor,mod}_{K^pK_p,\Sigma} \rightarrow \mathfrak{S}^{\star,mod}_{K^pK_p}$. 

\begin{rem}\label{rem-minor-difference} The space $\mathfrak{S}^{tor,mod}_{K^pK_p,\Sigma}$ is not exactly the space considered in \cite{MR3512528}. Namely, in that reference, we considered further  blow ups in order to make the sheaf denoted $\omega_A^{mod}$ in \emph{loc. cit.} (the subsheaf of $\omega_A$ generated by the image of the Hodge-Tate period map) locally free. 
\end{rem}

We  let ${ \mathfrak{S}}^{tor,mod}_{K^p,\Sigma} = \lim_n { \mathfrak{S}}^{tor,mod}_{K^pK_{p,n},\Sigma}$ where the inverse limit is taken in the category of $p$-adic formal schemes.  This inverse limit exists because the transition morphisms are affine.  

\begin{prop} The formal scheme ${ \mathfrak{S}}^{tor,mod}_{K^p,\Sigma}$ is integral perfectoid and its generic fiber is $\mathcal{S}_{K^p, \Sigma}^{tor}$. 
\end{prop}
\begin{proof} This follows from almost verbatim from \cite{MR3512528}, section A.12, taking into account remark \ref{rem-minor-difference}. 
\end{proof}

We now let $\mathcal{U} \hookrightarrow \mathcal{FL}_{G,\mu}$ be a quasi-compact open subset. Our goal is to define formal models for $\pi_{HT}^{-1}(\mathcal{U})$ and $(\pi^{tor}_{HT})^{-1}(\mathcal{U})$. 

We first need a formal model for $\mathcal{U}$. By \cite{MR1032938}, thm. 1.6,  there exists an ideal $\mathcal{I}$ of $\oscr_{\mathfrak{FL}_{G,\mu}}$ and an open subscheme $\mathfrak{U}$ of   $\mathrm{NBL}_{\mathcal{I}}(\mathfrak{FL}_{G,\mu})$ such that the generic fiber of $\mathfrak{U}$ is $\mathcal{U}$. 

We can define $\mathfrak{S}_{K^p, \Sigma,\mathfrak{U}}^{tor,mod} \rightarrow \mathfrak{S}^{\star,mod}_{K^p,\mathfrak{U}}$ which is a formal model for $(\pi^{tor}_{HT})^{-1}(\mathcal{U}) \rightarrow \pi_{HT}^{-1}(\mathcal{U})$ as follows.  The ideal $\mathcal{I}$ pulls back to ideals $\mathcal{I}_1$ and $\mathcal{I}_2$ of $\mathfrak{S}_{K^p, \Sigma}^{tor,mod}$ and $\mathfrak{S}^{\star,mod}_{K^p}$ respectively.   We now wish to consider the normalized  blow up of 
$\mathfrak{S}_{K^p, \Sigma}^{tor,mod}$ and $\mathfrak{S}^{\star,mod}_{K^p}$ at $\mathcal{I}_1$ and $\mathcal{I}_2$ respectively. We actually show that the ideals come from finite level, perform the normalized blow-up at a finite level,  and then pass to the limit.

\begin{lem}\label{lem-approximate-ideal-finitelevel} The ideal $\mathcal{I}_1$ and $\mathcal{I}_2$ are pull backs of ideals $\mathcal{I}_{1,K_p}$ and $\mathcal{I}_{2,K_p}$ of $\mathfrak{S}_{K^pK_{p}, \Sigma}^{tor,mod}$ and $\mathfrak{S}^{\star,mod}_{K^pK_{p,n}}$ for $K_p$ small enough. 
 \end{lem}
 \begin{proof} It suffices to prove the claim for $\mathcal{I}_{2}$. We need to prove that $\mathcal{I}_{2}$ is locally  generated by $K_{p,n}$-invariant sections for $n$ large enough. 
 First observe that the ideal $\mathcal{I}$ contains $p^s$ for some integer $s$. Over each standard affine $\mathfrak{U}_i$, we have $\mathcal{I}(\mathfrak{U}_i) = ( s_{1,i}, \cdots, s_{k,i})$. We can find sections $s'_{1,i}, \cdots, s'_{k,i} \in \HH^0( \pi_{HT}^{-1}(\mathcal{U}_i), \oscr_{\mathcal{S}_{K^p}^{\star}}^+)$ such that $s'_{j,i} = s_{j,i}~\mod p^s$, and $s'_{j,i}$ comes from some finite level $K_{p,n}$ by proposition \ref{prop-bretagne}, (3). Thus,  the $s'_{j,i}$ and $p^s$ generate $\mathcal{I}_{2}$ over the image of $\pi_{HT}^{-1}(\mathcal{U}_i)$ in $\mathcal{S}^{\star}_{K^pK_{p,n}}$. 
 \end{proof}

We can therefore consider $$\mathrm{NBL}_{\mathcal{I}_{1,K_p}}(\mathfrak{S}_{K^pK_p, \Sigma}^{tor,mod})$$ and $$\mathrm{NBL}_{\mathcal{I}_{2,K_p}}(\mathfrak{S}_{K^pK_p}^{\star,mod})$$ for $K_p$ small enough. 

We let $\mathrm{NBL}_{\mathcal{I}_1}(\mathfrak{S}_{K^p, \Sigma}^{tor,mod}) = \lim_{K_p} \mathrm{NBL}_{\mathcal{I}_{1,K_p}}(\mathfrak{S}_{K^pK_p, \Sigma}^{tor,mod})$ and $\mathrm{NBL}_{\mathcal{I}_2}(\mathfrak{S}_{K^p, \Sigma}^{\star,mod}) = \lim_{K_p} \mathrm{NBL}_{\mathcal{I}_{2,K_p}}(\mathfrak{S}_{K^pK_p, \Sigma}^{\star,mod})$  (the limits are taken in the category of $p$-adic formal schemes).

We have maps $$\mathrm{NBL}_{\mathcal{I}_1}(\mathfrak{S}_{K^p, \Sigma}^{tor,mod}) \rightarrow \mathrm{NBL}_{\mathcal{I}_2}(\mathfrak{S}_{K^p}^{\star,mod}) \rightarrow \mathrm{NBL}_{\mathcal{I}}(\mathfrak{FL}_{G,\mu}).$$ We let  $\mathfrak{S}_{K^p, \Sigma,\mathfrak{U}}^{tor,mod}$ and $\mathfrak{S}^{\star,mod}_{K^p,\mathfrak{U}}$ be the preimages of $\mathfrak{U}$. 
These open formal subschemes come from open formal subschemes $\mathfrak{S}_{K^pK_p, \Sigma,\mathfrak{U}}^{tor,mod}$ and $\mathfrak{S}^{\star,mod}_{K^pK_p,\mathfrak{U}}$  of  $\mathrm{NBL}_{\mathcal{I}_{1,K_p}}(\mathfrak{S}_{K^pK_p, \Sigma}^{tor,mod})$ and $\mathrm{NBL}_{\mathcal{I}_{2,K_p}}(\mathfrak{S}_{K^pK_p}^{\star,mod})$ for  large enough $K_p$ (the equations defining $\mathfrak{U}$ are defined at finite level).
 
 \medskip
 
 The following theorem shows that the spaces $\mathfrak{S}^{tor}_{K^pK_p,\Sigma}$, $\mathfrak{S}^{tor,mod}_{K^pK_p,\Sigma} $ and $\mathfrak{S}^{tor,mod}_{K^pK_p,\Sigma,\mathfrak{U}}$ admit the usual description at the boundary in terms of  certain formal charts.   The case of $\mathfrak{S}^{tor}_{K^pK_p,\Sigma}$ is available in the literature. The other cases are deduced from the case of $\mathfrak{S}^{tor}_{K^pK_p,\Sigma}$, by tracing down what happens with the various normalized blow ups at the level of formal charts. This is possible because  the Hodge-Tate period map behaves nicely for  degenerations of abelian varieties, since the Hodge-Tate period morphism for \'etale and multiplicative $p$-divisible group is trivial.

\begin{thm} \label{thm-description-toro-compactificiation}
\begin{enumerate}

\item Let $K = K^pK_p$ with $K_p \subseteq \mathrm{GSp}_{2g}(\Z_p)$. We have a decomposition $\mathfrak{S}^{tor}_{K^pK_p,\Sigma} = \coprod_{\Phi} Z_K(\Phi)$ into locally closed formal subschemes,  indexed by certain cusp label representatives $\Phi$. 

\item The formal completion $\widehat{\mathfrak{S}^{tor}_{K^pK_p,\Sigma}}^{Z_K(\Phi)}$ admits the following canonical  description:
\begin{itemize}
\item  There is a tower of $p$-adic formal schemes:
$$ \mathbf{S}_{K_\Phi}(Q_\Phi, D_\Phi) \rightarrow \mathbf{S}_{K_\Phi}( \overline{Q}_\Phi, \overline{D}_\Phi) \rightarrow \mathbf{S}_{K_\Phi}(G_{\Phi,h}, D_{\Phi,h})$$
where  $ \mathbf{S}_{K_\Phi}(Q_\Phi, D_\Phi) \rightarrow \mathbf{S}_{K_\Phi}( \overline{Q}_\Phi, \overline{D}_\Phi)$ is a torsor under a torus  $\mathbf{E}_K(\Phi)$, and $\mathbf{S}_{K_\Phi}(G_{\Phi,h}, D_{\Phi,h})$ is an integral model of a lower dimensional Siegel variety. 

In the situation that $K$  is the principal level $N$ congruence subgroup, $\mathbf{S}_{K_\Phi}(G_{\Phi,h}, D_{\Phi,h})$ carries a full level $N$ structure, and $ \mathbf{S}_{K_\Phi}( \overline{Q}_\Phi, \overline{D}_\Phi) \rightarrow \mathbf{S}_{K_\Phi}(G_{\Phi,h}, D_{\Phi,h})$ is an abelian scheme torsor. It parametrizes semi-abelian schemes $0 \rightarrow T \rightarrow \mathcal{G} \rightarrow A \rightarrow 0$ (together with certain level structure) where $A$ is the universal abelian scheme over $ \mathbf{S}_{K_\Phi}(G_{\Phi,h}, D_{\Phi,h})$ and $T$ is a split torus of rank $g-\dim A$. 

\item There is a twisted torus embedding  $\mathbf{S}_{K_\Phi}(Q_\Phi, D_\Phi) \rightarrow \mathbf{S}_{K_\Phi}(Q_\Phi, D_\Phi,\Sigma(\Phi))$ depending on $\Sigma$. 

\item There is an arithmetic group $\Delta_{K}(\Phi)$ acting on $X^\star(\mathbf{E}_K(\Phi))$ and on $\mathbf{S}_{K_\Phi}(Q_\Phi, D_\Phi) \hookrightarrow \mathbf{S}_{K_\Phi}(Q_\Phi, D_\Phi,\Sigma(\Phi))$. 
\item  We have a $\Delta_K(\Phi)$-invariant closed   subscheme $\mathbf{Z}_{K_\Phi}(Q_\Phi, D_\Phi,\Sigma(\Phi)) \hookrightarrow \mathbf{S}_{K_\Phi}(Q_\Phi, D_\Phi, \Sigma(\Phi))$.
\item There is a finite morphism  $\mathbf{S}_{K_\Phi}(G_{\Phi,h}, D_{\Phi,h}) \rightarrow  \mathfrak{S}^{\star}_{K}$. 
\item There is a series of morphisms: 
$$ \mathbf{Z}_{K_\Phi}(Q_\Phi, D_\Phi,\Sigma(\Phi)) \rightarrow \mathbf{S}_{K_\Phi}( \overline{Q}_\Phi, \overline{D}_\Phi) \rightarrow \mathbf{S}_{K_\Phi}(G_{\Phi,h}, D_{\Phi,h}).$$
\item  We have a canonical isomorphism $Z_K(\Phi) =\Delta_{K}(\Phi) \backslash \mathbf{Z}_{K_\Phi}(Q_\Phi, D_\Phi,\Sigma(\Phi))$. 

\item We have a canonical isomorphism  $$\widehat{\mathfrak{S}^{tor}_{K, \Sigma}}^{{Z}_{K}(\Phi)}  \simeq \Delta_K(\Phi) \backslash \big( \widehat{ \mathbf{S}_{K_\Phi}(Q_\Phi, D_\Phi,\Sigma(\Phi))}^{\mathbf{Z}_{K_\Phi}(Q_\Phi, D_\Phi,\Sigma(\Phi))}\big).$$
 
 \end{itemize}
 
 \item  For small enough $K_p$:
 
 \begin{itemize}
 \item The ideal $\mathcal{J}_{K_p}$ restricted to $\widehat{\mathfrak{S}^{tor}_{K, \Sigma}}^{{Z}_{K}(\Phi)}$ is the pull back of an ideal $\mathcal{J}_{K_p,\Phi}$ of $\mathbf{S}_{K_\Phi}(G_{\Phi,h}, D_{\Phi,h})$.
 \item Let $\mathbf{Z}^{mod}_{K_\Phi}(Q_\Phi, D_\Phi,\Sigma(\Phi))=  \mathrm{NBL}_{\mathcal{J}_{K_p,\Phi}}\mathbf{Z}_{K_\Phi}(Q_\Phi, D_\Phi,\Sigma(\Phi))$ and $Z^{mod}_K(\Phi)= \Delta_K(\Phi)\backslash \mathbf{Z}^{mod}_{K_\Phi}(Q_\Phi, D_\Phi,\Sigma(\Phi))$. 
  We have a decomposition  $\mathfrak{S}^{tor,mod}_{K^pK_p,\Sigma} =  \coprod_{\Phi} Z^{mod}_K(\Phi)$. 
  
  \item Let $\mathbf{S}^{mod}_{K_\Phi}(G_{\Phi,h}, D_{\Phi,h}) = \mathrm{NBL}_{\mathcal{J}_{K_p,\Phi}}(\mathbf{S}_{K_\Phi}(G_{\Phi,h}, D_{\Phi,h}))$. We have a finite morphism $\mathbf{S}^{mod}_{K_\Phi}(G_{\Phi,h}, D_{\Phi,h}) \rightarrow \mathfrak{S}^{\star,mod}_{K^pK_p}$. 
  
   \item Let $\mathbf{S}^{mod}_{K_\Phi}(Q_\Phi, D_\Phi,\Sigma(\Phi)) = \mathrm{NBL}_{\mathcal{J}_{K_p,\Phi}}\mathbf{S}_{K_\Phi}(Q_\Phi, D_\Phi,\Sigma(\Phi))$.  We have a canonical isomorphism  $$\widehat{\mathfrak{S}^{tor,mod}_{K, \Sigma}}^{{Z}^{mod}_{K}(\Phi)}  \simeq \Delta_K(\Phi) \backslash \big( \widehat{ \mathbf{S}^{mod}_{K_\Phi}(Q_\Phi, D_\Phi,\Sigma(\Phi))}^{\mathbf{Z}^{mod}_{K_\Phi}(Q_\Phi, D_\Phi,\Sigma(\Phi))}\big).$$ 
\end{itemize}

 \item For all small enough $K_p$:
 \begin{itemize}
 \item The ideal $\mathcal{I}_{1,K_p}$ restricted to $\widehat{\mathfrak{S}^{tor,mod}_{K, \Sigma}}^{{Z}^{mod}_{K}(\Phi)}$ is the pull back of an ideal $\mathcal{I}_{1,K_p,\Phi}$ of $\mathbf{S}^{mod}_{K_\Phi}(G_{\Phi,h}, D_{\Phi,h})$.
 \item We have a decomposition  $\mathfrak{S}^{tor,mod}_{K^pK_p,\Sigma,\mathfrak{U}} =  \coprod_{\Phi} Z^{mod}_K(\Phi)_{\mathfrak{U}}$  where $Z^{mod}_K(\Phi)_{\mathfrak{U}}$ is an open subset of $\mathrm{NBL}_{\mathcal{I}_{1,K_p}}(Z^{mod}_K(\Phi))$. 
 
 \item  There is an open subset  $\mathbf{S}^{mod}_{K_\Phi}(G_{\Phi,h}, D_{\Phi,h})_{\mathfrak{U}}$ of $\mathrm{NBL}_{\mathcal{I}_{1,K_p,\Phi}}(\mathbf{S}^{mod}_{K_\Phi}(G_{\Phi,h}, D_{\Phi,h}))$ such that we  have a finite morphism $\mathbf{S}^{mod}_{K_\Phi}(G_{\Phi,h}, D_{\Phi,h})_{\mathfrak{U}} \rightarrow \mathfrak{S}^{\star,mod}_{K^pK_p,\mathfrak{U}}$.
 \item There exists an open subset  $\mathbf{S}^{mod}_{K_\Phi}(Q_\Phi, D_\Phi,\Sigma(\Phi))_{\mathfrak{U}}$ of  $\mathrm{NBL}_{\mathcal{I}_{1, K_p,\Phi}}(\mathbf{S}^{mod}_{K_\Phi}(Q_\Phi, D_\Phi,\Sigma(\Phi)))$ and an open subset $ \mathbf{Z}^{mod}_{K_\Phi}(Q_\Phi, D_\Phi,\Sigma(\Phi))_{\mathfrak{U}}$ of $ \mathrm{NBL}_{\mathcal{I}_{1, K_p,\Phi}}(\mathbf{Z}^{mod}_{K_\Phi}(Q_\Phi, D_\Phi,\Sigma(\Phi)))$ such that 
  we have a canonical isomorphism  $$\widehat{\mathfrak{S}^{tor,mod}_{K, \Sigma,\mathfrak{U}}}^{{Z}^{mod}_{K}(\Phi)_{\mathfrak{U}}} \simeq \Delta_K(\Phi) \backslash \big( \widehat{\mathbf{S}^{mod}_{K_\Phi}(Q_\Phi, D_\Phi,\Sigma(\Phi))_{\mathfrak{U}}}^{\mathbf{Z}^{mod}_{K_\Phi}(Q_\Phi, D_\Phi,\Sigma(\Phi))_{\mathfrak{U}}}\big).$$ 
\end{itemize}

 \end{enumerate}
 \end{thm}
 
 \begin{rem} The main observation needed to prove this theorem is the property that the ideals $\mathcal{J}_{K_p}$ and  $\mathcal{I}_{1,K_p}$ come from ideals $\mathcal{J}_{K_p, \Phi}$ and $\mathcal{I}_{1, K_p,\Phi}$ on each formal chart (first item in $(3)$ and $(4)$). This is reminiscent of the concept of well positioned  subset or subscheme of \cite{MR3859178}, def. 2.2.1.
 \end{rem}

\begin{proof} The first two items follow from \cite{Lan2016}, \cite{MR3948111} or \cite{MR3512528} (for principal level structures). The point  $(3)$,  follows from \cite{MR3512528}, section A.12. It remains to check the point $(4)$. The key property is  that the  ideal $\mathcal{I}_{1,K_p}$ restricted to $\widehat{\mathfrak{S}^{tor,mod}_{K, \Sigma}}^{{Z}^{mod}_{K}(\Phi)}$ is the pull back of an ideal $\mathcal{I}_{1,K_p,\Phi}$ of $\mathbf{S}^{mod}_{K_\Phi}(G_{\Phi,h}, D_{\Phi,h})$. The rest  follows easily. 
We argue as follows. Let $\Phi$ be a cusp label at some finite level $K^pK_p$. For all $n$ large enough such that $K_{p,n} \subseteq K_p$, we let $\{\Phi_n\}_n$ be a compatible sequence of cusp label of level $K^pK_{p,n}$ mapping to $\Phi$. We let $Z^{mod}_{K^p}(\Phi_\infty) = \lim_n Z^{mod}_{K^pK_{p,n}}(\Phi_n)$.  We can consider the completion of $\mathfrak{S}^{tor,mod}_{K^p,\Sigma}$ at $Z^{mod}_{K^p}(\Phi_\infty)$. 
Then we have a map $\widehat{\mathfrak{S}^{tor,mod}_{K^p,\Sigma}}^{Z^{mod}_{K^p}(\Phi_\infty)} \rightarrow \mathbf{S}^{mod}_{K^p_{\Phi_\infty}}(G_{\Phi,h}, D_{\Phi,h})$ where 
$\mathbf{S}^{mod}_{K^p_{\Phi_\infty}}(G_{\Phi,h}, D_{\Phi,h})$ is a perfectoid formal scheme attached to a Siegel Shimura datum of lower dimension. Moreover, we have a factorization of the Hodge-Tate period map as follows (compare with \cite{caraiani2019generic}, coro. 4.2.2):
\begin{eqnarray*}
\xymatrix{\widehat{\mathfrak{S}^{tor,mod}_{K^p,\Sigma}}^{Z^{mod}_{K^p}(\Phi_\infty)} \ar[r]^{\pi^{tor}_{HT}} \ar[d] & \mathfrak{FL}_{G,\mu} \\
\mathbf{S}^{mod}_{K^p_{\Phi_\infty}}(G_{\Phi,h}, D_{\Phi,h}) \ar[r]^{\pi_{HT}} & \mathfrak{FL}_{G_{\Phi,h}, \mu_{G_{\Phi,h}}}\ar[u] }
\end{eqnarray*}

The sheaf of  ideals $\mathcal{I}$ of $\mathfrak{FL}$ restricts to a sheaf of  ideals $\mathcal{I}_{\Phi}$ of $\mathfrak{FL}_{G_{\Phi,h}, \mu_{G_{\Phi,h}}}$. We deduce that at infinite level, the sheaf $\mathcal{I}_{1}$ restricted  to $\widehat{\mathfrak{S}^{tor,mod}_{K^p,\Sigma}}^{Z^{mod}_{K^p}(\Phi_\infty)}$ is the pull back of a sheaf of ideals $\mathcal{I}_{1,\Phi}$ on $\mathbf{S}^{mod}_{K^p_{\Phi_\infty}}(G_{\Phi,h}, D_{\Phi,h})$. We can then prove that $\mathcal{I}_{1,\Phi}$ comes from finite level $K_p$ as in lemma \ref{lem-approximate-ideal-finitelevel}. 
\end{proof}

\begin{thm}\label{thm-formal-vanishingtominimal} Consider the map $\pi : \mathfrak{S}_{K^pK_p, \Sigma,\mathfrak{U}}^{tor,mod} \rightarrow \mathfrak{S}^{\star,mod}_{K^pK_p,\mathfrak{U}}$. Then we have that $\mathrm{R}^i\pi_\star \oscr_{ \mathfrak{S}_{K^pK_p, \Sigma,\mathfrak{U}}^{tor,mod}}(-nD) = 0$ for all $i >0$ and $n \geq 1$.
\end{thm}
\begin{proof} The proof of theorem \ref{vanishingtominimal} transposes  verbatim, given theorem \ref{thm-description-toro-compactificiation}.
\end{proof}

\subsubsection{The Hodge type case}\label{section-Hodge-case}

We fix an embedding $(G,X) \hookrightarrow (\mathrm{GSp}_{2g}, \mathcal{H}^{\pm}_g)$ of a Hodge type  Shimura datum into a Siegel Shimura datum. 
Let $V$ be a $2g$-dimensional vector space over $\qq$, equipped with a symplectic pairing so that $G \hookrightarrow \mathrm{GSp}_{2g}\hookrightarrow \mathrm{GL}(V)$.  We can view $G \hookrightarrow \mathrm{GL}(V)$ as the  group stabilizing a finite number of tensors $\{s_\alpha\} \in V^{\otimes}$ (where $V^\otimes = \oplus_{m,n \in \ZZ_{\geq 0}} V^{\otimes m} \otimes (V^\vee)^{\otimes n}$).  
The group $\mathrm{GSp}_{2g}$ will be denoted by $\tilde{G}$ in this section. All the objects corresponding to this group will carry a $\sim$.

Over $S_{K^p K_p}$ we have  a pro-\'etale $K_p$-torsor represented by the tower of Shimura varieties $\lim_{K'_p\subseteq K_p} S_{K^pK_p'}$. By pushout along the map $K_p \rightarrow G(\qq_p)$ we get a $G(\qq_p)$-torsor.

 One can give a more ``modular'' description of this torsor using the closed embedding  ${S}_K \hookrightarrow \tilde{{S}}_{\tilde{K}}$,  where $\tilde{{S}}_{\tilde{K}}$ is a Siegel Shimura variety (over $E$) and the embedding is the one induced  by  $(G,X) \hookrightarrow (\mathrm{GSp}_{2g}, \mathcal{H}^{\pm}_g)$ for a suitable compact open subgroup $\tilde{K}$. 
The tensors $\{s_\alpha\}$ can be used to produce sections $\{s_{\alpha,p} \in \HH^0({S}_K, \HH_1(A, \qq_p)^\otimes)\}$. More precisely, one first produces tensors  $\{s_{\alpha,B} \in \HH^0(S_K(\C), \HH_1(A, \qq)^\otimes)\}$ using the complex uniformization of $S_K(\C)$. They  give  tensors $\{s_{\alpha,p} \in \HH^0(S_K(\C), \HH_1(A, \qq_p)^\otimes)\}$.   Lemma 2.3.2 of \cite{MR3702677} proves that these tensors are defined over $E$.  Therefore, we can consider the torsor of isomorphisms $V \otimes_\qq \qq_p \rightarrow \HH_1(A, \qq_p)$, preserving all the tensors $s_{\alpha,p}$. 

There is also the  $M_\mu$-torsor $M_{dR}$ over $S_K$.  We  recall its description. The tensors $\{s_\alpha\}$ can be used to produce sections $\{s_{\alpha,dR} \in \HH^0({S}_K, \HH_{1,dR}(A)^\otimes)\}$. One first define the $P_{\mu}^{std}$-torsor $P_{dR}$, to be the the torsor of isomorphisms $V \otimes_\qq \oscr_{S_K} \rightarrow \HH_{1,dR}(A)$ matching the filtration on $V$ corresponding to $P_{\mu}^{std}$ with the Hodge filtration, and preserving the tensors $s_{\alpha,dR}$. 
By pushout along  $P_{\mu}^{std} \rightarrow M_\mu$, we have $M_{dR} = P_{dR} \times^{P_{\mu}^{std}} M_\mu$. 

\begin{rem} The closed immersion ${S}_K \hookrightarrow \tilde{{S}}_{\tilde{K}}$ extends to a closed immersion $M_{dR} \hookrightarrow \tilde{M}_{dR}$ where $\tilde{M}_{dR}$ is the $M_{\tilde{\mu}}$- ``de Rham'' torsor over $\tilde{{S}}_{\tilde{K}}$. 
\end{rem}

By pull back to the analytic space $\mathcal{S}^{an}_{K}$  we get a pro-\'etale $G(\qq_p)$-torsor $\mathcal{G}^{an}_{pet, p}$, as well as a $\mathcal{M}_\mu^{an}$-torsor $\mathcal{M}_{dR}^{an}$.

In section 2.3 of \cite{MR3702677} the authors define two other  pro-\'etale torsors  over $\mathcal{S}^{an}_K$: $\mathcal{P}^{an}_{HT}$ and $\mathcal{M}^{an}_{HT}$, under the groups  $\mathcal{P}^{an}_\mu$ and  $\mathcal{M}_\mu^{an}$. These definitions extend those given in the Siegel case  (see section \ref{section-torsors-Siegel}) and use the tensors $s_{\alpha,p}$. 
 Moreover, $\mathcal{P}^{an}_{HT} \times^{\mathcal{P}^{an}_\mu} \mathcal{G}^{an} = \mathcal{G}^{an}_{et,p}\times^{G(\qq_p)} \mathcal{G}^{an}$ and  $\mathcal{M}^{an}_{HT} = \mathcal{P}^{an}_{HT} \times^{\mathcal{P}^{an}_\mu} \mathcal{M}^{an}_\mu$. By \cite{MR3702677}, prop. 2.3.9, there is a canonical identification of torsors    $\mathcal{M}^{an}_{dR}$ and $\mathcal{M}^{an}_{HT}$ 

\subsubsection{Perfectoid Hodge type Shimura varieties}
By \cite{scholze-torsion}, there is a perfectoid space $\mathcal{S}^{an}_{K^p} \sim \lim_{K_p} \mathcal{S}^{an}_{K^p K_p}$. Since the torsor $\mathcal{G}^{an}_{pet, p}$ becomes trivial over $\mathcal{S}^{an}_{K^p}$, we obtain a Hodge-Tate period map  (\cite{MR3702677}, thm 2.1.3):   $\pi_{HT} : \mathcal{S}^{an}_{{K^p}} \rightarrow \mathcal{FL}_{G,\mu}$ which is $G(\qq_p)$-equivariant.

A key property is that the pull back of the $\mathcal{M}_\mu^{an}$- torsor $\mathcal{G}^{an}/U_{\mathcal{P}_\mu^{an}} \rightarrow  \mathcal{FL}_{G,\mu}$ via $\pi_{HT}$ is $\mathcal{M}_{HT}^{an}$ and this is canonically identified with the pull back via $\mathcal{S}^{an}_{{K^p}} \rightarrow \mathcal{S}^{an}_{{K^pK_p}}$ of $\mathcal{M}_{dR}^{an}$.

Moreover, the relation with the Siegel datum is expressed by the following diagram, where the horizontal maps are closed immersions: 

\begin{eqnarray*}
\xymatrix{ \mathcal{S}^{an}_{K^p} \ar[r] \ar[d] & \tilde{\mathcal{S}}^{an}_{\tilde{K}^p} \ar[d] \\
\mathcal{FL}_{G,\mu} \ar[r] & {\mathcal{FL}}_{\tilde{G},\tilde{\mu}}}
\end{eqnarray*}

\subsubsection{Compactifications in the Hodge case}\label{Section-Hodge-perfectoid}
For any compact open $K \subset G(\mathbb{A}_f)$, we have the minimal compactification $\mathcal{S}^{\star}_K$ and  there is a finite surjective map  
$ \mathcal{S}^\star_{K} \rightarrow \mathcal{S}^{\star}_{\overline{K}}$ where $\mathcal{S}^{\star}_{\overline{K}}$ is defined before \cite{scholze-torsion}, thm. IV. 1.1.    This is  the schematic image of the morphism $ \mathcal{S}^\star_{K} \rightarrow \tilde{\mathcal{S}}^{\star}_{\tilde{ K}}$ where $\tilde{\mathcal{S}}^{\star}_{\tilde{K}}$ is the minimal compactification of the Shimura variety for $\tilde{G}$, and $\tilde{K}$ is a small enough compact open subgroup of $\tilde{G}(\mathbb{A}_f)$ such that $\tilde{K} \cap G(\mathbb{A}_f) = K$. 
  By \cite{scholze-torsion},  thm. IV. 1.1 there is a  perfectoid space 
 $\mathcal{S}^\star_{\overline{K^p}} \sim \lim_{K_p} \mathcal{S}^\star_{\overline{K^pK_p}}$  
and there is a   map $\pi_{HT} : \mathcal{S}^\star_{\overline{K^p}} \rightarrow \mathcal{FL}_{{G},{\mu}}$.  
Strictly speaking the map constructed  in  \cite{scholze-torsion},  took values in $\mathcal{FL}_{ \tilde{G}, \tilde{\mu}}$. Since  $\mathcal{FL}_{G, \mu} \hookrightarrow \mathcal{FL}_{ \tilde{G}, \tilde{\mu}}$ is Zariski closed and the map  is known to factors through $ \mathcal{FL}_{G, \mu} $ over the Zariski dense open subset $\mathcal{S}^{an}_{{K^p}}$ by \cite{MR3702677}, thm 2.1.3, we deduce that we have a map $\pi_{HT} : \mathcal{S}^\star_{\overline{K^p}} \rightarrow \mathcal{FL}_{G,\mu}$. 

 Since the tower $\{\mathcal{S}^\star_{\overline{K^pK_p}}\}_{K_p}$ carries a $G(\qq_p)$-action, the space $\mathcal{S}^\star_{\overline{K^p}} $ inherits a $G(\qq_p)$-action, and the map $\pi_{HT}$ is $G(\qq_p)$-equivariant. Moreover, the map $\pi_{HT}$ is affinoid in the sense  that there exists an affinoid covering $\mathcal{FL}_{G,\mu} = \cup_i V_i$ such each $\pi_{HT}^{-1}(V_i)$ is a good affinoid perfectoid open subset of  $\mathcal{S}^\star_{\overline{K^p}}$ (see definition \ref{defi-of-good}).

We also know that there is  a perfectoid space  $\mathcal{S}^\star_{K^p} = \lim_{K_p} \mathcal{S}^{\star,\Diamond}_{K^pK_p}$ (this last inverse limit is taken in the category of diamonds) by \cite{bhatt2019prisms}, thm. 1.16. We therefore have a  $G(\qq_p)$-equivariant map $\mathcal{S}^\star_{K^p} \rightarrow \mathcal{S}^\star_{\overline{K^p}}$. It is not known whether $\mathcal{S}^\star_{K^p} \sim \lim \mathcal{S}^{\star}_{K^pK_p}$ in general.

By \cite{KWLan}, for a cofinal subset of cone decompositions $\Sigma$, we have a perfectoid space   $\mathcal{S}^{tor}_{K^p,\Sigma} \sim \lim \mathcal{S}^{tor}_{K^pK_p,\Sigma}$.    More precisely, for each such cone decomposition $\Sigma$, there exists a cone decomposition $\tilde{\Sigma}$ and a closed immersion of perfectoid spaces 
$\mathcal{S}^{tor}_{K^p,\Sigma} \hookrightarrow \mathcal{S}^{tor}_{\tilde{K}^p, \tilde{\Sigma}}$.   Let us call these cone decompositions \emph{perfect} cone decompositions because they give rise to perfectoid toroidal compactifications. 

\begin{rem} We did not prove that    there is a perfectoid space  $\mathcal{S}^{tor}_{K^p,\Sigma} \sim \lim \mathcal{S}^{tor}_{K^pK_p,\Sigma}$ for any cone decomposition $\Sigma$. It seems likely that one could reproduce the argument of \cite{MR3512528} in the Hodge case, using the explicit description of the boundary of the integral toroidal compactifications given in \cite{MR3948111}.  
\end{rem} 

For a perfect cone decomposition $\Sigma$, we have a series of maps $\mathcal{S}^{tor}_{K^p,\Sigma} \rightarrow \mathcal{S}^\star_{K^p}  \rightarrow \mathcal{S}^\star_{\overline{K^p}}$ and we therefore get a map $\pi_{HT}^{tor} : \mathcal{S}^{tor}_{K^p,\Sigma} \rightarrow \mathcal{FL}_{G,\mu}$. 

The relation to the Siegel Shimura varieties is given by the following diagram where all horizontal maps are closed immersions: 

\begin{eqnarray*}
\xymatrix{ \mathcal{S}^{tor}_{{K^p},\Sigma} \ar[r] \ar[d] & \tilde{\mathcal{S}}^{tor}_{\tilde{K}^p, \tilde{\Sigma}} \ar[dd]\\
\mathcal{S}^{\star}_{{K^p}}  \ar[d] &  \\
\mathcal{S}^{\star}_{\overline{K^p}} \ar[r] \ar[d] & \tilde{\mathcal{S}}^{\star}_{\tilde{K}^p} \ar[d] \\
\mathcal{FL}_{G,\mu} \ar[r] & {\mathcal{FL}}_{\tilde{G},\tilde{\mu}}}
\end{eqnarray*}

Using the map $\pi_{HT}^{tor}$ we can define a canonical extension of the torsors $\mathcal{P}_{HT}^{an}$ and $\mathcal{M}^{an}_{HT}$ to $\mathcal{S}^{tor}_{K^p,\Sigma}$, by simply by pulling back the universal $\mathcal{P}_\mu^{an}$-torsor over $\mathcal{FL}_{G,\mu}$ and pushing out along the map $\mathcal{P}_\mu^{an} \rightarrow \mathcal{M}_\mu^{an}$. There is also the torsor $\mathcal{M}^{an}_{dR}$  over $\mathcal{S}_{K,\Sigma}^{tor}$ which is pulled back from the map  $\mathcal{S}_{K^p,\Sigma}^{tor} \rightarrow \mathcal{S}_{K^pK_p,\Sigma}^{tor}$. The following proposition is corollary 5.2 of \cite{epiga:5865}. 

\begin{prop}[\cite{epiga:5865}] The torsors $\mathcal{M}^{an}_{HT}$  and $\mathcal{M}^{an}_{dR} \times_{\mathcal{S}_{K^pK_p,\Sigma}^{tor}} \mathcal{S}^{tor}_{K^p,\Sigma} $ are canonically isomorphic.
\end{prop}
\begin{proof} Let us denote by $\tilde{M}$ the restriction to $\mathcal{S}^{tor}_{K^p,\Sigma}$ of the torsor $\tilde{\mathcal{M}}^{an}_{dR} = \tilde{\mathcal{M}}^{an}_{HT}$ defined over the Siegel perfectoid Shimura variety $\tilde{\mathcal{S}}_{\tilde{K}^p, \tilde{\Sigma}}$.   By construction, $\tilde{M}$ is  a $\mathcal{M}^{an}_{\tilde{\mu}}$-torsor. The two torsors $\mathcal{M}^{an}_{HT}$ and $\mathcal{M}^{an}_{dR}$ are $\mathcal{M}_{\mu}^{an}$-reductions of this torsor.  They coincide over $\mathcal{S}^{an}_{K^p}$ by \cite{MR3702677}, prop. 2.3.9. But $\mathcal{M}^{an}_{HT}$ (resp. $\mathcal{M}^{an}_{dR}$) is equal to the Zariski closure of $\mathcal{M}^{an}_{HT}\vert_{\mathcal{S}^{an}_{K^p}}$ (resp. $\mathcal{M}^{an}_{dR}\vert_{\mathcal{S}^{an}_{K^p}}$) in $\tilde{M}$. Therefore, these torsors have to coincide everywhere.
\end{proof}

\begin{rem}  The isomorphism of torsors is Hecke equivariant.
\end{rem} 

\subsubsection{Integral models in the Hodge case}\label{subsec-integral-hodge} We need to consider formal models. This section is entirely parallel to section \ref{sect-formal-perf-Siegel}.  Let $K = K^p K_p$ with $K = \tilde{K} \cap G(\mathbb{A}_f)$.   Suppose that  $ \tilde{K}_p \subseteq  \tilde{G}(\ZZ_p)$. We can define $\mathfrak{S}_{K^pK_p}$, $\mathfrak{S}^\star_{K^pK_p}$, and $\mathfrak{S}^{tor}_{K^pK_p,\Sigma}$ as the normalizations of $\tilde{\mathfrak{S}}_{\tilde{K}}$, $\tilde{\mathfrak{S}}^\star_{\tilde{K}}$, and $\tilde{\mathfrak{S}}^{tor}_{\tilde{K},\tilde{\Sigma}}$ in $\mathcal{S}_{K^pK_p}$, $\mathcal{S}^\star_{K^pK_p}$, and $\mathcal{S}^{tor}_{K^pK_p,\Sigma}$ respectively. 

Similarly, for $K_p$ small enough, we define  $\mathfrak{S}^{\star,mod}_{K^pK_p}$, $\mathfrak{S}^{tor,mod}_{K^pK_p,\Sigma}$ as the normalizations of $\tilde{\mathfrak{S}}^{\star,mod}_{\tilde{K}}$ and $\tilde{\mathfrak{S}}^{tor,mod}_{\tilde{K},\Sigma}$ in  $\mathcal{S}^\star_{K^pK_p}$ and $\mathcal{S}^{tor}_{K^pK_p,\Sigma}$ respectively. 
Alternatively, these can also be constructed as normalized blow ups of  $\mathfrak{S}^\star_{K^pK_p}$ and $\mathfrak{S}^{tor}_{K^pK_p,\Sigma}$ for the ideals $\mathcal{I}_{K_p}$ and $\mathcal{J}_{K_p}$ which are the pull back of the ideals $\mathcal{I}_{\tilde{K}_p}$ and $\mathcal{J}_{\tilde{K}_p}$ from  $\tilde{\mathfrak{S}}^\star_{\tilde{K}}$ and $\tilde{\mathfrak{S}}^{tor}_{\tilde{K},\Sigma}$.

We  let ${ \mathfrak{S}}^{\star,mod}_{K^p} = \lim_{K_p} { \mathfrak{S}}^{\star,mod}_{K^pK_p}$, where the inverse limit is taken in the category of $p$-adic formal schemes.  This inverse limit exists because the transition morphisms are affine. In the limit we have a  map ${ \mathfrak{S}}^{\star,mod}_{K^p} \rightarrow \tilde{ \mathfrak{S}}^{\star,mod}_{\tilde{K}^p}$ and therefore a map $\pi_{HT} : \mathfrak{S}^{\star,mod}_{K^p} \rightarrow \mathfrak{FL}_{\tilde{G}, \tilde{\mu}}$. 

\begin{prop}\label{prop-bretagne2}
 The Hodge-Tate map factors through a map $\pi_{HT} : \mathfrak{S}^{\star,mod}_{K^p} \rightarrow \mathfrak{FL}_{G, \mu}$.
\end{prop}
\begin{proof} One first checks that the space $\mathfrak{S}^{mod}_{K^p}$ (the complement of the boundary in $\mathfrak{S}^{\star,mod}_{K^p}$) is integral perfectoid and its generic fiber is the quasi-compact open perfectoid Shimura variety $\mathcal{S}_{K^p}$. This follows from the Siegel case, using that $\mathcal{S}_{K^p} \hookrightarrow \tilde{S}_{\tilde{K}^p}$ is a Zariski closed immersion. Therefore the factorization of the period morphism through $\mathfrak{FL}_{G, \mu}$ holds over the Zariski dense subspace  $\mathfrak{S}^{mod}_{K^p}$, and thus everywhere. 
\end{proof}

\begin{rem} We do not know if $\mathfrak{S}^{\star,mod}_{K^p}$ is integral perfectoid. 
\end{rem}

We  let ${ \mathfrak{S}}^{tor,mod}_{K^p,\Sigma} = \lim_{K_p} { \mathfrak{S}}^{tor,mod}_{K^pK_p,\Sigma}$ where the inverse limit is taken in the category of $p$-adic formal schemes.  This inverse limit exists because the transition morphisms are affine.  

\begin{prop} Assume that $\Sigma$ is perfect. The formal scheme ${ \mathfrak{S}}^{tor,mod}_{K^p,\Sigma}$ is integral perfectoid and its generic fiber is $\mathcal{S}_{K^p, \Sigma}^{tor}$. 
\end{prop}
\begin{proof} This follows from the Siegel case. 
\end{proof}

\begin{rem} We do not know if $\mathfrak{S}^{tor,mod}_{K^p,\Sigma}$ is integral perfectoid for all $\Sigma$. 
\end{rem}

We now let $\mathcal{U} \hookrightarrow \mathcal{FL}_{G,\mu}$ be a quasi-compact open subset. It is induced by a quasi-compact open subset $\tilde{\mathcal{U}}$ of $\mathcal{FL}_{\tilde{G}, \tilde{\mu}}$. Our goal is to define formal models for $\pi_{HT}^{-1}(\mathcal{U})$ and $(\pi^{tor}_{HT})^{-1}(\mathcal{U})$.

 There exists an ideal $\tilde{\mathcal{I}}$ of $\oscr_{\mathfrak{FL}_{\tilde{G},\tilde{\mu}}}$ and an open subscheme $\tilde{\mathfrak{U}}$ of   $\mathrm{NBL}_{\tilde{\mathcal{I}}}(\mathfrak{FL}_{\tilde{G},\mu_{\tilde{G}}})$ such that the generic fiber of $\tilde{\mathfrak{U}}$ is $\tilde{\mathcal{U}}$.

For $K_p$ small enough, we define  $\mathfrak{S}^{\star,mod}_{K^pK_p, \mathfrak{U}}$, $\mathfrak{S}^{tor,mod}_{K^pK_p,\Sigma, \mathfrak{U}}$ as the normalizations of $\tilde{\mathfrak{S}}^{\star,mod}_{\tilde{K}, \tilde{\mathfrak{U}}}$ and $\tilde{\mathfrak{S}}^{tor,mod}_{\tilde{K}, \tilde{\Sigma}, \tilde{\mathfrak{U}}}$ in  $\pi_{HT}^{-1}(\mathcal{U})$ and $(\pi^{tor}_{HT})^{-1}(\mathcal{U})$ respectively. 
Alternatively, these can also be constructed as  suitable opens in normalized blow ups of  $\mathfrak{S}^{\star,mod}_{K^pK_p}$ and $\mathfrak{S}^{tor,mod}_{K^pK_p,\Sigma}$ at the ideals $\mathcal{I}_{1, {K}_p}$ and $\mathcal{I}_{2, {K}_p}$ which are  the pull back of the ideals $\mathcal{I}_{1, \tilde{K}_p}$ and $\mathcal{I}_{2, \tilde{K}_p}$ from  $\tilde{\mathfrak{S}}^{\star,mod}_{\tilde{K}}$ and $\tilde{\mathfrak{S}}^{tor,mod}_{\tilde{K},\tilde{\Sigma}}$.

\begin{thm} \label{thm-description-toro-compactificiation2}
\begin{enumerate}

\item Let $K = K^pK_p$ with $K_p \subseteq \tilde{G}(\Z_p)$. We have a decomposition $\mathfrak{S}^{tor}_{K^pK_p,\Sigma} = \coprod_{\Phi} Z_K(\Phi)$ into locally closed formal subschemes,  indexed by certain cusp label representatives $\Phi$. 

\item The formal completion $\widehat{\mathfrak{S}^{tor}_{K^pK_p,\Sigma}}^{Z_K(\Phi)}$ admits the following canonical  description:
\begin{itemize}
\item  There is a tower of $p$-adic formal schemes:
$$ \mathbf{S}_{K_\Phi}(Q_\Phi, D_\Phi) \rightarrow \mathbf{S}_{K_\Phi}( \overline{Q}_\Phi, \overline{D}_\Phi) \rightarrow \mathbf{S}_{K_\Phi}(G_{\Phi,h}, D_{\Phi,h})$$
where  $ \mathbf{S}_{K_\Phi}(Q_\Phi, D_\Phi) \rightarrow \mathbf{S}_{K_\Phi}( \overline{Q}_\Phi, \overline{D}_\Phi)$ is a torsor under a torus  $\mathbf{E}_K(\Phi)$, and $\mathbf{S}_{K_\Phi}(G_{\Phi,h}, D_{\Phi,h})$ is an integral model of a lower dimensional Shimura variety.

\item There is a twisted torus embedding  $\mathbf{S}_{K_\Phi}(Q_\Phi, D_\Phi) \rightarrow \mathbf{S}_{K_\Phi}(Q_\Phi, D_\Phi,\Sigma(\Phi))$ depending on $\Sigma$. 

\item There is an arithmetic group $\Delta_{K}(\Phi)$ acting on $X^\star(\mathbf{E}_K(\Phi))$ and on $\mathbf{S}_{K_\Phi}(Q_\Phi, D_\Phi) \hookrightarrow \mathbf{S}_{K_\Phi}(Q_\Phi, D_\Phi,\Sigma(\Phi))$. 
\item  We have a $\Delta_K(\Phi)$-invariant closed   subscheme $\mathbf{Z}_{K_\Phi}(Q_\Phi, D_\Phi,\Sigma(\Phi)) \hookrightarrow \mathbf{S}_{K_\Phi}(Q_\Phi, D_\Phi, \Sigma(\Phi))$.
\item There is a finite morphism  $\mathbf{S}_{K_\Phi}(G_{\Phi,h}, D_{\Phi,h}) \rightarrow  \mathfrak{S}^{\star}_{K}$. 
\item There is a series of morphisms: 
$$ \mathbf{Z}_{K_\Phi}(Q_\Phi, D_\Phi,\Sigma(\Phi)) \rightarrow \mathbf{S}_{K_\Phi}( \overline{Q}_\Phi, \overline{D}_\Phi) \rightarrow \mathbf{S}_{K_\Phi}(G_{\Phi,h}, D_{\Phi,h}).$$
\item  We have a canonical isomorphism $Z_K(\Phi) =\Delta_{K}(\Phi) \backslash \mathbf{Z}_{K_\Phi}(Q_\Phi, D_\Phi,\Sigma(\Phi))$. 

\item We have a canonical isomorphism  $$\widehat{\mathfrak{S}^{tor}_{K, \Sigma}}^{{Z}_{K}(\Phi)}  \simeq \Delta_K(\Phi) \backslash \big( \widehat{ \mathbf{S}_{K_\Phi}(Q_\Phi, D_\Phi,\Sigma(\Phi))}^{\mathbf{Z}_{K_\Phi}(Q_\Phi, D_\Phi,\Sigma(\Phi))}\big).$$
 
 \end{itemize}
 
 \item  For small enough $K_p$:
 
 \begin{itemize}
 \item The ideal $\mathcal{J}_{K_p}$ restricted to $\widehat{\mathfrak{S}^{tor}_{K, \Sigma}}^{{Z}_{K}(\Phi)}$ is the pull back of an ideal $\mathcal{J}_{K_p,\Phi}$ of $\mathbf{S}_{K_\Phi}(G_{\Phi,h}, D_{\Phi,h})$.
 \item Let $\mathbf{Z}^{mod}_{K_\Phi}(Q_\Phi, D_\Phi,\Sigma(\Phi))=  \mathrm{NBL}_{\mathcal{J}_{K_p,\Phi}}\mathbf{Z}_{K_\Phi}(Q_\Phi, D_\Phi,\Sigma(\Phi))$ and $Z^{mod}_K(\Phi)= \Delta_K(\Phi)\backslash \mathbf{Z}^{mod}_{K_\Phi}(Q_\Phi, D_\Phi,\Sigma(\Phi))$. 
  We have a decomposition  $\mathfrak{S}^{tor,mod}_{K^pK_p,\Sigma} =  \coprod_{\Phi} Z^{mod}_K(\Phi)$. 
  
  \item Let $\mathbf{S}^{mod}_{K_\Phi}(G_{\Phi,h}, D_{\Phi,h}) = \mathrm{NBL}_{\mathcal{J}_{K_p,\Phi}}(\mathbf{S}_{K_\Phi}(G_{\Phi,h}, D_{\Phi,h}))$. We have a finite morphism $\mathbf{S}^{mod}_{K_\Phi}(G_{\Phi,h}, D_{\Phi,h}) \rightarrow \mathfrak{S}^{\star,mod}_{K^pK_p}$. 
  
   \item Let $\mathbf{S}^{mod}_{K_\Phi}(Q_\Phi, D_\Phi,\Sigma(\Phi)) = \mathrm{NBL}_{\mathcal{J}_{K_p,\Phi}}\mathbf{S}_{K_\Phi}(Q_\Phi, D_\Phi,\Sigma(\Phi))$.  We have a canonical isomorphism  $$\widehat{\mathfrak{S}^{tor,mod}_{K, \Sigma}}^{{Z}^{mod}_{K}(\Phi)}  \simeq \Delta_K(\Phi) \backslash \big( \widehat{ \mathbf{S}^{mod}_{K_\Phi}(Q_\Phi, D_\Phi,\Sigma(\Phi))}^{\mathbf{Z}^{mod}_{K_\Phi}(Q_\Phi, D_\Phi,\Sigma(\Phi))}\big).$$ 
\end{itemize}

 \item For all small enough $K_p$,
 \begin{itemize}
 \item The ideal $\mathcal{I}_{1,K_p}$ restricted to $\widehat{\mathfrak{S}^{tor,mod}_{K, \Sigma}}^{{Z}^{mod}_{K}(\Phi)}$ is the pull back of an ideal $\mathcal{I}_{1,K_p,\Phi}$ of $\mathbf{S}^{mod}_{K_\Phi}(G_{\Phi,h}, D_{\Phi,h})$.
 \item We have a decomposition  $\mathfrak{S}^{tor,mod}_{K^pK_p,\Sigma,\mathfrak{U}} =  \coprod_{\Phi} Z^{mod}_K(\Phi)_{\mathfrak{U}}$  where $Z^{mod}_K(\Phi)_{\mathfrak{U}}$ is an open subset of $\mathrm{NBL}_{\mathcal{I}_{1,K_p}}(Z^{mod}_K(\Phi))$. 
 
 \item  There is an open subset  $\mathbf{S}^{mod}_{K_\Phi}(G_{\Phi,h}, D_{\Phi,h})_{\mathfrak{U}}$ of $\mathrm{NBL}_{\mathcal{I}_{1,K_p,\Phi}}(\mathbf{S}^{mod}_{K_\Phi}(G_{\Phi,h}, D_{\Phi,h}))$ such that we  have a finite morphism $\mathbf{S}^{mod}_{K_\Phi}(G_{\Phi,h}, D_{\Phi,h})_{\mathfrak{U}} \rightarrow \mathfrak{S}^{\star,mod}_{K^pK_p,\mathfrak{U}}$.
 \item There exists an open subset  $\mathbf{S}^{mod}_{K_\Phi}(Q_\Phi, D_\Phi,\Sigma(\Phi))_{\mathfrak{U}}$ of  $\mathrm{NBL}_{\mathcal{I}_{1, K_p,\Phi}}(\mathbf{S}^{mod}_{K_\Phi}(Q_\Phi, D_\Phi,\Sigma(\Phi)))$ and an open subset $ \mathbf{Z}^{mod}_{K_\Phi}(Q_\Phi, D_\Phi,\Sigma(\Phi))_{\mathfrak{U}}$ of $ \mathrm{NBL}_{\mathcal{I}_{1, K_p,\Phi}}(\mathbf{Z}^{mod}_{K_\Phi}(Q_\Phi, D_\Phi,\Sigma(\Phi)))$ such that 
  we have a canonical isomorphism  $$\widehat{\mathfrak{S}^{tor,mod}_{K, \Sigma,\mathfrak{U}}}^{{Z}^{mod}_{K}(\Phi)_{\mathfrak{U}}} \simeq \Delta_K(\Phi) \backslash \big( \widehat{\mathbf{S}^{mod}_{K_\Phi}(Q_\Phi, D_\Phi,\Sigma(\Phi))_{\mathfrak{U}}}^{\mathbf{Z}^{mod}_{K_\Phi}(Q_\Phi, D_\Phi,\Sigma(\Phi))_{\mathfrak{U}}}\big).$$ 
\end{itemize}

 \end{enumerate}
 \end{thm}

\begin{proof} The first two items follow from  \cite{MR3948111}. The point  $(3)$ and the point $(4)$ follow from the analogous statement in  theorem \ref{thm-description-toro-compactificiation}.\end{proof}
\begin{thm}\label{thm-formal-vanishingtominimal2} Consider the map $\pi : \mathfrak{S}_{K^pK_p, \Sigma,\mathfrak{U}}^{tor,mod} \rightarrow \mathfrak{S}^{\star,mod}_{K^pK_p,\mathfrak{U}}$. Then we have that $\mathrm{R}^i\pi_\star \oscr_{ \mathfrak{S}_{K^pK_p, \Sigma,\mathfrak{U}}^{tor,mod}}(-nD) = 0$ for all $i >0$ and $n \geq 1$.
\end{thm}
\begin{proof} The proof of theorem \ref{vanishingtominimal} transposes  verbatim, given theorem \ref{thm-description-toro-compactificiation2}.
\end{proof}

\subsubsection{General Shimura varieties}\label{section-Abelian-type1} We now extend part of the preceding  to general Shimura varieties.  For the moment, let $(G,X)$ be an arbitrary Shimura datum and let $K = K^p K_p \subseteq G(\mathbb{A}_f)$ be a compact open subgroup. We consider the compactified Shimura variety $\mathcal{S}_{K,\Sigma}^{tor}$.
 For any algebraic representation $W$ of $G^c$, we have a  local system on the pro-kummer-\'etale site $\mathcal{W}_p$ as 
well as a filtered vector bundle with integrable log-connection $\mathcal{W}_{dR}$. 
We need some period sheaves relative to $\mathcal{S}_{K,\Sigma}^{tor}$ (see \cite{MR3090230}, sect. 6 and \cite{diao2019logarithmic}, sect. 2): $\mathbb{B}_{dR}^+$, $\mathbb{B}_{dR}$, $\oscr \mathbb{B}^+_{dR,log}$ and $\oscr \mathbb{B}_{dR,log}$.
The following property is a consequence of  \cite{diao2019logarithmic}, thm. 5.3.1. The local system $\mathcal{W}_p$ and the filtered vector bundle with integrable log-connection $\mathcal{W}_{dR}$ are associated (\cite{MR3090230}, def 7.5) in the sense  that there is a canonical isomorphism compatible with filtrations, connection and Hecke action:

$$ \mathcal{W}_p \otimes_{\qq_p} \oscr \mathbb{B}_{dR,log} =  \mathcal{W}_{dR} \otimes_{\oscr_{\mathcal{S}_{K,\Sigma}^{tor}}} \oscr\mathbb{B}_{dR,log}$$

\begin{rem} This identity is a consequence of the comparison theorems  for (semi)-abelian varieties in the Siegel case. In the Hodge case and outside of the boundary, it is proved in \cite{MR3702677} section 2.2 and 2.3, as a consequence of the comparison theorems for abelian varieties, together with results  on Hodge tensors under comparison isomorphisms due to \cite{MR1265557}. 
\end{rem}

Using this canonical isomorphism, we define the Hodge-Tate filtration on $\mathcal{W}_p \otimes \hat{\oscr}_{\mathcal{S}_{K,\Sigma}^{tor}}$ as follows. Let $\mathbb{M} = \mathcal{W}_p \otimes_{\qq_p} \mathbb{B}^{+}_{dR,log}$ and $\mathbb{M}_0 = (\mathcal{W}_{dR} \otimes_{\oscr_{\mathcal{S}_{K,\Sigma}^{tor}}} \oscr \mathbb{B}^+_{dR,log})^{\nabla =0}$.

These are two $\mathbb{B}^+_{dR}$-local systems and are  lattices inside $\mathbb{M} \otimes_{\mathbb{B}^+_{dR}} \mathbb{B}_{dR} = \mathbb{M}_0 \otimes_{\mathbb{B}^+_{dR}} \mathbb{B}_{dR} $. There are therefore two  filtrations on this $\mathbb{B}_{dR}$-local system $\mathrm{Fil}^i \mathbb{M}$ and $\mathrm{Fil}^i \mathbb{M}_0$. 
One defines an ascending filtration by: $$\mathrm{Fil}_{-j} \mathcal{W}_p \otimes_{\qq_p} \hat{\oscr}_{\mathcal{S}_{K,\Sigma}^{tor}} = (\mathbb{M} \cap \mathrm{Fil}^j \mathbb{M}^0)/ (\mathrm{Fil}^1 \mathbb{M} \cap  \mathrm{Fil}^j \mathbb{M}^0).$$ 
and we have the relation  $$\mathrm{Gr}_{j} \mathcal{W}_p \otimes_{\qq_p} \hat{\oscr}_{\mathcal{S}_{K,\Sigma}^{tor}} = \mathrm{Gr}^j \mathcal{W}_{dR} \otimes_{\oscr_{\mathcal{S}^{tor}_{K,\Sigma}}} \hat{\oscr}_{\mathcal{S}_{K,\Sigma}^{tor}}.$$

This construction is functorial with respect to the Hecke action and compatible with the Tannakian formalism. One can therefore derive consequences at the level of torsors.
Over $\mathcal{S}_{K,\Sigma}^{tor}$, we have a pro-kummer-\'etale  $G^c(\qq_p)$-torsor $\mathcal{G}_{pet,p}^{an}$ as well as a principal $\mathcal{G}^{c,an}$-torsor $\mathcal{G}^{an}_{dR}$ (defined over the \'etale site). The torsor $\mathcal{G}^{an}_{dR}$ has a reduction of structure group to a $\mathcal{P}_\mu^{std,c,an}$-torsor (corresponding to the filtration on $W_{dR}$) that we denote by $\mathcal{P}^{an}_{dR}$, and we have $$\mathcal{M}^{an}_{dR} = \mathcal{P}^{an}_{dR} \times^{\mathcal{P}_\mu^{std,c,an}} \mathcal{M}_\mu^{c,an}.$$

On the other hand, the torsor $\mathcal{G}_{pet,p}^{an} \times^{G^c(\qq_p)} \mathcal{G}^{c,an}$ has a reduction of structure group to a $\mathcal{P}_\mu^{c,an}$-torsor (corresponding to the filtration on $W_{p}  \otimes_{\qq_p} \hat{\oscr}_{\mathcal{S}_{K,\Sigma}^{tor}} $) that we denote by $\mathcal{P}^{an}_{HT}$, and we define $$\mathcal{M}^{an}_{HT} = \mathcal{P}^{an}_{HT} \times^{\mathcal{P}_\mu^{c,an}} \mathcal{M}_\mu^{c,an}.$$
We have the identification of torsors on the pro-\'etale site $\mathcal{M}^{an}_{HT} = \mathcal{M}^{an}_{dR}$, compatible with the Hecke action. This implies that $\mathcal{M}^{an}_{HT}$ is already defined on the \'etale site.

We can consider the inverse limit in the category of diamonds $\mathcal{S}_{K^p,\Sigma}^{tor, \Diamond} = \lim_{K_p} \mathcal{S}_{K^pK_p,\Sigma}^{tor, \Diamond}$.

To summarize the situation, we have the following theorem, which is really a consequence of \cite{diao2019logarithmic}, thm. 5.3.1.

\begin{thm}\label{thm-diamond-general-Shimura}  The torsor $\mathcal{G}_{pet,p}^{an}$ is trivial over $\mathcal{S}_{K^p,\Sigma}^{tor, \Diamond}$, and we therefore obtain a morphism: 
$$\pi_{HT}^{tor} : \mathcal{S}_{K^p,\Sigma}^{tor, \Diamond} \rightarrow \mathcal{FL}_{G,\mu}.$$
Moreover, the pull back of the torsor $\mathcal{G}^{c,an}/U_{\mathcal{P}_\mu^{an}} \rightarrow  \mathcal{FL}_{G,\mu}$ via  $ \pi_{HT}^{tor} : \mathcal{S}_{K^p,\Sigma}^{tor, \Diamond} \rightarrow \mathcal{FL}_{G,\mu}$ is $\mathcal{M}^{an}_{HT}$, and   this is also  the pull back of $\mathcal{M}_{dR}^{an}$ via the map $ \mathcal{S}_{K^p,\Sigma}^{tor, \Diamond} \rightarrow \mathcal{S}_{K^pK_p,\Sigma}^{tor}$. 
\end{thm}

We have thus extended the picture to the case of general Shimura varieties at the expense of working with diamonds. This is unfortunately not enough information for our purposes. In particular, we need to address the question of affineness of the Hodge-Tate period map. For this reason, we restrict to abelian type Shimura data and work out the connection between abelian and Hodge type Shimura varieties in the next section. 

\subsubsection{Abelian type Shimura varieties}\label{section-Abelian-type2}

We now consider an abelian type Shimura datum $(G,X)$. This means that there is a Hodge type Shimura datum $(G_1,X_1)$ such that there is a central isogeny $G_1^{der}\rightarrow G^{der}$ inducing an isomorphism of the associated connected Shimura data $(G^{ad}, X^+)$ and $(G_1^{ad}, X_1^+)$ (see \cite{MR2192012}, definition 4.4).

We recall a nice way to connect the Shimura data $(G,X)$ and $(G_1,X_1)$ from \cite{MR3720933}, section 4.6. Let  $E$ be the composite of the reflex fields of both Shimura data.
One can construct  a diagram of  Shimura data 
\begin{eqnarray*}
\xymatrix{ (B_1,X_{B_1}) \ar[r] \ar[d] & (B, X_B) \ar[d] \\
(G_1,X_1) &  (G,X)}
\end{eqnarray*}
where $B_1 \rightarrow G_1$ and $B \rightarrow G$ induce isomorphisms $B_1^{der} \simeq G_1^{der}$ and $B^{der} \simeq G^{der}$, and
the map $B_1 \rightarrow B$ is a central isogeny. 

Actually $B_1 = G_1 \times_{G_{1}^{ab}} \mathrm{Res}_{E/\qq} \mathbb{G}_m$ for a suitable map $ \mathrm{Res}_{E/\qq} \mathbb{G}_m \rightarrow  G_{1}^{ab}$ induced by the cocharacter $\mu_{G_1}$. We let $T = \mathrm{Res}_{E/\qq} \mathbb{G}_m$. 

Since we are going to deal with several Shimura varieties at the same time in this section,  we change our notation, and denote by $S(G,X)_K$ the Shimura variety attached to a datum $(G,X)$ and compact open $K$.  We will similarly write $M_{dR}(G,X)$ rather than $M_{dR}$, etc.  We also adopt the standard notation that for $H/\qq$ a reductive group, $H^{ad}(\qq)^+$ is the intersection of $H^{ad}(\qq)$ with the identity component of $H^{ad}(\mathbb{R})$, while $H(\qq)_+$ is the preimage of $H^{ad}(\qq)^+$ under $H(\qq)\to H^{ad}(\qq)$.

Let us first recall some classical results. We choose connected components $X_{B_1}^+$, $X_B^+$, $X_1^+$ and $X^+$ compatibly with our morphisms of Shimura data. This allows us to identify the set of geometric connected component of a Shimura variety $S(H,X_H)_K$ with ${H(\qq)_+} \backslash H(\mathbb{A}_f)/K$.  We let $S^0(H,X_H)_K$ be the connected component corresponding to the class of $1$ (we allow ourselves to extend the base field). We have $S^0(H,X_H)_K(\C) = \Gamma_K \backslash X_H^+$ where $\Gamma_K =  H(\qq)_+ \cap K $. 
We adopt the same notation of adding a supscript $0$ for the connected components of minimal and toroidal compactifications. 

We recall a functoriality between Shimura data and compactifications. Let $g : (H,X_H) \rightarrow (S,X_S)$ be a morphism of Shimura data. We assume that  $g$ induces an isomorphism $H^{ad} = S^{ad}$. Let $K' \subseteq H(\mathbb{A}_f)$ and $K \subset S(\mathbb{A}_f)$ be compact open subgroups such that $g(K') \subseteq K$. Let $\Sigma$ be a cone decomposition for $S$. 

\begin{prop}\label{proposition-toro-functo} The following point are satisfied. 

\begin{enumerate} 
\item The morphism $S(H, X_H)_{K'} \rightarrow S(S,X_S)_{K}$ is  finite \'etale.
\item The cone decomposition $\Sigma$ for $S$ induces a cone decomposition for $H$ and the morphism 
$S^{tor}(H, X_H)_{K',\Sigma} \rightarrow S^{tor}(S,X_S)_{K,\Sigma}$ is finite. 
\item If $g(K')$ is normal inside $K$, the morphism $S^0(H, X_H)_{K'} \rightarrow S^0(S,X_S)_{K}$ is Galois with finite group $\Delta(K,K')$. The action of $\Delta(K,K')$ extends to an action on $S^{tor,0}(H, X_H)_{K',\Sigma}$,  and $ S^{tor,0}(S,X_S)_{K,\Sigma}$ is the quotient of $S^{tor,0}(H,X_H)_{K'}$ by the action of $\Delta(K,K')$. 
\end{enumerate}
\end{prop}

\begin{proof} For the first point we reduce to check this on connected components over the complex numbers. We have $X_H^+ = X_S^+=X^+$. We have $S^0(H, X_H)_{K'} = X^+/\Gamma_{K'}$ and $S^0(S, X_S)_{K} = X^+/\Gamma_{K}$. The images of $\Gamma_{K}$ and $\Gamma_{K}'$  in $H^{ad}(\qq)^+$, denoted respectively $\Gamma^{ad}_{K}$ and $\Gamma^{ad}_{K'}$,  are arithmetic subgroups by \cite{MR2192012}, prop. 3.2. We deduce that $\Gamma^{ad}_{K'}$ has finite index in  $\Gamma^{ad}_{K}$.  
We next claim that $\Sigma$ also induces a cone decomposition for $H$. By our assumption, the map $P \mapsto g^{-1}(P)$ induces a bijection between the rational parabolics of $S$ and $H$ and this is basically all we need. See \cite{MR997249}, sect. 2.5.  The finiteness of the map follows from the description of the local charts. 
For the last point, the group $\Delta(K,K')$ is $\Gamma_{K}^{ad}/\Gamma_{K'}^{ad}$. By normality, the action of $\Delta(K,K')$ extends and the last point follows.
\end{proof}

We also recall a construction of fiber products of torsors. Let $S$ be a scheme and let $H_1,H_2,H_3$ be three group schemes with morphisms $H_1 \rightarrow H_2$  and $H_3 \rightarrow H_2$. We assume that we have $H_i$-torsors $P_i$ over $S$ and isomorphisms $\theta_1 : P_{H_1} \times^{H_1} H_2 \simeq P_{H_2}$ and $\theta_3 : P_{H_3} \times^{H_3} H_2 \simeq P_{H_2}$. Then one can form $P_{H_1 \times_{H_2} H_3} = P_{H_1} \times_{P_{H_2}} P_{H_1}$ which is an $H_1 \times_{H_2} H_3$ torsor.

\begin{thm}\label{thm-abelian-torsor}  The following points are satisfied: 
\begin{enumerate} 
\item The towers of connected components of the Shimura varieties $\lim_{K \subseteq B_1(\mathbb{A}_f)}S^0(B_1,X_{B_1})_{K}$ and $\lim_{K \subseteq G_1(\mathbb{A}_f)} S^0(G_1,X_1)_{K}$ are canonically isomorphic. 
The same holds for the minimal and toroidal compactifications. 

\item  Let $K^p \subseteq B_1(\mathbb{A}_f^p)$ be compact open. There exists a compact open subgroup $(K^p)' \subseteq G_1(\mathbb{A}_f^p)$  such that we have a finite \'etale Galois  morphism:  
$$\lim_{K'_p \subseteq G_1(\qq_p)}S^0(G_1,X_1)_{(K^p)' K_p'} \rightarrow \lim_{K_p \subseteq B_1(\qq_p)} S^0(B_1,X_{B_1})_{K^pK_p}.$$
\item Let $K\subseteq B_1(\mathbb{A}_f)$ be a compact open subgroup and $\Sigma$ be a cone decomposition for $G_1$. 
There are compact open subgroups $K_1 \subseteq G_1(\mathbb{A}_f)$, $K_2 \subseteq T(\mathbb{A}_f)$, $K_3 \subseteq G^{ab}_1(\mathbb{A}_f)$,
\begin{eqnarray*}
\xymatrix{  {S}^{tor}(B_1, X_{B_1})_{{K},\Sigma}  \ar[r]^{\pi_1} \ar[d]^{\pi_2} &  {S}^{tor}(G_1, X_{_1})_{{K_{1}},\Sigma} \ar[d] \\
 {S}(T, X_{T})_{K_{2}} \ar[r] &  {S}(G_1^{ab}, X_{G_1^{ab}})_{K_{3}}}
 \end{eqnarray*}
Denote by $\pi_3 : {S}^{tor}(B_1, X_{B_1})_{{K},\Sigma} \rightarrow {S}(G_1^{ab}, X_{G_1^{ab}})_{K_{3}}$. 
 The torsor $M_{dR}(B_1,X_{B_1})$ over ${S}^{tor}(B_1, X_{B_1})_{{K^pK_p},\Sigma}$  is canonically isomorphic to the fiber product  $$\pi_1^\star M_{dR}(G_1,X_1) \times_{ \pi_3^\star M_{dR}(G_1^{ab}, X_{G_1^{ab}})} \pi_2^\star M_{dR}(T,X_{T}).$$

\item  Let $g : B_1 \rightarrow G$.   Let $K\subseteq G(\mathbb{A}_f)$ be neat open compact.  There exists a compact open subgroup $K' \subseteq B_1(\mathbb{A}_f)$ such that $g(K') \subseteq K$ and the morphism $S^0(B_1,X_{B_1})_{K'} \rightarrow S^0(G,X)_{K}$ is finite \'etale and Galois with group $\Delta(K,K')$.

\item   Let $\Sigma$ be a cone decomposition for $G$. The torsor $M_{dR}(G,X)$ over $S^{tor,0}(G,X)_{K,\Sigma}$  is  the quotient by $\Delta(K,K')$ of the torsor $$M_{dR}(B_1,X_1) \times^{M_{\mu_{B_1}}^c} M_{\mu_{G}}^c$$  over $S^{tor,0}(B_1,X_{B_1})_{K',\Sigma}$.

\item Let $K^p \subseteq G(\mathbb{A}^p_f)$. There exists a compact open subgroup $(K')^p \subseteq B_1(\mathbb{A}_f^p)$ such that $g((K')^p) \subseteq K^p$ and the morphism 
$$\lim_{K_p' \subseteq B_1(\qq_p)} S^0(B_1,X_{B_1})_{(K')^pK'_p} \rightarrow \lim_{K_p \subseteq G(\qq_p)}S^0(G,X)_{K^p K_p}$$ is finite \'etale and Galois.

\end{enumerate}
\end{thm}
\begin{proof} 
We recall  that for a reductive group $H$ defined over $\qq$, the topologies on $H^{ad}(\qq)^+$ defined by the  images of congruence subgroups of $H(\qq)_+$ and $H^{der}(\qq)_+$ are the same. This implies that the connected components of the tower of Shimura varieties is determined by the connected Shimura datum. By construction, the  datum $(G_1^{der}, X_1^+)$ and $(B^{der}_1, X_{B_1}^+)$  are the same. 

We now check the second point. Let us fix a decreasing sequence of compact open subgroups $K_{p,n} \subseteq B_1(\qq_p)$ with $\cap_n K_{p,n} = \{1\}$ and $K_{p,n}$ normal inside $K_{p,1}$. Let $K_{p,n}^{ad}$ be the image of $K_{p,n}$ in $B_1^{ad}(\qq_p)$.
Let $\Gamma_n= B_1(\qq)_+ \cap K^p K_{p,n}$ and let $\Gamma_n^{ad}$ the image of $\Gamma_n$ in $B_1^{ad}(\qq)^+$. 
 
 By the first point, there exists a compact open subgroup $(K^p)'K_{p,1}' \subseteq G_1(\mathbb{A}_f)$ such that $\Gamma'_1 =  G_1(\qq)_+ \cap (K^p)' K'_{p,1}$ has the property that its image in $B_1^{ad}(\qq)^+$ is a finite index normal subgroup of $\Gamma_1^{ad}$. 
 Actually, we may also choose $\Gamma'_1$ such that $\Gamma'_1 \subseteq G_1^{der}(\qq)_+$ (this follows from the fact that $G_1$ is of Hodge type, hence $G_1^{ab}(\qq)$ is discrete in $G_1^{ab}(\mathbb{A}_f)$). 

We may also assume that the map $G_1^{der}(\qq)_+ \rightarrow B_1^{ad}(\qq)^+$ induces an isomorphism from $\Gamma'_1$ to its image (indeed, the kernel of this map is finite).  
Let us denote by $\Gamma'_n = \Gamma^{ad}_n \cap \Gamma'_1$.  Let $\Delta_n = \Gamma_n^{ad}/\Gamma_n'$. Clearly $\Delta_{n} \rightarrow \Delta_{n-1}$ is injective and therefore $\Delta = \lim_n \Delta_n$ is a finite group. 

Since the map $G_1^{der}(\qq_p) \rightarrow B_1^{ad}(\qq_p)$ is a local homeomorphism, there is for $n$ large enough a subgroup $K_{p,n}^{der} \subseteq G_1^{der}(\qq_p)$ which maps isomorphically to $K_{p,n}^{ad}$, the image of $K_{p,n}$ in $B_1^{ad}(\qq_p)$. We deduce that $\Gamma'_n = \Gamma'_1 \cap  K^{der}_{p,n}$. Finally, we can find a decreasing sequence of compact open subgroups $\{K'_{p,n} \subseteq G_1(\qq_p)\}$  such that $\cap_n K'_{p,n} = \{1\}$ and $K'_{p,n} \cap G^{der}_1(\qq_p)=  K_{p,n}^{der}$. 
We deduce that $\lim_{K'_{p,n}}S^0(G_1,X_{B_1})_{(K^p)' K'_{p,n}} \rightarrow \lim_{K_{p,n} } S^0(B_1,X_{B_1})_{K^pK_{p,n}}$ is finite \'etale with group $\Delta$. 
For the third point, using the various functorialities of Shimura varieties (see for instance \cite{MR3720933}, lemma 3.1.6), we  see that we have a map of torsors $$M_{dR}(B_1,X_{B_1}) \rightarrow  \pi_1^\star M_{dR}(G_1,X_1) \times_{ \pi_3^\star M_{dR}(G_1^{ab}, X_{G_1^{ab}})} \pi_2^\star M_{dR}(T,X_{T}).$$ and a map of torsors is an isomorphism. 
We check the fourth point. We can find a normal compact open subgroup $K' \subseteq g^{-1}(K)$ and we apply proposition \ref{proposition-toro-functo}. 
The fifth point follows again from the functorialities between Shimura varieties. 
We now check the last point. Let us fix a decreasing sequence of compact open subrgoups $\{K_{p,n} \subseteq G(\qq_p)\}$ with $\cap_n K_{p,n} = \{1\}$ and $K_{p,n}$ is normal inside $K_{p,1}$. Let us define $\Gamma_n = G(\qq)_+ \cap K^pK_{p,n}$. 
Let us fix a compact open subgroup $(K')^pK_{p,1}' \subseteq B_1(\mathbb{A}_f)$ with $g((K')^pK_{p,1}') \subseteq K^pK_{p,n}$.  Let $K_{p,n}' = g^{-1}(K_{p,n}) \cap K'_{p,1}$. Let $K_{p,n}^{der} = K'_{p,n} \cap B^{der}_1(\qq_p)$. 
Let $\Gamma'_n = B_1(\qq)_+ \cap (K^p)'K'_{p,n}$. Let $\Gamma^{der}_n = \Gamma_n \cap B_1(\qq)_+^{der}$. 
For any of $\Gamma_n, \Gamma'_n, \Gamma_n$, we add a subscript $ad$ to mean their image in $G^{ad}(\qq)^+$. 
We see that $(\Gamma_n^{der})^{ad} \subseteq (\Gamma'_n)^{ad} \subseteq \Gamma_n^{ad}$ are all arithmetic subgroups.  Let $\Delta_n = \Gamma_n^{ad}/(\Gamma_n^{der})^{ad}$. We see that $\Delta_n \rightarrow \Delta_{n-1}$ is injective and therefore $\Delta = \lim_n \Delta_n$ is a finite group. 
We observe that $\cap_n K_{p,n}^{der} = \mathrm{Ker}( B_1^{der}(\qq_p) \rightarrow G(\qq_p))$ is a finite group. 
Let us fix a compact open subgroup $L_p \subseteq B_1^{der}(\qq_p)$ such that $L_p \cap \mathrm{Ker}( B_1^{der}(\qq_p) \rightarrow G(\qq_p))= \{1\}$.  We let $\Gamma^{der,'}_n = \Gamma^{der}_n  \cap L_p$. 
The morphism $\lim_n \Gamma^{der,'}_n \backslash X^+ \rightarrow \lim_n \Gamma_n \backslash X^+$ is finite \'etale and Galois. Since this map factors the map $\lim_{K_p' \subseteq B_1(\qq_p)} S^0(B_1,X_{B_1})_{(K')^pK'_p}(\C) \rightarrow \lim_{K_p \subseteq G(\qq_p)}S^0(G,X)_{K^p K_p}(\C)$, we conclude. 
\end{proof}

In this paper, we will be able to reduce a number of statements (and in particular vanishing theorems) on  abelian type Shimura varieties to the Hodge type case using the   following  principle:
\begin{Principle}\label{abelianstrategy}
\begin{enumerate} 
\item We can extend  a construction (eg. of a torsor) from the Shimura variety for $(G_1,X_1)$ to the Shimura variety for $(B_1, X_{B_1})$   by   performing a (trivial) construction on the Shimura variety attached to  $T$  and using the identity $M_{\mu_{B_1}}^c = M^c_{\mu_{G_1}} \times_{M^{ab,c}_{\mu_{G_1}}} T^c$.
\item The torsors we construct on a connected component of a Shimura variety for $(G,X)$ is obtained  from a torsor on a connected  component of the Shimura variety for $(B_1,X_{B_1})$ by   pushing out along the map  $B_1^c \rightarrow G^c$ and taking the quotient by a finite group.
\end{enumerate}
\end{Principle}

We  work until the end of this paragraph  over the algebraically closed field $\C\simeq \C_p$ (this means that all Shimura varieties are base changed to $\C_p$). This simplifies our treatment of connected components of Shimura varieties. We also recall that the rationality questions with respect to the reflex field are irrelevant for us. 

The principle \ref{abelianstrategy} is  based on the following theorem: 

\begin{thm}\label{thm-abelian-strategy} There is a commutative diagram 
\begin{eqnarray*} 
\xymatrix{  & \mathcal{S}^{\star}(B_1,X_{B_1}) \ar[ld] \ar[dd] \ar[rd]&  \\
 \mathcal{S}^{\star}(G_1,X_1)  \ar[rd] & & \mathcal{S}^{\star}(G,X)  \ar[ld]\\
 & \mathcal{FL}_{G,\mu} &}
 \end{eqnarray*}
where $\mathcal{S}^{\star}(B_1,X_{B_1})$, $\mathcal{S}^{\star}(G_1,X_1)$, and $\mathcal{S}^{\star}(G,X)$ are perfectoid spaces representing (in the category of diamonds)
$\lim_K \mathcal{S}^{\star}(B_1,X_{B_1})_{K}$, $\lim_K \mathcal{S}^{\star}(G_1,X_{G_1})_{K}$, and $\lim_K \mathcal{S}^{\star}(G,X)_{K}$.
The maps are equivariant for the respective actions of $B_1(\mathbb{A}_f)$, $G_1(\mathbb{A}_f)$, and $G(\mathbb{A}_f)$. 
The three maps to the flag variety $\mathcal{FL}_{G,\mu} = \mathcal{FL}_{G_1,\mu_{G_1}} = \mathcal{FL}_{B_1,\mu_{B_1}}$ are  Hodge-Tate period maps, and the composite map $\mathcal{S}^{tor,\Diamond}(G,X)_{\Sigma} = \lim_K   \mathcal{S}^{tor,\Diamond}(G,X)_{K,\Sigma} \rightarrow \mathcal{S}^{\star}(G,X) \rightarrow   \mathcal{FL}_{G,\mu} $ is the map $\pi_{HT}^{tor}$ of theorem \ref{thm-diamond-general-Shimura}. 
\end{thm}

\begin{rem} In the course of the proof, we also introduce spaces  $\overline{\mathcal{S}}^{\star}(G,X) = \lim_K {\mathcal{S}}^{\star}(G,X)_{\overline{K}} $ which depend on the embedding of $(G_1,X_1)$ in a Siegel datum. We have finite  maps ${\mathcal{S}}^{\star}(G,X)_{{K}} \rightarrow {\mathcal{S}}^{\star}(G,X)_{\overline{K}}$ which are isomorphisms away from the boundary. We prove that $\overline{\mathcal{S}}^{\star}(G,X) \sim \lim_K {\mathcal{S}}^{\star}(G,X)_{\overline{K}} $ and that the map $\pi_{HT}$ factors through an affinoid map $\pi_{HT} : \overline{\mathcal{S}}^{\star}(G,X)  \rightarrow \mathcal{FL}_{G,\mu}$. See proposition \ref{prop-final-perfectoid-abelian}.
\end{rem}

\begin{rem} Perfectoid (minimal compactifications of)  abelian type Shimura varieties  have been constructed for the first time in \cite{DHansen2}. Since we also need to construct the Hodge-Tate period map and prove several compatibilities, we give a complete argument. The argument is very similar to \cite{DHansen2} in the sense that this is a reduction to the Hodge type case. 
\end{rem} 

\subsubsection{Proof of theorem \ref{thm-abelian-strategy}}
 We first  briefly recall how to reconstruct a Shimura variety from its connected Shimura variety following  \cite{MR546620}, section 2.1.  
Let $(G,X)$ be a Shimura datum. 
The center $Z(G)$ of $G$ is simply denoted by $Z$ unless some confusion may arise. 

The inverse system $S(G,X) = \lim_K S(G,X)_K$ has a right  action of the group $\mathcal{A}(G) = G(\mathbb{A}_f)/\overline{Z(\qq)} \ast_{G(\qq)_+} G^{ad}(\qq)^+$. Here $\overline{Z(\qq)}$ is the closure of $Z(\qq)$ in $G(\mathbb{A}_f)$. 
\begin{rem} If $Z_{s} = \{1\}$,   then $\overline{Z(\qq)} = Z(\qq)$.  If the morphism $G(\qq) \rightarrow G^{ad}(\qq)$ is surjective (for example, if $Z(G)$ is a split torus), then $\mathcal{A}(G)  = G(\mathbb{A}_f)/\overline{Z(\qq)}$. 
\end{rem}

 There is a canonical  bijection  $\pi_0 (S(G,X)) \rightarrow \overline{G(\qq)_+} \backslash G(\mathbb{A}_f)$ of right $\mathcal{A}(G)$-profinite sets where $\overline{G(\qq)_+} $ is the closure of $G(\qq)_+$ in $G(\mathbb{A}_f)$.   This is also the completion of $G(\qq)_+$ for the topology where a basis of open neighborhoods of the identity is given by the congruence subgroups of $G(\qq)_+$.
 
   Let us pick $1 \in  \overline{G(\qq)_+} \backslash G(\mathbb{A}_f)$ and let $S^0(G,X) \hookrightarrow S(G,X)$ be the connected component corresponding to $1$. This is the tower of connected Shimura varieties. The stabilizer of $S^0(G,X)$ for the action of $\mathcal{A}(G)$ is $\mathcal{A}^0(G) = \overline{G(\qq)_+}/\overline{Z(\qq)} \ast_{G(\qq)_+} G^{ad}(\qq)^+$ and we have the formula:
$$ S(G,X) = [S^0(G,X) \times \mathcal{A}(G)]/\mathcal{A}^0(G)$$
for the action $(s,g).h = (sh, h^{-1}g)$. 
 One important observation is that $S^0(G,X)$ depends only on the connected Shimura datum $(G^{der}, X^+)$ and the group $\mathcal{A}^0(G)$ depends also only on $G^{der}$. More precisely, we have that 
 $\mathcal{A}^0(G) = \overline{G^{der}(\qq)_+} \ast_{G^{der}(\qq)_+} G^{ad}(\qq)^+$, and  $\mathcal{A}^0(G)$ is therefore the completion of $G^{ad}(\qq)^+$ with respect to the topology where a basis  of open neighborhoods of $1$ is the images of congruence subgroups of $G^{der}(\qq)_+$. 
 We have that $S^0(G,X) = \lim_{K} S^0(G,X)/K$ where the limit runs over the compact open subgroups of $\overline{G^{der}(\qq)_+}$,  and  $ (S^0(G,X)/K)(\C) = \Gamma_K \backslash X^+$ where $\Gamma_K = G^{der}(\qq)_+\cap K$. 
 
  Of course, there is also a simpler formula: $ S(G,X) = [S^0(G,X) \times G(\mathbb{A}_f)]/\overline{G(\qq)_+}$, but the disadvantage of this formula is that $\overline{G(\qq)_+}$ depends on $G$!

All of this extends to the minimal compactification 
since a Shimura variety and its minimal compactification have the same connected components, and group actions extend by normality of minimal compactifications.

We now go back to our Hodge type datum $(G_1,X_1)$, and we fix an embedding in a Siegel datum $(\tilde{G}, \tilde{X})$. Attached to this fixed embedding, for any compact open $K$ we have  the minimal compactification $S^\star(G_1,X_1)_{K}$ together with a finite surjective map to $S^\star(G_1,X_1)_{\overline{K}}$, where $S^\star(G_1,X_1)_{\overline{K}}$ is the schematic image of $S^\star(G_1,X_1)_{K} \rightarrow S^\star(\tilde{G}, \tilde{X})_{\tilde{K}}$ for  all small enough $\tilde{K}$ with $\tilde{K} \cap G(\mathbb{A}_f) = K$. Let us denote $\overline{S}^\star(G_1,X_1) = \lim_K S^\star(G_1,X_1)_{\overline{K}}$. 
For the Siegel datum $(\tilde{G}, \tilde{X})$, we have that $\mathcal{A}(\tilde{G}) = \tilde{G}(\mathbb{A}_f)/Z(\tilde{G})(\qq)$. We deduce that the closed subgroup $\mathcal{A}(G_1)$ of $\mathcal{A}(\tilde{G})$ acts on $\overline{S}^\star(G_1,X_1)$ and ${S}^\star(G_1,X_1)$ in a compatible way. We also deduce that the Hodge-Tate period map
$ \mathcal{S}^\star(G_1,X_1) \rightarrow \overline{\mathcal{S}}^\star(G_1,X_1) \rightarrow \mathcal{FL}_{G_1, \mu_{G_1}}$ is $\mathcal{A}(G_1)$-equivariant (by reduction to the Siegel case).  We remark that the $\mathcal{A}(G_1)$-action on $\mathcal{FL}_{G_1, \mu_{G_1}}$ factors through the map $\mathcal{A}(G_1) \rightarrow G_1^{ad}(\mathbb{A}_f) \rightarrow G^{ad}_1(\qq_p)$.

We are now ready to extend things to $(B_1, X_{B_1})$. First, we have that $S^\star(B_1,X_{B_1}) = [S^{\star,0}(G_1,X_1) \times \mathcal{A}(B_1)]/\mathcal{A}^0(B_1)$. We may also define  $\overline{S}^\star(B_1,X_{B_1}) = [\overline{S}^{\star,0}(G_1,X_1) \times \mathcal{A}(B_1)]/\mathcal{A}^0(B_1)$. By taking $K$-invariants, we define $S^\star(B_1,X_{B_1})_{\overline{K}}=\overline{S}^\star(B_1,X_{B_1})/K$. 

Our first lemma is: 

\begin{lem}\label{existence-perf-B1} There is a perfectoid space $ \mathcal{S}^{\star}(B_1,X_{B_1})_{\overline{K^p}} \sim \lim_{K_p}  \mathcal{S}^{\star}(B_1,X_{B_1})_{\overline{K^pK_p}}$ and a perfectoid space $ \mathcal{S}^{\star}(B_1,X_{B_1})_{{K^p}} =  \lim_{K_p}  \mathcal{S}^{\star,\Diamond}(B_1,X_{B_1})_{{K^pK_p}}$.
\end{lem}
\begin{proof} We prove the first statement. Fix a compact open subgroup $K_p\subseteq B_1(\qq_p)$.  It acts on the scheme ${S}^\star(B_1,X_{B_1})_{\overline{K^p}}$, and it acts on the connected components with finitely many orbits. We choose representatives for the orbits, which are all of the form ${S}^{\star,0}(B_1,X_{B_1})_{\overline{k_iK^pk_i^{-1}}}$ for suitable elements $k_i \in B_1(\mathbb{A}_f^p)$. Let $S(k_i)$ be the stabilizer of the connected component ${S}^{\star,0}(B_1,X_{B_1})_{\overline{k_i K^pk_i^{-1}}}$ in $K_p$. This is a closed subgroup of $K_p$. 
We deduce that ${S}^\star(B_1,X_{B_1})_{\overline{K^p}} = \coprod_i  [{S}^{\star,0}(B_1,X_{B_1})_{\overline{k_iK^pk_i^{-1}}} \times K_p]/{S(k_i)}$.

By theorem \ref{thm-abelian-torsor} (2), we can find compact open subgroups $ (K_i^p)' \subseteq G_1(\mathbb{A}^p_f)$ such that 
 we have finite morphisms $S^{\star,0}(G_1,X_{1})_{\overline{(K_i^p)'} } \rightarrow S^{\star,0}(B_1,X_{B_1})_{\overline{k_iK^pk_i^{-1}}}$, identifying $S^{\star,0}(B_1,X_{B_1})_{\overline{k_iK^pk_i^{-1}}}$ as the quotient of  $S^{\star,0}(G_1,X_{1})_{\overline{(K_i^p)'} } $ by a finite group. 

Now we pass to analytic geometry.  
It follows from  lemma \ref{lem-connected-perfectoid} and  the results of \ref{Section-Hodge-perfectoid} that the $\mathcal{S}^{\star,0}(G_1,X_{1})_{\overline{(K_i^p)'} } $ are perfectoid. We deduce from lemma \ref{lem-group-perfectoid} that the $\mathcal{S}^{\star,0}(B_1,X_{B_1})_{\overline{k_iK^pk_i^{-1}}}$ are perfectoid spaces.

We now want to define  $\mathcal{S}^\star(B_1,X_{B_1})_{\overline{K^p}} = \coprod_i [\mathcal{S}^{\star,0}(B_1,X_{B_1})_{\overline{k_iK^pk_i^{-1}}} \times K_p]/{S(k_i)}$.  The product  $\mathcal{S}^{\star,0}(B_1,X_{B_1})_{\overline{k_iK^pk_i^{-1}}} \times (K_p/{S(k_i)})$ is a  perfectoid space, being the product of a perfectoid space with a profinite set. 
Checking that $\mathcal{S}^{\star}(B_1,X_{B_1})_{\overline{K^p}} \sim \lim_{K_p}  \mathcal{S}^{\star}(B_1,X_{B_1})_{\overline{K^pK_p}}$ is now an easy exercise, again using \ref{Section-Hodge-perfectoid}. 
We deduce that $ \mathcal{S}^{\star}(B_1,X_{B_1})_{{K^p}} =  \lim_{K_p}  \mathcal{S}^{\star,\Diamond}(B_1,X_{B_1})_{{K^pK_p}}$ by \cite{bhatt2019prisms}, thm. 1.16.
\end{proof}

We can now construct the Hodge-Tate period map.

\begin{lem}\label{lem-HT-equivariant} We have a  Hodge-Tate period map  $$\pi_{HT} :  \mathcal{S}^{\star}(B_1,X_{B_1})_{{K^p}} \rightarrow \mathcal{S}^{\star}(B_1,X_{B_1})_{\overline{K^p}} \rightarrow \mathcal{FL}_{B_1, \mu_{B_1}}$$ which agrees on connected components with Hodge-Tate period map for $G_1$. 
The Hodge-Tate period map is $B_1(\qq_p)$-equivariant and Hecke equivariant away from $p$. The map $\mathcal{S}^{\star}(B_1,X_{B_1})_{\overline{K^p}} \rightarrow \mathcal{FL}_{B_1, \mu_{B_1}}$ is affinoid in the sense that $\mathcal{FL}_{B_1, \mu_{B_1}}$ has a cover by affinoid opens whose preimages in $\mathcal{S}^{\star}(B_1,X_{B_1})_{\overline{K^p}}$ are good affinoid perfectoid opens.
\end{lem}
\begin{proof} We may pass to the limit over $K^p$ and consider the map $\overline{\mathcal{S}}^{\star}(B_1,X_{B_1}) = [\overline{\mathcal{S}}^{\star,0}(G_1,X_1) \times \mathcal{A}(B_1)]/\mathcal{A}^0(B_1) \rightarrow \mathcal{FL}_{B_1, \mu_{B_1}}$ sending $(x,a) \in \overline{\mathcal{S}}^{\star,0}(G_1,X_1) \times \mathcal{A}(B_1)$ to $a.\pi_{HT}(x)$ where $\pi_{HT}(x) \in  \mathcal{FL}_{B_1, \mu_{B_1}}  =  \mathcal{FL}_{G_1, \mu_{G_1}}$ and $\mathcal{A}(B_1)$ acts via its quotient $B^{ad}_1(\qq_p)$ on $\mathcal{FL}_{B_1, \mu_{B_1}}$. Since the map $ \overline{\mathcal{S}}^{\star,0}(G_1,X_1)  \rightarrow  \mathcal{FL}_{B_1, \mu_{B_1}}$ is $\mathcal{A}^0(G_1)$-equivariant, we deduce that we have a well-defined  $\mathcal{A}(B_1)$-equivariant map $\overline{\mathcal{S}}^{\star}(B_1,X_{B_1}) \rightarrow  \mathcal{FL}_{B_1, \mu_{B_1}}$. The affine property follows from the Hodge case. 
\end{proof}

We now consider the descent from $(B_1, X_{B_1})$ to $(G,X)$. We have a map  $f : B_1 \rightarrow G$ inducing an isogeny $B_1^{der} \rightarrow G^{der}$. 

\begin{lem} \begin{enumerate}
\item  We have a continuous surjective map $\mathcal{A}^0(B_1) \rightarrow \mathcal{A}^0(G)$ with kernel a pro-finite group $\Delta$. 
\item We have a pro-finite \'etale morphism $S^0(B_1,X_{B_1}) \rightarrow S^0(G,X)$ which is Galois with group $\Delta$.
\end{enumerate}
\end{lem}
\begin{proof}  Let $Cong(B_1)$ be the set of congruence subgroups of $B_1^{der}(\qq)_+$, and let $Cong(G)$ be the set of congruence subgroups of $G^{der}(\qq)_+$. 
We claim that  the map $f : B_1^{der}(\qq)_+ \rightarrow G^{der}(\qq)_+$ induces a map   $f^{-1} : Cong(G) \rightarrow Cong(B_1)$. Indeed, let $K \subseteq G^{der}(\mathbb{A}_f)$ be  a compact open subgroup, then $f^{-1}(K)$ is compact open in $B_1^{der}(\mathbb{A}_f)$ and $f^{-1}( G^{der}(\qq)_+ \cap K ) = B_1^{der}(\qq)_+\cap f^{-1} (K)$. 
If $\Gamma \in Cong(B_1)$, then $f(\Gamma)$ is an arithmetic subgroup of $G^{der}(\qq)_+$ but not necessarily a congruence subgroup! 

We consider the index set  $I$ of compact open subgroups $ (K,K') \in B_1^{der}(\mathbb{A}_f) \times G^{der}(\mathbb{A}_f)$  such that $f(K) \subseteq K'$. 

Let $\Gamma_K = B_1^{der}(\qq)_+ \cap K$ and $\Gamma_{K'} = G^{der}(\qq)_+ \cap K'$. Let $\Delta(K,K') = \Gamma_{K} \backslash \Gamma_{K'}$. 

We observe that for any sufficiently small $\Gamma \in Cong(B_1)$ or $\Gamma \in Cong(G)$, the map $f\circ g : \Gamma \rightarrow G^{ad}(\qq)^+$ or $f : \Gamma \rightarrow G^{ad}(\qq)^+$ is injective. Indeed, remark that the  kernel of $G^{der}(\qq)_+ \rightarrow G^{ad}(\qq)^+$ is a finite group.

The groups $ g \circ f(Cong(B_1))$ and $f(Cong(G))$ are the basis of open neighborhoods of $1$ for  two topologies $\tau(B_1)$ and $\tau(G)$ on $G^{ad}(\qq)^+$. We know that $\mathcal{A}^0(B_1)$ and $\mathcal{A}^0(G)$ are the completion of $G^{ad}(\qq)^+$ for the topologies $\tau(B_1)$ and $\tau(G)$ respectively. 
Therefore, $\mathcal{A}^0(B_1) = \lim_{\Gamma \in Cong(B_1)} G^{ad}(\qq)^+/\Gamma$ and $\mathcal{A}^0(G) =  \lim_{\Gamma \in Cong(G)} G^{ad}(\qq)^+/\Gamma$. 
By the Mittag-Leffler criterion the sequence   $$1 \rightarrow \lim_{(K,K') \in I} \Delta(K,K') \rightarrow \lim_{I}  G^{ad}(\qq)^+/(\Gamma_{K}) \rightarrow \lim_{I} G^{ad}(\qq)^+/\Gamma_{K'} \rightarrow 1$$
is exact and equals $1 \rightarrow \Delta \rightarrow \mathcal{A}^0(B_1) \rightarrow \mathcal{A}^0(G) \rightarrow 1$. 
\end{proof}

We are now ready to descend everything from  $(B_1, X_{B_1})$ to $(G,X)$. First, we have that $S^{\star,0}(G,X) = S^{\star,0}(B_1, X_{B_1})/\Delta$ and we may also define $\overline{S}^{\star, 0} (G,X) =  \overline{S}^{\star,0}(B_1, X_{B_1})/\Delta$. We have  $S^{\star}(G,X)= [S^{\star,0}(G,X) \times \mathcal{A}(G)]/\mathcal{A}^0(G)$. We may also define  $\overline{S}^\star(G,X) = [\overline{S}^{\star,0}(G,X) \times \mathcal{A}(G)]/\mathcal{A}^0(G)$. By taking $K$-invariants, we define $\overline{S}^\star(G,X)/K = {S}^\star(G,X)_{\overline{K}}$. 

\begin{prop}\label{prop-final-perfectoid-abelian} Let $(G,X)$ be an abelian type Shimura datum as above.  \begin{enumerate}
\item There is a perfectoid space $ \mathcal{S}^{\star}(G,X)_{\overline{K^p}} \sim \lim_{K_p}  \mathcal{S}^{\star}(G,X)_{\overline{K^pK_p}}$ and a perfectoid space $ \mathcal{S}^{\star}(G,X)_{{K^p}} =  \lim_{K_p}  \mathcal{S}^{\star,\Diamond}(G,X)_{{K^pK_p}}$.
\item We have a Hodge-Tate period map $\pi_{HT} :  \mathcal{S}^{\star}(G,X)_{{K^p}} \rightarrow \mathcal{S}^{\star}(G,X)_{\overline{K^p}} \rightarrow \mathcal{FL}_{G, \mu}$ which is $G(\qq_p)$-equivariant and Hecke equivariant. The map $\mathcal{S}^{\star}(G,X)_{\overline{K^p}} \rightarrow \mathcal{FL}_{G, \mu}$ is affinoid in the sense that $\mathcal{FL}_{G, \mu}$ has a cover by affinoid opens whose preimages in $\mathcal{S}^{\star}(G,X)_{\overline{K^p}}$ are good affinoid perfectoid opens. 
\end{enumerate}
\end{prop}
\begin{proof} For the first point, we observe that the group $\Delta $ acts trivially on the flag variety. Since $\pi_{HT}$ was affinoid for $(B_1,X_{B_1})$, we can use theorem \ref{thm-abelian-torsor} (6), and  apply lemma \ref{lem-connected-perfectoid} and lemma \ref{lem-group-perfectoid} to deduce that $\mathcal{S}^{\star,0}(G,X)_{\overline{K^p}}$ is perfectoid. We may now extend the result from $\mathcal{S}^{\star,0}(G,X)_{\overline{K^p}}$ to $\mathcal{S}^{\star}(G,X)_{\overline{K^p}}$ as in the proof of lemma \ref{existence-perf-B1}. 

Passing to the limit over $K^p$, we have a Hodge-Tate period map $ \overline{\mathcal{S}}^{\star,0}(B_1,X_{B_1})/\Delta = \overline{\mathcal{S}}^{\star,0}(G,X) \rightarrow \mathcal{FL}_{G, \mu} = \mathcal{FL}_{B_1,\mu_{B_1}}$ which is $\mathcal{A}^0(G)$-equivariant. We deduce that there is a $\mathcal{A}(G)$-equivariant Hodge-Tate period map $\overline{\mathcal{S}}^{\star}(G,X) \rightarrow \mathcal{FL}_{G, \mu} $. 
The representability of $\lim_{K_p}  \mathcal{S}^{\star,\Diamond}(G,X)_{{K^pK_p}}$ follows from \cite{bhatt2019prisms}, thm. 1.16.

\end{proof}

\subsection{The truncated Hodge-Tate period map} In most of this paper, we work on finite level Shimura varieties rather than perfectoid Shimura varieties. For this reason we introduce some truncated Hodge-Tate period map. 

Let  $(G,X)$ be an abelian type Shimura datum. Let $K = K^p K_p \subseteq G(\mathbb{A}_f)$ be a compact open subgroup. 
 The group $K_p$ acts on $\mathcal{FL}_{G,\mu}$. We form the quotient space $ \mathcal{FL}_{G,\mu}/K_p$, equipped with the quotient topology from the surjective map    $ \pi_{K_p} : \mathcal{FL}_{G,\mu} \rightarrow \mathcal{FL}_{G,\mu}/K_p$. We merely view  $\mathcal{FL}_{G,\mu}/K_p$ as a topological space. We adopt some definitions. We say that an open $U \subseteq \mathcal{FL}_{G,\mu}/K_p$ is affinoid if $\pi_{K_p}^{-1}(U)$ is affinoid.  If $V\subseteq U\subseteq \mathcal{FL}_{G,\mu}/K_p$ are open we say that $V$ is a rational subset of $U$ if $\pi_{K_p}^{-1}(V)$ is a rational subset of $\pi_{K_p}^{-1}(U)$.  We let the residue field of a  point $x \in \mathcal{FL}_{G,\mu}/K_p$ be the residue field of any lift of this point to $\mathcal{FL}_{G,\mu}$.  To illustrate all these definitions, we have the following lemma: 

\begin{lem}\label{lem-example-neighborhoods} Any point $x \in \mathcal{FL}_{G,\mu}/K_p$  with finite residue field over $\qq_p$ (i.e. a classical point in the sense of rigid analytic geometry) has a basis of neighborhoods consisting of affinoids $\{U_n\}_{n \geq 1}$ with the property that $U_{n+1} \subseteq U_n$ is a rational subset. 
\end{lem}
\begin{proof} The group $\mathcal{G}$ admits a filtration by open affinoid normal subgroups $\mathcal{G}_n$ which form a basis of neighborhoods of the identity (take $\mathcal{G}_n$ the subgroup of elements which reduce to $1$ modulo $p^n$). Let $y \in  \mathcal{FL}_{G,\mu}$ be a lift of $x$.  Then   $y \mathcal{G}_n \hookrightarrow \mathcal{FL}_{G,\mu}$ is an affinoid for any $n \geq 1$ (this is  a closed tube centered at the point $y$, this is also where we use that $x$ has finite residue field over $\qq_p$). We now consider the group $\mathcal{G}_n K_p \subset \mathcal{G}$. The connected component of the identity of this group is $\mathcal{G}_n$, and the quotient $\mathcal{G}_n K_p/\mathcal{G}_n$ is a finite group. Then we find that  $y \mathcal{G}_nK_p = \coprod_{i \in I} y k_i \mathcal{G}_n$ for a finite set $I$ and elements $k_i \in K_p$, and is therefore an affinoid which is $K_p$-invariant. The $\{y \mathcal{G}_nK_p/K_p\}_{n\geq 1}$ form a basis of open neighborhoods of $x$. Moreover, $y \mathcal{G}_nK_p \subseteq y \mathcal{G}_{n-1}K_p$ is a rational subset because $y k_i \mathcal{G}_n \subseteq yk_i\mathcal{G}_{n-1}$ is a rational subset. 
\end{proof}

\bigskip

 The main result of this section  is the following:  

\begin{thm}\label{thm-affineness} There is a continuous map:
$$ \pi_{HT, K_p}  : \mathcal{S}^\star_{K} \rightarrow \mathcal{FL}_{G,\mu}/K_p$$
which is equivariant for the action  of the Hecke algebra $\mathcal{C}^{\infty}_c(K\backslash G(\mathbb{A}_f)/K, \ZZ)$ by correspondences. 

Moreover, any point $x \in \mathcal{FL}_{G,\mu}/K_p$  with finite residue field over $\qq_p$ (i.e. a classical point in the sense of rigid analytic geometry) has an affinoid   neighborhood $U$  such that  for any rational subset $V \subseteq U$, $(\pi_{HT, K_{p}})^{-1}(V)$ is affinoid. 
\end{thm}

By proposition \ref{prop-final-perfectoid-abelian} there is a  map $\pi_{HT} : \mathcal{S}^\star_{\overline{K^p}} \rightarrow \mathcal{FL}_{G,\mu}$ which is equivariant for the action of the $G(\qq_p)$. For  any point $x \in \mathcal{FL}_{G,\mu}$, there is an affinoid neighborhood $U$ of $x$ such that $\pi_{HT}^{-1}(U)$ is affinoid, and moreover, $\pi_{HT}^{-1}(U) = \lim_{K_p} \pi_{HT}^{-1}(U)_{K_p}$ where for $K_p$ small enough, $\pi_{HT}^{-1}(U)_{K_p} \hookrightarrow \mathcal{S}^\star_{\overline{K^pK_p}}$ is affinoid. 
We call an open affinoid in $\mathcal{FL}_{G,\mu}$ with these properties a \emph{very good affinoid}. Clearly,  a rational subset  of a  very good affinoid is a very good affinoid.

We can define  truncated Hodge-Tate period maps: 
$$\pi_{HT,K_p} : \mathcal{S}^\star_{K^pK_p}  \rightarrow \mathcal{S}^\star_{\overline{K^pK_p}}  \rightarrow \mathcal{FL}_{G,\mu}/K_p.$$

\begin{lem}\label{lemma-continuous} The map $\pi_{HT,K_p}$ is continuous.
\end{lem}
\begin{proof} We have a continuous map (of topological spaces) $\lim_{K'_p} \mathcal{S}^\star_{K^pK_p'} \rightarrow  \mathcal{S}^\star_{\overline{K^p}} \rightarrow \mathcal{FL}_{G,\mu}$. For all normal subgroup $K'_p \subseteq K_p$, the map 
$\mathcal{S}^\star_{K^pK_p'} \rightarrow \mathcal{S}^\star_{K^pK_p} $ is surjective and the target carries the quotient topology (\cite{DHansen}, theorem 1.1). The continuity of the map of the lemma follows. 
\end{proof}

\begin{lem}\label{lem-good-neighb} Let $V \subseteq \mathcal{FL}_{G,\mu} $ be a very  good affinoid invariant under a compact open subgroup $K_p$. Let $\overline{V}$ be its image in $\mathcal{FL}_{G,\mu}/K_p$. Then $\pi_{HT,K_p}^{-1}(\overline{V}) \subseteq \mathcal{S}^\star_{{K^pK_p}}$ is affinoid.
\end{lem}
\begin{proof} It is part of the definition that  $\pi_{HT}^{-1}(V)$ is the pullback of an affinoid $\pi_{HT}^{-1}(V)_{K'_p} \subset \mathcal{S}^\star_{\overline{K^pK_p'}}$ for some $K'_p \subseteq K_p$, a normal compact open subgroup.   We can further pull back $\pi_{HT}^{-1}(V)_{K'_p}$ to an affinoid $\tilde{U} \subseteq \mathcal{S}^\star_{{K^pK_p'}}$.  The space $\mathcal{S}^\star_{K^pK_p}$ is the categorical quotient of $\mathcal{S}^\star_{K^pK_p'}$ by $K_p/K'_p$, and the image of $\tilde{U}$ in $\mathcal{S}^\star_{K^pK_p}$ is indeed affinoid.
\end{proof}

\begin{proof}[Proof of Theorem \ref{thm-affineness}]  The continuity of the map is lemma \ref{lemma-continuous}.   Let $x \in \mathcal{FL}_{G,\mu}/K_p$  with finite residue field over $\qq_p$. Let $y$ be a lift of $x$ in  $\mathcal{FL}_{G,\mu}$. For $n$ large enough, $x \mathcal{G}_n \hookrightarrow \mathcal{FL}_{G,\mu}$ is a very good affinoid (where $\mathcal{G}_n$ is the subgroup of elements which reduce to $1$ modulo $p^n$ as in lemma \ref{lem-example-neighborhoods}). We form the group $\mathcal{G}_nK_p$. Then $y \mathcal{G}_nK_p = \coprod_{i \in I} x k_i \mathcal{G}_n$ for a finite set $I$ and elements $k_i \in K_p$. Moreover, $y k_i \mathcal{G}_n = y  \mathcal{G}_nk_i$ is a very good affinoid (because the property of being a very good affinoid is preserved under $G(\qq_p)$-action). 
A finite disjoint union  of very good affinoids is a very good affinoid. We can apply lemma \ref{lem-good-neighb}, to $y \mathcal{G}_nK_p$. The image of $y \mathcal{G}_nK_p$ in $\mathcal{FL}_{G,\mu}/K_p$ provides the open neighborhood $U$ of $x$ as in  the theorem.  Finally, since any rational subset of a very good affinoid is again a very good affinoid, a second application of lemma \ref{lem-good-neighb} proves the last point. 
\end{proof}

We also adopt the notation $\pi_{HT,K_p}^{tor} :  \mathcal{S}^{tor}_{K^pK_p,\Sigma} \rightarrow  \mathcal{S}^{\star}_{K^pK_p} \stackrel{\pi_{HT,K_p}}\rightarrow \mathcal{FL}_{G,\mu}/K_p$.

\subsection{Reductions of the torsor $\mathcal{M}^{an}_{dR}$}\label{section-reduction}
The group $G$  of the Shimura datum  is defined over $\qq$. We recall that we have fixed  a finite extension $F$ of $\qq_p$  which splits $G$. We have also fixed a representative of the cocharacter $\mu$ over $F$ and let  $P_\mu$ and $M_\mu$ be the corresponding parabolic and Levi subgroups of $G$.  We will soon assume that $G$ is quasi-split over $\qq_p$. When this is the case we assume that $P_\mu$ contains a  Borel $B$ defined over $\qq_p$.  We  take a reductive model for $G_F$ defined over $\Spec~\ocal_F$. By abuse of notation, we also denote this model by $G$. We also then have models for $P_\mu$ and $M_\mu$ over $\Spec~\ocal_F$. 
On the analytic side, we have the (non-quasi-compact) groups $\mathcal{G}^{an}$, $\mathcal{P}_\mu^{an}$, and $\mathcal{M}_\mu^{an}$, all considered over $\Spa (F, \ocal_F)$, and there is an embedding $G(\qq_p) \hookrightarrow \mathcal{G}^{an}$. 
Because we have fixed an integral model for $G$, we also  have the quasi-compact groups $\mathcal{G} \hookrightarrow \mathcal{G}^{an}$, $\mathcal{P}_\mu\hookrightarrow \mathcal{P}_\mu^{an}$, and $\mathcal{M}_\mu\hookrightarrow \mathcal{M}_\mu^{an}$.

The goal of this section is to use the Hodge-Tate period morphism to produce some finer structure on the torsor $\mathcal{M}^{an}_{dR}$. These result generalize those obtained in \cite{MR3275848}, prop. 4.3.1 for example. These refined structure will allow us to $p$-adically interpolate  automorphic vector bundles.

\subsubsection{Preparations} We start by giving a  detailed description of the torsor pulled back from the map $\pi_{HT}^{tor} : \mathcal{S}^{tor}_{K^p, \Sigma}  \rightarrow \mathcal{FL}_{G,\mu}$.

If $(G,X)$ is Hodge type Shimura datum,  and  $\Sigma$ is a perfect cone decomposition, then $ \mathcal{S}^{tor}_{K^p, \Sigma}$ is a perfectoid space. In the general abelian case, it is only known to be a diamond. 

Let $S \rightarrow \mathcal{S}^{tor}_{K^p, \Sigma}$ be a map from a perfectoid space $S$.  By composing with $\pi_{HT}^{tor}$,  we get a map $S \rightarrow \mathcal{FL}_{G,\mu}$ which can be described as follows. Over $S$, we have a map  $\mathcal{P}^{an}_{HT} \hookrightarrow \mathcal{G}^{c,an}\times S$, from the $\mathcal{P}^{c,an}_{\mu}$-torsor   $\mathcal{P}^{an}_{HT}$ to the trivial $\mathcal{G}^{c,an}$-torsor, which is equivariant for the natural morphism of groups $\mathcal{P}^{c,an}_{\mu} \rightarrow \mathcal{G}^{c,an}$. 

There is an \'etale cover $\tilde{S} \rightarrow S$ such that  the torsor $\mathcal{P}^{an}_{HT} \times_{S} \tilde{S}$ becomes trivial and aquires a section $g_{\tilde{S}} \in \mathcal{G}^{c,an}(\tilde{S})$, unique up to left multiplication by elements of $\mathcal{P}^{c,an}_{\mu}(\tilde{S})$, so that we get a commutative diagram of torsors: 
\begin{eqnarray*}
\xymatrix{ \mathcal{P}^{an}_{HT}\times_S \tilde{S} \ar[r] \ar[d]^{\sim}& \mathcal{G}^{c,an}\times \tilde{S} \ar[d]^{\mathrm{Id}} \\
(\mathcal{P}^{c,an}_{\mu} \times \tilde{S})\cdot g_{\tilde{S}} \ar[r] &\mathcal{G}^{c,an}\times \tilde{S}}
\end{eqnarray*}

Over $\tilde{S} \times_S \tilde{S}$, there is a section $h_{\tilde{S} \times_S \tilde{S}} \in \mathcal{P}^{c,an}_{\mu}(\tilde{S} \times_S \tilde{S})$ such that $h_{\tilde{S} \times_S \tilde{S}}\cdot p_1^\star g_{\tilde{S}} = p_2^\star g_{\tilde{S}}$.  The section $h_{\tilde{S} \times_S \tilde{S}}$ is a $1$-cocycle which describes the original torsor $\mathcal{P}^{an}_{HT}$ . Changing $g_{\tilde{S}}$ by left multiplication by an element of $\mathcal{P}^{c,an}_{\mu}(\tilde{S})$ will change $h_{\tilde{S} \times_S \tilde{S}}$ by a coboundary. 

The image of $g_{\tilde{S}} \in \mathcal{FL}_{G,\mu}(\tilde{S})$ descends to give a point in $\mathcal{FL}_{G,\mu}({S})$:  the morphism $S \rightarrow \mathcal{FL}_{G,\mu}$ we started with.

The group $\mathcal{G}^{an}$ acts on the right on  $\mathcal{FL}_{G,\mu}$.  Concretely this action sends $g_{\tilde{S}}$ to $g_{\tilde{S}}.g$. This action does not affect the construction of the torsor $\mathcal{P}_{HT}^{an}$ which is indeed $\mathcal{G}^{an}$-equivariant. 

\subsubsection{Integral structure} 
Recall that $\mathcal{M}_\mu^c$ is a (quasi-compact) open subgroup of $\mathcal{M}_{\mu}^{c,an}$. The following proposition can be interpreted as  the existence of an integral structure on the torsors $\mathcal{M}_{dR}^{an}$ or $\mathcal{M}_{HT}^{an}$.  

\begin{prop}\label{prop-first-reduction} Let $K_p \subset G(\qq_p) \cap G(\ocal_F)$. The \'etale torsor $\mathcal{M}_{dR}^{an} = \mathcal{M}_{HT}^{an}$  over $\mathcal{S}_{K^pK_p, \Sigma}^{tor}$ has a reduction of structure group to  an \'etale $\mathcal{M}^c_{\mu}$-torsor $\mathcal{M}_{dR} = \mathcal{M}_{HT}$. 
\end{prop}

\begin{rem} In the Siegel case, the torsor $\mathcal{M}_{dR}^{an}$ is  (ignoring the center) the torsor of trivializations of the vector bundle $\omega_{A}$, the conormal sheaf of the    universal semi-abelian scheme over $\mathcal{S}^{tor}_{K^pK_p,\Sigma}$ (well defined up to prime to $p$-isogeny by our choice of level structure). A possible integral structure is obtained by 
declaring  that an invariant differential form is integral if it extends to an invariant differential form on an integral model of the universal semi-abelian scheme. However, the integral structure we consider here is different. Namely, we declare that a differential form is integral if it is in the span of the image of the integral Tate module for the Hodge-Tate period map.  By \cite{MR2919687}, Theorem 2,  section 5.3.2.  we can  explicitly bound the difference between both integral structures. 
\end{rem}

\begin{proof}      
We  work over $\mathcal{S}_{K^p,\Sigma}^{tor}$. Let $S \rightarrow \mathcal{S}_{K^p,\Sigma}^{tor}$ be a map from a perfectoid space. We first explain how to define the torsor $\mathcal{M}_{HT}$ over $S$ (in a functorial way). 

The map $S \rightarrow \mathcal{FL}_{G,\mu}$ is described by an element $g_{\tilde{S}} \in \mathcal{G}^{c,an}(\tilde{S})$ for some \'etale cover $\tilde{S} \rightarrow S$. We are free to change $g_{\tilde{S}}$ by left multiplication by an element of $\mathcal{P}^{c,an}_{\mu}(\tilde{S})$. Thus, up to passing to some further cover of $\tilde{S}$, we may actually assume that $g_{\tilde{S}} \in \mathcal{G}^c(\tilde{S})$ (because $\mathcal{FL}_{G,\mu} = \mathcal{P}^c_\mu \backslash \mathcal{G}^c =  \mathcal{P}^{c,an}_\mu \backslash \mathcal{G}^{c,an}$), and this new element is well defined up to multiplication by an element of $\mathcal{P}^c_{\mu}(\tilde{S})$.  The torsor $\mathcal{P}^{an}_{HT}$ is defined by the cocycle $h_{\tilde{S} \times_S \tilde{S}}=p_2^\star g_{\tilde{S}}\cdot (p_1^\star g_{\tilde{S}})^{-1} \in \mathcal{P}^c_{\mu}( \tilde{S} \times_S \tilde{S}) \subset \mathcal{P}^{c,an}_{\mu}( \tilde{S} \times_S \tilde{S})$.  We therefore have produced a reduction of the torsor $\mathcal{P}^{an}_{HT}$ to  a torsor $\mathcal{P}_{HT}$ under the group $\mathcal{P}_\mu^c$. We can take the pushout under the map $\mathcal{P}^c_\mu \rightarrow \mathcal{M}^c_\mu$ (which amounts to projecting  $h_{\tilde{S} \times_S \tilde{S}} $ in $\mathcal{M}^c_\mu(\tilde{S} \times_S \tilde{S})$) to get the desired torsor $\mathcal{M}_{HT}$. 

We now proceed to descend from $\mathcal{S}^{tor}_{K^p,\Sigma}$ to $\mathcal{S}_{K^pK_p, \Sigma}^{tor}$. We have an \'etale torsor $\mathcal{M}_{HT}^{an} \rightarrow \mathcal{S}_{K^pK_p, \Sigma}^{tor}$, and we have defined an open subset $\mathcal{M}_{HT} \subset \mathcal{M}_{HT}^{an} \times _{\mathcal{S}_{K^pK_p, \Sigma}^{tor}} \mathcal{S}^{tor}_{K^p,\Sigma}$.  We claim that this open descends to an open subset of  $\mathcal{M}_{HT}^{an}$.  The map $\mathcal{M}_{HT}^{an} \times_{\mathcal{S}_{K^pK_p,\Sigma}^{tor}}\mathcal{S}_{K^p,\Sigma}^{tor} \rightarrow \mathcal{M}_{HT}^{an}$ identifies the topological space $\vert \mathcal{M}^{an}_{HT} \vert $ with the quotient of $\vert \mathcal{M}_{HT}^{an} \times_{\mathcal{S}_{K^pK_p,\Sigma}^{tor}} \mathcal{S}_{K^p,\Sigma}^{tor}\vert$ by the action of $K_p$. We have an identification between $K_p$-invariant open subsets of $\mathcal{M}_{HT}^{an} \times_{\mathcal{S}_{K,\Sigma}^{tor}} \mathcal{S}_{K^p,\Sigma}^{tor}$ and open subsets of $\mathcal{M}_{HT}^{an}$. 
The only thing to check is therefore that $\mathcal{M}_{HT}$ is indeed invariant under the action of $K_p$.  We go back to considering a map $S \rightarrow \mathcal{S}_{K^p,\Sigma}^{tor}$, described by an element $g_{\tilde{S}} \in  \mathcal{G}^c(\tilde{S})$. Under right multiplication by $k \in K_p$, we get a new element $g_{\tilde{S}} k \in \mathcal{G}^c(\tilde{S})$ (it is crucial here that $k \in K_p \subseteq G(\ocal_F)$) and the corresponding reduction of the torsor described by the element $p_2^\star g_{\tilde{S}} k . (p_1^\star g_{\tilde{S}}. k)^{-1} = p_2^\star g_{\tilde{S}}. (p_1^\star g_{\tilde{S}})^{-1}$ doesn't depend on $k$. 

We need to prove that the action map $\mathcal{M}^{an}_\mu \times \mathcal{M}^{an}_{HT} \rightarrow \mathcal{M}^{an}_{HT} \times \mathcal{M}^{an}_{HT}$ (which is an isomorphism), induces an isomorphism
 $\mathcal{M}_\mu \times \mathcal{M}_{HT} \rightarrow \mathcal{M}_{HT} \times \mathcal{M}_{HT}$. In other words, we need to prove that the two open subsets $\mathcal{M}_\mu \times \mathcal{M}_{HT}$ and $\mathcal{M}_{HT} \times \mathcal{M}_{HT}$ identify via the action map. This can be checked after pull back to $S^{tor}_{K^p,\Sigma}$ and this is true. 
Finally, the morphism $\mathcal{M}_{HT} \rightarrow \mathcal{S}_{K,\Sigma}^{tor}$ is smooth, and surjective on geometric points. Therefore, there are sections \'etale-locally and $\mathcal{M}_{HT}$ is an \'etale  torsor. 
\end{proof}
\begin{rem} The above argument uses crucially that we know that $\mathcal{M}_{HT}^{an}$ descends from $\mathcal{S}^{tor}_{K^p,\Sigma}$ to $\mathcal{S}^{tor}_{K^pK_p,\Sigma}$, because pro-\'etale descent is not effective in general. 
\end{rem}

\begin{rem} \label{rem-connection-Hodge-abelian}
We briefly explain how the construction   in the abelian case connects with the construction in the Hodge case, in order to illustrate principle \ref{abelianstrategy}.  We consider a diagram of Shimura datum with $(G,X)$ of abelian type and $(G_1,X_1)$ of  Hodge type as in section \ref{section-Abelian-type2}:
\begin{eqnarray*}
\xymatrix{ (B_1,X_{B_1}) \ar[r] \ar[d] & (B, X_B) \ar[d] \\
(G_1,X_1) &  (G,X)}
\end{eqnarray*}
We may assume that the various morphisms between the groups $G_1, B_1, B, G$ over $\qq$ extend to morphisms over $\ocal_F$.  We assume that we have a diagram for suitable compacts $K', K_1, K_2, K_3$ with $K'_p \subseteq B_1(\ocal_F)$ and $K_{1,p} \subseteq  G_1(\ocal_F)$, and   cone decomposition for $G_1$: 

\begin{eqnarray*}
\xymatrix{  {S}^{tor}(B_1, X_{B_1})_{{K'},\Sigma}  \ar[r]^{\pi_1} \ar[d]^{\pi_2} \ar[rd]^{\pi} &  {S}^{tor}(G_1, X_{_1})_{{K_{1}},\Sigma} \ar[d] \\
 {S}(T, X_{T})_{K_{2}} \ar[r] &  {S}(G_1^{ab}, X_{G_1^{ab}})_{K_{3}}}
 \end{eqnarray*}
  
 We have the first formula:  $$\mathcal{M}_{HT}(B_1,X_{B_1}) = \pi_1^\star \mathcal{M}_{HT}(G_1,X_1) \times_{\pi^\star \mathcal{M}_{HT}(G^{ab}_1,X_{G^{ab}_1})} \pi_2^\star \mathcal{M}_{HT}(T,X_{T})$$
 
 Secondly, we have a finite  map: ${S}^{tor,0}(B_1, X_{B_1})_{{K'},\Sigma}   \rightarrow {S}^{tor,0}(G, X)_{{K},\Sigma}$, generically finite \'etale with group $\Delta(K,K')$ (possibly after refining $\Sigma$).
and the second formula $$(\mathcal{M}_{HT}(B_1,X_{B_1})  \times^{\mathcal{M}^c_{\mu_{B_1}}} \mathcal{M}_{\mu_{G}}^c) /\Delta(K,K') = \mathcal{M}_{HT}(G,X_{G})$$ (valid for the restrictions of the torsor to the connected component). 

\end{rem}

For $\kappa \in X^\star(T^c)^{M_\mu,+}$, the sheaf $\mathcal{V}_\kappa$ is modeled on the representation $V_\kappa$ of $\mathcal{M}_\mu^{an}$ defined over $F$. Recall that $V_\kappa$ is defined as the module of sections $f(m) \in \HH^0(\mathcal{M}^{an}_\mu, \oscr_{\mathcal{M}^{an}_\mu})$ such that $f(mb ) = - w_{0, M} \kappa(b) f(m)$ for $b \in \mathcal{B}^{an} \cap \mathcal{M}^{an}_\mu$. Using that $$\mathcal{M}^{an}_\mu/(\mathcal{B}^{an} \cap \mathcal{M}^{an}_\mu) = \mathcal{M}_\mu/(\mathcal{B} \cap \mathcal{M}_\mu)$$ we find that this is also the module of sections $f(m) \in \HH^0(\mathcal{M}_\mu, \oscr_{\mathcal{M}_\mu})$ such that $f(mb ) = - w_{0, M} \kappa(b) f(m)$ for $b \in \mathcal{B} \cap \mathcal{M}_\mu$. 

We can define an  $\ocal_{F}$-submodule $V_\kappa^+ \subset V_\kappa$ by considering sections $f(m) \in \HH^0(\mathcal{M}_\mu, \oscr_{\mathcal{M}_\mu}^+)$ such that $f(mb ) = - w_{0, M_\mu} \kappa(b) f(m)$.  Since $\HH^0(\mathcal{M}_\mu, \oscr_{\mathcal{M}_\mu}^+)$ is open and bounded in $\HH^0(\mathcal{M}_\mu, \oscr_{\mathcal{M}_\mu})$, we deduce that  $V_\kappa^+$ is a lattice in $V_\kappa$, stable under the action of $\mathcal{M}_\mu$.

\begin{coro}\label{coro-existence-integral-struct} Assume that $K_p \subseteq G(\ocal_F)$. For all $\kappa \in X^\star(T^c)^{M_\mu,+}$, the locally free sheaf $\mathcal{V}_\kappa$ over $\mathcal{S}^{tor}_{K, \Sigma}$ has an integral structure $\mathcal{V}_\kappa^+$ in the sense of definition \ref{defi-integral-structure-sheaf}.
\end{coro}

\begin{proof} We consider the map $g : \mathcal{M}_{dR} \rightarrow \mathcal{S}^{tor}_{K, \Sigma}$. We let $\mathcal{V}_\kappa^+$ be the subsheaf of $g_\star \oscr^+_{\mathcal{M}_{dR}}$ of sections $f(m)$ which satisfy $f(mb) = - w_{0,M}\kappa(b) f(m)$.
\end{proof}
\subsubsection{Further reductions of the group structure}\label{sect-further-restriction-group-structure} We assume that $K_p = K_{p,m',b'}$ for $m' \geq b' \in \ZZ_{\geq 0}$ and $m' >0$ (see section \ref{section-compact-open-subgroups}). This is a compact open subgroup of $G(\qq_p)$ with an Iwahori decomposition.  
For any $w \in \WM$, we let  $K_{p,w,M_{\mu}}$ be the projection of $wK_pw^{-1} \cap \mathcal{P}_\mu$ to $\mathcal{M}_{\mu}$.  

We  will describe this group. We can define $N_{p,w, M_{\mu}} = K_{p,w,M_{\mu}} \cap U_{M_\mu}$ and $\overline{N}_{p,w, M_{\mu}} = K_{p,w,M_{\mu}} \cap \overline{U}_{M_\mu}$, where we let $B_{M_\mu}$ be the Borel subgroup of $M_{\mu}$ which is the image of $B$ in $M_{\mu}$, $\overline{B}_{M_\mu}$ the opposite Borel,  and $U_{M_\mu}$ (resp. $\overline{U_{M_\mu}}$) be the unipotent radical of $B_{M_\mu}$ (resp. $\overline{B}_{M_\mu}$). We also recall that $T_{b'} = \mathrm{Ker}(T(\ocal_F) \rightarrow T(\ocal_F/p^{b'})) \cap T(\qq_p) = T \cap K_p$.

\begin{prop}\label{prop-Iwahori-deco-of-weird-group}  For any $w \in \WM$, the group $K_{p,w,M_{\mu}}$ is a subgroup of the Iwahori subgroup of $M_\mu(\ocal_{F})$. Moreover, it admits an Iwahori decomposition.  Namely, the product map:
$$ N_{p,w, M_{\mu}} \times wT_{b'}w^{-1} \times \overline{N}_{p,w, M_{\mu}}  \rightarrow K_{p,w,M_{\mu}}$$ is an isomorphism. 

When $G$ is unramified, $M_\mu$ is defined over $\qq_p$,  $w$ is $\mathrm{Gal}(F/\qq_p)$-invariant,  and $K_p = K_{p,1,0}$ is the Iwahori subgroup of $G(\ZZ_p)$,  then $K_{p,w,M_{\mu}}$ is the Iwahori subgroup of $M_\mu(\ZZ_p)$. 
\end{prop} 

\begin{proof}
Let  $U$ and $\overline{U}$ be respectively the unipotent radicals of $B$ and $\overline{B}$. We have the Iwahori decomposition  $K_p = \overline{N}_p \times T_{b'}(\ZZ_p) \times N_p$ where $N_p = K_p \cap U$ is $U(\qq_p)\cap U(\ocal_F)$  and $\overline{N}_p = K_p \cap \overline{U} \subseteq \overline{U}(\qq_p)$ is the subgroup of elements reducing to $1$ modulo $p^{m'}$ 
It is useful to give a more precise version of this decomposition. We let $\Phi$ be the set of roots (defined over $F$).  We have $\Phi = \Phi^+_M \coprod \Phi^-_M \coprod \Phi^{+,M} \coprod \Phi^{-,M}$, where $\Phi_M = \Phi^+_M \coprod \Phi^-_M$ is the set of roots in $M$. 
We also let $\Phi_0 = \Phi/\mathrm{Gal}(F/\qq_p)$ and have $\Phi_0 =\Phi_0^+ \coprod \Phi_0^-$ (because $G$ is quasi-split). 

For all $\alpha_0 \in \Phi_0$, we let $U_{\alpha_0} \hookrightarrow G$ be the corresponding unipotent group. For all $\alpha \in \Phi$ we also denote by $U_\alpha \hookrightarrow G_{\ocal_{F}}$ the one parameter subgroup. We have $U_{\alpha_0} \times_{\Spec~\qq_p} \Spec~F = \prod_{\alpha \in \Phi, \alpha \mapsto \alpha_0} U_{\alpha} \times_{\Spec~\ocal_F} \Spec~F$. 
We consider the product map (in any order of the factors): 
$$ \prod_{\alpha_0} U_{\alpha_0}(\qq_p) \times T(\qq_p) \rightarrow G(\qq_p).$$
This maps induces a bijection between $K_p$ and  the set of elements $((n_{\alpha_0})_{\alpha_0 \in \Phi_0}, t)$ which satisfy:  $n_{\alpha_0}$, $t \in G(\ocal_F)$,  $n_{\alpha_0} = 1~\mod p^{m'}$ if $\alpha_0 \in \Phi_0^{-}$, and $t = 1~\mod p^{b'}$.
We get that $w K_p w^{-1} \cap P_\mu $ identifies (via the product map) with the set of elements $((wn_{\alpha_0}w^{-1})_{\alpha_0 \in \Phi_0 } \in wU_{\alpha_0}(\qq_p)w^{-1}, wtw^{-1} \in w T(\qq_p) w^{-1} )$ such that:
\begin{itemize}
\item $n_{\alpha_0},~t \in G(\ocal_F)$,
\item $n_{\alpha_0} = 1~\mod p^{m'}$ if $\alpha_0 \in \Phi_0^{-}$,
\item $t = 1~\mod p^{b'}$,
\item $n_{\alpha_0} \in U'_{\alpha_0}=  \mathrm{Ker} \big(U_{\alpha_0}(\qq_p) \cap G(\ocal_F) \rightarrow \prod_{ \alpha \in w^{-1}\Phi^{-,M}, \alpha \mapsto \alpha_0} U_{\alpha}(\ocal_{F})\big)$.
\end{itemize}

We let $U''_{\alpha_0} = \mathrm{Im} \big( U'_{\alpha_0} \rightarrow \prod_{ \alpha \in w^{-1}\Phi_M, \alpha \mapsto \alpha_0} U_\alpha(\ocal_{F})\big)$. 
We deduce that $K_{p,w,M_{\mu}}$  is in bijection (via the product map) with the elements $((wn_{\alpha_0}w^{-1})_{\alpha_0 \in \Phi_0} \in wU''_{\alpha_0}w^{-1}, wtw^{-1} \in w T_{b'} w^{-1} )$ such that 
 $n_{\alpha_0} = 1~\mod p^{m'}$ if $\alpha_0 \in \Phi_0^{-}$.  The fact that any element of $K_{p,w,M_{\mu}}$ can be written in this way follows from the previous discussion. The injectivity of the product map is a general fact.

We now observe that $w \in \WM$ and therefore $\Phi_M^+ \subset w(\Phi^+)$ and $\Phi_M^- \subset w(\Phi^-)$.  We deduce that $w^{-1} \Phi_M \cap \Phi^- = w^{-1} \Phi_M^-$ and $w^{-1} \Phi_M \cap \Phi^+ = w^{-1} \Phi_M^+$.

It follows that if $\alpha_0 \in \Phi_0^+$, $w U''_{\alpha_0} w^{-1} \subset U_{M_\mu}(\ocal_{F})$ and if $\alpha_0 \in \Phi_0^-$, $w U''_{\alpha_0} w^{-1} \subset \overline{U}_{M_\mu}(\ocal_{F})$. Therefore, we deduce that $K_{p,w,M_\mu} =  N_{p,w, M_{\mu}} \times wT_{b'}w^{-1} \times \overline{N}_{p,w, M_{\mu}}$.  By the condition that  $n_{\alpha_0} = 1~\mod p^{m'}$ if $\alpha_0 \in   \Phi_0^{-}$, we deduce that $K_{p, w, M_{\mu}}$ is a subgroup of the Iwahori of $M_\mu(\ocal_{F})$.

Finally, if $G$ is unramified, and $M_\mu$ is defined over $\qq_p$,  the partition $\Phi = \Phi^+_M \coprod \Phi^-_M \coprod \Phi^{+,M} \coprod \Phi^{-,M}$ descends to a partition $\Phi_0 = \Phi^+_{0,M} \coprod \Phi^-_{0,M} \coprod \Phi_0^{+,M} \coprod \Phi_0^{-,M}$.  If $w$ is rational, then it acts on $\Phi_0$. Assume that $K_p = K_{p,1,0}$.  The description of $K_{p,w,M_{\mu}}$ simplifies  and we find that $U''_{\alpha_0} = \{1\}$ if $\alpha_0 \notin w^{-1} \Phi_{0,M}$, $U''_{\alpha_0} = U_{\alpha_0}(\ZZ_p)$ if $\alpha_0 \in w^{-1} \Phi_{0,M}$. It follows that $K_{p,w,M_\mu}$ is the Iwahori of $M_\mu(\ZZ_p)$.

\end{proof}

\begin{ex} The group $K_{p, w,  M_\mu}$ may be a little strange. Let us consider the following example. We assume that $G_{\qq_p} $ is $\mathrm{Res}_{\qq_{p^2}/\qq_p} \mathrm{GL}_2$, with standard diagonal torus $T_{\qq_p}$ and upper triangular borel $B_{\qq_p}$. We let $K_p = K_{p,1,0}$.   We identify $G_{\qq_{p^2}} = \mathrm{GL}_2 \times \mathrm{GL}_2$. We assume that $\mu$ is defined over $\qq_{p^2}$ and is given by the cocharacter $t \mapsto \mathrm{diag}(t,1) \times \mathrm{diag}(1,1)$ of $T_{\qq_{p^2}}$. We deduce that $P_{\mu}$ is $B_{\qq_{p^2}} \times \mathrm{GL}_2$, and that $M_\mu$ is $T_{\qq_{p^2}} \times GL_2$. We finally observe that $K_{p} \cap P_\mu$ is  $B(\ZZ_p)$ and therefore $K_{p,1, \mu}$ is  the image of $B(\Z_p)$ in $T(\qq_{p^2}) \times GL_2(\qq_{p^2})$. \end{ex} 

For all $m,n \in \qq_{\geq 0}$, we let $\mathcal{G}^1_{m,n}$ be the subgroup of $\mathcal{G}$ of elements which reduce to $\mathcal{U}$ modulo $p^{m + \epsilon}$ for all $\epsilon >0$ and to $\overline{\mathcal{U}}$ modulo $p^n$ (see section \ref{section-orbitsofcells}). We let $\mathcal{M}^1_{\mu,m,n} \subseteq \mathcal{M}_{\mu}$ be the group of elements which reduce to $\mathcal{U}_{\mathcal{M}_\mu}$ modulo $p^{m + \epsilon}$ for all $\epsilon >0$ and to $\overline{\mathcal{U}}_{ \mathcal{M}_\mu}$ modulo $p^n$.

\begin{lem}Let $w \in \WM$ and let $K_p = K_{p,m',b'}$ with $m' \in \ZZ_{>0}$ and $b' \in \ZZ_{\geq 0}$. Let $m, n \geq 0$ and assume that  $0 \leq m-n \leq m'-1$.  Then 
$K_{p,w,M_{\mu}}$ normalizes $\mathcal{M}^1_{\mu,m,n}$.
\end{lem}
\begin{proof}  Observe that $K_p \subseteq \mathcal{G}_{m-n,0}$. By lemma \ref{lem-normalsubgroups}, $K_p$ normalizes $\mathcal{G}^1_{m,n}$. Moreover, $\mathrm{Im} (w \mathcal{G}^1_{m,n} w^{-1} \cap \mathcal{P}_\mu  \rightarrow \mathcal{M}_\mu) = \mathcal{M}^1_{\mu, m,n}$. 

\end{proof}

It follows from this lemma that $K_{p,w,M_{\mu}} \mathcal{M}^1_{\mu,m,n}$ is a subgroup of $\mathcal{M}_\mu$. 

\begin{prop}\label{prop-reduction-2} Let $w \in \WM$ and let $K_p = K_{p,m',b'}$ with $m' \in \ZZ_{>0}$ and $m' \geq b' $. Let $m, n \geq 0$ and assume that  $0 \leq m-n \leq m'-1$. 
 Over $(\pi^{tor}_{HT,K_p})^{-1} (
  ]C_{w,k}[_{m,n} K_p)  \subseteq \mathcal{S}_{K^pK_p,\Sigma}^{tor}$, the torsor $\mathcal{M}_{HT}$ has a  reduction to an \'etale torsor $\mathcal{M}_{HT,m,n,K_p}$ under the group $K^c_{p,w,M_{\mu}} \mathcal{M}^{1,c}_{\mu,m,n}$.
  
\end{prop}

\begin{proof} The proof is very similar to the proof of proposition \ref{prop-first-reduction}.  We observe that  $ ]C_{w,k}[_{m,n} K_p =  \mathcal{P}_\mu^{an} \backslash \mathcal{P}_\mu^{an} w \mathcal{G}^1_{m,n} K_p$. Since $K_p$ normalizes $\mathcal{G}^1_{m,n}$, $\mathcal{G}^1_{m,n} K_p = K_p \mathcal{G}^1_{m,n} $ is a group. It follows that 
$$ ]C_{w,k}[_{m,n} K_p =  (\mathcal{P}_\mu^{an} \cap w  K_p \mathcal{G}^1_{m,n} w^{-1}) \backslash w  K_p \mathcal{G}^1_{m,n} w^{-1}.w$$ 
Observe that we have the following $K_p$-equivariant $K^c_{p,w,M_{\mu}} \mathcal{M}^{1,c}_{\mu,m,n}$-torsor: 
$$(\mathcal{U}_\mu^{an,c} \cap w  K^c_p \mathcal{G}^{1,c}_{m,n} w^{-1}) \backslash w  K^c_p \mathcal{G}^{1,c}_{m,n} w^{-1}.w \rightarrow (\mathcal{P}_\mu^{an} \cap w  K_p \mathcal{G}^1_{m,n} w^{-1}) \backslash w  K_p \mathcal{G}^1_{m,n} w^{-1}.w$$ and we proceed to pull back and descend this torsor. 

 We first construct the torsor $\mathcal{M}_{HT,m,n,K_p}$ as an open subset of $$\mathcal{M}_{HT} \times_{S^{tor}_{K,\Sigma}}( \pi^{tor}_{HT})^{-1} (]C_{w,k}[_{m,n} K_p),$$ and prove it is $K_p$-invariant to descend it to an open subset of $$\mathcal{M}_{HT} \times_{S^{tor}_{K,\Sigma}} (\pi^{tor}_{HT,K_p})^{-1} (]C_{w,k}[_{m,n} K_p).$$ 

Let $S \rightarrow (\pi^{tor}_{HT})^{-1} (]C_{w,k}[_{m,n} K_p)$ be a map from a perfectoid space $S$. The torsor $\mathcal{P}_{HT}$ is described as follows: there is a cover $\tilde{S} \rightarrow S$ and an element $g_{\tilde{S}} \in \mathcal{G}(\tilde{S})$ such that $h_{\tilde{S}\times_S \tilde{S}} = p_2^\star g_{\tilde{S}}. (p_1^\star g_{\tilde{S}})^{-1}  \in \mathcal{P}_{\mu}( \tilde{S} \times_S \tilde{S})$ is a $1$-cocycle describing the torsor. By assumption, we may assume $g_{\tilde{S}} \in w \mathcal{G}^1_{m,n} K_p$, so that $h_{\tilde{S}\times_S \tilde{S}} \in \mathcal{P}_{\mu}( \tilde{S} \times_S \tilde{S}) \cap w \mathcal{G}^1_{m,n} K_p w^{-1}$ and its image in $\mathcal{M}_\mu$ describes a reduction $\mathcal{M}_{HT,m,n,K_p}$  of $\mathcal{M}_{HT} \times_{S^{tor}_{K,\Sigma}}( \pi^{tor}_{HT})^{-1} (]C_{w,k}[_{m,n} K_p)$ to a $K_{p,w,M_\mu} \mathcal{M}^1_{\mu,m,n}$-torsor. One checks easily that this torsor is $K_p$-invariant and therefore descends to a torsor over $(\pi^{tor}_{HT,K_p})^{-1} (]C_{w,k}[_{m,n} K_p)$. 
\end{proof}

\begin{rem}\label{rem-on-the-abelian-case} The connection between the abelian case and Hodge case can be described similarly  as in remark \ref{rem-connection-Hodge-abelian}. We may use section \ref{subsecchangegroup} to make sure that the subset  $]C_{w,k}[_{m,n} K_p$ of the flag variety doesn't change when we consider the different groups $G_1, B_1, G$. 
\end{rem}

\begin{prop} Let $w \in \WM$ and let $K_p = K_{p,m',b'}$ with $m' \in \ZZ_{>0}$ and $m' \geq b' $.
Let $K_p' = K_{p,m'', b''}$ with $m'' \geq m'$ and $m'' \geq b'' \geq b'$.  Let $m, n \geq 0$ and assume that $0 \leq m-n \leq m'-1$. Let $r,s \geq 0$ with $r \geq m$, $s \geq m$ and $0 \leq r-s \leq m''-1$. 
There is a commutative diagram:
\begin{eqnarray*}
\xymatrix{ \mathcal{M}_{HT,r,s,K'_p} \ar[r] \ar[d] & \mathcal{M}_{HT,m,n, K_p} \ar[d] \\
(\pi^{tor}_{HT,K'_p})^{-1} (
  ]C_{w,k}[_{r,s} K'_p) \ar[r] & (\pi^{tor}_{HT,K_p})^{-1} (
  ]C_{w,k}[_{m,n} K_p)  }
\end{eqnarray*}
The top horizontal map is  equivariant  for the map $K'_{p,w,M_{\mu}} \mathcal{M}^1_{\mu,r,s} \rightarrow K_{p,w,M_{\mu}} \mathcal{M}^1_{\mu,m,n}$.
\end{prop}
\begin{proof} This follows from the construction of the torsors.
\end{proof}

We have a map $\mathcal{M}^1_{\mu,n,n} \rightarrow \mathcal{M}_{\mu,n}$ where $\mathcal{M}_{\mu,n}$ is the  sub-group of $\mathcal{M}_\mu$ of elements reducing to $1$ modulo $p^n$. Let $K_{p} = K_{p,m',b'}$ with $m' >0$, $m' \geq b'$ and let $m,n \geq 0$ be such that $0 \leq m-n \leq m'-1$. Over  $(\pi^{tor}_{HT,K_p})^{-1} (
  ]C_{w,k}[_{n,n} K_p)$, we define the pushout:  $$\mathcal{M}_{HT,n,K_p} = \mathcal{M}_{HT,n,n,K_p} \times^{K^c_{p,w,M_\mu}\mathcal{M}^{1,c}_{\mu,n,n}} K^c_{p,w,M_\mu}\mathcal{M}^{c}_{\mu,n}.$$ 

It is sometimes more convenient to work with this torsor because the group $\mathcal{M}^c_{\mu,n}$ is affinoid. 

\begin{prop}\label{prop-torsor-become-trivial} Assume that $(G,X)$ is a Hodge-type Shimura datum and $\Sigma$ a perfect cone decomposition. Let $K_{p} = K_{p,m',b}$. For any affinoid open   $\Spa (R, R^+) \rightarrow (\pi^{tor}_{HT,K_p})^{-1} (
  ]C_{w,k}[_{n,n} K_p)$  which we assume to be  pregood (see definition \ref{defi-of-good}), there exists $K'_p \subseteq K_p$ such  that over $\Spa(R,R^+) \times_{\mathcal{S}_{K^pK_p, \Sigma}^{tor}} \mathcal{S}_{K^pK_p', \Sigma}^{tor}$ the torsor $\mathcal{M}_{HT, n, K_p}$ is trivial.
  \end{prop}
  
  \begin{rem} It will be important for certain  vanishing theorems 
  that we are able to prove the triviality of the torsor after a finite flat cover  for  pregood  affinoid opens.
  \end{rem}
  
  \begin{proof}
   Let us consider a decreasing sequence of compacts $K_{p,k}$ with $K_{p,0} = K_p$ and $\cap_k K_{p,k} = \{1\}$. Let $\Spa (R_k, R_{k}^+) = \Spa(R,R^+) \times_{\mathcal{S}_{K^pK_p, \Sigma}^{tor}} \mathcal{S}_{K^pK_{p,k}, \Sigma}^{tor}$. Let $\Spa (R_\infty, R_{\infty}^+) = \lim \Spa (R_k, R_{k}^+)$ be an affinoid open of  $\mathcal{S}_{K^p, \Sigma}^{tor}$.  We first observe that the torsor $\mathcal{M}_{HT}^{an}\vert_{\Spa (R, R^+)}$ is a Stein space, which can be written as an increasing union of quasi-compact  affinoid subsets: $$\mathcal{M}_{HT}^{an}\vert_{\Spa (R, R^+)} = \cup_{i \geq 0} (\mathcal{M}_{HT}^{an})_i.$$
  Over $\Spa (R_\infty, R_\infty^+)$ we  observe that the the torsor $\mathcal{M}_{HT,n,K_p}$ is trivial. Indeed, this torsor is pulled back from the following torsor (over the flag variety): 
  $$  (U_{\mathcal{P}_\mu} \cap wK_p \mathcal{G}_n w^{-1}) \backslash wK_p \mathcal{G}_n w^{-1} \rightarrow  ( {\mathcal{P}_\mu} \cap wK_p \mathcal{G}_n w^{-1})\backslash wK_p \mathcal{G}_n w^{-1} $$ which is trivial because of the Iwahori decomposition of the group $ wK_p \mathcal{G}_n w^{-1}$. 
  It follows that $\mathcal{M}_{HT,n,K_p} \times \Spa (R_\infty, R_\infty^+)$ is affinoid, and is a rational open subset of $(\mathcal{M}_{HT}^{an})_i \times \Spa (R_\infty, R_\infty^+)$ for $i$ large enough. 
 
 We deduce that $\mathcal{M}_{HT,n,K_p} \times \Spa (R_k, R_k^+)$ is a rational subset of $(\mathcal{M}_{HT}^{an})_i \times \Spa (R_k, R_k^+)$ for $k$ large enough (by approximating the equations defining $\mathcal{M}_{HT,n,K_p} \times \Spa (R_\infty, R_\infty^+)$ ). Therefore, $\mathcal{M}_{HT,n,K_p} \times \Spa (R_k, R_k^+) = \Spa (T, T^+)$ where $(T,T^+)$ is an $(R_k,R_k^+)$ algebra topologically of finite type. Moreover, there is a section $T^+\hat{ \otimes}_{R_k^+} R_\infty^+ \rightarrow R_\infty^+$.  We now prove that this section can be approximated to a section $T^+\hat{ \otimes}_{R_k^+} R_{k'}^+ \rightarrow R_{k'}^+$ for $k'$ large enough. By \cite{MR345966}, theorem 7,  there is a finite type $R_{k}^+$-algebra $A$, such that $A[1/p]$ is smooth over $R_{k}^+[1/p]$, and whose $p$-adic completion is isomorphic to  $T^+$. By  \cite{MR345966}, theorem 2, there exists integers $n_0, r \geq 0$ with the property that  for any $n \geq n_0$, for any map of $R_{k}^+$-algebras $f : A \rightarrow R_{k'}^+/p^{n}$ there is a map $ \tilde{f} : A \rightarrow R_{k'}^+$ with the property that $f \mod p^{n-r} = \tilde{f} \mod p^{n-r}$ (the ideal  denoted $H_B$ of the reference contains $p^h$ for a large enough integer $h$, because $A[1/p]$ is smooth over  $R_{k}^+[1/p]$). Let $n = n_0 + r$.  We consider the section $ s: A \rightarrow R_\infty^+$.  Its reduction mod $p^n$ factors through  $s ~\mod p^n : A \rightarrow R_{k'}^+/p^n$ for $k'$ large enough and therefore we find   a lift to a section $\tilde{s} : A \rightarrow R_{k'}^+$. 
  \end{proof}
  \begin{rem} We ask the following question: let $\mathcal{S}$ be an affinoid adic space and let $\mathcal{H}$ be an affinoid group over $\mathcal{S}$. Let $\mathcal{T} \rightarrow \mathcal{S}$ be a $\mathcal{H}$-torsor for the \'etale topology. Is $\mathcal{T}$  affinoid over $\mathcal{S}$? It follows from proposition \ref{prop-torsor-become-trivial} that the torsor $\mathcal{M}_{HT,n,K_p}$ is affinoid  over any pregood affinoid open. 
  \end{rem}

\subsubsection{Maps between torsors}  We consider for the moment the following abstract situation. We assume that we have two analytic groups $\mathcal{K} \hookrightarrow \mathcal{G}$. For $i \in \{1,2\}$, we let $\mathcal{K}_{i}$ be $\mathcal{K}$-torsors over an adic space $\mathcal{X}$. We let $\mathcal{G}_i$ be their push-out via the map $\mathcal{K} \hookrightarrow \mathcal{G}$. 
Let  $\alpha : \mathcal{G}_1 \rightarrow \mathcal{G}_2$ be a map of $\mathcal{G}$-torsors over $\mathcal{X}$. Over any cover $\mathcal{U} \rightarrow \mathcal{X}$ which trivializes both $\mathcal{K}_1$ and $\mathcal{K}_2$, we can represent the map $\alpha$ by an element $g \in \mathcal{G}(\mathcal{U})$, well defined up to right and left multiplication by $\mathcal{K}(\mathcal{U})$.  We shall say that the map $\alpha$ is locally represented (over $\mathcal{U}$) by $\mathcal{K}g\mathcal{K}$. When $\mathcal{K}$ is clear from the context, we also say for simplicity that the map is locally represented by $g$.

Let $(G,X)$ be an abelian Shimura datum. Let $t \in G(\qq_p)$. Let $K^pK_p\subset G(\mathbb{A}_f)$ be a compact open subgroup. For suitable choices of polyhedral cone decomposition, we have a correspondence:  

\begin{eqnarray*}
\xymatrix{ & S^{tor}_{K^p (K_p  \cap t K_p t^{-1}), \Sigma''} \ar[rd]^{p_1} \ar[ld]_{p_2} & \\
S^{tor}_{K^pK_p, \Sigma} & & S^{tor}_{K^pK_p, \Sigma'}}
\end{eqnarray*}

We get an associated map of pro-Kummer-\'etale right torsors: 
\begin{eqnarray*}
\xymatrix{ p_1^\star \mathcal{G}^{an}_{pet} \ar[r] &  p_2^\star \mathcal{G}^{an}_{pet} \\
p_1^\star \mathcal{G}_{pet} \ar[u] &  p_2^\star \mathcal{G}_{pet} \ar[u]}
\end{eqnarray*}
which by definition is locally represented by $t$.  We also deduce a map of \'etale right torsors:
 \begin{eqnarray*}
\xymatrix{ p_1^\star \mathcal{M}^{an}_{HT} \ar[r] &  p_2^\star \mathcal{M}^{an}_{HT} \\
p_1^\star \mathcal{M}_{HT} \ar[u] &  p_2^\star \mathcal{M}_{HT} \ar[u]}
\end{eqnarray*}
  Let $w \in \WM$ and let $K_p = K_{p,m',b'}$ with $m'>0$ and $b' \geq 0$. Let $m, n \geq 0$ and assume that  $0 \leq m-n \leq m'-1$. 
Over $p_2^{-1} \big( (\pi^{tor}_{HT, K_p})^{-1}( ]C_{w,k}[_{m,n} K_p)\big) \cap p_1^{-1} \big( (\pi^{tor}_{HT, K_p})^{-1}( ]C_{w,k}[_{m,n} K_p) \big)$,
 we have a map of  \'etale right torsors: 
 \begin{eqnarray*}
\xymatrix{ p_1^\star \mathcal{M}^{an}_{HT} \ar[r] &  p_2^\star \mathcal{M}^{an}_{HT} \\
p_1^\star \mathcal{M}_{HT} \ar[u] &  p_2^\star \mathcal{M}_{HT} \ar[u] \\
p_1^\star \mathcal{M}_{HT,m,n,K_p} \ar[u] &  p_2^\star \mathcal{M}_{HT,m,n,K_p}\ar[u]}
\end{eqnarray*}
 
\begin{prop}\label{prop-representing-torsor-map} Let $w \in \WM$. Let $t \in T(\qq_p)$. The map 
$ p_1^\star \mathcal{M}^{an}_{HT} \rightarrow p_2^\star \mathcal{M}^{an}_{HT}$ restricted to  $$p_2^{-1} \big( (\pi^{tor}_{HT, K_p})^{-1}( ]C_{w,k}[_{m,n} K_p)\big) \cap p_1^{-1} \big( (\pi^{tor}_{HT, K_p})^{-1}( ]C_{w,k}[_{m,n} K_p) \big)$$ is locally represented by $$K^c_{p,w, M_\mu} \mathcal{M}^{1,c}_{\mu,m,n} wtw^{-1}K^c_{p,w, M_\mu} \mathcal{M}^{1,c}_{\mu,m,n}.$$
\end{prop}

\begin{proof} By definition of the Hecke correspondence,  the map  $ p_1^\star \mathcal{G}^{an}_{pet} \rightarrow p_2^\star \mathcal{G}^{an}_{pet}$ is locally represented by $ t$. This means that 
locally for the pro-Kummer-\'etale topology we have sections $x_2 \in p_2^\star\mathcal{G}_{pet}$ and $x_1 \in  p_1^\star\mathcal{G}_{pet}$ and there is an isomorphism 
\begin{eqnarray*}
 \xymatrix{p_1^\star \mathcal{G}^{an}_{pet} \ar[r] & p_2^\star \mathcal{G}^{an}_{pet} \\
x_1 G(\qq_p)   \ar[u] \ar[r]^{t}  & x_2G(\qq_p)   \ar[u] }
 \end{eqnarray*}
 where the bottom map  is $ x_1 g  \mapsto  x_2 g = x_1 t g $.  We now get by pushforward to $\mathcal{G}^{an} $  a diagram: 
 
 \begin{eqnarray*}
 \xymatrix{p_1^\star \mathcal{G}^{an}_{pet}\times_{G(\qq_p)} \mathcal{G}^{an} \ar[r] & p_2^\star \mathcal{G}^{an}_{pet} \times_{G(\qq_p)} \mathcal{G}^{an}\\
 x_1  \mathcal{G}^{an} \ar[u] \ar[r]^{t}  &  x_2 \mathcal{G}^{an}  \ar[u] }
 \end{eqnarray*}
 
The torsors  $p_1^\star \mathcal{G}^{an}_{pet}$ and $p_2^\star \mathcal{G}^{an}_{pet}$ arise by pushforward from torsors $p_1^\star \mathcal{P}^{an}_{HT}$ and $p_2^\star \mathcal{P}^{an}_{HT}$  and we have a diagram: 
 
 \begin{eqnarray*}
 \xymatrix{p_1^\star \mathcal{G}^{an}_{pet} \ar[r] & p_2^\star \mathcal{G}^{an}_{pet} \\
p_1^\star \mathcal{P}_{HT}^{an}    \ar[u] \ar[r]  & p_2^\star \mathcal{P}_{HT}^{an}   \ar[u] }
 \end{eqnarray*}

 We  pick $w \in \WM$ and we work locally over $$p_2^{-1} \big( (\pi^{tor}_{HT, K_p})^{-1}( ]C_{w,k}[_{m,n} K_p)\big) \cap p_1^{-1} \big( (\pi^{tor}_{HT, K_p})^{-1}( ]C_{w,k}[_{m,n} K_p) \big) \subseteq S^{tor}_{K^p (K_p  \cap t K_p t^{-1}), \Sigma''}.$$
 
 Concretely, this means that we can find trivializations $x'_1$ and $x'_2$ of $p_1^\star \mathcal{P}_{HT}^{an}$ and $p_2^\star \mathcal{P}_{HT}^{an}$ of the form $x_2 = x'_2 w h_2$ and $x_1 = x'_1 wh_1$ for $h_i \in \mathcal{G}^{1,c}_{m,n} K^c_p$.
 
 We deduce that 
 $$ x'_2 = x'_1  w h_1  t  h_2^{-1} w^{-1}$$
  This forces  $w h_1  t  h_2^{-1} w^{-1} \in \mathcal{P}^{an,c}_\mu$.
 By pushforward to $\mathcal{M}^{an}_{HT}$ we get:

\begin{eqnarray*}
 \xymatrix{p_1^\star \mathcal{M}^{an}_{HT} \ar[r] & p_2^\star \mathcal{M}^{an}_{HT} \\
 x'_1 \mathcal{M}_\mu^{an}  \ar[u] \ar[r]  & x'_2\mathcal{M}_\mu^{an}   \ar[u] }
 \end{eqnarray*}

where $x'_2$ and $x'_1$ are now viewed as  sections of $p_2^\star\mathcal{M}_{HT,m,n,K_p}$ and $p_1^\star \mathcal{M}_{HT,m,n,K_p}$ and the bottom map is given by  $$x'_2 m =  x'_1 \overline{(w h_1  t  h_2^{-1} w^{-1})}  m$$ where $\overline{(w h_1  t  h_2^{-1} w^{-1})}$ is the image of $(w h_1  t  h_2^{-1} w^{-1})$ via the map $\mathcal{P}_\mu^{an} \rightarrow \mathcal{M}_{\mu}^{an}$.  By  lemma \ref{lem-double-cosets}, $\overline{(w h_1  t  h_2^{-1} w^{-1})} \in K^c_{p,w, M_\mu} \mathcal{M}^{1,c}_{\mu, m,n} wt w^{-1}K^c_{p,w, M_\mu} \mathcal{M}^{1,c}_{\mu, m,n}.$

\end{proof}

\begin{lem}\label{lem-double-cosets} For any $t \in T(\qq_p)$, we have that $$\mathrm{Im} (wK_{p} \mathcal{G}^1_{m,n} t  K_p \mathcal{G}^1_{m,n} w^{-1} \cap \mathcal{P}_\mu^{an} \rightarrow \mathcal{M}_\mu^{an})  = K_{p,w, M_\mu} \mathcal{M}^1_{\mu,m,n} wt w^{-1}K_{p,w, M_\mu} \mathcal{M}^1_{\mu, m,n}.$$
\end{lem}

\begin{proof}  We will prove that $(wK_{p} \mathcal{G}^1_{m,n} t K_p \mathcal{G}^1_{m,n} w^{-1} )\cap \mathcal{P}_\mu^{an} = $ $$((wK_{p} \mathcal{G}^1_{m,n}w^{-1}) \cap \mathcal{P}_\mu^{an}) w t w^{-1} ((wK_{p} \mathcal{G}^1_{m,n}w^{-1}) \cap \mathcal{P}_\mu^{an}).$$
Any element in $k \in wK_{p} \mathcal{G}^1_{m,n}w^{-1}$ writes uniquely $\prod_{\alpha \in \Phi} k_\alpha$ (for any fixed ordering of the roots). Let  $$k wtw^{-1} k' \in (wK_{p} \mathcal{G}^1_{m,n}w^{-1} w t w^{-1} w K_p \mathcal{G}^1_{m,n} w^{-1} )\cap \mathcal{P}_\mu^{an} $$ with $k = \prod_{\alpha \in \Phi^+ \cup \Phi^-_M} k_\alpha  \prod_{\alpha \in \Phi^{-,M}} k_\alpha $ and $k' =  \prod_{\alpha \in \Phi^{-,M}} k'_\alpha \prod_{\alpha \in \Phi^+ \cup \Phi^-_M} k'_\alpha$. A necessary and sufficient condition that $k wtw^{-1} k' \in \mathcal{P}_\mu^{an}$ is that $k'' wtw^{-1} k''' \in \mathcal{P}_\mu^{an}$ where 
$k'' = \prod_{\alpha \in \Phi^-_M} k_\alpha$ and $k''' = \prod_{\alpha \in \Phi^-_M} k'''_\alpha$. But $k'' wtw^{-1} k''' \in \mathcal{U}_{\overline{P}_\mu^{an}} \rtimes \mathcal{T}^{an}$ and necessarily, $k'' = k''' = 1$. 
\end{proof}

\section{Overconvergent cohomologies and the spectral sequence}\label{section-overconvergent-coho}

Our goal in this section is to introduce a spectral sequence which computes classical finite slope cohomology in terms of the finite slope parts of certain overconvergent cohomologies indexed by $w\in\WM$.  Moreover we will prove a classicality theorem comparing the small slope part of classical cohomology in regular weight with the small slope part of a single overconvergent cohomology for a $w$ determined by the weight. We will also prove a vanishing theorem for the classical cohomology in all weights, including non regular ones.

\subsection{The finite slope part} We briefly recall the spectral theory of compact operators over a non-archimedean field. 
\subsubsection{Slope decomposition}\label{sect-slope-decompo} Let  $F$ be a non archimedean field extension of $\qq_p$. The valuation $v$ on $F$ is normalized by $v(p) =1$.  A polynomial $Q   \in F[X]$ has a Newton polygon. The slopes of the newton polygon are the inverse of the valuations of the roots of $Q$ (in an algebraic closure of $F$). We let $Q^\star = X^{\deg Q}Q(1/X)$ be the reciprocal polynomial. Let $h \in \qq$. A polynomial $Q$ is said to have slope $\leq h$ if the slopes of its Newton polygon are $\leq h$ (equivalently,  the roots of $Q^\star(X)$ have valuation less or equal to $h$).  Let $M$ be a vector space over $F$ and let $T$ be an endomorphism of the vector space $M$. Let $h \in \qq$. Following  \cite{ashstevensslopes}, def. 4.6.3,  an $h$-slope decomposition of $M$ with respect to $T$ is a direct sum decomposition  of $F$-vector spaces $M = M^{\leq h} \oplus M^{>h}$ such that:
\begin{enumerate}
\item  $M^{\leq h}$  and $M^{>h}$ are stable under the action of $T$.
\item  $M^{\leq h} $  is finite  dimensional over $F$.
\item  All the eigenvalues  (in an algebraic closure of $F$) of $T$ acting on $M^{\leq h}$ are of valuation less or equal to $h$. 
\item  For any polynomial $Q$ with slope $\leq h$, the restriction of $Q^\ast(T)$  to $M^{>h}$
is an invertible endomorphism.
\end{enumerate} 
By \cite{MR2846490}, coro. 2.3.3, if such a slope decomposition exists, it is unique.  If  $M$ has an $h$-slope decomposition for all $h \in \qq$, we simply say that $M$ has slope decomposition. In this situation we can obviously  define submodules  $M^{=h}$ and $M^{<h}$ of $M$ for all $h \in \qq$.  We let $M^{fs} = \lim_{ h } M^{\leq h}$ be the projective limit of all the slope $\leq h$ factors of $M$ and we call it the finite slope part of $M$ with respect to $T$. There is a projection $M \rightarrow M^{fs}$. The kernel of this projection is the space $M^{\infty s}$ of infinite slope vectors. 

\begin{lem}\label{lemma-slope-decomp} Let $u : M \rightarrow N$ be a map of $F$-vector spaces. Let $T_M$ and $T_N$ be endomorphisms of $M$ and $N$ respectively, such that $T_N \circ u = u   \circ T_M$. Assume that $M$ and $N$ have $h$-slope decompositions for $T_M$ and $T_N$. Then $u(M^{\leq h}) \subseteq N^{\leq h}$ and $u(M^{>h}) \subseteq N^{>h}$. Moreover, $\mathrm{ker}(u)$, $\mathrm{Im}(u)$ and $\mathrm{coker}(u)$ have $h$-slope decompositions. 
\end{lem}
\begin{proof} The inclusions $u(M^{\leq h}) \subseteq N^{\leq h}$ and $u(M^{>h}) \subseteq N^{>h}$ are evident from the property that $u$ is equivariant. It follows that $\mathrm{ker}(u) =  \mathrm{ker}(u) \cap M^{\leq h} \oplus \mathrm{ker}(u) \cap M^{>h}$ and this is an $h$-slope decomposition. The remaining points are left to the reader.
\end{proof}

\subsubsection{Compact operators}\label{section-pointwise-spectral} Let $F$ be a non archimedean field extension of $\qq_p$.  Let $M \in Ob( \mathbf{Ban}(F))$, and let $T $ be a compact endomorphism of $M$. Then by \cite{MR144186}, $M$ has a slope decomposition with respect to $T$. We can generalize this slightly. 

\begin{prop}\label{prop-establishing-slope-decomp} Let $M^\bullet \in \mathcal{D}(F)$ and let $T \in \mathrm{End}_{\mathcal{D}(F)}(M^\bullet)$ be a compact operator (in the sense of definition \ref{defi-notop-compact}). Then for any $h \in \qq$, we have a direct sum decomposition in $\mathcal{D}(F)$, $M^\bullet = M^{\bullet,\leq h} \oplus M^{\bullet,>h}$ characterized by the property that  
the cohomology groups $\HH^i(M^\bullet)$ have $h$-slope decompositions for $T$, and $\HH^i(M^\bullet)^{\leq h} = \HH^i(M^{\bullet, \leq h})$.
\end{prop}
\begin{proof} We first check that if we have such a decomposition, then it is unique. Indeed, let $M^\bullet = M^{\bullet,\leq h} \oplus M^{\bullet,>h} = N^{\bullet,\leq h} \oplus N^{\bullet,>h}$ be two decompositions. Then we see that the map $M^{\bullet,\leq h} \rightarrow N^{\bullet,\leq h}$ obtained by composing the inclusion into $M^\bullet$ and the projection orthogonal to 
$N^{\bullet, >h}$ induces a quasi-isomorphism.  We deduce similarly that the map $M^{\bullet,> h} \rightarrow N^{\bullet,> h}$ is a quasi-isomorphism. On the other hand, the map $M^{\bullet,\leq h} \rightarrow N^{\bullet,> h}$ and $M^{\bullet,> h} \rightarrow N^{\bullet, \leq h}$ are the zero map on cohomology, and they are therefore the zero map (since $F$ is a field).

Now check the existence.  First assume that $T$ is represented by a compact morphism in $\mathcal{K}^{proj}(\mathbf{Ban}(F))$. Then, by applying the spectral theory to each term of the complex, we get a direct sum decomposition $M^\bullet = M^{\bullet,\leq h} \oplus M^{\bullet,>h}$, where for each $i$, $M^i = M^{i,\leq h} \oplus M^{i,>h}$ is the $h$-slope decomposition of $M^i$. Moreover, one deduces from lemma \ref{lemma-slope-decomp} that $\HH^i(M^\bullet)$ has an $h$-slope  decomposition with  $\HH^i(M^\bullet)^{\leq h} = \HH^i(M^{\bullet, \leq h})$, $\HH^i(M^\bullet)^{> h} = \HH^i(M^{\bullet, > h})$.

Now assume that $M^\bullet$ is represented by  $``\lim_i"M_i^\bullet \in  Ob(\mathrm{Pro}_{\N}(\mathcal{K}^{proj}(\mathbf{Ban}(F))))$ and $T$ by a compact operator of  $``\lim_i"M_i^\bullet \in  Ob(\mathrm{Pro}_{\N}(\mathcal{K}^{proj}(\mathbf{Ban}(F))))$. By lemma \ref{lem-potent-compact-factor},   $T$ induces canonically a  compact endomorphism $T_i$ of $M^\bullet_i$ for $i$ large enough and there are factorization diagrams: 
\begin{eqnarray*}
\xymatrix{ M^\bullet_{i+1} \ar[d]\ar[r]^{T_{i+1}} & M_{i+1}^\bullet \ar[d] \\
M_i^\bullet \ar[ru]\ar[r]^{T_i} & M_i^\bullet}
\end{eqnarray*}
For any $h \in \qq$, we deduce that $M_{i+1}^{\bullet, \leq h} \rightarrow M_{i}^{\bullet, \leq h}$ is a quasi-isomorphism. We also deduce that  $M_{i+1}^{\bullet, \leq h} \rightarrow M_i^{\bullet, > h}$ and $M_{i+1}^{\bullet, > h} \rightarrow M_i^{\bullet, \leq h}$ are the zero map on cohomology. Using that we are over a field, we deduce that these are the zero map.  We therefore deduce that $``\lim_i" M_{i}^{\bullet} = ``\lim_i" M_{i}^{\bullet, \leq h} \oplus ``\lim_i" M_{i}^{\bullet, >h} $. 
By taking the limit, we deduce that there is a decomposition $M^\bullet = M^{\bullet, \leq h} \oplus M^{\bullet, >h}$. We now immediately check that this gives an $h$-slope decomposition on cohomology. 
The case where  $M^\bullet$ is represented by  $``\colim_i"M_i^\bullet \in  Ob(\mathrm{Ind}_{\N}(\mathcal{K}^{proj}(\mathbf{Ban}(F))))$  is similar and left to the reader. 

\end{proof}

We have natural projections $M^{\bullet, \leq h'} \rightarrow M^{\bullet, \leq h}$ for any $h' > h \in \qq$.  We let $M^{\bullet, fs} = \mathrm{lim}_{h \geq 0} M^{\bullet, \leq h}$ be the finite slope part of $M^\bullet$. There is a map $M^\bullet \rightarrow M^{\bullet,fs}$. 

\subsubsection{Action of an algebra}\label{section-algebra-acting-fs} Let  $T^+$ be a commutative monoid (with neutral element)  and $T^{++} \subseteq T^+$ be a sub-semigroup (possibly without neutral element), such that $T^+ . T^{++} \subseteq T^{++}$. We also assume that  for any $t, t' \in T^{++}$, there  exists $n \in \mathbb{N}$, $t'' \in T^+$ such that $t^n = t' t''$. 
Let $M \in \mathbf{Ban}(F)$. We assume that we have an algebra action $\ZZ[T^+] \rightarrow \mathrm{End}_F (M)$, such that the ideal $\ZZ[T^{++}]$ acts by potent compact operators.

\begin{lem}\label{lem-indep-t1} For any $t,t' \in T^{++}$ acting compactly on $M$, and any $h \in \qq$, there exists $h' \in \qq$ such that $M^{\leq_t h} \subseteq M^{\leq_{t'} h'}$,  where $M^{\leq_t h}$ and $M^{\leq_{t'} h'}$ are the slope $\leq h$ part for $t$ and $\leq h'$ part for $t'$ respectively. 
\end{lem} 

\begin{proof} We see that $M^{\leq_t h}$ is stable under $T^+$. In particular it is stable under the action of $t'$. Moreover, $t'$ is invertible since there exists $n \in \mathbb{N}$, $t'' \in T^+$ such that $t^n = t' t''$. It follows that $M^{\leq_t h} \subseteq M^{\leq_{t'} h'}$ where $h'$ is the maximum of the valuation of the eigenvalues of $t'$ on $M^{\leq_t h}$.
\end{proof}

Let  $M^\bullet \in Ob(\mathcal{D}(F))$. We assume that we have an algebra action $\ZZ[T^+] \rightarrow \mathrm{End}_{\mathcal{D}(F)}(M^\bullet)$, such that the ideal $\ZZ[T^{++}]$ acts by potent compact operators.  

\begin{lem}\label{lem-indep-t} For any $t,t' \in T^{++}$ acting compactly on $M^\bullet$, the corresponding finite slope parts $M^{\bullet, t-fs}$ and $M^{\bullet, t'-fs}$ are  quasi-isomorphic. 
\end{lem}
\begin{proof}   We deduce from lemma \ref{lem-indep-t1}, that for any $h \in \qq$, the projection $M^\bullet \rightarrow M^{\bullet, \leq_t h } $ factors through $M^\bullet \rightarrow M^{\bullet, \leq_{t'} h'}$ for some $h' \in \qq$.  Passing to the limit over $h$ and $h'$ we deduce that we have a map $M^{\bullet, t-fs} \rightarrow M^{\bullet, t'-fs}$. This map is a quasi-isomorphism. 
\end{proof}

In view of this lemma, we use the notation $M^{\bullet, fs} $ to mean $M^{\bullet, t-fs}$  for any compact operator $t \in T^{++}$.

\subsection{Correspondences and cohomology with support} We now discuss the action of a cohomological correspondence on the cohomology with support of  a sheaf. 
Let $(F, \ocal_F)$ be a non-archimedean local field and let   $\mathcal{X}$ be an adic space of finite type over $\Spa(F, \ocal_F)$. Let 
\begin{eqnarray*}
\xymatrix{ & \mathcal{C} \ar[rd]^{p_1} \ar[ld]_{p_2} & \\
\mathcal{X} & & \mathcal{X}}
\end{eqnarray*}
be a correspondence, where $\mathcal{C}$ is of finite type and $p_1$ and $p_2$ are  morphisms of adic spaces. 

\subsubsection{Action of the correspondences on subsets of $\mathcal{X}$} Let $\mathcal{P}(\mathcal{X})$ be the set of  subsets of $\mathcal{X}$. Let us denote by $T : \mathcal{P}(\mathcal{X}) \rightarrow \mathcal{P}(\mathcal{X})$ the map which takes $\mathcal{B} \in \mathcal{P}(\mathcal{X})$ to $p_2 (p_1^{-1}(\mathcal{B}))$ and $T^t : \mathcal{P}(\mathcal{X}) \rightarrow \mathcal{P}(\mathcal{X})$ the map which takes $\mathcal{B} \in \mathcal{P}(\mathcal{X})$ to $p_1 (p_2^{-1}(\mathcal{B}))$. 

\begin{lem}\label{lem-T-acting-on-set} If $p_1$ is proper  (see \cite{MR1734903}, section 1.3), the map $T^t$ preserves closed subsets and closed subsets with quasi-compact complements.  If $p_1$ is flat, the map $T^t$ preserves open subsets and quasi-compact open subsets.
\end{lem}

\begin{proof} 
Let $\mathcal{Z}$ be a closed subset of $\mathcal{X}$. Then $p_2^{-1} (\mathcal{Z})$ is a closed subset of $\mathcal{C}$, and $T^t(\mathcal{Z})$ is closed by properness. Moreover if $\mathcal{Z}$ has quasi-compact complement, then $p_2^{-1} (\mathcal{Z})$ has quasi-compact complement. We can find a formal model for the projection $p_1$, $\mathfrak{p}_1 : \mathfrak{C} \rightarrow \mathfrak{X}$, which is proper and with the property that the complement of $p_2^{-1} \mathcal{Z}$ is the generic fiber of an open subset  of $\mathfrak{C}$ (\cite{MR1032938}, thm. 1.6, thm. 3.1). Therefore $p_2^{-1} (\mathcal{Z}) = \mathrm{sp}^{-1}( Z_k)$ for  some closed subset $Z_k$ of the special fiber of $\mathfrak{C}$. Since $\mathfrak{p}_1$ is proper, $\mathfrak{p}_1(Z_k)$ is closed, and $T^t(\mathcal{Z}) =\mathrm{sp}^{-1}( \mathfrak{p}_1(Z_k))$ has quasi-compact complement. 
Let $\mathcal{U} \subseteq \mathcal{X}$ be an open subset. Then $p_2^{-1}(\mathcal{U})$ is open in $\mathcal{C}$ and   $T^t(\mathcal{U})$ is open since  flat morphisms are open (\cite{MR1734903}, lemma 1.7.9).  If we assume that $\mathcal{U}$ is quasi-compact, then $p_2^{-1}(\mathcal{U})$ and hence $T^t(\mathcal{U})$ are as well.
\end{proof}

\subsubsection{Action of the correspondence on cohomology}\label{section-action-sheaf}
Let $\mathscr{F}$ be a sheaf of  $\oscr_{\mathcal{X}}$-modules.  Let $\mathcal{U} \subseteq \mathcal{X}$ be an open subset of $\mathcal{X}$ and let $\mathcal{Z} \subseteq \mathcal{X}$ be a closed subset.   We assume that $T(\mathcal{U})$ is open and that $p_1$ is finite flat.   We also assume that we have a map $T : p_2^\star \mathscr{F} \rightarrow p_1^\star \mathscr{F}$.  We can define a map $$T : \mathrm{R}\Gamma_{\mathcal{Z} \cap T(\mathcal{U})}(T(\mathcal{U}), \mathscr{F}) \rightarrow \mathrm{R}\Gamma_{T^t(\mathcal{Z}) \cap \mathcal{U}}(\mathcal{ U}, \mathscr{F})$$ as the following composite:

$$ \mathrm{R}\Gamma_{\mathcal{Z} \cap T(\mathcal{U})}(T(\mathcal{U}), \mathscr{F}) \stackrel{a}\rightarrow \mathrm{R}\Gamma_{p_2^{-1}(\mathcal{Z} \cap T(\mathcal{U}))} (p_2^{-1}( T(\mathcal{U})), p_2^\star \mathscr{F}) \stackrel{b} \rightarrow  \mathrm{R}\Gamma_{p_2^{-1}(\mathcal{Z}) \cap p_1^{-1}(\mathcal{U})} (p_1^{-1} (\mathcal{U}), p_2^\star \mathscr{F}) $$
$$\stackrel{c}\rightarrow \mathrm{R}\Gamma_{p_2^{-1}(\mathcal{Z}) \cap p_1^{-1}(\mathcal{U})} (p_1^{-1} (\mathcal{U}), p_1^\star \mathscr{F}) \stackrel{d}\rightarrow \mathrm{R}\Gamma_{T^t(\mathcal{Z}) \cap \mathcal{U}}( \mathcal{U}, \mathscr{F}),$$

where:
\begin{itemize}
\item $a$ is a pull back map along $p_2^{-1}(T(\mathcal{U})) \rightarrow T(\mathcal{U})$,
\item $b$ is a pull back map along $p_1^{-1} (\mathcal{U} ) \rightarrow p_2^{-1}(T(\mathcal{U}))$. Notice that $p_2^{-1}(\mathcal{Z} \cap T(\mathcal{U}))\cap p_1^{-1}(\mathcal{U}) = p_2^{-1}(\mathcal{Z}) \cap p_1^{-1}(\mathcal{U})$,
\item  $c$ is given by the map $T : p_2^\star \mathscr{F} \rightarrow p_1^\star \mathscr{F}$,
\item $d$ is given by the trace map of lemma \ref{lem-trace-1}.  Notice that $p_1(p_2^{-1}(\mathcal{Z}) \cap p_1^{-1}(\mathcal{U})) \subset T^t(\mathcal{Z}) \cap \mathcal{U}$. 
\end{itemize}
\begin{rem}\label{rem-simplification-map-sheaf} We observe that in the definition of the correspondence, the map $p_2^\star \mathscr{F} \rightarrow p_1^\star \mathscr{F}$ is only used  in  a neighborhood of 
$$ p_2^{-1}(\mathcal{Z} ) \cap p_1^{-1}(\mathcal{U}).$$
Indeed, let $\mathcal{V}$ be  a neighborhood of $p_2^{-1}(\mathcal{Z} ) \cap p_1^{-1}(\mathcal{U})$ in $p_1^{-1}(\mathcal{U})$, then we have $$\mathrm{R}\Gamma_{p_2^{-1}(\mathcal{Z}) \cap p_1^{-1}(\mathcal{U})} (p_1^{-1} (\mathcal{U}), p_i^\star \mathscr{F}) = \mathrm{R}\Gamma_{p_2^{-1}(\mathcal{Z}) \cap p_1^{-1}(\mathcal{U})} ( \mathcal{V}, p_i^\star \mathscr{F}) $$  for $i = 1,2$. 
\end{rem}

\subsubsection{Representing the correspondence by a map of complexes} In this section, we prove under some mild assumptions that the map $T$ of the last section can be represented by a morphism of complexes of Banach spaces. We assume that $\mathscr{F}$ is a locally projective Banach sheaf.  

\begin{lem}\label{lem-representing-mapT} Assume that $\mathcal{U}$ and $T(\mathcal{U})$ are quasi-compact open subsets, and that  $\mathcal{Z}$ and $T^{t}(\mathcal{Z})$ have quasi-compact complements. Then the map $T : \mathrm{R}\Gamma_{\mathcal{Z} \cap T(\mathcal{U})} (T(\mathcal{U}), \mathscr{F}) \rightarrow \mathrm{R}\Gamma_{T^t(\mathcal{Z})\cap \mathcal{U}} (\mathcal{U}, \mathscr{F})$ is represented by a morphism in $\mathcal{K}^{proj}(\mathbf{Ban}(F))$. 
\end{lem}
\begin{proof} We consider finite affinoid $\mathscr{F}$-acyclic coverings $T(\mathcal{U}) = \cup_{i \in I} U_i$ and $\mathcal{U} = \cup_{i \in I'} U'_i$. We also let $Z = \cap_{j \in J} Z_j$, $T^t(\mathcal{Z}) = \cap_{j \in J'} Z'_j$ where $Z_j^c$ and $(Z'_j)^c$  are affinoid and $\mathscr{F}$-acyclic and $J$, $J'$ are finite. 
We see that $\mathrm{R}\Gamma_{\mathcal{Z} \cap T(\mathcal{U})} (T(\mathcal{U}), \mathscr{F})$ is represented by a cone of a map of \v{C}ech complexes
$$ \mathrm{Cone} \big( \check{C}( \{U_i\}_{i \in I}, \mathscr{F}) \rightarrow \check{C}( \{U_i \cap Z_j^c\}_{(i,j) \in I\times J}, \mathscr{F}) \big)[-1],$$
and similarly $\mathrm{R}\Gamma_{T^t(\mathcal{Z}) \cap \mathcal{U}} (\mathcal{U}, \mathscr{F})$ is represented by
$$ \mathrm{Cone} \big( \check{C}( \{U'_i\}_{i \in I'}, \mathscr{F}) \rightarrow \check{C}( \{U'_i \cap (Z'_j)^c\}_{(i,j) \in I'\times J'}, \mathscr{F}) \big)[-1].$$
We now proceed to represent the map $T$. 
We first have a map (pullback under $p_2$): 
$$A :  \mathrm{Cone} \big( \check{C}( \{U_i\}_{i \in I}, \mathscr{F}) \rightarrow \check{C}( \{U_i \cap Z_j^c\}_{(i,j) \in I\times J}, \mathscr{F}) \big)[-1] \rightarrow  $$ $$\mathrm{Cone} \big( \check{C}( \{p_2^{-1}(U_i)\}_{i \in I}, p_2^\star \mathscr{F})  \rightarrow \check{C}( \{p_2^{-1}(U_i \cap Z_j^c)\}_{(i,j) \in I\times J}, p_2^\star \mathscr{F}) \big)[-1]$$
We then have a restriction map, composed with the map $p_2^\star \mathscr{F} \rightarrow p_1^\star \mathscr{F}$: 
$$B: \mathrm{Cone} \big( \check{C}( \{p_2^{-1}(U_i)\}_{i \in I}, p_2^\star \mathscr{F}) \rightarrow \check{C}( \{p_2^{-1}(U_i \cap Z_j^c)\}_{(i,j) \in I\times J}, p_2^\star \mathscr{F}) \big)[-1]
\rightarrow $$ $$ \mathrm{Cone} \big( \check{C}( \{p_2^{-1}(U_i) \cap p_1^{-1}(\mathcal{U})\}_{i \in I}, p_1^\star\mathscr{F}) \rightarrow \check{C}( \{p_2^{-1}(U_i \cap Z_j^c) \cap p_1^{-1}(\mathcal{U} \cap T^t(\mathcal{Z})^c)\}_{(i,j) \in I\times J}, p_1^\star\mathscr{F}) \big)[-1].$$
We now observe that $\{p_2^{-1}(U_i) \cap p_1^{-1}(\mathcal{U})\}_{i \in I}$ is a covering of $p_1^{-1}(\mathcal{U})$ since $p_1^{-1} (\mathcal{U}) \subseteq p_2^{-1}(T(\mathcal{U}))$. We also observe that $\{p_2^{-1}(U_i \cap Z_j^c) \cap p_1^{-1}(\mathcal{U} \cap T^t(\mathcal{Z})^c)\}_{(i,j) \in I\times J}$ is a covering of $p_1^{-1}(\mathcal{U} \cap T^t(\mathcal{Z})^c)$ since $p_2^{-1}(\mathcal{Z})  \subseteq p_1^{-1} (T^t(\mathcal{Z}))$. 

We now pick an $\mathscr{F}$-acyclic finite affinoid covering  $ p_1^{-1}(\mathcal{U}) = \cup_{i \in I''} U''_i$ which refines both $\{p_2^{-1}U_i \cap p_1^{-1}(\mathcal{U})\}_{i \in I}$ and $\{ p_1^{-1} U_i'\}_{i\in I'}$. We also let $p_1^{-1} (\mathcal{Z}) = \cap_{j \in J''} Z''_j$ where $(Z''_j)^c$ is affinoid and $\mathscr{F}$-acyclic, $J''$ is finite, and the covering $\cup_{j \in J''} (Z''_j)^c$ of $p_1^{-1} (\mathcal{Z})^c$ refines both $\{p_1^{-1} (Z'_j)^c\}_{j \in J'}$ and $\{p_2^{-1} (Z_j)^c \cap p_1^{-1}(\mathcal{Z})^c\}_{j \in J}$. 
We deduce that there are refinement maps:
$$C : \mathrm{Cone} \big( \check{C}( \{p_2^{-1}(U_i) \cap p_1^{-1}(\mathcal{U})\}_{i \in I}, p_1^\star\mathscr{F}) \rightarrow \check{C}( \{p_2^{-1}(U_i \cap Z_j^c) \cap p_1^{-1}(\mathcal{U} \cap T^t(\mathcal{Z})^c)\}_{(i,j) \in I\times J}, p_1^\star\mathscr{F}) \big)[-1] \rightarrow $$ $$ \mathrm{Cone} \big( \check{C}( \{U_i''\}_{i \in I''}, p_1^\star\mathscr{F}) \rightarrow \check{C}( \{U''_i \cap (Z''_j)^c\}_{(i,j) \in I''\times J''}, p_1^\star\mathscr{F}) \big)[-1]$$
and
$$D : \mathrm{Cone} \big( \check{C}( \{ p_1^{-1}(U'_i)\}_{i \in I'}, p_1^\star\mathscr{F}) \rightarrow \check{C}( \{p_1^{-1}(U'_i \cap (Z'_j)^c)\}_{(i,j) \in I'\times J'}, p_1^\star\mathscr{F}) \big)[-1] \rightarrow $$ $$ \mathrm{Cone} \big( \check{C}( \{U_i''\}_{i \in I''}, p_1^\star\mathscr{F}) \rightarrow \check{C}( \{U''_i \cap (Z''_j)^c\}_{(i,j) \in I''\times J''}, p_1^\star\mathscr{F}) \big)[-1].$$
But since $p_1$ is finite, $p_1^{-1}$ sends $\mathscr{F}$-acyclic affinoids to $p_1^\star\mathscr{F}$-acyclic affinoids.  Hence $D$ is actually a quasi-isomorphism, and we can pick an inverse $D^{-1}$ for $D$ up to homotopy by lemma \ref{lem-invert-quasi-isom}.

We finally have a trace map: 
 $$ E : \mathrm{Cone} \big( \check{C}( \{ p_1^{-1}(U'_i)\}_{i \in I'}, p_1^\star\mathscr{F}) \rightarrow \check{C}( \{p_1^{-1}(U'_i \cap (Z'_j)^c)\}_{(i,j) \in I'\times J'}, p_1^\star\mathscr{F}) \big)[-1] \rightarrow $$ $$  \mathrm{Cone} \big( \check{C}( \{U'_i\}_{i \in I'}, \mathscr{F}) \rightarrow \check{C}( \{U'_i \cap (Z'_j)^c\}_{(i,j) \in I'\times J'}, \mathscr{F}) \big)[-1]$$
Then $E \circ D^{-1} \circ C \circ B \circ A$ represents $T$. 
\end{proof}

\begin{coro}\label{coro-representing-mapT} Assume that $\mathcal{U} = \cup_n \mathcal{U}_n$  is a countable union of quasi-compact open subsets. Let $\mathcal{Z} = \cap_n \mathcal{Z}_n$ be a countable intersection of closed subsets with quasi-compact complement.  Assume that $T(\mathcal{U}_n)$ is quasi-compact  and $T^t(\mathcal{Z}_n)$ is the complement of a quasi-compact space.  Assume also that $\mathcal{U}$, $T(\mathcal{U})$  are a finite union of quasi-Stein spaces and that the complements of $\mathcal{Z}$ and $\mathcal{T}^t(\mathcal{Z})$ are finite union of quasi-Stein spaces.  Then the map $T : \mathrm{R}\Gamma_{\mathcal{Z} \cap T(\mathcal{U})} (T(\mathcal{U}), \mathscr{F}) \rightarrow \mathrm{R}\Gamma_{T^t(\mathcal{Z})\cap \mathcal{U}} (\mathcal{U}, \mathscr{F})$ is represented in $\mathrm{Pro}_{\N}(\mathcal{K}^{proj}(\mathbf{Ban}(F)))$.
\end{coro}

\begin{proof} We can assume that $\mathcal{U}_n \subseteq \mathcal{U}_{n+1}$ and $\mathcal{Z}_n \subseteq \mathcal{Z}_{n-1}$. 
By lemma \ref{lem-representing-mapT}, the maps: $\mathrm{R}\Gamma_{\mathcal{Z}_n \cap T(\mathcal{U}_n)} (T(\mathcal{U}_n), \mathscr{F}) \rightarrow \mathrm{R}\Gamma_{T^t(\mathcal{Z}_n)\cap \mathcal{U}_n} (\mathcal{U}_n, \mathscr{F})$ are representable by maps in $\mathcal{K}^{proj}(\mathbf{Ban}(F))$. 
We then pass to the limit over $n$. See also lemma \ref{lem-complex-frechet}.
\end{proof} 

\begin{rem} We note that the assumptions of lemma \ref{lem-representing-mapT} and corollary \ref{coro-representing-mapT} are satisfied if $p_1$ and $p_2$ are   finite flat by lemma \ref{lem-T-acting-on-set}.
In our applications to Shimura varieties, the maps of the (toroidal compactifications of the) correspondences are only finite flat for suitable choices of cone decompositions. However, we will only consider cohomology with support of certain subsets which are ``well positioned'' with respect to the boundary which will have the consequence that the cohomologies we consider can be defined for any choice of cone decomposition and are further independent of this choice.
\end{rem}
 
\subsection{A formal analytic continuation result}\label{subsection-analytic-continuation}  In this section we prove a  result which will identify the finite slope part of different cohomologies.  This can be seen  as an (abstract)  generalization of \cite{MR1937198}, thm. 5.2. Let $\mathcal{X}$ be an adic space of finite type over $\Spa (F, \ocal_F)$, and let $p_1, p_2 : \mathcal{C} \rightarrow \mathcal{X}$ be a correspondence. We assume that $p_1$ is finite flat. We also let  $\mathscr{F}$ be a locally projective Banach sheaf (see definition \ref{defi-banach-sheaf}), and we assume that the map $T : p_2^\star \mathscr{F} \rightarrow p_1^\star \mathscr{F}$ is compact (see definition \ref{defi-compact-map}).  We recall that when $\mathscr{F}$ is a locally free coherent sheaf, any map $p_2^\star \mathscr{F} \rightarrow p_1^\star \mathscr{F}$ is compact. 

\subsubsection{The diamond of a correspondence} We now make the further assumption that $T^t(\mathcal{Z}) \subseteq \mathcal{Z}$ and $T(\mathcal{U}) \subseteq \mathcal{U}$. We can build the following diamond shaped diagram:

\begin{eqnarray*}
\xymatrix{ & \mathrm{R}\Gamma_{\mathcal{Z} \cap \mathcal{U}}(\mathcal{U}, \mathscr{F})  \ar[ld]^{res} &  \\
\mathrm{R}\Gamma_{\mathcal{Z} \cap T(\mathcal{U})} (T(\mathcal{U}), \mathscr{F}) \ar[rr]^{T}& & \mathrm{R}\Gamma_{T^t(\mathcal{Z}) \cap \mathcal{U}}(\mathcal{U}, \mathscr{F}) \ar[lu]^{cores} \ar[ld]^{res} \\
& \mathrm{R}\Gamma_{T^t(\mathcal{Z}) \cap T(\mathcal{U})}(T(\mathcal{U}), \mathscr{F}) \ar[lu]^{cores}& }
\end{eqnarray*}
where the composition of the top and lower triangle  $ res \circ cores$ and $cores \circ res$ are equal.  We can define an endomorphism of each of the four cohomologies occurring in this diagram by suitably composing $T$ with restriction and corestriction maps.  By abuse of notation we also call these endomorphisms $T$.  One checks immediately that all the maps in the diamond are equivariant for these endomorphisms.
 
\begin{prop}\label{last-prop-diamond} Assume that all the complexes appearing are objects of $\mathrm{Pro}_{\N}(\mathcal{K}^{proj}(\mathbf{Ban}(F)))$ and that the endomorphisms $T$ are potent compact. After taking finite slope parts, all the maps in the resulting diamond 
\begin{eqnarray*}
\xymatrix{ & \mathrm{R}\Gamma_{\mathcal{Z} \cap \mathcal{U}}(\mathcal{U}, \mathscr{F})^{fs}  \ar[ld]^{res} &  \\
\mathrm{R}\Gamma_{\mathcal{Z} \cap T(\mathcal{U})} (T(\mathcal{U}), \mathscr{F})^{fs} \ar[rr]^{T}& & \mathrm{R}\Gamma_{T^t(\mathcal{Z}) \cap \mathcal{U}}(\mathcal{U}, \mathscr{F})^{fs} \ar[lu]^{cores} \ar[ld]^{res} \\
& \mathrm{R}\Gamma_{T^t(\mathcal{Z}) \cap T(\mathcal{U})}(T(\mathcal{U}), \mathscr{F})^{fs} \ar[lu]^{cores}& }
\end{eqnarray*}
are quasi-isomorphisms.
\end{prop}

\begin{proof} For each map $f$ in the diagram, there is another map $g$ (the composition of two other arrows in the diagram) so that $gf$ and $fg$ are the endomorphism $T$ on the source and target of $f$.  As $T$ becomes invertible on the finite slope part, so does $f$.
\end{proof}

\begin{rem} Under the assumptions of corollary \ref{coro-representing-mapT}, the operator $T$ is potent compact if  the map ``restriction-corestriction'' obtained by composing the top  or bottom arrows of the diamond $\mathrm{R}\Gamma_{T^t(\mathcal{Z}) \cap \mathcal{U}}(\mathcal{U}, \mathscr{F}) \rightarrow \mathrm{R}\Gamma_{\mathcal{Z} \cap T(\mathcal{U})} (T(\mathcal{U}), \mathscr{F})$ is compact (see lemma \ref{lem-criterium-compacity2} for a criterium).
\end{rem}

As a corollary, we note the following fact:

\begin{coro}\label{coro-allowing-more-support}  Under the hypothesis of  proposition \ref{last-prop-diamond}, let $\mathcal{U}'$ and $\mathcal{Z}'$ be open and closed subsets, such that  $T(\mathcal{U})\cap \mathcal{Z} \subseteq \mathcal{U}' \subseteq \mathcal{U} $ and $\mathcal{U} \cap T^t(\mathcal{Z}) \subseteq \mathcal{Z}' \cap \mathcal{U}' \subseteq \mathcal{Z}$. Assume also that   $\mathrm{R}\Gamma_{\mathcal{Z}' \cap \mathcal{U}'} (\mathcal{U}, \mathscr{F})$ is an object of $\mathrm{Pro}_{\N}(\mathcal{K}^{proj}(\mathbf{Ban}(F)))$.  Then the operator $T$ is well defined and potent compact on $\mathrm{R}\Gamma_{\mathcal{Z}' \cap \mathcal{U}'} (\mathcal{U}', \mathscr{F})$ and the finite slope part of $\mathrm{R}\Gamma_{\mathcal{Z}' \cap \mathcal{U}'} (\mathcal{U}', \mathscr{F})$ is canonically quasi-isomorphic to the finite slope part of $\mathrm{R}\Gamma_{\mathcal{Z}\cap\mathcal{ U}} (\mathcal{U}, \mathscr{F})$.
\end{coro}

\begin{proof}   Under our assumptions, we have restriction-corestriction maps: $$\mathrm{R}\Gamma_{\mathcal{Z}' \cap \mathcal{U}'} (\mathcal{U}') \rightarrow \mathrm{R}\Gamma_{\mathcal{Z}' \cap \mathcal{U}' \cap T(\mathcal{U})} (\mathcal{U}' \cap T(\mathcal{U}))  \rightarrow \mathrm{R}\Gamma_{\mathcal{Z} \cap T(\mathcal{U})} (T(\mathcal{U}) \cap \mathcal{U}', \mathscr{F}) = \mathrm{R}\Gamma_{\mathcal{Z} \cap T(\mathcal{U})} (T(\mathcal{U}), \mathscr{F})$$ and $$ \mathrm{R}\Gamma_{T^t(\mathcal{Z}) \cap \mathcal{U}}(\mathcal{U}, \mathscr{F}) \rightarrow  \mathrm{R}\Gamma_{T^t(\mathcal{Z}) \cap \mathcal{U}'}(\mathcal{U}', \mathscr{F}) \rightarrow  \mathrm{R}\Gamma_{\mathcal{Z}' \cap \mathcal{U}'}(\mathcal{U}', \mathscr{F}). $$
We can build the following diagram: 
\begin{eqnarray*}
\xymatrix{ & \mathrm{R}\Gamma_{\mathcal{Z}' \cap \mathcal{U}'}(\mathcal{U}', \mathscr{F})  \ar[ld] &  \\
\mathrm{R}\Gamma_{\mathcal{Z} \cap T(\mathcal{U})} (T(\mathcal{U}), \mathscr{F}) \ar[rr]^{T}& & \mathrm{R}\Gamma_{T^t(\mathcal{Z}) \cap \mathcal{U}}(\mathcal{U}, \mathscr{F}) \ar[lu] }
\end{eqnarray*}
and $T$ is therefore well defined and potent compact on $\mathrm{R}\Gamma_{\mathcal{Z}' \cap \mathcal{U}'}(\mathcal{U}', \mathscr{F})$.  We deduce that $\mathrm{R}\Gamma_{\mathcal{Z}' \cap \mathcal{U}'}(\mathcal{U}', \mathscr{F})^{fs}$ is quasi-isomorphic to  $\mathrm{R}\Gamma_{\mathcal{Z} \cap T(\mathcal{U})} (T(\mathcal{U}), \mathscr{F})^{fs}$ which is in turn quasi-isomorphic to $\mathrm{R}\Gamma_{\mathcal{Z}\cap \mathcal{U}} (\mathcal{U}, \mathscr{F})^{fs}$.
\end{proof}

\subsubsection{The infinite diamond}\label{section-infinite-diamond}
  It is interesting to iterate the operator $T$.  We  now work under the stronger  assumption that $p_1$ and $p_2$ are finite flat. These assumptions imply that for any $n \geq 0$, the $n$-th iterate $\mathcal{C}^{(n)}$ of the correspondence $\mathcal{C}$ comes with two finite flat projections $p_{1,n}$ and $p_{2,n}$.

\begin{rem} Later we will apply this material to toroidal compactifications of Shimura varieties. We  therefore have to work under slightly more general assumptions. Namely, it is not possible in general to find cone decompositions such that all the maps between compactifications of Hecke correspondences are finite flat. Nevertheless, by allowing suitable changes of the cone decompositions, we can always assume that a given map is finite flat. Moreover,  the composition of compactified  Hecke correspondences and their action on the cohomology has been explained in detail in section \ref{section-compos-Hecke-op}. Therefore, for the clarity of the exposition, we will keep the assumption that $p_1$ and $p_2$ are finite flat here.  

\end{rem}

We let $\mathcal{U}_m = T^m(\mathcal{U})$ and $\mathcal{Z}_n = (T^t)^n(\mathcal{Z})$. We assume that $T(\mathcal{U}) \subseteq \mathcal{U}$ and $T^t(\mathcal{Z}) \subseteq \mathcal{Z}$.  The sequences $\{\mathcal{U}_m\}_{m \geq 0}$ and $\{\mathcal{Z}_n\}_{n \geq 0}$ are therefore decreasing.

We can then construct diamonds as above for all $n, m \geq 0$: 

\begin{eqnarray*}
\xymatrix{ & \mathrm{R}\Gamma_{\mathcal{Z}_n \cap \mathcal{U}_m}(\mathcal{U}_m, \mathscr{F})  \ar[ld]^{res} &  \\
\mathrm{R}\Gamma_{\mathcal{Z}_n \cap \mathcal{U}_{m+1}} ( \mathcal{U}_{m+1}, \mathscr{F}) \ar[rr]^{T}& & \mathrm{R}\Gamma_{\mathcal{Z}_{n+1} \cap \mathcal{U}_m}(\mathcal{U}_m, \mathscr{F}) \ar[lu]^{cores} \ar[ld]^{res} \\
& \mathrm{R}\Gamma_{\mathcal{Z}_{n+1}\cap \mathcal{U}_{m+1}}(\mathcal{U}_{m+1}, \mathscr{F}) \ar[lu]^{cores}& }
\end{eqnarray*}

and we can add them to get an infinite  diamond diagram looking like:

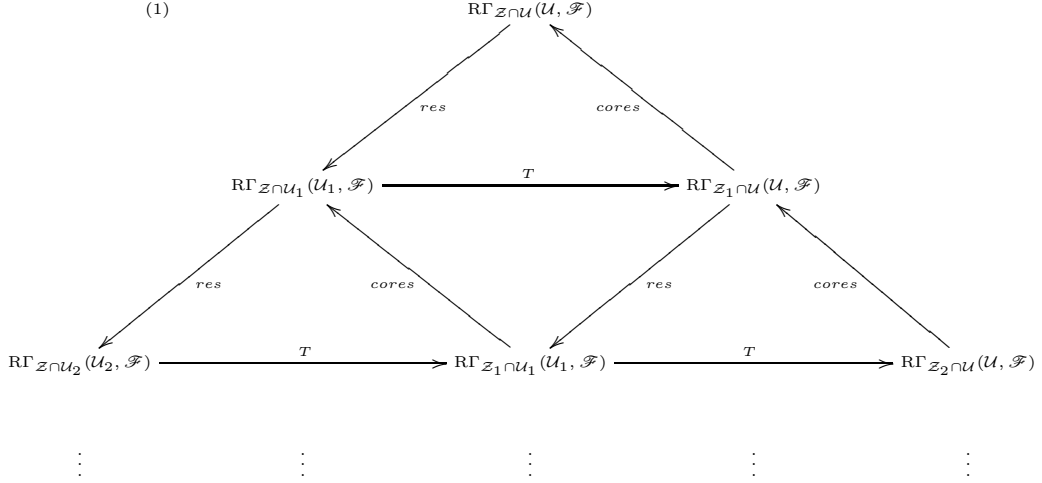
\begin{tiny} \begin{eqnarray}\label{eqn-diamond}
\xymatrix{ & & \mathrm{R}\Gamma_{\mathcal{Z} \cap \mathcal{U}}(\mathcal{U}, \mathscr{F})  \ar[ldd]^{res} & & \\
& & & & \\
&\mathrm{R}\Gamma_{\mathcal{Z} \cap \mathcal{U}_1} (\mathcal{U}_1, \mathscr{F}) \ar[ldd]^{res} \ar[rr]^{T}& & \mathrm{R}\Gamma_{\mathcal{Z}_1 \cap \mathcal{U}}(\mathcal{U}, \mathscr{F}) \ar[luu]^{cores} \ar[ldd]^{res}  & \\
& & & & \\
\mathrm{R}\Gamma_{\mathcal{Z} \cap \mathcal{U}_2} (\mathcal{U}_2, \mathscr{F}) \ar[rr]^{T}& & \mathrm{R}\Gamma_{\mathcal{Z}_1 \cap \mathcal{U}_1}(\mathcal{U}_1, \mathscr{F}) \ar[luu]^{cores} \ar[rr]^{T}&  & \mathrm{R}\Gamma_{\mathcal{Z}_2 \cap \mathcal{U}} (\mathcal{U}, \mathscr{F}) \ar[luu]^{cores} \\
\vdots &  \vdots & \vdots &\vdots & \vdots\\
}
\end{eqnarray}

\end{tiny}

We assume that all the objects of the above diagram belong to $\mathrm{Pro}_{\N}(\mathcal{K}^{proj}(\mathbf{Ban}(F)))$. We now make  the further assumption that there exists $(m_0,n_0)$ such that one morphism  ``restriction-corestriction'' obtained by composing the top  or low arrows of a diamond $\mathrm{R}\Gamma_{\mathcal{Z}_{n_0+1} \cap \mathcal{U}_{m_0}}(\mathcal{U}_{m_0}, \mathscr{F}) \rightarrow \mathrm{R}\Gamma_{\mathcal{Z}_{n_0} \cap \mathcal{U}_{m_0+1}} (\mathcal{U}_{m_0+1}, \mathscr{F})$ is compact. 

For any $m, n$ with $m, n \geq 0$ and $(m,n) \neq (0,0)$, we can define an endomorphism $T_{m,n} : \mathrm{R}\Gamma_{\mathcal{Z}_n \cap\mathcal{ U}_m}(\mathcal{U}_m, \mathscr{F}) \rightarrow \mathrm{R}\Gamma_{\mathcal{Z}_n \cap \mathcal{U}_m}(\mathcal{U}_m, \mathscr{F})$  by composing $T$, $res$ and $cores$ in a suitable order. We abuse notation and denote this operator by $T$. The operator $T$ is potent compact because some power of it will factor over the compact ``restriction-corestriction'' map above.     In any case, we can speak of the finite slope direct factor of $\mathrm{R}\Gamma_{\mathcal{Z}_n\cap \mathcal{U}_m}(\mathcal{U}_m, \mathscr{F})$ for $T$.

\begin{thm}\label{last-thm-diamond} On the finite slope part, all the morphisms of the infinite diamond are quasi-isomorphisms.
\end{thm}

\begin{proof} This follows from proposition \ref{last-prop-diamond}.
\end{proof}

\begin{coro}\label{coro-infinitediamond}  Under the hypothesis of theorem \ref{last-thm-diamond}, assume that there is $m,n,s \in \ZZ_{\geq 0}$ such that   $ (T^t)^{n+s} (\mathcal{Z}) \cap T^m(\mathcal{U}) \subseteq \mathcal{Z}'  \cap \mathcal{U}'\subseteq (T^t)^n(\mathcal{Z})$ is a closed subset  and $T^{m+s}(\mathcal{U}) \cap (T^t)^n(\mathcal{Z}) \subseteq \mathcal{U}' \subseteq T^m(\mathcal{U}) $. Assume moreover that $\mathrm{R}\Gamma_{\mathcal{Z}' \cap \mathcal{U}'} (\mathcal{U}', \mathscr{F})$ is an object of $\mathrm{Pro}_{\N}(\mathcal{K}^{proj}(\mathbf{Ban}(F)))$.   Then the operator $T^s$ is well defined and potent compact on $\mathrm{R}\Gamma_{\mathcal{Z}' \cap \mathcal{U}'} (\mathcal{U}', \mathscr{F})$ and the finite slope part of $\mathrm{R}\Gamma_{\mathcal{Z}' \cap \mathcal{U}'} (\mathcal{U}', \mathscr{F})$ is canonically quasi-isomorphic to the finite slope part of $\mathrm{R}\Gamma_{\mathcal{Z}\cap \mathcal{U}} (\mathcal{U}, \mathscr{F})$.
\end{coro}

\begin{proof} This follows from corollary \ref{coro-allowing-more-support}.
\end{proof}

\subsection{Overconvergent cohomologies}\label{section-overconvergent-cohomologies} Let $(G,X)$ be an abelian type Shimura datum such that $G_{\qq_p}$ is quasi split. 
Let $w \in \WM$.  For a choice of $+$ or $-$ and a weight $\kappa \in  X^\star(T^c)^{M_\mu,+}$ we want to define finite slope overconvergent cohomologies $\mathrm{R}\Gamma_w(K^p, \kappa)^{\pm,fs}$ and their cuspidal counterparts $\mathrm{R}\Gamma_w(K^p, \kappa, cusp)^{\pm, fs}$ by taking cohomologies with suitable support conditions of neighborhoods of the inverse image under the Hodge-Tate period map of $\mathcal{P_\mu}\backslash\mathcal{P_\mu}wK_p$, and by taking finite slope parts for a suitable Hecke operator.  We will also define variants $\mathrm{R}\Gamma_w(K^p,\kappa,\chi)^{\pm,fs}$ and $\mathrm{R}\Gamma_w(K^p,\kappa,\chi,cusp)^{\pm,fs}$ where $\chi:T(\ZZ_p)\to {\overline{F}}^\times$ is a finite order character.

\subsubsection{First definition}\label{section-first-defin}

For a level $K_p=K_{p,m',b}$ with $m'\geq b\geq 0$ and $m'>0$ and a weight $\kappa \in  X^\star(T^c)^{M_\mu,+}$, we define:
$$\mathrm{R}\Gamma_w(K^pK_p, \kappa)^{+,fs}  : =   \mathrm{R}\Gamma_{ (\pi_{HT, K_p}^{tor})^{-1}(]C_{w,k}[_{0, \overline{0}})}((\pi_{HT, K_p}^{tor})^{-1}(]X_{w,k}[), \mathcal{V}_\kappa)^{+,fs}.$$
Implicit in this definition is that this cohomology is an object of $\mathrm{Pro}_{\N}(\mathcal{K}^{proj}(\mathbf{Ban}(F)))$ and that $\mathcal{H}_{p,m',b}^+$ acts on it in a way that $\mathcal{H}_{p,m',b}^{++}$ acts by potent compact operators (this will be proved below in Theorem \ref{thm-finiteslopecoho}). 

Similarly for a weight $\kappa \in  X^\star(T^c)^{M_\mu,+}$, we define:
$$\mathrm{R}\Gamma_w(K^pK_p, \kappa)^{-,fs}  : =   \mathrm{R}\Gamma_{(\pi_{HT, K_p}^{tor})^{-1}(]C_{w,k}[_{\overline{0}, {0}})}((\pi_{HT, K_p}^{tor})^{-1}(]Y_{w,k}[), \mathcal{V}_\kappa)^{-,fs}.$$
Again implicit in this definition is that this cohomology is an object of $\mathrm{Pro}_{\N}(\mathcal{K}^{proj}(\mathbf{Ban}(F)))$ and that $\mathcal{H}_{p,m',b}^-$ acts on it in a way that $\mathcal{H}_{p,m',b}^{--}$ acts by potent compact operators (this will also be proved below in Theorem \ref{thm-finiteslopecoho}). 

We have similar definitions for cuspidal cohomology.

\subsubsection{Existence of finite slope cohomology}

\begin{thm}\label{thm-finiteslopecoho} Let $K_p=K_{p,m,b}$ for some $m\geq b\geq0$ with $m>0$, and fix $w\in\WM$ and $\kappa \in  X^\star(T^c)^{M_\mu,+}$.
 \begin{enumerate}
\item  There is an action of $\mathcal{H}_{p,m,b}^+$ on $\mathrm{R}\Gamma_{ (\pi_{HT, K_p}^{tor})^{-1}(]C_{w,k}[_{0, \overline{0}})}((\pi_{HT, K_p}^{tor})^{-1}(]X_{w,k}[), \mathcal{V}_\kappa)$ for which $\mathcal{H}_{p,m,b}^{++}$ acts via potent compact operators.  The same statement holds for $\mathrm{R}\Gamma_{ (\pi_{HT, K_p}^{tor})^{-1}(]C_{w,k}[_{0, \overline{0}})}((\pi_{HT, K_p}^{tor})^{-1}(]X_{w,k}[), \mathcal{V}_\kappa(-D))$.
\item  There is an action of $\mathcal{H}_{p,m,b}^-$ on $\mathrm{R}\Gamma_{(\pi_{HT, K_p}^{tor})^{-1}(]C_{w,k}[_{\overline{0}, {0}})}((\pi_{HT, K_p}^{tor})^{-1}(]Y_{w,k}[), \mathcal{V}_\kappa)$ for which $\mathcal{H}_{p,m,b}^{--}$ acts via potent compact operators.  The same statement holds for $\mathrm{R}\Gamma_{(\pi_{HT, K_p}^{tor})^{-1}(]C_{w,k}[_{\overline{0}, {0}})}((\pi_{HT, K_p}^{tor})^{-1}(]Y_{w,k}[), \mathcal{V}_\kappa(-D))$.
 \end{enumerate}
\end{thm}

\begin{proof} We only prove point $1$ for non cuspidal cohomology. The rest of the argument is very similar and left to the reader.

Let $U=(\pi^{tor}_{HT, K_{p}})^{-1}(]X_{w,k}[)$ and $Z=(\pi^{tor}_{HT, K_{p}})^{-1}(\overline{]Y_{w,k}[})$.  By lemma \ref{lem-topological-description} we have
$$\mathrm{R}\Gamma_{U\cap Z}(U,\mathcal{V}_\kappa)=\mathrm{R}\Gamma_{ (\pi_{HT, K_p}^{tor})^{-1}(]C_{w,k}[_{0, \overline{0}})}((\pi_{HT, K_p}^{tor})^{-1}(]X_{w,k}[), \mathcal{V}_\kappa).$$

Now if $T$ is the Hecke operator $[K_{p}tK_{p}]$ for any $t\in T^+$, we have $T(U)\subseteq U$ and $T^t(Z)\subseteq Z$ by lemma \ref{lem-dynamic-corres} and hence there is an action of $\mathcal{H}_{p,m,b}^+$ on $\mathrm{R}\Gamma_{U\cap Z}(U,\mathcal{V}_\kappa)$ via the construction explained at the beginning of Section \ref{subsection-analytic-continuation}. That this action defines an action of the Hecke algebra $\mathcal{H}_{p,m,b}^{+}$  follows from the discussion of section \ref{section-compos-Hecke-op}. 

Now suppose that $T$ is associated to $t\in T^{++}$. In order to simplify notations, we choose $t$ such  that $\min(t) = \inf_{\alpha \in \Phi^+} v(\alpha(t)) \geq 1$.  In order to show the action of $T$ is potent compact, it suffices to show that the ``restriction-corestriction'' map
$$\mathrm{R\Gamma}_{U \cap T^t(Z)} ( U, \mathcal{V}_\kappa) \rightarrow \mathrm{R\Gamma}_{T(U) \cap Z} ( T(U), \mathcal{V}_\kappa)$$
is compact.  We need to check the assumptions of  lemma \ref{lem-criterium-compacity2}. 
 
 We will be done if we can find a quasi-compact open subset $U'$ such that $T(U) \cap Z \subseteq U'$ and $\overline{U'} \subseteq U$, and closed subset  $Z_1, Z_2$, with quasi-compact complement, such that $U \cap T^t(Z) \subseteq Z_1 $, $Z_1 \subseteq \overset{\circ}Z_2$ and $Z_2 \subseteq Z$. 
By  lemma \ref{lem-dynamic-corres},  we have:
 $$U \cap T^t(Z) \subseteq (\pi^{tor}_{HT, K_{p}})^{-1}\big( ]C_{w,k}[_{0, \overline{1}}K_{p} \big)$$
and  $$T(U) \cap Z \subseteq (\pi^{tor}_{HT, K_{p}})^{-1}\big( ]C_{w,k}[_{1, \overline{0}}K_{p} \big)$$
We first find $U'$. 
By lemma \ref{lem-topological-description}, $$(\pi_{HT, K_{p}}^{tor})^{-1}\big( ]C_{w,k}[_{0, -\infty}K_{p} \big) \subseteq U.$$  We may now take $U_1$ a quasi-compact open  of $\mathcal{FL}$ with the property  that  $]C_{w,k}[_{1, \overline{0}} \subseteq U_1 \subseteq ]C_{w,k}[_{\frac{1}{2}, -1}$ and let $U' = (\pi_{HT,K_p}^{tor})^{-1}(U_1 K_p)$. Observe that $U_1  K_p$ is a finite union of translates of $U_1$.  It follows easily from lemma \ref{lem-topological-description}, (1) that $U'$ has the required property. 

We now proceed to find $Z_1$ and $Z_2$. By lemma \ref{lem-topological-description}, $$(\pi_{HT, K_{p,m,b}}^{tor})^{-1}\big( ]C_{w,k}[_{-\infty, {0}}\big) \subseteq (\pi_{HT, K_{p,m,b}}^{tor})^{-1}(]Y_{w,k}[).$$

We can find $]C_{w,k}[_{0, \overline{1}} \subseteq W_1 \subseteq W_2 \subseteq  ]C_{w,k}[_{-\infty, {0}}$ with $W_1$, $W_2$  closed with quasi-compact complements and $W_1 \subseteq \overset{\circ}W_2$.  We deduce that $]C_{w,k}[_{0, \overline{1}} K_p \subseteq W_1 K_p \subseteq W_2K_p \subseteq ]Y_{w,k}[$ and that $W_1 K_p \subseteq \overset{\circ}{W_2 K_p} = \overset{\circ}{ W_2} K_p$ since $W_1 K_p $ and $W_2K_p$ are a finite union of translates of $W_1$ and $W_2$. We let $Z_1 = (\pi_{HT, K_{p}}^{tor})^{-1} ( W_1K_p)$ and $Z_2 = (\pi_{HT, K_{p}}^{tor})^{-1} ( W_2K_p)$.
We let $Z_1 = (\pi_{HT, K_{p,m,b}}^{tor})^{-1} ( \overline{\cup_{ s < \frac{1}{2}} ]C_{w,k}[_{\overline{0}, s}} K_{p,m,b})$.  We now find $Z_2$.  We first observe that $ \overline{\cup_{ s < \frac{1}{2}} ]C_{w,k}[_{\overline{0}, s}} \subseteq ]C^{w}_k[ \subseteq ]X^{w}_k[$, and that $\overline{]X^{w}_k[} = \mathrm{sp}^{-1}(X^{w}_k) \subseteq ]Y_{w,k}[$. Since $]X^{w}_k[ K_{p,m,b}$ is a finite union of translates of $]X^{w}_k[$, we can take $$Z_2 = (\pi_{HT, K_{p,m,b}}^{tor})^{-1} ( \overline{]X^{w}_k[} K_{p,m,b}).$$ 
 We observe that $Z_1^{c} \cap  (\pi_{HT, K_{p,m,0}}^{tor})^{-1}(]Y_{w,k}[)$ is a non-empty quasi-compact open, call it $W$. We can find a non-empty quasi-compact open $W'$ with the property that $\overline{W} \subset W$. 
We let $$Z' =]C_{w,k}[_{\overline{0}, \overline{0}}$$ and $Z = (\pi_{HT, K_{p,m,0}}^{tor})^{-1}(Z')$. 

\end{proof}

\subsubsection{The finite slope cohomology and the cohomology of dagger spaces} We   construct two exact triangles which give an interpretation of   finite slope cohomology for a given $w \in \WM$ as being quasi-isomorphic to the cone of  a map between finite slope cohomology of certain dagger spaces, or the cone of a map between  cohomology with compact support of dagger spaces. This will be useful for duality and vanishing. 

\begin{prop}\label{prop-dagger1}
Let $K_p=K_{p,m,b}$ for some $m\geq b\geq0$ with $m>0$, and fix $w\in\WM$ and $\kappa \in  X^\star(T^c)^{M_\mu,+}$.
\begin{enumerate} 
\item The Hecke algebra $\mathcal{H}_{p,m,b}^{-}$ acts on 
$\mathrm{R}\Gamma((\pi_{HT, K_p}^{tor})^{-1}(]Y_{w,k}[), \mathcal{V}_\kappa)$ and $\mathrm{R}\Gamma((\pi_{HT, K_p}^{tor})^{-1}(]\cup_{w' >w}Y_{w',k}[), \mathcal{V}_\kappa)$ and the ideal $\mathcal{H}_{p,m,b}^{--}$ acts by potent compact operators.  We have an exact triangle:
$$\mathrm{R}\Gamma_w(K^pK_p, \kappa)^{-,fs}  \rightarrow  \mathrm{R}\Gamma((\pi_{HT, K_p}^{tor})^{-1}(]Y_{w,k}[), \mathcal{V}_\kappa)^{-,fs}  $$ $$ \rightarrow
\mathrm{R}\Gamma((\pi_{HT, K_p}^{tor})^{-1}(]\cup_{w' >w}Y_{w',k}[), \mathcal{V}_\kappa)^{-,fs} \stackrel{+1}\rightarrow$$
\item The Hecke algebra $\mathcal{H}_{p,m,b}^{+}$ acts on 
$\mathrm{R}\Gamma_{(\pi_{HT, K_p}^{tor})^{-1}(\overline{]Y_{w,k}[})}( \mathcal{S}^{tor}_{K^pK_p,\Sigma}, \mathcal{V}_\kappa)$ and $\mathrm{R}\Gamma_{(\pi_{HT, K_p}^{tor})^{-1}(\overline{]\cup_{w' >w}Y_{w',k}[})}( \mathcal{S}^{tor}_{K^pK_p,\Sigma}, \mathcal{V}_\kappa)$ and the ideal $\mathcal{H}_{p,m,b}^{++}$ acts by potent compact operators.  We have an exact triangle:
$$\mathrm{R}\Gamma_{(\pi_{HT, K_p}^{tor})^{-1}(\overline{]\cup_{w' >w}Y_{w',k}[})}( \mathcal{S}^{tor}_{K^pK_p,\Sigma}, \mathcal{V}_\kappa)^{+,fs}   \rightarrow  \mathrm{R}\Gamma_{(\pi_{HT, K_p}^{tor})^{-1}(\overline{]Y_{w,k}[})}( \mathcal{S}^{tor}_{K^pK_p,\Sigma}, \mathcal{V}_\kappa)^{+,fs}$$ $$ \rightarrow
\mathrm{R}\Gamma_w(K^pK_p, \kappa)^{+,fs}  \stackrel{+1}\rightarrow$$
\end{enumerate}
\end{prop}
\begin{proof}  We have an exact triangle: 
$$\mathrm{R}\Gamma_{(\pi_{HT, K_p}^{tor})^{-1}(]C_{w,k}[_{\overline{0}, {0}})}((\pi_{HT, K_p}^{tor})^{-1}(]Y_{w,k}[), \mathcal{V}_\kappa)  \rightarrow  \mathrm{R}\Gamma((\pi_{HT, K_p}^{tor})^{-1}(]Y_{w,k}[), \mathcal{V}_\kappa)  $$ $$ \rightarrow
\mathrm{R}\Gamma((\pi_{HT, K_p}^{tor})^{-1}(]\cup_{w' >w}Y_{w',k}[), \mathcal{V}_\kappa) \stackrel{+1}\rightarrow$$
since  $]Y_{w,k}[ = ] \cup_{w'>w}Y_{w,k}[ \coprod ]C_{w,k}[_{\overline{0}, {0}}$ by lemma \ref{lem-decomposition-cells}. 
We also have an exact triangle 
$$\mathrm{R}\Gamma_{(\pi_{HT, K_p}^{tor})^{-1}(\overline{]\cup_{w' >w}Y_{w',k}[})}( \mathcal{S}^{tor}_{K^pK_p,\Sigma}, \mathcal{V}_\kappa)   \rightarrow  \mathrm{R}\Gamma_{(\pi_{HT, K_p}^{tor})^{-1}(\overline{]Y_{w,k}[})}( \mathcal{S}^{tor}_{K^pK_p,\Sigma}, \mathcal{V}_\kappa)$$ $$ \rightarrow \mathrm{R}\Gamma_{(\pi_{HT, K_p}^{tor})^{-1}(]C_{w,k}[_{{0}, \overline{0}})}((\pi_{HT, K_p}^{tor})^{-1}(]X_{w,k}[), \mathcal{V}_\kappa) 
 \stackrel{+1}\rightarrow$$
since  $\overline{]Y_{w,k}[ }=\overline{ ] \cup_{w'>w}Y_{w,k}[ } \coprod ]C_{w,k}[_{{0}, \overline{0}}$ by lemma \ref{lem-decomposition-cells}, and $]X_{w,k}[$ is an open neighborhood of  $]C_{w,k}[_{{0}, \overline{0}}$ in $\mathcal{FL} \setminus \overline{ ] \cup_{w'>w}Y_{w,k}[ }$. 
It therefore remains to prove that the relevant Hecke algebras act in an equivariant way on the triangle and to take the finite slope part. 
We observe that the cohomologies $\mathrm{R}\Gamma((\pi_{HT, K_p}^{tor})^{-1}(]Y_{w,k}[), \mathcal{V}_\kappa)$ and $\mathrm{R}\Gamma((\pi_{HT, K_p}^{tor})^{-1}( ] \cup_{w'>w}Y_{w',k}[), \mathcal{V}_\kappa)$ are represented by complexes in $\mathcal{K}^{proj}(\mathbf{Ban}(F))$. It follows from lemma \ref{lem-dynamic-corres}, (2) and (7) that $\mathcal{H}_{p,m',b}^{-}$ is acting and that $\mathcal{H}_{p,m',b}^{--}$ is acting via potent compact operators. 
We observe that the cohomologies $\mathrm{R}\Gamma_{(\pi_{HT, K_p}^{tor})^{-1}(\overline{]\cup_{w' >w}Y_{w',k}[})}( \mathcal{S}^{tor}_{K^pK_p,\Sigma}, \mathcal{V}_\kappa)$   and  $\mathrm{R}\Gamma_{(\pi_{HT, K_p}^{tor})^{-1}(\overline{]Y_{w,k}[})}( \mathcal{S}^{tor}_{K^pK_p,\Sigma}, \mathcal{V}_\kappa)$ are represented by complexes in $\mathrm{Pro}_{\N} (\mathcal{K}^{proj}(\mathbf{Ban}(F)))$ and it follows form lemma \ref{lem-dynamic-corres} (2) and (7) that $\mathcal{H}_{p,m',b}^{+}$ is acting and that $\mathcal{H}_{p,m',b}^{++}$ is acting via potent compact operators. 
\end{proof}

\begin{rem} One could also consider two other triangles: one for the $-$ theory, using cohomology with support in the closed subsets $(\pi_{HT, K_p}^{tor})^{-1}(\overline{]\cup_{w'<w}X_{w',k}[})$ and $(\pi_{HT, K_p}^{tor})^{-1}(\overline{]X_{w,k}[})$, and one for the $+$ theory using cohomology of the  open subsets $(\pi_{HT, K_p}^{tor})^{-1}({]\cup_{w'<w}X_{w',k}[})$ and $(\pi_{HT, K_p}^{tor})^{-1}({]X_{w,k}[})$. 
\end{rem}

By remark \ref{rem-dagger}, the cohomologies $$\mathrm{R}\Gamma_{(\pi_{HT, K_p}^{tor})^{-1}(\overline{]\cup_{w' >w}Y_{w',k}[})}( \mathcal{S}^{tor}_{K^pK_p,\Sigma}, \mathcal{V}_\kappa),~\textrm{and}~  \mathrm{R}\Gamma_{(\pi_{HT, K_p}^{tor})^{-1}(\overline{]Y_{w,k}[})}( \mathcal{S}^{tor}_{K^pK_p,\Sigma}, \mathcal{V}_\kappa)$$ are the cohomologies with compact support of the dagger spaces $$(\pi_{HT, K_p}^{tor})^{-1}({]\cup_{w' >w}Y_{w',k}[})^\dag ~\textrm{and}~(\pi_{HT, K_p}^{tor})^{-1}({]Y_{w,k}[})^\dag.$$ Similarly, 
$$ \mathrm{R}\Gamma((\pi_{HT, K_p}^{tor})^{-1}(\overline{]Y_{w,k}[}), \mathcal{V}_\kappa), ~\textrm{and}~ 
\mathrm{R}\Gamma((\pi_{HT, K_p}^{tor})^{-1}(\overline{]\cup_{w' >w}Y_{w',k}[}), \mathcal{V}_\kappa)$$ are the cohomologies of the dagger spaces $$(\pi_{HT, K_p}^{tor})^{-1}({]\cup_{w' >w}Y_{w',k}[})^\dag~\textrm{and}~(\pi_{HT, K_p}^{tor})^{-1}({]Y_{w,k}[})^\dag.$$

\begin{prop}\label{prop-dagger2} The Hecke algebra $\mathcal{H}_{p,m,b}^{-}$ acts on 
$\mathrm{R}\Gamma((\pi_{HT, K_p}^{tor})^{-1}(\overline{]Y_{w,k}[}), \mathcal{V}_\kappa)$ and $\mathrm{R}\Gamma((\pi_{HT, K_p}^{tor})^{-1}(\overline{]\cup_{w' >w}Y_{w',k}[}), \mathcal{V}_\kappa)$ and the ideal $\mathcal{H}_{p,m,b}^{--}$ acts by potent compact operators.
Moreover, the morphisms: 
\begin{eqnarray*}
 \mathrm{R}\Gamma((\pi_{HT, K_p}^{tor})^{-1}(\overline{]Y_{w,k}[}), \mathcal{V}_\kappa) &\rightarrow &\mathrm{R}\Gamma((\pi_{HT, K_p}^{tor})^{-1}({]Y_{w,k}[}), \mathcal{V}_\kappa) \\
\mathrm{R}\Gamma((\pi_{HT, K_p}^{tor})^{-1}(\overline{]\cup_{w' >w}Y_{w',k}[}), \mathcal{V}_\kappa) &\rightarrow& \mathrm{R}\Gamma((\pi_{HT, K_p}^{tor})^{-1}({]\cup_{w' >w}Y_{w',k}[}), \mathcal{V}_\kappa)
\end{eqnarray*}
induce isomorphisms on the finite slope part:
\begin{eqnarray*}
 \mathrm{R}\Gamma((\pi_{HT, K_p}^{tor})^{-1}(\overline{]Y_{w,k}[}), \mathcal{V}_\kappa)^{-,fs} &\stackrel{\sim}\rightarrow &\mathrm{R}\Gamma((\pi_{HT, K_p}^{tor})^{-1}({]Y_{w,k}[}), \mathcal{V}_\kappa)^{-,fs} \\
\mathrm{R}\Gamma((\pi_{HT, K_p}^{tor})^{-1}(\overline{]\cup_{w' >w}Y_{w',k}[}), \mathcal{V}_\kappa)^{-,fs} & \stackrel{\sim}\rightarrow& \mathrm{R}\Gamma((\pi_{HT, K_p}^{tor})^{-1}({]\cup_{w' >w}Y_{w',k}[}), \mathcal{V}_\kappa)^{-,fs}
\end{eqnarray*}
\end{prop}
\begin{proof} The cohomologies $\mathrm{R}\Gamma((\pi_{HT, K_p}^{tor})^{-1}(\overline{]Y_{w,k}[}), \mathcal{V}_\kappa)$ and $\mathrm{R}\Gamma((\pi_{HT, K_p}^{tor})^{-1}(\overline{]\cup_{w' >w}Y_{w',k}[}), \mathcal{V}_\kappa)$ are represented by complexes in $\mathrm{Ind}_{\N}(\mathcal{K}^{proj}(\mathbf{Ban}(F)))$. The Hecke algebra $\mathcal{H}_{p,m',b}^{-}$ acts and the operators are potent compact as a consequence of lemma \ref{lem-dynamic-corres} (2) and (7). The same lemma shows the quasi-isomorphism on the finite slope part. 
\end{proof}

\begin{coro}\label{coro-dagger3} We have an exact triangle:
$$\mathrm{R}\Gamma_w(K^pK_p, \kappa)^{-,fs}  \rightarrow  \mathrm{R}\Gamma((\pi_{HT, K_p}^{tor})^{-1}(\overline{]Y_{w,k}[}), \mathcal{V}_\kappa)^{-,fs}  $$ $$ \rightarrow
\mathrm{R}\Gamma((\pi_{HT, K_p}^{tor})^{-1}(\overline{]\cup_{w' >w}Y_{w',k}[}), \mathcal{V}_\kappa)^{-,fs} \stackrel{+1}\rightarrow$$
\end{coro}
\begin{proof} This is a consequence of propositions \ref{prop-dagger1} and \ref{prop-dagger2}.
\end{proof}
\subsubsection{Change of support condition} It is important to us that the cohomology $\mathrm{R}\Gamma_w(K^pK_p, \kappa)^{\pm,fs}$ can actually be realized as the finite slope part of cohomology groups with different support conditions.  The following definition is motivated by Lemma \ref{lem-complex-frechet} and the discussion of Section \ref{subsection-analytic-continuation}, especially Corollary \ref{coro-infinitediamond}.  We start by fixing an element $t \in T^{++}$ such that $\min(t) \geq 1$. We let $C = \mathrm{max}(t)$. 

\begin{defi}\label{defi-support-condition} Let $m'\geq b\geq0$ with $m'>0$.  A $(+,w, K_{p,m',b})$-allowed support is a pair $(\mathcal{U}, \mathcal{Z})$  where:

\begin{enumerate}
\item $\mathcal{U}$ is an open subset of  $\mathcal{S}_{K^pK_{p,m',b}, \Sigma}^{tor}$ which is a finite union of quasi-Stein open subsets. 
\item $\mathcal{Z}$ is a closed subset of  $\mathcal{S}_{K^pK_{p,m',b}, \Sigma}^{tor}$ whose complement is a finite union of quasi-Stein open subsets.
\item   There exists $m,n,s \in \ZZ_{\geq 0}$ such that:
 $$(\pi_{HT, K_{p, m',b}}^{tor})^{-1}(]C_{w,k}[_{m, \overline{0}} K_{p,m',b} \cap ]C_{w,k}[_{0, \overline{n +s}} K_{p,m',b}) \subseteq \mathcal{Z} \cap \mathcal{U}
  \subseteq (\pi_{HT, K_{p, m',b}}^{tor})^{-1}(]C_{w,k}[_{{0}, C\overline{n}}K_{p,m',b}),$$
 $$(\pi_{HT, K_{p, m',b}}^{tor})^{-1}(]C_{w,k}[_{m+s, \overline{0}} K_{p,m',b} \cap ]C_{w,k}[_{0, \overline{n}} K_{p,m',b}) \subseteq \mathcal{U} \subseteq (\pi_{HT, K_{p, m',b}}^{tor})^{-1}(]C_{w,k}[_{Cm,-1} K_{p,m',b}).$$
 
 \end{enumerate}

Let $m'\geq b\geq0$ with $m'>0$. A $(-,w, K_{p,m',b})$-allowed support is a pair $(\mathcal{U}, \mathcal{Z})$  where:
\begin{enumerate}
\item $\mathcal{U}$ is an open subset of  $\mathcal{S}_{K^pK_{p,m',b}, \Sigma}^{tor}$ which is a finite union of quasi-Stein open subsets. 
\item $\mathcal{Z}$ is a closed subset of  $\mathcal{S}_{K^pK_{p,m',b}, \Sigma}^{tor}$ whose complement is a finite union of quasi-Stein open subsets.
\item There exists $m,n,s \in \ZZ_{\geq 0}$ such that:
$$(\pi_{HT, K_{p,m',b}}^{tor})^{-1}(]C_{w,k}[_{\overline{m+s}, {0} } K_{p,m',b} \cap ]C_{w,k}[_{\overline{0}, n} K_{p,m',b}) \subseteq \mathcal{Z} \cap \mathcal{U} \subseteq (\pi_{HT, K_{p,m',b}}^{tor})^{-1}( ]C_{w,k}[_{C\overline{m}, {0}} K_{p,m',b}),$$
 $$(\pi_{HT, K_{p,m',b}}^{tor})^{-1}(]C_{w,k}[_{\overline{m}, 0} K_{p,m',b} \cap ]C_{w,k}[_{\overline{0}, n+s}  K_{p,m',b} ) \subseteq \mathcal{U} \subseteq (\pi_{HT, K_{p,m',b}}^{tor})^{-1}(]C_{w,k}[_{-1, Cn} K_{p,m',b}).$$
 \end{enumerate}
\end{defi}

 \begin{ex}\label{example-allowed-support} For any $m'\geq b\geq 0$ with $m'>0$ and any $s\geq 0$, the pair $$ \mathcal{U} = (\pi_{HT, K_{p,m',b}}^{tor})^{-1}(]C_{w,k}[_{s,-1} K_{p,m',b}),\quad \mathcal{Z} = (\pi_{HT, K_{p,m',b}}^{tor})^{-1}(]C_{w,k}[_{\overline{0}, \overline{s}} K_{p,m',b})$$ is a $(+,w, K_{p,m',b})$-allowed support (since $]C_{w,k}[_{s,-1} K_{p,m',b} \cap ]C_{w,k}[_{\overline{0}, \overline{s}} K_{p,m',b} \subseteq ]X_{w,k}[ \cap \overline{]Y_{w,k}[} = ]C_{w,k}[_{0,\overline{0}}$).  Similarly, the pair $$\mathcal{U} = (\pi_{HT, K_{p,m',b}}^{tor})^{-1}(]C_{w,k}[_{-1,s} K_{p,m',b}),\quad \mathcal{Z} = (\pi_{HT, K_{p,m',b}}^{tor})^{-1}(]C_{w,k}[_{\overline{s}, \overline{0}} K_{p,m',b})$$ is a $(-,w, K_{p,m',b})$-allowed support.
 \end{ex}
 
\begin{thm}\label{thm-finiteslopesupport}Let $m'\geq b\geq0$ with $m'>0$, and fix $w\in\WM$ and $\kappa \in  X^\star(T^c)^{M_\mu,+}$. \begin{enumerate} 
\item  Let $(\mathcal{U}, \mathcal{Z})$ be a $(+,w,K_{p,m',b})$-allowed support condition.  Then $\mathrm{R}\Gamma_{\mathcal{Z} \cap \mathcal{U}}(\mathcal{U}, \mathcal{V}_\kappa)$ and $\mathrm{R}\Gamma_{\mathcal{Z} \cap \mathcal{U}}(\mathcal{U}, \mathcal{V}_\kappa(-D))$ are objects of $\mathrm{Pro}_{\N}(\mathcal{K}^{proj}(\mathbf{Ban}(F)))$ and carry a canonically defined, potent compact action of $T^s$ where $T = [K_{p,m',b} t K_{p,m',b}]$ be the corresponding Hecke operator. .  Moreover there are canonical isomorphisms
$$\mathrm{R}\Gamma_w(K^pK_{p,m',b},\kappa)^{+,fs}\simeq \mathrm{R}\Gamma_{\mathcal{Z} \cap \mathcal{U}}(\mathcal{U}, \mathcal{V}_\kappa)^{T^s-fs}$$
and
$$\mathrm{R}\Gamma_w(K^pK_{p,m',b},\kappa,cusp)^{+,fs}\simeq \mathrm{R}\Gamma_{\mathcal{Z} \cap \mathcal{U}}(\mathcal{U}, \mathcal{V}_\kappa(-D))^{T^s-fs}.$$

\item  Let $(\mathcal{U}, \mathcal{Z})$ be a $(-,w,K_{p,m',b})$-allowed support condition.  Then $\mathrm{R}\Gamma_{\mathcal{Z} \cap \mathcal{U}}(\mathcal{U}, \mathcal{V}_\kappa)$ and $\mathrm{R}\Gamma_{\mathcal{Z} \cap \mathcal{U}}(\mathcal{U}, \mathcal{V}_\kappa(-D))$ are objects of $\mathrm{Pro}_{\N}(\mathcal{K}^{proj}(\mathbf{Ban}(F)))$ and carry a canonically defined, potent compact action of $T^s$ where $T = [K_{p,m',b} t K_{p,m',b}]$.  Moreover there are canonical isomorphisms
$$\mathrm{R}\Gamma_w(K^pK_{p,m',b},\kappa)^{-,fs}\simeq \mathrm{R}\Gamma_{\mathcal{Z} \cap \mathcal{U}}(\mathcal{U}, \mathcal{V}_\kappa)^{T^s-fs}$$
and
$$\mathrm{R}\Gamma_w(K^pK_{p,m',b},\kappa,cusp)^{-,fs}\simeq \mathrm{R}\Gamma_{\mathcal{Z} \cap \mathcal{U}}(\mathcal{U}, \mathcal{V}_\kappa(-D))^{T^s-fs}.$$
 \end{enumerate}
\end{thm}
\begin{proof}
We will only prove the first point for non cuspidal cohomology.  It follows from lemma \ref{lem-complex-frechet} that $\mathrm{R}\Gamma_{\mathcal{Z} \cap \mathcal{U}}(\mathcal{U}, \mathcal{V}_\kappa)$ is an object of $\mathrm{Pro}_{\N}(\mathcal{K}^{proj}(\mathbf{Ban}(F)))$.

For the rest of the statement we will use the infinite diamond construction of section \ref{section-infinite-diamond} and corollary \ref{coro-infinitediamond}.  For $m,n\in\ZZ_{\geq 0}$ we let $U_m = (\pi^{tor}_{HT, K_{p,m',b}})^{-1} (T^m(]X_{w,k}[))$ and $Z_n = (\pi^{tor}_{HT, K_{p,m',b}})^{-1}((T^t)^n (\overline{]Y_{w,k}[}))$.  Since we have already checked in the proof of theorem \ref{thm-finiteslopecoho} that the ``restriction-corestriction'' map
$$\mathrm{R\Gamma}_{U_0 \cap Z_1} ( U_0, \mathcal{V}_\kappa) \rightarrow \mathrm{R\Gamma}_{U_1 \cap Z_0} ( U_1, \mathcal{V}_\kappa)$$
is compact, it follows that the conclusion of corollary \ref{coro-infinitediamond} hold.

By lemma \ref{lem-dynamic-corres} we have:
\begin{eqnarray*} U_m \cap Z_{n+s}& \subseteq&(\pi_{HT, K_{p, m',b}}^{tor})^{-1}( ]C_{w,k}[_{{m}, \overline{0}} K_{p, m',b}\cap ]C_{w,k}[_{0, \overline{n+s}} K_{p, m',b})\\
 U_{m+s} \cap Z_{n} & \subseteq &(\pi^{tor}_{HT, K_{p, m',b}})^{-1}( ]C_{w,k}[_{{m+s}, \overline{0}} K_{p, m',b}\cap ]C_{w,k}[_{0, \overline{n}} K_{p, m',b}) \\
 (\pi_{HT, K_{p, m',b}}^{tor})^{-1}( ]C_{w,k}[_{{m}C, -1} K_{p, m',b}) & \subseteq & U_m \\
(\pi_{HT, K_{p, m',b}}^{tor})^{-1}( ]C_{w,k}[_{0, \overline{n}C} K_{p, m',b}) & \subseteq & Z_n.
 \end{eqnarray*}
 
It follows from corollary \ref{coro-infinitediamond} that if $(\mathcal{U}, \mathcal{Z})$ is a $(+,w,K_{p,m',b})$-allowed support condition, then $\mathrm{R}\Gamma_{\mathcal{U} \cap \mathcal{Z}}( \mathcal{U}, \mathcal{V}_\kappa)$ carries a canonical action of $T^s$ and we have canonical quasi-isomorphisms
$$\mathrm{R}\Gamma_{\mathcal{U} \cap \mathcal{Z}}( \mathcal{U}, \mathcal{V}_\kappa)^{T^s-fs}\simeq\mathrm{R}\Gamma_{U\cap Z}(U,\mathcal{V}_\kappa)^{T-fs}\simeq\mathrm{R}\Gamma_w(K^pK_{p,m',b},\kappa)^{+,fs}.$$
\end{proof}
 
\subsubsection{Change of level}Now we investigate how the finite slope cohomologies $\mathrm{R}\Gamma_w(K^pK_p,\kappa)^{\pm,fs}$ and $\mathrm{R}\Gamma_w(K^pK_p,\kappa,cusp)^{\pm,fs}$ vary with the level $K_p$.
\begin{thm}\label{thm-change-of-level}
\begin{enumerate}
\item For all $w\in\WM$ and all $m' \geq m\geq b\geq0$ with $m>0$, the pullback map 
$$ \mathrm{R}\Gamma_w(K^pK_{p,m,b},\kappa)^{+, fs}  \rightarrow  \mathrm{R}\Gamma_w(K^pK_{p,m',b},\kappa)^{+, fs}$$ 
and the trace map
$$ \mathrm{R}\Gamma_w(K^pK_{p,m',b},\kappa)^{-, fs}  \rightarrow  \mathrm{R}\Gamma_w(K^pK_{p,m,b},\kappa)^{-, fs}$$ are quasi-isomorphisms, compatible with the action of $\qq[T(\qq_p)/T_b]$, and the same statements are true for cuspidal cohomology.
\item For all $w\in\WM$ and all $m\geq b'\geq b\geq 0$ with $m>0$, the pullback map
$$ \mathrm{R}\Gamma_w(K^pK_{p,m,b},\kappa)^{+, fs}  \rightarrow  (\mathrm{R}\Gamma_w(K^pK_{p,m,b'},\kappa)^{+, fs})^{T_b/T_{b'}}$$
and the trace map
$$ (\mathrm{R}\Gamma_w(K^pK_{p,m,b'},\kappa)^{-, fs})^{T_b/T_{b'}}  \rightarrow  \mathrm{R}\Gamma_w(K^pK_{p,m,b},\kappa)^{-, fs}$$ are quasi-isomorphisms, compatible with the action of $\qq[T(\qq_p)/T_b]$, and the same statements are true for cuspidal cohomology. 
\end{enumerate}
\end{thm}

\begin{proof}
This follows from lemma \ref{lemma-relations-Hecke-algebra}.
\end{proof}

As a result of the theorem, we can let $\mathrm{R}\Gamma_w(K^p,\kappa,\chi)^{\pm,fs}$ and $\mathrm{R}\Gamma_w(K^p,\kappa,\chi,cusp)^{\pm,fs}$ denote $\mathrm{R}\Gamma_w(K^pK_{p,m,b},\kappa)^{\pm,fs}[\chi]$ and $\mathrm{R}\Gamma_w(K^pK_{p,m,b},\kappa,cusp)^{\pm,fs}[\chi]$ for any $m\geq b\geq\mathrm{cond}(\chi)$ with $m>0$, as these spaces have been canonically identified.

\subsection{The spectral sequence associated with the Bruhat stratification}\label{section-spectral-sequences}
Recall that there is a length function $\ell : \WM \rightarrow [0, d]$ where $d = \dim \mathcal{FL} = \dim S_{K}$ with the property that $\ell (w) = \dim C_w$. We let $\ell_+(w) = \ell(w)$ and $\ell_-(w) = d-\ell(w)$.  

The main result of this section is the following theorem:

\begin{thm}\label{thm-spectral-sequence} Let $\kappa \in X^\star(T^c)^{M_{\mu},+}$ be a weight and let $\chi:T(\ZZ_p)\to {\overline{F}}^\times$ be a finite order character.  For a choice of $+$ or $-$, there is a $\mathcal{H}_{p,m,b}^{\pm}$-equivariant spectral sequence  $\mathbf{E}^{p,q}(K^p, \kappa,\chi)^\pm$ converging to classical finite slope cohomology $ \HH^{p+q}(K^p,\kappa,\chi)^{\pm,fs}$,
such that $$\mathbf{E}_1^{p,q}(K^p, \kappa,\chi)^{\pm} = \oplus_{w \in \WM, \ell_\pm(w) = p} \HH^{p+q}_w( K^p, \kappa,\chi)^{\pm,fs}.$$
There are also spectral sequences $\mathbf{E}^{p,q}(K^p, \kappa,\chi,cusp)^\pm$ converging to $ \HH^{p+q}(K^p,\kappa,\chi,cusp)^{\pm,fs}$ such that $$\mathbf{E}_1^{p,q}(K^p, \kappa, \chi,cusp)^{\pm} = \oplus_{w \in \WM, \ell_\pm(w) = p} \HH^{p+q}_w( K^p, \kappa,\chi, cusp)^{\pm,fs}.$$
\end{thm}
\subsubsection{Construction of a filtration}\label{sect-construction-filtration} 
We consider the following two  stratifications of the special fiber of the flag variety $FL_k $, the first one is by open subsets:
$$ \{ FL_{k}^{\geq r} = \cup_{\ell(w) \geq r} C_{w,k}\}_{0 \leq r \leq d} $$
and the second one is by closed subsets:  $$ \{ FL_{k}^{\leq_{d-r}} = \cup_{\ell(w) \leq d-r} C_{w,k}\}_{0 \leq r \leq d}.$$

We then define $Z^+_{r}=\overline{] FL_{k}^{\geq r}  [}$ and $Z^-_{r}=\overline{] FL_{k}^{\leq d- r}  [}$. This gives two filtrations  $$\mathcal{FL} = Z^\pm_0 \supset Z^\pm_1 \supset \cdots \supset Z^\pm_d \supset Z^\pm_{d+1} = \emptyset$$ by closed subspaces invariant under the Iwahori $\mathrm{Iw}$.

Now let $K_p=K_{p,m,b}$ for some $m\geq b\geq0$ with $m>0$.  We can consider the pullback of these filtrations by $\pi^{tor}_{HT, K_{p}}$ to get two filtrations $\mathcal{S}^{tor} _{K, \Sigma} = \mathcal{Z}^\pm_0 \supset \mathcal{Z}^\pm_1 \supset \cdots \supset \mathcal{Z}^\pm_d \supset \mathcal{Z}^\pm_{d+1} = \emptyset$ by closed subspaces.

For any weight $\kappa$, we can consider the associated spectral sequence (see section \ref{section-spectral-sequence}): 
$$ \HH^{p+q}_{\mathcal{Z}^\pm_p/\mathcal{Z}^\pm_{p+1}} ( \mathcal{S}^{tor}_{K, \Sigma}, \mathcal{V}_{\kappa}) \Rightarrow  \HH^{p+q} ( \mathcal{S}^{tor}_{K, \Sigma}, \mathcal{V}_{\kappa}).$$

By definition and lemma \ref{lem-writing-intersections},

$$  \HH^{\star}_{\mathcal{Z}^+_p/\mathcal{Z}^+_{p+1}} ( \mathcal{S}^{tor}_{K^pK_{p}}, \mathcal{V}_{\kappa})
= $$
$$\HH^\star_{(\pi_{HT, K_{p}}^{tor})^{-1}(\overline{] \cup_{\ell(w) \geq p} C_w[} \cap { ] \cup_{\ell(w) \leq p} C_w[})} ( (\pi_{HT, K_{p}}^{tor})^{-1}( { ] \cup_{\ell(w) \leq p} C_w[)} , \mathcal{V}_{\kappa}).$$

and $$  \HH^{\star}_{\mathcal{Z}^-_p/\mathcal{Z}^-_{p+1}} ( \mathcal{S}^{tor}_{K^pK_{p}}, \mathcal{V}_{\kappa})
= $$
$$\HH^\star_{(\pi_{HT, K_{p}}^{tor})^{-1}(\overline{] \cup_{\ell(w) \leq d-p} C_w[} \cap { ] \cup_{\ell(w) \geq d- p} C_w[})} ( (\pi_{HT, K_{p}}^{tor})^{-1}({ ] \cup_{\ell(w) \geq d- p} C_w[)} , \mathcal{V}_{\kappa}).$$

We also have a cuspidal version $$ \HH^{p+q}_{\mathcal{Z}^\pm_p/\mathcal{Z}^\pm_{p+1}} ( \mathcal{S}^{tor}_{K, \Sigma}, \mathcal{V}_{\kappa}(-D)) \Rightarrow  \HH^{p+q} ( \mathcal{S}^{tor}_{K, \Sigma}, \mathcal{V}_{\kappa}(-D)).$$

We relate the $\mathbf{E}_1$ pages of these spectral sequences to the overconvergent cohomologies considered in the previous section.

\begin{lem} For all $p$,
$$\mathrm{R}\Gamma_{\mathcal{Z}^+_p/\mathcal{Z}^+_{p+1}} ( \mathcal{S}^{tor}_{K^pK_{p},\Sigma}, \mathcal{V}_{\kappa}) = $$ $$ \oplus_{w \in \WM,~\ell(w) = p} \mathrm{R}\Gamma_{(\pi_{HT, K_{p}}^{tor})^{-1}(]X_{w,k}[ \cap \overline{]Y_{w,k}[})} ( (\pi_{HT, K_{p}}^{tor})^{-1}(]X_{w,k}[), \mathcal{V}_{\kappa}) $$
$$\mathrm{R}\Gamma_{\mathcal{Z}^-_p/\mathcal{Z}^-_{p+1}} ( \mathcal{S}^{tor}_{K^pK_{p},\Sigma}, \mathcal{V}_{\kappa}) = $$ $$\oplus_{w \in \WM,~\ell(w) = d- p} \mathrm{R}\Gamma_{(\pi_{HT, K_{p}}^{tor})^{-1}(]Y_{w,k}[ \cap \overline{]X_{w,k}[})} ( (\pi_{HT, K_{p}}^{tor})^{-1}(]Y_{w,k}[), \mathcal{V}_{\kappa}) $$
We have similar results for cuspidal cohomology. 
\end{lem}
\begin{proof} We have $$\overline{] \cup_{\ell(w) \geq p} C_w[} \cap { ] \cup_{\ell(w) \leq p} C_w[} = \cup_{w, \ell(w)=p} 
\overline{]Y_{w,k}[} \cap ]X_{w,k}[$$ and $$\overline{] \cup_{\ell(w) \leq d-p} C_w[} \cap { ] \cup_{\ell(w) \geq d- p} C_w[} = \cup_{w, \ell(w)=d-p} 
{]Y_{w,k}[} \cap \overline{ ]X_{w,k}[}$$ by lemma \ref{topological-lemma-easy}. The conclusion follows from lemma \ref{lem-direct-sum-complex-support}.
\end{proof}

\begin{lem}\label{lem-hecke-spectral-sequence}
For a choice of $+$ or $-$, $\mathcal{H}_{p,m,b}^{\pm}$ acts on $\mathrm{R}\Gamma_{\mathcal{Z}^\pm_p/\mathcal{Z}^\pm_{p+1}} ( \mathcal{S}^{tor}_{K^pK_{p},\Sigma}, \mathcal{V}_{\kappa})$, and the spectral sequence 
$$\HH^{p+q}_{\mathcal{Z}^\pm_p/\mathcal{Z}^\pm_{p+1}} ( \mathcal{S}^{tor}_{K, \Sigma}, \mathcal{V}_{\kappa}) \Rightarrow  \HH^{p+q} ( \mathcal{S}^{tor}_{K, \Sigma}, \mathcal{V}_{\kappa})$$ is $\mathcal{H}_{p,m,b}^{\pm}$-equivariant.  The same result holds for cuspidal cohomology.
\end{lem}
\begin{proof} Easy and left to the reader.
\end{proof}

\subsubsection{Proof of Theorem \ref{thm-spectral-sequence}}  For a choice of $\pm$, we have constructed a spectral sequence 
$$ \HH^{p+q}_{\mathcal{Z}^\pm_p/\mathcal{Z}^\pm_{p+1}} ( \mathcal{S}^{tor}_{K, \Sigma}, \mathcal{V}_{\kappa}) \Rightarrow  \HH^{p+q} ( \mathcal{S}^{tor}_{K, \Sigma}, \mathcal{V}_{\kappa}).$$
The Hecke algebra $\mathcal{H}_{p,m,b}^{\pm}$ acts on this spectral sequence and it makes sense to take the finite slope part by lemma \ref{lem-hecke-spectral-sequence}. Applying the finite slope projector, we obtain the spectral sequence of the theorem. 

\subsubsection{Cousin complexes}\label{subsection-Cousin}  We now extract from the spectral sequence certain complexes that play a prominent role, in particular in view of  corollaries   \ref{coro-concentration-compact} and \ref{coro-concentration-interior}. 
We let $\mathcal{C}ous (K^p, \kappa, \chi)^{\pm}$ be the complex $\mathbf{E}_1^{\bullet,0}(K^p, \kappa,\chi)^{\pm}$ ($w_0^M$ is the longest element of $\WM$): 
$$\HH^0_{Id/w_0^M}( K^p, \kappa,\chi)^{\pm,fs} \rightarrow \oplus_{w \in \WM, \ell_\pm(w) = 1} \HH^{1}_w( K^p, \kappa,\chi)^{\pm,fs} \rightarrow $$ $$ \oplus_{w \in \WM, \ell_\pm(w) = 2} \HH^{2}_w( K^p, \kappa,\chi)^{\pm,fs}   \rightarrow \cdots  \rightarrow \HH^d_{w_0^M/Id}( K^p, \kappa,\chi)^{\pm,fs}$$
and we let $\mathcal{C}ous (K^p, \kappa, \chi, cusp )^{\pm}$ be the complex $\mathbf{E}_1^{\bullet,0}(K^p, \kappa,\chi,cusp)^{\pm}$: 
$$\HH^0_{Id/w_0^M}( K^p, \kappa,\chi, cusp)^{\pm,fs} \rightarrow \oplus_{w \in \WM, \ell_\pm(w) = 1} \HH^{1}_w( K^p, \kappa,\chi, cusp)^{\pm,fs} \rightarrow $$ $$ \oplus_{w \in \WM, \ell_\pm(w) = 2} \HH^{2}_w( K^p, \kappa,\chi, cusp)^{\pm,fs}   \rightarrow \cdots  \rightarrow \HH^d_{w_0^M/Id}( K^p, \kappa,\chi,cusp)^{\pm,fs}$$
 
\subsection{Cohomological vanishing}  The following vanishing theorem is crucial in order to study the spectral sequence.
\begin{thm} \label{theorem-coho-van} The cohomology complex $\mathrm{R}\Gamma_w(K^p, \kappa, \chi, cusp)^{\pm,fs}$ has amplitude  $[0, \ell_\pm(w)]$. 
\end{thm}
\begin{proof} We only give the argument for the $+$ case, as the $-$ case follows with minor modifications. Let $b$ be the conductor of $\chi$.  We can realize $\mathrm{R}\Gamma_w(K^p, \kappa, cusp)^{+,fs}$ as the $\chi$-isotypic part of the finite slope part of 
$$ \mathrm{R}\Gamma_{(\pi_{HT, K_{p, m',b}}^{tor})^{-1}(]C_{w,k}[_{s,-1}K_{p,m',b} \cap C_{w,k}[_{\overline{0},\overline{s}}K_{p,m',b} )}((\pi^{tor}_{HT,K_{p,m',b}})^{-1}(]C_{w,k}[_{s,-1}K_{p,m',b}), \mathcal{V}_{\kappa}(-D))$$
for any $s \geq 0$ and $m' \geq b$, $m' >0$ by example \ref{example-allowed-support} and theorem \ref{thm-finiteslopesupport}. 

We fix $s$ large enough so that $\pi_{HT, K_{p,m',b}}^{-1}(]C_{w,k}[_{s,s-1})$ is quasi-Stein in the minimal compactification.  We also fix $m' = s$. To simplify notations, we let $K_p = K_{p,m',b}$. 

We observe that under these assumptions, $]C_{w,k}[_{s,-1}K_{p} \cap ]C_{w,k}[_{\overline{0},\overline{s}}K_{p} = ]C_{w,k}[_{s,\overline{s}} K_{p}$ by lemma \ref{lem-nice-intersections}, and therefore the above cohomology writes 
$$ \mathrm{R}\Gamma_{(\pi^{tor}_{HT,K_{p}})^{-1}(]C_{w,k}[_{s,\overline{s}}K_{p})}((\pi^{tor}_{HT,K_{p}})^{-1}(]C_{w,k}[_{s,s-1}K_{p}), \mathcal{V}_{\kappa}(-D)).$$

 We shall prove that $$ \mathrm{R}\Gamma_{(\pi_{HT, K_{p}}^{tor})^{-1}(]C_{w,k}[_{s,\overline{s}}K_{p})}((\pi^{tor}_{HT,K_{p}})^{-1}(]C_{w,k}[_{s,s-1}K_{p}), \mathcal{V}_{\kappa}(-D))$$ has cohomological amplitude $[0, \ell_+(w)]$. 
 
 Projecting via $\pi_{K^pK_p,\Sigma} : \mathcal{S}_{K^pK_p, \Sigma}^{tor} \rightarrow \mathcal{S}_{K^pK_p}^\star$ and using theorem \ref{vanishingtominimal}, this cohomology is quasi-isomorphic to 
 
 $ \mathrm{R}\Gamma_{(\pi_{HT, K_{p}})^{-1}(]C_{w,k}[_{s,\overline{s}}K_{p})}((\pi_{HT,K_{p}})^{-1}(]C_{w,k}[_{s,s-1}K_{p}), (\pi_{K^pK_p,\Sigma})_\star \mathcal{V}_{\kappa}(-D))$
 
 Le $K'_p$ be the principal level $m'$ congruence  subgroup and let $ \pi_{K'_p, K_p} : \mathcal{S}_{K^pK'_p}^\star \rightarrow \mathcal{S}_{K^pK_p}^\star$. Since  
 $$ \mathrm{R}\Gamma_{(\pi_{HT, K_{p}})^{-1}(]C_{w,k}[_{s,\overline{s}}K_{p})}((\pi^{tor}_{HT,K_{p}})^{-1}(]C_{w,k}[_{s,s-1}K_{p}),  (\pi_{K^pK_p,\Sigma})_\star\mathcal{V}_{\kappa}(-D)) =$$ $$ \mathrm{R}\Gamma\big( K_p/K'_p, \mathrm{R}\Gamma_{(\pi_{HT, K'_{p}})^{-1}(]C_{w,k}[_{s,\overline{s}}K_{p})}((\pi_{HT,K'_{p}})^{-1}(]C_{w,k}[_{s,s-1}K_{p}), \pi_{K'_p, K_p} ^\star (\pi_{K^pK_p,\Sigma})_\star \mathcal{V}_{\kappa}(-D))\big)$$
 
 It will suffice to prove that $$\mathrm{R}\Gamma_{(\pi_{HT, K'_{p}})^{-1}(]C_{w,k}[_{s,\overline{s}}K_{p})}((\pi_{HT,K'_{p}})^{-1}(]C_{w,k}[_{s,s-1}K_{p}), \pi_{K'_p, K_p} ^\star (\pi_{K^pK_p,\Sigma})_\star \mathcal{V}_{\kappa}(-D))$$ has cohomological amplitude in $[0, \ell_+(w)]$. 
 
 Since $]C_{w,k}[_{s,\overline{s}}K_{p}$ is a finite disjoint union of translates of $\{]C_{w,k}[_{s,\overline{s}} k_i\}$ for elements $k_1, \cdots, k_n$ by lemma \ref{lem-nice-disjoint-union},  we are left to prove that $$\mathrm{R}\Gamma_{(\pi_{HT, K'_{p}})^{-1}(]C_{w,k}[_{s,\overline{s}}k_i)}((\pi_{HT,K'_{p}})^{-1}(]C_{w,k}[_{s,s-1}k_i), \pi_{K'_p, K_p} ^\star (\pi_{K^pK_p,\Sigma})_\star \mathcal{V}_{\kappa}(-D) )$$  for $1 \leq i \leq n$ has cohomological amplitude in $[0, \ell_+(w)]$. Also, using the action of $K_p$ it suffices to treat the case $k_i = 1$. 
The cohomology fits in a distinguished triangle:

  $$ \mathrm{R}\Gamma_{(\pi_{HT,K'_{p}})^{-1}(]C_{w,k}[_{s,\overline{s}})}((\pi_{HT,K'_{p}})^{-1}(]C_{w,k}[_{s,s-1}), \pi_{K'_p, K_p} ^\star (\pi_{K^pK_p,\Sigma})_\star \mathcal{V}_{\kappa}(-D) ) \rightarrow $$ $$\mathrm{R}\Gamma((\pi_{HT,K'_{p}})^{-1}({]C_{w,k}[_{s,s-1}}), \pi_{K'_p, K_p} ^\star (\pi_{K^pK_p,\Sigma})_\star \mathcal{V}_{\kappa}(-D) ) \rightarrow  $$ $$\mathrm{R}\Gamma((\pi_{HT,K'_{p}})^{-1}({]C_{w,k}[_{s,s-1}} \setminus {]C_{w,k}[_{s, \overline{s}}}), \pi_{K'_p, K_p} ^\star (\pi_{K^pK_p,\Sigma})_\star \mathcal{V}_{\kappa}(-D) ) \stackrel{+1}\rightarrow$$
  
  Since $\pi_{HT, K'_p}^{-1}({]C_{w,k}[_{s,s-1}})$ is quasi-Stein in the minimal compactification $\mathcal{S}^\star_{K^pK'_p}$ we have  that $$ \mathrm{R}\Gamma((\pi_{HT,K'_{p}})^{-1}({]C_{w,k}[_{s,s-1}}), \pi_{K'_p, K_p} ^\star (\pi_{K^pK_p,\Sigma})_\star \mathcal{V}_{\kappa}(-D) )$$ is concentrated in degree $0$. 

We will now prove that $\pi_{HT,K'_p}^{-1}({]C_{w,k}[_{s,s-1}} \setminus {]C_{w,k}[_{s,\overline{s}}})$ admits a covering by $\ell_+(w)$ acyclic spaces. This will show that 
$$\mathrm{R}\Gamma((\pi^{tor}_{HT,K'_{p}})^{-1}({]C_{w,k}[_{s,s-1}} \setminus {]C_{w,k}[_{s, \overline{s}}}), \pi_{K'_p, K_p} ^\star (\pi_{K^pK_p,\Sigma})_\star \mathcal{V}_{\kappa}(-D) ) $$
has only cohomology in degree $0$ to $\ell_+(w)-1$ and the theorem will follow. 

We recall from corollary \ref{coro-description-bruhat-cell} the isomorphism: 

\begin{eqnarray*}
 \prod_{\alpha \in  (w^{-1} \Phi^{-,M}) \cap \Phi^+} \mathcal{U}_{\alpha} \times \prod_{ \alpha \in  (w^{-1} \Phi^{-,M}) \cap \Phi^-} \mathcal{U}_{\alpha}^o &\rightarrow&   ]C_{w,k}[_{\mathcal{FL}}  \\
 (u_{\alpha})_{\alpha \in  w^{-1}\Phi^{-,M}} & \mapsto & w\prod_\alpha u_{\alpha}
 \end{eqnarray*}
 
 Let us fix coordinates $1+ u_{\alpha}$ on each of the one parameter groups. For these coordinates, the equation of ${]C_{w,k}[_{s,s-1}} \setminus ]C_{w,k}[_{s,\overline{s}}$ is:
 
\begin{itemize}
\item $\forall \alpha \in (w^{-1} \Phi^{-,M}) \cap \Phi^+$, $ \vert u_\alpha \vert \leq \vert p^{s-1} \vert$, 
\item $\forall \alpha \in  (w^{-1} \Phi^{-,M}) \cap \Phi^-$, $\exists \epsilon >0$,  $ \vert u_\alpha \vert \leq \vert p^{s+\epsilon}\vert$, 
\item $\exists \alpha \in  (w^{-1} \Phi^{-,M}) \cap \Phi^+$, $\exists \nu >0$,  $ \vert u_\alpha \vert \geq \vert p^{s-\nu}\vert$.
\end{itemize}

Since $\sharp  (w^{-1} \Phi^{-,M}) \cap \Phi^+ = \ell(w) = \ell_+(w)$, we deduce that ${]C_{w,k}[_{s,s-1}} \setminus {]C_{w,k}[_{s,\overline{s}}}$ is indeed covered by $\ell(w)$ acyclic spaces and the same holds for its pre-image by $\pi_{HT,K'_p}$. 
\end{proof}

\begin{prop}\label{prop-deg-spect-sequ1} For the spectral sequence $\mathbf{E}^{p,q}(K^p, \kappa,\chi,cusp)^\pm$ converging to $ \HH^{p+q}(K^p,\kappa,\chi,cusp)^{\pm,fs}$  we have $\mathbf{E}_1^{p,q}(K^p, \kappa,\chi,cusp)^\pm = 0$ if  $q >0$.  In particular $\mathbf{E}_\infty^{p,q}(K^p, \kappa,\chi,cusp)^\pm = \mathrm{Gr}^p(  \HH^{p+q}(K^p,\kappa,\chi,cusp)^{\pm,fs}) =0$ if $p< p+q$. 
\end{prop}

\begin{proof}  This follows from  theorem \ref{theorem-coho-van}.  The picture of the first page of the spectral sequence is as follows ($w_0^M$ is the longest element of $\WM$):

\begin{eqnarray*} \tiny{
\xymatrix{  \HH^0_{Id/w_0^M}( K^p, \kappa,\chi, cusp)^{\pm,fs}  & \oplus_{w \in \WM, \ell_\pm(w) = 1} \HH^{1}_w( K^p, \kappa,\chi, cusp)^{\pm,fs} & \oplus_{w \in \WM, \ell_\pm(w) = 2} \HH^{2}_w( K^p, \kappa,\chi, cusp)^{\pm,fs}  & \cdots \\
& \oplus_{w \in \WM, \ell_\pm(w) = 1} \HH^{0}_w( K^p, \kappa,\chi, cusp)^{\pm,fs}  & \oplus_{w \in \WM, \ell_\pm(w) = 2} \HH^{1}_w( K^p, \kappa,\chi, cusp)^{\pm,fs}  & \cdots \\
&& \oplus_{w \in \WM, \ell_\pm(w) = 2} \HH^{0}_w( K^p, \kappa,\chi, cusp)^{\pm,fs}  & \cdots }}
\end{eqnarray*}
\end{proof}

\subsection{Duality} In this section we investigate Serre duality on overconvergent cohomologies. 

\begin{thm}\label{thm-pairing-over-coho} For all $w \in \WM$,  there is a  pairing: 
$$\langle, \rangle :  \HH^{i}_w(K^p, \kappa, \chi, cusp)^{\pm,fs} \times \HH^{d-i}_w(K^p, -2\rho_{nc}- w_{0,M} \kappa, \chi^{-1})^{\mp,fs} \rightarrow \overline{F}$$
 such that for $t \in T^{\pm}$ we have $\langle t -,- \rangle = \langle  -,t^{-1}- \rangle $. This pairing induces a pairing between the spectral sequences:
$$\langle, \rangle_{p,q,r} :  \mathbf{E}^{p,q}_r(K^p, \kappa,\chi,cusp)^\pm \times \mathbf{E}^{d-p,-q}_r(K^p, -2\rho_{nc}- w_{0,M} \kappa,\chi^{-1})^\mp \rightarrow \overline{F}$$
On the abutment of the spectral sequence the pairing $\langle, \rangle_{p,q,\infty}$ is induced by the perfect Serre duality pairing:
$$\HH^{p+q}(K^p,\kappa,\chi,cusp)^{\pm,fs} \times \HH^{d-p-q}(K^p,-2\rho_{nc}- w_{0,M} \kappa, \chi^{-1})^{\mp,fs} \rightarrow \overline{F}.$$
\end{thm}
\begin{proof} We construct the pairing $$\langle, \rangle :  \HH^{i}_w(K^p, \kappa, \chi, cusp)^{+,fs} \times \HH^{d-i}_w(K^p, -2\rho_{nc}- w_{0,M} \kappa, \chi^{-1})^{-,fs} \rightarrow \overline{F}.$$  Choose $b$ so that $\chi$ is trivial on $T_b$.  We can realize $\mathrm{R}\Gamma_w(K^p, \kappa, \chi, cusp)^{+,fs}$ as the $\chi$-isotypic part of the finite slope part of 
$$ \mathrm{R}\Gamma_{(\pi_{HT, K_{p, m',b}}^{tor})^{-1}(]C_{w,k}[_{s,\overline{s}}K_{p,m',b})}((\pi^{tor}_{HT,K_{p,m',b}})^{-1}(]C_{w,k}[_{s,-1}K_{p,m',b}), \mathcal{V}_{\kappa}(-D))$$
for any  $s > 0$ and $m' \geq b$, $m' >s$.  by example \ref{example-allowed-support} and theorem \ref{thm-finiteslopesupport}. 

We can realize $\mathrm{R}\Gamma_w(K^p, -2\rho_{nc}- w_{0,M} \kappa, \chi^{-1} )^{-,fs}$ as the $\chi^{-1}$-isotypic part of the finite slope part of 
$$ \mathrm{R}\Gamma_{(\pi_{HT, K_{p, m',b}}^{tor})^{-1}(]C_{w,k}[_{\overline{s+1},s-1}K_{p,m',b})}((\pi^{tor}_{HT,K_{p,m',b}})^{-1}(]C_{w,k}[_{-1,s-1}K_{p,m',b}), \mathcal{V}_{-2\rho_{nc}- w_{0,M} \kappa}).$$

We have a cup-product by proposition \ref{prop-construction-cuprod}: 
$$ \HH^i_{(\pi_{HT, K_{p, m',b}}^{tor})^{-1}(]C_{w,k}[_{s,\overline{s}}K_{p,m',b})}((\pi^{tor}_{HT,K_{p,m',b}})^{-1}(]C_{w,k}[_{s,-1}K_{p,m',b}), \mathcal{V}_{\kappa}(-D))   \times $$
$$  
\HH^{d-i}_{(\pi_{HT, K_{p, m',b}}^{tor})^{-1}(]C_{w,k}[_{\overline{s+1},s-1}K_{p,m',b})}((\pi^{tor}_{HT,K_{p,m',b}})^{-1}(]C_{w,k}[_{-1,s-1}K_{p,m',b}), \mathcal{V}_{-2\rho_{nc}- w_{0,M} \kappa})$$
$$ \rightarrow \HH^d_{ (\pi_{HT, K_{p, m',b}}^{tor})^{-1}(]C_{w,k}[_{\overline{s+1},\overline{s}}K_{p,m',b})}( (\pi^{tor}_{HT,K_{p,m',b}})^{-1}(]C_{w,k}[_{s,s-1}K_{p,m',b},  \mathcal{V}_{-2\rho_{nc}}(-D))$$
and there is a trace map (by theorem \ref{thm-duality-GK}): $$\HH^d_{ (\pi_{HT, K_{p, m',b}}^{tor})^{-1}(]C_{w,k}[_{\overline{s+1},\overline{s}}K_{p,m',b})}( (\pi^{tor}_{HT,K_{p,m',b}})^{-1}(]C_{w,k}[_{s,s-1}K_{p,m',b}),  \mathcal{V}_{-2\rho_{nc}}(-D)) \rightarrow F.$$

This pairing intertwines the actions of $\mathcal{H}_{p,m',b}^+$ and $\mathcal{H}_{p,m',b}^{-}$. It is straightforward (but painful) to check that the induced pairing $$\langle, \rangle :  \HH^{i}_w(K^p, \kappa, \chi, cusp)^{+,fs} \times \HH^{d-i}_w(K^p, -2\rho_{nc}- w_{0,M} \kappa, \chi^{-1})^{-,fs} \rightarrow \overline{F}$$ is independent of choices.

The rest of the theorem follows from the functoriality of the trace map. 

\end{proof}

We now prove the following theorem: 

\begin{thm}\label{thm-perfect-pairing} The pairing $$\langle, \rangle :  \HH^{i}_w(K^p, \kappa, \chi,cusp)^{\pm,fs} \times \HH^{d-i}_w(K^p, -2\rho_{nc}- w_{0,M} \kappa, \chi^{-1})^{\mp,fs} \rightarrow \overline{F}$$ is non degenerate.  (Equivalently it induces perfect pairings between the finite dimensional generalized eigenspaces for $T^\pm$).
\end{thm}

\begin{proof} We consider the pairing $\langle, \rangle :  \HH^{i}_w(K^p, \kappa, \chi,cusp)^{+,fs} \times \HH^{d-i}_w(K^p, -2\rho_{nc}- w_{0,M} \kappa, \chi^{-1})^{-,fs} \rightarrow \overline{F}$. The other case is similar. 
Recall from propositions \ref{prop-dagger1} and  corollary \ref{coro-dagger3}, the exact triangles:
$$\mathrm{R}\Gamma_w(K^pK_p, -2\rho_{nc}-w_{0,M}\kappa)^{-,fs}  \rightarrow  \mathrm{R}\Gamma((\pi_{HT, K_p}^{tor})^{-1}(\overline{]Y_{w,k}[}), \mathcal{V}_{-2\rho_{nc}-w_{0,M}\kappa})^{-,fs}  $$ $$ \rightarrow
\mathrm{R}\Gamma((\pi_{HT, K_p}^{tor})^{-1}(\overline{]\cup_{w' >w}Y_{w',k}[}), \mathcal{V}_{-2\rho_{nc}-w_{0,M}\kappa})^{-,fs} \stackrel{+1}\rightarrow$$

and $$\mathrm{R}\Gamma_{(\pi_{HT, K_p}^{tor})^{-1}(\overline{]\cup_{w' >w}Y_{w',k}[})}( \mathcal{S}^{tor}_{K^pK_p,\Sigma}, \mathcal{V}_\kappa(-D))^{+,fs}   \rightarrow  \mathrm{R}\Gamma_{(\pi_{HT, K_p}^{tor})^{-1}(\overline{]Y_{w,k}[})}( \mathcal{S}^{tor}_{K^pK_p,\Sigma}, \mathcal{V}_\kappa(-D))^{+,fs}$$ $$ \rightarrow
\mathrm{R}\Gamma_w(K^pK_p, \kappa, cusp)^{+,fs}  \stackrel{+1}\rightarrow$$

Using the naturality of the duality pairing, it is sufficient to show that the pairings: 
 $$\HH^i((\pi_{HT, K_p}^{tor})^{-1}(\overline{]Y_{w,k}[}), \mathcal{V}_{-2\rho_{nc}-w_{0,M}\kappa})^{-,fs} \times \HH^{d-i}_{(\pi_{HT, K_p}^{tor})^{-1}(\overline{]Y_{w,k}[})}( \mathcal{S}^{tor}_{K^pK_p,\Sigma}, \mathcal{V}_\kappa(-D))^{+,fs} \rightarrow F$$
 and 
 $$\HH^{i}((\pi_{HT, K_p}^{tor})^{-1}(\overline{] \cup_{w'>w}Y_{w,k}[}), \mathcal{V}_{-2\rho_{nc}-w_{0,M}\kappa})^{-,fs} \times \HH^{d-i}_{(\pi_{HT, K_p}^{tor})^{-1}(\overline{] \cup_{w'>w}Y_{w,k}[})}( \mathcal{S}^{tor}_{K^pK_p,\Sigma}, \mathcal{V}_\kappa(-D))^{+,fs} \rightarrow F$$ are non degenerate. 
 We only explain how to show that the first pairing is non-degenerate. The proof for the second pairing is identical. We take a finite affinoid covering  $(\pi_{HT, K_p}^{tor})^{-1}(]Y_{w,k}[) = \cup_{i \in I} U_i$. 
 
 We deduce that $\mathrm{R}\Gamma((\pi_{HT, K_p}^{tor})^{-1}(\overline{] Y_{w,k}[}), \mathcal{V}_{-2\rho_{nc}-w_{0,M}\kappa})$, which is the cohomology of the dagger space $(\pi_{HT, K_p}^{tor})^{-1}(]Y_{w,k}[)^\dag$,   is represented by the \v{C}ech complex: 
 $$ M^\bullet :  \prod_{i} \HH^0(U_i^\dag, \mathcal{V}_{-2\rho_{nc}-w_{0,M}\kappa}) \rightarrow \prod_{i,j} \HH^0((U_i \cap U_j)^\dag, \mathcal{V}_{-2\rho_{nc}-w_{0,M}\kappa}) 
 \rightarrow \cdots $$
 
 We deduce that $\mathrm{R}\Gamma_{(\pi_{HT, K_p}^{tor})^{-1}(\overline{] \cup_{w'>w}Y_{w,k}[})}(\mathcal{S}_{K^pK_p,\Sigma}^{tor}, \mathcal{V}_{\kappa})$ (which is the cohomology with compact support of the dagger space $(\pi_{HT, K_p}^{tor})^{-1}(]Y_{w,k}[)^\dag$)  is represented by the \v{C}ech complex:
 $$ N^\bullet :  \cdots \rightarrow \prod_{i,j} \HH^d_c((U_i \cap U_j)^\dag, \mathcal{V}_{\kappa}(-D)) 
 \rightarrow \prod_{i} \HH^d_c(U_i^\dag, \mathcal{V}_{\kappa}(-D))$$
 
 Using the duality for affinoid dagger spaces (see theorem \ref{thm-duality-GK}), we see that $M^\bullet$ is  a complex of inductive limit of Banach spaces, $N^\bullet$ is a complex of projective limit of Banach spaces and these complexes are strict dual of each other. More precisely, the Serre pairing is induced by  the perfect duality on each of the modules (which identifies each module as the strong dual of the other):
 $$  \HH^d_c((U_{J})^\dag, \mathcal{V}_{\kappa}(-D)) \times \HH^0( (U_J)^\dag, \mathcal{V}_{-2\rho_{nc}-w_{0,M}\kappa}) \rightarrow F,$$
 where $J \subseteq I$ and $U_J = \cap_{j \in J} U_{j}$. 
 In general, this does not induce a  perfect duality between the cohomology of $M^\bullet$ and $N^\bullet$. However, we will prove that this is the case on the finite slope part. Now, let $t \in T^{--}$ be an operator acting compactly. We claim that we can represent it by an endomorphism of $M^\bullet$.
 Indeed, for all $i$ we can find an affinoid $V_i$ with  $\overline{U}_i  \subseteq V_i$ (by \cite{MR1032938}, thm. 5.1). Let $V = \cup V_i$. By shrinking the $V_i$'s,  we may assume  by lemma \ref{lem-dynamic-corres}, 7),  that $t(V) \subseteq (\pi_{HT, K_p}^{tor})^{-1}(\overline{] Y_{w,k}[}) \subseteq V$. Let us finally take an affinoid covering $U'_i$ refining $t(V) \cap U_i$. 
 By lemma \ref{lem-representing-mapT}, there is a map $$t : \check{C}( \{U_i'\} ,  \mathcal{V}_{-2\rho_{nc}-w_{0,M}\kappa}) \rightarrow \check{C}(\{V_i\}, \mathcal{V}_{-2\rho_{nc}-w_{0,M}\kappa}).$$
 There are restriction  maps $A : \check{C}(\{U_i^\dag\},  \mathcal{V}_{-2\rho_{nc}-w_{0,M}\kappa}) \rightarrow \check{C}( \{U_i'\} ,  \mathcal{V}_{-2\rho_{nc}-w_{0,M}\kappa}) $ and
 $B : \check{C}( \{V_i\} ,  \mathcal{V}_{-2\rho_{nc}-w_{0,M}\kappa})  \rightarrow \check{C}( \{U_i^\dag\} ,  \mathcal{V}_{-2\rho_{nc}-w_{0,M}\kappa})$. We deduce that $B \circ t \circ A$ represents $t$ acting on $M^\bullet$ (and is compact since $A$ and $B$ are easily seen to be compact). 
 We easily deduce that the  adjoint of $t$ represents $t^{-1}$ acting on $N^\bullet$. 
 Passing to the slope $\leq h$ part on both $M^\bullet$ and $N^\bullet$, we obtain a perfect pairing between complexes of finite dimensional vector spaces. It induces a perfect pairing on slope $\leq h$ cohomology groups. 
 \end{proof}
 
Using duality, we can prove vanishing theorems for non-cuspidal cohomology: 

\begin{thm}\label{theorem-coho-van2} The cohomology complex $\mathrm{R}\Gamma_w(K^p, \kappa, \chi)^{\pm,fs}$ has amplitude  $[ \ell_\pm(w),d]$. 
\end{thm}

\begin{proof} This follows by combining theorem \ref{thm-perfect-pairing} and theorem \ref{theorem-coho-van}.
\end{proof}

\begin{prop}\label{prop-deg-spect-sequ2} For the spectral sequence $\mathbf{E}^{p,q}(K^p, \kappa,\chi)^\pm$ converging to $ \HH^{p+q}(K^p,\kappa,\chi)^{\pm,fs}$  we have $\mathbf{E}_1^{p,q}(K^p, \kappa,\chi)^\pm = 0$ if  $q <0$.  In particular $\mathbf{E}_\infty^{p,q}(K^p, \kappa,\chi)^\pm = \mathrm{Gr}^p(  \HH^{p+q}(K^p,\kappa,\chi)^{\pm,fs}) =0$ if $p> p+q$. 
\end{prop}

\begin{proof} This follows from theorem \ref{theorem-coho-van2}.
\end{proof}

\begin{coro}\label{coro-cousin-computes} If the Shimura variety is compact, $\mathrm{R}\Gamma_w(K^p, \kappa, \chi)^{\pm,fs}$ is concentrated in degree $\ell_{\pm}(w)$. 
\end{coro}
\begin{proof} This is a combination of theorems \ref{theorem-coho-van} and \ref{theorem-coho-van2}.
\end{proof}

\begin{rem}
Even if the Shimura variety is not compact, it can still happen that for certain $w\in\WM$, the locally closed on which $\mathrm{R}\Gamma_w(K^p,\kappa,\chi)^{\pm,fs}$ is computed does not meet the boundary $D$, and hence $\mathrm{R}\Gamma_w(K^p,\kappa,\chi)^{\pm,fs}=\mathrm{R}\Gamma_w(K^p,\kappa,\chi,cusp)^{\pm,fs}$ is concentrated in degree $\ell_{\pm}(w)$ just as in the compact case.  As a basic example, consider the Hilbert Shimura datum $(\mathrm{Res}_{F/\qq}\mathrm{GL}_2,\mathcal{H}_1^{[F:\qq]})$ and a prime $p$ which is totally inert in $F$.  Then $\WM=W=\{1,w_0\}^{[F:\qq]}$ and we have $\mathrm{R}\Gamma_w(K^p,\kappa,\chi)^{\pm,fs}=\mathrm{R}\Gamma_w(K^p,\kappa,\chi,cusp)^{\pm,fs}$ unless $w=(1,\ldots,1)$ or $(w_0,\ldots,w_0)$.
\end{rem}

The following corollary illustrates the importance of the Cousin complex:

\begin{coro}\label{coro-concentration-compact} If the Shimura variety is compact, we have a quasi-isomorphism: 
${\mathcal{C}ous}(K^p,\kappa, \chi)^{\pm} = \mathrm{R}\Gamma(K^p, \kappa, \chi)^{\pm,fs}$
\end{coro}
\begin{proof} For  the spectral sequence $\mathbf{E}^{p,q}(K^p, \kappa,\chi)^\pm$ converging to $ \HH^{p+q}(K^p,\kappa,\chi)^{\pm,fs}$,  we have  $\mathbf{E}_1^{p,q}(K^p, \kappa,\chi)^\pm =0$ if $q\neq 0$.
\end{proof}

\subsection{Interior cohomology}\label{section-interior-coho}

For non compact Shimura varieties, we can introduce the interior cohomology: $$ \overline{\HH}^{i}(K^p, \kappa, \chi)^{\pm,fs} = \mathrm{Im}({\HH}^{i}(K^p, \kappa, \chi,cusp)^{\pm,fs} \rightarrow {\HH}^{i}(K^p, \kappa, \chi)^{\pm,fs}). $$ 

We can also consider the interior overconvergent cohomology
$$\overline{\HH}^i_w(K^p,\kappa, \chi)^{\pm,fs}  =  \mathrm{Im} \big( \HH^i_w (K^p, \kappa, \chi, cusp)^{\pm} \rightarrow \HH^i_w (K_p, \kappa, \chi)^{\pm} \big)$$
as well as the interior Cousin complex  $$\overline{\mathcal{C}ous}(K^p,\kappa, \chi)^{\pm} = $$ $$ \mathrm{Im} \big( \mathcal{C}ous (K^p, \kappa, \chi, cusp)^{\pm} \rightarrow \mathcal{C}ous (K_p, \kappa, \chi)^{\pm} \big).$$

By definition, $\overline{\mathcal{C}ous}(K^p,\kappa, \chi)^{\pm}$ is concentrated in degrees in the interval $[0,d]$ and its degree $i$ object is $\oplus_{w \in \WM, \ell_{\pm}(w)=i}  \overline{\HH}^{i}_w(K^p,\kappa, \chi)^{\pm,fs}$.

\begin{coro}\label{coro-concentration-interior}  We have the formula:  $$\overline{\HH}^{p}(K^p,\kappa, \chi)^{\pm,fs}  = \mathrm{Im} (\mathbf{E}_\infty^{p,0}(K^p, \kappa,\chi,cusp)^\pm \rightarrow \mathbf{E}^{p,0}_\infty(K^p, \kappa,\chi)^{\pm}).$$

Moreover, for all $p$, $\overline{\HH}^{p}(K^p,\kappa, \chi)^{\pm,fs}$ is a subquotient of $\HH^p(\overline{\mathcal{C}ous}(K^p,\kappa, \chi)^{\pm})$.
\end{coro}
\begin{proof}  We have a  commutative diagram:

\begin{eqnarray*}
\xymatrix{  {\HH}^{p}(K^p,\kappa, \chi, cusp)^{\pm,fs} \ar[r]  \ar[d] & {\HH}^{p}(K^p,\kappa, \chi)^{\pm,fs} \\
\mathbf{E}_\infty^{p,0}(K^p, \kappa,\chi,cusp)^\pm \ar[r] \ar[d] & \mathbf{E}_\infty^{p,0}(K^p, \kappa,\chi)^\pm \ar[u]\\
\HH^p({\mathcal{C}ous}(K^p,\kappa, \chi,cusp)^{\pm})  \ar[r] & \HH^p({\mathcal{C}ous}(K^p,\kappa, \chi)^{\pm}) \ar[u]}
\end{eqnarray*}
where the map ${\HH}^{p}(K^p,\kappa, \chi, cusp)^{\pm,fs}  \rightarrow \mathbf{E}_\infty^{p,0}(K^p, \kappa,\chi,cusp)^\pm$ is surjective (this last module is the last non-zero graded piece in the filtration given by the spectral sequence) and the map $ \mathbf{E}_\infty^{p,0}(K^p, \kappa,\chi)^\pm \rightarrow {\HH}^{p}(K^p,\kappa, \chi)^{\pm,fs}$ is injective (the first module is the first non-zero graded piece in the filtration given by the spectral sequence). 

We deduce that $$\overline{\HH}^{p}(K^p,\kappa, \chi)^{\pm,fs}  = \mathrm{Im} (\mathbf{E}_\infty^{p,0}(K^p, \kappa,\chi,cusp)^\pm \rightarrow \mathbf{E}^{p,0}_\infty(K^p, \kappa,\chi)^{\pm}).$$

The map $$\mathbf{E}_\infty^{p,0}(K^p, \kappa,\chi,cusp)^\pm \rightarrow \HH^p({\mathcal{C}ous}(K^p,\kappa, \chi,cusp)^{\pm}) = \mathbf{E}_2^{p,0}(K^p, \kappa,\chi,cusp)^\pm$$ is injective, and the map $$\HH^p({\mathcal{C}ous}(K^p,\kappa, \chi)^{\pm}) = \mathbf{E}_2^{p,0}(K^p, \kappa,\chi)^\pm \rightarrow \mathbf{E}_\infty^{p,0}(K^p, \kappa,\chi)^\pm$$ is surjective. 

We deduce  that $\overline{\HH}^{p}(K^p,\kappa, \chi)^{\pm,fs} $ is a subquotient of $$\mathrm{Im} ( \HH^p({\mathcal{C}ous}(K^p,\kappa, \chi,cusp)^{\pm})\rightarrow \HH^p({\mathcal{C}ous}(K^p,\kappa, \chi)^{\pm})).$$
This last group is a subquotient of  $\HH^p(\overline{\mathcal{C}ous}(K^p,\kappa, \chi)^{\pm})$.
\end{proof}
 
\subsection{Lower bounds on slopes}  In this section we will write $\langle -,-\rangle$  for the usual pairing $X_\star(T)\times X^\star(T)\to\ZZ$. We denote by $T^d$ the maximal $\qq_p$-split subtorus of $T$.  We have a relative root system $\Phi_d\subset X^*(T^d)$ with a choice of positive and simple roots $\Delta_d\subset\Phi^+_d\subset\Phi_d$.  Because $G/\qq_p$ is quasi-split, restriction from $T$ to $T^d$ defines a surjective map $r:\Phi\rightarrow\Phi_d$, which restricts to a surjective map $r:\Delta\rightarrow\Delta_d$ (the fibers of $r$ are exactly the Galois orbits of absolute roots.)

We have on $X^\star(T)_{\mathbb{R}}$ a partial order $\preceq$ where $\lambda\preceq \lambda'$ if and only if $\lambda'-\lambda\in\mathbb{R}_{\geq 0}\Delta$.  We have on $X^\star(T^d)_{\mathbb{R}}$ a partial order $\leq$ where $\lambda\leq\lambda'$ if and only if $\lambda'-\lambda\in\mathbb{R}_{\geq 0}\Delta_d$.  We extend the symbol $\leq$ to the case that one or both sides are in $X^\star(T)_{\mathbb{R}}$, in which case we apply the restriction map $X^\star(T)\rightarrow X^\star(T^d)$ (so in particular for $\lambda,\lambda'\in X^\star(T)$, $\lambda\preceq\lambda'$ implies $\lambda\leq\lambda'$, but not necessarily conversely.) 

Recall that we have monoids $T^+$ and $T^-$ in $T(\qq_p)$.  In section \ref{section-compact-open-subgroups}, we defined a valuation morphism $v : T(\qq_p) \rightarrow X_\star(T^d) \otimes \qq$, whose image is a lattice, and whose kernel is the maximal compact subgroup of $T(\qq_p)$ which we denoted (slightly abusively) by $T(\ZZ_p)$.  For $\lambda\in X^\star(T^d)_{\mathbb{R}}$ and $t\in T(\qq_p)$ we will abusively write $v(\lambda(t))$ for $\langle v(t),\lambda\rangle$.  The partial order $\leq$ has another characterization that we frequently use:
\begin{lem}
Let $\lambda,\lambda'\in X^\star(T^d)_{\mathbb{R}}$.  Then $\lambda\leq\lambda'$ if and only if $v(\lambda(t))\leq v(\lambda'(t))$ for all $t\in T^+$. 
\end{lem}

Given a homomorphism $\lambda:T(\qq_p)\rightarrow \overline{F}^\times$, the composition with the valuation $v:\overline{F}^\times\to\mathbb{R}$ factors through a morphism $v(T(\qq_p))=T(\qq_p)/T(\ZZ_p)\rightarrow\mathbb{R}$.  This extends by linearity to a $\mathbb{R}$-linear map $X_\star(T^d)_{\mathbb{R}}\to\mathbb{R}$ and thus defines an element of $X^\star(T^d)_{\mathbb{R}}$, which we will denote by $v(\lambda)$ and call the \emph{slope} of $\lambda$.  Unravelling the definition we have $\langle v(\lambda),v(t)\rangle=v(\lambda(t))$.

If we start instead with a monoid homomorphism $T^{\pm}\rightarrow \overline{F}^\times$, we also define the slope $v(\lambda)$ of $\lambda$  by first extending $\lambda$ to a group homomorphism $T(\qq_p)\to\overline{F}^\times$ (recall that $T(\qq_p)$ is generated by the monoids $T^\pm$).

We now formulate a general conjectural lower bound on the slopes of $\mathrm{R}\Gamma_w(K^p, \kappa,\chi)^{\pm,fs}$ and $\mathrm{R}\Gamma_w(K^p, \kappa,\chi,cusp)^{\pm,fs}$

\begin{conj}\label{conj-strongslopes}
Fix $w\in\WM$, $\kappa\in X^\star(T^c)^{M_\mu,+}$, and $\chi:T(\ZZ_p)\to {\overline{F}}^\times$ of finite order.
\begin{enumerate}
\item For any  character $\lambda$ of $T^+$ on $\mathrm{R}\Gamma_w(K^p, \kappa,\chi)^{+,fs}$ or $\mathrm{R}\Gamma_w(K^p, \kappa,\chi,cusp)^{+,fs}$ we have $v(\lambda) \geq w^{-1} w_{0,M}(\kappa + \rho) +  \rho$. 

\item For any  character $\lambda$ of $T^-$ on $\mathrm{R}\Gamma_w(K^p, \kappa,\chi)^{-,fs}$ or $\mathrm{R}\Gamma_w(K^p, \kappa,\chi,cusp)^{-,fs}$ we have $v(\lambda) \leq w^{-1} (\kappa + \rho) -  \rho$.
\end{enumerate}
\end{conj}

\begin{rem} We can spell out the meaning of these inequalities.  The  inequality $v(\lambda) \geq w^{-1} w_{0,M}(\kappa + \rho) +  \rho$ means that for all  $t \in T^+$ (and corresponding $v(t) \in X_\star(T^d)_{\qq}^+$), we have $$v(\lambda(t)) \geq  \langle v(t), w^{-1} w_{0,M}(\kappa + \rho) +  \rho\rangle.$$
The inequality $v(\lambda) \leq w^{-1} (\kappa + \rho) -  \rho$ means that for all $t \in T^-$ and  (and corresponding $v(t) \in X_\star(T^d)_{\qq}^-$), we have 
$$ v(\lambda(t)) \geq \langle v(t), w^{-1}(\kappa + \rho) - \rho\rangle.$$
\end{rem}

\begin{rem}
We recall that for any $w\in W$ we have $\rho+w\rho,\rho-w\rho\in X^\star(T)^+$ (even if $\rho$ is not itself in $X^\star(T)$.)  It follows that for all $t\in T^+$ we have $\langle t,\rho+w^{-1}w_{0,M}\rho\rangle\in\mathbb{Z}_{\geq0}$ and for all $t\in T^-$ we have $-\langle t,\rho-w^{-1}\rho\rangle\in\mathbb{Z}_{\geq 0}.$
\end{rem}

\begin{rem}
The bounds of Conjecture \ref{conj-strongslopes} are compatible with duality in the sense that they are exchanged upon replacing $t$ by $t^{-1}$ and $\kappa$ by $-2\rho_{nc}-w_{0,M}\kappa$.
\end{rem}

On the right hand side of the inequality  of conjecture \ref{conj-strongslopes} we have $w^{-1} w_{0,M}(\kappa + \rho) +  \rho$ and  $w^{-1} (\kappa + \rho) -  \rho$. Each of these expressions can be separated into  $(w^{-1} w_{0,M}\kappa )  + (w^{-1} w_{0,M}\rho +  \rho)$ and $(w^{-1} \kappa )  + (w^{-1}\rho -  \rho)$ where the first term depends on $\kappa$ and is related to the action of the Hecke correspondences on the sheaf, while the second term is independent of $\kappa$ and is related to the geometry of the correspondence (and in particular to the ramification of integral models of the correspondence). The second term is the more delicate to study. 

The main result of this section is a bound which is slightly weaker than the conjecture (and concerns only the first term).

\begin{thm}\label{thm-slopes} 
Fix $w\in\WM$, $\kappa\in X^\star(T^c)^{M_\mu,+}$, and $\chi:T(\ZZ_p)\to {\overline{F}}^\times$ of finite order.
\begin{enumerate}
\item For any  character $\lambda$ of $T^+$ on $\mathrm{R}\Gamma_w(K^p, \kappa,\chi)^{+,fs}$ or $\mathrm{R}\Gamma_w(K^p, \kappa,\chi,cusp)^{+,fs}$ we have $v(\lambda) \geq w^{-1} w_{0,M}\kappa$ and $v(\lambda) \geq w^{-1} w_{0,M}\kappa + w^{-1} 2 \rho_{nc}$. 

\item For any  character $\lambda$ of $T^-$ on $\mathrm{R}\Gamma_w(K^p, \kappa,\chi)^{-,fs}$ or $\mathrm{R}\Gamma_w(K^p, \kappa,\chi,cusp)^{-,fs}$ we have $v(\lambda) \leq w^{-1}\kappa$ and $v(\lambda) \leq w^{-1} \kappa + w^{-1} 2 \rho_{nc}$. 

\end{enumerate}
\end{thm}

\begin{rem} Using eigenvarieties, we will be able to prove the conjecture \ref{conj-strongslopes} for the interior cohomology in theorem \ref{thm-strongslopes-interior}.
\end{rem}
\subsubsection{Proof of theorem \ref{thm-slopes}}   We begin with a reduction.  The bounds in  theorem \ref{thm-slopes} are compatible with the duality theorem \ref{thm-perfect-pairing}. Actually, it will suffice to prove: \begin{enumerate}
\item For any  character $\lambda$ of $T^+$ on $\mathrm{R}\Gamma_w(K^p, \kappa,\chi)^{+,fs}$ or $\mathrm{R}\Gamma_w(K^p, \kappa,\chi,cusp)^{+,fs}$ we have $v(\lambda) \geq w^{-1} w_{0,M}\kappa$. 

\item For any  character $\lambda$ of $T^-$ on $\mathrm{R}\Gamma_w(K^p, \kappa,\chi)^{-,fs}$ or $\mathrm{R}\Gamma_w(K^p, \kappa,\chi,cusp)^{-,fs}$ we have $v(\lambda) \leq w^{-1}\kappa$. 

\end{enumerate}
and the rest will follow from duality.

 Let $\kappa \in X^\star(T^c)^{M_\mu,+}$. The definition of the sheaf $\mathcal{V}_\kappa$ is given in \ref{section-automorphicvectorbundles} with the help of the torsor $\mathcal{M}^{an}_{dR}$ and modeled on the highest weight representation $V_\kappa$. By corollary \ref{coro-existence-integral-struct}, the sheaf $\mathcal{V}_\kappa$ has an integral structure $\mathcal{V}_\kappa^+$ (in the sense of definition \ref{defi-integral-structure-sheaf}), constructed with the help of  the $\mathcal{M}_\mu$-torsor $\mathcal{M}_{dR}$ (see proposition \ref{prop-first-reduction}) and modeled on the submodule $V_\kappa^+ \subseteq V_\kappa$.

\begin{lem}\label{lem-sheaf-map-integral} 
Let $K_p=K_{p,m',b}$ for $m'\geq b\geq 0$ and $m'>0$.
\begin{enumerate}
\item  Let $t \in T^+$.  For all $n \geq 1$,
the isomorphism  $p_2^\star \mathcal{V}_\kappa \rightarrow p_1^\star \mathcal{V}_\kappa$ induces a map
$p_2^\star \mathcal{V}_\kappa^+ \rightarrow p^{ \langle  wv(t), w_{0,M}\kappa \rangle} p_1^\star \mathcal{V}_\kappa^+$
on $$p_2^{-1} \big( (\pi^{tor}_{HT, K_p})^{-1}( w\mathcal{G}_n K_p)\big) \cap p_1^{-1} \big( (\pi^{tor}_{HT, K_p})^{-1}( w\mathcal{G}_n K_p) \big).$$
\item Let $t \in T^-$.  For all $n \geq 1$,
the isomorphism  $p_2^\star \mathcal{V}_\kappa \rightarrow p_1^\star \mathcal{V}_\kappa$ induces a map
$p_2^\star \mathcal{V}_\kappa^+ \rightarrow p^{ \langle wv(t),  \kappa \rangle} p_1^\star \mathcal{V}_\kappa^+$
on $$p_2^{-1} \big( (\pi^{tor}_{HT, K_p})^{-1}( w\mathcal{G}_n K_p)\big) \cap p_1^{-1} \big( (\pi^{tor}_{HT, K_p})^{-1}( w\mathcal{G}_n K_p) \big).$$
\end{enumerate}
\end{lem}

\begin{proof} 
We prove the first point. We have a map $ t : p_2^\star \mathcal{M}_{dR}^{an} \rightarrow p_1^\star \mathcal{M}_{dR}^{an}$, which is locally represented by $wt$ by proposition \ref{prop-representing-torsor-map}. 

Therefore, we have an isomorphism $ t^\star : p_1^\star \mathcal{V}_\kappa \rightarrow p_2^\star \mathcal{V}_\kappa$ which is locally given by
$ t^\star f (x'_2 m) = f( x'_1 wt m)$ for trivializations $x'_1$ and $x'_2$ of $p_1^\star \mathcal{M}_{dR}$ and $p_2^\star \mathcal{M}_{dR}$. The map $p_2^\star \mathcal{V}_\kappa \rightarrow p_1^\star \mathcal{V}_\kappa$ of the lemma is the inverse of the map $t^\star$. 

This map is locally isomorphic to the map 
\begin{eqnarray*}
V_\kappa & \rightarrow & V_\kappa \\
v & \mapsto & (wt)^{-1} v
\end{eqnarray*}
which has eigenvectors of valuation $\langle (wv(t))^{-1}, \nu \rangle$ where $\nu$ ranges through the weights of $V_\kappa$. Since $t \in T^{d,+}$ and $w \in \WM$, it follows that $wv(t) \in X_\star(T)_{\qq}^{M_\mu,+}$ so that $(wv(t))^{-1} \in X_\star(T)_{\qq}^{M_\mu, -}$. 
The lowest weight of $V_\kappa$ is $w_{0,M} \kappa$ and therefore,  $p^{\langle (wv(t))^{-1}, w_{0,M}\kappa \rangle} V_\kappa^+ \subseteq (wt)^{-1} V_\kappa^+$. We deduce that  $p^{\langle (wv(t))^{-1}, w_{0,M}\kappa \rangle} p_2^\star \mathcal{V}_\kappa^+ \subseteq  t^\star  p_1^\star \mathcal{V}_\kappa^+$ from which we deduce that $p_2^\star \mathcal{V}_\kappa^+ \rightarrow p^{ \langle  wv(t),  w_{0,M} \kappa \rangle} p_1^\star \mathcal{V}_\kappa^+$.
The proof of the second point is almost identical, it is enough to observe that now $(wv(t))^{-1} \in X_\star(T)_{\qq}^{M_\mu,+}$ and that $\kappa$ is the highest weight of $V_\kappa$. Details are left to the reader. 
\end{proof}

\begin{lem}\label{lemma-construct-lattice} Let $w \in \WM$, $\kappa \in X^{\star}(T^c)^{M_\mu, +}$.  
\begin{enumerate}
\item Let $(\mathcal{U}, \mathcal{Z})$ be a $(+, w , K_{p,m',b})$-allowed support condition. Assume further that $\mathcal{U}$ is a quasi-compact open and that the complement of $\mathcal{Z}$ is quasi-compact. The image of 
$ \HH^i_{\mathcal{U} \cap \mathcal{Z}}(\mathcal{U}, \mathcal{V}_\kappa^+)$ in $\HH^i_w(K^p, \kappa)^{+,\leq fs}$ is an open and bounded submodule.
\item Let $(\mathcal{U}, \mathcal{Z})$ be a $(-, w , K_{p,m',b})$-allowed support condition. Assume further that $\mathcal{U}$ is a quasi-compact open and that the complement of $\mathcal{Z}$ is quasi-compact. The image of 
$ \HH^i_{\mathcal{U} \cap \mathcal{Z}}(\mathcal{U}, \mathcal{V}_\kappa^+)$ in $\HH^i_w(K^p, \kappa)^{-,\leq fs}$  is an open and bounded submodule. 
\end{enumerate}
\end{lem}

\begin{proof} We only treat the first item since the second one follows with minor modifications. We can represent the cohomology $ \mathrm{R}\Gamma_{\mathcal{U} \cap \mathcal{Z}}(\mathcal{U}, \mathcal{V}_\kappa)$ by an explicit complex of Banach modules $C^\bullet$. We choose an affinoid covering  $\mathfrak{U}_1$  of $\mathcal{U}$ and  consider the \v{C}ech complex $\check{C}(\mathfrak{U}_1, \mathcal{V}_\kappa)$ which computes $\mathrm{R}\Gamma(\mathcal{U}, \mathcal{V}_\kappa)$. Next we take an affinoid covering  $\mathfrak{U}_2$ of $\mathcal{U} \cap \mathcal{Z}^c$ refining the covering $\mathfrak{U}_1 \cap \mathcal{Z}^c$ and consider the \v{C}ech complex $\check{C}(\mathfrak{U}_2, \mathcal{V}_\kappa)$ which computes $\mathrm{R}\Gamma(\mathcal{U} \cap \mathcal{Z}^c, \mathcal{V}_\kappa)$. Finally, we  represent the cohomology $ \mathrm{R}\Gamma_{\mathcal{U} \cap \mathcal{Z}}(\mathcal{U}, \mathcal{V}_\kappa)$ by $$C^\bullet = \mathrm{Cone}( \check{C}(\mathfrak{U}_1, \mathcal{V}_\kappa) \rightarrow \check{C}(\mathfrak{U}_2, \mathcal{V}_\kappa))[-1].$$    We also consider the subcomplex of open and bounded submodules of $C^\bullet$: 
$$C^{+, \bullet} =  \mathrm{Cone}( \check{C}(\mathfrak{U}_1, \mathcal{V}^+_\kappa) \rightarrow \check{C}(\mathfrak{U}_2, \mathcal{V}^+_\kappa))[-1].$$
Any sufficiently regular element $T$ of $T^{++}$  lifts to a compact endomorphism $\tilde{T}$ of $C^\bullet$ and we can consider the direct ``slope less  or equal than $h$'' factor  $C^{\bullet, \leq h }$ of $C^\bullet$. This is a perfect complex of $F$-vector spaces,  whose cohomology groups compute $\HH^i_w(K^p, \kappa)^{+,\leq h}$. Denote the projection of  $C^{+, \bullet}$ in $C^{\bullet, \leq h }$ by $C^{+, \bullet, \leq h}$. This is   a perfect complex of $\ocal_F$-modules and the image of $\HH^i(C^{+,\bullet,  \leq h })$ in $\HH^i(C^{\bullet, \leq h }) = \HH^i_w(K^p, \kappa)^{+,\leq h}$ is therefore an open and bounded submodule.  Therefore, for any $h$, the image of $\HH^i(C^{+,\bullet})$ in $ \HH^i_w(K^p, \kappa)^{+,\leq h}$ is open and bounded. Passing to the limit over $h$, we deduce that the image of $\HH^i(C^{+,\bullet})$ in $ \HH^i_w(K^p, \kappa)^{+,fs}$ is open and bounded.
To prove the lemma, it  suffices to show that the map $$ \HH^i(C^{+, \bullet}) \rightarrow \HH^i_{\mathcal{U} \cap \mathcal{Z}}(\mathcal{U}, \mathcal{V}_\kappa^+)$$ has kernel and cokernel of bounded torsion.  Using the \v{C}ech to cohomology spectral sequence, this follows from lemma \ref{lem-Cechversusanalytic}. 
\end{proof}

\begin{proof}[Proof of Theorem \ref{thm-slopes}]  We only prove the $+$ case. The $-$ case follows with minor modifications. We let $t \in T^{+}$, and let $T = [K_{p, m', b} t K_{p, m', b}]$.   We   take $(\mathcal{U}, \mathcal{Z})$  as in lemma \ref{lemma-construct-lattice}.
We find by using lemma \ref{lem-sheaf-map-integral} that we have an endomorphism $p^{ -\langle  wv(t), w_{0,M}\kappa \rangle}T : \mathrm{R}\Gamma_{\mathcal{U} \cap \mathcal{Z}}( \mathcal{U}, \mathcal{V}_\kappa^+) \rightarrow \mathrm{R}\Gamma_{\mathcal{U} \cap \mathcal{Z}}( \mathcal{U}, \mathcal{V}_\kappa^+)$. It follows from lemma \ref{lemma-construct-lattice} that $p^{ -\langle  wv(t), w_{0,M}\kappa \rangle}T$ preserves an open and bounded submodule in $\HH^i_w(K^p, \kappa)^{+,fs}$ for all $i$. For any character of $\lambda$ of $T^+$ on $\HH^i_w(K^p, \kappa)^{+,fs}$,    this implies that $v(\lambda(t)) \geq \langle  wv(t), w_{0,M}\kappa \rangle$. The theorem is thus proven. 
\end{proof}

\subsection{Comparisons with slope bounds on classical cohomology}
In this section, we make the connection between the slope bounds on overconvergent cohomology with other slope bounds on classical cohomology.
\subsubsection{Some combinatorics} Let $\kappa\in X^\star(T^c)^{M_\mu,+}$. We first attach to $\kappa$ certain subsets of the Weyl group. 
 We let $W(\kappa)^+ = \{ w \in W, w w_{0,M}(\kappa + \rho) = w_{0,M}(\kappa + \rho) \}$. We let $W(\kappa)^{-} =  \{ w \in W, w(\kappa + \rho) = \kappa + \rho \}$. We let $C(\kappa)^+ = \{ w \in W, w^{-1} w_{0,M}(\kappa + \rho) \in X^\star(T)_{\qq}^- \}$ and $C(\kappa)^- = \{ w \in W, w^{-1} (\kappa + \rho) \in X^\star(T)_{\qq}^+ \}$. 

\begin{prop} \begin{enumerate}
\item The set $C(\kappa)^{\pm}$ is a left principal homogeneous space under $W(\kappa)^{\pm}$. 
\item $C(\kappa)^{\pm} \subseteq \WM$.
\item $\kappa+\rho$ is regular if and only if $C(\kappa)^{\pm}$ is reduced to a single element.
\item We have $W(\kappa)^+ = w_{0,M} W(\kappa)^- w_{0,M}$ and $C(\kappa)^+ = w_{0,M} C(\kappa)^- w_{0}$.
\item We have $C(\kappa)^{\pm} = C(-w_{0,M}\kappa-2\rho_{nc})^{\mp}$.
\end{enumerate}
\end{prop}
\begin{proof} Left multiplication defines an action of $W(\kappa)^{\pm}$ on $C(\kappa)^{\pm}$.   Given a weight $\lambda \in X^\star(T)^-$, we have $w\lambda \geq \lambda$ for all $w \in W$. It follows that if $w, w' \in C(\kappa)^{\pm}$, $w(w')^{-1} \in W(\kappa)^{\pm}$. The elements of $\WM$ are characterized among $W$ by the property that $wX^\star(T)^{+} \subseteq X^\star(T)^{M_\mu,+}$. Since $\kappa + \rho$ is $M_\mu$-dominant and regular, the second point follows. The remaining points are evident.
\end{proof}

We now give some more explanations concerning the meaning of these sets and the connection with infinitesimal characters. The element $-w^{-1} w_{0,M}(\kappa + \rho) \in X^\star(T)_{\qq}^+$ is independent of  $w \in C(\kappa)^+$, and we denote it by $\nu + \rho$ for $\nu \in X^\star(T)$.

\begin{prop}[\cite{harris-ann-arb}, prop. 3.1.4] The character $\nu + \rho$ is the dominant representative of the infinitesimal character of the automorphic representations contributing to the cohomology of the sheaves $\mathcal{V}_\kappa$ or $\mathcal{V}_{\kappa}(-D)$ over $S^{tor}_{K,\Sigma}(\C)$.
\end{prop}

\begin{rem} The infinitesimal character of the automorphic representations contributing to the cohomology of the Serre dual sheaves $\mathcal{V}_{-w_{0,M} \kappa- 2 \rho_{nc}}$ and $\mathcal{V}_{-w_{0,M} \kappa- 2 \rho_{nc}}(-D)$ is $-\nu-\rho$. Its dominant representative is therefore $-w_{0}\nu+\rho$.
\end{rem}

It is important to record the formulas that allow us to switch between the infinitesimal character and the weight:
\begin{eqnarray*}
\nu& =& -w^{-1} w_{0,M}(\kappa + \rho) - \rho, ~\forall w \in C(\kappa)^+ \\
\kappa &=& - w_{0,M}w( \nu + \rho)-\rho, ~\forall w \in C(\kappa)^+ \\
\nu &=&  - w_{0}w^{-1}(\kappa + \rho) - \rho,~\forall w \in C(\kappa)^-\\
\kappa &=& -w w_{0}( \nu + \rho)-\rho, ~\forall w \in C(\kappa)^-.
\end{eqnarray*}

We also introduce the notation $\ell_{\min}(\kappa)=\min_{w\in C(\kappa)^+}\ell_+(w)=\min_{w\in C(\kappa)^-}\ell_-(w)$ and $\ell_{\max}(\kappa)=\max_{w\in C(\kappa)^+}\ell_+(w)=\max_{w\in C(\kappa)^-}\ell_-(w)$.  Here the equalities follow from the fact that for $w\in\WM$, $\ell_+(w_{0,M}ww_0)=d-\ell_+(w)=\ell_-(w)$. Moreover $\ell_{\min}(\kappa)=\ell_{\max}(\kappa)$ if and only if $\kappa+\rho$ is regular.  We note that we expect the automorphic vector bundle $\mathcal{V}_\kappa$ to have interesting cohomology exactly in the range $[\ell_{\min}(\kappa),\ell_{\max}(\kappa)]$.   

\begin{thm} Let $\pi$ be an automorphic representation contributing to the cohomology of the sheaves $\mathcal{V}_\kappa$ or $\mathcal{V}_{\kappa}(-D)$ over $S^{tor}_{K,\Sigma}(\C)$. Assume that $\pi_\infty$ is (essentially) tempered. Then $\pi$ can only contribute to the cohomology in degree in the range $[\ell_{\min}(\kappa),\ell_{\max}(\kappa)]$.
\end{thm}
\begin{proof} This follows from a combination of results of Blasius-Harris-Ramakrishnan, Mirkovich, Schmid and Williams. See \cite{harris-ann-arb}, thm. 3.4 and thm. 3.5.
\end{proof}

\subsubsection{Slope bounds on classical coherent cohomology} In light of the spectral sequences of section \ref{section-spectral-sequences}, conjecture \ref{conj-strongslopes} suggests the following conjectural slope bound for classical cohomology:
\begin{conj}\label{conj-strongslopes-classical} Let $\kappa \in X^\star(T^c)^{M_\mu,+}$ and let $\chi:T(\ZZ_p)\to \overline{F}^\times$ be a finite order character.  Let $\nu=-w^{-1}w_{0,M}(\kappa+\rho)-\rho$ for any $w\in C(\kappa)^+$.  For any eigensystem $\lambda : T^\pm \rightarrow \overline{F}^\times$ occurring in the classical cohomologies $\mathrm{R}\Gamma(K^p, \kappa, \chi)^{\pm,fs}$ or  $\mathrm{R}\Gamma(K^p, \kappa, \chi, cusp)^{\pm,fs}$, we have:
\begin{enumerate}
\item In the $+$ case, $v(\lambda) \geq -\nu$.
\item In the $-$ case, $v(\lambda) \leq -w_0\nu$.
\end{enumerate}
\end{conj}

\begin{rem}
The $+$ and $-$ statements are in fact equivalent, in view of the discussion of section \ref{subsection-jacquet} and in particular the isomorphism between the Jacquet modules for $U$ and $\overline{U}$ given by $w_0$.
\end{rem}

\begin{prop}\label{prop-strongslopes-implies-classical}
Conjecture \ref{conj-strongslopes} implies conjecture \ref{conj-strongslopes-classical}.
\end{prop}
\begin{proof}
We treat the non cuspidal $+$ case, the others are identical.  By the spectral sequence of section \ref{section-spectral-sequences} $\lambda$ occurs in $\mathrm{R}\Gamma_{w'}(K^p,\kappa,\chi)^{+,fs}$ for some $w'\in\WM$, and hence by conjecture \ref{conj-strongslopes} we have
\begin{equation*}
s(\lambda)\geq {w'}^{-1}w_{0,M}(\kappa+\rho)+\rho=-(w'^{-1}w)\cdot\nu\geq-\nu
\end{equation*}
where the last inequality follows from lemma \ref{lem-bruhat-inequality} below, using that $\nu+\rho\in X^\star(T)^+_{\mathbb{R}}$.
\end{proof}

We are not able to prove conjecture \ref{conj-strongslopes-classical} completely, however we will see that it holds when $\kappa+\rho$ is regular in theorem \ref{coro-lafforgue-estimates} in the next section, and we will eventually use $p$-adic interpolation to prove it for interior cohomology in theorem \ref{thm-strongslopes-interior}, and so in particular it holds for compact Shimura varieties.

We now explain the relation of this conjecture with other known and conjectured slope bounds on classical cohomology.

\subsubsection{Slope bounds on Betti cohomology}  Let $\nu\in X^\star(T^c)^+$. Let $W_\nu$ be the corresponding irreducible representation of $G$ with highest weight $\nu$ defined over $F$. Over $S_K(\C)$, we can construct a local system $\mathcal{W}_\nu^\vee$ attached to $W_\nu^\vee$ and we can consider the Betti cohomology groups $ \HH^\star(S_K(\C), \mathcal{W}_\nu^\vee)$ and $\HH^\star_c(S_K(\C), \mathcal{W}_\nu^\vee).$

\begin{prop}\label{prop-lafforgue} Assume that $K = K^pK_p$ with  $K_p = K_{p,m,b}$. Let $\nu \in X^\star(T^c)^+$.  For any eigensystem $\lambda : T^\pm \rightarrow \overline{F}^\times$ for the action of $\mathcal{H}_{p,m,b}^\pm$ on $ \HH^\star(S_K(\C), \mathcal{W}_\nu^\vee)^{\pm,fs}$ or $\HH^\star_c(S_K(\C), \mathcal{W}_\nu^\vee)^{\pm,fs}$, we have:
\begin{enumerate}
\item $v(\lambda) \geq -\nu$ in the $+$ case,
\item $v(\lambda) \leq -w_0\nu$ in the $-$ case.
\end{enumerate}
\end{prop}
\begin{proof} This is a straightforward adaptation of \cite{MR2869300}, prop. 3.1 which considers the case where $K_p$ is hyperspecial in an unramified group $G$. Recall that $G$ splits over $F$. The representation $W_\nu^\vee$ admits a lattice $W_\nu^{\vee,+}$ which is stable under the action of $G(\ocal_F)$ and admits  a weight decomposition with respect to the action of $T$.  For any $t \in T^+$, we have that $t (W_\nu^{\vee,+}) \subseteq (-\nu)(t) W_\nu^{\vee,+}$ and for any $t \in T^-$, we have that $t (W_\nu^{+,\vee}) \subseteq (-w_0\nu)(t) W_\nu^{+,\vee}$, as $-\nu$ and $-w_0\nu$ are the lowest and highest weights of $W_\nu^\vee$. The lattice $W_\nu^{+,\vee}$ gives an $\ocal_F$-local system $\mathcal{W}_\nu^{+,\vee}$ such that $\mathcal{W}_\nu^{+,\vee} \otimes_{\ocal_F} F = \mathcal{W}_\nu^\vee$.  The image of $ \HH^\star(S_K(\C), \mathcal{W}_\nu^{+,\vee})$ in $ \HH^\star(S_K(\C), \mathcal{W}_\nu^\vee)$ (resp.  of $ \HH_c^\star(S_K(\C), \mathcal{W}_\nu^{+,\vee})$ in $ \HH_c^\star(S_K(\C), \mathcal{W}_\nu^\vee)$) is a lattice $L$ (resp. $L_c$) by the finiteness of Betti Cohomology. For any $t \in T^+$, we find that $[K_p t K_p](L) \subseteq  (-\nu)(t) L$ and $[K_p t K_p](L_c)\subseteq  (-\nu)(t) L_c$. For any $t \in T^-$, we find that $[K_p t K_p](L) \subseteq  (-w_0\nu)(t) L$ and $[K_p t K_p](L_c)\subseteq  (-w_0\nu)(t) L_c$. 
\end{proof}

Using this proposition, we can prove conjecture \ref{conj-strongslopes-classical} when the weight is regular.

\begin{coro}\label{coro-lafforgue-estimates} Conjecture \ref{conj-strongslopes-classical} holds when $\kappa+\rho$ is $G$-regular.
\end{coro}
\begin{proof} Since $\kappa+\rho$ is $G$ regular, there is a unique $\nu \in X^\star(T)^+$ and a unique $v \in W$ such that $-\kappa-\rho =  v(\nu + \rho)$.  By the definition we have $C(\kappa)^+=\{w_{0,M}v\}$ and $C(\kappa)^-=\{vw_0\}$.

By the degeneration of Faltings's dual BGG spectral sequences (see for example \cite{harris-ann-arb} section 4, \cite{HZIII} Cor. 4.2.3), $\bigoplus_{w'\in \WM} \HH^{i-l(w')}(K^p,-w'w_0(\nu+\rho)-\rho, \chi)^{\pm,fs}$ embeds Hecke-equivariantly in  $\HH^i(S_K(\C),  \mathcal{W}_\nu^\vee)$ and $\bigoplus_{w'\in\WM} \HH^{i-l(w')}(K^p,-w'w_0(\nu+\rho)-\rho, \chi,cusp)^{\pm,fs}$ embeds Hecke-equivariantly in  $\HH^i_c(S_K(\C),  \mathcal{W}_\nu^\vee)$.   The estimate follows from proposition \ref{prop-lafforgue}. 
\end{proof}

\subsubsection{Connection with  \cite{F-Pilloni}}\label{section-connection-FP} We make a digression in order  to explain the relation between conjecture \ref{conj-strongslopes-classical}  and conjecture 4.5 of \cite{F-Pilloni} which concerns the action of  the spherical Hecke algebra on coherent cohomology at spherical level.  This conjecture is inspired by \cite{MR2869300}, and   is a translation of the Katz-Mazur inequality on the cohomology of algebraic varieties to the automorphic setting.  

Let $\Gamma$ be the Galois group of $F/\qq_p$, acting on  $X^\star(T)$. The restriction map $X^\star(T)_{\mathbb{R}} \rightarrow X^\star(T^d)_{\mathbb{R}}$ induces an isomorphism $X^\star(T)_{\mathbb{R}}^\Gamma \rightarrow X^\star(T^d)_{\mathbb{R}}$, with an inverse given by $\lambda \mapsto \frac{1}{\vert \Gamma \vert} \sum_{\sigma \in \Gamma} \sigma. \tilde{\lambda}$ where $\tilde{\lambda}\in X^\star(T)_{\mathbb{R}}$ is any lift of $\lambda\in X^\star(T^d)_{\mathbb{R}}$.

We can therefore identify $X^\star(T^d)_{\mathbb{R}}$ as a subspace of $X^\star(T)_{\mathbb{R}}$ and the partial order on $X^\star(T^d)_{\mathbb{R}}$ is the one induced by the partial order on $X^\star(T)_{\mathbb{R}}$. This is the point of view adopted in \cite{F-Pilloni}. 

We will assume that $G_{\qq_p}$ is of the form $\mathrm{Res}_{L/\qq_p} G_0$  where $L$ is a finite extension of $\qq_p$ and $G_0$ is an unramified reductive group over $L$. In \cite{F-Pilloni}, the group $G_{\qq_p}$ was assumed to be unramified, but the same conjecture can be made in this level of generality, and is interesting for applications.  We assume that $K_p \subseteq G(\qq_p) = G_0(L)$ is a hyperspecial subgroup of $G_0(L)$
We consider the classical cohomology $\mathrm{R}\Gamma(S_{K^pK_p, \Sigma}^{tor}, \mathcal{V}_\kappa)$ or $\mathrm{R}\Gamma(S_{K^pK_p, \Sigma}^{tor}, \mathcal{V}_\kappa(-D))$. We let $\infty(\kappa) \in X^\star(T)^+$ be the dominant representative of $-\kappa-\rho$ which is the infinitesimal character of automorphic representations contributing to $\mathrm{R}\Gamma(S_{K^pK_p, \Sigma}^{tor}, \mathcal{V}_\kappa)$ or $\mathrm{R}\Gamma(S_{K^pK_p, \Sigma}^{tor}, \mathcal{V}_\kappa(-D))$. Let $\mathcal{H}(G(\qq_p), K_p)$ be the spherical Hecke algebra. 

Let $T_0$ be a maximal torus of $G_0$. We can assume that $T = \mathrm{Res}_{L/\qq_p} T_0$. Let $T_0^d$ be the maximal split sub-torus of $T_0$. Then $T_0^d$ is naturally defined over $\qq_p$ and the diagonal map $T_0^d \hookrightarrow \mathrm{Res}_{L/\qq_p} T_0^d \hookrightarrow T$ indentifies $T_0^d$ and $T^d$. We let $W_0^d$ be the sub-group of the (geometric) Weyl group of $G_0$ which stabilizes $T_0^d$. 
We  fix an element $p^{\frac{1}{2}} \in F$ (this is always  possible if we enlarge $F$).  
We have the  Satake isomorphism:
$$ \mathcal{S} : \mathcal{H}(G(\qq_p), K_p) \otimes_{\ZZ} F \rightarrow F[X_\star(T_0^d)]^{W_0^d}$$
and to any $\lambda : \mathcal{H}(G(\qq_p), K_p) \rightarrow \overline{F}$, we can attach a semi-simple $\sigma$-conjugacy class $ c \in (X^\star(T_0^d) \otimes \bar{F}^\times)/W_0^d = ~^L{G_0}(\bar{F})^{ss}/\sigma-\mathrm{conj}$, where $~^LG_0(\bar{F}) = \hat{G_0} \rtimes \ZZ$ is the Langlands group of $G_0$. This is  the semi-direct product of the dual group $\hat{G_0}$ by the free group $\ZZ$ generated by $\sigma$. The action of $\sigma$ on $G_0^D$ is the one induced by the Frobenius which is a generator of the Galois group of  the unramified extension of $L$ which splits $G_0$. In particular, if $G_0$ is split, this action is trivial. 

Applying the valuation to $c$  gives the element  $\mathrm{Newt}_v(c) \in X^\star(T_0^d)_{\mathbb{R}}/W_0^d = X^\star(T_0)_{\mathbb{R}}^\Gamma/W_0^d =  X^\star(T_0^d)^+_{\mathbb{R}} = (X^\star(T_0)_{\mathbb{R}}^{ +})^{\Gamma}.$

We also remark that $T(\qq_p)/T(\ZZ_p) = T_0(L)/T_0(\ocal_L) = X_\star (T_0^d)$ via the map $\lambda \in  X_\star (T_0^d) \mapsto \lambda(\varpi_L)$ where $\varpi_L$ is a uniformizing element in $L$. We have also defined a valuation map $v : T(\qq_p)/T(\ZZ_p) \rightarrow X_\star (T^d) \otimes \qq$. We therefore get a  map $v : X_\star (T_0^d) \rightarrow 
 X_\star (T^d) \otimes \qq$. This map is given by multiplication by $v(\varpi_L)$ (via  the identification $X_\star (T_0^d) = X_\star(T^d)$). 

We have  the following conjecture (which is \cite{F-Pilloni}, conj. 4.5 in the unramified case):

\begin{conj}\label{conj-K-M} Let $\kappa \in X^\star(T^c)^{M_\mu, +}$. Let $c \in   ~^L{G_0}(\bar{F})^{ss}/\sigma-\mathrm{conj}$ arising via the Satake isomorphism from an eigensystem for the spherical Hecke alegbra action on   $\mathrm{R}\Gamma(S_{K^pK_p, \Sigma}^{tor}, \mathcal{V}_\kappa)$ or $\mathrm{R}\Gamma(S_{K^pK_p, \Sigma}^{tor}, \mathcal{V}_\kappa(-D))$ For any $t \in X_\star(T_0^d)_{\mathbb{R}}^+$, we have  
$$ \langle t,  \mathrm{Newt}_v(c) \rangle \leq \langle v(t),  -w_0  \frac{1}{\vert \Gamma \vert} \sum_{\sigma \in \Gamma} \sigma.  \infty(\kappa) \rangle.$$
\end{conj} 

\begin{rem} If $L$ is unramified, $X_\star(T_0^d)$ is canonically identified with $X_\star(T^d)$ via the valuation map $v$, and the above identity simply writes: 
$ \mathrm{Newt}_v(c)  \leq  -w_0  \frac{1}{\vert \Gamma \vert} \sum_{\sigma \in \Gamma} \sigma.  \infty(\kappa) $ in $X_\star(T)$. 
\end{rem}
We can reformulate this conjecture in another  way. Any element $t \in X_\star(T_0^d)^+$ gives a $\sigma$-equivariant representation of $\hat{G_0}$ (of highest weight $t$) and therefore a representation of $~^LG_0$. 
\begin{lem}\label{equivalentformconjecture}  The conjecture \ref{conj-K-M} holds if and only if,  for any $t \in X_\star(T_0^d)^+$ viewed as a dominant character of the Langlands group, and with associated highest weight representation $V_t$, we have that for any eigenvalue $x$ of $c$ on $V_t$, $$v(x) \geq - \langle v(t),  \infty(\kappa) \rangle.$$
\end{lem}

\begin{rem} We remark that $-\infty(\kappa)$ is the anti-dominant representative of $\kappa + \rho$. 
\end{rem}
\begin{proof} It follows from lemma 3.6 of \cite{F-Pilloni} that the conjecture is equivalent to the statement that $v(\mathrm{Tr} ( c  \vert V_t)) \geq \langle w_0(v(t)), -w_0\infty(\kappa) \rangle = \langle v(t), -\infty(\kappa) \rangle$. Therefore, the converse implication holds. Let us prove the direct implication.  If there is a unique eigenvalue of $c$ on $V_t$ with minimal   valuation, we deduce  that  for any eigenvalue $x$ of $c$ on $V_t$, $v(x) \geq v(\mathrm{Tr} ( c  \vert V_t))$. Otherwise, let $x_1, \cdots, x_i$ be the $i$-th eigenvalues of minimal valuation. Then one considers the representation $\Lambda^i V_t$ and we find 
$iv(x_1) = v(\mathrm{Tr} ( c \rtimes \sigma \vert  \Lambda^iV_t)) \geq i \langle v(t), -\infty(\kappa) \rangle$.
\end{proof}

Before we state our main compatibility, we need to recall certain relations between the spherical and Iwahori Hecke algebras. We have the spherical Hecke algebra $\mathcal{H}(G,K_p)$ and the Iwahori Hecke algebra $\mathcal{H}(G, K_{p,1,0})$.  These are algebras for the convolution product for a Haar measure normalized by  $\mathrm{vol} (K_p) = 1$ (respectively $\mathrm{vol} (K_{p,1,0}) = 1$).  We have also introduced a subalgebra $\mathcal{H}^+_{p,1,0}$ of $\mathcal{H}(G, K_{p,1,0})$, isomorphic to $\ZZ[T^+]$, and generated by the elements $[K_{p,1,0} t K_{p,1,0}]$ with $t \in T^+$.

We now consider the twisted embedding $$F[X_\star(T_0^d)^+] \hookrightarrow \mathcal{H}(G, K_{p,1,0}) \otimes_{\ZZ}F$$
which sends $t \in T^+/T(\ZZ_p) = X_\star(T_0^d)^+$ to $q^{ -\langle t, \rho_0 \rangle} [K_{p,1,0} t K_{p,1,0}]$ where $q$ is the cardinal of $\ocal_L/\varpi_L$ and $\rho_0$ is half the sum of the positive roots in $G_0$. All the operators $q^{ -\langle t, \rho \rangle} [K_{p,1,0} t K_{p,1,0}]$ are invertible in  $\mathcal{H}(G, K_{p,1,0}) \otimes_{\ZZ}F$ and this map extends to an embedding  $$F[X_\star(T_0^d)] \hookrightarrow \mathcal{H}(G, K_{p,1,0}) \otimes_{\ZZ} F.$$ Moreover, $F[X_\star(T_0^d)]^{W_d}$ is the center of $\mathcal{H}(G, K_{p,1,0}) \otimes_{\ZZ} F$. Let $e_{K_p} \in \mathcal{H}(G, K_{p,1,0})$ be the idempotent equal to  characteristic function of $K_p$ divided by the volume of $K_p$. The natural isomorphism:
$$ \mathcal{H}(G, K_p) \otimes F  \rightarrow  e_{K_p}(\mathcal{H}(G, K_{p,1,0}) \otimes F)e_{K_p}$$ induces an isomorphism  $$\mathcal{H}(G, K_p) \otimes F \rightarrow e_{K_p} F[X_\star(T_0^d)]^{W_d}$$ which is the Satake isomorphism.  

 \begin{coro} Let $\pi$ be an irreducible smooth representation of $G(\qq_p)$ defined over $\bar{F}$. Assume that $\pi^{K_p} \neq 0$. Then $\pi^{K_p}$ is one dimensional. Let $c \in (X^\star(T_0^d) \otimes \bar{F}^\times)/W_0^d$ be the  semi-simple $\sigma$-conjugacy class corresponding to the action of $\mathcal{H}(G,K_p)$ on $\pi^{K_p}$.  Any eigensystem of $F[X_\star(T_0^d)] $ acting on $\pi^{K_{p,1,0}}$ is given by a lift $\tilde{c} \in X^\star(T_0^d) \otimes \bar{F}^\times$ of $c$.  
\end{coro}

\begin{rem} In particular  for any $t \in T^+/T(\ZZ_p) = X_\star(T_0^d)^+$, the eigenvalues of $q^{ -\langle t, \rho_0 \rangle} [K_{p,1,0} t K_{p,1,0}]$  acting on $\pi^{K_{p,1,0}}$ are among the eigenvalues of $c$ acting on the representation $V_t$ of the Langlands dual group.
\end{rem}

\begin{prop}\label{prop-reductiontoLaff} 
Let $\kappa \in X^\star(T^c)^{M_\mu, +}$. 
Consider the submodule $\mathcal{H}(G,K_{p,1,0})\cdot  \HH^i(\mathcal{S}^{tor}_{K^pK_p,\Sigma}, \mathcal{V}_\kappa)$ of $\HH^i(\mathcal{S}^{tor}_{K^pK_{p,1,0},\Sigma}, \mathcal{V}_\kappa)$   generated by the cohomology at spherical level $ \HH^i(\mathcal{S}^{tor}_{K^pK_p,\Sigma}, \mathcal{V}_\kappa)$. If conjecture  \ref{conj-K-M} holds for all eigensystems in $ \HH^i(\mathcal{S}^{tor}_{K^pK_p,\Sigma}, \mathcal{V}_\kappa)$ then conjecture \ref{conj-strongslopes-classical} holds for all  eigensystems in $\mathcal{H}(G,K_{p,1,0})\cdot  \HH^i(\mathcal{S}^{tor}_{K^pK_p,\Sigma}, \mathcal{V}_\kappa)$. 
\end{prop} 
\begin{proof}  Let $c$ be the semi-simple conjugacy class arising from a spherical eigenclass. Let $\lambda$ be the character of $\mathcal{H}^+_{p,1,0}$ on the span of this eigenclass at Iwahori level. For each $t_0 \in T^+/T(\ZZ_p) = X_\star(T_0^d)^+$, we see that $\lambda(t_0) q^{ -\langle t_0, \rho_0 \rangle} $ is an eigenvalue for $c$ acting on $V_{t_0}$. It follows from lemma \ref{equivalentformconjecture} that $v(\lambda(t_0)) - v(q) \langle t_0, \rho_0 \rangle  \geq   \langle v(t_0), -\infty(\kappa) \rangle$. Where via the valuation map $v : X_\star(T_0^d) \rightarrow X_\star(T^d)\otimes \qq$, $t_0$ maps to $v(\varpi_L) t$, where we let $t$ be the element corresponding to $t_0$ via the isomorphism $X_\star(T_0^d) = X_\star(T^d)$. So this identity can be re-written: 
$$v(\lambda(t)) v(\varpi_L) - v(q) \langle t_0, \rho_0 \rangle  \geq  v( \varpi_L) \langle t, -\infty(\kappa) \rangle.$$
It remains to remark that $v(q) v(\varpi_L)^{-1} = [L:\qq_p]$ and that $\langle t_0, \rho_0 \rangle [L:\qq_p] = \langle t, \rho \rangle$. 
\end{proof} 

\subsection{Slopes and small slope conditions}\label{section-slope-conditions}
We will now define several ``small slope'' conditions that will occur in this paper.  The reason that there are so many conditions is explained as follows. First we consider small slope condition on the coherent cohomology in weight $\kappa$, but also conditions on the Betti cohomology in weight $-w_0\nu$. Second, the conditions needed to obtain a vanishing theorem are not exactly the same as those needed to obtain classicality theorems. For example, on a Shimura set the vanishing theorem is trivial and does not require any slope condition, but when one considers the theory of $p$-adic algebraic automorphic forms, there is a slope condition to achieve classicality. Finally, in our setting there are two types of control theorems. We have control theorems for cohomologies of classical automorphic sheaves, but also control theorems for cohomologies valued in Banach sheaves. All this explains the large variety of slope conditions we need.  We begin with the small slope condition and then turn to the strongly small slope condition which we need to use because we were not able to prove conjecture \ref{conj-strongslopes}. 
\subsubsection{Small slope conditions}

\begin{defi}\label{defi-ss}
Let $\lambda\in X^\star(T^d)_{\mathbb{R}}$.
\begin{itemize}
\item Let $\nu\in X^\star(T)$ satisfy $\nu+\rho\in X^\star(T)^+_{\mathbb{R}}$.
\begin{itemize}
\item We say $\lambda$ satisfies $+,ss(\nu)$ if for all $w\in W$ with $w\cdot\nu\not=\nu$, $$\lambda\not\geq-w\cdot\nu.$$
\item We say $\lambda$ satisfies $-,ss(\nu)$ if for all $w\in W$ with $w\cdot\nu\not=\nu$, $$\lambda\not\leq -w\cdot(w_0\nu).$$
\end{itemize}
\item Let $\kappa\in X^\star(T)^{M,+}$.
\begin{itemize}
\item We say $\lambda$ satisfies $+,ss^M(\kappa)$ if for all $w\in \WM\setminus C(\kappa)^+$, $$\lambda\not\geq w^{-1}w_{0,M}(\kappa+\rho)+\rho.$$
\item We say $\lambda$ satisfies $-,ss^M(\kappa)$ if for all $w\in \WM\setminus C(\kappa)^-$, $$\lambda\not\leq w^{-1}(\kappa+\rho)-\rho.$$
\end{itemize}

\item Let $\nu\in X^\star(T)^+$.
\begin{itemize}
\item We say $\lambda$ satisfies $+,ss_b(\nu)$ if for all $w\in \WM$, $\lambda$ satisfies $+, ss^M(-w_{0,M} w(\nu + \rho)- \rho)$. 
\item We say $\lambda$ satisfies $-,ss_b(\nu)$ if for all $w\in \WM$, $\lambda$ satisfies $-, ss^M(-w_{0,M} w(\nu + \rho)- \rho)$. 
\end{itemize}

\item Let $\kappa\in X^\star(T)$ and let $w\in \WM$.
\begin{itemize}
\item We say $\lambda$ satisfies $+,ss_{M,w}(\kappa)$ if for all $w'\in W_M$, $w'\not=1$, $$\lambda\not\geq w^{-1}w_{0,M}w'(\kappa+\rho)+\rho.$$
\item We say $\lambda$ satisfies $-,ss_{M,w}(\kappa)$ if for all $w'\in W_M$, $w'\not=1$, $$\lambda\not\leq w^{-1}w'(\kappa+\rho)-\rho.$$
\end{itemize}
\end{itemize}
\end{defi}

To orient the reader, we give a brief summary of how these conditions arise:
\begin{itemize}
\item The condition $\pm,ss^M(\kappa)$ will appear in the ``geometric'' classicality theorem relating classical cohomology and overconvergent cohomology, as well as vanishing theorems for classical cohomology in section \ref{subsection-classicality-vanishing}.
\item The condition $\pm,ss_{M,w}(\kappa)$ will arise in the second classicality theorem ``at the level of the sheaf'' relating overconvergent  and locally analytic cohomologies in algebraic weights in section \ref{subsection-control2}.
\item The condition $\pm,ss(\nu)$ is the usual small slope condition that arises in works on $p$-adic modular forms from the Betti perspective.  We shall see in proposition \ref{prop-small-slope-rel} below that it is the combination of the other two conditions.
\item The condition $\pm, ss_b(\nu)$ is the condition that arises in the vanishing theorems for classical Betti cohomology. We shall see in proposition \ref{prop-ssb-cond} that it is weaker than $\pm,ss(\nu)$. 
\end{itemize}

We first observe that the $+$ and $-$ conditions are related by two symmetries:
\begin{prop}\label{prop-plusminus-symmetry} Let $\lambda\in X^\star(T^d)_{\mathbb{R}}$.
\begin{enumerate}
\item Let $\nu\in X^\star(T)$ satisfy $\nu+\rho\in X^\star(T)^+_{\mathbb{R}}$.  Then the following are equivalent:
\begin{itemize}
\item $\lambda$ satisfies $-,ss(\nu)$.
\item $w_0(\lambda)$ satisfies $+,ss(\nu)$.
\item $-\lambda$ satisfies $+,ss(-w_0\nu)$.
\end{itemize}
\item Let $\nu\in X^\star(T)$ satisfy $\nu+\rho\in X^\star(T)^+_{\mathbb{R}}$.  Then the following are equivalent:
\begin{itemize}
\item $\lambda$ satisfies $-,ss_b(\nu)$.
\item $w_0(\lambda)$ satisfies $+,ss_b(\nu)$.
\item $-\lambda$ satisfies $+,ss_b(-w_0\nu)$.
\end{itemize}
\item Let $\kappa\in X^\star(T)^{M,+}$.  Then the following are equivalent:
\begin{itemize}
\item $\lambda$ satisfies $-,ss^M(\kappa)$.
\item $w_0(\lambda)$ satisfies $+,ss^M(\kappa)$.
\item $-\lambda$ satisfies $+,ss^M(-w_{0,M}\kappa-2\rho_{nc})$.
\end{itemize}
\item Let $\kappa\in X^\star(T)^{M,+}$ and let $w\in \WM$.  Then the following are equivalent:
\begin{itemize}
\item $\lambda$ satisfies $-,ss_{M,w}(\kappa)$.
\item $w_0(\lambda)$ satisfies $+,ss_{M,w_{0,M}ww_0}(\kappa)$.
\item $-\lambda$ satisfies $+,ss_{M,w}(-w_{0,M}\kappa-2\rho_{nc})$.
\end{itemize}
\end{enumerate}
\end{prop}
The first symmetry is related to the fact that when we have a smooth, admissible representation of $G(\qq_p)$, the action of $w_0$ exchanges the $+$ and $-$ finite slope parts (see section \ref{subsubsection-jacquet} below).  The second symmetry is related to Poincare and Serre duality.

Now we try to further explain the meaning of these small slope conditions and make them more explicit.  In view of the symmetries above we only consider the $+$ case.  For $\nu\in X^\star(T)$ we introduce the notation $W_\nu=\{w\in W\mid w\cdot\nu=\nu\}$.

We will use the following standard lemma.
\begin{lem}\label{lem-bruhat-inequality}
Let $\nu\in X^\star(T)^+_{\mathbb{R}}-\rho$ and let $w,w'\in W$.  If $w\leq w'$ then $w'\cdot\nu\preceq w\cdot\nu$.
\end{lem}
\begin{proof}
By the definition of the Bruhat order and induction, it suffices to treat the case that $w'=s_\alpha w$ with $\alpha\in\Phi^+$ with $l(w')>l(w)$, which implies that $w^{-1}\alpha\in\Phi^+$ by \cite{MR1066460} 5.7.  Then $w'\cdot\nu=w\cdot\nu-\langle\alpha^\vee,w(\nu+\rho)\rangle\alpha$, and $\langle\alpha^\vee,w(\nu+\rho)\rangle=\langle (w^{-1}\alpha)^\vee,\nu+\rho\rangle\geq0$.
\end{proof}

We now give some alternative characterizations of the condition $+,ss(\nu)$.

\begin{prop}\label{prop-ss-equiv}
The following conditions on $\lambda\in X^\star(T^d)_{\mathbb{R}}$ are equivalent:
\begin{enumerate}
\item $\lambda\not\geq-w\cdot\nu$ for all $w\in W\setminus W_\nu$, i.e. $\lambda$ satisfies $+,ss(\nu)$.
\item $\lambda\not\geq-s_\alpha\cdot\nu$ for all $\alpha\in\Delta$ with $s_\alpha\not\in W_\nu$.
\end{enumerate}
Moreover if we additionally assume that $\lambda\geq-\nu$ then we have the further equivalent condition:
\begin{enumerate}
\setcounter{enumi}{2}
\item $\lambda=-\nu+\sum_{\alpha\in\Delta_d}c_\alpha\alpha$ with $$c_\alpha<\min_{\beta\in r^{-1}(\alpha),s_\beta\not\in W_\nu}\langle\beta^\vee,\nu\rangle+1.$$
\end{enumerate}
\end{prop}
\begin{proof}
Clearly the first condition implies the second.  For the converse, given $w\in W\setminus W_\nu$, we have $w\geq s_\alpha$ for some $\alpha\in\Delta$, $s_\alpha\not\in W_\nu$ (write $w$ as a reduced product of simple reflections, not all the factors can fix $\nu$ as $w$ doesn't.)  Then $-w\cdot\nu\geq-s_\alpha\cdot\nu$ by lemma \ref{lem-bruhat-inequality} and so $\lambda\not\geq -s_\alpha\cdot\nu$ implies $\lambda\not\geq -w\cdot\nu$.

Under the hypothesis $\lambda\geq-\nu$, the equivalence of the second and third conditions is immediate from the formula $-s_{\beta}\cdot\nu=-\nu+(\langle\beta^\vee,\nu\rangle+1)\beta$.
\end{proof}

Now we consider the condition $+,ss^M(\kappa)$.  For $\kappa\in X^\star(T)$ we can write $-\kappa-\rho=v(\nu+\rho)$ for a unique $\nu\in X^\star(T)$ with $\nu+\rho\in X^\star(T)^+_{\mathbb{R}}$, and $v\in W$ uniquely determined up to right multiplication by $W_\nu$.
\begin{prop}\label{prop-ssnc-equiv}
Let $\kappa\in X^\star(T)^{M,+}$.  Then with $\nu$ and $v$ as above, the following conditions on $\lambda\in X^\star(T^d)_{\mathbb{R}}$ are equivalent:
\begin{enumerate}
\item $\lambda\not\geq w^{-1}w_{0,M}(\kappa+\rho)+\rho$ for all $w\in \WM \setminus C(\kappa)^+$, i.e. $\lambda$ satisfies $+,ss^M(\kappa)$.
\item $\lambda\not\geq -w\cdot\nu$ for all $w\in (\WM)^{-1} \cdot C(\kappa)^+\setminus W_\nu$.
\item $\lambda\not\geq -s_\alpha\cdot\nu$ for all $\alpha\in\Delta$ with $s_\alpha\in (\WM)^{-1}\cdot C(\kappa)^+\setminus W_\nu$.
\end{enumerate}
Moreover if we additionally assume that $\lambda\geq-\nu$ then we have the further equivalent condition:
\begin{enumerate}
\setcounter{enumi}{3}
\item $\lambda=-\nu+\sum_{\alpha\in\Delta_d}c_\alpha\alpha$ with $$c_\alpha<\min_{\beta\in r^{-1}(\alpha),s_\beta\in(\WM)^{-1}\cdot C(\kappa)^+\setminus W_\nu}\langle\beta^\vee,\nu\rangle+1.$$
\end{enumerate}
\end{prop}
\begin{proof}
The second condition is a direct translation of the first: we have $C(\kappa)^+=w_{0,M}vW_\nu$, and we can write
\begin{equation*}
w^{-1}w_{0,M}(\kappa+\rho)+\rho=(w^{-1}w_{0,M}v)v^{-1}(\kappa+\rho)+\rho=-(w^{-1}w_{0,M}v)\cdot\nu
\end{equation*}
and so the first condition is equivalent to $\lambda\not\geq -w\cdot\nu$ for $w\in (\WM)^{-1}w_{0,M}v\setminus W_\nu$, which is equivalent to condition 2 because $(\WM)^{-1}w_{0,M}vW_\nu=(\WM)^{-1}\cdot W_\nu$.

The second condition clearly implies the third.  For the converse, to argue as in the proof of proposition \ref{prop-ss-equiv} we need to show that for all $w\in W$ with $w\in(\WM)^{-1}\cdot C(\kappa)^+\setminus W_\nu$, we have $w\geq s_\alpha$ with $\alpha\in\Delta$ and $s_\alpha\in(\WM)^{-1}\cdot C(\kappa)^+\setminus W_\nu$.  To do this, suppose $w=(w')^{-1}w''$ with $w'\in\WM$, $w''\in C(\kappa)^+$, and let $w=s_1\cdots s_n$ be a reduced expression as a product of simple reflections.  Choose $k$ such that $s_k\not\in W_\nu$ but $s_{k+1},\ldots,s_n\in W_\nu$ (such a $k$ exists as $w\not\in W_\nu$.)  Then
$$s_k=(w's_1\cdots s_{k-1})^{-1}(w''s_{k+1}\cdots s_n)\in (\WM)^{-1}\cdot C(\kappa)^+$$
using lemma \ref{lem-WM-gallery} below to see that $w's_1\cdots s_{k-1}\in\WM$.  Moreover $s_k\leq w$, and $s_k\not\in W_\nu$, so we are done.

The equivalence of the third and fourth conditions is exactly as in proposition \ref{prop-ss-equiv}.
\end{proof}

\begin{lem}\label{lem-WM-gallery}
Let $w,w'\in\WM$ and let $w^{-1}w'=s_1\cdots s_n$ be a reduced expression as a product of simple roots.  Then $ws_1\cdots s_i\in\WM$ for all $1\leq i\leq n$.
\end{lem}
\begin{proof}
We begin with the following claim: if $\alpha\in \Phi^+$, $\beta\in\Delta$, and $u\in W$ are such that $l(us_\beta)>l(u)$ and $u^{-1}\alpha\in\Phi^-$, then $(us_\beta)^{-1}\alpha\in\Phi^-$.  Indeed, $(us_\beta)^{-1}\alpha=s_\beta(u^{-1}\alpha)\in\Phi^+$ would imply $u^{-1}(\alpha)=-\beta$, and hence $u\beta\in\Phi^-$, contradicting $l(us_\beta)>l(u)$.

Applying the claim inductively we see that if $\alpha\in\Phi^+$, and $(s_1\cdots s_i)^{-1}\alpha\in\Phi^-$, then $(s_1\cdots s_n)^{-1}\alpha\in\Phi^-$.

Now if $ws_1\cdots s_i\not\in\WM$, there exists $\beta\in\Delta_M$ with $(ws_1\cdots s_i)^{-1}\beta=(s_1\cdots s_i)^{-1}(w^{-1}\beta)\in\Phi^-$.  But $w\in \WM$ implies $w^{-1}\beta\in\Phi^+$, and so from the above with $\alpha=w^{-1}\beta$ we deduce $(s_1\cdots s_n)^{-1}(w^{-1}(\beta))={w'}^{-1}(\beta)\in\Phi^-$, and hence $w'\not\in\WM$.
\end{proof}

Now we turn to the conditions $+,ss_{M,w}(\kappa)$.  Let $\kappa\in X^\star(T)^+$ and let $w\in \WM$.  We let $v=w_{0,M}w$ and we let $\nu$ be defined by the formula $v(\nu+\rho)=-\kappa-\rho$.  Note that we have $\nu+\rho\in X^\star(T)^+_{\mathbb{R}}$ if and only if $w\in C(\kappa)^+$, which we have not assumed for the moment.
\begin{prop}\label{prop-ssc-equiv}
Let $\kappa\in X^\star(T)^{M,+}$ and let $w\in W$.  Let $\nu$ and $v$ be as above.  The following conditions on $\lambda\in X^\star(T^d)_{\mathbb{R}}$ are equivalent:
\begin{enumerate}
\item $\lambda\not\geq w^{-1}w_{0,M}w'(\kappa+\rho)+\rho$ for all $w'\in W_M$, $w'\not=1$, i.e. $\lambda$ satisfies $+,ss_{M,w}(\kappa)$
\item $\lambda\not\geq -(v^{-1}w'v)\cdot \nu$ for all $w'\in W_M$, $w'\not=1$.
\item $\lambda\not\geq -(v^{-1}s_\alpha v)\cdot\nu$ for all $\alpha\in\Delta_M$.
\end{enumerate}
\end{prop}
\begin{proof}
The equivalence of the first and second conditions is a direct translation.

For the equivalence of the second and third points, we introduce temporarily the notation $\lambda_1\preceq_{M,w}\lambda_2$ if $\lambda_2-\lambda_1\in\mathbb{R}_{\geq0}w^{-1}\Delta_M$ for $\lambda_1,\lambda_2\in X^\star(T)_{\mathbb{R}}$.  Since $w\in\WM$ we have $w^{-1}\Delta_M\subseteq \Phi^+$ and hence $\lambda_1\preceq_{M,w}\lambda_2$ implies $\lambda_1\preceq\lambda_2$ and hence $\lambda_1\leq\lambda_2$.

Now applying lemma \ref{lem-bruhat-inequality} for the group $M$, we see that for $\nu'\in X^\star(T)_{\mathbb{R}}^{-,M}-\rho_M$ and $w',w''\in W_M$ with $w'\leq w''$ we have $w'\cdot\nu'\preceq_{M,1}w''\cdot\nu'$, hence $(w_{0,M}w'')\cdot\nu'\preceq_{M,1}(w_{0,M}w'')\cdot\nu'$, and hence $(w^{-1}w_{0,M}w'')\cdot\nu'\preceq_{M,w}(w^{-1}w_{0,M}w')\cdot\nu'$.  We apply this with $\nu'=v\cdot\nu=-\kappa-\rho\in X^\star(T)^{-,M}_{\mathbb{R}}-\rho_M$ to deduce that for $w'\leq w''$ we have $-(v^{-1}w''v)\cdot\nu\geq -(v^{-1}w'v)\cdot\nu$, and hence the third condition implies the second.
\end{proof}

\begin{rem}
Note that for $w\in W_M$, $(v^{-1}wv)\cdot\nu=\nu$ implies $w=1$.  Indeed, then $w\cdot(-\kappa-2\rho)=-\kappa-2\rho$ or equivalently $w(-\kappa-\rho_M)=-\kappa-\rho_M$, and hence that $w=1$, since $\kappa\in X^\star(T)^{M,+}$.
\end{rem}

Note that we have now expressed all the small slope conditions as $\lambda\not\geq-w\cdot\nu$ as $w$ ranges over a certain subset of $W$.  We may use this to compare them.

\begin{prop}\label{prop-small-slope-rel}
Let $\kappa\in X^\star(T)^{M,+}$ and let $w\in C(\kappa)^+$.  Let $v=w_{0,M}w$ and let $\nu$ be given by $v(\nu+\rho)=-\kappa-\rho$ so that $\nu+\rho\in X^\star(T)^+_{\mathbb{R}}$.  Then a slope $\lambda\in X^\star(T^d)_{\mathbb{R}}$ satisfies $+,ss(\nu)$ if and only if it satisfies both $+,ss^M(\kappa)$ and $+,ss_{M,w}(\kappa)$.
\end{prop}
\begin{proof}
We apply the characterizations of propositions \ref{prop-ss-equiv}, \ref{prop-ssnc-equiv}, \ref{prop-ssc-equiv}.  Then it is clear that $+,ss(\nu)$ implies both $+,ss^M(\kappa)$ and $+,ss_{M,w}(\kappa)$.  For the other direction, we need to show that for each $\alpha\in\Delta$, then either $s_\alpha\in (\WM)^{-1}C(\kappa)^+$ or $vs_\alpha v^{-1}\in W_M$.

We apply lemma \ref{lem-weyl-alternative} to $w\in\WM$ and $s_\alpha$.  If $ws_\alpha\in\WM$ then $s_\alpha=(ws_{\alpha})^{-1}w\in(\WM^{-1})w$.  Otherwise there exists $\beta\in\Delta_M$ so that $s_\beta w=ws_\alpha$ and hence $vs_\alpha v^{-1}=w_{0,M}^{-1}s_\beta w_{0,M}\in W_M$.
\end{proof}

\begin{lem}\label{lem-weyl-alternative}
Let $w\in\WM$ and $\alpha\in\Delta$.  Then either $ws_\alpha\in\WM$ or $ws_\alpha=s_\beta w$ for $\beta\in\Delta_M$.
\end{lem}
\begin{proof}
If $ws_\alpha\not\in\WM$ then there exists $\beta\in\Delta_M$ with $s_\alpha(w^{-1}(\beta))=(ws_\alpha)^{-1}(\beta)\in\Phi^-$.  But $w^{-1}(\beta)\in\Phi^+$.  Hence $w^{-1}(\beta)=\alpha$, so $w^{-1}s_\beta w=s_\alpha$, hence $ws_\alpha=s_\beta w$.
\end{proof}

We can write $\Phi=\coprod_i\Phi_i$ where the $\Phi_i$ are simple root systems.  Then $\Phi_M=\coprod_i\Phi_{M,i}$ where $\Phi_{M,i}=\Phi_M\cap\Phi_i$.  We let $\Phi_b=\coprod_{i,\Phi_i\not=\Phi_{M,i}}\Phi_i$ be the union of the simple factors where $M$ is a proper levi.  Let $\Delta_b=\Phi_b\cap \Delta$.
\begin{prop}\label{prop-ssb-cond}
Suppose that $\nu\in X^\star(T)^+$.  Then the following conditions on a slope $\lambda\in X^\star(T^d)_{\mathbb{R}}$ are equivalent.
\begin{enumerate}
\item $\lambda$ satisfies $+,ss^{M}(-w_{0,M}w(\nu+\rho)-\rho)$ for all $w\in \WM$.
\item $\lambda\not\geq -s_{\alpha}\cdot\nu$ for all $\alpha\in\Delta_b$.
\end{enumerate}
\end{prop}
\begin{proof}
We need to show that for $\alpha\in\Delta$, we have $\alpha\in\Delta_b$ if and only if there is $w\in\WM$ so that $ws_\alpha w^{-1}\not\in W_M$.

The later condition is equivalent to $w\alpha\not\in \Phi_M$ for all $w\in W$ (since if we write $w=w_Mw^M$ then $w\alpha\in \Phi_M$ implies $w^M\alpha\in\Phi_M$) but $w\alpha$ for $w\in W$ span $\mathbb{Q}\Phi_i$ where $\Phi_i$ is the simple factor containing $\alpha$ and so the only way that we will have $w\alpha\in\Phi_M$ for all $w\in W$ is if $\Phi_i=\Phi_{i,M}$.
\end{proof}

\subsubsection{The strongly small slope conditions}
We now introduce some slightly stronger versions of the small slope conditions of the last section.  We need these because we cannot prove the slope bounds of conjecture \ref{conj-strongslopes}, but only the weaker bounds of theorem \ref{thm-slopes}.  

\begin{defi}\label{defi-sss}
Let $\lambda\in X^\star(T^d)_{\mathbb{R}}$.
\begin{itemize}
\item Let $\kappa\in X^\star(T)^{M,+}$.
\begin{itemize}
\item We say $\lambda$ satisfies $+,sss^M(\kappa)$ if for all $w\in \WM\setminus C(\kappa)^+$, $$\lambda\not\geq w^{-1}w_{0,M}\kappa~\textrm{or}~\lambda\not\geq w^{-1}w_{0,M}\kappa + w^{-1}2\rho_{nc}.$$
\item We say $\lambda$ satisfies $-,sss^M(\kappa)$ if for all $w\in \WM\setminus C(\kappa)^-$, $$\lambda\not\leq w^{-1}\kappa~\textrm{or}~\lambda\not\leq w^{-1}\kappa + w^{-1}2\rho_{nc}.$$
\end{itemize}
\item Let $\kappa\in X^\star(T)$ and let $w\in \WM$.
\begin{itemize}
\item We say $\lambda$ satisfies $+,sss_{M,w}(\kappa)$ if for all $w'\in W_M$, $w'\not=1$, $$\lambda\not\geq w^{-1}w_{0,M}w'(\kappa)~\textrm{or}~\lambda \not \geq w^{-1}w_{0,M}w'(\kappa) + 2 w^{-1} w_{0,M} w' 2\rho_{nc}.$$
\item We say $\lambda$ satisfies $-,sss_{M,w}(\kappa)$ if for all $w'\in W_M$, $w'\not=1$, $$\lambda\not\leq w^{-1}(w'(\kappa+\rho)-\rho) ~\textrm{or}~\lambda \not \leq w^{-1}w'(\kappa) + 2 w^{-1} w' 2\rho_{nc}.$$
\end{itemize}
\end{itemize}
\end{defi}
It is immediate from the definitions that $\lambda$ satisfies $\pm,sss^M(\kappa)$ if and only if $w_0(\lambda)$ satisfies $\mp,sss^M(\kappa)$ and if and only if  $-\lambda$ satisfies $\mp,sss^M(-w_{0,M}\kappa-2\rho_{nc})$.

We introduce combinations of these conditions motivated by propositions \ref{prop-ssb-cond} and \ref{prop-small-slope-rel}.
\begin{itemize}
\item We say that $\lambda$ satisfies $+,sss_w(\nu)$ if it satisfies $+,sss^M(-w_{0,M}w(\nu+\rho)-\rho)$ and $+,sss_{M,w}(-w_{0,M}w(\nu+\rho)-\rho)$.  We say that $\lambda$ satisfies $-,sss_w(\nu)$ if it satisfies $-,sss^M(-ww_0(\nu+\rho)-\rho)$ and $-,sss_{M,w}(-ww_0(\nu+\rho)-\rho)$.
\item We say that $\lambda$ satisfies $\pm,sss_b(\nu)$ if it satisfies $\pm,sss^M(-w_{0,M}w(\nu+\rho)-\rho)$ for all $w\in\WM$.
\end{itemize}

\subsection{Small slopes, classicality, and vanishing}\label{subsection-classicality-vanishing}

We  introduce a notation. Let $M$ be a module or a complex carrying an action of $T(\qq_p)$ and admitting a slope decomposition. Let $? $ be a condition on the slope of characters of $T(\qq_p)$ (for instance the condition introduced in definition \ref{defi-ss}), then $M^? $ means the factor of $M$ which satisfies the condition $?$. 

\subsubsection{Coherent cohomology}

Theorem \ref{thm-slopes} and the definition of the small slope condition implies the following vanishing.
\begin{coro}\label{coro-theorem-slopes}
Let $\kappa\in X^\star(T^c)^{M_\mu,+}$ and let $\chi:T(\ZZ_p)\rightarrow \overline{F}^\times$ be a finite order character.  Then $$\mathrm{R}\Gamma_w(K^p,\kappa,\chi)^{\pm,sss^M(\kappa)}=\mathrm{R}\Gamma_w(K^p,\kappa,\chi,cusp)^{\pm,sss^M(\kappa)}=0$$
for $w\not\in C(\kappa)^\pm$.
\end{coro}

This implies that if we take the strongly small slope part of the spectral sequences from finite slope overconvergent cohomology to finite slope classical cohomology of Theorem \ref{thm-spectral-sequence}, all the terms for $w\not\in C(\kappa)^\pm$ vanish.  We immediately deduce our first main classicality theorem.

\begin{thm}\label{thm-control-thm-Coleman}
Let $\kappa\in X^\star(T^c)^{M_\mu,+}$ and assume that $\kappa+\rho$ is regular so that $C(\kappa)^{\pm} = \{w_{\pm}\}$.    Let $\chi:T(\ZZ_p)\rightarrow {\overline{F}}^\times$ be a finite order character. 
Then the spectral sequences of Theorem \ref{thm-spectral-sequence} induces isomorphisms
$$\mathrm{R}\Gamma_{w_\pm}(K^p,\kappa,\chi)^{\pm,sss^M(\kappa)}\simeq \mathrm{R}\Gamma(K^p, \kappa, \chi)^{\pm,sss^M(\kappa)}$$
$$\mathrm{R}\Gamma_{w_\pm}(K^p,\kappa,\chi,cusp)^{\pm,sss^M(\kappa)}\simeq \mathrm{R}\Gamma(K^p, \kappa, \chi,cusp)^{\pm,sss^M(\kappa)}$$
\end{thm}

\begin{rem} Cases of this theorem for the degree $0$  cohomology of  PEL Shimura varieties  were already proven. See for example  \cite{MR1369416}, \cite{MR2219265}, \cite{MR2783930}, \cite{MR3488741}. 
\end{rem}

We also deduce vanishing theorems for classical cohomology.

\begin{thm}\label{thm-first-vanishing-sss} Let $\kappa\in X^\star(T^c)^{M_\mu,+}$ and let $\chi:T(\ZZ_p)\rightarrow \overline{F}^\times$ be a finite order character.
\begin{enumerate}
\item $\mathrm{R}\Gamma(K^p, \kappa, \chi, cusp)^{\pm,sss^M(\kappa)}$ is concentrated in degree $[0,\ell_{\max}(\kappa)]$.
\item $\mathrm{R}\Gamma(K^p, \kappa, \chi)^{\pm,sss^M(\kappa)}$   is  concentrated in degree $[\ell_{\min}(\kappa), d]$. 
\item $\overline{\HH}^i(K^p, \kappa, \chi)^{\pm, sss^M(\kappa)}$   is concentrated in degree $[\ell_{\min}(\kappa),\ell_{\max}(\kappa)]$.
\end{enumerate}
\end{thm}

\begin{proof}  
The vanishing result  follows from the spectral sequences of theorem \ref{thm-spectral-sequence} and the vanishing results of propositions \ref{prop-deg-spect-sequ1} and \ref{prop-deg-spect-sequ2}, together with theorem \ref{thm-slopes}. 
\end{proof}

\begin{rem} If we assume the conjecture \ref{conj-strongslopes}, then the strongly small slope conditions in theorems \ref{thm-first-vanishing-sss} 
can be weakened to small slope condition.  In theorem \ref{thm-strong-vanishing-ss} we will actually be able to prove the theorems  for the small slope condition for the interior cohomology only, using the eigenvariety.  
\end{rem}

\begin{rem}
In \cite{MR3571345}, an analog of theorem \ref{thm-first-vanishing-sss} 
is proved without any small slope condition, but with a regularity condition on the weight $\kappa$  instead.
\end{rem}

\begin{rem}
The classical coherent  cohomology can be computed in terms of automorphic forms for $G$ by the result of Su \cite{su2018coherent}.  Thus it may be possible to reprove theorem \ref{thm-first-vanishing-sss}
with sufficient knowledge of automorphic forms on $G$.
\end{rem}

The following corollary gives a situation for which the interior Cousin complex computes the classical interior cohomology. 

\begin{coro}\label{coro-cousin-exactly-computes} Let $\kappa\in X^\star(T^c)^{M_\mu,+}$ and assume that $\kappa+\rho$ is regular. 
Then  $\overline{\mathcal{C}ous}(K^p, \kappa, \chi)^{\pm, sss^M(\kappa)}$ computes  $\overline{\HH}^\star(K^p, \kappa, \chi)^{\pm,sss^M(\kappa)}$.
\end{coro}
\begin{proof} We have $C(\kappa)^{\pm} = \{w_{\pm}\}$. We deduce that $$\overline{\mathcal{C}ous}(K^p, \kappa, \chi)^{\pm, sss^M(\kappa)} =$$$$ \mathrm{Im} \big( \HH^{\ell_\pm(w)}_w(K^p, \kappa, \chi, cusp)^{\pm, sss^M(\kappa)} \rightarrow $$ $$ \HH^{\ell_{\pm}(w)}_w (K_p, \kappa, \chi)^{\pm, sss^M(\kappa)} \big)[-\ell_{\pm}(w)]$$ which computes $\overline{\HH}^\star(K^p, \kappa, \chi)^{\pm,sss^M(\kappa)}$ by theorem \ref{thm-control-thm-Coleman} and theorem \ref{thm-first-vanishing-sss}.
\end{proof}

\subsubsection{Betti cohomology}

Let $\nu\in X^\star(T^c)^+$. We let $W_\nu$ be the corresponding irreducible representation of $G$ with highest weight $\nu$ and $W_\nu^\vee$ be its contragredient. We have an associated local system $\mathcal{W}^\vee_\nu$  over $S_K(\C)$. 

Using Faltings's dual BGG spectral sequence we deduce vanishing results for the small slope parts of Betti cohomology.
\begin{thm}\label{thm-betti-sss-vanishing}
Let $\nu\in X^\star(T^c)^+$.
\begin{enumerate}
\item $\HH^i(S_K(\mathbb{C}),\mathcal{W}_\nu^\vee)^{\pm,sss_b(\nu)}$ is concentrated in degree $[d,2d]$.
\item $\HH_c^i(S_K(\mathbb{C}),\mathcal{W}_\nu^\vee)^{\pm,sss_b(\nu)}$  is concentrated in degree $[0,d]$.
\item $\overline{\HH}^i(S_K(\mathbb{C}),\mathcal{W}_\nu^\vee)^{\pm,sss_b(\nu)}$ is concentrated in degree $d$.
\end{enumerate}
\end{thm}

\begin{rem} If we assume the conjecture \ref{conj-strongslopes}, then the strongly small slope conditions in \ref{thm-betti-sss-vanishing} can be weakened to small slope condition.  In theorem \ref{thm-strong-vanishing-ss} we will actually be able to prove the theorems  for the small slope condition for the interior cohomology only, using the eigenvariety.  
\end{rem}

\begin{rem}
In \cite{MR3571345}, analogs of theorem  \ref{thm-betti-sss-vanishing} are proved without any small slope condition, but with a regularity condition on the weight $\nu$ instead.
\end{rem}

\begin{rem}
The classical  Betti cohomology can be computed in terms of automorphic forms for $G$ by the results of Franke \cite{MR1603257}.  Thus it may be possible to reprove theorem  \ref{thm-betti-sss-vanishing} with sufficient knowledge of automorphic forms on $G$.
\end{rem}

\subsubsection{Small slope conditions and Jacquet modules}\label{subsubsection-jacquet} We now use these small slope condition to define certain direct summands of smooth admissible representations and apply this to the cohomology of the Shimura variety. 

\begin{prop}\label{prop-jacquet-ss}
Let $\pi$ be a smooth admissible representation of $G(\qq_p)$ and let $\nu\in X^\star(T)$ satisfy $\nu+\rho\in X^\star(T)_{\mathbb{R}}^+$.  Then the following are equivalent:
\begin{enumerate}
\item There exists $m\geq b\geq 0$ such that $\pi^{K_{p,m,b},+,ss(\nu)}\not=0$.
\item There exists $m\geq b\geq 0$ such that $\pi^{K_{p,m,b},-,ss(\nu)}\not=0$.
\item $(\pi_U)^{+,ss(\nu)}\not=0$.
\item $(\pi_{\overline{U}})^{-,ss(\nu)}\not=0$.
\end{enumerate}

We have the same equivalent properties when the condition $ss(\nu)$ replaced by $ss_b(\nu)$, $sss_b(\nu)$, or $ss^M(\kappa)$, $sss^M(\kappa)$ for $\kappa \in X^\star(T)^{M,+}$. 
\end{prop}
\begin{proof}
The equivalence of (1) with (3) and (2) with (4) is immediate from proposition \ref{prop-casell-JM}.  The equivalence of (3) and (4) follows from proposition \ref{prop-plusminus-symmetry} and the isomorphism between $\pi_U$ and $\pi_{\overline{U}}$ given by $w_0$.
\end{proof}

\begin{defi}\label{defi-small-slope-rep}
Let $\pi$ be a smooth admissible representation of $G(\qq_p)$, let $\nu\in X^\star(T)$ satisfy $\nu+\rho\in X^\star(T)_{\mathbb{R}}^+$ and let $\kappa \in X^\star(T)^{M,+}$.  We define $\pi^{ss(\nu)}\subseteq\pi$ to be the sum of all indecomposable summands of $\pi$ which satisfy the equivalent conditions of proposition \ref{prop-jacquet-ss} for $ss(\nu)$. We define $\pi^{ss_b(\nu)}$, $\pi^{ss^M(\kappa)}$, $\pi^{sss_b(\nu)}$, and $\pi^{sss^M(\kappa)}$ similarly.
\end{defi}

\begin{rem} It is not necessarily true that any irreducible subquotient of $\pi^{ss(\nu)}$ satisfies the condition $ss(\nu)$ and similarly for the other conditions. 
\end{rem}
\begin{rem} If $\pi$ is irreducible, then $\pi^{ss(\nu)} = \pi$ means  that $\pi$ admits an embedding in $\iota_{B}^G\psi$ for a character $\psi : T(\qq_p) \rightarrow \overline{\qq}_p^\times$ with $v(\psi)$ satisfying $+,ss(\nu)$.  A similar remark holds for the other slope conditions.
\end{rem}

With these definition in place, we can deduce (most of) theorems \ref{thm-1-intro} and \ref{thm-2-intro}  of the introduction (we will be able to use the small slope condition for interior cohomology after we prove theorem \ref{thm-strong-vanishing-ss}):  

\begin{thm}\label{thm-12-main}  For any $\kappa \in X_\star(T^c)^{M_\mu,+}$,
\begin{enumerate}
\item  $\overline{\HH}^i(K^p, \kappa)^{sss^M(\kappa)}$ is concentrated in the range  $[\ell_{\min}(\kappa), \ell_{\max}(\kappa)]$,
\item ${\HH}^i(K^p, \kappa, cusp)^{sss^M(\kappa)}$ is concentrated in the range  $[0, \ell_{\max}(\kappa)]$,
\item ${\HH}^i(K^p, \kappa)^{sss^M(\kappa)}$ is concentrated in the range  $[\ell_{\min}(\kappa),d]$.
\end{enumerate}

For any $\nu \in X_\star(T^c)^{+}$,
\begin{enumerate}
\item $\overline{\HH}^i(K_p, \mathcal{W}_\nu^\vee)^{sss_b(\nu)}$ is concentrated in the middle degree $d$, 
\item ${\HH}_c^i(K^p, \mathcal{W}_\nu^\vee)^{sss_b(\nu)}$ is concentrated in the range  $[0, d]$,
\item ${\HH}^i(K^p, \mathcal{W}_\nu^\vee)^{sss_b(\nu)}$ is concentrated in the range  $[d,2d]$.
\end{enumerate}
\end{thm}
\begin{proof} This follows immediately from theorems \ref{thm-first-vanishing-sss} and \ref{thm-betti-sss-vanishing}.
\end{proof}

\subsection{De Rham and rigid cohomology}
 Let $\nu\in X^\star(T^c)^+$. Let $(\mathcal{W}^\vee_{\nu, dR}, \nabla)$ be the associated  filtered vector bundle with integrable logarithmic connection over $S_{K,\Sigma}^{tor}$.  Over $S_K(\C)$, the set of horizontal sections of the corresponding holomorphic vector bundle is the local system $\mathcal{W}^\vee_\nu$.
 Let $\mathrm{DR}(\mathcal{W}_\nu^\vee) =  \mathcal{W}_{\nu,dR}^\vee \otimes_{\oscr_{S^{tor}_{K,\Sigma}} }\Omega^\bullet_{S^{tor}_{K,\Sigma}/F}(\mathrm{log}(D))$ be the filtered de Rham complex with logarithmic poles associated to $(\mathcal{W}^\vee_{\nu, dR}, \nabla)$. Its cohomology will be denoted   $\mathrm{R}\Gamma_{dR}(S_{K,\Sigma}^{tor}, \mathcal{W}_\nu^\vee)$.  We also consider the sub-complex  $\mathrm{DR}(\mathcal{W}_\nu^\vee(-D))$ and   its cohomology will be denoted   $\mathrm{R}\Gamma_{dR,c}(S_{K,\Sigma}^{tor}, \mathcal{W}_\nu^\vee)$ where the subscript $c$ stands for compact support. 
 
 Faltings's dual BGG  complex  for $\mathcal{W}_\nu^\vee$  is a filtered     complex $\mathrm{BGG}(\mathcal{W}_\nu^\vee)$ in the category of vector bundles with maps given by differential operators. See for example \cite{MR3818616} or \cite{lan2019rham}, sect. 6.1. We have 
 $\mathrm{BGG}(\mathcal{W}_\nu^\vee)^i = \oplus_{w \in \WM, \ell(w)= i} \mathcal{V}_{-ww_0( \nu + \rho)-\rho}$. 
 We also have a subcomplex  $\mathrm{BGG}(\mathcal{W}_\nu^\vee(-D))$ with $$\mathrm{BGG}(\mathcal{W}_\nu^\vee(-D))^i = \oplus_{w \in \WM, \ell(w)= i} \mathcal{V}_{-ww_0( \nu + \rho)-\rho}(-D).$$
We have (see for example \cite{lan2019rham}, thm. 6.1.10):
 \begin{thm}\label{thm-H-dR} There is a filtered quasi-isomorphism $\mathrm{BGG}(\mathcal{W}_\nu^\vee) \rightarrow \mathrm{DR}(\mathcal{W}_\nu^\vee)$ in the category of vector bundles over ${S}_{K,\Sigma}^{tor}$, with morphisms given by differential operators. The stupid filtration on  $\mathrm{BGG}(\mathcal{W}_\nu^\vee)$ induces   is a  spectral sequence $$E_1^{p,q} = \oplus_{w \in \WM, \ell(w)=p} \HH^q(S_{K,\Sigma}^{tor}, \mathcal{V}_{-ww_0( \nu + \rho)-\rho}) \Rightarrow \HH^{p+q}_{dR}( S_{K,\Sigma}^{tor}, \mathcal{W}_\nu^\vee)$$ degenerating at $E_1$. 
  There is a quasi-isomorphism $\mathrm{BGG}(\mathcal{W}_\nu^\vee(-D)) \rightarrow \mathrm{DR}(\mathcal{W}_\nu^\vee(-D))$. The stupid filtration on  $\mathrm{BGG}(\mathcal{W}_\nu^\vee(-D))$ induces a spectral sequence $$E_1^{p,q} = \oplus_{w \in \WM, \ell(w)=p} \HH^q(S_{K,\Sigma}^{tor}, \mathcal{V}_{-ww_0( \nu + \rho)-\rho}(-D)) \Rightarrow \HH^{p+q}_{dR,c}( S_{K,\Sigma}^{tor}, \mathcal{W}_\nu^\vee)$$ degenerating at $E_1$. 
 \end{thm}

\begin{rem} Instead of using the stupid filtration on $\mathrm{BGG}(\mathcal{W}_\nu^\vee)$ one can use the filtration $F$ corresponding to the Hodge filtration on $\mathrm{DR}(\mathcal{W}_\nu^\vee)$. The associated graded of this  filtration $F$ are complexes of automorphic vector bundles (featuring those appearing as objects in $\mathrm{BGG}(\mathcal{W}_\nu^\vee)$), with trivial differential. The spectral sequence for the $F$-filtration is then the Hodge-to-de Rham spectral sequence. It also degenerates at $E_1$. The difference between the Hodge-to-de Rham spectral sequence and the stupid filtration spectral sequence  is therefore basically a reindexing of the terms. 
\end{rem}

We now pass to $p$-adic geometry. We can consider   $\mathrm{DR}(\mathcal{W}_\nu^\vee)$ and $\mathrm{BGG}(\mathcal{W}_\nu^\vee)$  as complexes of vector bundles with maps given by differential operators over the adic space $\mathcal{S}^{tor}_{K,\Sigma}$ and the GAGA theorem ensures that  $\mathrm{R}\Gamma_{dR}(\mathcal{S}_{K,\Sigma}^{tor}, \mathcal{W}_\nu^\vee)$ still computes the algebraic  de Rham cohomology groups. If $K = K^p K_p$ with $K_{p} = K_{p,m,b}$, and $\chi:T(\ZZ_p)\to {\overline{F}}^\times$  is a finite order character,  we can  define    $\mathrm{R}\Gamma_{dR}(K^p, \mathcal{W}_\nu^\vee,\chi)^{\pm,fs}$  as a direct factor of the complex of $\mathrm{R}\Gamma_{dR}(\mathcal{S}_{K^pK_p,\Sigma}^{tor}, \mathcal{W}_\nu^\vee)$ (see section \ref{sect-finite-slope-classical-coho}). We can also define $\mathrm{R}\Gamma_{dR,c}(K^p, \mathcal{W}_\nu^\vee,\chi)^{\pm,fs}$  as a direct factor of the complex of $\mathrm{R}\Gamma_{dR,c}(\mathcal{S}_{K^pK_p,\Sigma}^{tor}, \mathcal{W}_\nu^\vee)$

For any $w \in \WM$, we can also make sense of $\mathrm{R}\Gamma_{dR, w}(K^p, \mathcal{W}_\nu^\vee,\chi)^{\pm,fs}$ as in section \ref{section-overconvergent-cohomologies}. Namely, one just copies verbatim this section with the automorphic sheaf $\mathcal{V}_\kappa$ replaced by the complex of automorphic sheaves $\mathrm{DR}(\mathcal{W}_\nu^\vee)$. Similarly, we can define $\mathrm{R}\Gamma_{dR,c, w}(K^p, \mathcal{W}_\nu^\vee,\chi)^{\pm,fs}$ by considering the cohomology of the complex $\mathrm{DR}(\mathcal{W}_\nu^\vee(-D))$. 

\begin{rem} It would be interesting to study in depth the cohomologies $$\mathrm{R}\Gamma_{dR, w}(K^p, \mathcal{W}_\nu^\vee,\chi)^{\pm,fs}.$$ For example, are the cohomology groups finite dimensional $F$-vector spaces? 
In the Siegel case,  the cohomologies $\mathrm{R}\Gamma_{dR, w}(K^p, \mathcal{W}_\nu^\vee,\chi)^{\pm,fs}$ and $\mathrm{R}\Gamma_{dR, c, w}(K^p, \mathcal{W}_\nu^\vee,\chi)^{\pm,fs}$ for $w \in \{Id, w_0^M\}$ are rigid cohomologies with support conditions of certain coverings of the ordinary locus  in the Shimura variety.  
Is this a general phenomena (with the ordinary locus replaced by the Igusa variety corresponding to $w$)? 
\end{rem}

The following theorem is the direct analogue of theorem \ref{thm-spectral-sequence}. The proof proceeds in the exact same way. 
 
 \begin{thm}\label{thm-spectral-sequence-rigid} Let $\nu \in X^\star(T^c)^{+}$ be a weight and let $\chi:T(\ZZ_p)\to {\overline{F}}^\times$ be a finite order character.  For a choice of $+$ or $-$, there is a $\mathcal{H}_{p,m,b}^{\pm}$-equivariant spectral sequence  $\mathbf{E}_{dR}^{p,q}(K^p, \mathcal{W}_\nu^\vee,\chi)^\pm$ converging to classical finite slope de Rham cohomology $ \HH^{p+q}_{dR}(K^p, \mathcal{W}_\nu^\vee,\chi)^{\pm,fs}$,
such that $$\mathbf{E}_{dR,1}^{p,q}(K^p, \mathcal{W}_\nu^\vee ,\chi)^{\pm} = \oplus_{w \in \WM, \ell_\pm(w) = p} \HH^{p+q}_{dR,w}( K^p, \mathcal{W}_\nu^\vee,\chi)^{\pm,fs}.$$
There are also spectral sequences $\mathbf{E}_{dR,c}^{p,q}(K^p, \mathcal{W}_\nu^\vee,\chi)^\pm$ converging to $ \HH^{p+q}_{dR,c}(K^p, \mathcal{W}_\nu^\vee,\chi)^{\pm,fs}$ such that $$\mathbf{E}_{dR,c,1}^{p,q}(K^p, \mathcal{W}_\nu^\vee, \chi)^{\pm} = \oplus_{w \in \WM, \ell_\pm(w) = p} \HH^{p+q}_{dR,c,w}( K^p, \mathcal{W}_\nu^\vee,\chi)^{\pm,fs}.$$
\end{thm}

It follows that we have two spectral sequences converging to the classical cohomology $\mathrm{R}\Gamma_{dR}(K^p, \mathcal{W}_\nu^\vee,\chi)^{\pm,fs}$ (and the compactly supported one). The first one  is associated to the stupid filtration on the de Rham complex (and is basically the Hodge-to-de Rham spectral sequence). The other one is the spectral sequence of theorem \ref{thm-spectral-sequence-rigid}, and it is really coming from the Bruhat stratification on the Flag variety. 

By comparing both spectral sequences on the strongly small slope part, we obtain the following decomposition of the de Rham cohomology. 

\begin{thm}\label{thm-hodge-decomposition} For all $\nu \in X^\star(T^c)^{+}$,  we have that:
$$(\HH^{n}_{dR}(K^p, \mathcal{W}_\nu^\vee,\chi)^{+,sss_b(\nu)}) =  \bigoplus_{p+q = n} \oplus_{w, \ell(w) = p} \HH^{q} (K^p, -ww_{0}(\nu+\rho)-\rho, \chi)^{+,sss_{b}(\nu)}$$
and that 
$$(\HH^{n}_{dR}(K^p, \mathcal{W}_\nu^\vee,\chi))^{-,sss_b(\nu)} = \bigoplus_{p+q = n} \oplus_{w, \ell_-(w) = p} \HH^{q} (K^p, -ww_{0}(\nu+\rho)-\rho, \chi)^{-,sss_b(\nu)}$$

We have similarly:
$$(\HH^{n}_{dR,c}(K^p, \mathcal{W}_\nu^\vee,\chi)^{+,sss_b(\nu)}) =  \bigoplus_{p+q = n} \oplus_{w, \ell(w) = p} \HH^{q} (K^p, -ww_{0}(\nu+\rho)-\rho, \chi,cusp)^{+,sss_{b}(\nu)}$$
and that 
$$(\HH^{n}_{dR,c}(K^p, \mathcal{W}_\nu^\vee,\chi))^{-,sss_b(\nu)} = \bigoplus_{p+q = n} \oplus_{w, \ell_-(w) = p} \HH^{q} (K^p, -ww_{0}(\nu+\rho)-\rho, \chi,cusp)^{-,sss_b(\nu)}.$$
\end{thm}

\begin{proof} We only prove the first displayed equation. The idea of the proof is that the two spectral spectral sequences are in a certain sense opposite to each other on the strongly small slope part, and therefore, not only do we have  degeneration at $E_1$, but also the induced filtration on cohomology is split.  

 From theorem \ref{thm-control-thm-Coleman} and corollary \ref{coro-theorem-slopes}, we have that 
\begin{eqnarray*}
\HH^{p+q}_{dR,w}( K^p, \mathcal{W}_\nu^\vee,\chi)^{+,sss_b(\nu)} &=& \HH^{p+q-\ell_-(w)}_{w}( K^p,  -w_{0,M}w(\nu+\rho)-\rho, \chi)^{+,sss_{b}(\nu)} \\
&=&\HH^{p+q-\ell_-(w)}( K^p,  -w_{0,M}w(\nu+\rho)-\rho, \chi)^{+,sss_{b}(\nu)}
\end{eqnarray*}

For the spectral sequence of theorem \ref{thm-spectral-sequence-rigid}, we deduce that 
\begin{eqnarray*} \mathbf{E}_{dR,1}^{p,q}(K^p, \mathcal{W}_\nu^\vee ,\chi)^{+, sss_b(\nu)} & = &\oplus_{w \in \WM, \ell(w) = p} \HH^{p+q}_{dR,w}( K^p, \mathcal{W}_\nu^\vee,\chi)^{\pm,fs} \\
&=& \oplus_{w \in \WM, \ell(w) = p} \HH^{p+q-\ell_-(w)}( K^p,  -w_{0,M}w(\nu+\rho)-\rho, \chi)^{+,sss_{b}(\nu)} 
\end{eqnarray*}

For the spectral sequence of theorem  \ref{thm-H-dR}, we have 
\begin{eqnarray*}
(E_1^{p,q})^{+,sss_b(\nu)} & = &\oplus_{w \in \WM, \ell(w)=p} \HH^q(K^p, -ww_0( \nu + \rho)-\rho,\chi)^{+,sss_b(\nu)} \\
&=& \mathbf{E}_{dR,1}^{p',q'}(K^p, \mathcal{W}_\nu^\vee ,\chi)^{+, sss_b(\nu)},~\textrm{for}~p'=d-p,~p'+q'=p+q.
\end{eqnarray*}
The spectral sequence of theorem  \ref{thm-H-dR} degenerates at $E_1$, and we deduce that we get two opposite filtrations on the cohomology. This gives the splitting. 
\end{proof}

\begin{rem} This theorem is reminiscent of complex Hodge theory, where one obtains a splitting of the Hodge filtration given by harmonic $\mathcal{C}^\infty$-differential forms.  In the $p$-adic setting, the de Rham cohomology comes with additional structure (notably a Frobenius on the  Hyodo-Kato cohomology \cite{MR1293974}). Is the above splitting induced by these additional structure?
\end{rem}

We conclude this section by   constructing an interior Cousin bi-complex analogue to the interior Cousin complex of section \ref{section-interior-coho}.

We define $\overline{\mathcal{C}ous}(K^p, \mathcal{W}_\nu^\vee, \chi)^{\pm}$ as the bi-complex concentrated in degrees in $[0,d]\times [0,d]$, where for all $(i,j) \in \ZZ \times \ZZ$, we have: 
$$(\overline{\mathcal{C}ous}(K^p, \mathcal{W}_\nu^\vee, \chi)^{\pm})^{(i,j)} = \oplus_{w \in \WM, \ell_{\pm}(w)=i} \oplus_{w' \in \WM, \ell(w') = j} \overline{\HH}^{i}_w(K^p,-w'w_0(\nu + \rho)-\rho, \chi)^{\pm,fs}$$

The horizontal complexes $(\overline{\mathcal{C}ous}(K^p, \mathcal{W}_\nu^\vee, \chi)^{\pm})^{(\bullet,j)}$ are $$\oplus_{w' \in \WM, \ell(w') = j}\overline{\mathcal{C}ous}(K^p, -w'w_0(\nu + \rho)-\rho, \chi)^{\pm}$$ and the vertical differentials are those given by the maps in the BGG complexes.

Let us define the interior de Rham cohomology by  $$\overline{\HH}^i_{dR}(K^p, \mathcal{W}_\nu^\vee, \chi)^{\pm} = \mathrm{Im}( {\HH}^i_{dR,c}(K^p, \mathcal{W}_\nu^\vee, \chi)^{\pm}  \rightarrow {\HH}^i_{dR,c}(K^p, \mathcal{W}_\nu^\vee, \chi)^{\pm}).$$

If the Shimura variety is compact we simply write ${\mathcal{C}ous}(K^p, \mathcal{W}_\nu^\vee, \chi)^{\pm}$ instead of $\overline{\mathcal{C}ous}(K^p, \mathcal{W}_\nu^\vee, \chi)^{\pm}$.

We find:

\begin{prop}\label{interior-deRham-cousin} If the Shimura variety is compact, the complex $ \mathrm{Tot}({\mathcal{C}ous}(K^p, \mathcal{W}_\nu^\vee, \chi)^{\pm})$ is quasi-isomorphic to $\mathrm{R}\Gamma_{dR}(K^p, \mathcal{W}_\nu^\vee, \chi)^{\pm}$. 
For a general Shimura variety, the interior cohomology  $\overline{\HH}^i_{dR}(K^p, \mathcal{W}_\nu^\vee, \chi)^{\pm}$ is a subquotient of   $$\HH^i ( \mathrm{Tot}(\overline{\mathcal{C}ous}(K^p, \mathcal{W}_\nu^\vee, \chi)^{\pm})).$$

\end{prop}
\begin{proof} This follows from corollary \ref{coro-concentration-interior}. 
\end{proof}

\subsection{Explicit formulas in the symplectic case}
Let $V$ be a $\qq$-vector space of dimension $2g$. Let $\Psi$ be the symplectic form on $V$ given in the canonical basis $e_1, \cdots, e_{2g}$   by $\Psi (e_i, e_j) = 1$ if $i \leq g$ and $j = 2g-i+1$ and $\Psi(e_i, e_j) = 0$ if $i \leq g$ and $j \neq 2g-i+1$. 
Let $G = \mathrm{GSp}_{2g}$, be the subgroup of automorphisms of $V$ respecting $\Psi$ up to a similitude factor $\nu$.   We pick the maximal diagonal torus $T$  in $G$. A typical element $t \in T$   is labelled $(t_1, \cdots, t_g; c) = \mathrm{diag}(t_1c,\cdots, t_g c, t_g^{-1}c, \cdots, t_1^{-1} c)$. The character group $X^\star(T)$ identifies with  $\{(k_1, \cdots, k_g; k) \in \ZZ^{g+1},~\sum k_i = k ~\mathrm{mod}~2\}$. The action is  given by $(k_1, \cdots, k_g; k) (t_1, \cdots, t_g; c) = \prod t_i^{k_i} c^k$. 
We let $P_\mu^{std}$ be the stabilizer of the Lagrangian plan $\langle e_1, \cdots, e_g \rangle$. We therefore let $P_\mu$ be the stabilizer of the Lagrangian plan $\langle e_{g+1}, \cdots, e_{2g} \rangle$. 
We have  $M_\mu \simeq \mathrm{GL}_g \times \mathbb{G}_m$. We choose the upper triangular Borel in $M_\mu$, which fixes the positive compact roots. Recall that we choose the positive non-compact roots to be in $\mathfrak{g}/\mathfrak{p}_\mu^{std}$. It follows that $X^\star(T)^{M_\mu,+} = \{ (k_1, \cdots, k_g; k), k_1 \geq k_2 \geq \cdots \geq k_g\}$ and $X^\star(T)^{+} = \{ (k_1, \cdots, k_g; k), 0 \geq k_1 \geq k_2 \geq \cdots \geq k_g\}$. 
We have $\rho = (-1,-2,-3,\cdots, g ; 0)$, $2\rho_{nc} = (-g-1, \cdots, -g-1;0)$. The usual way to normalize the central character in the theory of Siegel modular forms is to consider weights of the form $(k_1, \cdots, k_g ; -\sum k_i)$, with $k_1 \geq \cdots \geq k_g$.
For example (consult \cite{F-Pilloni}, example 5.2 for the details), when $g=1$, $\mathcal{V}_{(k;-k)}  = \omega_E^{\otimes k}$ where $E \rightarrow S_K^{\star}$ is the universal semi-abelian scheme and $\omega_E$ is its conormal sheaf. When $g=2$, $\mathcal{V}_{(k_1,k_2;-k_1-k_2)}  = \mathrm{Sym}^{k_1-k_2}\omega_A \otimes \det^{k_2} \omega_A$ where $A \rightarrow S_{K,\Sigma}^{tor}$ is the universal semi-abelian scheme and $\omega_A$ is its conormal sheaf. 

A standard basis of   Hecke operators in $\mathcal{H}^+_{p,1,0}$ is given by  the classes: 
\begin{enumerate}
\item $ U_g =  [K_{p,1,0} (- 1/2, \cdots, -1/2; -1/2)(p) K_{p,1,0}] $, where $(- 1/2, \cdots, -1/2; -1/2)(p) = \mathrm{diag}(p^{-1}, \cdots, p^{-1}, 1, \cdots, 1)$. 
\item $U_i = [K_{p,1,0} (0, \cdots,0, -1, \cdots, -1; -1)(p) K_{p,1,0}]$ for $1 \leq i \leq g-1$ with $i$ many $-1$ before the $;$. We have that  $(0, \cdots,0, -1, \cdots, -1; -1)(p)  = \mathrm{diag}(p^{-1}\mathrm{Id}_{g-i},p^{-2} \mathrm{Id}_i,  \mathrm{Id}_i, p\mathrm{Id}_{g-i})$
\item $S=  [pK_{p,1,0}]$, $S^{-1}$. 
\end{enumerate}
\begin{rem} Following \cite{F-Pilloni},  remark 5.6, we justify that  for  $g=1$,  the double class  $[K_{p,1,0} \mathrm{diag}(p^{-1}, 1) K_{p,1,0}]$ indeed corresponds to the  standard $U_p$-operator!  For simplicity let us assume that $K^p \subseteq \mathrm{GL}_2( \prod_{\ell \neq p} \ZZ_\ell)$.  The corresponding moduli space (ignoring cusps) parametrizes two  elliptic curves $E_1, E_2$ up to isomorphisms, with $K$ level structures, together with a quasi isogeny $E_2 \rightarrow E_1$ giving a map  $ V(E_2) \rightarrow V(E_1)$  where $V(E_ i) = \lim E_i[N] \otimes \qq$ is the adelic Tate module, and we ask that this map is  represented by $K \mathrm{diag}(p^{-1},1) K$.    Concretely this means that the quasi-isogeny  $E_2 \rightarrow E_1$  comes from a degree $p$ isogeny $E_1 \rightarrow E_2$ and that this isogeny matches $K$-level structures on the Tate modules.  The $K_{p,1,0}$ level structure is the data of a rank $p$-subgroup $H_i \subseteq E_i[p]$.  Because we choose the lower triangular Borel, we find that the isogeny $E_1 \rightarrow E_2$ induces and isomorphism between $H_1$ and $H_2$. 
\end{rem} 
\subsubsection{$GL_2/\qq$} In this case, everything is already in \cite{Boxer-Pilloni}.  Let $\kappa = (k;-k)$. The Cousin complex is 
$\mathcal{C}ous(K^p, \kappa, \chi) : \HH^0_{1} ( K^p, \kappa, \chi)^{+,fs} \rightarrow \HH^1_{w} ( K^p, \kappa, \chi)^{+,fs}$
where $\HH_{1} ( K^p, \kappa,\chi)^{+,fs} $ is the space   of finite slope overconvergent modular forms of weight $k$, nebentypus $\chi$ and 
$\HH^1_{w} ( K^p, \kappa, \chi)^{+,fs}$ 
 is the finite slope part of the cohomology with compact support of the dagger space ``ordinary locus'' in weight $k$ and nebentypus $\chi$. 
 We have $w_{0,M} = 1$ for $\mathrm{GL}_2$.  
For $w = 1$, we find that   $w^{-1} w_{0,M}(\kappa +  \rho ) + \rho = (k -2; -k)$ and that $ \langle (-1/2; -1/2) , (k-2;-k) \rangle = 1$.   The un-normalized $U_p$-operator acts on $q$-expansion by $\sum a_n q^n = p \sum a_{np} q^n$ and has indeed slope greater or equal to $1$ on  $\HH^0_{1} ( K^p, \kappa, \chi)^{+,fs} $. 
For $w \neq 1$, we   find that $w^{-1} w_{0,M}(\kappa +  \rho ) + \rho = (-k; -k)$ and we get that $-\langle (-1/2; -1/2) , (-k;-k) \rangle  = k $.  On $\HH^1_{w} ( K^p, \kappa, \chi)^{+,fs}$, $U_p$ acts like Frobenius, and this explains why it is of slope greater or equal than $k$.  See lemma 5.3 in \cite{Boxer-Pilloni}.  We deduce from this lemma the classicality theorem. 

\subsubsection{$GSp_{4}/\qq$} The Weyl group is generated by the following transposition: $ s_0(k_1,k_2; k) = (k_2, k_1; k)$ and $s_1 (k_1,k_2;k) = (-k_1, k_2; k)$.  
The elements of $\WM$ are $ Id, s_1 ,  s_1 s_0, s_1s_0 s_1$.  We consider the weight $\kappa=(k_1,k_2;-k_1-k_2)$ so that $-w_{0,M}\kappa=(-k_2,-k_1; k_1 +k_2)$. The following table indicates the value of the pairing 
$$ \langle t, w^{-1} w_{0,M}(\kappa + \rho) + \rho \rangle = \langle t, w^{-1} w_{0,M}(\kappa)  \rangle + { \langle t, w^{-1} w_{0,M}( \rho) + \rho \rangle}$$  where $t = (-1/2,-1/2;-1/2)(p)$ or $(0,-1;-1)(p)$, and $w \in \WM$. 

\medskip

\begin{tabular}{|c|c|c|c|c|}
\hline
&Id&$s_1$&$s_1s_0$&$s_1s_0s_1$\\
\hline
(-1/2,-1/2;-1/2)&3&$k_2 +1$&$k_2 + 1$&$k_1 + k_2$ \\
\hline
(0,-1;-1)&$k_2+3$&$k_2 +3$& $ 2k_2 +k_1$ &$2k_2+ k_1$\\
\hline
\end{tabular}

\medskip

The table below gives the value $$ \max\{ \langle t, w^{-1} w_{0,M}(\kappa)  \rangle,  \langle t, w^{-1} w_{0,M}(\kappa)  + 2w^{-1} \rho_{nc} \rangle  \}$$ which appears in the strictly small slope condition. 
\medskip

\begin{tabular}{|c|c|c|c|c|}
\hline
&Id&$s_1$&$s_1s_0$&$s_1s_0s_1$\\
\hline
(-1/2,-1/2;-1/2)& 3 &$k_2 $&$k_2 $&$k_1 + k_2$ \\
\hline
(0,-1;-1)&$k_2+3 $&$k_2+3$& $ 2k_2 +k_1$ &$2k_2+ k_1$\\
\hline
\end{tabular}

\medskip

 The difference between these two tables illustrates the difference between conjecture \ref{conj-strongslopes} (first table) and theorem \ref{thm-slopes} (second table). Let us explain how to interpret this information.

Let us take a weight in the interior of the holomorphic chamber: $k_2 > 2$. We assume that  $U_2$ has slopes $< k_2 +1$. This is the ``small slope'' condition because there are conjecturally  no overconvergent cohomology classes for $w \neq 1$ which satisfy this slope condition. Therefore the $U_2$-slope $< k_2+1$ part of classical cohomology identifies conjecturally with the overconvergent cohomology for $w= 1$.  Under the bounds of theorem \ref{thm-slopes}, we have to  use  the strongly small slope condition $< k_2 $ instead. 

Let us take a weight in the interior of the  ``$\mathrm{H}^1$ chamber'': $k_2 < 2$ and $k_1 + k_2 > 3$. The small slope condition  is now $U_2$-slope $<3$ and $U_1$-slope $< k_1+2k_2$. The first condition kills  the overconvergent cohomology for $1$. The second condition kills  the overconvergent cohomology for $s_1s_0$ and $s_1s_0s_1$. 
Therefore, the small slope classical cohomology identifies  with the small slope overconvergent cohomology for $w=s_1$.  

\section{$p$-adic families of overconvergent cohomology}

\subsection{Relative spectral theory and slope decompositions}\label{section-construction-eigenvariety}

\subsubsection{Slope decomposition over Tate algebras} Let $(A,A^+)$ be  a finite type Tate $(F,\ocal_F)$-algebra.   For any point $x \in \Spa (A, A^+)$, we let $k(x)$ be the complete residue field. It possesses a  rank one  valuation $v$  which we normalize by  $v(p)=1$ (this valuation corresponds  to the maximal generalization $\tilde{x}$ of $x$).  

\begin{defi}
A polynomial $Q \in A[X]$ whose leading coefficient is a unit is said to be of slope $\leq h$ for $h\in\qq$ if for all $x \in \Spa(A,A^+)$, the image of $Q$ in $k(x)[X]$ has slope less or equal than $h$ in the sense of section \ref{sect-slope-decompo}.
\end{defi}

\begin{rem}This definition  of being of slope $\leq h$  agrees with \cite{ashstevensslopes}, def. 4.3.2, in the case that $A$ is reduced, we take the supremum norm as a norm on $A$, and the leading coefficient of $Q$ is a multiplicative unit.
\end{rem}

\begin{defi}\label{defi-slope-module}
Let $M$ be an $A$-module and let $T$ be an $A$-linear endomorphism of $M$.  Let $h \in \qq$.  An $h$-slope decomposition of $M$ with respect to $T$ is a direct sum decomposition  of $A$-modules $M = M^{\leq h} \oplus M^{>h}$ such that:
\begin{enumerate}
\item  $M^{\leq h}$  and $M^{>h}$ are stable under the action of $T$.
\item  $M^{\leq h} $  is a finite $A$-module.
\item  There is a unitary polynomial $Q \in A[X]$ with slope $\leq h$ such that $Q^\star(T)$ is zero on  $M^{\leq h}$. 
\item  For any unitary polynomial $Q \in A[X]$ with slope $\leq h$, the restriction of $Q^\ast(T)$  to $M^{>h}$
is an invertible endomorphism.
\end{enumerate} 
\end{defi}

 This definition fits in the framework of \cite{ashstevensslopes}, def. 4.1.1. It is easy to see that if  such a slope decomposition exists, it is unique.   

Now let $S = \Spa(A,A^+)$ and let $\mathcal{M}$ be a sheaf of $\oscr_S$-modules. We let $T \in \mathrm{End}_{\oscr_S}(\mathcal{M})$ be an endomorphism.

\begin{defi}
We say that $\mathcal{M}$ has slope decomposition with respect to $T$ if for any $x \in S$ and $h\in\qq$, there exists a affinoid neighborhood $U$ of $x$ in $S$, $h'\geq h$, and a $T$-stable direct sum decomposition of sheaves of $\oscr_U$-modules $$\mathcal{M}\vert_U=(\mathcal{M}\vert_{U})^{\leq h'} \oplus (\mathcal{M}\vert_{U})^{> h'}$$ such that:
\begin{enumerate}
\item $(\mathcal{M}\vert_{U})^{\leq h'}$ is a coherent sheaf of $\oscr_U$-modules.
\item For any affinoid open $V = \Spa (B,B^+) \subseteq U$, $$\mathcal{M}(V)=(\mathcal{M}\vert_{U})^{\leq h'}(V) \oplus (\mathcal{M}\vert_{U})^{> h'}(V)$$ is a $h'$-slope decomposition of the $B$-module $\mathcal{M}(V)$ in the sense of definition \ref{defi-slope-module}.
\end{enumerate}
\end{defi}

\begin{rem}  It is a subtle but important point  in the definition that we only ask for a slope $h'$ decomposition for  some $h' \geq h$ and not for an $h$-slope decomposition for any $h$. 
\end{rem}

We say that a section $s \in \mathcal{M}(U)$ for an open $U \subseteq S$ has infinite slope if for any affinoid open $V\subseteq U$ and $h\in\qq$ such that $\mathcal{M}(V)$ admits an $h$-slope decomposition, $s\vert_V\in\mathcal{M}(V)^{>h}$.  The infinite slope sections form a subsheaf $\mathcal{M}^{\infty s}\subseteq\mathcal{M}$.  We let $\mathcal{M}^{fs}=\mathcal{M}/\mathcal{M}^{\infty s}$ and call it the finite slope part of $\mathcal{M}$.

\begin{lem}\label{lem-functorialityoffs}  

\begin{enumerate} 
\item Let $ f : \mathcal{M} \rightarrow \mathcal{N}$ be a map of sheaves of $\oscr_S$-modules. Let $T_{\mathcal{M}}$ and $T_{\mathcal{N}}$ be endomorphisms of $\mathcal{M}$ and $\mathcal{N}$ respectively, commuting with $f$. Assume that $\mathcal{M}$ and $\mathcal{N}$ have slope decompositions. Then $\mathrm{ker}(f)$, $\mathrm{im}(f)$, and $\mathrm{coker}(f)$ have slope decompositions. 

\item Let $0 \rightarrow \mathcal{L} \rightarrow \mathcal{M} \rightarrow  \mathcal{N} \rightarrow 0$ be an exact sequence of sheaves of $\oscr_S$-modules. Let $T \in \mathrm{End}_{\oscr_S} (\mathcal{M})$ be such that $T(\mathcal{L}) \subseteq \mathcal{L}$. Then if two out of $\mathcal{L}, \mathcal{M}, \mathcal{N}$ have slope decomposition with respect to $T$, then so does the third.
Moreover, taking the infinite slope part yields an exact sequence $0 \rightarrow \mathcal{L}^{\infty s}  \rightarrow \mathcal{M}^{\infty s}  \rightarrow  \mathcal{N}^{\infty s} \rightarrow 0$  
 and taking the finite slope part gives an exact sequence $0 \rightarrow \mathcal{L}^{f s}  \rightarrow \mathcal{M}^{fs}  \rightarrow  \mathcal{N}^{fs} \rightarrow 0$. 
 \end{enumerate}
 \end{lem}
 \begin{proof} We reduce easily to the case of modules where this is  \cite{ashstevensslopes}, prop. 4.1.2.
 \end{proof}

We now explain the construction of  the spectral variety $\mathcal{Z} \hookrightarrow \mathbb{A}^{1,an}_S$.  For any affinoid open $U = \Spa (B, B^+)$ for which we have a slope decomposition $ \mathcal{M}(U) = \mathcal{M}(U)^{\leq h} \oplus \mathcal{M}(U)^{> h}$,  we  define a map $B[X] \rightarrow \mathrm{End}_{B} ( \mathcal{M}(U)^{\leq h})$ by sending $X$ to $T^{-1}$ (note that $T$ is invertible on $ \mathcal{M}(U)^{\leq h}$). We let $I$ be the kernel of this map. We let $\mathcal{Z}_{U,h} \hookrightarrow \mathbb{A}^{1,an}_S$ be $\Spa ( B[X]/I, (B[X]/I)^+)$ where $(B[X]/I)^+$ is the integral closure of $B^+$ in $B[X]/I$. If $U'   \subseteq U$, is an open affinoid,  we see (using the flatness of $\oscr_S(U) \rightarrow \oscr_S(U')$) that $\mathcal{Z}_{U,h} \times_{U} U' = \mathcal{Z}_{U',h}$. If $h' \geq h$ and  $\mathcal{M}(U)$ also has an $h'$-slope decomposition, we see that $\mathcal{Z}_{U,h} \hookrightarrow \mathcal{Z}_{U,h'}$ is a union of connected components since $\mathcal{M}(U) = \mathcal{M}(U)^{\leq h} \oplus (\mathcal{M}(U)^{\leq h'}  \cap  \mathcal{M}(U)^{> h}) \oplus  \mathcal{M}(U)^{>h'}$. 

We let $\mathcal{Z} = \coprod_{U,h} \mathcal{Z}_{U,h} / \sim$ where $\sim$ is the equivalence relation defined by the inclusions  $\mathcal{Z}_{U,h} \hookrightarrow \mathcal{Z}_{U,h'}$ for $h \geq h'$ and  $\mathcal{Z}_{U',h} \hookrightarrow \mathcal{Z}_{U,h}$ for $U' \subseteq U$. Let $\pi : \mathcal{Z} \rightarrow S$ be the projection which is locally quasi-finite, and partially proper. 
There is a coherent sheaf $\mathcal{M}_{\mathcal{Z}}^{fs} \hookrightarrow \mathcal{M}_{\mathcal{Z}} = \pi^\star \mathcal{M}$ defined by $\mathcal{M}_{\mathcal{Z}}^{fs}(\mathcal{Z}_{U,h}) = \mathcal{M}(U)^{\leq h}$.  We see that $\pi_\star \mathcal{M}_{\mathcal{Z}}^{fs} = \mathcal{M}^{fs}$ is the finite slope part of $\mathcal{M}$. 

In many  applications, there is also a commutative algebra $H$ with $T \in H$,  which acts on $\mathcal{M}$.  It is clear that $H$ acts also on $\mathcal{M}^{fs}$. 
We can consider $\oscr_{\mathcal{E}}$ which is the coherent $\oscr_{\mathcal{Z}}$-subalgebra of $\mathrm{End}_{\oscr_{\mathcal{Z}}}( \mathcal{M}_{\mathcal{Z}}^{fs})$ generated by $H$. We let $\mathcal{E}$ be its relative adic spectrum over $\mathcal{Z}$. The morphism $\mathcal{E} \rightarrow \mathcal{Z}$ is finite surjective. This is the eigenvariety attached to $H$ acting on $\mathcal{M}$. 

 \subsubsection{Finite slope decomposition  for an algebra of operators} It is often the case that we want to consider the finite slope decomposition of a module or a sheaf when we have an algebra acting. 
We now let $\mathbb{Z}_{\geq 0}^r$ be the free monoid on $r$ generators, generated by element $T_1, \cdots, T_r$. We assume that we have an action of $\mathbb{Z}_{\geq 0}^r$ on a sheaf $\mathcal{M}$. We also assume that for all the operators $T \in \mathbb{Z}_{\geq 1}^r$, the sheaf $\mathcal{M}$ has slope decomposition.

\begin{prop}\label{prop-independence-choiceT} Let $T,T' \in  \mathbb{Z}_{\geq 1}^r$, and consider the finite slope projections $p_T : \mathcal{M} \rightarrow \mathcal{M}^{T, fs}$ and $p_{T'} : \mathcal{M} \rightarrow \mathcal{M}^{T',fs}$. Then there is an isomorphism $\mathcal{M}^{T,fs} \rightarrow \mathcal{M}^{T',fs}$ compatible with the projections $p_T$ and $p_{T'}$. 
\end{prop}

\begin{proof} Let $x \in S$, let $h \in \qq$, and let $V$ be a neighborhood of $x$ such that we have  $h$-slope decomposition with respect to $T$:
$$ \mathcal{M}_{\vert{V}}= \mathcal{M}_{\vert V}^{\leq_T h} \oplus \mathcal{M}_{\vert V}^{>_Th}.$$
We claim that  $\mathcal{M}_{\vert V}^{\leq_T h}$ has slope decomposition with respect to $T'$ and that the projection $\mathcal{M}_{\vert V} \rightarrow \mathcal{M}_{\vert V}^{\leq_T h}$ orthogonal to $ \mathcal{M}_{\vert V}^{>_Th}$ factors through $\mathcal{M}^{T',fs}_{\vert V}$.  This will show that the projection $p_T : \mathcal{M} \rightarrow \mathcal{M}^{T, fs}$ factors $\mathcal{M} \rightarrow \mathcal{M}^{T', fs} \rightarrow \mathcal{M}^{T,fs}$. Reversing the roles of $T$ and $T'$ we conclude the proof of the proposition. 

We observe that $T$ is invertible on $\mathcal{M}_{\vert V}^{\leq_T h}$ and therefore $T'$ is also invertible as there is $n \in \ZZ_{\geq 0}$ and $u \in \mathbb{Z}_{\geq 0}^r$ such that $T^n = T' u $. 
Note that  $\mathcal{M}_{\vert V}^{\leq_T h}$ is a coherent sheaf.  For any $y \in V$, we claim that we can find a neighborhood $W$ of $y$ and  a unitary polynomial $Q$ such that $Q(0)$ is a unit and  such that $Q^\star(T') = 0$ on $(\mathcal{M}_{\vert V}^{\leq_T h})_{\vert W}$. Indeed, consider the stalk $(\mathcal{M}_{\vert V}^{\leq_T h})_y$. This is a finite $\oscr_{S,y}$-module. By Nakayama's lemma, we can pick a  surjection 
$\oscr_{S,y}^m \rightarrow (\mathcal{M}_{\vert V}^{\leq_T h})_y$ and lift $T'$ to an invertible endomorphism $\tilde{T}'$ of $\oscr_{S,y}^m$. We let $Q = \det(\tilde{T'}^{-1}-XId)$. This is a unitary polynomial whose constant coefficient is a unit and $Q^\star(T') = 0$ by Cayley-Hamilton. 
It follows that  $\mathcal{M}_{\vert V}^{\leq_T h}$ has a slope decomposition with respect to $T'$ and that it is equal to its finite slope part for $T'$.

\end{proof} 

\subsubsection{Relative spectral theory}\label{relative-spectral-1} We recall briefly the relative spectral theory for compact operators.  The original reference is \cite{1997InMat.127..417C}. 
Let $(A,A^+)$ be  a finite type complete Tate $(F,\ocal_F)$-algebra, and let $S = \Spa(A,A^+)$.  Let  $M^\bullet \in Ob(\mathcal{K}^{proj}(\mathbf{Ban}(A)))$, and  let $\mathcal{M}^{\bullet}$ be the associated complex of  Banach sheaves on $S$. Let $T \in  \mathrm{End}_{\mathcal{D}(\mathbf{Ban}(A))}(M^\bullet)$  be a compact operator.   Let $\HH^i(\mathcal{M}^\bullet)$ be the $i$th cohomology sheaf.

\begin{prop}\label{prop-existence-decompositionT} The sheaves  $\HH^{i}(\mathcal{M}^\bullet)$ have  slope decomposition for $T$.  Moreover, there is an object   $\mathcal{M}^{\bullet,fs} \in \mathcal{D}(\mathrm{Mod}(\oscr_S))$ and a morphism $\mathcal{M}^{\bullet} \rightarrow \mathcal{M}^{\bullet,fs}$ (unique up to a non unique quasi-isomorphism)  with the property that $\HH^i(\mathcal{M}^{\bullet,fs}) = \HH^{i}(\mathcal{M}^\bullet)^{fs}$.
\end{prop}

\begin{proof} We begin by describing a  construction which depends on  a lift $\tilde{T} \in \mathrm{End}_{A}(M^\bullet)$ which is compact in all degrees. We let $$\tilde{{P}}(X) = \prod_k \mathrm{det}(1-X\tilde{T}\vert M^k) \in A[[X]]$$ be the (total) Fredholm determinant of $\tilde{T}$ (see \cite{1997InMat.127..417C}, section A 2). The series $\tilde{{P}}(X)$ is an entire series, and we let $\tilde{\mathcal{Z}} \hookrightarrow \mathbb{G}^{an}_m \times \Spa (A, A^+)$ be the vanishing locus of $\tilde{{P}}$. This is the spectral variety associated to $\tilde{T}$.  The morphism $\pi : \tilde{\mathcal{Z}} \rightarrow \Spa (A, A^+)$ is locally quasi-finite, flat, and partially proper (see \cite{MR3831033}, thm. B1 for a  short proof in the language of adic spaces). Any point $z \in  \tilde{\mathcal{Z}} $ has a neighborhood $\mathcal{U}_z$ such that $\overline{\{z\}} \subseteq \mathcal{U}_z$ and $\mathcal{U}_z \rightarrow \pi(\mathcal{U}_z)$ is finite flat. 
We can describe such a neighborhood more precisely. Let  $x \in \Spa(A,A^+)$ be the image of $z$. Then there is a neighborhood $\Spa(B,B^+)$ of  $x$ in $\Spa(A, A^+)$ and a factorization in $B\{ \{ X\}\}$ (the ring of global functions on $\mathbb{A}^{1,an}_{\Spa(B,B^+)}$) $\tilde{P}(X) = \tilde{R}(X) \tilde{Q}(X)$ where $\tilde{R}(X) $ is a Fredholm series (that is $R(0)=1$), $\tilde{Q}(X) = 1 + a_1 X + \cdots + a_d X^d$ is a polynomial with $a_d \in B^\times$,  $ \tilde{R}(X)$ and  $\tilde{Q}(X)$ are prime to each other, and $\tilde{Q}(z)=0$. Then $\mathcal{U}_z = V(\tilde{Q}(X)) \subseteq \mathbb{A}^1 \times \Spa (B, B^+)$ is a neighborhood of $z$ in $\tilde{\mathcal{Z}}$.

Over $\tilde{\mathcal{Z}}$ we have a complex of coherent sheaves $\tilde{\mathcal{M}}_{\tilde{\mathcal{Z}}}^{\bullet,fs}$ whose definition we briefly recall.  Let $\tilde{Q}^{\ast}(X) = X^d \tilde{Q}(X^{-1})$.  Over $\Spa(B,B^+)$ we have a unique decomposition $M^\bullet \hat{\otimes}_A B = M^\bullet(\tilde{Q}) \oplus N^\bullet(\tilde{Q})$ where $\tilde{Q}^\ast(\tilde{T})$ is zero on $M^\bullet(\tilde{Q})$ and invertible on $N^\bullet(\tilde{Q})$. Moreover, the projection on each factor is given by an entire series in $\tilde{T}$ (see \cite{MR144186}, proposition 12).  

It follows that $M^\bullet(\tilde{Q})$ has a natural structure of a complex of $B[X]/\tilde{Q}(X)$ modules of finite type (with $X^{-1}$ acting as $\tilde{T}$), and we let ${\mathcal{M}}_{\tilde{\mathcal{Z}}}^{\bullet,fs}\vert_{V(\tilde{Q}(X))} = M^\bullet(\tilde{Q})$. These glue to give the complex ${\mathcal{M}}_{\tilde{\mathcal{Z}}}^{\bullet,fs}$ over $\tilde{\mathcal{Z}}$.  We observe that by construction ${\mathcal{M}}_{\tilde{\mathcal{Z}}}^{\bullet,fs}$ is a perfect complex of $\pi^{-1} \oscr_{\Spa(A,A^+)}$-modules. Moreover, if $M^\bullet$ is concentrated in the range $[a,b]$, so is ${\mathcal{M}}_{\tilde{\mathcal{Z}}}^{\bullet,fs}$.

We can actually assume (after possibly changing the neighborhood $\mathcal{U}_z$) that in the factorization  $\tilde{P}(X) = \tilde{R}(X) \tilde{Q}(X)$, there exists $h \in \qq$ such that $\tilde{Q}(X)$ has slope $\leq h$ and $\tilde{R}(X)$ has slope $>h$. We deduce that 
$$M^\bullet \hat{\otimes}_A B = M^\bullet(\tilde{Q}) \oplus N^\bullet(\tilde{Q}) = (M^\bullet \hat{\otimes}_A B)^{\leq h} \oplus (M^\bullet \hat{\otimes}_A B)^{>h}$$ is an $h$-slope decomposition of the complex. Indeed, for any unitary  polynomial $S(X)$ of slope $\leq h$, we deduce that $S(X)$ and $\tilde{R}(X)$ are coprime to each other in $B\{ \{ X\}\}$ since the resultant $\mathrm{Res}(S,\tilde{R}) \in B^\times$. It follows from \cite{1997InMat.127..417C}, lem A4.1 that $S^\star(\tilde{T})$ acts invertibly on $(M^\bullet \hat{\otimes}_A B)^{>h}$. 

By varying $z$ in the fiber of $x$, we deduce that $\mathcal{M}^\bullet$ has a slope decomposition with respect to $\tilde{T}$. We let $\mathcal{M}^{\bullet,fs}$ be the finite slope quotient.  Observe that  $\mathcal{M}^{\bullet,fs} = \pi_\star {\mathcal{M}}_{\tilde{\mathcal{Z}}}^{\bullet,fs}$. The map $\mathcal{M}^\bullet \rightarrow \mathcal{M}^{\bullet, fs}$  is given by adjunction   by the map  $\pi^\star \mathcal{M}^\bullet \rightarrow {\mathcal{M}}_{\tilde{\mathcal{Z}}}^{\bullet, fs}$  (which in the above notations, is locally given by the projection $M^\bullet \hat{\otimes} B \rightarrow  M^\bullet(\tilde{Q})$, orthogonal to  $N^\bullet(\tilde{Q})$).

It follows from  lemma \ref{lem-functorialityoffs} that $\HH^i(\mathcal{M}^\bullet)$ has a slope decomposition and that $\HH^i(\mathcal{M}^\bullet)^{fs} = \HH^i(\mathcal{M})^{fs}$

We now discuss unicity. We have an exact triangle $\mathcal{M}^{\bullet, \infty s} \rightarrow \mathcal{M}^{\bullet} \rightarrow \mathcal{M}^{\bullet, fs} \stackrel{+1}\rightarrow $ in $\mathcal{D}(\mathrm{Mod}(S))$. Another choice of compact lift $\tilde{T}'$  of $T$ gives another exact triangle 
$\mathcal{M}^{',\bullet, \infty s} \rightarrow \mathcal{M}^{\bullet} \rightarrow \mathcal{M}^{',\bullet, fs} \stackrel{+1}\rightarrow $.

 Let $C^\bullet$ be the cone of the morphism of complexes $ \mathcal{M}^{\bullet, \infty s} \rightarrow \mathcal{M}^\bullet \rightarrow  \mathcal{M}^{',\bullet,fs}$ obtained by taking inclusion in $\mathcal{M}^\bullet$ followed by projection. 
  The diagram:
 \begin{eqnarray*}
 \xymatrix{  \mathcal{M}^{\bullet, \infty s} \ar[r] \ar[d]^{\tilde{T}} & \mathcal{M}^{',\bullet,fs} \ar[d]^{\tilde{T}'} \\
  \mathcal{M}^{\bullet, \infty s}  \ar[r] &  \mathcal{M}^{',\bullet,fs}}
  \end{eqnarray*}
   commutes up to homotopy. We deduce that there is a morphism $\tilde{T}'' : C^\bullet \rightarrow C^\bullet$  which on $C^n = \mathcal{M}^{',n,fs}\oplus  \mathcal{M}^{n+1, \infty s} $ is given by    $\begin{pmatrix} 
        \tilde{T}' & b_n \\
        0& \tilde{T} \\
     \end{pmatrix}$ where $b_n$ depends on the choice of a homotopy between $\tilde{T}$ and $\tilde{T}'$ (see \cite{stacks-project}, TAG 014F). 
  We deduce that the complex $C^\bullet$ has slope decomposition with respect to $\tilde{T}''$ and we let $C^{\bullet, fs}$ be its finite slope quotient. 
  The map $ \mathcal{M}^{',\bullet,fs} \rightarrow C^{\bullet, fs}$ is a quasi-isomorphism. 
 By the TR4 axiom on triangulated categories, we find that there is a commutative diagram 
 \begin{eqnarray*}  
 \xymatrix{ \mathcal{M}^\bullet \ar[r] \ar[d] &  \mathcal{M}^{\bullet, fs}  \ar[d] \\
 \mathcal{M}^{', \bullet, fs} \ar[r] & C^{\bullet}}
 \end{eqnarray*} 
and we deduce that in the commutative diagram: 
\begin{eqnarray*}  
 \xymatrix{ \mathcal{M}^\bullet \ar[r] \ar[d] &  \mathcal{M}^{\bullet, fs}  \ar[d] \\
 \mathcal{M}^{', \bullet, fs} \ar[r] & C^{\bullet, fs}}
 \end{eqnarray*}
the right vertical map and the bottom horizontal line are quasi-isomorphisms. 

\end{proof}

\begin{rem}\label{rem-useful-perfect} The proof reveals that there is a non-canonical spectral variety $\pi : \tilde{\mathcal{Z}}  \rightarrow S$ such that $\mathcal{M}^{\bullet,fs} = \pi_\star \mathcal{M}_{\tilde{\mathcal{Z}}}^{\bullet, fs}$ where $\mathcal{M}_{\tilde{\mathcal{Z}}}^{\bullet, fs}$ is a complex of coherent $\oscr_{\tilde{Z}}$-modules, perfect as a complex of $\pi^{-1}\oscr_{S}$-modules, of the same  amplitude as  $M^\bullet$. Moreover, we have a projection $\pi^\star \mathcal{M}^{\bullet} \rightarrow \pi_\star \mathcal{M}_{\tilde{\mathcal{Z}}}^{\bullet, fs}$. This projection is locally on $\tilde{\mathcal{Z}}$ given by an entire series in $\tilde{T}$ with coefficients in $\pi^{-1} \oscr_S$. 
\end{rem}

\subsubsection{Relative spectral theory for an algebra of operators}\label{relative-spectral-2}
We now let $\mathbb{Z}_{\geq 0}^r$ be the free monoid on $r$ generators, generated by element $T_1, \cdots, T_r$. We assume that we have an action of $\mathbb{Z}_{\geq 0}^r$ on an object  $M^\bullet \in Ob(\mathcal{K}^{proj}(\mathbf{Ban}(A)))$. We also assume that the operators $T \in \mathbb{Z}_{\geq 1}^r$ are potent compact.

For any choice of $T \in \mathbb{Z}_{\geq 1}^r$ acting compactly, we have a finite slope projection $\mathcal{M}^\bullet \rightarrow \mathcal{M}^{\bullet, T,  fs}$ and we can consider the cohomlogy $\HH^i (\mathcal{M}^{\bullet, T,fs}) = \HH^i (\mathcal{M}^{\bullet})^{T,fs}$. 

\begin{lem} For any $T,T'$ acting compactly, we have a canonical isomorphism  $\HH^i (\mathcal{M}^{\bullet})^{T,fs} \simeq \HH^i (\mathcal{M}^{\bullet})^{T',fs}$ commuting with the projections $\HH^i (\mathcal{M}^{\bullet}) \rightarrow  \HH^i (\mathcal{M}^{\bullet})^{T,fs}$ and $\HH^i (\mathcal{M}^{\bullet}) \rightarrow  \HH^i (\mathcal{M}^{\bullet})^{T',fs}$.
\end{lem}

\begin{proof} This follows from proposition \ref{prop-independence-choiceT}.
\end{proof}

\begin{rem} It is not clear to us whether   there is a quasi-isomorphism $\mathcal{M}^{\bullet,T',fs} \rightarrow \mathcal{M}^{\bullet, T,  fs}$. In this paper, what we will do is choose a compact operator $T$ once and for all. Ultimately, we are only interested in the cohomology, and the ambiguity disappears.
 We  note that we could only prove that the object $\mathcal{M}^{\bullet, T,  fs}$ is well defined up to a non-canonical quasi-isomorphism (in some sense, it depends on the choice of a compact lift $\tilde{T}$ of $T$). It is not even clear to us that $T'$ will act on $\mathcal{M}^{\bullet, T,  fs}$. One can prove this locally on the spectral variety $\tilde{\mathcal{Z}}_T$: the projection $\pi^\star \mathcal{M}^\bullet \rightarrow \mathcal{M}_{\tilde{\mathcal{Z}}_T}^{\bullet, T,  fs}$ is locally given by an entire series in $\tilde{T}$ and if we fix a lift $\tilde{T}'$ of $T'$ commuting up to homotopy with $\tilde{T}$ one can prove that $\tilde{T}'$ commutes up to homotopy with the projector, hence giving an action of $T'$ on $\mathcal{M}_{\tilde{\mathcal{Z}}_T}^{\bullet, T,  fs}$, locally on $\tilde{\mathcal{Z}}_T$.
\end{rem}

\subsubsection{The case of projective or inductive limits of complexes}\label{relative-spectral-3}   We now work in a slightly more general setting than in section \ref{relative-spectral-1}.  We  consider an object $``\lim_i"M_i^\bullet \in  Ob(\mathrm{Pro}_{\N}(\mathcal{K}^{proj}(\mathbf{Ban}(A))))$ or $``\colim_i"M_i^\bullet \in  Ob(\mathrm{Ind}_{\N}(\mathcal{K}^{proj}(\mathbf{Ban}(A))))$ together with a compact operator $T$. By lemma \ref{lem-potent-compact-factor}, $T$ induces a compact operator $T_i$ of ${M}_i^\bullet$ for $i$ large enough. 

We let $\mathcal{M}_i^\bullet$ be the complex of Banach sheaves over $S$ attached to $M_i^\bullet$. By proposition \ref{prop-existence-decompositionT}, for all $i $ large enough and all $k \in \ZZ$,  we have a projection $\HH^k(\mathcal{M}^\bullet_i) \rightarrow \HH^k(\mathcal{M}_i^\bullet)^{fs}$. 

\begin{lem} \begin{enumerate}
\item For $``\lim_i"M_i^\bullet \in  Ob(\mathrm{Pro}_{\N}(\mathcal{K}^{proj}(\mathbf{Ban}(A))))$, the maps $\HH^k (\mathcal{M}^\bullet_i)^{fs} \rightarrow \HH^k(\mathcal{M}^\bullet_{i-1})^{fs}$ are isomorphisms.
 
\item For $``\colim_i"M_i^\bullet \in  Ob(\mathrm{Ind}_{\N}(\mathcal{K}^{proj}(\mathbf{Ban}(A))))$, the maps $\HH^k (\mathcal{M}^\bullet_i)^{fs} \rightarrow \HH^k(\mathcal{M}^\bullet_{i+1})^{fs}$ are isomorphisms.
\end{enumerate}

\end{lem}

\begin{proof} We only prove the first item. The action of $T_i$ factorizes as  $T_i  : \HH^k(\mathcal{M}^\bullet_i) \rightarrow \HH^k(\mathcal{M}^\bullet_{i-1}) \rightarrow \HH^k(\mathcal{M}^\bullet_i)$. We deduce that there is a bijective map  $T_i  : \HH^k(\mathcal{M}^\bullet_i)^{fs} \rightarrow \HH^k(\mathcal{M}^\bullet_{i-1})^{fs} \rightarrow \HH^k(\mathcal{M}^\bullet_i)^{fs}$. The action of $T_{i-1}$ factorizes as well as  $T_{i-1}  : \HH^k(\mathcal{M}^\bullet_{i-1}) \rightarrow \HH^k(\mathcal{M}^\bullet_{i}) \rightarrow \HH^k(\mathcal{M}^\bullet_{i-1})$ from which we deduce that there is a bijective map $T_{i-1}  : \HH^k(\mathcal{M}^\bullet_{i-1})^{fs} \rightarrow \HH^k(\mathcal{M}^\bullet_{i})^{fs} \rightarrow \HH^k(\mathcal{M}^\bullet_{i-1})^{fs}$.
\end{proof}

Let  $\mathcal{M}^\bullet \in \mathcal{D}(\mathrm{Mod}(\oscr_S))$ be either $\lim_i \mathcal{M}_i^\bullet$ or $\colim_i\mathcal{M}_i^\bullet$. Then we let $\HH^k(\mathcal{M}^\bullet)^{fs} = \HH^k(\mathcal{M}_i^\bullet)^{fs}$ for any large enough $i$. 
We also define $\mathcal{M}^{\bullet, fs} = \mathcal{M}_i^{\bullet, fs}$ for some chosen $i$ large enough. We remark that this definition depends on the choice of $i$. But the ambiguity disappears when we pass to cohomology. 

\subsection{Locally analytic inductions and the locally analytic BGG resolution}

In this section we recall some basic facts about analytic inductions and BGG resolutions for $p$-adic groups. Standard references for this material are \cite{MR2846490}, section 3 and \cite{JONES20111616}. The notations for this section are as follows. We let $F$ be a finite extension of $\qq_p$, and we let   $M \rightarrow \Spec~\ocal_F$  be a split reductive group. We fix a maximal torus $T$ and a Borel $B$ containing $T$. We let $\Phi_M = \Phi_M^+ \coprod \Phi_M^-$ be the root system. We also let $\Delta_M \subset \Phi_M^+$ be the positive simple roots. We denote by $W_M$ the Weyl group, and denote by $\ell : W_M \rightarrow \ZZ_{\geq 0}$ the length function.  For all $i \in  \ZZ_{\geq 0}$, we let $W_M^{(i)}$ be the set of elements in $W_M$ of length $i$.  We denote by $w_{0,M}$ the longest element of $W_M$. We let $\rho_M$ be half the sum of the positive roots.  The Weyl group acts on $T$, $X_\star(T)$ and $X^\star(T)$. We also have the  dotted action  on $X^\star(T)$ given by $w \cdot \kappa = w(\kappa + \rho_M)- \rho_M$. We let $X^\star(T)^{M,+}$ be the cones of dominant characters, and we let  $X^\star(T)^{M,++}$ be the cones of dominant regular  characters.  We use a $-$ sign to denote the opposite cones. We assume that $T_F = T \times_{\Spec~\ocal_F} \Spec~F$ is in fact defined over $\qq_p$. Namely, there is a torus $T_{\qq_p}$ and an isomorphism $T_{\qq_p} \times \Spec~F = T_F$.  We often drop the subscripts $F$ or $\qq_p$ when the context is clear. 

\begin{rem} In our applications, $M$ will be the Levi  $M_\mu^c$ of the group $G$ which is part of the Shimura datum.  A slight warning is that   the torus $T_{\qq_p}$ will in general  be  the conjugate of a maximal torus of $G_{\qq_p}$ by an element of  the absolute Weyl group of $G$ which is not necessarily rational.  
\end{rem}

We let $T^d \subseteq T_{\qq_p}$ be the maximal split subtorus. We let $T(\ZZ_p) \subseteq T(\qq_p)$ be the maximal compact subgroup. 
There is a valuation map $v : T(\qq_p) \rightarrow X_\star(T^d) \otimes \qq$ whose image is a lattice and whose kernel is $T(\ZZ_p)$.  
We let $T^{M,+}$ be the monoid in $T(\qq_p)$ of elements $t$ such that $v(\alpha(t)) \geq 0$ for all $\alpha \in \Phi_M^+$ and we let $T{M,++}$ be the monoid in $T^{M,+}$ of elements $t$ such that $v(\alpha(t)) > 0$ for all $\alpha \in \Phi_M^+$.

 For any $\kappa \in X^\star(T)$,
  we let $F(\kappa)$ be the one dimensional $F$ vector space endowed with the action of $T(\qq_p)$ via the character $\kappa$.  If $V$ is a $F$-vector space endowed with an action of (a submonoid of) $T(\qq_p)$, we let $V(\kappa) =   V \otimes F(\kappa)$.   

 We let $\mathcal{M}$ be the  quasi-compact adic space over $\Spa(F, \ocal_F)$ attached to $M$ and we denote by $\mathcal{M}_n$ the subgroup of $\mathcal{M}$ of elements reducing to $1$ modulo $p^n$. We define in a similar fashion $\mathcal{T}$, the quasi-compact torus over $\Spa(F, \ocal_F)$ attached to $T$ and $\mathcal{T}_n$ the subgroup of $\mathcal{T}$ of elements reducing to $1$ modulo $p^n$. 
 
 We let ${M}_1 \subseteq {M}(\ocal_F)$ be a closed subgroup possessing an Iwahori decomposition, in the sense that the product map  $$\overline{N}_1 \times T_1 \times N_1 \rightarrow M_1$$ is an isomorphism,  where $N_1 = M_1 \cap U$, $T_1 = M_1 \cap T$, $\overline{N}_1 = M_1 \cap \overline{U}$ (for $U$ and $\overline{U}$ the unipotent radical of $B$ and the opposite Borel $\overline{B}$ respectively).  We also assume that $T_1 = T(\ZZ_p)$, and that $T^{M,-}$ normalizes $\overline{N}_1$.

\subsubsection{Algebraic inductions}
Let $\kappa \in X^{\star}(T)^{M,+}$.  We have the algebraic representation $V_\kappa$ of $M$ with highest weight $\kappa$. It can be realized as an algebraic induction:
\begin{eqnarray*}
V_{\kappa} &=& \mathrm{Ind}_{B}^{M}(w_{0, M} \kappa) \\
&=& \{ f : M \rightarrow \mathbb{A}^1\mid f(mb) = (w_{0, M} \kappa)(b^{-1}) f(m),~\forall~(m,b) \in M \times B  \} 
\end{eqnarray*}
We have a left action of $M$ given by $h  f (m) = f(h^{-1}m)$ for any $h \in M$.

\subsubsection{Analytic weights} Let $(A,A^+)$ be a complete Tate algebra over $(F, \ocal_F)$  and let $\kappa_A$ be a continuous morphism $T(\ZZ_p) \rightarrow A^\times$. Let $n \in \ZZ_{\geq 0}$. We say that $\kappa_A$ is $n$-analytic if the map $\kappa_A$ can be extended to a pairing:
$$ T(\ZZ_p) \mathcal{T}_n \times \Spa(A,A^+) \rightarrow \mathbb{G}_m^{an}$$ where $T(\ZZ_p)\mathcal{T}_n$ is the subgroup of $\mathcal{T}$ generated by $T(\ZZ_p)$ and $\mathcal{T}_n$ (this is a finite union of translates of $\mathcal{T}_n$).
We recall that any continuous character $\kappa_A$ is $n$-analytic for some $n$ (see \cite{MR2846490}, lemma 3.2.5. for example).

\subsubsection{Analytic inductions}\label{subsection-analyinduction} Let $(A,A^+)$ be a complete Tate algebra over $(F, \ocal_F)$ and let $S = \Spa(A,A^+)$. We  assume that $A$ is uniform (i.e $A^0$ is bounded) and equip $A$ with the supremum norm.  Let $n_0 \in \ZZ_{\geq 0}$ be an integer.  We now fix a character $\kappa_A : w_{0,M}^{-1}T(\ZZ_p)w_{0,M} \rightarrow A^\times$ which is $n_0$-analytic, or equivalently a character $w_{0,M}\kappa_A:T(\ZZ_p)\rightarrow A^\times$ which is $n_0$-analytic.  

\begin{rem} The notation $w_{0,M}\kappa_A$ may seem strange so let us explain it to orient the reader. Let $\kappa \in X^\star(T)$ be an algebraic character of $T$. Then for any $w \in W_M$, we have $\langle w \kappa, t \rangle = \langle \kappa, w^{-1} t \rangle$ so  that $w\kappa (t)  = \kappa (w^{-1} t w)$. 
\end{rem}

Let $\mathcal{B}$ be the quasi-compact adic space attached to $B$. For all $n \geq n_0$ we  can define $$V_{\kappa_A}^{n-an} =  \mathrm{an-Ind}^{M_1\mathcal{M}_n}_{\mathcal{B}\cap (M_1\mathcal{M}_n)}(w_{0,M}\kappa_A) = $$ $$  \{ f : (M_1\mathcal{M}_n)_S \rightarrow \mathbb{A}^{1,an}_S\mid f(mb) = (w_{0,M}\kappa_A)(b^{-1})  f(m),~\forall~(m,b) \in   (M_1\mathcal{M}_n)_S \times (\mathcal{B} \cap (M_1\mathcal{M}_n))_S \}.$$ 

This is a  Banach $A$-module  for the supremum norm. We let $V_{\kappa_A}^{n-an,+}$ be the module of elements with supremum norm less or equal than one. 
The  space $V_{\kappa_A}^{n-an}$ carries  the following actions of the group $(M_1 \mathcal{M}_{n})_S$  and of the monoid  $T^{M,+}$:
\begin{itemize}
\item $hf(m) = f(h^{-1}m)$ for $h,m \in (M_1 \mathcal{M}_{n})_S$,
\item $tf(m)=  f( t^{-1} \overline{n}_m t_m t) $ for $t \in T^{M,+}$, $m \in (M_1 \mathcal{M}_{n})_S$, and $m =  \overline{n}_m t_m n_m$ the Iwahori decomposition of $m$. 
\end{itemize}

These actions respect the submodule $V_{\kappa_A}^{n-an,+}$. 

We record the following lemma for later use. Let $M_1^p$ be the  quotient of $M_1$ by its maximal normal pro-$p$ subgroup. 

\begin{lem}\label{lem-inductive-lim-rep}   The representation $V_{\kappa_A}^{n-an,+} \otimes_{A^+} A^+/A^{++}$ of $M_1\mathcal{M}_n$  is a countable inductive limit of finite projective $A^+/A^{++}$-submodules $V_i$ stable under the action of  $M_1\mathcal{M}_n$, and with the property that  the action on $V_{i+1}/V_i$ factors through an action of   $M_1^p$.
\end{lem} 
\begin{proof} The character $\kappa_A ~\mod A^{++}$ factors through a finite character and therefore it is locally constant  on the reduced ring $A^+/A^{++}$. We may assume it is constant.  We are then reduced to the case that $A$ is a finite field extension of $F$ and $A^+/A^{++}$ is a finite field. The stabilizer of any vector $v \in V_{\kappa_A}^{n-an,+} \otimes_{A^+} A^+/A^{++}$ is open in $M_1\mathcal{M}_n$ and we deduce that $V_{\kappa_A}^{n-an,+} \otimes_{A^+} A^+/A^{++}$ is a countable inductive limit of finite dimensional representations. Since pro-$p$ groups have non-zero fixed vectors on non-trivial finite dimensional representations in characteristic $p$, the claim follows.
\end{proof} 

We also let $V_{\kappa_A}^{lan} = \mathrm{colim}_n V_{\kappa_A}^{n-an}$ be the locally analytic induction. 

\begin{lem}\label{lem-slopeboundBGG} The operators $t \in T^{M,++}$ are compact on $V_{\kappa_A}^{n-an}$.  Moreover, if $A$ is a field, the maps $V_{\kappa_A}^{n-an} \rightarrow V_{\kappa_A}^{n+1-an}$ induce isomorphisms on the finite slope part, and the slopes of $t \in T^{M,+}$ on   $V_{\kappa_A}^{lan, fs}$ are $\geq0$.  
\end{lem}
\begin{proof} If $t \in T^{M,+ +}$, one sees easily that the map $t : V_{\kappa_A}^{n-an} \rightarrow V_{\kappa_A}^{n-an}$ improves analyticity. In particular, if $\min_{ \alpha \in \Delta_M} v( \alpha(t)) \geq 1$, one  has a factorization $$V_{\kappa_A}^{n-an} \rightarrow V_{\kappa_A}^{n+1-an} \rightarrow V_{\kappa_A}^{n-an}$$ where the first inclusion is compact.  Moreover, we deduce from the definition  of the action that if $t \in T^{M,+}$, then $t$ preserves the open and bounded submodule $V_{\kappa_A}^{n-an,+}$.
\end{proof}

The space  $V_{\kappa_A}^{n-an}$ embeds in $\HH^0( M_1\mathcal{M}_n, \oscr_{M_1\mathcal{M}_n})$ and similarly $V_{\kappa_A}^{lan}$ embeds in $$\colim_n \HH^0( M_1\mathcal{M}_n, \oscr_{M_1\mathcal{M}_n}).$$
We have a left action of $M_1\mathcal{M}_n$ on $\HH^0( M_1\mathcal{M}_n, \oscr_{M_1\mathcal{M}_n})$ given by $h \ast f(m) = f(mh)$. 
Passing to the limit over $n$ and differentiating, we get an action of the Lie algebra $\mathfrak{m}$ of $M$ on  $\colim_n \HH^0( M_1\mathcal{M}_n, \oscr_{M_1\mathcal{M}_n})$ which can be extended to an action of  the enveloping algebra $\mathcal{U}(\mathfrak{m})$. 

\begin{rem} We are not requiring that  $M_1 \subseteq M(\ocal_F)$ is an open subgroup, and $M_1$ will not always be Zariski dense in $M$. 
In our definition of the locally analytic induction, we consider analytic functions on neighborhoods of $M_1$ in $\mathcal{M}$. These functions are not necessarily determined by their restrictions to $M_1$. 
\end{rem}

\subsubsection{Twist by a finite order character} We  fix a character $\kappa_A : w_{0,M}^{-1}T(\ZZ_p)w_{0,M} \rightarrow A^\times$ which is $n$-analytic. Let $w_{0,M}\chi : M_1 \rightarrow \overline{F}^\times$ be a finite order character which we assume is trivial on $M_1\cap \mathcal{M}_n$. We also denote by $w_{0,M}\chi$ its restriction to $T(\ZZ_p)$ and by $\chi$ the corresponding character of $w_{0,M}^{-1}T(\ZZ_p)w_{0,M}$, which is trivial on $w_{0,M}^{-1}T(\ZZ_p)w_{0,M}\cap\mathcal{T}_n$ and hence $n$-analytic.  We also denote by $w_{0,M}\chi$ the corresponding $1$-dimensional representation of $M_1$. We endow it with the trivial action  of $T^{M,+}$. 
We have the following lemma:  
\begin{lem}\label{lem-add-character} There is  a canonical map $V^{n-an}_{\kappa_A} \otimes_F w_{0,M} \chi \rightarrow V^{n-an}_{\kappa_A \otimes \chi} $ which is an isomorphism of $(M_1, T^{M,+})$-modules.
\end{lem} 
\begin{proof} The map is defined by  sending a function $a \otimes 1 \in V^{n-an}_{\kappa_A}$ to the  function $a w_{0,M} \chi^{-1}$ on $M_1\mathcal{M}_n$. The rest of the lemma follows easily and is left to the reader.
\end{proof}

\subsubsection{Locally algebraic induction}\label{section-locally-algebraic-induction} Let $\kappa \in X^\star(T)^{M,+}$.  From $\kappa$ we obtain a character $\kappa:w_{0,M}^{-1}T(\ZZ_p)w_{0,M}\to F^\times$ as the composition of the inclusion $w_{0,M}^{-1}T(\ZZ_p)w_{0,M}\subset T(F)$ and $\kappa:T(F)\to F^\times$.
From the definitions, restriction induces a natural inclusion $\iota : V_\kappa \hookrightarrow V_\kappa^{lan}$. There is an action of $M$ on $V_\kappa$, and therefore actions of $M_1$ and $T^{M,+}$. The map $\iota $ is $M_1$-equivariant, but not $T^{M,+}$-equivariant. More precisely, we have the following formula: $$\iota (t v) = (w_{0,M}\kappa)(t) t \iota (v)$$
from which we deduce that the twisted map $\iota : V_\kappa \rightarrow V_\kappa^{lan} (w_{0,M} \kappa)$ is $(M_1, T^{M,+})$-equivariant. 
We can define the subspace $V_{\kappa}^{lalg}$ of $V_{\kappa}^{lan}$, consisting of elements arising from functions on $M_1\mathcal{M}_n$ for some $n$ which on each component of $M_1\mathcal{M}_n$ are the restriction of a polynomial function on $M$. This is a $(M_1,T^{M,+})$-subrepresentation. This space contains the space $V_{\kappa}(-w_{0,M} \kappa)$ of algebraic functions. Let 
$$ V_{\mathbf{1}}^{sm} = \mathrm{sm-Ind}_{B\cap M_1}^{M_1} \mathbf{1} = $$ $$ \{f : M_1 \rightarrow F\mid f(mb) = f(m)~\forall~(m,b) \in M_1 \times B\cap M_1,~f ~\textrm{is locally constant}\}.$$
The map $V_{\kappa} \otimes V_{\mathbf{1}}^{sm} \rightarrow V_{\kappa}^{lalg}(w_{0,M} \kappa)$ is an isomorphism of $(M_1, T_1^+)$-representations.

\subsubsection{The BGG complex}  For all $\alpha \in \Delta_M$, we fix a generator $X_\alpha$ of the root space $\mathfrak{u}_\alpha \subseteq \mathfrak{m}$ and we have the corresponding generator $X_{-\alpha}$ of $\mathfrak{u}_{-\alpha}$. Let $\kappa \in X^\star(T)$. 
\begin{lem}\label{lem-defi-maps-teta} For all $\alpha \in \Delta_M$ such that $\langle \kappa, \alpha^\vee\rangle \geq -1$, we have   maps:  
\begin{eqnarray*}
\Theta_{\alpha} : V_{\kappa}^{lan}(w_{0,M} \kappa) &\rightarrow& V_{s_\alpha \cdot  \kappa}^{lan}( w_{0,M} (s_{\alpha}\cdot  \kappa)) \\
f & \mapsto & X_{w_{0,M}\alpha}^{\langle   \kappa, \check{\alpha}\rangle +1} \ast f 
\end{eqnarray*}
equivariant for the action of $(M_1,T^{M,+})$.
\end{lem}
\begin{proof} These maps are constructed in \cite{MR2846490}, proposition 3.2.11, remark 3.3.11, proposition 3.3.12, as well as  \cite{JONES20111616}, section 5, see also the remark below theorem 13.
\end{proof}

\begin{rem}  The normalization of \cite{MR2846490} and \cite{JONES20111616} is slightly different from the one we use here. They realize the highest weight $\kappa$ representation 
$V_\kappa$ as the following induction: $ V_\kappa' = \{ f : M \rightarrow \mathbb{A}^1\mid f(\overline{b} m) = \kappa(\overline{b}) f(m),~\forall~(m,\overline{b}) \in M \times B  \}$. To translate to our  setting, one simply  applies the involution  $m \mapsto w_{0,M} m^{-1} w_{0,M}$.   We therefore get an isomorphism $V_\kappa \rightarrow V_\kappa'$ which sends $f$ to $f'$ defined by $f'(m) = f(w_{0,M} m^{-1} w_{0,M})$. There is an action of $M$ on $V_\kappa'$ induced from the right translation action of $M$ on itself. We find that $m f' = ((w_{0,M} m w_{0,M}^{-1}) f)'$. This explains the twist by $w_{0,M}$ appearing in lemma \ref{lem-defi-maps-teta} compared to \emph{loc. cit}.
\end{rem}

\begin{rem} We have $$w_{0,M} ( s_\alpha \cdot \kappa-\kappa) = - (\langle \kappa, \alpha^\vee\rangle +1) w_{0,M} \alpha.$$ In particular, for any $t \in T^{M,+}$, 
$(\langle \kappa, \alpha^\vee\rangle +1) \langle - w_{0,M} \alpha, v(t) \rangle \geq 0$.  This means that $\Theta_\alpha$ increases the slopes. 
\end{rem}

\begin{thm}[\cite{JONES20111616}, thm. 26, \cite{MR2846490}, sect. 3.3.9]\label{thm-BGG-resolution} There is an exact sequence of $(M_1,T^{M,+})$-representations 
$$ 0 \rightarrow  V_{\kappa} \otimes V_{\mathbf{1}}^{sm}  \rightarrow V_{\kappa}^{lan}(w_{0,M} \kappa) \rightarrow \bigoplus_{ w \in W_M^{(1)}} V^{lan}_{ w \cdot \kappa}(  w_{0,M}(w \cdot \kappa)) \rightarrow \cdots \rightarrow $$ $$ 
\bigoplus_{ w \in W_M^{(i)}} V^{lan}_{ w \cdot  \kappa}(w_{0,M}(w \cdot \kappa)) \rightarrow  \cdots \rightarrow V_{w_{0,M} \cdot \kappa}^{lan}(w_{0,M} (w_{0,M} \cdot \kappa)) \rightarrow 0$$
where the first map $V_{\kappa} \otimes V_{\mathbf{1}}^{sm}   \rightarrow V_{\kappa}^{lan}(w_{0,M} \kappa)$ is the natural inclusion, the second map  $ V_{\kappa}^{lan}(w_{0,M} \kappa) \rightarrow \bigoplus_{ w \in W_M^{(1)}} V^{lan}_{ w \cdot \kappa}(  w_{0,M}(w \cdot \kappa))$ is a linear combination of the maps $\Theta_{\alpha}$ for $\alpha \in \Delta_M$, and more generally  the  differentials $\oplus_{ w \in W_M^{(i)}} V^{lan}_{ w \cdot  \kappa}(w_{0,M} (w \cdot \kappa))\rightarrow \oplus_{ w \in W_M^{(i+1)}} V^{lan}_{ w \cdot  \kappa}(w_{0,M} (w \cdot \kappa))$ are linear combinations  of  maps of the form $X\ast \cdot$ for suitable elements $X \in \mathcal{U}(\mathfrak{m})$.  
\end{thm}

\begin{defi} Let $\kappa \in X_\star(T)^{M,+}$. We say that a $T^{M,+}$- eigensystem $\lambda$  in $V_\kappa^{lan}$  is of $M$-small slope (abbreviated $+,ss_M$) if $$v(\lambda) < -  (\langle \kappa, \alpha^\vee\rangle +1)w_{0,M} \alpha$$ for some $\alpha \in \Delta_M$.

We say that a $T^{M,+}$- eigensystem $\lambda$  in $V_\kappa$  is of $M$-small slope (abbreviated $+,ss_M$) if $$v(\lambda) <  w_{0,M}\kappa -  (\langle \kappa, \alpha^\vee\rangle +1)w_{0,M} \alpha$$ for some $\alpha \in \Delta_M$.
\end{defi}

\begin{coro}\label{coro-BGG-classical-1} The map $V_{\kappa}^{+,ss_M} \rightarrow V_{\kappa}^{lan,+,ss_M}(w_{0,M} \kappa)$ is an isomorphism of $(M_1, T^{M,+})$-modules.

\end{coro}
\begin{proof} It follows from theorem \ref{thm-BGG-resolution} and lemma \ref{lem-slopeboundBGG},  that the map $V_{\kappa} \otimes V_{\mathbf{1}}^{sm} \rightarrow V_{\kappa}^{lan}(w_{0,M} \kappa)$ is an isomorphism on the small slope part. On the other hand, the map $V_\kappa \rightarrow V_{\kappa} \otimes V_{\mathbf{1}}$ is an isomorphism on the finite slope part. 
\end{proof}

\begin{ex} Let $M = \mathrm{SL}_2/\qq_p$, with diagonal torus $T$ and upper triangular Borel $B$. Then $X^\star(T) = \ZZ$, and there is a unique simple root $\alpha = 2$. For any $k \in X^\star(T)^{M,+} = \ZZ_{\geq 0}$, we have $V_k = \mathrm{Sym}^k(\mathrm{St})$. The valuation of the eigenvalues of $t = \mathrm{diag}(p,p^{-1})$ on $V_k$ are $-k, -k+2, \cdots, k$. 
The $M-ss$ condition on $V_k$ translates into the condition that the eigenvalues of $t$ have valuation $< -k + 2(k+1) = k+2$. This condition is always satisfied (in the case of $\mathrm{SL}_2$) and we have $V_k = V_k^{M-ss}$. The space $V_{k}^{lan}$ identifies with the space  of locally analytic functions on $p\ZZ_p$ and the action of $t$ is given by $f(z) \mapsto f(p^2z)$. A basis of finite slope vectors in $V_k^{lan}$ is given by the monomial functions $z \mapsto z^n$ for $n \in \ZZ_{\geq 0}$. The slopes of $t$ on this basis are $0, 2, \cdots, 2n, \cdots $. 
The inclusion $V_k  \hookrightarrow V_k^{lan}$ identifies $V_k$ with the polynomial functions of degree $\leq k$, which is indeed the space defined by the slope condition $M-ss$. 
\end{ex}

\subsubsection{Distributions}
Let $\kappa_A :w_{0,M}^{-1}T(\ZZ_p)w_{0,M} \rightarrow A^\times$ be $n_0$-analytic.  For all $n \geq n_0$,  
we now consider $ (V_{\kappa_A}^{n-an})^\vee$  the continuous $A$-dual. This is a Banach $A$-module. However it is not a projective Banach $A$-module in general (it is so if $A$ is a field).  In order to remedy this, we introduce another related space. Let  $\overset{\circ}{\mathcal{M}}_n$ be the open subset of $\mathcal{M}_n$ of elements $m \in \mathcal{M}_n$ such that $m = 1 ~\mod p^{n + \epsilon}$ for some $\epsilon >0$ (this is an ``open polydisc'').  Let 
\begin{eqnarray*}
{V}_{\kappa_A}^{\circ, n-an} &= &\\
  \{  f : (M_1\overset{\circ}{\mathcal{M}}_n)_S \rightarrow \mathbb{A}^{1,an}_S&,& f(mb) = (w_{0,M}\kappa_A)(b^{-1})  f(m),~\forall~(m,b) \in   (M_1\mathcal{M}_n)_S \times (\mathcal{B} \cap (M_1\mathcal{M}_n))_S \\ && f ~\textrm{ is bounded} \}
\end{eqnarray*} 
Let ${V}_{\kappa_A}^{\circ,n-an,+} \subseteq{V}_{\kappa_A}^{\circ,n-an}$ be the subset of functions which are bounded by $1$. The space ${V}_{\kappa_A}^{\circ,n-an,+}$ is a projective limit of finite free $A^+$-module, which defines a topology on it, and this determines a topology on $V_{\kappa_A}^{\circ,n-an}$ (with this topology$V_{\kappa_A}^{\circ,n-an}$ is not a Banach module).

We have a map $V_{\kappa_A}^{n-an} \rightarrow{V}_{\kappa_A}^{\circ, n-an} $ with dense image. 
 We let ${D}^{n-an}_{\kappa_A} = ({V}_{\kappa_A}^{\circ,n-an})^\vee$ be the continuous $A$-dual, equipped with the strong topology. This is a projective Banach $A$-module and we have a map ${D}^{n-an}_{\kappa_A}  \rightarrow  (V_{\kappa_A}^{n-an})^\vee$ with dense image.  We let $D^{n-an, +}_{\kappa_A}$ be the continuous $A^+$-dual of $V_{\kappa_A}^{\circ,n-an, +}$. It is an open and bounded submodule of ${D}^{n-an}_{\kappa_A}$. We let ${D}^{lan}_{\kappa_A}= \lim_n {D}^{n-an}_{\kappa_A} = \lim_n (V_{\kappa_A}^{lan})^\vee $.  This is a compact projective limit of  Banach $A$-modules (the distributions of weight $\kappa_A$). There is a perfect pairing:
$$ \langle-,- \rangle :  {V}^{lan}_{ \kappa_A} \times D_{\kappa_A}^{lan} \rightarrow A.$$
The space   $D^{lan}_{\kappa_A}$ carries a right action of $(M_1, T^{M,+})$ defined by $\langle m f , \mu \rangle = \langle f, \mu m \rangle$, $\langle t f, \mu \rangle = \langle f, \mu t \rangle$ for $(t, m, f, \mu) \in T^{M,+} \times M_1 \times V_{\kappa_A}^{lan} \times D^{lan}_{ \kappa_A}$. The space $D^{lan}_{\kappa_A}$ therefore carries a left  action of $(M_1, T^{M,-})$ defined by $\langle m^{-1} f, \mu \rangle = \langle f, m \mu \rangle$, $\langle t^{-1} f, \mu \rangle = \langle f, t \mu \rangle$ for $ (t, m, f, \mu) \in T^{M,-} \times M_1 \times V_{\kappa_A}^{lan} \times D^{lan}_{ \kappa_A}$. 
The action of $T^{M,--}$ is by compact operators on $D^{lan}_{\kappa_A}$.  

Let $\kappa \in X^\star(T)^+$.  By dualizing the exact sequence of theorem \ref{thm-BGG-resolution}, we get the following complex of $(M_1, T^{M,-})$-representations: 
$$ 0 \rightarrow   D^{lan}_{ w_{0,M}.\kappa} ( -w_{0,M} (w_{0,M} \cdot \kappa)) \rightarrow  \cdots \rightarrow \oplus_{ w \in W_M^{(i)}} D^{lan}_{ w \cdot  \kappa}(-w_{0,M}  (w \cdot \kappa)) $$ $$\rightarrow \cdots  \rightarrow \oplus_{ w \in W_M^{(1)}} D^{lan}_{ w \cdot \kappa}( -w_{0,M} (w \cdot \kappa) ) \rightarrow D^{lan}_{\kappa}(-w_{0,M} \kappa) \rightarrow  0 $$
This complex is exact except in the last degree. The cokernel  of $\oplus_{ w \in W_M^{(1)}} D^{lan}_{ w \cdot \kappa}(- w_{0,M} (w \cdot \kappa) ) \rightarrow D^{lan}_{\kappa}(-w_{0,M} \kappa)$ maps to $V_\kappa^\vee$. 
Passing to the finite slope part  gives an exact sequence: 
$$ 0 \rightarrow   D^{lan}_{ w_{0,M}.\kappa} (-w_{0,M} (w_{0,M} \cdot \kappa) )^{fs} \rightarrow  \cdots \rightarrow \oplus_{ w \in W_M^{(i)}} D^{lan}_{ w \cdot  \kappa}(-w_{0,M}  (w \cdot \kappa))^{fs} $$ $$\rightarrow \cdots  \rightarrow \oplus_{ w \in W_M^{(1)}} D^{lan}_{ w \cdot \kappa}( -w_{0,M} (w \cdot \kappa))^{fs} \rightarrow D^{lan}_{\kappa}(-w_{0,M} \kappa)^{fs} \rightarrow  V_\kappa^\vee \rightarrow 0 $$

\begin{defi} Let $\kappa \in X_\star(T)^{M,+}$. We say that a $T^{M,-}$- eigensystem $\lambda$  in $D_\kappa^{lan}$  is of $M$-small slope (abbreviated $-,ss_M$) if $$v(\lambda) >  (\langle \kappa, \alpha^\vee\rangle +1)w_{0,M} \alpha$$ for some $\alpha \in \Delta_M$.

We say that a $T^{M,-}$- eigensystem $\lambda$  in $V^\vee_\kappa$  is of $M$-small slope (abbreviated $-,ss_M$) if $$v(\lambda) > -  w_{0,M}\kappa +  (\langle \kappa, \alpha^\vee\rangle +1)w_{0,M} \alpha$$ for some $\alpha \in \Delta_M$.
\end{defi}

We have the following control theorem:

\begin{coro} The map $(D_\kappa^{lan})^{-,ss_M}(-w_{0,M}\kappa) \rightarrow (V^\vee_{\kappa})^{-,ss_M}$ is an isomorphism of $(M_1,T^{M,-})$-modules.
\end{coro}

\begin{proof} This is the dual  of corollary \ref{coro-BGG-classical-1}. \end{proof}

\subsection{$p$-adic families of  sheaves} 

\subsubsection{Definition of the sheaves}
We  let $K_p = K_{p,m',0}$ with $m' >0$. Let $w \in \WM$.  By the results of section \ref{sect-further-restriction-group-structure}, for any $n \geq 0$, over $(\pi_{HT, K_p}^{tor})^{-1}(]C_{w,k}[_{n,n} K_p)$ the torsor $\mathcal{M}_{dR}^{an}$ has a reduction to a torsor $\mathcal{M}_{HT,n,K_p}$ under the group $K_{p,w,M_\mu}^c\mathcal{M}_{\mu,n}^c$.
The group $K_{p,w,M_\mu}^c$ has an Iwahori decomposition by proposition \ref{prop-Iwahori-deco-of-weird-group}. Moreover, $K_{p,w,M_\mu}^c \cap T^c = wT^c(\ZZ_p)w^{-1}$. 

Let  $(A, A^+)$ be a Tate algebra over $(F, \ocal_{F})$.  Let $\nu_A:T^c(\ZZ_p) \rightarrow A^\times$ be an $n$-analytic character.   Let $\kappa_A:w_{0,M}wT^c(\ZZ_p)(w_{0,M}w)^{-1}\rightarrow A^\times$ be given by $\kappa_A=-w_{0,M}w\nu_A-(w_{0,M}w\rho+\rho)$.

We can construct a sheaf $\mathcal{V}_{\nu_A}^{n-an}$ over $(\pi_{HT,K_p}^{tor})^{-1}(]C_{w,k}[_{n,n}K_p)$,  modeled on $V_{\kappa_A}^{n-an}$.  Namely consider the torsor $\pi : \mathcal{M}_{HT,n,K_p} \rightarrow (\pi_{HT,K_p}^{tor})^{-1}(]C_{w,k}[_{n,n}K_p)$ and $\pi  \times 1 : \mathcal{M}_{HT,n,K_p} \times \Spa(A,A^+) \rightarrow (\pi_{HT}^{tor})^{-1}(]C_{w,k}[_{n,n}K_p) \times \Spa (A, A^+)$. We let $\mathcal{V}_{\nu_A}^{n-an}$ be the subsheaf of $(\pi\times 1)_\star \oscr_{\mathcal{M}_{HT,n,K_p} \times \Spa (A,A^+)}$ of sections which satisfy $f(m b) = (w_{0,M}\kappa_A)(b^{-1}) f(m)$ for all $b \in \mathcal{B}^c \cap (K_{p,w,M_\mu}^c\mathcal{M}^c_{\mu,n})$.  We also have $\mathcal{D}_{\nu_A}^{n-an} \subseteq (\mathcal{V}_{\nu_A}^{n-an})^\vee\otimes \mathcal{V}_{-2\rho_{nc}}=(\mathcal{V}_{\nu_A-2w^{-1}\rho_{nc}}^{n-an})^\vee$, a sheaf locally modeled on $D_{\kappa_A+2\rho_{nc}}^{n-an}$.

\subsubsection{First properties}  We prove that the interpolation sheaves are locally projective Banach sheaves.

\begin{prop}\label{dist-sheaves-Banach} The sheaves $\mathcal{V}_{\nu_A}^{n-an}$ and $\mathcal{D}_{\nu_A}^{n-an}$ are locally projective Banach sheaves over  $(\pi_{HT,K_p}^{tor})^{-1}(]C_{w,k}[_{n,n}K_p)$. More precisely, for any affinoid  $\mathcal{U} = \Spa(R,R^+) \hookrightarrow (\pi_{HT,K_p}^{tor})^{-1}(]C_{w,k}[_{n,n}K_p)$, $\mathcal{V}_{\nu_A}^{n-an}(\mathcal{U})$ and $\mathcal{D}_{\nu_A}^{n-an}(\mathcal{U})$ are projective Banach $R \hat{\otimes}_F A$-modules and the maps  $\mathcal{V}_{\nu_A}^{n-an}(\mathcal{U}) \hat{\otimes}_{R} \oscr_{\mathcal{U}} \rightarrow \mathcal{V}_{\nu_A}^{n-an}\vert_{\mathcal{U}}$ and $\mathcal{D}_{\nu_A}^{n-an}(\mathcal{U}) \hat{\otimes}_{R} \oscr_{\mathcal{U}} \rightarrow \mathcal{D}_{\nu_A}^{n-an}\vert_{\mathcal{U}}$ are isomorphisms. 
\end{prop}

\begin{proof} It follows from proposition \ref{prop-torsor-become-trivial}, that there is a finite flat morphism $f : \mathcal{U}' = \Spa (R', (R')^+) \rightarrow \mathcal{U}$,  a finite group $H$ acting on $\mathcal{U}'$ such that $\mathcal{U}'/H = \mathcal{U}$, such that the torsor $\mathcal{M}_{HT,n,K_p}\vert_{\mathcal{U}'}$ is trivial.  It follows that $\mathcal{V}_{\nu_A}^{n-an}(\mathcal{U}) = (V_{\kappa_A}^{n-an} \hat{\otimes} R')^H$ (where $H$ acts semi-linearly) is a direct factor of the projective $A \hat{\otimes}_F R$-module $V_{\kappa_A}^{n-an} \hat{\otimes} R'$.  Similarly, $$ \mathcal{V}_{\nu_A}^{n-an}\vert_{\mathcal{U}} = (V_{\kappa_A}^{n-an} \otimes_F \oscr_{\mathcal{U}'})^H =  (V_{\kappa_A}^{n-an} \otimes_F R' \otimes_R  \oscr_{\mathcal{U}})^H = \mathcal{V}_{\nu_A}^{n-an}(\mathcal{U}) \hat{\otimes}_{R} \oscr_{\mathcal{U}}.$$
The case of the distribution sheaf is similar and left to the reader. 
\end{proof}

\begin{coro}\label{coro-inpro}   Let $U \subseteq (\pi_{HT, K_p}^{tor})^{-1}(]C_{w,k}[_{n,n} K_p)$ be an open subset which is a finite union of quasi-Stein opens. Then $$\mathrm{R}\Gamma_{et}(U, \mathcal{V}_{\nu_A}^{n-an}) = \mathrm{R}\Gamma_{an}(U, \mathcal{V}_{\nu_A}^{n-an}),\quad\text{and}\quad \mathrm{R}\Gamma_{et}(U, \mathcal{D}_{\nu_A}^{n-an}) = \mathrm{R}\Gamma_{an}(U, \mathcal{D}_{\nu_A}^{n-an})$$ and they are both representable by objects of $\mathrm{Pro}_{\N}(\mathcal{K}^{proj}(\mathbf{Ban}(A)))$.
\end{coro}
\begin{proof} This follows from proposition \ref{dist-sheaves-Banach} and lemma \ref{lem-complex-frechet} (or a slight elaboration of it).
\end{proof}

\subsubsection{Interpolation sheaves and locally algebraic weights}
 
 Let $\nu\in X^\star(T^c)$ and let $\kappa=-w_{0,M}w\nu-(w_{0,M}w\rho+\rho)$ be algebraic characters of $T$.  By restriction they also define characters $\nu:T^c(\ZZ_p)\to F^\times$, $\kappa:w_{0,M}wT^c(\ZZ_p)(w_{0,M}w)^{-1}\to F^\times$.
 
\begin{prop}\label{prop-special1} Suppose $\kappa\in X^\star(T^c)^{M_\mu,+}$.  Let $K_p = K_{p,m', b}$ for $m'\geq b\geq 0$, $m'>0$.  Over $(\pi_{HT,K_p}^{tor})^{-1}(]C_{w,k}[_{n,n}K_p)$  we have dual morphisms $\mathcal{V}_\kappa \rightarrow \mathcal{V}_{\nu}^{n-an}$ and $\mathcal{D}^{n-an}_{\nu} \rightarrow \mathcal{V}_{\kappa+2\rho_{nc}}^\vee= \mathcal{V}_{-w_{0, M}\kappa-2\rho_{nc}}$.
\end{prop}
\begin{proof} This follows from the construction. Compare with  section \ref{section-locally-algebraic-induction}.
\end{proof}

We would like to have a similar formula for locally algebraic  dominant weights.  We introduce some notation.  Let $m' \geq b \geq 0$.  We let $K_p'=K_{p,m',b}Z_s(\ZZ_p)$.  Then the map $\mathcal{S}_{K^p K_p', \Sigma}^{tor} \rightarrow \mathcal{S}_{K^p K_{p, m',0}, \Sigma}^{tor}$ is an \'etale cover with group $T^c(\ZZ_p)/T^c_b$.  For any character $\chi : T^c(\ZZ_p)/T^c_b \rightarrow F^\times$, we get an invertible sheaf $\oscr_{ \mathcal{S}_{K^p K_{p, m',0}, \Sigma}^{tor}}(\chi)$. For any sheaf of $\oscr_{ \mathcal{S}_{K^p K_{p, m',0}, \Sigma}^{tor}}$-modules $\mathscr{F}$, we let $\mathscr{F}(\chi)$ denote $\mathscr{F} \otimes_{\oscr_{ \mathcal{S}_{K^p K_{p, m',0}, \Sigma}^{tor}}} \oscr_{ \mathcal{S}_{K^p K_{p, m',0}, \Sigma}^{tor}}(\chi)$.

\begin{prop}\label{prop-special2}  Let $K_p = K_{p,m',0}$ with $m'>0$.  Let $\chi : T^c(\ZZ_p)/T^c_b \rightarrow F^\times$ be a character with $b\leq m',n$. Let $\nu_A:T^c(\ZZ_p)\to A^\times$ be an $n$-analytic character. Over $(\pi_{HT,K_p}^{tor})^{-1}(]C_{w,k}[_{n,n}K_p)$ we have $\mathcal{V}^{n-an}_{\nu_A \chi} = \mathcal{V}^{n-an}_{\nu_A}(\chi)$.
\end{prop}

\begin{proof} The map 
$(\pi^{tor}_{HT,K'_p})^{-1} (]C_{w,k}[_{n,n} K'_p)  \rightarrow (\pi^{tor}_{HT,K_p})^{-1} (
  ]C_{w,k}[_{n,n} K_p) $ is an \'etale cover of group $T^c(\ZZ_p)/T^c_b$ since $]C_{w,k}[_{n,n} K'_p= ]C_{w,k}[_{n,n} K_p$.  We also have a map of torsors 
$\mathcal{M}_{HT,n,n,K'_p} \rightarrow \mathcal{M}_{HT,n,n, K_p}$ equivariant for the map: $K'_{p,w,M_{\mu}} \mathcal{M}^1_{\mu,n,n} \rightarrow K_{p,w,M_{\mu}} \mathcal{M}^1_{\mu,n,n}$. We form the quotient: $$K_{p,w,M_{\mu}} \mathcal{M}^1_{\mu,m,n}/ K'_{p,w,M_{\mu}} \mathcal{M}^1_{\mu,m,n} = wT(\ZZ_p)w^{-1}/wT_b w^{-1}.$$ 
By taking the pushout of the map $\mathcal{M}_{HT,m,n,K'_p} \rightarrow \mathcal{M}_{HT,m,n, K_p}$  via $K_{p,w,M_{\mu}} \mathcal{M}^1_{\mu,m,n} \rightarrow wT(\ZZ_p)w^{-1}/wT_b w^{-1}$, we get a map:
$$ (\pi^{tor}_{HT,K'_p})^{-1} ( ]C_{w,k}[_{m,n} K'_p) \rightarrow  \mathcal{M}_{HT,m,n, K_p} \times_{K_{p,w,M_{\mu}} \mathcal{M}^1_{\mu,m,n}} (wT(\ZZ_p)w^{-1}/wT_b w^{-1})$$
  This map is necessarily an isomorphism, because the left hand side is an \'etale cover of $(\pi^{tor}_{HT,K_p})^{-1} (
  ]C_{w,k}[_{m,n} K_p) $ of group $T(\ZZ_p)/T_b$ and the right hand side is an \'etale cover of group $wT(\ZZ_p)w^{-1}/wT_b w^{-1}$. Moreover, the map is equivariant under the isomorphism $T(\ZZ_p)/T_b \rightarrow wT(\ZZ_p)w^{-1}/wT_b w^{-1}$ given by conjugation by $w$. 
It follows that  the \'etale cover $(\pi^{tor}_{HT,K'_p})^{-1} (
  ]C_{w,k}[_{m,n} K'_p)  \rightarrow (\pi^{tor}_{HT,K_p})^{-1} (
  ]C_{w,k}[_{m,n} K_p) $ can be realized as a pushout of the torsor $\mathcal{M}_{HT,n, K_p}$.   
 We deduce from lemma \ref{lem-add-character} that $V^{n-an}_{\kappa_A} \otimes w\chi =  V_{\kappa_A \otimes (w_{0,M} w  \chi^{-1})}$. 
 The lemma follows.
\end{proof} 

\begin{coro}\label{cor-special-localg} Let $\nu=\nu_{alg}\chi$ be a locally algebraic character of $T^c(\ZZ_p)$, with $\nu_{alg}\in X^*(T^c)$ algebraic and $\chi:T^c(\ZZ_p)\to \overline{F}^\times$ a finite order character of conductor $b\leq n$.   Let $K_p=K_{p,m',0}$ for $m'\geq b$, $m'>0$.  If $\kappa_{alg}=-w_{0,M}w\nu_{alg}+(w_{0,M}w\rho-\rho)$ is $M$-dominant, then over $(\pi_{HT,K_p}^{tor})^{-1}(]C_{w,k}[_{n,n}K_p)$ we have maps of sheaves $$\mathcal{V}_{\kappa_{alg}}(\chi)\to\mathcal{V}_\nu^{n-an}\quad\mathrm{and}\quad\mathcal{D}_{\nu}^{n-an}\to\mathcal{V}_{-2\rho_{nc}-\kappa_{alg}}(-\chi).$$
\end{coro}
\begin{proof} This is a combination of propositions \ref{prop-special1} and \ref{prop-special2}.
\end{proof}

\subsubsection{Definition of the Hecke action}\label{subsubsec-hecke-action}

Let $t \in T^+$. Let $K_p = K_{p, m',0}$.    We consider the Hecke correspondence:  
\begin{eqnarray*}
\xymatrix{ & S^{tor}_{K^p (K_p  \cap t K_p t^{-1}), \Sigma''} \ar[rd]^{p_1} \ar[ld]_{p_2} & \\
S^{tor}_{K^pK_p, \Sigma} & & S^{tor}_{K^pK_p, \Sigma'}}
\end{eqnarray*}

Over $p_2^{-1} \big( (\pi^{tor}_{HT, K_p})^{-1}( ]C_{w,k}[_{n,n}K_p)\big) \cap p_1^{-1} \big( (\pi^{tor}_{HT, K_p})^{-1}( ]C_{w,k}[_{n,n}K_p) \big)$ we have a map  
\begin{eqnarray*}
\xymatrix{ p_1^\star \mathcal{M}^{an}_{HT} \ar[r]^{[t^{-1}]} &  p_2^\star \mathcal{M}^{an}_{HT} \\
p_1^\star \mathcal{M}_{HT,n,K_p} \ar[u] &  p_2^\star \mathcal{M}_{HT,n,K_p}\ar[u]}
\end{eqnarray*}
which is represented by $K_{p,w, M_\mu}^c \mathcal{M}_{\mu,n}^c wt^{-1}w^{-1} K_{p,w, M_\mu}^c \mathcal{M}_{\mu,n}^c$ by proposition \ref{prop-representing-torsor-map}.  So far, all this discussion depends only on the double coset $K_p t K_p$ and therefore only on the image of $t$ in $T^+/T(\ZZ_p)$. In particular, if $t \in T(\ZZ_p)$, all the maps are the identity. 

We now consider the following map: $$ \widetilde{[t^{-1}]} :   p_1^\star \mathcal{M}^{an}_{HT} /\mathcal{U}^{an}  \rightarrow   p_2^\star \mathcal{M}^{an}_{HT}/\mathcal{U}^{an}$$ which is given by  $ x  \mathcal{U}^{an}  \mapsto [t^{-1}] x  wtw^{-1}\mathcal{U}^{an}$.  This map depends on $t$ and not only on its class in $T^+/T(\ZZ_p)$. 

\begin{lem}\label{lem-hecke-local} \begin{enumerate}
\item The map $\widetilde{[t^{-1}]}$ restricts to a map $$p_1^\star \mathcal{M}_{HT,n,K_p}/\mathcal{U}_n \rightarrow p_2^\star \mathcal{M}_{HT,n,K_p}/\mathcal{U}_n.$$
\item If $t \in T^{++}$ and $n \geq 1$,   this map factors $$p_1^\star \mathcal{M}_{HT,n,K_p}/\mathcal{U}_n \hookrightarrow p_1^\star \mathcal{M}_{HT,n-1,K_p}/\mathcal{U}_{n-1} \rightarrow p_2^\star \mathcal{M}_{HT,{n},K_p}/\mathcal{U}_n.$$
\item Let $\nu_A$ be an $n$-analytic weight. Then $\widetilde{[t^{-1}]}$ induces a morphism 
$$ \widetilde{[t^{-1}]} :  p_2^\star \mathcal{V}_{\nu_A}^{n-an} \rightarrow p_1^\star \mathcal{V}_{\nu_A}^{n-an}$$
which is locally modeled on the morphism $wtw^{-1} : V_{\kappa_A}^{n-an} \rightarrow V_{\kappa_A}^{n-an}$ defined in section \ref{subsection-analyinduction}.  If $t\in T^{++}$ then this is a compact  morphism (in the sense of definition \ref{defi-compact-map}).
\end{enumerate}
\end{lem}
\begin{proof} Easy and left to the reader.  Observe that $wtw^{-1} \in T^{M_\mu,+}$. 
\end{proof}

Now for $\nu_A$ an $n$-analytic weight, we can now define a map
$$ t : \mathrm{R}(p_1)_\star p_2^\star \mathcal{V}_{\nu_A}^{n-an} \rightarrow \mathcal{V}_{\nu_A}^{n-an}$$ 
as the composite of $ \widetilde{[t^{-1}]} : \mathrm{R}(p_1)_\star p_2^\star \mathcal{V}_{\nu_A}^{n-an}  \rightarrow \mathrm{R}(p_1)_\star p_1^\star \mathcal{V}_{\nu_A}^{n-an}$ and $(-w^{-1} w_{0,M} \rho + \rho)(t) \mathrm{Tr}_{p_1} : \mathrm{R}(p_1)_\star p_1^\star \mathcal{V}_{\nu_A}^{n-an} \rightarrow \mathcal{V}_{\nu_A}^{n-an}$.

\begin{lem}\label{lem-action-by-weight} If $t \in T(\ZZ_p)$, then $$t : \mathrm{R}(p_1)_\star p_2^\star \mathcal{V}_{\nu_A}^{n-an} =  \mathcal{V}_{\nu_A}^{n-an} \rightarrow \mathcal{V}_{\nu_A}^{n-an}$$
acts via scalar multiplication by $\nu_A(t)$.
\end{lem}
\begin{proof} This follows from the identity $\nu_A = w^{-1} w_{0,M} \kappa_A-w^{-1} w_{0,M} \rho + \rho$. The  scalar multiplication by $w^{-1} w_{0,M} \kappa_A(t)$ comes from the map $\widetilde{[t^{-1}]}$ and the multiplication by $(-w^{-1} w_{0,M} \rho + \rho)(t)$ comes from the normalization of the trace map.
\end{proof}

By duality we also obtain a morphism:  $ p_2^\star \mathcal{D}_{\nu_A}^{n-an} \rightarrow p_1^\star \mathcal{D}_{\nu_A}^{n-an}$ which is locally modeled on $(w{t}w^{-1})^{-1} : D_{\kappa_A+2\rho_{nc}}^{n-an} \rightarrow D_{\kappa_A+2\rho_{nc}}^{n-an}$.

\subsection{Locally analytic overconvergent cohomology}\label{section-locally-overconvergent-coho}

Fix $w \in \WM$.  For a choice of $+$ or $-$ and a weight $\nu_A : T^c(\ZZ_p) \rightarrow A^\times$ we want to define a finite slope, overconvergent, locally analytic  cohomology $\mathrm{R}\Gamma_{w, an}(K^p, \nu_A)^{\pm,fs}$ and the cuspidal counterpart $\mathrm{R}\Gamma_{w,an}(K^p, \nu_A, cusp)^{\pm, fs}$ by taking cohomologies of the analytic sheaves $\mathcal{V}_{\nu_A}^{n-an}$ or $\mathcal{D}_{\nu_A}^{n-an}$ (for $n$ large enough),  with suitable support conditions of neighborhoods on the inverse image of $\mathcal{P_\mu}\backslash\mathcal{P_\mu}wK_p$ by the Hodge-Tate period map, and passing to finite slope parts for a suitable Hecke operator.

\subsubsection{First definition}\label{section-first-defi2}

Let $K_p=K_{p,m',0}$ with $m'>n$ and let $\nu_A : T^c(\ZZ_p) \rightarrow A^\times$ be an $n$-analytic weight. We fix $t \in T^{++}$ and we assume that $\min(t) = \inf_{\alpha \in \Phi^+} v(\alpha(t)) \geq 1$ in order to simplify notations. We let $T$ be the associated Hecke operator. 

We have that $$T^{n+1}  ((\pi_{HT, K_p}^{tor})^{-1}(]X_{w,k}[) ) \cap (T^t)^{n+1}((\pi_{HT, K_p}^{tor})^{-1}(\overline{]Y_{w,k}[}))  \subseteq  ( \pi_{HT,K_p}^{tor})^{-1}(]C_{w,k}[_{n+1,\overline{n+1}} K_p) $$ $$ \subseteq ( \pi_{HT,K_p}^{tor})^{-1}(]C_{w,k}[_{n,n} K_p)$$ 
by lemma \ref{lem-dynamic-corres}, and hence the sheaves $\mathcal{V}_{\nu_A}^{n-an}$ and $\mathcal{D}_{\nu_A}^{n-an}$ are defined over a neighborhood of $T^{n+1}  ((\pi_{HT, K_p}^{tor})^{-1}(]X_{w,k}[) ) \cap (T^t)^{n+1}((\pi_{HT, K_p}^{tor})^{-1}(\overline{]Y_{w,k}[})) $. 

We also have that $$T^{n+1}  ((\pi_{HT, K_p}^{tor})^{-1}(\overline{]X_{w,k}[} )) \cap (T^t)^{n+1}((\pi_{HT, K_p}^{tor})^{-1}(]Y_{w,k}[))  \subseteq  ( \pi_{HT,K_p}^{tor})^{-1}(]C_{w,k}[_{\overline{n+1},{n+1}} K_p) $$ $$ \subseteq ( \pi_{HT,K_p}^{tor})^{-1}(]C_{w,k}[_{n,n} K_p)$$ 
and therefore, the sheaves $\mathcal{V}_{\nu_A}^{n-an}$ and $\mathcal{D}_{\nu_A}^{n-an}$ are defined over a neighborhood of $T^{n+1}  ((\pi_{HT, K_p}^{tor})^{-1}(\overline{]X_{w,k}[} )) \cap (T^t)^{n+1}((\pi_{HT, K_p}^{tor})^{-1}(]Y_{w,k}[))$ as well.

We define:
$$\mathrm{R}\Gamma_{w,n-an}(K^pK_p, \nu_A)^{+,fs}  : = $$
$$  \mathrm{R}\Gamma_{ T^{n+1}  ((\pi_{HT, K_p}^{tor})^{-1}(]X_{w,k}[) ) \cap (T^t)^{n+1}((\pi_{HT, K_p}^{tor})^{-1}(\overline{]Y_{w,k}[}))}(T^{n+1}((\pi_{HT, K_p}^{tor})^{-1}(]X_{w,k}[)), \mathcal{V}^{n-an}_{\nu_A})^{+,fs}.$$
Implicit in this definition is that it makes sense to take the finite slope part: namely the cohomology is an object of $\mathrm{Pro}_{\N}(\mathcal{K}^{proj}(\mathbf{Ban}(A)))$ and that $\mathcal{H}_{p,m',0}^+$ acts on it in a way that $\mathcal{H}_{p,m',0}^{++}$ acts by potent compact operators.  This is proved in Theorem \ref{thm-finiteslopecoho2} below.  Note that $\mathrm{R}\Gamma_{w,n-an}(K^pK_p, \nu_A)^{+,fs}$ is an object of the derived category of sheaves of $\oscr_{\Spa(A,A^+)}$-modules, which possibly depends on several choices (see sections \ref{relative-spectral-1}, \ref{relative-spectral-2}, \ref{relative-spectral-3}), but the cohomologies $H^k_{w,n-an}(K^pK_p,\nu_A)^{+,fs}$ are well defined.

Similarly, we define:
$$\mathrm{R}\Gamma_{w, n-an}(K^pK_p, \nu_A)^{-,fs}  : =  $$
$$  \mathrm{R}\Gamma_{ T^{n+1}  ((\pi_{HT, K_p}^{tor})^{-1}(\overline{]X_{w,k}[})) \cap (T^t)^{n+1}((\pi_{HT, K_p}^{tor})^{-1}(]Y_{w,k}[))}((T^{t})^{n+1}((\pi_{HT, K_p}^{tor})^{-1}(]Y_{w,k}[)), \mathcal{D}^{n-an}_{\nu_A})^{-,fs}.$$ Again implicit in this definition is that it makes sense to take the finite slope part: namely the cohomology is an object of $\mathrm{Pro}_{\N}(\mathcal{K}^{proj}(\mathbf{Ban}(A)))$ and that $\mathcal{H}_{p,m',0}^-$ acts on it in a way that $\mathcal{H}_{p,m',0}^{--}$ acts by potent compact operators.  This is also proved in Theorem \ref{thm-finiteslopecoho2} below.

We have similar definitions for cuspidal cohomologies $\mathrm{R}\Gamma_{w,n-an}(K^pK_p, \nu_A,cusp)^{+,fs}$ and $\mathrm{R}\Gamma_{w, n-an}(K^pK_p, \nu_A,cusp)^{-,fs}$

\subsubsection{Existence of finite slope cohomology} We now justify that the cohomologies introduced in the previous section are well defined. 

\begin{thm} \label{thm-finiteslopecoho2}  Let $K_p=K_{p,m',0}$ for some $m'>n$,  $w\in\WM$, and $\nu_A: T^c(\ZZ_p) \rightarrow A^\times$ an $n$-analytic character.
\begin{enumerate}
\item The cohomologies  $$  \mathrm{R}\Gamma_{ T^{n+1}  ((\pi_{HT, K_p}^{tor})^{-1}(]X_{w,k}[) ) \cap (T^t)^{n+1}((\pi_{HT, K_p}^{tor})^{-1}(\overline{]Y_{w,k}[}))}(T^{n+1}((\pi_{HT, K_p}^{tor})^{-1}(]X_{w,k}[)), \mathcal{V}^{n-an}_{\nu_A}),$$ $$  \mathrm{R}\Gamma_{ T^{n+1}  ((\pi_{HT, K_p}^{tor})^{-1}(\overline{]X_{w,k}[}) ) \cap (T^t)^{n+1}((\pi_{HT, K_p}^{tor})^{-1}(]Y_{w,k}[))}((T^{t})^{n+1}((\pi_{HT, K_p}^{tor})^{-1}(]Y_{w,k}[)), \mathcal{D}^{n-an}_{\nu_A})$$  are objects of $\mathrm{Pro}_{\N}(\mathcal{K}^{proj}(\mathbf{Ban}(A)))$. 
\item  There is an action of $\mathcal{H}_{p,m',0}^+$ on $$  \mathrm{R}\Gamma_{ T^{n+1}  ((\pi_{HT, K_p}^{tor})^{-1}(]X_{w,k}[ )) \cap (T^t)^{n+1}((\pi_{HT, K_p}^{tor})^{-1}(\overline{]Y_{w,k}[}))}(T^{n+1}((\pi_{HT, K_p}^{tor})^{-1}(]X_{w,k}[)), \mathcal{V}^{n-an}_{\nu_A})$$ for which $\mathcal{H}_{p,m',0}^{++}$ acts via compact operators.   
\item  There is an action of $\mathcal{H}_{p,m',0}^-$ on $$  \mathrm{R}\Gamma_{ T^{n+1}  ((\pi_{HT, K_p}^{tor})^{-1}(\overline{]X_{w,k}[}) ) \cap (T^t)^{n+1}((\pi_{HT, K_p}^{tor})^{-1}(]Y_{w,k}[))}((T^{t})^{n+1}((\pi_{HT, K_p}^{tor})^{-1}(]Y_{w,k}[)), \mathcal{D}^{n-an}_{\nu_A})$$ for which $\mathcal{H}_{p,m',0}^{--}$ acts via compact operators.  
\item All the analogous statements hold for cuspidal cohomology. 
 \end{enumerate}
\end{thm}

\begin{proof} The property that the objects at hand are objects of $\mathrm{Pro}_{\N}(\mathcal{K}^{proj}(\mathbf{Ban}(A)))$ is corollary \ref{coro-inpro}. The rest of the argument follows almost verbatim the proof of theorem \ref{thm-finiteslopecoho}, now using lemmas \ref{last-lemma-compact} and \ref{lem-hecke-local} to see that the operators in $\mathcal{H}_{p,m',0}^{\pm\pm}$ are potent compact. The details  are left  to the reader.
\end{proof}
 
\subsubsection{Change of analyticity radius}

\begin{thm}\label{thm-change-of-radius}  Let $m'>n+1$. Let $\nu_A$ be an $n$-analytic character.  The maps 
$\HH^i_{w,n+1-an}(K^pK_p, \nu_A)^{+,fs} \to \HH^i_{w,n-an}(K^pK_p, \nu_A)^{+,fs}$ and $\HH^i_{w,n-an}(K^pK_p, \nu_A)^{-,fs} \to \HH^i_{w,n+1-an}(K^pK_p, \nu_A)^{-,fs}$
are quasi-isomorphisms. The same results hold for cuspidal cohomologies. 
\end{thm}
\begin{proof} Let $T = [K_{p,m',0} t K_{p,m',0}]$ for $t \in T^{++}$ satisfying $\min(t) \geq 1$.  The endomorphism  $T $ of $$ \mathrm{R}\Gamma_{ T^{n+1}  ((\pi_{HT, K_p}^{tor})^{-1}(]X_{w,k}[) ) \cap (T^t)^{n+1}((\pi_{HT, K_p}^{tor})^{-1}(\overline{]Y_{w,k}[}))}(T^{n+1}((\pi_{HT, K_p}^{tor})^{-1}(]X_{w,k}[)), \mathcal{V}^{n-an}_{\nu_A}) $$
factors into 
$$  \mathrm{R}\Gamma_{ T^{n+1}  ((\pi_{HT, K_p}^{tor})^{-1}(]X_{w,k}[) ) \cap (T^t)^{n+1}((\pi_{HT, K_p}^{tor})^{-1}(\overline{]Y_{w,k}[}))}(T^{n+1}((\pi_{HT, K_p}^{tor})^{-1}(]X_{w,k}[)), \mathcal{V}^{n-an}_{\nu_A})$$
$$\longrightarrow $$
$$ \mathrm{R}\Gamma_{ T^{n+1}  ((\pi_{HT, K_p}^{tor})^{-1}(]X_{w,k}[) ) \cap (T^t)^{n+1}((\pi_{HT, K_p}^{tor})^{-1}(\overline{]Y_{w,k}[}))}(T^{n+1}((\pi_{HT, K_p}^{tor})^{-1}(]X_{w,k}[)), \mathcal{V}^{n+1-an}_{\nu_A}) $$ $$\longrightarrow$$
$$ \mathrm{R}\Gamma_{ T^{n+1}  ((\pi_{HT, K_p}^{tor})^{-1}(]X_{w,k}[) ) \cap (T^t)^{n+1}((\pi_{HT, K_p}^{tor})^{-1}(\overline{]Y_{w,k}[}))}(T^{n+1}((\pi_{HT, K_p}^{tor})^{-1}(]X_{w,k}[)), \mathcal{V}^{n-an}_{\nu_A}).$$
The $-$ case follows similarly.
\end{proof}

\subsubsection{Change of support condition} It is important to us that the cohomology $\mathrm{R}\Gamma_{w,n-an}(K^pK_p, \nu_A)^{\pm,fs}$ can actually be realized as the finite slope part of cohomology groups with different support conditions.  The following definition is similar to definition \ref{defi-support-condition}. We start by fixing an element $t \in T^{++}$ such that $\min(t) \geq 1$ and we set $C = \max(t)$. 

\begin{defi}\label{defi-support-condition2} Let $m'>n$.  A $(+,w, K_{p,m',0}, n-an)$-allowed support is a pair $(\mathcal{U}, \mathcal{Z})$  where:

\begin{enumerate}
\item $\mathcal{U}$ is an open subset of  $\mathcal{S}_{K^pK_{p,m',0}, \Sigma}^{tor}$ which is a finite union of quasi-Stein open subsets. 
\item $\mathcal{Z}$ is a closed subset of  $\mathcal{S}_{K^pK_{p,m',0}, \Sigma}^{tor}$ whose complement is a finite union of quasi-Stein open subsets.
\item   There exists $m,l,s \in \ZZ_{\geq 0}$  such that:
 $$(\pi_{HT, K_{p, m',0}}^{tor})^{-1}(]C_{w,k}[_{m, \overline{0}} K_{p,m',0} \cap ]C_{w,k}[_{0, \overline{l +s}} K_{p,m',0}) \subseteq \mathcal{Z} \cap \mathcal{U}
  \subseteq (\pi_{HT, K_{p, m',0}}^{tor})^{-1}(]C_{w,k}[_{{0}, C\overline{l}}K_{p,m',0}),$$
 $$(\pi_{HT, K_{p, m',0}}^{tor})^{-1}(]C_{w,k}[_{m+s, \overline{0}} K_{p,m',0} \cap ]C_{w,k}[_{0, \overline{l}} K_{p,m',0}) \subseteq \mathcal{U} \subseteq (\pi_{HT, K_{p, m',0}}^{tor})^{-1}(]C_{w,k}[_{Cm,-1} K_{p,m',0}),$$
 $$ \mathcal{Z} \cap \mathcal{U} \subseteq  (\pi_{HT,K_{p,m',0}}^{tor})^{-1}(]C_{w,k}[_{n,n} K_{p,m',0}).$$
 \end{enumerate}

Let $m'>n$.  A $(-,w, K_{p,m',0},n-an)$-allowed support is a pair $(\mathcal{U}, \mathcal{Z})$  where:
\begin{enumerate}
\item $\mathcal{U}$ is an open subset of  $\mathcal{S}_{K^pK_{p,m',0}, \Sigma}^{tor}$ which is a finite union of quasi-Stein open subsets. 
\item $\mathcal{Z}$ is a closed subset of  $\mathcal{S}_{K^pK_{p,m',0}, \Sigma}^{tor}$ whose complement is a finite union of quasi-Stein open subsets.
\item There exists $m,l,s \in \ZZ_{\geq 0}$ such that:
$$(\pi_{HT, K_{p,m',0}}^{tor})^{-1}(]C_{w,k}[_{\overline{m+s}, {0} } K_{p,m',0} \cap ]C_{w,k}[_{\overline{0}, n} K_{p,m',0}) \subseteq \mathcal{Z}  \cap \mathcal{U}\subseteq (\pi_{HT, K_{p,m',0}}^{tor})^{-1}( ]C_{w,k}[_{C\overline{m}, {0}} K_{p,m',0}),$$
 $$(\pi_{HT, K_{p,m',0}}^{tor})^{-1}(]C_{w,k}[_{\overline{m}, 0} K_{p,m',0} \cap ]C_{w,k}[_{\overline{0}, l+s}  K_{p,m',0} ) \subseteq \mathcal{U} \subseteq (\pi_{HT, K_{p,m',0}}^{tor})^{-1}(]C_{w,k}[_{-1, Cl} K_{p,m',0}),$$
  $$ \mathcal{Z} \cap \mathcal{U} \subseteq  (\pi_{HT,K_{p,m',0}}^{tor})^{-1}(]C_{w,k}[_{n,n} K_{p,m',0}).$$
 \end{enumerate}
\end{defi}
 
\begin{thm}\label{thm-finiteslopesupport2} Let $m'>n$, $w\in\WM$, and $\nu_A $ an $n$-analytic character.

\begin{enumerate} 
\item  Let $(\mathcal{U}, \mathcal{Z})$ be a $(+,w,K_{p,m',0},n-an)$-allowed support condition.  Then $\mathrm{R}\Gamma_{\mathcal{Z} \cap \mathcal{U}}(\mathcal{U}, \mathcal{V}^{n-an}_{\nu_A})$ and $\mathrm{R}\Gamma_{\mathcal{Z} \cap \mathcal{U}}(\mathcal{U}, \mathcal{V}^{n-an}_{\nu_A}(-D))$ are objects of $\mathrm{Pro}_{\N}(\mathcal{K}^{proj}(\mathbf{Ban}(A)))$ and carry a   compact action of $T^s$ for sufficiently large $s$.  Moreover there are canonical isomorphisms
$$\HH^i_{w, n-an}(K^pK_{p,m',0},\nu_A)^{+,fs}\simeq \HH^i_{\mathcal{Z} \cap \mathcal{U}}(\mathcal{U}, \mathcal{V}^{n-an}_{\nu_A})^{T^s-fs}$$
and
$$\HH^i_{w, n-an}(K^pK_{p,m',0},\nu_A,cusp)^{+,fs}\simeq \HH^i_{\mathcal{Z} \cap \mathcal{U}}(\mathcal{U}, \mathcal{V}^{n-an}_{\nu_A}(-D))^{T^s-fs}.$$

\item  Let $(\mathcal{U}, \mathcal{Z})$ be a $(-,w,K_{p,m',0}, n-an)$-allowed support condition.  Then $\mathrm{R}\Gamma_{\mathcal{Z} \cap \mathcal{U}}(\mathcal{U}, \mathcal{D}^{n-an}_{\nu_A})$ and $\mathrm{R}\Gamma_{\mathcal{Z} \cap \mathcal{U}}(\mathcal{U}, \mathcal{D}^{n-an}_{\nu_A}(-D))$ are objects of $\mathrm{Pro}_{\N}(\mathcal{K}^{proj}(\mathbf{Ban}(A)))$ and carry a  compact action of $T^s$ for sufficiently large $s$.  Moreover there are canonical isomorphisms
$$\HH^i_{w, n-an}(K^pK_{p,m',0},\nu_A)^{-,fs}\simeq \HH^i_{\mathcal{Z} \cap \mathcal{U}}(\mathcal{U}, \mathcal{D}^{n-an}_{\nu_A})^{T^s-fs}$$
and
$$\HH^i_{w, n-an}(K^pK_{p,m',0},\nu_A,cusp)^{-,fs}\simeq \HH^i_{\mathcal{Z} \cap \mathcal{U}}(\mathcal{U}, \mathcal{D}^{n-an}_{\nu_A}(-D))^{T^s-fs}.$$
 \end{enumerate}
\end{thm}
\begin{proof} This is very similar to the proof of theorem \ref{thm-finiteslopesupport}, and left to the reader.
\end{proof}
 
\subsubsection{Change of level}Now we investigate how the finite slope cohomologies $\mathrm{R}\Gamma_{w, n-an}(K^pK_p,\nu_A)^{\pm,fs}$ and $\mathrm{R}\Gamma_{w, n-an}(K^pK_p,\nu_A,cusp)^{\pm,fs}$ vary with the level $K_p$.
\begin{thm}\label{thm-change-of-level2}
For all $w\in\WM$ and all $m'' \geq m'>n$, the pullback map 
$$ \HH^i_{w, n-an}(K^pK_{p,m',0},\nu_A)^{+, fs}  \rightarrow  \HH^i_{w, n-an}(K^pK_{p,m'',0},\nu_A)^{+, fs}$$ 
and the trace map
$$ \HH^i_{w, n-an}(K^pK_{p,m'',0},\nu_A)^{-, fs}  \rightarrow  \HH^i_{w, n-an}(K^pK_{p,m',0},\nu_A)^{-, fs}$$ are quasi-isomorphisms, compatible with the action of $T(\qq_p)$, and the same statements are true for cuspidal cohomology.
\end{thm}

\begin{proof} This is very similar to the proof of theorem \ref{thm-change-of-level}. Details are left to the reader.
\end{proof}

As a result of the theorems \ref{thm-change-of-radius} and \ref{thm-change-of-level2}, we can let $$\mathrm{R}\Gamma_{w,an}(K^p,\nu_A)^{\pm,fs}~\textrm{and}~\mathrm{R}\Gamma_{w,an}(K^p,\nu_A,cusp)^{\pm,fs}$$ denote respectively $$\mathrm{R}\Gamma_{w, n-an}(K^pK_{p,m',0},\nu_A)^{\pm,fs} ~\textrm{and}~\mathrm{R}\Gamma_{w,n-an}(K^pK_{p,m',0},\nu_A,cusp)^{\pm,fs}$$ for  some choice of $n$ and  $m'>n$.  Although the cohomology complexes depend on a number of choices, the cohomology groups are independent of any choice. 

\subsection{Two spectral sequences} We consider two spectral sequences. The first spectral sequence is a Tor spectral sequence for specialization in the weight space. The second spectral sequence comes from the locally analytic BGG resolution and relates locally analytic overconvergent and overconvergent cohomology in classical weights. 
\subsubsection{The Tor spectral sequence}

\begin{thm}\label{thm-tor-spectral} Let $w\in\WM$, let $\nu_{A}$ be a weight, and let $x \in \Spa (A, A^+) = S$. Let $\nu_x$ be the corresponding weight. We have spectral sequences:
$$ E_2^{p,q} = \mathrm{Tor}^{\oscr_S}_{-p}( \HH^q_{w,an}(K^p, \nu_A)^{\pm,fs}, k(x)) \Rightarrow \HH^{p+q}_{w,an}(K_p, \nu_x)^{\pm,fs}$$
and similarly for cuspidal cohomology.
\end{thm}

\begin{rem} Concretely, if we are interested in computing the slope $\leq h$ part of $\mathrm{R}\Gamma_{w,an}(K_p, \nu_x)^{\pm,fs}$ for a given compact operator $t \in T^{\pm}$, we first can find $h' > h$ and $\Spa (B, B^+) \hookrightarrow S$ an affinoid open containing $x$ such that $\HH^i_{w,an}(K_p, \nu_x)^{\pm,fs}\vert_{\Spa(B,B^+)}$ have slope $\leq h'$ decomposition (actually, we can suppose that $x$ is  a maximal point and we can take $h=h'$). We can thus consider the finite $B$-module $\HH^i_{w,an}(K_p, \nu_x)^{\pm, \leq h'}( \Spa(B,B^+)) $. The spectral sequence specializes  to $$\mathrm{Tor}^{B}_{-p}( \HH^q_{w,an}(K^p, \nu_A)^{\pm,\leq h'}(\Spa(B,B^+)), k(x)) \Rightarrow \HH^{p+q}_{w,an}(K_p, \nu_x)^{\pm,\leq h'}.$$
\end{rem}

\begin{proof} Following the above remark, it suffices to show that for any $h$, there exists $B$ and $h'$ and a spectral sequence $$\mathrm{Tor}^{B}_{-p}( \HH^q_{w,an}(K^p, \nu_A)^{\pm,\leq h'}(\Spa(B,B^+)), k(x)) \Rightarrow \HH^{p+q}_{w,an}(K_p, \nu_x)^{\pm,\leq h'}.$$  
But by construction (see remark \ref{rem-useful-perfect}),  we can find a  complex  $M^\bullet$ of $B$-modules which computes $\mathrm{R}\Gamma_{w,an}(K^p, \nu_B)^{\pm,\leq h'}$ and such that $M^\bullet \otimes_B k(x)$ computes $\mathrm{R}\Gamma_{w,an}(K^p, \nu_x)^{\pm,\leq h'}$. We can therefore apply \cite{stacks-project}, example TAG 061Z. 
\end{proof}

\subsubsection{The spectral sequence from locally analytic overconvergent to overconvergent cohomology}

Let $w\in\WM$ and let $\nu=\nu_{alg}\chi:T^c(\ZZ_p) \rightarrow F^\times$ be a locally algebraic weight so that $\kappa_{alg}=-w_{0,M}w\nu_{alg}-(w_{0,M}w\rho+\rho)$ is $M$-dominant.

By corollary \ref{cor-special-localg}, we have morphisms:
$$ \mathrm{R}\Gamma_w( K^p, \kappa_{alg}, \chi )^{+,fs} \rightarrow \mathrm{R}\Gamma_{w,an}(K^p,\nu)^{+,fs}$$
$$\mathrm{R}\Gamma_{w,an}(K^p,\nu)^{-,fs} \rightarrow \mathrm{R}\Gamma_w( K^p, -w_{0,M}\kappa_{alg}-2\rho_{nc}, \chi^{-1} )^{-,fs}$$
and similarly for cuspidal cohomology.   We remark that these morphisms are not Hecke equivariant. Namely, on $ \mathrm{R}\Gamma_{w,an}( K^p, \nu )^{+,fs}$ we have an action of $T^+$ with the property that $T(\ZZ_p)$ acts via $\nu$ and on $\mathrm{R}\Gamma_{w,an}(K^p,\nu)^{-,fs}$, we have an action of $T^{-}$ such that $T(\ZZ_p)$ acts by $-\nu$ (see lemma \ref{lem-action-by-weight}). We deduce (compare with section \ref{section-locally-algebraic-induction}) that there is a $T^+$-equivariant map:
$\mathrm{R}\Gamma_w( K^p, \kappa_{alg}, \chi )^{+,fs} \rightarrow \mathrm{R}\Gamma_{w,an}(K^p,\nu)^{+,fs}(-\nu_{alg})$ and a $T^-$-equivariant map: 
$\mathrm{R}\Gamma_{w,an}(K^p,\nu)^{-,fs}(\nu_{alg}) \rightarrow \mathrm{R}\Gamma_w( K^p, -w_{0,M}\kappa_{alg}-2\rho_{nc}, \chi^{-1} )^{-,fs}$

We can study these maps with the help of the locally analytic BGG resolution.

\begin{thm}\label{spectral-sequence-lan-to-classical} In the setting above, there is a $\mathcal{H}_{p,m,0}^{+}$-equivariant spectral sequence  $\mathbf{E}^{p,q}_w(K^p, \kappa,\chi)^+$ converging to  finite slope overconvergent cohomology $ \HH^{p+q}_w(K^p,\kappa_{alg},\chi)^{+,fs}$,
such that $$\mathbf{E}_{w,1}^{p,q}(K^p, \kappa_{alg},\chi)^{+} = $$ $$ \oplus_{v \in W_M, \ell(v) = p} \HH^{q}_{w,an}\big( K^p,(((w_{0,M}w)^{-1}vw_{0,M}w)\cdot\nu_{alg})\chi\big)^{+,fs} \big( -(((w_{0,M}w)^{-1}vw_{0,M}w)\cdot\nu_{alg})\big).$$

There is a $\mathcal{H}_{p,m,0}^{-}$-equivariant spectral sequence  $\mathbf{E}^{p,q}_w(K^p, \kappa,\chi)^-$ converging to  finite slope overconvergent cohomology $ \HH^{p+q}_w(K^p,-w_{0,M}\kappa_{alg}-2\rho_{nc},\chi^{-1})^{-,fs}$,
such that $$\mathbf{E}_{w,1}^{p,q}(K^p, \kappa_{alg},\chi)^{-} = $$ $$\oplus_{v \in W_M, \ell(v) = -p} \HH^{q}_{w,an}\big( K^p,  (((w_{0,M}w)^{-1}vw_{0,M}w)\cdot\nu_{alg})\chi\big)^{-,fs}\big((((w_{0,M}w)^{-1}vw_{0,M}w)\cdot\nu_{alg})\big).$$ 
\end{thm}
\begin{proof} This is the spectral sequence associated to the BGG sequence of theorem \ref{thm-BGG-resolution}, as well as its variant for distribution sheaves. We use the following identities: $\nu_{alg} + \rho = -w^{-1} w_{0,M} ( \kappa_{alg} + \rho)$ and $-w^{-1} w_{0,M} (v.\cdot \kappa_{alg} + \rho)- \rho = ((w_{0,M}w)^{-1}vw_{0,M}w)\cdot\nu_{alg}$. 
\end{proof}

\subsection{Cohomological vanishing} 
The following is the analogue of  theorem \ref{theorem-coho-van}.

\begin{thm}\label{thm-one-half-conj} For all $w \in \WM$ and weights $\nu_A$, the cohomology  $\mathrm{R}\Gamma_{w,an}(K^p, \nu_A,cusp)^{\pm,fs}$ is concentrated in degree $[0, \ell_{\pm}(w)]$.
\end{thm}
\begin{proof} This is exactly the same as the proof of theorem \ref{theorem-coho-van}, granting lemmas \ref{lemma-local-vanishing} and \ref{lemma-local-vanishing1} below.
\end{proof}

The key  lemma to prove the theorem is the following (this lemma and its proof are inspired by \cite{MR3275848}): 

\begin{lem}\label{lemma-local-vanishing} Assume that $(G,X)$ is a Hodge type Shimura datum. Suppose the weight $\nu_A$ is $n$-analytic, and let $K_p \subseteq  K_{p,m',0}$ for some $m'>n$. For all affinoid opens $\mathcal{U} \subseteq ]C_{w,k}[_{n,n}K_p$, $\mathrm{R}\Gamma( (\pi_{HT,K_p}^{tor})^{-1}(\mathcal{U}), \mathcal{V}_{\nu_A}^{n-an}(-D))$ and $\mathrm{R}\Gamma((\pi_{HT,K_p}^{tor})^{-1}(\mathcal{U}), \mathcal{D}_{\nu_A}^{n-an}(-D))$ are concentrated in degree $0$. 
\end{lem}
\begin{proof} We only consider the case of the sheaf $\mathcal{V}_{\nu_A}^{n-an}(-D)$, the other case is identical.   For any compact open $K_p' \subseteq K_p$, let $\mathfrak{S}^{tor}_{K^pK_p', \Sigma}$ be an integral toroidal compactification, and let $\mathfrak{S}^{\star}_{K^pK_p'}$ be an integral minimal compactification (whose existence is given by the main results of  \cite{MR3948111}, see also section \ref{subsec-integral-hodge}). Let $\mathfrak{U}$ be a formal model for $\mathcal{U}$, realized as an open subset of the normalization of a blow-up of $\mathfrak{FL}_{G,\mu}$.  By theorem \ref{thm-description-toro-compactificiation2},
for $K_p' \subseteq K_p$ small enough, we have  normal formal models $\mathfrak{S}^{\star,mod}_{K^pK_p', \mathfrak{U}} \rightarrow \mathfrak{S}^{\star}_{K^pK_p'}$ for the map $\pi_{HT,K_p'}^{-1}(\mathcal{U}) \rightarrow \mathcal{S}^{\star}_{K^pK_p'}$ and $\mathfrak{S}^{tor,mod}_{K^pK_p', \Sigma, \mathfrak{U}} \rightarrow \mathfrak{S}^{tor}_{K^pK_p',\Sigma}$ for the map $(\pi^{tor}_{HT,K_p'})^{-1}(\mathcal{U}) \rightarrow \mathcal{S}^{tor}_{K^pK_p',\Sigma}$ and we have a map $f_{K_p'} : \mathfrak{S}^{tor,mod}_{K^pK_p', \Sigma, \mathfrak{U}} \rightarrow \mathfrak{S}^{\star,mod}_{K^pK_p', \mathfrak{U}}$. 

 We may actually assume that $K_p'$ is a subgroup of the pro $p$-group $K_{p,m',1}$. 
By proposition \ref{prop-torsor-become-trivial} for $K_p'' \subseteq K_p'$ small enough, the torsor $\mathcal{M}_{HT,n,K_{p,m',1}}$
  is  trivial  on the generic fiber of a Zariski cover  of $\mathfrak{S}^{tor,mod}_{K^pK_p'', \Sigma, \mathfrak{U}}$. 

We can construct a formal model for the Banach sheaf $\mathcal{V}_{\nu_A}^{n-an}$ that we denote $\mathfrak{V}_{\nu_A}^{n-an,+}$ over $\mathfrak{S}^{tor,mod}_{K^pK_p'', \Sigma, \mathfrak{U}} $. Indeed let $\mathfrak{S}^{tor,mod}_{K^pK_p'', \Sigma, \mathfrak{U}} = \cup_i \mathfrak{V}_{K_p'',i}$ be a Zariski covering with the property that the torsor $\mathcal{M}_{HT,n, K_{p,m',1}}$ is trivial over the generic fiber $\mathcal{V}_{K_p'',i}$ of $\mathfrak{V}_{K_p'',i}$. We fix such a trivialization. We can construct the associated $1$-cocycle: over each intersection $\mathfrak{V}_{K_p'', i,j} = \mathfrak{V}_{K_p'',i} \cap \mathfrak{V}_{K_p'',j}$ we have an element $m_{i,j} \in K_{p,m',1,w,M_\mu}^c\mathcal{M}_{\mu,n}^c$ describing the change of trivialization of the torsor $\mathcal{M}_{HT,n, K_{p,m',1}}$.  Over $\mathfrak{V}_{K_p'',i}$ we let $\mathfrak{V}_{\nu_A}^{n-an,+} = V_{\kappa_A}^{n-an,+} \hat{\otimes}_{\ocal_F} \oscr_{\mathfrak{V}_{K_p'',i}}$ and we glue these sheaves using multiplication by $m_{i,j}$ on each intersection.  The sheaf $\mathfrak{V}_{\nu_A}^{n-an,+}$ is a flat formal Banach sheaf (see section \ref{section-Cohomological properties of Banach sheaves}).  We claim that it is also small (see again section \ref{section-Cohomological properties of Banach sheaves}). Indeed,  it follows from lemma \ref{lem-inductive-lim-rep} that the representation $V_{\kappa_A}^{n-an,+} \otimes_{A^+} A^+/A^{++}$ of $K_{p,w,M_\mu}^c\mathcal{M}_{\mu,n}^c$ is an inductive limit of finite free $A^+/A^{++}$-submodules $\mathrm{colim}_i V_i$, stable under  $K_{p,w,M_\mu}^c\mathcal{M}_{\mu,n}^c$,  and such that over $V_i/V_{i-1}$ the action  factors over the quotient $K_{p,w,M_\mu}^{c,p}$ of  $K_{p,w,M_\mu}^c$ by its maximal normal pro-$p$ subgroup.  We find that $K_{p,w,M_\mu}^{c,p} = wT^c(\ZZ_p)w^{-1} /wT_1^cw^{-1}$ is a finite abelian group.  Moreover,  the elements $m_{i,j}$  map trivially to $K_{p,w,M_\mu}^{c,p}$. 

It follows  that the sheaf $\mathfrak{V}_{\nu_A}^{n-an,+}\otimes_{A^+}A^+/A^{++} = \mathrm{colim}_i \mathscr{F}_i$ is an inductive limit of locally free sheaves  of $A^+/A^{++} \otimes \oscr_{\mathfrak{S}^{tor, mod}_{K^pK_p'', \Sigma, \mathfrak{U}}}$-modules  with the property that $\mathscr{F}_i/\mathscr{F}_{i-1}$ is the trivial sheaf.  We now fix $\mathcal{L}$ an ample line bundle over $\mathfrak{S}^{\star, mod}_{K^pK_p'', \mathfrak{U}} $ (this last formal scheme is quasi-projective, so $\mathcal{L}$ exists). 

We claim that there exists $m\geq 0$ such that $\mathrm{R}\Gamma( \mathfrak{S}^{tor,mod}_{K^pK_p'', \Sigma, \mathfrak{U}} , f_{K_p''}^\star \mathcal{L}^m \otimes \mathfrak{V}_{\nu_A}^{n-an,+}(-D_{K_p''}))$ (where $-D_{K''_p}$ is the  boundary divisor at level $K''_p$) is concentrated in degree $0$. 
 The cohomology is represented by a complex of $A^+$-modules which are completions of free $A^+$-modules (take the \v{C}ech complex associated with  the covering $\cup \mathfrak{V}_{K_p'',i}$). It follows from lemma \ref{lem-handy} below that it suffices to prove that there exists $m \geq 0$ such that 
 $\mathrm{R}\Gamma( \mathfrak{S}^{tor,mod}_{K^pK_p'', \Sigma, \mathfrak{U}} , f_{K_p''}^\star\mathcal{L}^m  \otimes \mathfrak{V}_{\nu_A}^{n-an,+}/A^{++}(-D_{K_p}))$ is concentrated in degree $0$. We are therefore reduced to prove that there exists $m$ such that   $\mathrm{R}\Gamma( \mathfrak{S}^{tor,mod}_{K^pK_p'', \Sigma, \mathfrak{U}} , f_{K_p''}^\star \mathcal{L}^m \otimes \mathscr{F}_i/\mathscr{F}_{i-1}(-D_{K''_p}))$ is concentrated in degree $0$.  By the projection formula, $$\mathrm{R}\Gamma( \mathfrak{S}^{tor,mod}_{K^pK_p'', \Sigma, \mathfrak{U}} , f_{K_p''}^\star \mathcal{L}^m \otimes \mathscr{F}_i/\mathscr{F}_{i-1}(-D_{K''_p})) = \mathrm{R}\Gamma( \mathfrak{S}^{\star,mod}_{K^pK_p'', \Sigma, \mathfrak{U}} , \mathcal{L}^m \otimes \mathrm{R} (f_{K_p''})_\star \mathscr{F}_i/\mathscr{F}_{i-1}(-D_{K''_p})).$$
  We first observe that   $\mathrm{R}(f_{K_p''})_\star \mathscr{F}_i/\mathscr{F}_{i-1}(-D_{K''_p})$ is a sheaf concentrated in degree $0$ by  theorem  \ref{thm-formal-vanishingtominimal2}. 
Since $\mathcal{L}$ is very ample on $\mathfrak{S}^{\star, mod}_{K^pK_p'', \Sigma, \mathfrak{U}} $ and $(f_{K_p''})_\star \mathscr{F}_i/\mathscr{F}_{i-1}(-D_{K''_p})$  is a  coherent sheaf independant of $i$,  it follows that there exists $m$ such that $\mathrm{R}\Gamma( \mathfrak{S}^{tor,mod}_{K^pK_p', \Sigma, \mathfrak{U}} , f_{K_p''}^\star \mathcal{L}^m \otimes \mathscr{F}_i/\mathscr{F}_{i-1}(-D_{K''_p}))$ is concentrated in degree $0$ for all $i$. 

Let $L = \HH^0(\mathfrak{S}^{\star,mod}_{K^pK_p'', \mathfrak{U}} , \mathcal{L}[1/p])$. This is a rank $1$ projective $\oscr_{\mathcal{S}^{\star}_{K^pK_p''}}((\pi_{HT,K_p''})^{-1}(\mathcal{U}))$-module.   We observe that $$f_{K_p''}^\star \mathcal{L}^m \otimes_{\oscr_{\mathfrak{S}^{tor,mod}_{K^pK_p'', \Sigma, \mathfrak{U}}}} \mathfrak{V}_{\nu_A}^{n-an,+}(-D_{K''_p})[1/p]  = L^m \otimes_{\oscr_{\pi_{HT,K_p''}^{-1}(\mathcal{U})}} \mathfrak{V}_{\nu_A}^{n-an,+}(-D_{K''_p})[1/p].$$

We deduce that $$\mathrm{R}\Gamma( \mathfrak{S}^{tor,mod}_{K^pK_p'', \Sigma, \mathfrak{U}} , f_{K_p''}^\star\mathcal{L}^m \otimes \mathfrak{V}_{\nu_A}^{n-an,+}(-D_{K''_p})[1/p]) = $$ $$ L^m \otimes_{\oscr_{\pi_{HT,K_p''}^{-1}(\mathcal{U})}} \mathrm{R}\Gamma(\pi_{HT,K_p''}^{-1}(\mathcal{U}),   \mathfrak{V}_{\nu_A}^{n-an,+}(-D_{K''_p})[1/p])$$ and therefore  $\mathrm{R}\Gamma( \mathfrak{S}^{tor,mod}_{K^pK_p'', \Sigma, \mathfrak{U}},   \mathfrak{V}_{\nu_A}^{n-an,+}(-D_{K''_p})[1/p])$ is concentrated in degree $0$. Taking a Zariski covering of $ \mathfrak{S}^{tor,mod}_{K^pK_p'', \Sigma, \mathfrak{U}}$, we deduce that the associated augmented \v{C}ech complex on the generic fiber is exact. By theorem \ref{thm-formal-B-sheaves}, affinoids are acyclic for Banach sheaves arisings as generic fibers of flat small formal Banach sheaves. We deduce from the \v{C}ech to cohomology spectral sequence that  $\mathrm{R}\Gamma( (\pi_{HT,K_p''}^{tor})^{-1}(\mathcal{U}),   \mathcal{V}_{\nu_A}^{n-an,+}(-D_{K''_p}))$ is concentrated in degree $0$. The lemma follows   by taking the invariants under $K_p/K_p''$.

\end{proof}

\begin{lem}\label{lem-handy} Let $A^+$ be a ring, $A^++$ be an ideal of $A^+$. Assume that $A^+$ is $A^++$-adically complete and separated, Let $M_0 \stackrel{f_0}\rightarrow M_1 \stackrel{f_1}\rightarrow M_2$ be a complex of complete, separated, and torsion free $A^+$-modules. Assume that $$M_0 \otimes_{A^+} A^+/A^{++} \rightarrow M_1 \otimes_{A^+} A^+/A^{++} \rightarrow M_2 \otimes_{A^+} A^+/A^{++}$$ is exact.  Then the complex  $M_0 \stackrel{f_0}\rightarrow M_1 \stackrel{f_1}\rightarrow M_2$ is exact.
\end{lem}
\begin{proof} It follows from the assumptions that $\mathrm{Im}(f_0) + (A^{++}M_1 \cap  \mathrm{Ker}(f_1))= \mathrm{Ker}(f_1)$. Since $M_2$ is torsion free, we deduce that 
$ (A^{++}M_1 \cap  \mathrm{Ker}(f_1))= A^{++} \mathrm{Ker} (f_1)$. Let $m \in \mathrm{Ker}(f_1)$. By successive approximation, we can construct a sequence of elements $m_n \in M_0$ for $n \geq 0$,  with $m_n - m_{n+1} \in (A^{++})^n M_0$ and $f_0(m_n) - m \in \mathrm(A^{++})^{n+1} M_1$. The sequence $m_n$ converges to $m_\infty$ and $f_0(m_\infty) = m$. 
\end{proof}

\begin{lem}\label{lemma-local-vanishing1} Lemma \ref{lemma-local-vanishing}  holds without assuming that  $(G,X)$ is an Hodge type Shimura datum. 
\end{lem}
\begin{proof}  We consider a diagram of Shimura datum with $(G,X)$ of abelian type and $(G_1,X_1)$ of  Hodge type as in section \ref{section-Abelian-type2}:
\begin{eqnarray*}
\xymatrix{ (B_1,X_{B_1}) \ar[r] \ar[d] & (B, X_B) \ar[d] \\
(G_1,X_1) &  (G,X)}
\end{eqnarray*}
We may assume that the various morphisms between the groups $G_1, B_1, B, G$ over $\qq$ extend to morphisms over $\ocal_F$.
We will prove that lemma \ref{lemma-local-vanishing} holds first for  $B_1$ and second for $G$. We add decorations $H$ or $(H,X_H)$ to objects depending on the Shimura datum $(H,X_H)$ in order to make the distinction between the Shimura data. 
Let $\nu_A : T(B_1)(\ZZ_p) \rightarrow A^\times$ be a continuous, $n$-analytic character. It induces a character $\nu'_A : T(G_1)(\ZZ_p) \rightarrow A^\times$. 
We consider the morphism of  compact open subgroups $K_p = K_{p,m',0}(B_1) \rightarrow K'_p = K_{p,m',0}(G_1)$ and 
we deduce a finite morphism of Shimura varieties:
$$f : \mathcal{S}^{tor}(B_1,X_{B_1})_{K, \Sigma} \rightarrow \mathcal{S}^{tor}(G_1,X_{G_1})_{K',\Sigma}$$
with $K = K^pK_p $ and $K' = (K')^p K'_p$ and by restriction (see theorem \ref{thm-abelian-strategy}, lemma \ref{lem-compact-abelian})  we have  a finite morphism: 
$$ f : (\pi_{HT,K_p}^{tor})^{-1}(]C_{w,k}[_{n,n} K_p) \rightarrow (\pi_{HT,K'_p}^{tor})^{-1}(]C_{w,k}[_{n,n} K'_p)$$
We claim that there is an isomorphism $$f^\star \mathcal{V}_{\nu'_A}^{n-an} = \mathcal{V}_{\nu_A}^{n-an}.$$ This follows from remark \ref{rem-connection-Hodge-abelian} and remark \ref{rem-on-the-abelian-case}, observing that the torsor comming from the Shimura variety $(T,X_T)$ is trivial. We insist that  the above formula does not describe  the Hecke action at the level of $B_1$, but this is irrelevant for us now.  
We can the find another compact $K'' \subseteq K'$ such that there is a finite, generically finite \'etale morphism $f : \mathcal{S}^{tor,0}(G_1,X_{G_1})_{K'', \Sigma} \rightarrow \mathcal{S}^{tor,0}(B_1,X_{B_1})_{K,\Sigma}$ with group $\Delta$.
We now deduce that $\mathrm{R}\Gamma( (\pi_{HT,K_p}^{tor})^{-1}(\mathcal{U})^0, \mathcal{V}_{\nu_A}^{n-an}(-D))$  is the direct factor of invariants under $\Delta$ of $\mathrm{R}\Gamma( (\pi_{HT,K''_p}^{tor})^{-1}(\mathcal{U})^0, \mathcal{V}_{\nu'_A}^{n-an}(-D))$. We wrote  $(\pi_{HT,K_p}^{tor})^{-1}(\mathcal{U})^0$ for $\mathcal{S}^{tor,0}(B_1,X_{B_1})_{K,\Sigma} \cap (\pi_{HT,K_p}^{tor})^{-1}(\mathcal{U})$ and similarly for $G_1$. The lemma is thus proven for $B_1$. We observe that the lemma holds for any refinement of the cone decomposition $\Sigma$ by the projection formula.  
We now deduce the lemma  for $G$.  We reset the notations.  Let $\nu_A : T(G)(\ZZ_p) \rightarrow A^\times$ be a continuous, $n$-analytic character. It induces a character $\nu'_A : T(B_1)(\ZZ_p) \rightarrow A^\times$. 
Let $K = K^pK_p$ be a compact open subgroup of $G(\mathbb{A}_f)$. 
We can  find a compact open subgroup $K'  \subseteq B_1(\mathbb{A}_f)$ such that there is a finite, generically finite \'etale morphism $f : \mathcal{S}^{tor,0}(B_1,X_{B_1})_{K', \Sigma} \rightarrow \mathcal{S}^{tor,0}(G,X_{G})_{K,\Sigma}$ with group $\Delta(K,K')$.
This map induces a finite map : $ f : (\pi_{HT,K'_p}^{tor})^{-1}(]C_{w,k}[_{n,n} K'_p)^0 \rightarrow (\pi_{HT,K_p}^{tor})^{-1}(]C_{w,k}[_{n,n} K_p)^0$. We find that $$ \mathcal{V}^{n-an}_{\nu_A} = (f_\star \mathcal{V}^{n-an}_{\nu'_A})^{\Delta(K,K')}.$$ 
We deduce that  $\mathrm{R}\Gamma( (\pi_{HT,K_p}^{tor})^{-1}(\mathcal{U})^0, \mathcal{V}_{\nu_A}^{n-an}(-D))$  is the direct factor of invariants under $\Delta(K,K')$ of $\mathrm{R}\Gamma( (\pi_{HT,K'_p}^{tor})^{-1}(\mathcal{U})^0, \mathcal{V}_{\nu'_A}^{n-an}(-D))$. The lemma is thus proven for $G$.
\end{proof}
\subsection{Cup products and Serre duality}\label{subsection-analytic-duality} We now consider cup-products on locally analytic overconvergent cohomology. 

\begin{thm}\label{thm-cup-products} For all $w \in \WM$ and weights $\nu_A$ there is a  pairing: 
$$\langle, \rangle :  \HH^{i}_{w,an}(K_p, \nu_A, cusp)^{\pm,fs} \times \HH^{d-i}_{w,an}(K_p, \nu_A)^{\mp,fs} \rightarrow A$$
For any $t \in T(\qq_p)$ we have $\langle t -, -\rangle = \langle -, t^{-1} - \rangle$. 
Let $\nu=\nu_{alg}\chi$ be a locally algebraic weight so that $\kappa_{alg}=-w_{0,M}w\nu_{alg}-(w_{0,M}w\rho+\rho)$ is $M$-dominant.  The above pairing induces a pairing between the spectral sequences:
$$\langle, \rangle_{p,q,r} :  \mathbf{E}_{w,r}^{p,q}(K^p, \kappa_{alg}, \chi, cusp)^\pm \times \mathbf{E}^{-p,d-q}_{w,r}(K^p,\kappa,\chi^{-1})^\mp \rightarrow F$$
On the abutment of the spectral sequence the pairing $\langle, \rangle_{p,q,\infty}$ is induced by the  pairing  of theorem \ref{thm-pairing-over-coho}:
$$\HH^{p+q}_w(K^p,\kappa,\chi,cusp)^{\pm,fs} \times \HH^{d-p-q}_w(K^p,-2\rho_{nc}- w_{0,M} \kappa, \chi^{-1})^{\mp,fs} \rightarrow F.$$
\end{thm}
\begin{proof} We construct the pairing $$\langle, \rangle :  \HH^{i}_{w,an}(K_p, \nu_A, cusp)^{+,fs} \times \HH^{d-i}_{w,an}(K_p, \nu_A)^{-,fs} \rightarrow F.$$  We can realize $\mathrm{R}\Gamma_{w,an}(K^p, \nu_A, cusp)^{+,fs}$ as the the finite slope part of 
$$ \mathrm{R}\Gamma_{(\pi_{HT, K_{p, m',0}}^{tor})^{-1}(]C_{w,k}[_{s,\overline{s}}K_{p,m',0})}((\pi^{tor}_{HT,K_{p,m',0}})^{-1}(]C_{w,k}[_{s,-1}K_{p,m',0}), \mathcal{V}^{n-an}_{\nu_A}(-D))$$
for   $s >> 0$ and  $m' >>s$  by definition \ref{defi-support-condition2} and theorem \ref{thm-finiteslopesupport2}.  Similarly we can realize $\mathrm{R}\Gamma_{w,an}(K^p, \nu_A )^{-,fs}$ as the finite slope part of 
$$ \mathrm{R}\Gamma_{(\pi_{HT, K_{p, m',0}}^{tor})^{-1}(]C_{w,k}[_{\overline{s+1},s-1}K_{p,m',0})}((\pi^{tor}_{HT,K_{p,m',0}})^{-1}(]C_{w,k}[_{-1,s-1}K_{p,m',0}), \mathcal{D}^{n-an}_{\nu_A}).$$

Moreover, by construction we have a pairing $\mathcal{V}_{\nu_A}^{n-an}(-D)\times\mathcal{D}_{\nu_A}^{n-an}\to \mathcal{V}_{-2\rho_{nc}}(-D)\otimes A$.  We have a cup-product by proposition \ref{prop-construction-cuprod}: 
$$ \HH^i_{(\pi_{HT, K_{p, m',0}}^{tor})^{-1}(]C_{w,k}[_{s,\overline{s}}K_{p,m',0})}((\pi^{tor}_{HT,K_{p,m',0}})^{-1}(]C_{w,k}[_{s,-1}K_{p,m',0}), \mathcal{V}^{n-an}_{\nu_A}(-D))   \times $$
$$  
\HH^{d-i}_{(\pi_{HT, K_{p, m',0}}^{tor})^{-1}(]C_{w,k}[_{\overline{s+1},s-1}K_{p,m',0})}((\pi^{tor}_{HT,K_{p,m',0}})^{-1}(]C_{w,k}[_{-1,s-1}K_{p,m',0}), \mathcal{D}^{n-an}_{\nu_A})$$
$$ \rightarrow \HH^d_{ (\pi_{HT, K_{p, m',0}}^{tor})^{-1}(]C_{w,k}[_{\overline{s+1},\overline{s}}K_{p,m',0})}( (\pi^{tor}_{HT,K_{p,m',0}})^{-1}(]C_{w,k}[_{s,s-1}K_{p,m',0}),  \mathcal{V}_{-2\rho_{nc}}(-D) \otimes A)$$
and there is a trace map (by theorem \ref{thm-duality-GK}): $$\HH^d_{ (\pi_{HT, K_{p, m',0}}^{tor})^{-1}(]C_{w,k}[_{\overline{s+1},\overline{s}}K_{p,m',0})}( (\pi^{tor}_{HT,K_{p,m',0}})^{-1}(]C_{w,k}[_{s,s-1}K_{p,m',0}),  \mathcal{V}_{-2\rho_{nc}}(-D) \otimes A) \rightarrow A.$$

This pairing intertwines the actions of $\mathcal{H}_{p,m,0}^+$ and $\mathcal{H}_{p,m,0}^{-}$. It is straightforward (and painful) to check that the induced pairing $$\langle, \rangle :  \HH^{i}_{w,an}(K_p, \nu_A,  cusp)^{+,fs} \times \HH^{d-i}_{w,an}(K_p,\nu_A)^{-,fs} \rightarrow A$$ is independent of choices.

The rest of the theorem follows from the functoriality of the trace map. 

\end{proof}

\begin{thm}\label{thm-duality-field} Let $w \in \WM$, let $A$ be a field and fix a weight $\nu_A$.
The pairing $\langle, \rangle :  \HH^{i}_{w,an}(K_p, \nu_A, cusp)^{\pm,fs} \times \HH^{d-i}_{w,an}(K_p, \nu_A)^{\mp,fs} \rightarrow A$ is non-degenerate.
\end{thm}

\begin{proof} We follow closely the proof of  theorem \ref{thm-perfect-pairing}. Our strategy is to express the cohomologies as cones of cohomologies (possibly with compact support) of dagger spaces  and to use the duality for affinoid dagger spaces to produce dual complexes whose finite slope part computes the cohomology. 
We can realize $\mathrm{R}\Gamma_{w,an}(K^p, \nu_A, cusp)^{+,fs}$ as the the finite slope part of 
$$ \mathrm{R}\Gamma_{(\pi_{HT, K_{p, m',0}}^{tor})^{-1}(]C_{w,k}[_{s,\overline{s}}K_{p,m',0})}((\pi^{tor}_{HT,K_{p,m',0}})^{-1}(]C_{w,k}[_{s,-1}K_{p,m',0}), \mathcal{V}^{n-an}_{\nu_A}(-D))$$
for   $s >> 0$ and  $m' >>s$  by definition \ref{defi-support-condition2} and theorem \ref{thm-finiteslopesupport2}
 
 We now introduce the following closed subspaces  of $\mathcal{FL}_{G, \mu}$: 
  $$]C_{w,k}[_{\overline{s-1}',\overline{s}} = \cap_{ \epsilon >0} ]C_{w,k}[_{s-1- \epsilon,\overline{s}}, ~\textrm{and} ~]C_{w,k}[_{\overline{[s-1,s]}',\overline{s}} = ]C_{w,k}[_{\overline{s-1}',\overline{s}} \setminus ]C_{w,k}[_{s,\overline{s}}.$$ 
  Both are the closure of quasi-compact open subspaces of $\mathcal{FL}_{G,\mu}$.  We have an exact triangle:
  $$ \mathrm{R}\Gamma_{(\pi_{HT, K_{p, m',0}}^{tor})^{-1}(]C_{w,k}[_{\overline{[s-1,s]}',\overline{s}}K_{p,m',0})}((\pi^{tor}_{HT,K_{p,m',0}})^{-1}(]C_{w,k}[_{s-2,-1}K_{p,m',0}), \mathcal{V}^{n-an}_{\nu_A}(-D)) \rightarrow  $$ $$ \mathrm{R}\Gamma_{(\pi_{HT, K_{p, m',0}}^{tor})^{-1}(]C_{w,k}[_{\overline{s-1}',\overline{s}}K_{p,m',0})}((\pi^{tor}_{HT,K_{p,m',0}})^{-1}(]C_{w,k}[_{s-2,-1}K_{p,m',0}), \mathcal{V}^{n-an}_{\nu_A}(-D)) $$ $$ \rightarrow \mathrm{R}\Gamma_{(\pi_{HT, K_{p, m',0}}^{tor})^{-1}(]C_{w,k}[_{s,\overline{s}}K_{p,m',0})}((\pi^{tor}_{HT,K_{p,m',0}})^{-1}(]C_{w,k}[_{s,-1}K_{p,m',0}), \mathcal{V}^{n-an}_{\nu_A}(-D)) \stackrel{+1}\rightarrow$$
  
  We observe that $\mathrm{R}\Gamma_{(\pi_{HT, K_{p, m',0}}^{tor})^{-1}(]C_{w,k}[_{\overline{[s-1,s]}',\overline{s}}K_{p,m',0})}((\pi^{tor}_{HT,K_{p,m',0}})^{-1}(]C_{w,k}[_{s-2,-1}K_{p,m',0}), \mathcal{V}^{n-an}_{\nu_A}(-D))$ and $ \mathrm{R}\Gamma_{(\pi_{HT, K_{p, m',0}}^{tor})^{-1}(]C_{w,k}[_{\overline{s-1}',\overline{s}}K_{p,m',0})}((\pi^{tor}_{HT,K_{p,m',0}})^{-1}(]C_{w,k}[_{s-2,-1}K_{p,m',0}), \mathcal{V}^{n-an}_{\nu_A}(-D)) $ are the cohomology with compact support of the dagger spaces $(\pi_{HT, K_{p, m',0}}^{tor})^{-1}(]C_{w,k}[_{\overline{[s-1,s]}',\overline{s}}K_{p,m',0})$ and  $(\pi_{HT, K_{p, m',0}}^{tor})^{-1}(]C_{w,k}[_{\overline{s-1}',\overline{s}}K_{p,m',0})$ respectively.

  For the $-$ theory, we see that  for any $\epsilon >0$, the cohomology
   $$ \mathrm{R}\Gamma_{(\pi_{HT, K_{p, m',0}}^{tor})^{-1}(]C_{w,k}[_{\overline{s},s-\epsilon}K_{p,m',0})}((\pi^{tor}_{HT,K_{p,m',0}})^{-1}(]C_{w,k}[_{-1,s-\epsilon}K_{p,m',0}), (\mathcal{V}^{n-an}_{\nu_A})^\vee \otimes \mathcal{V}_{-2\rho_{nc}})$$
   is quasi-isomorphic to the cone of 
  $$\mathrm{R}\Gamma((\pi^{tor}_{HT,K_{p,m',0}})^{-1}(]C_{w,k}[_{\overline{s}',s-\epsilon}K_{p,m',0}), (\mathcal{V}^{n-an}_{\nu_A})^\vee \otimes \mathcal{V}_{-2\rho_{nc}})[-1]   \rightarrow $$ $$\mathrm{R}\Gamma((\pi^{tor}_{HT,K_{p,m',0}})^{-1}(]C_{w,k}[_{\overline{[s-1,s]}',s-\epsilon}K_{p,m',0}), (\mathcal{V}^{n-an}_{\nu_A})^\vee \otimes \mathcal{V}_{-2\rho_{nc}})[-1]$$
  We point out  that in this argument, in order to apply duality theorems, we prefer to use as coefficients $(\mathcal{V}^{n-an}_{\nu_A})^\vee \otimes \mathcal{V}_{-2\rho_{nc}}$ rather than $\mathcal{D}^{n-an}_{\nu_A}$. There is a map $\mathcal{D}^{n-an}_{\nu_A} \rightarrow   (\mathcal{V}^{n-an}_{\nu_A})^\vee \otimes \mathcal{V}_{-2\rho_{nc}}$. When we pass to the inverse limit on $n$, these maps induce an isomorphism. Since $A$ is a field, $(\mathcal{V}^{n-an}_{\nu_A})^\vee \otimes \mathcal{V}_{-2\rho_{nc}}$ is a projective Banach sheaf. 
  
  Passing to the limit over $\epsilon$, we deduce easily  that the finite slope part of the cone of  $$\mathrm{R}\Gamma((\pi^{tor}_{HT,K_{p,m',0}})^{-1}(]C_{w,k}[_{\overline{s}', \overline{s}}K_{p,m',0}), (\mathcal{V}^{n-an}_{\nu_A})^\vee \otimes \mathcal{V}_{-2\rho_{nc}})[-1]   \rightarrow $$ $$ \mathrm{R}\Gamma((\pi^{tor}_{HT,K_{p,m',0}})^{-1}(]C_{w,k}[_{\overline{[s-1,s]}', \overline{s}}K_{p,m',0}), (\mathcal{V}^{n-an}_{\nu_A})^\vee \otimes \mathcal{V}_{-2\rho_{nc}})[-1]$$
  computes $\mathrm{R}\Gamma_{w,an}(K^p, \nu_A)^{-,fs}$.

 We can consider a finite covering  $\cup_i U_i ^\dagger$  of $(\pi_{HT, K_{p, m',0}}^{tor})^{-1}(]C_{w,k}[_{\overline{s-1}',\overline{s}}K_{p,m',0})$ by affinoid dagger spaces (or equivalently by closed spaces which are the closure of quasi-compact open spaces) and we let $\cup_j V_j^\dagger$ be a covering of  $(\pi_{HT, K_{p, m',0}}^{tor})^{-1}(]C_{w,k}[_{\overline{[s-1,s]}',\overline{s}}K_{p,m',0})$ by affinoid dagger spaces, refining $\cup_i (U_i ^\dagger \cap (\pi_{HT, K_{p, m',0}}^{tor})^{-1}(]C_{w,k}[_{\overline{[s-1,s]}',\overline{s}}K_{p,m',0}))$.
 
 We let $A^\bullet = \check{C}_c(\{U_i^\dagger\}, \mathcal{V}^{n-an}_{\nu_A}(-D))$ and $B^\bullet =  \check{C}_c(\{V_i^\dagger\}, \mathcal{V}^{n-an}_{\nu_A}(-D))$
  be the \v{C}hech cohomology complex with compact support. There is a natural  map $B^\bullet \rightarrow A^\bullet$ and we let $C^\bullet$ be the cone of this map which represents $$\mathrm{R}\Gamma_{(\pi_{HT, K_{p, m',0}}^{tor})^{-1}(]C_{w,k}[_{s,\overline{s}}K_{p,m',0})}((\pi^{tor}_{HT,K_{p,m',0}})^{-1}(]C_{w,k}[_{s,-1}K_{p,m',0}), \mathcal{V}^{n-an}_{\nu_A}(-D)).$$
  
 Let $t \in T^{++}$. As explained in lemma \ref{lem-representing-mapT} , we can lift $t$ to a compact endomorphism of $C^\bullet$. 
 
 We let $D^\bullet$ be the complex which is the strict dual of $C^\bullet$.  By the duality for affinoid spaces, this complex computes the cone of $$\mathrm{R}\Gamma((\pi^{tor}_{HT,K_{p,m',0}})^{-1}(]C_{w,k}[_{\overline{s}',s-\epsilon}K_{p,m',0}), (\mathcal{V}^{n-an}_{\nu_A})^\vee \otimes \mathcal{V}_{-2\rho_{nc}})[-1] \rightarrow $$ $$   \mathrm{R}\Gamma((\pi^{tor}_{HT,K_{p,m',0}})^{-1}(]C_{w,k}[_{\overline{[s-1,s]}',s-\epsilon}K_{p,m',0}), (\mathcal{V}^{n-an}_{\nu_A})^\vee \otimes \mathcal{V}_{-2\rho_{nc}})[-1].$$ 
 For any slope $h$, taking the slope $\leq h$ part yields a perfect pairing between $C^{\bullet, \leq h}$ and $D^{\bullet, \leq h}$ and passing to cohomology gives the theorem.
 
  \end{proof}
  
Using duality we can deduce an improved vanishing theorem.

  \begin{thm}\label{theo-vanishing-big-sheaves} Let $w\in\WM$ and let $\nu_A$ be a weight.  The cohomology  $\HH^i_{w,an}(K^p, \nu_A,cusp)^{\pm,fs}$ is concentrated in degree $[0, \ell_{\pm}(w)]$ and  the cohomology  $\HH^i_{w,an}(K^p, \nu_A)^{\pm,fs}$ is concentrated in degree $[\ell_{\pm}(w),d]$. Moreover, if $A$ is a domain,  $\HH^{\ell_{\pm}(w)}_{w,an}(K^p, \nu_A)^{\pm,fs}$ is a torsion free sheaf of $\oscr_{\Spa(A,A^+)}$-modules. 
\end{thm}

\begin{proof} The case of $\HH^i_{w,an}(K^p, \nu_A,cusp)^{\pm,fs}$ is theorem \ref{thm-one-half-conj}. The case of  $\HH^i_{w,an}(K^p, \nu_A)^{\pm,fs}$ when $A$ is a field follows from the duality theorem \ref{thm-duality-field}. The case of general coefficients can be deduced as  follows. First observe that we can suppose that $A$ is reduced as the universal weight space is reduced.  Let us fix a compact operator $t \in T^{\pm}$. For any point $x \in \Spa(A,A^+)$ and any slope $h$, we can find a slope $h' \geq h$ and an open $\Spa (B, B^+) \hookrightarrow \Spa(A,A^+)$ such that  $\HH^i_{w,an}(K^p, \nu_A)\vert_{\Spa(B,B)^+}$ has slope $h'$ decomposition for $t$, and  $\HH^i_{w,an}(K^p, \nu_A)(\Spa(B,B)^+)^{\pm,\leq h'}$ is computed by a perfect complex  $C^\bullet$ of $B$-modules. Since we know that for any $x \in \Spa(B,B^+)$, $C^\bullet \otimes_B k(x)$ has cohomology concentrated in the range $[\ell_{\pm}(w),d]$, we deduce that $C^\bullet$ is quasi-isomorphic to $\tau_{\geq \ell_{\pm}(w)} C^\bullet$, and this truncated complex  is a perfect complex (by an application of \cite{mumford}, lemma 2 on p. 49). It follows that its cohomology in degree $\ell_{\pm}(w)$ is torsion free if $A$ is a domain. 
\end{proof}
 
\subsection{Slope estimates and control theorem}\label{subsection-control2}  We now formulate a conjecture regarding the slopes of the locally analytic overconvergent cohomologies. This is of course consistent with conjecture \ref{conj-strongslopes}. 

\begin{conj}\label{conj-strongslopes2}
Fix $w\in\WM$,  and let $\nu:T^c(\Z_p) \rightarrow \mathbb{C}_p^\times $ be a continuous character. 
\begin{enumerate}
\item  For any  character $\lambda$ of $T^+$ on $\mathrm{R}\Gamma_{w,an}(K^p, \nu)^{+,fs}$ or $\mathrm{R}\Gamma_{w,an}(K^p, \nu,cusp)^{+,fs}$,  we have $v(\lambda)\geq 0$.  
\item For any  character $\lambda$ of $T^-$ on  $\mathrm{R}\Gamma_{w,an}(K^p, \nu)^{-,fs}$ or $\mathrm{R}\Gamma_{w,an}(K^p, \nu,cusp)^{-,fs}$, we have $v(\lambda)\leq 0$.  
\end{enumerate}
\end{conj}

\begin{rem} The inequalities in conjecture \ref{conj-strongslopes2} are compatible with those of conjecture \ref{conj-strongslopes} due to the way we have renormalized the Hecke operators acting on locally analytic cohomology.  Similarly the slightly weaker bounds we prove in theorem \ref{thm-slopes2} are compatible with those of theorem \ref{thm-slopes}.
\end{rem} 

We can prove  a bound which is slightly weaker than the conjecture.

\begin{thm}\label{thm-slopes2} 
 Fix $w\in\WM$,  and let $\nu : T(\ZZ_p) \rightarrow \mathbb{C}_p^\times $ be a continuous character. 
 \begin{enumerate}
 \item  For any  character $\lambda$ of $T^+$ on $\mathrm{R}\Gamma_{w,an}(K^p, \nu)^{+,fs}$ and $\mathrm{R}\Gamma_{w,an}(K^p, \nu,cusp)^{+,fs}$,  we have $v(\lambda)\geq -w^{-1}w_{0,M}\rho-\rho$ and $v(\lambda)\geq -w^{-1}w_{0,M}\rho-\rho + w^{-1} 2\rho_{nc}$ 
\item For any  character $\lambda$ of $T^-$ on  $\mathrm{R}\Gamma_{w,an}(K^p, \nu)^{-,fs}$ and $\mathrm{R}\Gamma_{w,an}(K^p, \nu,cusp)^{-,fs}$, we have $v(\lambda)\leq -w^{-1}\rho+\rho$ and $v(\lambda) \leq -w^{-1}\rho+ w^{-1} 2\rho_{nc}$. 
\end{enumerate}
\end{thm}
\begin{proof} This is similar to the proof of theorem \ref{thm-slopes} and left to the reader.
\end{proof}

\begin{coro}\label{coro-class-overconvergent-analytic} Fix $w\in \WM$, and let $\nu=\nu_{alg}\chi$ be a locally algebraic character.  Let $\kappa_{alg}=-w_{0,M}w\nu_{alg}-(w_{0,M}w\rho+\rho)$ and suppose $\kappa_{alg}\in X^\star(T)^{M,+}$.  Then the morphisms 
$$ \mathrm{R}\Gamma_w( K^p, \kappa_{alg}, \chi )^{+,sss_{M,w}(\kappa_{alg})} \rightarrow \mathrm{R}\Gamma_{w,an}(K^p,\nu)(-\nu_{alg})^{+,sss_{M,w}(\kappa_{alg})}$$
$$\mathrm{R}\Gamma_{w,an}(K^p,\nu)(\nu_{alg})^{-,sss_{M,w}(-w_{0,M}\kappa_{alg}-2\rho_{nc})} \rightarrow \mathrm{R}\Gamma_w( K^p, -w_{0,M}\kappa-2\rho_{nc}, \chi^{-1} )^{-,sss_{M,w}(-w_{0,M}\kappa_{alg}-2\rho_{nc})}$$
and the corresponding morphisms for cuspidal cohomology are all quasi-isomorphisms.

\end{coro} 

\begin{proof} This follows from the spectral sequence of theorem \ref{spectral-sequence-lan-to-classical}, together with the slope bounds  of theorem \ref{thm-slopes2}.
\end{proof}

\begin{rem} If we assume conjectures \ref{conj-strongslopes} and \ref{conj-strongslopes2}, we can replace the $sss_{M,w}$ condition by the $ss_{M,w}$ condition in the above corollary.
\end{rem}

\subsection{Eigenvarieties}
Consider the Iwasawa algebra $\ZZ_p[[T^c(\ZZ_p)]]$ and the weight space $$\mathcal{W} = \Spa ( \ZZ_p[[T^c(\ZZ_p)]], \ZZ_p[[T^c(\ZZ_p)]]) \times_{\Spa (\ZZ_p, \ZZ_p)} \Spa (\qq_p, \ZZ_p).$$

For a $(A,A^+)$ a complete Tate $(F,\ocal_F)$-algebra, $$\mathrm{Hom} ( \Spa (A,A^+), \mathcal{W}) = \{ \textrm{Continuous characters}~\nu_A : T^c(\ZZ_p) \rightarrow A^\times\}.$$

Also let $\widehat{T}$ be the analytic adic  space of characters of $T(\qq_p)$, whose restriction to $T(\ZZ_p)$ factor through $T^c(\ZZ_p)$.  
If we fix a splitting $\xi$ for the map $T(\qq_p) \rightarrow T(\qq_p)/T(\ZZ_p)$, and we fix an isomorphism $T(\qq_p)/T(\ZZ_p) \simeq \ZZ^r$ (for  $r = \mathrm{rank}(X^\star(T^d))$), then we have an isomorphism $\widehat{T} \simeq \mathcal{W} \times (\mathbb{G}_m^{an})^{r}$, where the map sends a character $\lambda$ to $(\lambda\vert_{T^c(\ZZ_p)}, \lambda(\xi(e_1)), \cdots, \lambda(\xi(e_r)))$ (for the canonical basis $e_1, \cdots, e_r$ of $\ZZ^r$). 
We also observe that there is a natural map $T \rightarrow \oscr_{\widehat{T}}$.

Let $\Spa(A,A^+)\subset\mathcal{W}$ be an affinoid open and let $\nu_A^{un}:T^c(\ZZ_p)\rightarrow A^\times$ be the universal character.

For each $w\in \WM$ we have sheaves over $\Spa(A,A^+)$:  $\oplus_i  \HH^i_{w,an}(K^p,\nu_A^{un})^{\pm,fs}$ and $\oplus_i \HH^i_{w,an}(K^p, \nu_A^{un}, cusp)^{\pm,fs}$.
Gluing these sheaves for an affinoid covering of $\mathcal{W}$ we deduce that there are sheaves of $\oscr_{\mathcal{W}}$-modules $\oplus_{i=\ell_{\pm}(w)}^{d} \HH^i_{w,an}(K^p,\nu^{un})^{\pm,fs}$ and $\oplus_{i=0}^{\ell_{\pm}(w)} \HH^i_{w,an}(K^p, \nu^{un}, cusp)^{\pm,fs}$, admitting slope decomposition. 

\begin{rem} Our weight variable  $\nu$ is morally a translate of the infinitesimal character. We recall how to switch between the weight of our coherent cohomologies and the infinitesimal character. Let $\nu = \nu^{alg} \chi$ be a locally algebraic weight. 
Then $\HH^i_{w,an}(K^p,\nu^{un})^{+,fs} \otimes_{\oscr_{\mathcal{W}}} k(\nu)$ and $\HH^i_{w,an}(K^p,\nu^{un}, cusp)^{+,fs} \otimes_{\oscr_{\mathcal{W}}} k(\nu)$ are related to classical  or overconvergent  cohomology  in weight $\kappa_{alg}$ where $\nu_{alg} = -w^{-1}w_{0,M} (\kappa_{alg} + \rho)-\rho$.
On the other hand, $\HH^i_{w,an}(K^p,\nu^{un})^{-,fs} \otimes_{\oscr_{\mathcal{W}}} k(\nu)$ and $\HH^i_{w,an}(K^p,\nu^{un}, cusp)^{-,fs} \otimes_{\oscr_{\mathcal{W}}} k(\nu)$ are related to classical or overconvergent cohomology  in weight $\kappa_{alg}^\vee= -w_{0,M}\kappa_{alg}-2\rho_{nc}$ and $\nu_{alg} = w^{-1}(\kappa_{alg}^\vee + \rho)-\rho$.
\end{rem}

We also have cup-products: 
$$ \HH^i_{w,an}(K^p,\nu^{un})^{\pm,fs} \otimes \HH^{d-i}_{w,an}(K^p, \nu^{un}, cup)^{\mp,fs} \rightarrow  \oscr_{\mathcal{W}}.$$
We modify the action of $T(\qq_p)$ on $ \HH^i_{w,an}(K^p,\nu^{un})^{-,fs}$ and $\HH^i_{w,an}(K^p,\nu^{un},cusp)^{-,fs}$, by composing it with the inverse map $t \mapsto t^{-1}$. With this modification, the operators $t \in T(\qq_p)$ are self adjoint for the pairing. 
The action of $T(\qq_p)$ on the various sheaves  $ \HH^i_{w,an}(K^p,\nu^{un})^{\pm,fs}$ and $\HH^i_{w,an}(K^p,\nu^{un},cusp)^{\pm,fs}$ over $\mathcal{W}$ produces a spectral variety, which is a Zariski closed subset  $\mathcal{Z} \hookrightarrow \widehat{T}$ (see section \ref{section-construction-eigenvariety}).   We have coherent sheaves $ \HH^i_{w,an}(K^p,\nu^{un})_{\mathcal{Z}}^{\pm,fs}$ and $\HH^i_{w,an}(K^p,\nu^{un},cusp)_{\mathcal{Z}}^{\pm,fs}$  of $\oscr_{\widehat{T}}$-modules supported on $\mathcal{Z}$ and  $\oscr_\mathcal{Z}$ acts faithfully on the direct sum of all these sheaves. Moreover, if $\pi :  \widehat{T} \rightarrow \mathcal{W}$ is the projection to the weight space, 
$ \pi_\star \HH^i_{w,an}(K^p,\nu^{un})_{\mathcal{Z}}^{\pm,fs}  = \HH^i_{w,an}(K^p,\nu^{un})^{\pm,fs} $ and $\pi_\star \HH^i_{w,an}(K^p,\nu^{un},cusp)_{\mathcal{Z}}^{\pm,fs} = \HH^i_{w,an}(K^p,\nu^{un}, cusp)^{\pm,fs}$.

We  assume that $K^p = \prod K_\ell$ and we let $S'$ be the set of primes $\ell \neq p$ such that $K_\ell$ is not hyperspecial. We let $S = S' \cup \{p\}$.   Let $\mathcal{H}^{S} = \mathcal{C}^\infty_c( G(\mathbb{A}_f^S)// K^S, \qq)$ be the spherical Hecke algebra at places away from $S$.   We let $\mathcal{E}$ be the finite adic space over $\mathcal{Z}$ whose algebra is the coherent $\oscr_{\mathcal{Z}}$-algebra generated by the operators  $h \in \mathcal{H}^{S}$ acting on the sum of  the sheaves $ \HH^i_{w,an}(K^p,\nu^{un})_{\mathcal{Z}}^{\pm,fs}$ and $\HH^i_{w,an}(K^p,\nu^{un},cusp)_{\mathcal{Z}}^{\pm,fs}$, for all $i \in \ZZ$, all  $w \in \WM$, and all choices of $+$ or $-$.

A  (classical rigid analytic) point of $\mathcal{Z}$ corresponds to  a character $\lambda_p:T(\qq_p) \to \C_p^\times$. We can attach to $\lambda_p$ the weight $\nu = \lambda_p\vert_{T^c(\ZZ_p)} :  T^c(\ZZ_p)\to \C_p^\times$.  When the weight $\nu = \nu_{alg}\chi$ is locally algebraic, then we let $\lambda_p^{sm} = \lambda_p \nu_{alg}^{-1}$ with $\nu_{alg}$ viewed as a character of $T(\qq_p)$. Then $\lambda_p^{sm}$ factors through a character of $T(\qq_p)/T_b$ where $b$ is the conductor of $\chi$. 
\begin{rem} The superscript ${sm}$  stands for smooth, because the character $\lambda_p^{sm}$ is the smooth character attached to $\lambda_p$. Remark that the Hecke action on classical cohomology produces smooth characters. This is visible in point $(1)$ of theorem \ref{thm-eigenvariety}  below. 
\end{rem}

  A point of $\mathcal{E}$ corresponds to a pair $(\lambda_p,\lambda^S)$ where $\lambda^S: \mathcal{H}^{S}\to \C_p$ is a character.

\begin{thm}\label{thm-eigenvariety}   The eigenvariety $\pi:\mathcal{E} \rightarrow \mathcal{W}$ is locally quasi-finite and partially proper. It carries graded coherent sheaves $$\bigoplus_{w\in\WM,k\in\ZZ} (\HH^k_{w,an}(K^p,\nu^{un})_{\mathcal{Z}}^{+,fs} \oplus \HH^k_{w,an}(K^p,\nu^{un})_{\mathcal{Z}}^{-,fs}  $$ $$ \oplus \HH^k_{w,an}(K^p,\nu^{un}, cusp)_{\mathcal{Z}}^{+,fs} \oplus \HH^k_{w,an}(K^p,\nu^{un}, cusp)_{\mathcal{Z}}^{-,fs} )$$ 
 and they satisfy the following properties:
 \begin{enumerate} 
\item (Any classical, finite slope eigenclass gives a point of the eigenvariety) For any $\kappa_{alg}\in X^\star(T^c)^{M,+}$, finite order character $\chi:T^c(\ZZ_p)\to\overline{F}^\times$, and any system of Hecke eigenvalues $(\lambda^{sm}_p,\lambda^S)$ occurring in $\HH^i(K^p,\kappa_{alg},\chi)^{+,fs}$ (resp. $\HH^i(K^p,\kappa_{alg}^\vee,\chi^{-1})^{-,fs}$, $\HH^i(K^p,\kappa_{alg},\chi,cusp)^{+,fs}$, or $\HH^i(K^p,\kappa_{alg}^\vee,\chi^{-1},cusp)^{-,fs}$) there is a $w=w_Mw^M\in W$, so that if $\nu=\nu_{alg}\chi$ with $\nu_{alg}=-w^{-1}w_{0,M}(\kappa_{alg}+\rho)-\rho$, then $(\lambda^{sm}_p \nu_{alg},\lambda^S)$ is a point of the eigenvariety $\mathcal{E}$ which lies in the support of $\oplus_{k\in\mathbb{Z}}\HH^k_{w^M,an}(K^p,\nu^{un})_{\mathcal{Z}}^{+,fs} $ (resp. $\oplus_{k\in\mathbb{Z}}\HH^k_{w^M,an}(K^p,\nu^{un})_{\mathcal{Z}}^{-,fs} $, $\oplus_{k\in\mathbb{Z}}\HH^k_{w^M,an}(K^p,\nu^{un},cusp)_{\mathcal{Z}}^{+,fs} $, or $\oplus_{k\in\mathbb{Z}}\HH^k_{w^M,an}(K^p,\nu^{un},cusp)_{\mathcal{Z}}^{-,fs} $).

\item (Small slope points of the eigenvariety in regular, locally algebraic weights are classical) Conversely if $\nu=\nu_{alg}\chi$ is a locally algebraic weight with $\nu_{alg}\in X^\star(T)^+$, and $(\lambda_p,\lambda^S)$ is a point of $\mathcal{E}$ in the support of $\oplus_{k\in\mathbb{Z}}\HH^k_{w,an}(K^p,\nu^{un})_{\mathcal{Z}}^{+,fs} $  (resp. $\oplus_{k\in\mathbb{Z}}\HH^k_{w,an}(K^p,\nu^{un})_{\mathcal{Z}}^{-,fs} $, $\oplus_{k\in\mathbb{Z}}\HH^k_{w,an}(K^p,\nu^{un},cusp)_{\mathcal{Z}}^{+,fs} $, or $\oplus_{k\in\mathbb{Z}}\HH^k_{w,an}(K^p,\nu^{un},cusp)_{\mathcal{Z}}^{-,fs} $) for some $w\in\WM$, and if $\lambda_p^{sm}=\lambda_p\nu_{alg}^{-1}$ satisfies $+,sss_w(\nu)$ then $(\lambda^{sm}_p,\lambda^S)$ occurs in $\HH^i(K^p,\kappa_{alg},\chi)^{+,fs}$ (resp. $\HH^i(K^p,\kappa^\vee_{alg},\chi^{-1})^{-,fs}$, $\HH^i(K^p,\kappa_{alg},\chi,cusp)^{+,fs}$, or $\HH^i(K^p,\kappa^\vee_{alg},\chi^{-1},cusp)^{-,fs}$) for $\kappa_{alg}=-w_{0,M}w(\nu_{alg}+\rho)-\rho$.

\item (Serre duality interpolates over the eigenvariety) We have pairings: 
 $$\HH^k_{w,an}(K^p,\nu^{un})_{\mathcal{Z}}^{\pm,fs} \otimes \HH^{d-k}_{w,an}(K^p,\nu^{un}, cusp)_{\mathcal{Z}}^{\mp,fs} \rightarrow  \pi^{-1} \oscr_{\mathcal{E}}.$$
and these pairings are compatible with Serre duality under the classicality theorem. 
\end{enumerate} 
\end{thm}

\begin{proof} We only prove the first point for $\HH^i(K^p,\kappa_{alg},\chi)^{+,fs}$. The other cases are similar. There is a succession of three spectral sequences, having as input the sheaves $\HH^k_{w,an}(K^p, \nu^{un})^{+,fs}$ (for varying $k$ and $w$) and finally converging to the classical cohomology $\HH^i(K^p,\kappa_{alg},\chi)^{+,fs}$.
The first spectral sequence is the Tor spectral sequence  of theorem \ref{thm-tor-spectral}, from the cohomology sheaf over $\mathcal{W}$ to the locally analytic overconvergent cohomology in a given weight.  
The second spectral sequence is the BGG spectral sequence of theorem \ref{spectral-sequence-lan-to-classical}, from locally analytic overconvergent cohomology to overconvergent cohomology. The third spectral sequence is the spectral sequence of theorem \ref{thm-spectral-sequence}, from overconvergent to classical cohomology.
Therefore, starting from a classical class, we can lift it successively to the $E_1$ terms of the last two spectral sequences (for suitable choices of $w_M$  and $w^M$) and then to a class in a suitable  $\mathrm{Tor}^{\oscr_{\mathcal{W}}}_{-p}( \HH^q_{w,an}(K^p, \nu^{un})^{+,fs}, k(\nu))$ for $\nu = \nu_{alg}\chi$. Therefore,  $(\lambda^{sm}_p \nu_{alg},\lambda^S)$ is in the support of $\HH^q_{w,an}(K^p, \nu^{un})^{+,fs}$. 

We now prove point $(2)$. Let $(\lambda_p,\lambda^S)$ be a point  in the support of $\oplus_{k\in\mathbb{Z}}H^k_{w,an}(K^p, \nu^{un})^{+,fs}$ over a weight  $\nu$. Then one sees that there is a corresponding class in $\oplus_k \HH^{k}_{w,an}( K^p, \nu)^{+,fs}$. Indeed, take $k$ maximal such that $(\lambda_p,\lambda^S)$ is in the support of $\oplus_{k\in\mathbb{Z}}H^k_{w,an}(K^p, \nu^{un})^{+,fs}$. The Tor spectral sequence implies that $(\lambda_p, \lambda^S)$ is in the support of $\HH^{k}_{w,an}( K^p, \nu)^{+,fs}$. 
  Then we conclude by the control theorems (corollary \ref{coro-class-overconvergent-analytic} and theorem \ref{thm-control-thm-Coleman}): $$\oplus_k \HH^{k}_{w,an}( K^p, \nu)^{+,sss_w(\nu)} = \oplus_k \HH^{k}( K^p, -w_{0,M}w(\nu_{alg}+\rho)-\rho, \chi)^{+,sss_w(\nu)}.$$
The last point follows from the functoriality of the pairing.
\end{proof}

We now define certain components of the eigenvariety of maximal dimension and show that they contain all finite slope interior cohomology classes. We first define sheaves:
$$ \overline{\HH}^{\ell_{\pm}(w)}_{w,an}(K^p, \nu^{un})^{\pm,fs} = \mathrm{Im}( {\HH}^{\ell_{\pm}(w)}_{w,an}(K^p, \nu^{un}, cusp)^{\pm,fs} \rightarrow  {\HH}^{\ell_{\pm}(w)}_{w,an}(K^p, \nu^{un})^{\pm,fs}).$$

\begin{prop}\label{prop-properties-interior}

\begin{enumerate}
\item The sheaves  $\overline{\HH}^{\ell_{\pm}(w)}_{w,an}(K^p, \nu^{un})^{\pm,fs}$ are torsion free sheaves of $\oscr_{\mathcal{W}}$-modules.
\item For any $\nu:T^c(\ZZ_p)\to\C_p^\times$, the map $\overline{\HH}^{\ell_{\pm}(w)}_{w,an}(K^p, \nu^{un})^{\pm,fs} \otimes_{\oscr_{\mathcal{W}}} k(\nu) \rightarrow \overline{\HH}^{\ell_{\pm}(w)}_{w,an}(K^p, \nu)^{\pm,fs}$ is surjective. 
\item For any $\kappa_{alg} \in X^{\star}(T^c)^{M_\mu, +}$ and $\chi:T^c(\ZZ_p) \rightarrow \overline{F}^\times$ a finite order character, let $\nu=\nu^{alg}\chi$ with $\nu^{alg} = -w^{-1} w_{0,M} (\kappa_{alg} + \rho)-\rho$.  Then the map
$$\overline{\HH}^{\ell_+(w)}_{w}(K^p, \kappa, \chi)^{+,fs} \rightarrow \overline{\HH}^{\ell_{+}(w)}_{w,an}(K^p, \nu)^{+,fs}$$ is injective and the map 
$$ \overline{\HH}^{\ell_{-}(w)}_{w,an}(K^p, \nu)^{-,fs} \rightarrow \overline{\HH}^{\ell_-(w)}_{w}(K^p, -w_{0,M}\kappa-2\rho_{nc}, \chi^{-1})^{-,fs}$$ is surjective. 
\item There is a pairing $$\overline{\HH}^{\ell_{+}(w)}_{w,an}(K^p, \nu^{un})^{+,fs} \times \overline{\HH}^{\ell_{-}(w)}_{w,an}(K^p, \nu^{un})^{-,fs}  \rightarrow \oscr_{\mathcal{W}}$$ compatible with the preceding maps and the pairings $$\overline{\HH}^{\ell_{+}(w)}_{w,an}(K^p, \nu)^{+,fs} \times \overline{\HH}^{\ell_{-}(w)}_{w,an}(K^p, \nu)^{-,fs}  \rightarrow k(\nu)$$ and $$\overline{\HH}^{\ell_{+}(w)}_{w}(K^p, \kappa, \chi)^{+,fs} \times \overline{\HH}^{\ell_{-}(w)}_{w}(K^p, -w_{0,M} \kappa-2 \rho_{nc}, \chi^{-1})^{-,fs}  \rightarrow F.$$
\end{enumerate}
\end{prop}

\begin{proof} The first point follows from theorem \ref{theo-vanishing-big-sheaves}. The vanishing theorem and Tor spectral sequence imply that ${\HH}^{\ell_{\pm}(w)}_{w,an}(K^p, \nu^{un},cusp)^{\pm,fs} \otimes_{\oscr_{\mathcal{W}}} k(\nu) \rightarrow {\HH}^{\ell_{\pm}(w)}_{w,an}(K^p, \nu,cusp)^{\pm,fs}$ is an isomorphism, which gives the second point.  The vanishing theorems and the BGG spectral sequence imply that ${\HH}^{\ell_+(w)}_{w}(K^p, \kappa, \chi)^{+,fs} \rightarrow {\HH}^{\ell_{+}(w)}_{w,an}(K^p, \nu)^{+,fs}$ is injective and 
${\HH}^{\ell_{-}(w)}_{w,an}(K^p, \nu,cusp)^{-,fs} \rightarrow \overline{\HH}^{\ell_-(w)}_{w}(K^p, -w_{0,M}\kappa-2\rho_{nc}, \chi^{-1}, cusp)^{-,fs}$ is surjective, which gives the third point follows. The last point is clear.
\end{proof}

\begin{rem} There is  an asymmetry between  the $-$ and $+$ theories in point (3). This asymmetry results from our choice to develop the $+$ theory using locally analytic induction sheaves and the $-$ theory using their dual distribution sheaves.
\end{rem}

We let $\mathcal{Z}^!_w$, $\mathcal{E}^!_w$ be the supports of $ \overline{\HH}^{\ell_{\pm}(w)}_{w,an}(K^p, \nu^{un})_{\mathcal{Z}}^{\pm,fs}$ and $ \overline{\HH}^{\ell_{\pm}(w)}_{w,an}(K^p, \nu^{un})_{\mathcal{Z}}^{\pm,fs}$. We let $\mathcal{Z}^!$ and $\mathcal{E} ^!$ be the union of the $\mathcal{Z}^!_w$ and $\mathcal{E}^!_w$.

\begin{thm}\label{thm-eigenvariety2}   The following holds:
\begin{enumerate}
\item  The eigenvariety $\mathcal{E}^{!}$ and the spectral variety $\mathcal{Z}^{!}$ are either empty or equidimensional of dimension  $\dim \mathcal{W}$.  The same is true of the $\mathcal{E}^{!}_w$ and $\mathcal{Z}^{!}_w$ for each $w\in\WM$.

\item  For all $w \in \WM$, $\kappa_{alg}\in X^\star(T^c)^{M,+}$, and $\chi : T^c(\ZZ_p) \rightarrow \overline{F}^\times$ a finite order character, let $\nu_{alg}=-w^{-1}w_{0,M}(\kappa_{alg}+\rho)-\rho$. 
Any eigenclass  in $\overline{\HH}^{\ell_{-}(w)}_w(K^p, \kappa^\vee_{alg}, \chi^{-1})^{-,fs}$ or   $\overline{\HH}^{\ell_{+}(w)}_w(K^p, \kappa_{alg}, \chi)^{+,fs}$ gives a point of $\mathcal{E}_w^{!}$ of weight $\nu=\nu_{alg}\chi$.

\item Let $\kappa_{alg}\in X^\star(T^c)^{M,+}$ and $\chi:T^c(\Z_p)\to \overline{F}^\times$ be a finite order character.  For any eigenclass $c$ in $\overline{\HH}^i(K^p,\kappa_{alg},\chi)^{+,fs}$ or $\overline{\HH}^i(K^p,\kappa_{alg}^\vee,\chi^{-1})^{-,fs}$ there is a $w\in\WM$ with $\ell_{\pm}(w)=i$ so that $c$ gives a point of $\mathcal{E}_w^{!}$ of weight $\nu=\nu_{alg}\chi$ where $\nu_{alg}=-w^{-1}w_{0,M}(\kappa_{alg}+\rho)-\rho$.  
\end{enumerate}
\end{thm}

 \begin{proof} This first two points follow from proposition \ref{prop-properties-interior}.  The third point follows from the second point and corollary \ref{coro-concentration-interior}.
 \end{proof}
 
This theorem in particular implies that any finite slope system of Hecke eigenvalues in classical or overconvergent interior cohomology admits an analytic deformation of maximum dimension.

In light of point (3) of the theorem, it is natural to ask: given an eigenclass $c$ occurring in $\overline{\HH}^i(K^p,\kappa_{alg},\chi)^{+,fs}$, what is the set $\WM(c)$ of $w\in\WM$ with $\ell_+(w)=i$ for which $c$ occurs in $\overline{\HH}^i_w(K^p,\kappa_{alg},\chi)$?  It is natural to compare $\WM(c)$ to the set $C(\kappa_{alg})^+$, i.e. the set of $w\in\WM$ for which $\nu_{alg}+\rho\in X^{\star}(T^c)^+$.
  
\begin{prop}\label{determine-WM(c)} Assume that the eigenclass $c$ satisfies $+,ss^M(\kappa_{alg})$.  Then $\WM(c)\subseteq C(\kappa)^+$. 
\end{prop}
\begin{proof}
By theorem \ref{thm-strongslopes-interior} (2) below, $\overline{\HH}_w^i(K^p,\kappa_{alg},\chi)^{+,ss^M(\kappa)}=0$ unless $w\in C(\kappa)^+$. 
\end{proof}

In particular if $C(\kappa)^+$ has only one element of length $i$ (for example if $\kappa+\rho$ is regular) then $\WM(c)$ is determined under this small slope hypothesis.

\begin{ex}
We give a basic counter-example to the proposition if we drop the small slope hypothesis. We  assume the Shimura variety is compact, so that the constant function $1\in H^0(S_K,\oscr)$ defines a finite slope interior cohomology class.  Then $\WM(1)=\{Id\}$, and the corresponding point of the eigenvariety $\mathcal{E}$ is in weight $\nu_{alg}=-w_{0,M}\rho-\rho=-2\rho_{nc}$, and so in particular $\nu_{alg}+\rho$ is not dominant.  On the other hand we have $C(0)^+=\{w_0^M\}$.
\end{ex}

\begin{rem} If $c$ is a cohomology class represented by an automorphic representation  $\pi$ which is tempered at $\infty$ and contributes to coherent cohomology in weight $\kappa$, then $\pi_\infty$ is a limit of discrete series which is described by the pair consisting of its infinitesimal character $-\kappa-\rho$ and a chamber $wX^\star(T)^+_{\qq} \subseteq X^\star(T)^{M,+}_\qq$ for $w \in C(\kappa)^+$ (see \cite{harris-ann-arb}, sect. 3.3). We can ask if this class lifts to a point in $\mathcal{E}_{w}^!$.
\end{rem}

\subsection{Improved slope bounds for interior cohomology and applications}
Using the interior eigenvariety we are able to prove that the conjectured slopes bounds \ref{conj-strongslopes2}, \ref{conj-strongslopes}, and \ref{conj-strongslopes-classical} hold for interior cohomology.  The idea is that classical points in regular weight satisfy the correct slope bound by corollary \ref{coro-lafforgue-estimates}, and these points are dense in the interior eigenvariety.
\begin{thm}\label{thm-strongslopes-interior}
\begin{enumerate}
\item Fix $w\in \WM$ and a locally algebraic weight $\nu=\nu_{alg} \chi$ for $\nu_{alg} \in X^\star(T^c)$ and $\chi:T^c(\ZZ_p)\to \overline{F}^\times$ a finite order character.  For any character $\lambda$ of $T^\pm$ on $\overline{\HH}^{\ell_{\pm}(w)}_{w,an}(K^p,\nu)^{\pm,fs}$ we have $v(\lambda)\geq 0$ in the $+$ case and $v(\lambda)\leq 0$ in the $-$ case.
\item Fix $w\in\WM$, $\kappa\in X^\star(T^c)^{M,+}$, and a finite order character $\chi:T^c(\ZZ_p)\to F^\times$.  For any character $\lambda$ of $T^\pm$ on $\overline{\HH}^{\ell_{\pm}(w)}_w(K^p,\kappa,\chi)^{\pm,fs}$ we have $v(\lambda)\geq w^{-1}w_{0,M}(\kappa+\rho)+\rho$ in the $+$ case and $v(\lambda)\leq w^{-1}(\kappa+\rho)-\rho$ in the $-$ case.
\item Fix $\kappa\in X^\star(T^c)^{M,+}$ and a finite order character $\chi:T^c(\ZZ_p)\to \overline{F}^\times$.  Let $\nu=-w^{-1}w_{0,M}(\kappa+\rho)-\rho$ for any $w\in C(\kappa)^+$.  Then for any character $\lambda$ of $T\pm$ on $\overline{\HH}^i(K^p,\kappa,\chi)^{\pm,fs}$  we have $v(\lambda)\geq -\nu$ in the $+$ case and $v(\lambda)\leq -w_0\nu$ in the $-$ case.
\end{enumerate}
\end{thm}
\begin{proof}
In the $-$ case, the first point implies the second point by theorem \ref{thm-eigenvariety2} (3).  In the $+$ case, one repeats the constructions of this section, exchanging the roles of the sheaves of locally analytic functions and distributions.  The second point implies the third point as in the proof of proposition \ref{prop-strongslopes-implies-classical} using corollary \ref{coro-concentration-interior}.

We now prove the first point.  The eigensystem $\lambda$ gives a point $(\nu,\lambda_p,\lambda^S)$ of $\mathcal{E}^!$ (where $\lambda_p=\lambda$ in the $+$ case and $\lambda^t$ in the $-$ case.)  We can find another point $(\nu',\lambda_p',\lambda^{'S})$ which satisfies $v(\lambda'_p)=v(\lambda_p)$ and $\nu'=\nu'_{alg}\chi$ is locally algebraic, and $\nu'_{alg}\in X^\star(T^c)^+$ is sufficiently large so that $v(\lambda'_p)$ satisfies $+,sss(\nu)$.  Then this point is classical by theorem \ref{thm-eigenvariety} (2), and so the slope bound is satisfied by corollary \ref{coro-lafforgue-estimates}.
\end{proof}

As a consequence we deduce a vanishing theorem for interior cohomology which improves on theorem \ref{thm-first-vanishing-sss}. 

\begin{thm}\label{thm-strong-vanishing-ss}  Let $\kappa\in X^\star(T^c)^{M,+}$ and let  $\chi : T^c(\ZZ_p) \rightarrow \overline{F}^\times$ be a finite order character.  We have that $\overline{\HH}^i(K^p, \kappa, \chi)^{ss^M(\kappa)}$ is supported in the range $[\ell_{\min}(\kappa),\ell_{\max}(\kappa) ]$.
\end{thm}

\begin{proof}  
Theorem \ref{thm-strongslopes-interior} (2) implies that the complex $\overline{\mathcal{C}ous}(K^p, \kappa, \chi)^{\pm,ss^M(\kappa)}$ is concentrated in the range $[\ell_{\min}(\kappa),\ell_{\max}(\kappa) ]$, which implies the theorem by corollary \ref{coro-concentration-interior}
\end{proof}

As another application we are able to improve our classicality theorem for overconvergent cuspforms.

\begin{thm}\label{thm-strong-classicality}
Let $\kappa\in X^\star(T^c)^{M,+}$ satisfy $C(\kappa)^+=\{1\}$, and let $\chi:T^c(\ZZ_p)\to \overline{F}^\times$ be a finite order character.  Then $\HH^0_{Id}(K^p,\kappa,\chi,cusp)^{+,ss^M(\kappa)}\to\HH^0_{Id}(K^p,\kappa,\chi)^{+,ss^M(\kappa)}$ factors through $\HH^0(K^p,\kappa,\chi)^{+,ss^M(\kappa)}$.
\end{thm}

\begin{proof}
We have a commutative diagram
\begin{eqnarray*}
\xymatrix{&&\HH^0_{Id}(K^p,\kappa,\chi,cusp)^{+,fs}\ar[r]\ar[d]& \mathcal{C}ous^1(K^p,\kappa,\chi,cusp)^+\ar[d]\\
0\ar[r]&\HH^0(K^p,\kappa,\chi)^{+,fs}\ar[r]&\HH^0_{Id}(K^p,\kappa,\chi)^{+,fs}\ar[r]& \mathcal{C}ous^1(K^p,\kappa,\chi)^+}
\end{eqnarray*}
where the bottom row is exact by proposition \ref{prop-deg-spect-sequ2}.  After passing to the $+,ss^M(\kappa)$ part of this diagram, the right vertical arrow is 0 by theorem \ref{thm-strong-vanishing-ss}.
\end{proof}

\begin{rem}
Theorem \ref{thm-strong-classicality} says that a small slope overconvergent cuspform comes from a classical eigenform.  Although we have not proved it, this eigenform should also be cuspidal (for instance this would follow from conjeture \ref{conj-strongslopes}).  We note that it is possible for an overconvergent cuspform which does not have small slope to come from a non cuspidal eigenform.  A basic example is the non-ordinary $p$-stabilization of an Eisenstein series for $\mathrm{GL}_2$.
\end{rem}

\subsection{Application: construction of Galois representations and local-global compatibility $p$}\label{subsec-localglobal} 
Let $L$ be a totally real or CM number field.  The class of regular, algebraic, essentially (conjugate) self dual, cuspidal automorphic representations for $\mathrm{GL}_n/L$ have been studied intensively. In particular, one can attach to them compatible systems of Galois representations which satisfy many of the expected properties (see, e.g., \cite{MR3272052}, \cite{BLGGT}):
\begin{thm}[Bellaiche, Caraiani, Chenevier, Clozel, Harris, Kottwitz,
  Labesse, Shin, Taylor, \ldots]\label{thm-regular-galois}  Let $\pi$ be a regular, algebraic, essentially (conjugate) self dual,  cuspidal automorphic representation of $\mathrm{GL}_n/L$ with $\pi^c = \pi^\vee \otimes \chi$ and infinitesimal character $\lambda = ( ( \lambda_{1, \tau}, \ldots, \lambda_{n, \tau})_{\tau \in \mathrm{Hom}(L,\C)})$ with $\lambda_{1, \tau} > \cdots > \lambda_{n, \tau}$. 
  Then for each isomorphism $\iota:\overline{\qq}_p\simeq \C$ there is a continuous
  Galois representation $\rho_{\pi, \iota}: G_L \rightarrow \mathrm{GL}_n(\overline{\qq}_p)$
  such that:
  \begin{enumerate}
  \item $\rho_{\pi, \iota}^c \simeq \rho_{\pi, \iota}^\vee \otimes \epsilon_p^{1-n} \otimes \chi_{\iota}$ where $\chi_{\iota}$ is the $p$-adic realization of $\chi$ and $\epsilon_p$ is the cyclotomic character.
  \item $\rho_{\pi, \iota}$ is pure.
\item $\rho_{\pi, \iota}$ is de Rham at all places dividing $p$, with  $\iota^{-1} \circ \tau$-Hodge-Tate weights: $(-\lambda_{n,\tau} + \frac{n-1}{2}, \cdots, -\lambda_{1,\tau} + \frac{n-1}{2})$.
 \item  For all finite  place 
$v$ one has
$$ \iota \mathrm{WD} (\rho_{\pi, \iota} \vert_{G_{F_v}})^{F-ss} = \mathrm{rec} (\pi_v \otimes \vert \det \vert_v^{\frac{1-n}{2}}).$$
\end{enumerate}
 \end{thm}
The Galois representations of this theorem are usually found in the \'etale cohomology of unitary Shimura varieties, while in the remaining cases they are constructed by $p$-adic interpolation.

In \cite{F-Pilloni}, section 9, we defined a certain class of cuspidal automorphic representations for the group $\mathrm{GL}_n/L$ which (when $L$ is CM) realize in the coherent cohomology of unitary Shimura varieties, by weakening the condition of regular to weakly regular and odd.  To these weakly regular, odd, algebraic, essentially (conjugate) self dual, cuspidal automorphic representations, one can still attach compatible system of Galois representations (\cite{2015arXiv150705922B}, \cite{MR3989256}, \cite{MR3512528}), but at present many of their expected properties are not known. 
The techniques of this paper allows for a new construction of the Galois representation via analytic families. The advantage of this construction is that we can prove some instances of local-global compatibility at $p$, using results of Kisin \cite{MR1992017} on the interpolation of crystalline periods in analytic families, as in the work of Jorza and Mok \cite{MR3039824}, \cite{MR3200667}.

  \begin{thm} Let $\pi$ be a  weakly regular, odd, algebraic, essentially (conjugate) self dual, cuspidal automorphic representation of $\mathrm{GL}_n/L$ with $\pi^c = \pi^\vee \otimes \chi$ and infinitesimal character $\lambda = ( ( \lambda_{1, \tau}, \ldots, \lambda_{n, \tau})_{\tau \in \mathrm{Hom}(L,\C)})$ with $\lambda_{1, \tau} \geq \cdots \geq \lambda_{n, \tau}$. 
  Then for each isomorphism $\iota:\overline{\qq}_p\simeq \C$ there is a continuous
  Galois representation $\rho_{\pi, \iota}: G_L \rightarrow \mathrm{GL}_n(\overline{\qq}_p)$
  such that:
  \begin{enumerate}
  \item $\rho_{\pi, \iota}^c \simeq \rho^\vee \otimes \epsilon_p^{1-n} \otimes \chi_{\iota}$ where $\chi_{\iota}$ is the $p$-adic realization of $\chi$.
 \item $\rho_{\pi, \iota}$ is unramified at all finite places $v \nmid p$ for which $\pi_v$ is unramified and one has
$$\iota WD (\rho_{\pi, \iota} \vert_{G_{L_v}})^{F-ss} = \mathrm{rec} (\pi_v \otimes \vert \det \vert_v^{\frac{1-n}{2}}).$$
\item $\rho_{\pi, \iota}$ has generalized $\iota^{-1}\circ\tau$-Hodge--Tate weights $(-\lambda_{n,\tau} + \frac{n-1}{2}, \cdots,
 -\lambda_{1,\tau} + \frac{n-1}{2})$.
 \item Let $v \mid p$ be a place of $L$ and assume that $\pi_v$ is a regular principal series (i.e. the Jacquet module $(\pi_v)_U$ is a direct sum of $n!$ distinct characters of $T(L_v)$). Then $\rho_{\pi,\iota}\vert_{G_{L_v}}$ is potentially crystalline and 
 $$\iota WD (\rho_{\pi, \iota} \vert_{G_{L_v}})^{F-ss} = \mathrm{rec} (\pi_v \otimes \vert \det \vert_v^{\frac{1-n}{2}}).$$
\end{enumerate}
 \end{thm}

\begin{proof}
The existence of a Galois representation satisfying points (1) and (2) is already known (see \cite{F-Pilloni}, theorem 9.10).  In order to prove points (3) and (4), by base change we may assume that $L$ is CM with maximal totally real subfield $F$, all the primes above $p$ in $F$ split in $L$, and for all primes $v\mid p$ of $L$, $\pi_v$ is a constituent of a principal series representation (and so has finite slope in the sense of section \ref{subsection-jacquet}).

Let $\mathrm{GU}(n)/F$ be the quasi-split unitary similitude group in $n$ variables, and let $G\subset\mathrm{Res}_{F/\qq}$ be the subgroup where the similitude factor lands in $\mathbb{G}_m\subseteq\mathrm{Res}_{F/\qq}\mathbb{G}_m$.  The group $G$ admits a PEL Shimura datum $(G,X)$, so that $\pi$ realizes in the interior coherent cohomology of the corresponding Shimura variety, as explained in the proof of theorem 9.11 of \cite{F-Pilloni} (we note that this is conditional on the main results of \cite{MR3338302}).

The result now follows from theorems \ref{thm-eigenvariety2} and \ref{thm-regular-galois}, using the results of \cite{MR1992017} as explained in proof of theorem 4.1 of \cite{MR3039824}.
\end{proof}

\bibliographystyle{amsalpha}
\bibliography{H1eigenvariety}

\end{document}